\documentclass[10pt]{article}
\usepackage[margin=1in]{geometry}
\usepackage[english]{babel}
\usepackage[utf8]{inputenc}
\usepackage{amsmath, amssymb, amsthm, slashed, xcolor, mathtools}
\usepackage{hyperref}
\usepackage[T1]{fontenc}
\usepackage{graphicx}
\mathtoolsset{showonlyrefs}
\newtheorem{lemma}{Lemma}[section]
\newtheorem{theorem}{Theorem}[section]
\newtheorem{prop}{Proposition}[section] 

\newtheorem{remark}{Remark}
\newtheorem{definition}{Definition}
\newcommand{\ve}{\epsilon}
\newcommand{\R}{\mathbf{R}}

\newcommand{\ap}{\rs} 

\newcommand{\sBo}{X_1}
\newcommand{\sBt}{X_2}

\newcommand{\Sbo}{X_1} 
\newcommand{\Sbt}{X_2}

\newcommand{\mB}{m_{B}}
\newcommand{\mBB}{m_{B,a}} 
\newcommand{\mZB}{\mathcal{Z}_{\mB}}
\newcommand{\mZ}{\mathcal{Z}_{m}}
\newcommand{\ZB}{Z_{\mB}}
\newcommand{\Z}{{\mathcal{Z}_m}}

\newcommand{\dbeta}{\frac{d}{ds} \beta}
\newcommand{\bbeta}{\widetilde{\beta}}
\renewcommand{\S}{\mathbb{S}}

\renewcommand{\div}{\textrm{\,div\,}}
\newcommand{\curl}{\textrm{\,curl\,}}
\newcommand{\pa}{\partial}

\newcommand{\nas}{\slashed{\nabla}}
\newcommand{\pas}{\nas}
\newcommand{\sDelta}{\slashed{\Delta}}

\newcommand{\opa}{\overline{\pa}} 
\newcommand{\pg}{\widetilde{\gamma}} 
\newcommand{\mK}{\widetilde{K}} 
\newcommand{\mJ}{{\widetilde{J}}} 
\newcommand{\mQ}{{\widetilde{Q}}} 

\newcommand{\ev}{\ell} 
\newcommand{\evmB}{\ev^{\mB}} 

\newcommand{\eu}{n} 
\newcommand{\evg}{\ev^g}
\newcommand{\evm}{{\ev^m}}
\newcommand{\eug}{\eu}
\newcommand{\Xo}{X^{\ev}}
\newcommand{\Xt}{X^{\eu}}

\newcommand{\ls}{\xi} 
\newcommand{\rs}{\eta} 

\newcommand{\SL}{+} 

\newcommand{\wK}{\widehat{K}} 

\definecolor{psychedelicpurple}{rgb}{0.87, 0.0, 1.0}

\definecolor{brightpink}{rgb}{1.0, 0.0, 0.5}

\numberwithin{equation}{section}
\begin{document}
\title{The Stability of Irrotational Shocks and the Landau Law of Decay}

\author{Daniel Ginsberg and Igor Rodnianski}

\date{\today}

\maketitle
\abstract{We consider the long-time behavior of irrotational solutions of the three-dimensional
compressible Euler equations with shocks, hypersurfaces of discontinuity across which
the Rankine-Hugoniot conditions for irrotational flow hold.
Our analysis is motivated by Landau's analysis of spherically-symmetric shock waves,
who predicted that at large times, not just one, but two shocks emerge. These
shocks are
logarithmically-separated from the Minkowskian light cone and the fluid
velocity decays at the non-time-integrable rate $1/(t(\log t)^{1/2})$. 
We show that for initial data, which need not be spherically-symmetric, with two shocks in it
and which 
is sufficiently close, in appropriately weighted Sobolev norms,
to an $N$-wave profile, the solution to the shock-front initial value
problem can be continued for all time and does not develop any further
singularities. In particular this is the first proof of global existence
for solutions (which are necessarily singular) of a quasilinear wave equation in three space dimensions
which does not verify the null condition. The proof requires carefully-constructed multiplier
estimates and analysis of the geometry of the shock surfaces.}

\noindent

\tableofcontents

\section{Introduction}

We consider the isentropic compressible Euler equations
describing an ideal gas in $\mathbb{R}^3$,
\begin{align}
  \label{intromass}
  \pa_t \rho + \pa_i(\rho v^i) &=0,\\
  \pa_t(\rho v_i) + \pa_j(\rho v^j v_i) + \pa_i p &= 0,
  \qquad i = 1,2,3.
  \label{introeul}
\end{align}
Here, $v = (v_1, v_2, v_3)$ denotes the fluid velocity,
$\rho \geq 0$ denotes the mass density, $p$
denotes the pressure, and we are summing over repeated
upper and lower indices. The pressure $p$ is determined
from the density $p = P(\rho)$ for a given equation of
state $P$, which is assumed to be smooth, monotone and convex.
The local
well-posedness theory for the system
\eqref{intromass}-\eqref{introeul} with initial data lying
in appropriate function spaces is classical
\cite{Kato75}. On the other
hand, it is well-known that regular solutions to
\eqref{intromass}-\eqref{introeul} can develop singularities
in finite time \cite{Sideris85,Christodoulou07,Speck16,MRRS22}. 
 In particular, they may develop \emph{shocks},
surfaces across which the velocity $v$ and density
$\rho$ are bounded but not differentiable. It was shown by Majda \cite{Majda83,MajdaThomann87}
that given initial data for \eqref{intromass}-\eqref{introeul}
which already has a shock in it, the solution and the shock can be
continued for a short time.  Given this, it is natural to
ask what happens at large times, after the formation of a shock.

This question was first addressed by Landau \cite{Landau65} (whose conclusions were rediscovered in a somewhat sharper form
by Whitham \cite{Whitham}),
who considered the long-time behavior of \emph{irrotational} and \emph{spherically symmetric}
solutions to \eqref{intromass}-\eqref{introeul}.  Using a combination of geometric and approximation arguments, 
Landau argued that far away from a  spherically symmetric body, where sound waves decay like $1/r$,
not just one but \emph{two} shocks eventually emerge; these shocks
are approximately located at $\{r = t \pm (\log t)^{1/2}\}$,
and the velocity along the shocks decays
at the non-integrable rate $|v| \sim \frac{1}{t(\log t)^{1/2}}$.

To translate Landau's picture into precise mathematical language, we first observe 
that 
the \emph{irrotational} isentropic Euler equations reduce to a quasilinear wave equation
for the potential $\Phi$ such that $v = \nabla \Phi$, 
\begin{equation}
  \label{intromodel}
  \Box \Phi + \pa_\alpha({\gamma}^{\alpha\beta}(\pa \Phi) \pa_\beta \Phi) = 0,
\end{equation}
where $\Box$ denotes the Minkowskian wave operator and ${\gamma}(0) = 0$.
We note that under the conditions $P'(1)>0,  P''(1)\ne 0$,  \eqref{intromodel} does not satisfy the classical null
condition and as a consequence solutions may develop singularities in finite time even for initial data that are small, smooth 
and well-localized. In some situations these singularities are shocks in which case one can attempt to extend the local classical 
solution to a global weak solution containing shocks. 

{Landau's result can be interpreted as consisting of two statements:
\begin{itemize}
\item At least in the small data regime, or alternatively, far out, the final state of any solution contains two spherical shocks,
\item At large times $t$, the shocks are located at  $\{r - t\sim \pm (\log t)^{1/2}\}$,
and the velocity along the shocks decays with the rate  $\sim \frac{1}{t(\log t)^{1/2}}$.
\end{itemize}
These types of statements, but with the shock separation $\sim t^{1/2}$ and the shock strength decay $\sim t^{-1/2}$, 
are known for $1+1$-dimensional system of conservation laws  \cite{Lax1957}, \cite{Liu}, and even for large 
data in the case of {\it scalar} conservation laws. They are however completely out of reach for higher dimensional problems, and
even more so outside of spherical symmetry.}

The result of Landau can be motivated by the following heuristic, {very different from his original arguments.}
Introduce null coordinates $u = r-t$ and
$v = r+t$ and define $\Psi = r\Phi$. Restricting to the wave zone
$r \sim t$ and dropping nonlinear terms which verify
the null condition, and which should play no role
in the long-time behavior ({see below however}), in spherical symmetry
the equation \eqref{intromodel} takes the form
\begin{equation}
  \label{}
  -4\pa_v \pa_u \Psi + \frac{2}{v} \pa_u (\pa_u\Psi)^2 = 0,
\end{equation}
(see \eqref{wave0} and \eqref{ffs0})
and introducing $s = \log v$ and $B = \pa_u\Psi$, we find that
$B$ satisfies Burgers' equation
\begin{equation}
  \label{introburger}
    \pa_s B + \frac{1}{2} \pa_u B^2 = 0.
\end{equation}
It was shown by Hopf \cite{Hopf1950,Lax1957},
that at large times, the solution of \eqref{introburger} converges
to
 \begin{equation}
  \Sigma = \begin{cases} \frac{u}{s},  \qquad
\text{ when } |u|\leq s^{1/2},\\
0, \qquad \text{ otherwise }, \end{cases}
  \label{introprofile}
 \end{equation}
 which is a classical solution of \eqref{introburger}
 away from
 \begin{equation}
   \label{modelshockdef}
   \Gamma^L_{\Sigma} = \{u = s^{1/2}\},
   \qquad
   \Gamma^R_{\Sigma} = \{u = -s^{1/2}\},
 \end{equation}
 across which the classical Rankine-Hugoniot conditions for
 Burgers' equation hold; that is, the above is a solution of
 \eqref{introburger} with (Burgers') shocks.
 Unwinding definitions, the solution \eqref{intromodel}-\eqref{modelshockdef}
 has velocity of size $|v| \sim |B|/r \sim \frac{1}{t(\log t)^{1/2}}$ at the shocks
 $u = \pm s^{1/2}$. Thus, Landau's result {could be possibly} understood as the statement
 that at large times, the solution to \eqref{intromodel} should (in appropriate
 variables) approach the profile \eqref{introprofile}-\eqref{modelshockdef}.

{ It is worth pointing out that in the above heuristic replacing the original equation by the effective Burgers equation 
required removing the nonlinear terms satisfying the null condition. We argued that one can do so since such terms do not
influence the long term behavior. This assertion however is well established only for {\it smooth} solutions of quasilinear 
wave equations. A priori there is no reason as to why the same logic should apply to solutions with shocks. In fact, it is 
not difficult to see that the dropped terms, evaluated on the profile \eqref{introprofile}, will contribute $\delta$-functions to the equation.
To properly account for this one needs to observe that in fact the profile $\Sigma$ can be upgraded to a 2-dimensional
family of profiles 
  \begin{equation}
  \label{prof}
  \Sigma_{\ls,\rs} = \begin{cases} \frac{u}{s},  \qquad
\text{ when } -\rs\, s^{1/2}\leq u\leq \ls\, s^{1/2},\\
0, \qquad \text{ otherwise }, \end{cases}
 \end{equation}
 with {\it arbitrary} constants {$\ls,\rs\ge 0$}. Then, first, the correct statement about the solutions of the Burgers equation 
 is that they converge to {\it one} of the profiles $\Sigma_{\ls,\rs}$. This, of course, is already in \cite{Hopf1950}. And, second, is that it is precisely the freedom of choice of $\ls,\rs$ 
 that could allow one to {\it modulate}, that is to make $\rs$ and $\ls$ $s$-dependent, to have any hope to account for the terms neglected 
 in the original equation. None of this has been implemented even in spherical symmetry.}

The goal of this paper is to partially justify Landau's description
of the late-time behavior of irrotational solutions to
\eqref{intromass}-\eqref{introeul}. We do not address the formation
of a second shock ({or the first one, for that matter}), nor do we show that arbitrary solutions with
two shocks must behave as in Landau's prediction. What we do show is
that initial data, \emph{not necessarily spherically symmetric},
which is sufficiently close to the model shock
profile \eqref{introprofile}-\eqref{prof} (which
has two shocks already in the initial data) leads to a solution
to the shock front problem {which remains close to the modulated model shock
for all times. What that means is that the solution can be decomposed into the sum of the profile 
$\Sigma_{\rs(s,\omega),\ls(s,\omega)}$
with functions $\rs(s,\omega), \ls(s,\omega)$ depending on $s=\log (t+r)$ and $\omega\in \mathbb{S}^2$ (the shock surfaces are no longer spherically
symmetric) which converge to bounded limits $\rs(\omega),\ls(\omega)$ as $s\to+\infty$, and {\it sound waves}, which are smooth away from the shock 
surfaces } {(note that $u=t-r$ and thus the right shock {lies in the region} $u<0$)}
$$
  \Gamma^L_{\Sigma} = \{u = \ls(s,\omega)s^{1/2}\},
   \qquad
   \Gamma^R_{\Sigma} = \{u = -\rs(s,\omega) s^{1/2}\},
   $$
{and which decay faster (this statement applies only to the region between the shocks where the profile is nontrivial) than the profile itself.
The functions $\rs(\omega), \ls(\omega)$ encode the 
asymptotic behavior of the shocks and, together with the asymptotic behavior of the profile $\Sigma$, 
provide the precise statement of the {\it Landau law of decay} for weak compressible shocks.} {We note that the $N$-wave shape 
of the profile $\Sigma$ in \eqref{introprofile}, which we assume our initial data is close to, is precisely the shape that Landau claims 
should emerge at late times. }

{The statement about asymptotic behavior of such solutions contains their global existence as weak solutions containing two shocks.
 In particular we show that such solutions do not develop any further singularities, either away from the shocks or on their surfaces.
 The latter is particularly interesting in view of the fact that in the absence of spherical symmetry shock surfaces may be unstable 
 to {\it corrugation} \cite{LLbook}.}
 
 {The question of existence of higher dimensional {\it global} solutions containing shocks had been raised by Majda in his work 
on local well-posedness of shock solutions. This paper in particular resolves open problem 4.6.2 
from \cite{Majda84} in the irrotational setting. }

{From the point of view of theory of general quasilinear wave equations \eqref{intromodel}, our result is the first proof of 
global well-posedness (for solutions with initial data given in a small neighborhood, in a weighted Sobolev norm, of the two-shock 
profile) for such an equation that {\it does not} verify the null condition;
we emphasize that such solutions are not (and cannot be expected to be)
smooth, but instead are smooth away from two hypersurfaces across
which natural jump conditions hold. While the question of global well-posedness  (with small initial data in appropriately weighted Sobolev spaces) for quasilinear wave equations (even systems) of the type
\begin{equation}
  \notag
  \Box \Phi^i + \pa_\alpha(h^{\alpha\beta}(\pa \Phi) \pa_\beta \Phi^i)+q^{\alpha\beta}_{ijk}(\pa\Phi) \pa_\alpha\Phi^j\pa_\beta\Phi^k = 0,
\end{equation} 
satisfying the null condition, \cite{Kla1},
\begin{equation}
\label{null}
\pa_{\ell^\gamma} h^{\alpha\beta}(0)\ell_\gamma \ell_\alpha \ell_\beta=0,\qquad  q^{\alpha\beta}_{ijk}(0) \ell_\alpha\ell_\beta=0, \quad{\text{for all}}\,\, i,j,k \,\,{\text{and all null}}\, \,\ell:\, m^{-1}(\ell, \ell)=0
\end{equation}
with $m$ -- the Minkowski metric, is always answered in the affirmative and has been very well understood, going back to \cite{Christ86,Kla} and can, in some cases, be 
even extended with the same answer to systems satisfying the {\it weak null condition} \cite{LR} and nonlinearities depending on $\Phi$ 
instead of $\pa\Phi$, see e.g. \cite{LindbladRodnianski2010}, \cite{Keir}, in the absence of the null or the weak null conditions, the question has been completely open. In those cases, the analysis stopped at the statement of singularity formation, going back to \cite{John} and \cite{Sideris} and, in the specific context of the compressible Euler equations, followed by the more recent results referred to earlier, or the 
statement of {\it almost global existence}, going back to \cite{JK}, asserting that a classical solution will exist on the time interval exponential in the inverse size of initial data. 

To our knowledge, no examples of global solutions, classical or weak,  are known 
for either the wave equation \eqref{intromodel} without the null condition or the compressible Euler equations on 
$\mathbb{R}^3$ in the regime of small data or  
near the equilibrium state $v=0, \rho=1$, respectively. Even in other regimes, we are not aware of any results on the wave 
equation, and for the compressible Euler equations, the only exceptions are the results in \cite{Grassin,Sid,Jang} (and related works) where 
global {\it classical} solutions (in \cite{Sid,Jang} considered as a free boundary problem
with physical vacuum) had been constructed for initial data with velocity satisfying an {\it expansion}
condition with the density $\rho$ vanishing outside of compact set. Such problems and the corresponding solutions, of course, 
lie far away from the problem on the whole $\mathbb{R}^3$ near the equilibrium state $v=0, \rho=1$ studied here.
}

We now formulate the equation and jump conditions
as well as the notion of shock front initial data more precisely.
A rough version of our main theorem, Theorem \ref{mainthm} can be found in
Theorem \ref{roughthm}.

In terms of the enthalpy $w$,
\begin{equation}
  \label{introenthdef}
    w(\rho) = \int_1^\rho \frac{P'(\lambda)}{\lambda} d\lambda,
\end{equation}
the equations \eqref{introeul} read
\begin{equation}
  \label{introeulenth}
    \pa_t v_i + \pa_j(v^j v_i) + \pa_i w = 0.
\end{equation}
It follows from this equation in the usual way that if $\omega = \curl v$
vanishes initially and the solution remains smooth, then $\omega = 0$
at later times as well. It is therefore sensible to look for solutions
of the form $v = \nabla \Phi$, and inserting this into \eqref{introeulenth}
we find
\begin{equation}
  \label{introbern}
    \pa_t \Phi + \frac{1}{2} |\nabla_x \Phi|^2 = - w(\rho).
\end{equation}
If $P' > 0$, we can solve \eqref{introenthdef} for $\rho = \rho(w)$ and
we can then solve \eqref{introbern} for $\rho = \varrho(\pa \Phi)$.
The dynamics are then completely determined by the continuity equation \eqref{intromass},
which is the following quasilinear wave equation, 
\begin{equation}
  \label{introHeqn}
  \pa_\mu H^\mu(\pa\Phi) = 0,\qquad \text{ where } H^0(\pa \Phi) = \varrho(\pa \Phi)
  \text{ and }
  H^i(\pa \Phi) = \varrho(\pa \Phi) \nabla^i\Phi.
\end{equation}
Here, and in what
follows, Greek indices $\mu, \nu,...$ run over 0,1,2,3 and
Latin indices $i, j,...$ run over spatial indices 1,2,3. This is precisely the wave equation \eqref{intromodel}. Under the convexity assumption on the equation of state $P'(1)>0$ and $P''(1)\ne 0$, see Appendix \ref{derivation1}, the coefficients
$\gamma^{\alpha\beta}$ {\it do not} satisfy the null condition \eqref{null}:
\begin{equation}
\label{nullg'}
\pa_{\ell^\delta}\gamma^{\alpha\beta}(0)\ell_{\alpha} \ell_{\beta} \ell_\delta \ne 0,\quad \forall \ell: \, m^{-1}(\ell, \ell)=0
\end{equation}
In fact, if we parametrize all null vectors $\ell=\lambda (-1,\omega)$ with $\lambda\in \Bbb R$ and $\omega\in \Bbb S^2$, then 
the right hand side of the above is simply $c\lambda^3$ for some $c\ne 0$.
By rescaling $\Phi$ one can actually assume 
\begin{equation}
\label{nullg}
\pa_{\ell^\delta}\gamma^{\alpha\beta}(0)\ell_{\alpha} \ell_{\beta} \ell_\delta=-\lambda^3
\end{equation}

Now, let $\Gamma \subset \mathbb{R}^{1+3}$ be a $C^2$ hypersurface.
We say that $\Phi$
has a \emph{shock} along $\Gamma$ if 
$\Phi$ is a classical solution to \eqref{introHeqn} away from
$\Gamma$, and along $\Gamma$ the Rankine-Hugoniot conditions hold,
\begin{align}
  \label{introRH1}
  \zeta_\mu[H^\mu(\pa \Phi)] &= 0,\\
  \label{introRH2}
  [\Phi] &= 0.
\end{align}
Here, $\zeta$ is a space-time one-form whose null space
at each point $(t, x)$ is the tangent space $T_{(t,x)}\Gamma$
to $\Gamma$, and $[q]$ denotes the jump in the quantity $q$
across $\Gamma$: if $D^{\pm}$ denote the regions to either side
of $\Gamma$ and $q_{\pm}$ denote the limits of $q$ at $\Gamma$ taken from
the regions $D^{\pm}$, then $[q] = q_+ - q_-$.

We discuss the nature of the conditions \eqref{introRH1}-\eqref{introRH2}
in Section \ref{metricdefsection} and their relation to the compressible
Euler equations \eqref{intromass}-\eqref{introeul} in Section
\ref{introenth}. For now, just note that \eqref{introRH1} ensures that $\Phi$
is a weak solution to \eqref{introHeqn} and \eqref{introRH2} ensures that
$v = \nabla_x\Phi$ is a weak solution to $\curl v = 0$.
The surface $\Gamma$ needs to
be determined along with $\Phi$ so that \eqref{introRH1}-\eqref{introRH2}
hold. We then come to the following initial value problem.

\begin{definition}[The (restricted) shock front initial value problem]
\label{sfpdef} 
    Let $\Gamma_0 \subset \mathbb{R}^3$ be a $C^2$ surface and let
    $(\Phi_0^-,\Phi_1^-)$ and $(\Phi_0^+, \Phi_1^+)$ be 
    initial data posed at $t = t_0$ for the
    wave equation \eqref{introHeqn} defined on either side of $\Gamma^0$.
    We say that $(\Gamma, \Phi^-, \Phi^+)$
    is a solution to the (restricted) shock front problem if the 
    hypersurface $\Gamma \subset \mathbb{R}^{1+3}$ satisfies
    $\Gamma\cap\{t = t_0\} = \Gamma_0$, and if the
    $\Phi^{\pm}$ are classical solutions of \eqref{introHeqn} on either
    side of the surface $\Gamma$ with initial data $(\Phi_0^{\pm}, \Phi_1^{\pm})$
    so that at $\Gamma$, the jump conditions
    \eqref{introRH1}-\eqref{introRH2} hold.
\end{definition}
The above definition is extended to the case of more than one shock in the natural way.
It was shown in \cite{MajdaThomann87} that the above initial-value problem
has a unique local-in-time solution for \emph{shock front initial data}, initial
data $(\Gamma_0, \Phi^{\pm}_0, \Phi^{\pm}_1)$ satisfying certain 
compatibility and determinism conditions, discussed in Section
\ref{metricdefsection}.

We can now give the rough statement of our main theorem, Theorem \ref{mainthm}. 
\begin{theorem}
  \label{roughthm}
  Fix shock front initial data, posed at a large initial time, which is sufficiently close, in appropriate
weighted Sobolev norms, to the model shock profile
\eqref{prof}. That is, we assume for some sufficiently large time $t_0$, the data for the 
potential $\Phi$  is close to the profile 
$$
  \Phi= 
  \begin{cases} 
      \frac{(t_0-r)^2}{r\log (t_0+r)},  &\qquad
        \text{ when } -r^R(t_0,\omega)\, \log t_0^{1/2}\leq t_0-r\leq r^L(t_0,\omega)\,
\log t_0^{1/2},
        \\
    0, &\qquad \text{ otherwise }, \end{cases}
 $$
with the functions $r^L(t_0,\omega), r^R(t_0,\omega)$ sufficiently close to {constants $\xi,\eta>0$, bounded away from $0$}.
Then the shock front initial
value problem from Definition \ref{sfpdef} {for the equation \eqref{intromodel}, satisfying the condition \eqref{nullg},} has a unique \underline{global-in-time} solution
$(\Gamma^L, \Gamma^R, \Phi^L, \Phi^C, \Phi^R)$ with two shocks
$\Gamma^L, \Gamma^R$, where
$\Gamma^R$ lies to the exterior of $\Gamma^L$, and where
$\Phi^L$ is defined in the region $D^L$ to the
left of the left shock $\Gamma^L$, $\Phi^C$ is defined in the region $D^C$
between the shocks $\Gamma^L, \Gamma^R,$ and $\Phi^R$ is
defined in the region $D^R$ to the right of the right shock $\Gamma^R$.
See Figure \ref{shockdiagram}.

The solution has the following asymptotic behavior.

\begin{itemize}
    \item 
The time slices $\Gamma^A_{t'} = \Gamma^A \cap \{t = t'\}$
are described by
\begin{equation}
  \label{}
  \Gamma^L_t = \{ x  \in \mathbb{R}^3 : r = t - (\log t)^{1/2} r^L(t,\omega)\},
  \quad
  \Gamma^R_t = \{ x  \in \mathbb{R}^3 : r = t + (\log t)^{1/2} r^R(t,\omega)\},
\end{equation}
where $r = |x|$ and $\omega = x/r$,
for sufficiently smooth functions $r^L, r^R$ with bounded limits as $t\to \infty$.

\item The potentials $\Phi^L, \Phi^C, \Phi^R$ enjoy the following pointwise
    decay estimates along $D^A_{t'} = D^A \cap \{t = t'\}$,
\begin{equation}
  \label{pwbdrough}
  \lim_{t \to \infty}
  (1 + t) (1 + \log t)^{1/2} 
    \left( 
        \|\pa \Phi^L(t)\|_{L^\infty(D^L_t)} + \|\pa (\Phi^C - \tfrac{(t-r)^2}{2r\log (t+r)})\|_{L^\infty(D^C_t)}
        + \|\Phi^R(t)\|_{L^\infty(D^R_t)}
    \right) = 0.
\end{equation}
\end{itemize}
That is, the solution remains close to an appropriately modulated version of
the model shock profile \eqref{introprofile}-\eqref{modelshockdef}
for all time and \underline{no further singularities emerge}.
\end{theorem}
The function $(t-r)^2/(2r\log(t+r))=u^2/(2rs)$ is just the profile \eqref{introprofile}
expressed at the level of $\Phi$ instead of $B = \pa_u(r\Phi)$.
Note that $\pa (u^2/(2rs)) \sim \frac{1}{(1+t)(1+\log t)^{1/2}}$ when
$u \sim \pm \log t^{1/2}$. This is precisely the rate given by the Landau law for the decay of the shock strength.
Of course, we prove and will
require more detailed information than \eqref{pwbdrough}.

Note that at the level of the fluid variables $\rho, v$ the data at time $t_0$ is assumed to be close to 
$$
v= \begin{cases} \frac{(t_0-r)}{r\log (t_0+r)} \frac {x}r,  &\qquad
\text{ when } -r^R(t_0,\omega)\, \log t_0^{1/2}\leq t_0-r\leq r^L(t_0,\omega)\, 
\log t_0^{1/2},\\
0, &\qquad \text{ otherwise }, \end{cases}
$$
while $\rho$ can be found from the Bernoulli equation \eqref{introbern}. For the initial profile, $v$ vanishes identically both in front of the right 
shock and also behind the left shock. As a consequence, in those regions $\rho=1$. Moreover, globally, the value of $v$ is bounded 
by $1/(t_0\sqrt log t_0)$, so that in $L^\infty$ norm $v$ is globally close to $0$. By the same token, provided that the equation of state 
$p=P(\rho)$ is convex in the neighborhood of $\rho=1$, the density $\rho$ is uniformly globally close to $1$.  

\begin{figure}[h!]\label{shockdiagram}
\begin{center}
  \includegraphics[width=.42 \linewidth]{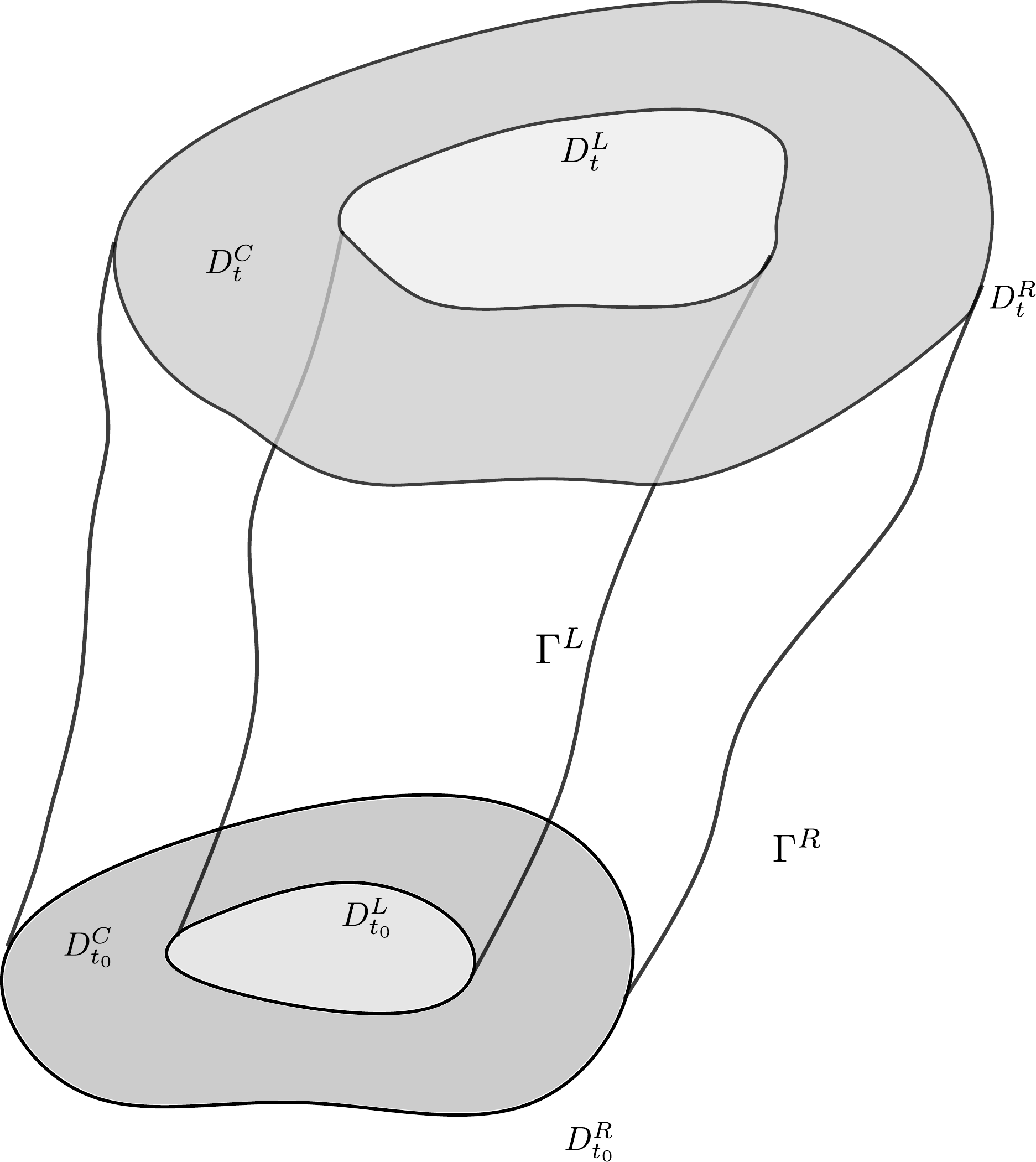}\qquad
  \caption{The surfaces in Theorem \ref{roughthm}. The initial data is
  posed along the time slices $D_{t_0}^L, D_{t_0}^C, D_{t_0}^R$. 
  The shocks are logarithmically-separated from the 
  Minkowskian light cone and satisfy $\Gamma^L \sim \{(t, x): t-|x| = \xi (\log t)^{1/2}\}$
  and $\Gamma^R \sim \{(t,x) : t-|x| = -\eta(\log t)^{1/2}\}$
  for positive constants $p, q$.
  Each shock
  is spacelike with respect to the
  wave equation \eqref{modelql} in the region to the exterior of
  the shock but timelike with respect to the 
  wave equation \eqref{modelql} in
  the region to the interior of the shock.
}
\end{center}
\end{figure}

\subsection{Strategy of the proof}
We now describe the nature of the problem \eqref{introHeqn} with jump conditions
\eqref{introRH1}-\eqref{introRH2} and the strategy we use to
prove our main theorem. In section \ref{metricdefsection} we reformulate
the system as an initial-boundary value problem for the potentials
$\psi_A$ and the positions of the shocks. 
In section \ref{introensec}, we describe
the construction of the energy norms that we will use, and in sections
\ref{issuesection}, we describe the main issues that come up
in the course of the proof of the energy estimates and their resolutions.

In the following sections we will continue using the null variables
 $$
 u = t-r,\quad v=t+r,\quad 
 s = \log v.
 $$

\subsubsection{The determinism conditions and formulation as an initial-boundary value problem}
\label{metricdefsection}

    Given the regions $D^L, D^C, D^R$ as in the above theorem,
    we let $\sigma$ denote the following approximate solution
    (the ``model shock profile'')
    to \eqref{introHeqn}
    \begin{equation}
      \label{sigmadef}
      \sigma(t, x) = 
      \begin{cases}
          \frac{u^2}{2rs}, \qquad &\text{ in } D^C\\
          0,\qquad &\text{ in } D^L, D^R,
      \end{cases}
    \end{equation}
    whose definition is motivated by \eqref{introprofile}.    If we let $\phi_A = \Phi_A - \sigma$, where index $A$ refers to the 
    regions $L, C$ and $R$, the perturbations $\phi_A$
    satisfy quasilinear wave equations of the form
    \begin{equation}
      \label{modelql}
      \pa_\alpha(h_A^{\alpha\beta}(\pa \phi_A)\pa_\beta \phi_A)
      = \pa_\alpha( g_A^{\alpha \beta}\pa_\beta\phi^A) + \pa_\alpha \widetilde{j}^\alpha(\pa \phi_A) 
      = 0, \qquad \text{ in } D^A,
    \end{equation}
    where $\widetilde{j}$ is a quadratic nonliearity,
    and where the linearized metrics $g_A$ are given by
    $g_L = g_R = m$ and $g_C = \mB$, where $m$ denotes the Minkowski metric and $\mB$
    is the ``Burgers' metric '',
    \begin{equation}
      \label{}
        m = -dt^2 + dx^2,
        \qquad
        \mB = m + \frac{u}{vs} dv^2.
    \end{equation}
    In \eqref{modelql}, we are suppressing various small and rapidly-decaying error terms 
    that appear in the central region,
    which arise from the fact that $\sigma$ is not an exact solution to
    \eqref{introHeqn}. 

    We then need to solve the equations \eqref{modelql} in the regions $D^L, D^C, D^R$. These
    regions are separated
    by the shocks $\Gamma^L, \Gamma^R$ (see Figure \ref{shockdiagram}), which are assumed close to the model
    shocks $\Gamma^L_{\Sigma} = \{u = s^{1/2}\}$, $\Gamma^R_{\Sigma} = \{u = -s^{1/2}\}$. A calculation
    (see Lemma \ref{causallemma}) reveals that the shocks and the metrics $m, m_B$ satisfy the following
    \emph{determinism conditions}: the right shock $\Gamma^R$ is spacelike with respect to the Minkowski
    metric (and thus with respect to small perturbations of the Minkowski metric), but
    \emph{timelike} with respect to the metric $m_B$. As a consequence, the solution to 
    \eqref{intrompsi} in the rightmost region $D^R$, and in particular along the right side of the right shock
    $\Gamma^R$,
    is entirely determined for $t \geq t_0$ by initial data posed in $D^R_{t_0}$. On the other hand,
    the solution of the equation \eqref{intromBpsi} in the central region $D^C$ is determined both 
    by initial data in $D^C_{t_0}$ and boundary data along $\Gamma^R$, which needs to be chosen
    so that the Rankine-Hugoniot conditions \eqref{introRH1}-\eqref{introRH2} hold.  
    Similarly, the left shock $\Gamma^L$ is spacelike with respect to the metric
    $\mB$ but timelike with respect to $m$, and so in the leftmost region we need to
    prescribe boundary data along the left side of $\Gamma^L$.

    In what follows, it will be more natural
    to work in terms of the variable $\psi_A = r\phi_A$, in which case, in
    $D^L$ and $D^R$, \eqref{modelql} takes the form
    \begin{equation}
      \label{intrompsi}
      -4\pa_v\pa_u \psi_A + \sDelta \psi_A + \pa_\alpha j^{\alpha}(\pa \psi_A) = 0, 
      \qquad \text{ for } A \in \{L, R\},
    \end{equation}
    for a nonlinearity $j$,
    and in $D^C$, it takes the form
    \begin{equation}
      \label{intromBpsi}
      -4 \left(\pa_v + \frac{u}{vs} \pa_u \right)\pa_u \psi_C  
      +\sDelta \psi_C + \pa_\alpha j^\alpha(\pa \psi_C)
      = 0.
    \end{equation}

    In Section \ref{RHsec}, we show that at the right shock, the Rankine-Hugoniot conditions
    \eqref{introRH1}-\eqref{introRH2} imply a nonlinear boundary condition
    for $\psi_C$ of the form
    \begin{equation}
      \label{modelbcrightshock}
      \left(\pa_v + \frac{u}{vs} \pa_u \right) \psi_C + N(\pa \psi_C)
      = \pa_v \psi_R + N(\pa \psi_R)
      \qquad \text{ at } \Gamma^R, 
    \end{equation}
    for a quadratic nonlinearity $N$, which determines (at least in the linearized sense when $N$ can be ignored) 
    $\psi_C$ along $\Gamma^R$
    in terms of $\psi_R$.
    At the left shock, we instead have a nonlinear boundary condition for $\psi_L$
    of the form
    \begin{equation}
      \label{modelbcleftshock}
      \pa_v \psi_L +  N(\pa \psi_L)
 = \left(\pa_v + \frac{u}{vs}  \pa_u\right) \psi_C 
 + N(\pa \psi_C)
 \qquad \text { at } \Gamma^L,
    \end{equation}
    which, together with \eqref{intrompsi}, determines $\psi_L$ along $\Gamma^L$ in terms of $\psi_C$.
    The above expressions are motivated by the fact that
    the fields $\pa_v, \pa_v + \frac{u}{vs}\pa_u$ 
    are null vectors for the metrics $m, \mB$ respectively.

    Since $\psi_R$ is determined entirely from initial data, once
    the position of the right shock is known, the 
    condition \eqref{modelbcrightshock} gives boundary data for
    $\psi_C$ along $\Gamma^R$ in terms of the data coming from
    $\psi_R$. This data and the equation \eqref{intromBpsi}
    determine $\psi_C$ uniquely in the region $D^C$ if 
    the position of the left shock $\Gamma^L$ is known.
    The condition \eqref{modelbcleftshock} then plays
    the same role at the left shock and determines $\psi_L$
    along $\Gamma^L$ in terms of the data coming from
    $\psi_C$.
    In the above discussion, we have assumed that the shocks were fixed
    but in reality we need to determine them at the same time as we
    determine the $\psi_A$. In Section \ref{RHsec}, we show
    that the Rankine-Hugoniot conditions \eqref{introRH1}-\eqref{introRH2}
    give evolution equations for the positions of the shocks. 
    We parametrize the shocks by $\Gamma^A = \{(t, x) \in \mathbb{R}^{1+3} : 
    u = \beta_s^A(\omega)\}$ with  $\omega = x/|x| \in \mathbb{S}^2$,
    for functions $\beta_s^A:\mathbb{S}^2\to\mathbb{R}$, and \eqref{introRH1}-\eqref{introRH2}
    lead to the following evolution equation,
    \begin{equation}
 \dbeta^A_s(\omega) - \frac{1}{2s}\beta_s^A(\omega) = 
 \left(\frac{1}{2} (\pa_u \psi_A + \pa_u \psi_C)
 + N(\pa \psi_A, \pa \psi_C)\right)\bigg|_{u = \beta_s(\omega)},
 \label{}
\end{equation}
where $N$ collects nonlinear error terms. Note that if
the right-hand side is negligible this gives $\beta_s^A \sim \beta_{s_0}^As^{1/2}$
for initial data $\beta_{s_0}^A$ and thus we can recover the assumption
that the shocks are close to the model shocks if this holds initially.

We have arrived at the following initial-boundary value problem.
Given functions $\beta_0^A$
for $A \in \{L, R\}$ which describe the positions of the initial shocks
and so that the initial shocks are close to the initial model shock surfaces
\eqref{modelshockdef}, and given
small initial data for the wave equations \eqref{intrompsi},\eqref{intromBpsi} 
on the initial time slices $D^L_{t_0},D^C_{t_0},D^R_{t_0}$ (defined in the natural
way in terms of the data $\beta^A_0$), solve the wave equations
\begin{alignat}{2}
    -4\pa_v\pa_u\psi_R + \sDelta \psi_R + \pa_\alpha \gamma(\pa \psi_R) &= 0, &&\qquad \text{ in } D^R,\\
    -4\left(\pa_v + \frac{u}{vs}\pa_u\right) \pa_u\psi_C + \sDelta \psi_C + \pa_\alpha \gamma(\pa \psi_C) &= 0, &&\qquad \text{ in } D^C,\\
    -4\pa_v\pa_u\psi_L + \sDelta \psi_L + \pa_\alpha \gamma(\pa \psi_L) &= 0, &&\qquad \text{ in } D^L,
\end{alignat}
subject to the boundary conditions
\begin{alignat}{2}
 \left(\pa_v  + \frac{u}{vs} \pa_u\right)\psi_C + N(\pa \psi_C) &=
 \pa_v \psi_R + N(\pa \psi_R) && \qquad \text{ along } \Gamma^R,\label{introbc1}\\
 \pa_v \psi_L + N(\pa \psi_L) &=
 \left(\pa_v  + \frac{u}{vs} \pa_u\right)\psi_C + N(\pa \psi_C)&& \qquad \text{ along } \Gamma^L,
 \label{introbc2}
\end{alignat}
and where the surfaces $\Gamma^A$ are given by 
$\Gamma^A = \{(t, x) : u = \beta_s^A(\omega)\}$ for 
$\beta_s^A$ solving
\begin{alignat}{2}
 \dbeta_s^A(\omega) - \frac{1}{2s} \beta_s^A(\omega) &= \left(\frac{1}{2} (\pa_u \psi_C + \pa_u \psi_A)
 + N(\pa \psi_A, \pa \psi_C) \right)\bigg|_{u = \beta_s^A(\omega)},
 \label{introevolution}
\end{alignat}
By the local existence theory from
\cite{MajdaThomann87} and the above-mentioned determinism conditions, 
we are guaranteed a local-in-time {\it unique} (in the class of 2-shock solutions) solution to the above problem.
Our goal
is to continue this local-in-time solution for all time.

\subsubsection{The energy estimates and the basic energy identity} 
\label{introensec}
Our proof of global existence uses a carefully constructed hierarchy
of weighted high-order energy estimates whose weights are designed to capture
the expected decay rate of solutions in each of the three regions
$D^L, D^C, D^R$.
These energy estimates are obtained by commuting the equations
\eqref{modelql} with families of vector fields (the ``commutator fields'')
that commute well with the linearized wave operators and then multiplying
the resulting equation by
by $X\psi_A^I$ for well-chosen vector fields $X$ (the ``multiplier fields'')
and integrating 
over the region bounded between two time slices $D^A_t$ and the
shocks, where $\psi_A^I = Z^I \psi_A$ denotes a collection of vector fields $Z$
applied to $\psi_A$. 

In the exterior regions $D^L, D^R$,
we use the standard Minkowskian vector fields as commutator fields and in the central
region $D^C$ we use the commutator fields $\mathcal{Z}_{\mB} = \{s \pa_u, v\pa_v, x_i\pa_j - x_j\pa_i\}$.
The multiplier fields we use are described below, and all of the fields we use 
are recorded in sections \ref{fields} and \ref{vfsection}.
The multiplier fields need to be chosen large
enough that bounds for the resulting energies are strong enough to imply
good pointwise decay estimates, but small enough that the nonlinear error 
terms we encounter in the course of proving the energy estimates can be handled.

The basic calculation that leads to energy estimates is as follows.
If the time slices $D^A_t$ are bounded
between a spacelike (with respect to the linearized metric $g_A$)
shock $\Gamma^S$ and a timelike shock $\Gamma^T$ (either of which
can be empty), integrating with respect to the measure $r^{-2}dxdt$, we arrive at the identity
\begin{equation}
  \label{intromult}
  \int_{D^A_{t_1}} Q_{h_A}(X, N_{h_A}^{D_t^A}) = \int_{D^A_{t_0}} 
  Q_{h_A}(X, N_{h_A}^{D^A_{t_0}}) +
  \int_{\Gamma^S_{t_0, t_1}} Q_{h_A}(X, N_{h_A}^{\Gamma_S})- \int_{\Gamma^T_{t_0, t_1}} Q_{h_A}(X, N_{h_A}^{\Gamma_T})
  + \int_{t_0}^{t_1}\int_{D^A_t}  K_{X,h_A}.
\end{equation}
On the spacelike surfaces $S \in\{D^A_t, \Gamma^S\}$ above,
$N^{\Gamma^S}_{h_A}$ denotes the future-directed normal vector field to $S$
defined with respect to the metric $h_A$, and on the timelike surface
$\Gamma^T$, $N^{\Gamma^T}_{h_A}$ denotes the outward-pointing
normal vector field.
We are also abbreviating $\Gamma_{t_0, t_1} = \Gamma\cap\{t_0 \leq t \leq t_1\}$,
and all surface integrals are taken with respect to the measure induced by
$r^{-2}dxdt$.

The quantity $Q_{h_A}$ is the energy-momentum tensor associated to 
the metric $h_A$ and $\psi_A^I$,
\begin{equation}
  \label{}
  Q_{h_A}(X, Y) = X\psi_A^I Y\psi_A^I - \frac{1}{2} h_A(X, Y)
  h_A^{-1}(\pa \psi_A^I, \pa \psi_A^I),
\end{equation}
and the scalar current $K_{X, h_A}$ associated to $X$ and $h_A$ takes the form
$ K_{X, h_A} = K_{X, g_A} + K_{X, \text{nonlinear}}$, where
$K_{X, \text{nonlinear}}$ collects the nonlinear terms and 
$K_{X, g_A}$ is the scalar current associated to the linearized
metric $g_A$. For the moment,
the exact expressions for these quantities are not important.

We now work out how we expect the above quantities to behave
if the shocks are close to the model shocks \eqref{modelshockdef}
and the potentials $\psi_A$ are sufficiently small.
First, the vector field $n = \pa_u$ is a null vector for both of the
linearized metrics $g_A \in \{m, \mB\}$, and these metrics each admit an additional
null field $\ev^{g_A}$ with $\evm = \pa_v, \evmB = \pa_v + \frac{u}{vs} \pa_u$.
If the multiplier field $X$ takes the form $X = X^n_{g_A} n + X^{\ev}_{g_A} \ell^{g_A}$,
and if $\psi_A$ is small enough that $h_A(\pa \psi_A) \sim g_A$,
then the quantity on the time slices is
\begin{equation}
  \label{}
  Q_{h_A}(X, N^{D_t^A}_{h_A}) \sim X_{g_A}^n (n\psi_A^I)^2 + X^\ell_{g_A}
  \left((\ev^{g_A} \psi^I_A)^2
  + |\nas \psi^I_A|^2\right),
\end{equation}
which is coercive (positive definite) if $X$ is future-directed and timelike
with respect to $g_A$, $g_A(X, X) < 0$.

Along the spacelike surface
$\Gamma^S$, provided $\Gamma^S$ is sufficiently close to
the appropriate model shock \eqref{prof}
and $\psi_A$ is sufficiently small,
we instead have (see Section \ref{emformulageneral})
\begin{equation}
  \label{introQpwSL}
  Q_{h_A}(X, N^{\Gamma_S}_{h_A}) \sim \lambda(v) X^n_{g_A} (n\psi_A^I)^2 + X^\ell_{g_A} (\ev^{g_A} \psi_A^I)^2
  + \left(\lambda(v) X^\ev_{g_A} + X^n_{g_A} \right) |\nas \psi_I^A|^2,
\end{equation}
where the weight $\lambda$ is given by $\lambda(v) = {\rs^A}(1+v)^{-1}(1+s)^{-1/2}$,
{with $\rs^L = \ls$ and $\rs^R = \rs$, the positive constants appearing in \eqref{prof}}.
The expression in \eqref{introQpwSL} is positive-definite if $X$ is timelike and future-directed.

On the other hand, even if $X$ is timelike and future-directed, the energy-momentum
tensor along the timelike surface $\Gamma_T$ is not coercive and we instead have
\begin{equation}
  \label{introQpwTL}
  -Q_{h_A}(X, N^{\Gamma_T}_{h_A} )\sim  \lambda(v) X^n_{g_A} (n\psi_A^I)^2 - X^\ell_{g_A}
  (\ev^{g_A} \psi_A^I)^2 + \left(\lambda(v) X^\ev_{g_A} - X^n_{g_A} \right) |\nas \psi_A^I|^2
\end{equation}
Note that the coefficient of $|\nas \psi_A^I|^2$ need not be positive. Combining the above,
for spacelike and future-directed multiplier fields $X$ we arrive at an energy identity
of the form
\begin{equation}
  \label{roughenident}
  E_X(t_1) + S_X(t_1) +  B^+_{X}(t_1) + \slashed{B}_X(t_1) \lesssim E_X(t_0) + B^-_X(t_1) + 
  \int_{t_0}^{t_1} \int_{D^A_t} |K_{X, \text{nonlinear}}|,
\end{equation}
where the energies on the time slices are
\begin{equation}
  \label{}
  E_X(t) = \int_{D^A_t}  X^n_{g_A} (n\psi_A^I)^2 + X^\ell_{g_A}
  \left((\ev^{g_A} \psi_A^I)^2 + |\nas \psi_A^I|^2\right),
\end{equation}
the space-time integrated quantity $S_X(t_1)$ is contributed
by the linear part of the scalar current,
\begin{equation}
  \label{}
  S_X(t_1) = \int_{t_0}^{t_1} \int_{D^A_t} -K_{X, g_A},
\end{equation}
and the boundary terms $B^{\pm}_X, \slashed{B}_X$ are
\begin{align}
  \label{}
  B^+_X(t_1) &= \int_{\Gamma^S_{t_0, t_1}} \lambda(v) X^n_{g_A} (n \psi_A^I)^2
  + X^\ell_{g_A} (\ev^{g_A} \psi_A^I)^2 +(\lambda(v) X^\ell_{g_A} + X^n_{g_A}) |\nas \psi|^2
  \\
             &\qquad\qquad+ \int_{\Gamma^T_{t_0, t_1}} \lambda(v) X^n_{g_A} (n \psi_A^I)^2\\
  B^-_X(t_1) &= \int_{\Gamma^T_{t_0, t_1}}
      X_{g_A}^\ell (\ev^{g_A} \psi)^2,
      \qquad
      \slashed{B}_X(t_1) = \int_{\Gamma^T_{t_0, t_1}}
      \left(\lambda(v)X^\ev_{g_A} - X^n_{g_A}\right)  |\nas \psi_A^I|^2.
\end{align}
{To illustrate the methodology of these energy estimates consider the exact Minkowski wave equation 
$
\Box \phi=0$ which, relative to $\psi=r\phi$, takes the form}
$$
-4\pa_v\pa_u \psi + \sDelta \psi=0,
$$
{in the right region $D^R$ where $u\le -\rs s^{1/2}$, i.e. $t-r\le \rs \log^{1/2} (t+r)$. Take $X$ to be the Killing field $X=\pa_t=2(\pa_u+\pa_v)$. 
Then
$
K_X=0,
$
the surface $\Gamma^R=\{u=-\rs s^{1/2}\}$ is spacelike, and our energy identity takes the form}
$$
\int_{D^R_{t_1}} \left(|\pa_u \psi|^2 + |\pa_v\psi|^2+|\nas\psi|^2\right) + \int_{\Gamma^R_{t_0,t_1}}  \left(\frac {|\pa_u\psi|^2}{(1+v)(1+s)^{1/2}}
+|\pa_v\psi|^2+|\nas\psi|^2\right)=\int_{D^R_{t_0}} \left(|\pa_u \psi|^2 + |\pa_v\psi|^2+|\nas\psi|^2\right) 
$$
{The small weight $\lambda(v)=\rs (1+v)^{-1}(1+s)^{1/2}$ appears in the above estimate due to the fact that the surface $\Gamma^R$ is 
    very close (within $\sim \log^{1/2}v$) to the null cone $u=0$. If we extended this estimate all the way to the null cone $\{u=0\}$, the corresponding 
energy flux would not contain the term $|\pa_u\psi|^2$ at all. }

On the other hand, the corresponding energy estimate in the left region $D^L=\{u\ge \ls s^{1/2}\}$, where the surface 
$\Gamma^L=\{u=\ls s^{1/2}\}$ is timelike, is 
\begin{align*}
\int_{D^L_{t_1}} &\left(|\pa_u \psi|^2 + |\pa_v\psi|^2+|\nas\psi|^2\right) + \int_{\Gamma^L_{t_0,t_1}}  \frac {|\pa_u\psi|^2}{(1+v)(1+s)^{1/2}}
\\ &=\int_{D^L_{t_0}} \left(|\pa_u \psi|^2 + |\pa_v\psi|^2+|\nas\psi|^2\right) +\int_{\Gamma^L_{t_0,t_1}} \left( |\pa_v\psi|^2+
\left(1-\frac 1{(1+v)(1+s)^{1/2}}\right)|\nas\psi|^2\right)
\end{align*}
{Unlike the previous case, the future energy at time $t_1$ requires not just the control of the energy at $t_0$ but also part of the 
energy flux along $\Gamma^L$. Note that the boundary condition \eqref{introbc2} would allow control 
of $|\pa_v\psi|^2$ along $\Gamma^L$ but {\it not} of the term involving $|\nas\psi|^2$. This indicates that even for 
 local existence theory, standard energy estimates with $X=\pa_t$ would not be sufficient.}

In general, as in the above example, the estimate \eqref{roughenident} only gives very weak control
over $n\psi_A^I$ along the spacelike and timelike sides
of the shocks, but strong control over $\ev\psi_A^I$ along the spacelike sides
of the shocks. On the other hand, in the regions
$D^L, D^C$, we need to treat $X^\ell_{g_A} (\ev \psi_A^I)^2$ as an error term
along the timelike sides of the shocks. Also, the term
$\slashed{B}_X$ need not have a sign and if $\lambda(v) X^\ev_{g_A} - X^n_{g_A} < 0$
we also need to be able to bound this term.

In reality, the argument we use to establish our energy estimates
is more delicate than the calculation
described above. In particular, the fact that our nonlinearities
do not satisfy the null condition means that we need to treat
the nonlinear error terms carefully; this is described in more
detail in Section \ref{modmultsec}. 
In Section \ref{energy0}, we collect various estimates involving
the energy momentum tensors which are used to justify
estimates of the form \eqref{introQpwSL} and \eqref{introQpwTL},
and we also derive expressions for the linear scalar currents
$K_{X, g_A}$. Finally, the basic energy estimates
(which are analogous to \eqref{roughenident})
we rely on are carried out in section \ref{ensec2}.

\subsubsection{The bootstrap argument, the decay estimates, and the choice
of multipliers}
\label{issuesection}

Our proof of global existence rests on a bootstrap argument, which requires
propagating a bound of the form $E_X(t) + B_X^+(t) \lesssim \epsilon^2$
for a small parameter $\epsilon$ from $t = t_0$ to $t = t_1$. All of the multipliers we will consider
will have the property that $S_X \geq 0$,
and in light of the identity \eqref{roughenident},
propagating this bound
requires getting control over \textbf{(a)} the nonlinear part
of the scalar current $K_{X, \text{nonlinear}}$,
\textbf{(b)} the boundary integral $B_X^-$, which does not
come with a favorable sign and must be treated as another 
error term, and \textbf{(c)} $\slashed{B}_X$, if
the multiplier $X$ is such that $\lambda(v)X^\ev_{g_A} - X^n_{g_A} < 0$. 
These issues are not independent and must be resolved in tandem
with one another.
We discuss these issues below.

\subsubsection*{Issue \textbf{(a)}: Controlling the nonlinear scalar current}
{As is the case with every supercritical nonlinear equation, the mechanism behind any global existence and stability statement is 
{\it decay}. For quasilinear wave equations \eqref{intromodel} on $\Bbb R^{3+1}$ this is a well known subtle issue in view of the 
slow decay rate of linear waves, i.e., solutions of the linear wave equation $\Box\phi=0$ on Minkowski space. Using the methods 
of energy estimates with appropriate commutators and multipliers which can then be adapted to the study of the nonlinear problems, 
such waves can be shown, \cite{Kla}, to satisfy 
the bounds}
\begin{equation}
\label{pointd}
|\pa^k_u\pa_v^j\nas^i\phi|\leq \frac {C_{ijk}}{(1+t)^{i+j} (1+|t-r|)^{1/2+k}}.
\end{equation}
In view of the fact that for the nonlinear wave equation \eqref{intromodel} with quadratic nonlinearities dependent on $\pa\phi$, the statement of global existence for {\it classical} small data solutions requires time integrability of the pointwise norm of the second derivatives of $\phi$,
that is 
\begin{equation}
\label{crit}
\int_{t_0}^\infty \|\pa^2\phi\|_{L^\infty} dt < \epsilon,
\end{equation}
{for a generic equation of the type \eqref{intromodel} such a statement will not hold true, since the linear waves already violate the required 
integrability criterion. It is precisely this phenomenon that led to the notion of the null condition, imposing structure on the form of the 
quadratic terms, which guarantees that for equations satisfying the null condition \eqref{crit} is not necessary and \eqref{pointd} is
sufficient, and also to the result that for (scalar) equations that do not satisfy the null condition small data solutions develop singularities 
in finite time.}

{For our solutions, which are no longer classical and contain shocks, the mechanism behind their global existence and stability statements
is still {\it decay}. As before, to control quadratic terms (which do not satisfy the null condition) requires the time integrability of the pointwise
norm of the second derivatives. The alert reader will notice that second derivatives for shock solutions contain $\delta$-functions 
of the shock surfaces and that even away from the shocks, such an estimate {\it does not} hold for either the model shock solution,
for which in the central region $\pa^2\Phi\sim 1/(t(\log t)^{1/2})$, or the linear waves (still).}

{The first issue is resolved by observing that 
the integrability statement should hold in {\it each} region $D^L, D^R, D^C$ separately. Of course, since the integrability/decay properties 
are derived from the energy estimates, both the latter and the derivation of the former from the latter now have to be properly localized.}

{The second merely suggests that we should rewrite 
our equation \eqref{intromodel} for  the perturbation 
$\phi_A=\Phi_A-\sigma$ as is done in \eqref{modelql} and hope that $\phi_A$ (and the source terms, omitted in \eqref{modelql}, 
coming from the profile $\sigma$) decays faster than the model shock profile. One of the challenges here is that the improvement of the rate of decay of $\phi_A$ over the one for the shock profile is truly marginal. In fact, pointwise, we can only establish that}
\begin{equation}
|\pa^2\phi_L|\lesssim \frac 1{t\log t (\log\log t)^{1/2}}
\label{marginal}
\end{equation}
which is still non-integrable. One of the novelties in this work is that {in the absence of an integrable pointwise estimate}, \eqref{crit} is established directly.

{Finally, to overcome slow decay of the linear waves we must take advantage of the geometries 
of the regions $D^L, D^R, D^C$. We begin with the region $D^R$ which is bounded from above by a spacelike (relative to the Minkowski metric) hypersurface which is close to the model shock $t=r-\eta (\log r)^{1/2}$. The solution of \eqref{modelql} in such a region is determined completely from its initial data. The region is located (logarithmically) below the light cone $t=r$. This indicates that the uniform bound on 
free waves $|\pa^2\phi_R|\lesssim 1/t$ is not sharp. In fact, \eqref{pointd} already suggests that using the fact that in such a region 
$|u|=|t-r|\ge (\log t)^{1/2}$ we could have the bound }
$$
|\pa^2\phi|\lesssim \frac 1{t(\log t)^{5/4}}
$$
{which is integrable. In this region, using the multiplier $X_R =(1 + |u|)^\mu \pa_t + r (\log r)^\nu\pa_v$,
with sufficiently large $\mu, \nu$ we can derive even stronger estimates. The analysis of both the linear and the full nonlinear 
problem is straightforward. This particular choice of the multiplier is motivated by the weighted
estimates from \cite{LindbladRodnianski2010} and the $r^p$ method from  \cite{DafermosRodnianski2010}. We note that 
the existence problem in regions which lie strictly below the light cone is connected with the so-called "boost problem" considered 
in \cite{ChristodoulouOMurchadha1981}, see also the recent work \cite{Wang}.}

{In the region $D^C$ the profile $\sigma$ is non-trivial and, as a result, the linearized problem \eqref{modelql} contains the wave 
equation with respect to the "Burgers metric" }
$$
m_B=m+\frac{t-r}{(t+r) \log (t+r)} dv^2.
$$
{Even though the deviation from the Minkowski metric is of order $1/(t(\log t)^{1/2})$ (since in this region $|t-r|\lesssim (\log (t+r))^{1/2}$) 
and decays, its influence on the behavior of linear waves is nontrivial and that behavior is very different from that of free waves on Minkowski space. 
The outgoing (radial) characteristics of the metric $m_B$ can be parametrized as }
$$
u=K\log v
$$
{(compare with the outgoing characteristics in Minkowski space given by $u=K$.) As with the 1-dimensional rarefaction waves, the characteristics are {\it spreading}. The quantitative effect of spreading on the behavior of linear waves on such background is additional 
decay. To capture it we use the multiplier $X_C=\log v \pa_u + v\pa_v$. In fact, both the multipliers and the commutators, employed 
in the energy estimates in this region, should be adapted to the metric $m_B$ and its properties. The result is that in this region }
$$
|\pa^2\phi|\lesssim \frac 1{t(\log t)^{3/2}}
$$
{The most difficult region is $D^L$. It is bounded on the right by a timelike (relative to the Minkowski metric) hypersurface close to the 
model shock $t=r+\xi (\log r)^{1/2}$. We are faced with the quasilinear wave equation \eqref{modelql} supplemented with the boundary condition \eqref{introbc2} along the timelike hypersurface. The behavior of free waves in Minkowski space given by \eqref{pointd} 
indicates that they decay faster in the interior of the light cone $t=r$. In particular, in the region $D_L$ \eqref{pointd} would suggest 
the bound }
\begin{equation}
  \label{fakeDLbd}
    |\pa^2\phi_L|\lesssim \frac 1{t(\log t)^{5/4}}.
\end{equation}
{We are however no longer dealing with the free waves on Minkowski space but rather with the solutions of the Minkowski wave equation
    on a bounded domain with a boundary condition, and as such there is
    no reason to expect \eqref{fakeDLbd} to hold. For an obvious example, consider such an equation in the cylindrical domain $r\le 1$ with Dirichlet or Neumann boundary conditions along $r=1$. Linear waves for such an equation do not decay at all! The behavior of linear 
(and nonlinear) waves in $D_L$ is entirely determined by the domain itself  and the boundary condition. To take advantage of both 
we employ two different multipliers. The first is a logarithmically amplified version of the scaling vector field }
\begin{equation}
    \label{introXLdef}
  X_L = uf(u) \pa_u + vf(v) \pa_v, \qquad f(z) = z \log z (\log \log z)^{\alpha},
  \qquad \text{ where } 1 < \alpha < 3/2
\end{equation}
  
and the second, a logarithmically enhanced version of the Morawetz multiplier
\begin{equation}
X_M =( \log(1+r)^{1/2} f(\log(1+r)) + 1)\pa_r.
\label{amplified}
\end{equation}
The latter is critical to establishing the integrability estimate \eqref{crit}.
{The logic behind our choice of multipliers will be explained momentarily.}

To summarize: to control the nonlinear scalar current, the crucial point is to show that a bound
for the energy $E_X \lesssim \epsilon^2$ implies the following time-integrated estimate  (recall that $\phi=r\psi$)
\begin{equation}
  \label{goalofestim}
\int_{t_0}^{t_1} \frac{1}{1+t} \|\pa^2 \psi_A\|_{L^\infty(D^A_t)}\, dt
  \lesssim \epsilon.
\end{equation}
Taking into account the definition of the energies $E_X$, by the Klainerman-Sobolev
inequality, simple properties of our commutator fields, and the fact that
by assumption $|u|\gtrsim s^{1/2} \sim (1+\log t)^{1/2}$ in the exterior regions
$D^L, D^R$, we find the pointwise bounds
\begin{equation}
  \label{goalofestimpw}
  |\pa^2 \psi_A| \lesssim \frac{1}{(1 + \log t)^{3/4}} \frac{1}{|X^n_{g^A}|^{1/2}}  \epsilon,
  \quad \text{ in } D^L, D^R, \quad\text{and}
  \qquad
  |\pa^2 \psi_C| \lesssim \frac{1}{1 + \log t} \frac{1}{|X^n_{g^C}|^{1/2}} \epsilon.
\end{equation}
We therefore want to pick $X$ so that the coefficients $X^n_{g^A}$ are large enough
that the right-hand sides here are time-integrable; in either case, we are ``just''
missing a few factors of $\log t$. However, we are not free to choose arbitrary
multiplier fields $X$; for one thing, we need to guarantee that $-K_{X, g_A} \leq 0$.
Moreover, the lack of null structure in this problem and the need to
be able to control various nonlinear error terms places a limit on the size
of the multipliers we consider.
We will see that this is relatively straightforward
to handle in the rightmost and central regions, but it presents serious
difficulties in the leftmost region, discussed in more detail in the
next section.

\subsubsection*{Issue \textbf{(b)}: Controlling the error terms along the timelike shocks}
We now consider issue \textbf{(b)}, which is only relevant in the central
region and the region to the left of the left shock.
To control the boundary term $B_X^-$, we use
the boundary conditions \eqref{introbc1}-\eqref{introbc2};
note that these
identities, at the linear level, relate $\ev^{g_A} \psi_A$ along 
the timelike side of $\Gamma_A$ to $\ev^{g_A^+} \psi_{A}^+$ with
$g_L^+ = g_C, g_C^+ = g_R$ denoting the linearized metric on the
spacelike side of $\Gamma_A$ and $\psi_A^+$ the corresponding potential.
Since our commutator fields are not tangent to the shocks, 
getting control of $\ev^{g_A} \psi_A^I$ requires first decomposing
the commutator fields into components which are tangent to the shock
and components which are transverse to the shock. This in turn requires
getting bounds for high-order derivatives of the function $\beta$, which
defines the shock, and for which we will need to differentiate the evolution
equation \eqref{introevolution}. This decomposition is performed in 
Section \ref{bcbootstrapsection} and control of high-order derivatives
of $\beta$ is established in Section \ref{bdsforbdfsec}.

Handling the above is somewhat involved, but the main difficulties
in handling the nonlinear boundary terms can
be understood already when $|I| = 0$.
If we directly use \eqref{introbc1}-\eqref{introbc2}
as appropriate, we find
\begin{equation}
  \label{}
  \int_{\Gamma^T_{t_0, t_1}} X^\ell_{g_A} (\ev^{g_A} \psi_A)^2
  \lesssim
  \int_{\Gamma^T_{t_0, t_1}} \frac{X^\ell_{g_A} }{(1+v)^2} 
  |\pa \psi_A|^4 
  + 
  \int_{\Gamma^T_{t_0, t_1}} X^\ell_g (\ev^{g_A^+} \psi_{A}^+)^2
  + \frac{X^\ell_{g_A}}{(1+v)^2} |\pa \psi_A^+|^4.
\end{equation}
The last two terms will not cause any serious difficulties: we will always
have better estimates available for $\psi_A^+$ than for $\psi_A$. 
We therefore focus on the first term.
In order to handle this term, it turns out that the main difficulty lies in establishing the estimate
\begin{equation}
  \label{}
  \int_{\Gamma^T_{t_0, t_1}} \frac{X^\ell_{g_A} }{(1+v)^2} 
   (n\psi_A)^4 \lesssim \epsilon^2 \int_{\Gamma^T_{t_0, t_1}}
   \lambda(v) X^n_{g_A} (n \psi_A)^2,
\end{equation}
where the quantity on the right-hand side is essentially the only
control we get over the solution along the timelike side of the shock from
\eqref{roughenident}. We remark that the fact that we need
to handle a term of this form ultimately derives from the fact that our nonlinearities 
do not satisfy the null condition.
Recalling $\lambda(v) = \ls (1+v)^{-1} (1+s)^{-1/2}$, this bound requires that
\begin{equation}
  \label{}
  \frac{X^\ell_{g_A}}{1+v} (n \psi_A)^2 \lesssim \frac{X^n_{g_A}}{(1+s)^{1/2}}.
\end{equation}
This places a limitation on the size of the multipliers we can afford to use, because
the Klainerman-Sobolev inequality and the bounds for our energies give us
\begin{equation}
  \label{}
  (n \psi_A)^2 \lesssim \frac{1}{1 + |u|} \frac{1}{X^n_{g^A}} E_X(t)
  \lesssim \frac{1}{(1 + s)^{1/2}}  \frac{1}{X^n_{g^A}} \epsilon
\end{equation}
along the shocks. Inserting this into the above we find that if we want
to close estimates for the nonlinear boundary terms, we must choose the 
multipliers so that the following condition holds true,
\begin{equation}
  \label{hardlimit}
  \frac{X^\ell_{g^A}}{1+v}\lesssim |X^n_{g^A}|^{3}.
\end{equation}
This same restriction also appears even in deriving the energy
identity \eqref{roughenident} and is needed to guarantee that
the statements \eqref{introQpwSL}-\eqref{introQpwTL} hold; see
in particular Section \ref{pertsec} and Lemmas 
\ref{spacelikeperturb}-\ref{timelikeperturb} where we prove bounds for the
nonlinear energy currents along the shocks.

We now come to the main difficulty:
we want to choose our multiplier field large enough that the pointwise
bounds \eqref{goalofestimpw} imply the time-integrated bound
\eqref{goalofestim}, but not so large that the condition \eqref{hardlimit} fails,
and at the same time we must ensure that the linear scalar current
satisfies $K_{X, g_A} \leq 0$. 
The above difficulty at the shocks is of course not present in $D^R$
since there is no timelike shock to contend with, and there, as mentioned
above, we can afford to 
use the multiplier $X_R =(1 + |u|)^\mu \pa_t + r (\log r)^\nu\pa_v$,
for large $\mu, \nu$.

In the central region, it turns out that the above issues are
not difficult to resolve and we can afford to use the multiplier
$X = \log v\pa_u + v\pa_v$, which is more than
large enough for our purposes. This strategy however raises issues in dealing
with point \textbf{(c)} above, see the next subsection.

In the leftmost region, on the other hand, this issue is nontrivial
to resolve,
and the above considerations lead to the fact that we cannot afford
to use a larger multiplier than \eqref{introXLdef}.

The estimate \eqref{roughenident} coming from $X_L$ is still useful,
since it gives control over the quantity $u f(u) (n \psi_A^I)^2
\sim (\log t)^{1/2} f((\log t)^{1/2}) (n\psi_A^I)^2$
near the shock (which is the most dangerous region from
the point of view of our estimates). This allows us to prove a Morawetz
inequality, obtained using the spacelike multiplier field \eqref{amplified}.
If we use this
in \eqref{intromult}, the resulting integrals over the time slices
are not positive-definite, but the field $X_M$ has been chosen so
that these integrals can be bounded in terms of the energies
$E_{X_L}$, which leads to a bound of the form
$S_{X_M}(t_1)\lesssim \epsilon^2$.
The advantage of using this multiplier is that the
scalar current $-K_{X_M, m}$ comes with a favorable sign
and it turns out that the above bound for $S_{X_M}$ directly implies
the bound \eqref{goalofestim}, which ultimately allows us to close our estimates.
This argument is carried out in section \ref{genmorsec}.

\subsubsection*{Issue \textbf{(c)}: Controlling the angular error term {$\slashed{B}_X$}}
\label{angularderivsshockpbm}
We now consider issue \textbf{(c)} above. Again, since there
is no timelike boundary to contend with in the rightmost region,
this only plays a role in the leftmost and central regions.

{In the leftmost region $D^L$, one interesting and important aspect of the choice of the multiplier $X_L$ \eqref{introXLdef} 
is that the angular 
flux $\slashed{B}_{X_L}$ along $\Gamma^L$ is actually positive. We have discussed earlier that if $X_L=\pa_t$ that term
is negative  and, unlike the term involving 
$|\pa_v\psi|^2$ in  ${B}_{X_L}^-$ , it could not be controlled from the boundary condition.}

 In the central region,
 however, the multiplier $X_C = {\left(s + \frac{u}{s}  \right) } \pa_u + v \pa_v$
which we use to establish our pointwise decay estimates
does not satisfy this condition.
As a result, {$\slashed{B}_X$} needs to be treated as an
error term and we need to find a way to control it.

For this, we couple the estimate obtained from $X_C$
with an estimate obtained by using the much weaker
multiplier $X_T = v\evmB + \left(  \frac{u}{s} + \frac{\ap}{4s^{1/2}}  \right)n$
(the ``top-order multiplier''). 
The resulting
estimate is too weak to give useful decay estimates, but this
multiplier has been chosen so that $\slashed{B}_{X_T} \geq 0$. 
To get the needed decay estimate, the idea is to prove the multiplier
estimate with $X_C$, but after commuting fewer vector fields
than we commute with in the estimate for $X_T$. It turns out that
one can control the resulting angular boundary term $\slashed{B}_{X_C}$
by integrating along the shock, after bounding $|\nas \psi_A| \lesssim
(1+v)^{-1} |\Omega \psi_A|$. This relies on the Hardy estimate
from Lemma \ref{controlangularhardyright} and is carried out
in Lemma \ref{bdsforpsiCalongshock}; see in particular the bound
\eqref{rightangularbound}.

\subsection{Modulated profiles and location of the shocks}
  Recall that the shocks 
  $$\Gamma^A = \{(t, x) \in \mathbb{R}^{1+3} : 
    t-r = \beta_s^A(\omega)\}
    $$ with  $A=L,R$ and $\omega = x/|x| \in \mathbb{S}^2$
 are parametrized by the  functions $\beta_s^A:\mathbb{S}^2\to\mathbb{R}$ with $s=\log (t+r)$ which satisfy the following evolution equation
\begin{equation}
\label{betaa}
 \dbeta_s^A(\omega) - \frac{1}{2s} \beta_s^A(\omega) = \left(\frac{1}{2} (\pa_u \psi_C + \pa_u \psi_A)
 + N(\pa \psi_A, \pa \psi_C) \right)\bigg|_{u = \beta_s^A(\omega)}.
\end{equation}
These functions appear as modulation parameters of our shock profile 
$$
  \sigma(t, x) = 
      \begin{cases}
          \frac{u^2}{2rs}, \qquad &\text{ in } \beta^L\le u\le \beta^R\\
          0,\qquad &\text{ otherwise}.
      \end{cases}
      $$
      When $\beta^A= C_A s^{1/2}$ with constant $C_A$, $\sigma$ is a 2-shock solution of the Burgers equation 
      $$
      \pa_s \sigma+\frac 12 \pa_u (\sigma^2)=0.
      $$
Modulating the profile $\sigma$ by making $\beta^A/s^{1/2}$ to depend nontrivially on $s$ and $\omega$ allows to adapt the profile 
to fit the equation \eqref{intromodel} and, in particular, account for the correct location of the shocks. The Rankine-Hugoniot conditions 
then lead to the evolution equations for $\beta^A$ and connect $\beta^A$ to the solutions of \eqref{modelql} or \eqref{intrompsi}.
Due to the dependence of  $\beta$ on $\omega$ and from the point of view of \eqref{betaa}, the space of modulation parameters is
infinite dimensional. 

As discussed in the previous section dealing with the higher derivatives of $\psi$ requires decomposing them along the shocks
into its transversal and tangential parts. This is done in order to take advantage of the (higher order) boundary conditions. Let $Z$ 
be an arbitrary vector field. Along the shock, it can be decomposed in the form 
$$
Z=Z_T+ Z(\beta-u) \pa_u 
$$
where $Z_T$ is 
tangent to the shock. Applying this repeatedly we see that the decomposition for $\pa^n(\pa \psi)$
will involve $(n+1)$ derivatives (with respect to $s$ or $\omega$) of $\beta$. Going back to \eqref{betaa} then shows that to 
control those would require either the  control of $(n+1)$ derivatives of $\psi$, if one the derivatives is the $s$-derivative, or 
even $(n+2)$ derivatives of $\psi$ otherwise. Even in the best case scenario, $\beta$ and $\psi$ couple to each other linearly 
at the highest order. This was already a major issue in the local existence theory of Majda \cite{Majda83,Majda83b,MajdaThomann87}.
In his work, the general approach is different as the shock is straightened at the expense of making the linearized equations for 
$\psi$ more complicated. For a global problem like the one considered here shock straightening can be costly and is avoided. To avoid the loss 
of derivatives which can arise when one commutes \eqref{betaa}
{and the boundary conditions for $\psi$}
with $(n+1)$ $\omega$-derivatives, we observe that $\beta$ also 
satisfies another equation
$$
\nas\beta=-\frac su [\nas\psi]+N(\pa\psi)
$$
see Appendix \ref{RHsec} and Remark \ref{eliminaterem}. 

Nonetheless, the conclusion of  this discussion  is that, unlike other problems where the modulation space is finite dimensional and 
the modulation parameters couple quadratically to the unknown fields, in this problem the coupling is linear and at the same order 
of differentiability. 

The linear coupling also has a major effect on the asymptotic behavior of the shocks. The logic of the proof requires that the shocks 
are close to the surfaces $u=\pm s^{1/2}$ (taking $\rs = \ls = 1$).
Quantitatively, at the very least, we need that the functions $\beta^A/s^{1/2}$ are uniformly 
bounded. From \eqref{betaa},
$$
\left|\frac {\beta^A_{s_1}}{s_1^{1/2}}\right|\lesssim
\left|\frac {\beta^A_{s_0}}{s_0^{1/2}} \right|
+ \int_{s_0}^{s_1} s^{-1/2} |\pa_u\psi| ds
$$ 
The energy estimate \eqref{roughenident} contains the boundary term 
$$
\int_{\Gamma_{t_0,t_1}} \frac 1{(1+v)(1+s)^{1/2}} X^n_{g_A} (n\psi_A)^2
$$
which, in view of our choices of multipliers $X$, discussed above, can be replaced by 
$$
\int_{\Gamma_{t_0,t_1}} \frac {\log s(\log\log s)^\alpha} {(1+v)} (\pa_u \psi_A)^2
$$
The same estimate also holds for the angular derivatives of $\psi_A$. Taking that into account (and that $ds=1/v dv$), 
$$
\int_{s_0}^{s_1} s^{-1/2} |\pa_u\psi| ds\lesssim \left(\int_{s_0}^{s_1} s^{-1} (\log s)^{-1}(\log\log s)^{-\alpha}\right)^{\frac 12} 
\left(\int_{s_0}^{s_1} \log s (\log\log s)^\alpha  |\pa_u\psi|^2 ds\right)^{\frac 12} \lesssim E_X^{\frac 12}
$$
since $\alpha>3/2$. This tells us that not only the functions $\beta/s^{1/2}$ are bounded but that they also have asymptotic 
limits as $s\to \infty$.

\subsection{The full compressible Euler equations and the restricted shock
front problem}
\label{introenth}
We now discuss how the problem \eqref{introHeqn}
with jump conditions \eqref{introRH1}-\eqref{introRH2}
is related to the original problem \eqref{intromass}-\eqref{introeul}.
For \emph{smooth} solutions, it is well-known that if the initial
data for \eqref{intromass}-\eqref{introeul} is irrotational,
then the solution is irrotational at later times as well. However,
this is \emph{not} true for solutions with shocks.
Indeed, consider \eqref{intromass}-\eqref{introeul} 
satisfying the classical Rankine-Hugoniot conditions
\begin{align}
  \label{clRH1}
  \zeta_t[\rho] + \zeta_i[\rho v^i] &=0,\\
  \label{clRH2}
  \zeta_t[\rho v_i] + \zeta_i[\rho v^iv_j] + \zeta_j[p]&=0, \qquad j = 1,2,3,
\end{align}
across a shock, where here $\zeta = \zeta_tdt + \zeta_i dx^i$
is a one-form as in \eqref{introRH1}, conormal to the shock.
These guarantee that \eqref{intromass}-\eqref{introeul} hold
in the weak sense across the shock. One can show that if 
\eqref{clRH1}-\eqref{clRH2} hold then in general $[\omega]\not=0$;
in particular, one cannot expect to have a solution to \eqref{intromass}-\eqref{introeul}
satisfying \eqref{clRH1}-\eqref{clRH2} which is irrotational on \emph{both} sides
of the shock.

To see what is behind the above, recall that \eqref{intromass}-\eqref{introeul} 
describe an isentropic fluid. If we want to take entropy into account, 
these
equations need to be supplemented with the conservation law for energy
\begin{equation}
  \label{encons}
  \pa_t (\rho E) + \pa_i (\rho v^i E + p v^i) = 0,
\end{equation}
where $E = \frac{1}{2} |v|^2 + e(\rho, p)$, where $e$ is the specific internal energy.
Here, $p$ is no longer determined by the density $\rho$ alone but instead $p = P(\rho, S)$
where $S$ is the specific entropy, related to the variables $e$, $p$, $\rho$
and the temperature $T$ by the second law of thermodynamics $de = TdS - p d(\rho^{-1})$. 
If $dP/dS\not=0$, irrotationality is not
preserved even for smooth solutions.
On the other hand, for classical solutions, the 
equation \eqref{encons} together with the other equations is equivalent to
\eqref{intromass}-\eqref{introeul} supplemented with
\begin{equation}
  \label{}
    \pa_t S + v^i\pa_i S = 0.
\end{equation}
As a result, if $S$ is initially constant and the solution remains smooth,
the motion is determined entirely by \eqref{intromass}-\eqref{introeul},
and the equation \eqref{introHeqn} completely determines the motion
if we additionally assume that the vorticity is initially zero.

However, if the solution develops a shock we need to supplement
the Rankine-Hugoniot conditions \eqref{clRH1}-\eqref{clRH2} 
with the jump condition
\begin{equation}
  \label{}
  \zeta_t[\rho E] + \zeta_i [\rho v^i E + p v^i] = 0.
\end{equation}
In order for the equations \eqref{intromass}-\eqref{introeul},
\eqref{encons} and the jump conditions \eqref{clRH1}-\eqref{clRH2}
to be deterministic, it turns out that one needs the entropy to 
have a nonzero jump aross the shock. In particular, the solution cannot
remain isentropic on both sides of the shock and as a result it 
cannot remain irrotational on both sides either, in light of the
fact that 
$[\omega] = O([S])$, see equation (1.275) in \cite{Christodoulou2019}.

The system \eqref{introHeqn} and jump conditions \eqref{introRH1}-\eqref{introRH2}
can then be understood as a version of the above non-isentropic problem where
we ignore variations due to entropy, even after the shock has formed.
This is precisely the setting of Christodoulou's ``restricted'' shock development program
\cite{Christodoulou2019}. The main advantage of working with restricted shocks, beyond
the conceptual simplifications of working with
\eqref{introHeqn} instead of \eqref{intromass}-\eqref{introeul}, \eqref{encons}, is that
one can can ignore the vorticity, which there is no known way to control at large times.

This problem is of interest in its own right from the point
of view of quasilinear wave equations, and
as explained in \cite{Christodoulou2019} and \cite{MajdaThomann87},
it is still physically relevant despite the above. First,
a calculation (see (1.260) in \cite{Christodoulou2019}) shows that the jump
in entropy is small if the jump in pressure is small, 
$[S] = O([p])^3$ and as a result $[\omega] = O([p])^3$ is also small, and so
solutions to \eqref{introHeqn} with jump conditions 
\eqref{introRH1}-\eqref{introRH2} are approximate solutions to the full problem
if $[p]$ is small (which is the case in our setting). {In fact, $[p]\sim 1/(t(\log t)^{1/2})$ and, as a 
 result, $[S],[\omega]= O\left(1/(t^3(\log t)^{3/2})\right)$ -- negligible from the point of
view of decay.}
 
We also note that \eqref{introRH1}-\eqref{introRH2} imply that three
of the conditions \eqref{clRH1}-\eqref{clRH2} hold. Indeed,
the condition \eqref{introRH1} is nothing but \eqref{clRH1},
and if we decompose \eqref{clRH2} into
its components parallel and transverse to $\zeta$, we find the relations
\begin{align}
  \label{projRH1}
  [\overline{v}_i] &= 0,\\
  \label{projRH2}
  \zeta_t[\rho v_\zeta] + [\rho v_\zeta^2] + [p] &=0,
\end{align}
with $v_\zeta = \zeta_i v^i$ and with $\overline{v}$ the component
of $v$ tangent to the shock. If $v = \nabla \Phi$ where $\Phi$
satisfies the jump condition \eqref{introRH2} then \eqref{projRH1}
holds, but there is no guarantee that \eqref{projRH2} holds.
Our jump conditions \eqref{introRH1}-\eqref{introRH2} then ensure
that the jump condition \eqref{clRH1} associated to the continuity
equation holds, and the tangential components of the conditions
\eqref{clRH2} associated to the momentum equations hold, but we do not
enforce the normal component of \eqref{clRH2}. 

\subsection{Further background on the problem and related results}
The mathematical theory of the compressible Euler equations has a long history,  with an enormous amount of literature devoted to it.
It would be impossible for us to survey it, but see
for example \cite{Christodoulou07, Dafermos2010, Lax1972} and the references therein. We will concentrate on the results 
more related to the subject of this paper which can be put into two categories: asymptotic behavior of  solutions for the equations in 
one space dimension and 
more recent work on the problems of breakdown and shock formation and evolution in higher dimensions.

Breakdown for smooth solutions of \eqref{intromass}-\eqref{introeul}
in one space dimension dates back to work of Challis \cite{Challis1848} and Riemann \cite{Riemann}. 
The long-time behavior of solutions of Burgers' equation, which serves
as an important model for the Euler equations, was studied
by Hopf \cite{Hopf1950} who was able to extract the asymptotic shape of
solutions after shock formation. This was generalized to other one-dimensional
scalar conservation laws by Lax \cite{Lax1957,Lax1972}. For a single convex scalar conservation law the 
following sharp result,  which is particularly instructive to compare to the main result of this paper, is proven in \cite{DiPerna}.
\begin{theorem}{\cite{DiPerna}}
Let $v$ be a BV solution with initial data of compact support of the equation 
$$
\pa_t v+\pa_x f(v)=0
$$
with $f''>0$ and $f'(0)=0$, $f''(0)=1$,
and let $N(t,\rs,\ls)$ denote the $N$-wave (cf. \eqref{prof})
$$
N(t,p,q)= \begin{cases} \frac{x}{t},  \qquad
\text{ when } -\rs\,t^{1/2}\le x\leq \ls\, t^{1/2},\\
0, \qquad \text{ otherwise }, \end{cases}
$$
Then there exist constants $\rs,\ls \ge  0$ depending on the initial data such that
$$
\|v(t,\cdot)-N(t,\rs,\ls)\|_{L^1}\lesssim t^{-1/2}
$$
for all sufficiently large $t$. 
\end{theorem}

This result was generalized (for small initial data) to systems of 2 conservation laws in \cite{GLax} and, finally, to systems of $n$ conservation 
laws in \cite{Liu}. These results should be compared with our Theorem \ref{roughthm} which gives the asymptotic behavior of 
2-shock solutions of the equation \eqref{intromodel} corresponding to the irrotational compressible Euler equations on $\Bbb R^3$
or, generally, the wave equation \eqref{intromodel} without the null condition, and the convergence
statement in $L^\infty$.

The first proof of singularity formation for the compressible
Euler equations
in higher dimensions was given by Sideris in \cite{Sideris85}.
There,
the proof is by a virial argument and does not give any information
about the nature of the singularity. Alinhac's work \cite{Al1,Al2} on the 2-dimensional 
version of the equation \eqref{intromodel} gave the first constructive proof of the "first time" singularity formation.

In the monumental work \cite{Christodoulou07} (see also \cite{ChM}),
Christodoulou was able to describe the maximal classical development for the solutions of the compressible Euler equations 
contained in the domain of dependence of the exterior of a sphere
of arbitrary, small, regular initial data which is constant
outside of a larger sphere, and gave a detailed description of the singular boundary.
These results were extended to different regimes,
 of initial data forming small open sets of specific profiles for the problem on $\Bbb R\times \Bbb T^2$  in \cite{ASpeck} and allowing for nontrivial vorticity at the
 singular boundary, {and where the authors were able to give a more complete description of the portion of the maximal development
 near the {\it crease} -- first singularity, {even in the absence of strict convexity.}} 
 For the corresponding results in 2d see also \cite{LukSpeck18} and \cite{ShVi23}, {
     where in the latter reference, the authors
{ gave a detailed description of a maximal development including the portion of a Cauchy horizon} for the problem on $\Bbb T^2$ for a 
 specific small open set of initial data.}
The "first time" singularity formation for the full problem, again for a specific small open set of initial data, 
were given in  \cite{LukSpeck21}, \cite{BuShVi23}.  Shock formation for a class of quasilinear wave equations in 2d was investigated 
in \cite{Holz} and for a class of large data in 3d in \cite{Miao}.

A different mechanism
for blowup for the compressible Euler equations with smooth data in three dimensions
with a very different character was recently discovered 
in \cite{MRRS22}. The singularities constructed
there arise from large initial data, they are not shocks and instead the density blows up in finite time. 

The problem of local-in-time existence for the multi-dimensional shock front
problem was solved by Majda in the works \cite{Majda83,Majda83b}. There, Majda
considered initial data for a large class of hyperbolic systems, including the compressible
Euler equations, which already has a shock in it and constructed a local-in-time solution
to the shock front problem. In \cite{MajdaThomann87},
Majda and Thomann gave a different proof of local existence for the restricted
shock front problem described in Definition \ref{sfpdef}. 

In recent years, starting with the breakthrough results 
\cite{Christodoulou2019}
of Christodoulou, there has also been a great interest in 
the \emph{shock development problem}, wherein one starts with the singular solution constructed in the process of the 
of solving the shock development problem and replaces/extends it with the {\it weak} solution containing a shock. 
A recent result for the 2d problem with azimuthal symmetry is in
\cite{BuDrShVi22}. For earlier results in spherical symmetry see \cite{ChLi},\cite{Yin04}. 
\subsection{Further developments}
As we mentioned earlier, our work addresses only part of the picture described by Landau (in spherical symmetry). 
In particular, the question of whether solutions arising from small smooth initial data for large times approach a 2-shock 
profile remains open (even in spherical symmetry). Already, constructing an example of the above scenario would be 
very interesting.

The next obvious step is to address the full problem \eqref{intromass}-\eqref{introeul}, \eqref{encons}, 
without the irrotational condition and allow for the production of vorticity and entropy across the shocks.  Such a problem 
in the whole space is completely intractable for the same reasons as the corresponding 3d problem of shock formation.
Vorticity and entropy waves propagate with the speed of the fluid  and do not decay. As a result, assuming that 
initially vorticity and entropy are of compact support, the support will remain compact and, eventually, will be contained 
in the interior of the left region 
$D^L$, where the vorticity waves could undergo vorticity stretching and form singularities of a very different kind. Nonetheless,
it would still be possible and desirable to consider the problem for the points which lie in the domain of dependence of 
the exterior of a sphere. Such a domain would necessarily contain the right shock $\Gamma^R$  in our  picture and the 
vorticity would decay there since it would be eventually transported away from this domain. As was discussed in Section 
\ref{introenth}, the vorticity and entropy produced by the shock are proportionate to the third power of the strength of the 
shock which is $\sim 1/(t(\log t)^{1/2})$. As a result, their influence is much weaker than that of the sound waves and should 
be easily controlled.

Landau's paper also discusses the 2-dimensional case. There, two shocks are supposed to be 
separated by distance of $\sim t^{1/4}$ and 
the strength of shocks should decay with the rate $\sim t^{-3/4}$. This rate is even further away from integrable than in the 3d case. 
Additionally, 2d free waves decay considerably slower than in 3d. Nonetheless, the shocks are further away from 
the null cone and both the geometry and preliminary analysis of the problem indicate that the 2d statements analogous to 
the ones proven in this paper likely hold true.

In this paper we considered the global problem involving spherical (but not spherically symmetric) shocks which are expected 
to emerge from compactly (or rapidly decaying) initial data. Of separate interest would be to consider other geometries and, in particular,
investigate the problem of stability of planar (non-symmetric) shocks.

\subsection{Acknowledgements}
IR acknowledges support through NSF grants DMS-2005464 and a Simons Investigator Award.
DG acknowledges support from the Simons Collaboration on Hidden Symmetries and a startup
grant from Brooklyn College.

\section{Notation and definitions}

Let $t, x^1, x^2, x^3$ denote the usual rectangular coordinates. We will work
in terms of the Minkowskian null coordinates
\begin{equation}
 u = t - |x|, \qquad v = t + |x|, \qquad \theta^1(x),\qquad \theta^2(x),
 \label{mainminknull}
\end{equation}
where $\theta^1, \theta^2$ are an arbitrary local coordinate system on the unit
sphere $\S^2$. We will use $s$ to denote
$$
s=\log v.
$$
We will write
\begin{equation}
 \omega_i = \delta_{ij} \omega^j = \frac{x_i}{|x|},
 \qquad
 \slashed{\Pi}_i^j = \delta_i^j - \omega_i \omega^j,
 \qquad
 \nas_i = \slashed{\Pi}_i^j \nabla_j,\quad i, j = 1,2,3,
 \label{angularproj}
\end{equation}
where $\nabla$ denotes the covariant derivative defined with respect to the
Minkowski metric,
\begin{equation}
 m = -dt^2 + dx^2 = -dudv + \frac{1}{4}(v-u)^2 dS(\omega),
 \label{minkowskidef}
\end{equation}
where $dS(\omega)$ denotes the metric on the unit sphere $\S^2$.

In the region between the shocks, the perturbation will satisfy
 a quasilinear wave equation which is a
perturbation of a wave equation with respect to the ``Burgers' '' metric $\mB$,
\begin{equation}
 \mB = -dt^2 + dx^2 + \frac{u}{vs} dv^2
 =
  - du dv+ \frac{u}{vs} dv^2 + \frac{1}{4} (v-u)^2 dS(\omega)
 \label{mBdef}
\end{equation}

We also record that the inverse metrics are given by
\begin{equation}
m^{-1}(\xi, \xi) = -4\xi_u\xi_v + 4(v-u)^{-2}|\slashed{\xi}|^2
 \label{}
\end{equation}
and
\begin{equation}
 \mB^{-1}(\xi, \xi) = -4\xi_u \xi_v -\tfrac{{4}u}{vs} \xi_u^2
 + 4(v-u)^{-2} |\slashed{\xi}|^2.
 \label{}
\end{equation}
The metrics $g = m, \mB$ admit two null vectors $(n, \ev^g)$ where
\begin{equation}
 \ev^m = \pa_v,\qquad
 \evmB = \pa_v + \frac{u}{vs} \pa_u,\qquad {n=\pa_u}
 \label{nullfields}
\end{equation}
which satisfy, in either case,
\begin{equation}
 g(n, \ev^g) = -\frac{1}{2}.
 \label{normalization}
\end{equation}

For a vector field $X$ we write
\begin{equation}
 X = \Xo_g \ev^g + \Xt_g \eu^g + \slashed{X}
 \label{}
\end{equation}
where
\begin{equation}
 \Xo_m = X^v, \quad
 \Xt_m = X^u,
 \qquad\qquad
 \Xo_{\mB} = X^v,
 \quad
 \Xt_{\mB} = X^u - \frac{u}{vs} X^v,
 \label{Xnullexplicit}
\end{equation}
and where the angular part
$\slashed{X} = \Pi\cdot X$ with $\Pi$ as in \eqref{angularproj}.
Note that from \eqref{normalization},
\begin{equation}
 g(X, Y) =-\frac{1}{2}(X^n_g Y^{\ev}_g + X^\ev_g Y^n_g)
 + g(\slashed{X}, \slashed{Y}).
 \label{metricnull}
\end{equation}

\subsection{The multiplier and commutator fields}
\label{fields}

For the convenience of the reader, we record here the
multiplier fields
we use in each of the three regions $D^R, D^C, D^L$.
\begin{center}
\begin{tabular}{ |c|c|c| }
  \hline
 Region & Multipliers  & Commutators\\[1ex]
 \hline
 $D^R$ $(u \lesssim -s^{1/2})$ & $X_R =w(u)(\pa_u + \pa_v) + r(\log r)^\nu \pa_v$ & $\mathcal{Z} = \{\pa_\mu,\Omega_{\mu\nu}= x_\mu\pa_\nu - x_\nu\pa_\mu, S =x^\alpha\pa_\alpha\}$
 \\[1ex]
 \hline
 $D^C$ $(|u| \lesssim
 s^{1/2})$ & $X_C ={ \left(s + \frac{u}{s}  \right)} \pa_u +v\pa_v$,
 $X_T = {\big(\tfrac{u}{s} + \tfrac{\ap}{{4}s^{1/2}}\big)\pa_u + v\pa_v}$
  & $\mZB = \{\Omega_{ij}, \sBo = s\pa_u, \sBt = v\pa_v\}$
 \\[1ex]
 \hline
 $D^L$ $(u \gtrsim s^{1/2})$ & $X_L = uf(u) \pa_u + vf(v) \pa_v$, $X_M = (g(r)+1) (\pa_v - \pa_u)$
  & $\mathcal{Z}$
  \\[1ex]
 \hline
\end{tabular}
\end{center}
In the above, we use the convention that $x^i = x_i$ for $i = 1,2,3$
and $x^0 = -x_0 = t$. The functions $f, g$ are 
\begin{equation}
 f(z) = \log z(\log \log z)^{\alpha} ,\qquad g(z) = (\log(1+z))^{1/2}f(\log(1+z))
 \label{fgdefintro}
\end{equation}
and the parameters $\mu, \alpha$ will be chosen subject to
\eqref{parameters}.

The roles of these multipliers are explained in section
\ref{issuesection} and the energies associated to these multipliers
along with the corresponding energy estimates can be found
in section \ref{ensec2}.

\subsection{Basic assumptions about the positions of the shocks}
\label{shocksetupsec}

We let $\Gamma^R, \Gamma^L$ denote the right and left shock respectively,
and write $\Gamma^A_{t'} = \Gamma^A \cap \{t = t'\}$ for $A = L, R$.
We will parametrize the shocks by functions $\beta_s^L, \beta_s^R:\S^2 \to \R$
where the parameter $s \in [s_0, s_1]$ for some $s_0,s_1 > 0$, so that the shocks
are given by
\begin{equation}
 \Gamma^A = \{(t, x) | u = \beta^A_{\log (t+|x|)}(x/|x|)\}.
 \label{betadef}
\end{equation}

We will prove energy estimates assuming that the left shock is sufficiently
close to the surface $ u = -\ls (\log v)^{1/2}$ and the right shock
is sufficiently close to the surface $u = \rs (\log v)^{1/2}$ for constants
$\ls, \rs > 0$.
In particular we will assume that the initial positions of the shocks
are parametrized by
\begin{equation}
 t_0 - r = \beta^L_0(\omega),
 \qquad
 t_0 - r = \beta^R_0(\omega),
 \label{}
\end{equation}
for functions $\beta^L_0, \beta^R_0$ which are sufficiently close to
the positions of the model shocks, 
\begin{equation}
    \label{eq:betainitial}
    \left|\nabla_\omega^j\left({\ls-}\frac{\beta^L_{0}(\omega)}{(1+ s^L_0)^{1/2}}\right)\right|
 +
 \left|\nabla_\omega^j\left( {\rs +}\frac{\beta^R_{0}(\omega)}{(1+ s^R_0)^{1/2}}\right)\right| \leq \ve,
\end{equation}
for $j = 0,1$, and
where $s^L_0, s^R_0$ denote the values of $s = \log (t + |x|)$ along the
shocks $\Gamma^L, \Gamma^R$ at $t = t_0$.

We will assume that that for $\ve_1, \ve_2$ sufficiently small,
we have the bounds
\begin{equation}
    | \beta^L_s(\omega) - \beta_{0}^L(\omega)s^{1/2}|
  + (1+s) \left|\dbeta^L_s - \frac{1}{2s} \beta^L_s\right|
  \leq \ve_1 (1+s)^{1/2},
  \qquad
  |\nabla_\omega \beta_s^L(\omega)| \leq \ve_2 (1+s)^{1/2},
 \label{betaLassump}
\end{equation}
with the same assumptions at the right shock,
\begin{equation}
    | \beta^R_s(\omega) - \beta_{0}^R(\omega) s^{1/2}|
  + (1+s) \left|\dbeta^R_s - \frac{1}{2s} \beta^R_s\right|
  \leq \ve_1 (1+s)^{1/2},
  \qquad
  |\nabla_\omega \beta_s^R(\omega)| \leq \ve_2 (1+s)^{1/2}.
 \label{betaRassump}
\end{equation}
To close the estimates along the timelike sides of the shocks,
we will also need to assume control of higher-order norms
of the functions $\beta^A_s$, see Section \ref{geomnormsection}.

%

It is convenient to introduce the following extension of $\beta^A_s$
to a neighborhood of the shocks,
\begin{equation}
 B^A(t, x) = \beta^A_{\log (t+|x|)}( x/|x|),
 \label{BAdef}
\end{equation}
which satisfies
\begin{equation}
 \pa_u B^A(t,x) = 0,
 \qquad
 \pa_v B^A(t,x) = \frac{1}{v} \dbeta^A_{\log (t+|x|)}(x/|x|),
 \qquad
 \nas B^A(t,x) = \frac{1}{r} \left(\nabla_\omega \beta^A_{\log (t+|x|)}\right)(x/|x|),
 \label{derivsofB}
\end{equation}
where in the last expression we have identified the abstract sphere $\S^2$
with the subset $\{|x| = 1\} \subset \R^4$.
Then the tangent space to $\Gamma^A$ at each point lies in the null space of the
one-form $\zeta^A$ given by
\begin{equation}
 \zeta^A = -\frac{1}{2}d(u-B^A) = - \frac{1}{2} du +\frac{1}{2} \pa_v B^A dv
 +\frac{1}{2} \nas B^A \cdot \slashed{dx}
 = -\frac{1}{2} du + \frac{1}{2v} \dbeta_s^A dv +
  \frac{1}{2r} \nabla_\omega \beta_s^A \cdot \slashed{dx},
 \label{zetadef}
\end{equation}
where $\slashed{dx}$ denotes the projection of $dx$ to the cotangent
space to the unit spheres and where
$s = \log (t+|x|)$.

We will work in terms of a vector field
$N_m^{\Gamma^A}$ which is normal to $\Gamma^A$ with respect to the
Minkowski metric, given by raising the index of \eqref{zetadef} with
the Minkowski metric,
\begin{equation}
 N^{\Gamma^A}_m = \pa_v - \pa_v B^A \pa_u {+\frac 12}\nas B^A\cdot \nas,
 \label{Nmdef}
\end{equation}
and similarly we will work in terms of a vector field
$N^{\Gamma^A}_{\mB}$ which is normal to $\Gamma^A$ with
respect to the Burgers' metric
$\mB$,
\begin{equation}
 N^{\Gamma^A}_{\mB} = \pa_v + \left( 2 \tfrac{u}{vs} - \pa_v B^A \right) \pa_u
 {+\frac 12}\nas B^A \cdot \nas.
 \label{NmBdef}
\end{equation}
We will often just write $N^A_g$ in place of $N^{\Gamma^A}_g$. We have chosen
$N^A_g$ so that when $\Gamma^A$ is spacelike with respect to $g$, $N^A_g$ is the future-directed
normal to $\Gamma^A$, and when $\Gamma^A$ is timelike with respect to $g$, $N^A_g$
is inward-pointing.
It will be convenient later on to write these formulas in terms of
the null vectors $(n, \ev^g)$ defined in \eqref{nullfields}. Writing $N_g = N_g^nn + N_g^{\ev} \ev^g
+\slashed{N}_g$ where $\slashed{N}_g$ denotes the angular part of $N_g$,
in either case we have
\begin{equation}
 N^\ev_g = \frac{g(N_g, n)}{g(\ev^g, n)} = -2\zeta(n)= 1,
 \qquad
 N^n_g = \frac{g(N_g, \ev^g)}{g(\ev^g, n)} = -2\zeta(\ev^g)
 = -\ev^g (B-u).
 \label{normalnullframe}
\end{equation}
We also record the following identities
\begin{equation}
	g(X, N_g) = -\frac{1}{2} X(u-B)
	= -\frac{1}{2}(X^n_g N_g^{\ev} + X^{\ev}_gN_g^n)
    -\frac{1}{2} X\cdot \nas B
	= -\frac{1}{2} (X^{n}_g - X^\ell_g \ev^g(B-u) ) - \frac{1}{2} X\cdot \nas B
 \label{innerproduct}
\end{equation}
where we used \eqref{metricnull} and \eqref{normalnullframe}.

A vector field $X$ is called timelike with respect to the
metric $g$ when $g(X,X) < 0$ and spacelike if $g(X, X) > 0$.
We say a surface $\Sigma$ is timelike (respectively spacelike) with respect to $g$
if the normal field $N_g^\Sigma$ to $\Sigma$, associated to the metric $g$ is
spacelike (respectively timelike). If we use \eqref{innerproduct} with $X = N_g$,
we find
\begin{equation}
 g(N_g^{\Gamma^A}, N_g^{\Gamma^A}) = \ev^g (B^A-u)
 + \frac 14 |\nas B^A|^2.
 \label{glength}
\end{equation}
When $g = m$ so $\evg = \pa_v$, along $\Gamma^A$ where $u = B^A$, the above reads
    \begin{equation}
      \label{}
        \ev^g (B^A - u) + \frac{1}{4} |\nas B^A|^2 
        = \pa_v B^A + \frac{1}{4} |\nas B^A|^2
        = \frac{u}{2vs} + \frac{1}{v} \left(\pa_s B^A - \frac{B^A}{2s}  \right)  + \frac{1}{4} |\nas B^A|^2,
    \end{equation}
    where the last two terms are negligible, 
    by \eqref{betaLassump}-\eqref{betaRassump} (which are written at the level of $\beta^A = B^A|_{\Gamma^A}$).
    When instead $g = \mB$, so $\evg = \pa_v + \frac{u}{vs} \pa_u$, we have
    \begin{equation}
      \label{}
        \evg (B^A -u) + \frac{1}{4} |\nas B^A|^2
        = \pa_v B^A - \frac{u}{vs} + \frac{1}{4} |\nas B^A|^2
        = -\frac{u}{2vs} + \frac{1}{v} \left(\pa_s B^A - \frac{B^A}{2s}  \right)  + \frac{1}{4} |\nas B^A|^2.
    \end{equation}
    
    Recalling that at $\Gamma^R$, $u \sim -s^{1/2} \rs$ and at $\Gamma^L$, $u \sim s^{1/2} \ls$,
    where $\rs, \ls$ are the positive constants from \eqref{eq:betainitial},
    if the assumptions \eqref{betaLassump}-\eqref{betaRassump} about the positions
    of the shocks hold, then in particular 
\begin{equation}
 \ev^g(B^A - u) \sim
 \begin{cases}-\frac{{\rs^A}}{2(1+v)(1+s)^{1/2}},
 \qquad &\text{ when } \Gamma^A \text{ is spacelike with respect to } g,\\
 \frac{{\rs^A}}{2(1+v)(1+s)^{1/2}},
 \qquad &\text{ when } \Gamma^A \text{ is timelike with respect to } g,
\end{cases}
 \label{signchange}
\end{equation}
with $\rs^L = \ls, \rs^R = \rs$ positive constants.

Since, by the same assumptions, the angular derivatives of $B^A$
are small, from \eqref{glength} it follows the left shock
is timelike with respect to the Minkowski
metric but spacelike with respect to $\mB$, while the right shock is
timelike with respect to $\mB$ but spacelike with respect to $m$.
We record the result of the above calculation.
\begin{lemma}
  \label{causallemma}
	For $g = m, \mB$, we have $g(N_g^{\Gamma^A}, N_g^{\Gamma^A}) = \ev^g (B^A-u)
	+ \frac 14 |\nas B^A|^2$. Explicitly,
  \begin{equation}
   m(N^{\Gamma^A}_m,N^{\Gamma^A}_m) =  \pa_v B^A +{\frac 14}|\nas B^A|^2,
   \qquad
   \mB(N^{\Gamma^A}_{\mB}, N^{\Gamma^A}_{\mB})
   = \pa_v B^A  - \frac{u}{vs}+ {\frac 14}|\nas B^A|^2.
   \label{Nlength}
  \end{equation}
  In particular, if the assumptions \eqref{betaLassump},\eqref{betaRassump}
  about the positions of the shocks hold,
  the left shock is timelike with respect to $m$ and spacelike with
  respect to $\mB$,
 \begin{equation}
  m(N^{\Gamma^L}_m, N^{\Gamma^L}_m) > 0, \qquad
  \mB(N^{\Gamma^L}_{\mB}, N^{\Gamma^L}_{\mB}) < 0,
  \label{causal1}
 \end{equation}
 and the right shock is timelike with respect to $\mB$ and
 spacelike with respect to $m$,
 \begin{equation}
  \mB(N^{\Gamma^R}_{\mB}, N^{\Gamma^R}_{\mB}) > 0,
  \qquad
  m(N^{\Gamma^R}_{m}, N^{\Gamma^R}_m) < 0.
  \label{causal2}
 \end{equation}

 {There is a constant $c_0$ so that if $h$ is a metric with
 $h^{-1} = m^{-1} + \gamma$ where $|\gamma|
 \leq c_0 \frac{1}{1+v}\frac{1}{(1+s)^{1/2}}$ then the same statements \eqref{causal1}, \eqref{causal2}
 hold with $m$ replaced by $h$. In the same way, if
 $h^{-1} = \mB^{-1} + \gamma$ where $|\gamma| \leq c_0 \frac{1}{1+v} \frac{1}{(1+s)^{1/2}}$
 then the same statements hold with $\mB$ replaced by $h$.}
\end{lemma}

We also record for later use that if \eqref{betaLassump}-\eqref{betaRassump}
hold then at the shock $\Gamma^A$
\begin{equation}
 -\frac{1}{2} g(X, N_g^{\Gamma^A})
 \sim
 \begin{cases}
     \frac{1}{4}\left( X^n_g + \frac{{\rs^A}}{{2}(1+v)(1+s)^{1/2}} X^\ell_g\right),
		&\qquad \text{ when } \Gamma^A \text{ is spacelike with respect to } g,\\
        \frac{1}{4}\left( X^n_g - \frac{{\rs^A}}{{2}(1+v)(1+s)^{1/2}} X^\ell_g\right),
		&\qquad \text{ when } \Gamma^A \text{ is timelike with respect to } g,\\
 \end{cases}
 \label{signchange2}
\end{equation}
{again with $\rs^R = \rs, \rs^L = \ls$.}
This follows from the formula \eqref{innerproduct} and \eqref{signchange}.
In particular we note that if $X$ is timelike and future-directed,
in the spacelike case, this quantity is positive-definite
but in the timelike case it may take either sign.

\subsection{The basic structure of the equations}

We assume that $\rho$
is given in terms of the density by an equation of state $p = P(\rho)$.
We will assume that the equation of state satisfies {$P'(1)>0, P''(1)\ne 0$} and $P \in C^{\infty}(\R\setminus \{0\}).$
The enthalpy $w = w(\rho)$ is defined by
\begin{equation}
 w(\rho) = \int_1^{\rho} \frac{P'(\lambda)}{\lambda} d\lambda.
 \label{}
\end{equation}
From Bernoulli's equation, $w$ is determined from $\pa \Phi$ according to
\begin{equation}
 w(\pa\Phi) = -\pa_t \Phi - \frac{1}{2} |\nabla \Phi|^2.
 \label{bern}
\end{equation}
Since $p' >0$ it follows that $\rho \mapsto w(\rho)$ is an invertible function,
which we denote $\rho = \rho(w)$. We then define
$\varrho$ by
$\varrho = \varrho(\pa\Phi) = \rho(w(\pa\Phi))$. For the
convenience of the reader we record that for the ``polytropic''
 equation of state $p(\rho) = \rho^\gamma$ with $\gamma > 1$, we have
\begin{equation}
 w(\rho) = \int_1^\rho \gamma \lambda^{\gamma-2}\, d\lambda
 = \frac{\gamma}{\gamma -1} \left(\rho^{\gamma-1} - 1\right),\qquad
 \rho(w) = \left( \frac{\gamma-1}{\gamma} w + 1\right)^{1/(\gamma-1)}.
 \label{}
\end{equation}

With the above notation, define
\begin{equation}
 H^0(\pa\Phi) = \varrho(\pa \Phi),
 \qquad
 H^i(\pa\Phi) = \varrho(\pa\Phi) \nabla^i\Phi.
 \label{Hdefintro}
\end{equation}
Then the continuity equation takes the form
\begin{equation}
 \pa_\alpha H^\alpha(\pa \Phi) = 0,
 \label{contderivintro}
\end{equation}
with $\pa_\alpha = \pa_{x^\alpha}$ where $x^\alpha$ denote
Cartesian coordinates on $\R^4$, and the jump conditions \eqref{introRH1},
\eqref{introRH2} take the
form
\begin{equation}
 [H^\alpha(\pa\Phi)]\zeta_\alpha = 0,
 \qquad
 [\Phi] = 0.
 \label{RHintro}
\end{equation}

After an appropriate rescaling of the dependent and
independent variables (see Lemma \ref{basicstructure}),
the quantities in \eqref{Hdefintro} take the form
\begin{equation}
  H^\alpha(\pa \Phi)
  = m^{\alpha\beta} \pa_\beta \Phi + \gamma^{\alpha \beta \delta}
  \pa_\beta \Phi \pa_\delta \Phi + G^\alpha(\pa\Phi),
 \label{Hexpintro}
\end{equation}
for constants $ \gamma^{\alpha\beta \delta}$, where $G^\alpha$ is a cubic
nonlinearity,
and where the quadratic terms are of the form
\begin{equation}
 \gamma^{\alpha\beta\delta}\pa_\beta\Phi \pa_\delta \Phi
 = -\delta^{\alpha}_{u} (\pa_u \Phi)^2 + \overline{\gamma}^{\alpha\beta\delta}
 \pa_\beta \Phi \pa_\delta \Phi.
 \label{}
\end{equation}
Here, we are writing
\begin{equation}
 \delta^{\alpha}_u = \delta^{\alpha\beta}\pa_\beta u = \delta^\alpha_0
 - \delta^{\alpha}_i \omega^i, \qquad \overline{\gamma}^{\alpha\beta\delta}
 = \gamma^{\alpha\beta\delta} + \delta_u^\alpha \delta_u^\beta\delta_u^\delta.
 \label{}
\end{equation}
(we have normalized so that  $\gamma^{uuu} = \gamma^{\alpha\beta\delta}
\pa_\alpha u \pa_\beta u\pa_\delta u = -1$).
With $\overline{\gamma}^{u\beta\delta} = \overline{\gamma}^{\alpha\beta\delta}
\pa_\alpha u$, the second term satisfies
\begin{equation}
  |\overline{\gamma}^{u \beta\delta}\pa_\beta \Psi_1 \pa_\delta \Psi_2|
  \lesssim
  |\opa \Psi_1| |\pa \Psi_2| + |\pa \Psi_1| |\opa \Psi_2|,\qquad \opa_\alpha:=\pa_\alpha -
  \pa_\alpha u\,\pa_u
 \label{}
\end{equation}
with $\overline{\gamma}^{u\beta\delta} = \overline{\gamma}^{0\beta\delta}
-\omega_i \overline{\gamma}^{i\beta\delta}$,
and the coefficients $\overline{\gamma}^{\alpha\beta\delta}
= \gamma^{\alpha\beta\delta} - \delta^{\alpha \alpha'}\gamma^{\alpha''\beta\delta}\pa_{\alpha'} u
\pa_{\alpha''} u$ satisfy the bound
\begin{equation}
 (1+v)^k|\pa^k \overline{\gamma}^{\alpha\beta\delta}|\lesssim
 1,\qquad \text{ when }|u| \leq \min(t/10, 1).
 \label{gammasymbs}
\end{equation}
Therefore the continuity equation takes the form
\begin{equation}
 \pa_\alpha (m^{\alpha\beta }\pa_\beta \Phi)
 - \pa_u (\pa_u\Phi)^2 + \pa_\alpha \left(\overline{\gamma}^{\alpha\beta\delta}
 \pa_\beta \Phi \pa_\delta \Phi\right) + \pa_\alpha G^\alpha(\pa\Phi) = 0,
 \label{contintro}
\end{equation}
and with $\zeta$ as in \eqref{zetadef}, the first jump condition in \eqref{RHintro}
reads
\begin{equation}
 [\pa_v \Phi ] - [\pa_u\Phi] \pa_v B + \frac 12[\nas_i\Phi] \nas^iB
 + [\gamma^{\alpha}(\pa \Phi)]\zeta_\alpha = 0.
 \label{basicRH}
\end{equation}
See Section \ref{bcbootstrapsection}.

\subsection{The wave equation for the perturbations}
Our results are more natural to state in terms of the variable
$\Psi = r\Phi$. The equation \eqref{contintro} then
takes the form
\begin{equation}
 -4\pa_u\pa_v \Psi + \sDelta \Psi + \pa_\mu (\gamma^{\mu\nu}(\pa \Phi)
 \pa_\nu \Psi )= F.
 \label{}
\end{equation}
We expand $\Psi = \Sigma + \psi$ where $\Sigma$ is the model shock profile
\begin{equation}
 \Sigma =\begin{cases}
\frac{u^2}{2s},\qquad \text{ in } D^C,\\
0,\qquad \text{ otherwise. } \end{cases}
 \label{introsigmadef}
\end{equation}
By Lemma \ref{higherorderexterioreqns}, in the exterior regions (where $\Sigma$ vanishes),
the perturbation $\psi$ satisfies
the following quasilinear perturbation of the usual Minkowskian wave equation,
\begin{equation}
   -4\pa_u\pa_v\psi + \sDelta\psi + \pa_\mu (\gamma^{\mu\nu}(\pa\phi) \pa_\nu \psi)
   + \pa_\mu Q^\mu = F,
 \label{}
\end{equation}
where $\phi = \psi/r$ and where $Q, F$ are given in Lemma
\ref{higherorderexterioreqns}.

In the region between the
shocks, the model shock profile contributes a non-perturbative top-order term
and the perturbation $\psi$ instead satisfies an equation of the form
\begin{equation}
  -4\pa_u\left(\pa_v  + \frac{u}{vs} \pa_u\right)\psi
  +\pa_\mu (\mBB^{\mu\nu} \pa_\nu \psi) + \sDelta \psi
  + \pa_\mu (\gamma^{\mu\nu}(\pa\phi) \pa_\nu \psi)
  + \pa_\mu Q^\mu = F,
 \label{setupmBeqn}
\end{equation}
where $P, F$ are given in Lemma \ref{higherorderexterioreqns} and where
$\pa_\mu (\mBB^{\mu\nu} \pa_\nu \psi)$ involves linear terms which can be
treated perturbatively.
The equation \eqref{setupmBeqn} is a quasilinear perturbation of the wave equation
with respect to the metric $\mB$ from \eqref{mBdef}.

\subsection{The commutator fields in each region}
\label{vfsection}
In the exterior regions $D^L, D^R$, we will commute the continuity equation
\eqref{contintro} with the usual family of Minkowski vector fields,
\begin{equation}
 \Z = \{\pa_\alpha, \Omega_{ij}, \Omega_{0i}, S\},
 \label{Zdefs}
\end{equation}
where $\pa_\alpha$ denotes differentiation with respect to the usual
rectangular coordinate system on $\R^4$ and
\begin{equation}
 \Omega_{ij} = x_i\pa_j - x_j \pa_i,
 \qquad
 \Omega_{0i} = t\pa_i + x_i \pa_t,
 \qquad
 S = x^\alpha\pa_\alpha.
 \label{Zdef}
\end{equation}
It is well-known that these vector fields form an algebra and
satisfy the following commutation
properties with the Minkowskian wave operator $\Box = -\pa_t^2 + \delta^{ij}\pa_i\pa_j$,
\begin{equation}
 Z \Box q - \Box Z q = c_Z q,
 \qquad \text{ where } c_S = -2, \quad c_Z = 0 \text{ otherwise. }
\end{equation}

In the region between the shocks, we will work with the family
\begin{equation}
 \mZB = \{\Omega_{ij}, \sBo = s\pa_u, \sBt = v\pa_v\},
 \label{ZBdef}
\end{equation}
which spans the tangent space at each point.
The field $\sBo$ satisfies
\begin{equation}
 \sBo \evmB q - \evmB \sBo q = 0,
 \label{}
\end{equation}
and so it commutes with the spherically-symmetric part of the equation
in the central region,
\begin{equation}
 \sBo \pa_u \evmB q-\pa_u \evmB \sBo q = 0.
 \label{}
\end{equation}
The field $\sBt$ satisfies
\begin{equation}
 \sBt \evmB q - \evmB \sBt q = -\evmB - \frac{u}{vs^2} \pa_u,
 \label{}
\end{equation}
so that in particular,
\begin{equation}
 \sBt \pa_u\evmB q - \pa_u\evmB \sBt q = -\pa_u \evmB q
 - \pa_u \left( \frac{u}{vs^2} \pa_uq\right).
 \label{}
\end{equation}

\subsection{Volumes and areas}
\label{volumesec}

In what follows, unless mentioned explicitly, all integrals over
spacetime regions are taken with respect to the measure
$d\widehat{\mu} = \frac{1}{r^2}dxdt$ as opposed to the standard $d\mu = dxdt$.
We have made this choice because we will be working in terms of the rescaled variables
 $\psi = r \phi$ and this simplifies many of the integration-by-parts identities
 we will encounter.

 As a result, all the surface integrals we encounter
 are taken with respect to the
 surface measure induced by $d\widehat{\mu}$. We will let
 $dS$ denote the induced surface measure on the spheres $\Gamma^A_t$.
 At each time $t$, $\Gamma^A_t$ is the graph over $\mathbb{S}^2$
 of the function $r^A(t, \omega)$, which is defined by the relation
 \begin{equation}
   \label{}
     t - r^A(t, \omega) = \beta^A_s(\omega),\qquad
     s = \log(t + r^A(t,\omega)).
 \end{equation}
 Under the assumptions \eqref{betaLassump}-\eqref{betaRassump}
 on $\beta^A_s(\omega)$, it follows
 that $dS$ is equivalent to $dS(\omega)$, the usual surface
 measure on the unit sphere $\mathbb{S}^2$,
 \begin{equation}
   \label{volumeformula}
     dS \sim dS(\omega).
 \end{equation}

\section{Multiplier identities}
\label{generalenergies}

The goal of this section is to collect the basic identities we will use to
construct energies for the continuity equation \eqref{introHeqn}.
We consider a linear wave equation of the form
\begin{equation}
  \pa_\mu(h^{\mu\nu}\pa_\nu \psi) + \pa_\mu P^\mu =  F,
  \label{modelwave}
\end{equation}
in a region $D$.
Here, and for the remainder of this section, the indices $\mu, \nu$
refer to quantities expressed in the following (Minkowskian) null
coordinate system,
\begin{equation}
 x^0 = u = t -r, \quad
 x^1 = v = t +r, \quad
 x^2 = \theta^1, \quad x^3 = \theta^2,
 \label{minknull0}
\end{equation}
where $(\theta^1, \theta^2)$ are an arbitrary local coordinate system
on the unit sphere $\mathbb{S}^2$. For our applications,
in the exterior regions the metric $h$ will take the form
will either take the form $h^{\mu\nu} = m^{\mu\nu} + \gamma^{\mu\nu}$
where $m^{\mu\nu}$ denote the components of the reciprocal
of the Minkowski metric.
\begin{equation}
  m = -dudv + \frac{1}{4} (v-u)^2 d \sigma_{\S^2}.
\end{equation}
We note that with our conventions, the Minkowskian wave operator takes the form
\begin{equation}
  \label{}
    \pa_\mu(m^{\mu\nu} \pa_\nu \psi) = -4\pa_u\pa_v \psi + \sDelta \psi.
\end{equation}

In the region between the shocks, the metric $h$ will take the form
$h^{\mu\nu} = \widetilde{\mB}^{\mu\nu} + \gamma^{\mu\nu}
= \mB^{\mu\nu} + \gamma^{\mu\nu}_a + \gamma^{\mu\nu}$,
where
$\mB^{\mu\nu}$ denote the components of the reciprocal of the metric
\begin{equation}
 \mB = -dudv + \frac{u}{vs} dv^2 + \frac{1}{4} (v-u)^2 d \sigma_{\S^2}.
\end{equation}
and where  $\gamma_a^{\mu\nu} = \frac{u}{vs} a^{\mu\nu}$. Here,
the $a^{\mu\nu} = a^{\mu\nu}(u,v,\omega)$ are smooth functions
satisfying the symbol condition $(1+v)^k|\pa^k a|\lesssim 1$ as well as the null condition
\begin{equation}
	a^{\mu\nu}\pa_\mu u\pa_\nu u = 0.
 \label{intronullcondn0}
\end{equation}
If $\overline{\xi}$ denotes projection of a one-form $\xi$ away from the cotangent
space to $\{u = const.\}$,
\begin{equation}
 \overline{\xi}_\mu = \left(\delta_\mu^\nu - \tfrac{1}{2} \delta^{\nu\nu'} \pa_{\nu'} u \pa_\mu u\right)\xi_\nu,
 \qquad
 |\overline{\xi}|\lesssim |\xi_v| + |\slashed{\xi}|,
 \label{xibardef}
\end{equation}
where $\slashed{\xi}$ denotes the angular part of $\xi$,
then for any one-forms $\xi, \tau$, writing $a(\xi, \tau) = a^{\mu\nu}\xi_\mu \tau_\nu$,
we have
$a(\xi, \tau) = a(\overline{\xi}, \tau) + a(\xi, \overline{\tau})$.
In particular, \eqref{intronullcondn0} implies the bound
\begin{equation}
 |a(\xi, \tau)| \lesssim |\overline{\xi}| |\tau| + |\xi| |\overline{\tau}|.
 \label{intronullcondn}
\end{equation}

For any symmetric (2,0)-tensor $g$ and a vectorfield $X$, define the energy current $J_{X, g}$ by
\begin{equation}
 J_{X,g}^\mu[\psi] =
 g^{\mu\nu}\pa_\nu \psi X\psi -
  \frac{1}{2} g^{\alpha\beta}\pa_\alpha \psi\pa_\beta\psi X^\mu
 \label{energycurrent}
\end{equation}
and the scalar current $K_{X, g}$ by
\begin{equation}
 K_{X,g}[\psi] =
 \frac{1}{2} \pa_\alpha (g^{\mu\nu} X^\alpha)
 \pa_\mu \psi \pa_\nu \psi - \pa_\mu X^\alpha g^{\mu\nu}
 \pa_\nu \psi \pa_\alpha \psi
 \label{scalarcurrent}
\end{equation}
Then we have the basic identity
\begin{equation}
 \pa_\mu(g^{\mu\nu}\pa_\nu \psi) X\psi = \pa_\mu J^\mu_{X, g} + K_{X, g}.
 \label{JKrelation}
\end{equation}
To keep track of lower-order  terms, it is helpful to introduce
\begin{equation}
 J_{X, g, P}^\mu[\psi] = J_{X, g}^\mu + P^\mu X\psi - X^\mu P\psi
 \label{JPdef}
\end{equation}
and
\begin{equation}
 K_{X, g, P}[\psi] = K_{X, g}[\psi] +
 {(X^\alpha\pa_\alpha P^\mu - P^\alpha\pa_\alpha X^\mu)\pa_\mu \psi + (\pa_\alpha X^\alpha) P\psi},
 \label{KPdef}
\end{equation}
which are defined so that
\begin{equation}
 \pa_\mu( g^{\mu\nu}\pa_\nu \psi + P^\mu) X\psi
 = \pa_\mu J_{X, g, P}^\mu
 + K_{X, g, P}.
 \label{mainlinident}
\end{equation}
If $\zeta$ is any one-form with $|\zeta| \leq 1$, the energy current
$J_{X, g, P}$ satisfies
\begin{align}
 |\zeta(J_{X, g,P})|
 &\lesssim |g| |\pa \psi| |X\psi| + |\zeta(X)| |g(\pa \psi, \pa \psi)|
 + |P| |X\psi| + |\zeta(X)| |P| |\psi| \label{energycurrentbd0}\\
 &\lesssim |g| |X| |\pa\psi|^2
 + |P| |X| |\pa\psi|
 \label{zetaJbound0}
\end{align}
where $g(\pa\psi, \pa \psi) = g^{\mu\nu}\pa_\mu\psi\pa_\nu\psi$
and where $\zeta(X)$ denotes the usual action of a one-form on a vector field.
 The first bound will be useful along the shocks.
 The scalar current satisfies the bound
 \begin{equation}
  |K_{X, g, P}| \lesssim \left( |\pa g| |X| + |g||\pa X| \right) |\pa \psi|^2
	+ (|\pa P| |X| + |P| |\pa X|)|\pa \psi|
  \label{naiveKbd}
 \end{equation}

For our applications, we will need to keep better track of the structure
of $K$. It is convenient to work in terms of covariant
derivatives $\nabla_X = X^\mu \nabla_\mu$ defined relative
to the Minkowksi metric.
We write
\begin{multline}
 K_{X, g, P}
 =
 \frac{1}{2} (\nabla_X g^{\mu\nu})\pa_\mu \psi \pa_\nu \psi+ \frac{1}{2} \pa_\alpha X^\alpha g^{\mu\nu}\pa_\mu \psi
 \pa_\nu \psi
 - \pa_\mu X^\alpha g^{\mu\nu}\pa_\nu \psi
 \pa_\alpha \psi\\
 + \nabla_X P^\mu \pa_\mu \psi
 - P^\alpha\pa_\alpha X^\mu \pa_\mu \psi
 + (\pa_\alpha X^\alpha)P\psi
 \\
 + X^\alpha\left(\Gamma_{\mu'\alpha}^\nu g^{\mu \mu'} + \Gamma_{\mu'\alpha}^\mu g^{\nu\mu'}\right)
 \pa_\mu\psi\pa_\nu \psi
 + X^\alpha \Gamma^{\mu}_{\nu \alpha} P^\nu \pa_\mu \psi,
 \label{KXgammacov}
\end{multline}
where the Christoffel symbols (relative to the null coordinate system $(u,v,\theta^1, \theta^2)$) $\Gamma_{\nu\alpha}^\mu$ satisfy $|\Gamma|\lesssim
\frac{1}{r}$.
In what follows, we just consider the case of
a spherically-symmetric multiplier $X$,
\begin{equation}
 X = X^u(u, v)\pa_u + X^v(u,v)\pa_v,
 \label{}
\end{equation}
and in this case we have
\begin{align}
 |K_{X, g, P}|
 &\lesssim
 |(\nabla_X g)(\pa \psi, \pa \psi)|
 + |\pa X| |g(\pa \psi, \pa \psi)|
 + |g||\pa \psi| \left(|\pa X^u| |\pa \psi| + |\pa X^v| |\pa_v \psi|\right)\\
 &+ |(\nabla_X P)\cdot \pa \psi|
 + |\pa X| |P\psi|
 + |P| \left(|\pa X^u| |\pa \psi| + |\pa X^v| |\pa_v \psi|\right)\\
 &+ \frac{1}{r} |X| \left( |g| |\pa \psi|^2 + |P| |\pa \psi|\right).
 \label{scalarcurrentbd0}
\end{align}
We also note at this point that it is possible to write the above in terms
of the Lie derivative of $g$. If we return to \eqref{scalarcurrent} and \eqref{KPdef}
and recall from the definitions that $X^\alpha \pa_\alpha g^{\mu\nu} =
\mathcal{L}_X g^{\mu\nu} + g^{\mu\alpha}\pa_\alpha X^\nu + g^{\nu\alpha}\pa_\alpha X^\mu$
and that $\mathcal{L}_X P^\mu = X^\alpha\pa_\alpha P^\mu - P^\alpha\pa_\alpha X^\mu$,
we find instead the bound
\begin{equation}
 |K_{X, g, P}|
  \lesssim |(\mathcal{L}_X g)(\pa \psi, \pa \psi)| + |\pa X| |g| |\pa \psi|^2
	+ |\mathcal{L}_X P| |\pa \psi|,
  \label{scalarcurrentbd0Lie}
\end{equation}
which we will use near $r = 0$ in place of \eqref{scalarcurrentbd0}
to avoid spurious singularities at $r = 0$.

\subsection{The energy-momentum tensor}

Given a metric $h$, define the energy-momentum tensor
\begin{equation}
 Q^h[\psi](X, Y) = h(J_{X, h}[\psi], Y)
 = X\psi Y\psi - \frac{1}{2} h(X, Y) h^{-1}(\pa \psi, \pa \psi).
 \label{Qdef}
\end{equation}
We will frequently drop $\psi$ from the notation
and just write $Q^h(X, Y)$.
For a vector field $P$ we also set
\begin{equation}
 Q^h_P(X, Y) = Q^h(X, Y) + h(P, Y) X\psi - h(X, Y) P\psi.
 \label{QPdef}
\end{equation}
With notation as in \eqref{JPdef}, \eqref{KPdef},
integrating the identity
\begin{equation}
 \pa_\mu \left( h^{\mu\nu} \pa_\nu \psi + P^\mu\right)X\psi =
 \pa_\mu J_{X, h, P}^\mu + K_{X, h, P}
 \label{basicnlident}
\end{equation}
over a region $D$ bounded by two time slices and a lateral boundary
$\Gamma$ and using Lemma \ref{divthm2}, we have the following
integral identity which will be used to get energy estimates.
\begin{lemma}
  \label{integralident}
  Fix a metric $h$, vector fields $X, P$ and define $Q^h_P$ as in \eqref{Qdef}-\eqref{QPdef}.
 Let $D = \cup_{t_0 \leq t \leq t_1} D_t$ be a region bounded by a
 (possibly empty) timelike boundary $\Gamma_{-}$  and a (possibly empty)
 spacelike boundary $\Gamma_{+}$,  lying to the future of $D_{t_0}$.
 Suppose that either $\{r = 0\}$ is
 not contained in $D$, or else that $\{r = 0\}$ is contained in
 $D$ and $\lim_{r\to 0}(h^{rr} X^r) = 0$. For a spacelike surface
 $\Sigma$, let $N_h^\Sigma$ denote the future-directed
 normal vector field to $\Sigma$ defined relative to the metric $h$,
 and for a timelike surface $\Sigma$, let $N_{h}^\Sigma$ denote
 the inward-pointing normal vector field defined relative to $h$.
 Suppose that $\lim_{r\to 0} |\psi/r| <\infty$.
 Then the following identity holds
 \begin{multline}
  \int_{D_{t_1}} Q^h_P(X, N_h^{D_{t_1}})
  +\int_{\Gamma_+} Q^h_P(X, N_h^{\Gamma_+})
  - \int_{\Gamma_-} Q^h_P(X, N_h^{\Gamma_-})
  - \int_{t_0}^{t_1} \int_{D_t} K_{X, h, P}\\
  =
  \int_{D_{t_0}} Q^h_P(X, N_h^{D_{t_0}})
  +
  \int_{t_0}^{t_1} \int_{D_t} -\pa_\mu( h^{\mu\nu} \pa_\nu \psi + P^\mu) X\psi
  \label{stokesident}
\end{multline}
 where all integrals over $D$ are taken with respect to the measure
 $dudv d\sigma_{\S^2} = \tfrac{1}{r^2}dxdt $, the integrals over the boundary terms are taken
 with respect to the induced surface measure, and $K_{X, h, P}$ is defined
 in \eqref{KPdef}.
\end{lemma}

\subsection{Modified multiplier identites}
\label{modmultsec}

For our estimates, we will be considering multipliers $X = X^v \pa_v + X^u \pa_u$
where the coefficient $X^v$ is much larger than $X^u$.
This causes an issue when closing the nonlinear estimates because will not be
able to control bulk error terms of the form $|\pa X^v| |\gamma| |\pa \psi|^2$,
which are present in $K_{X, \gamma}$.
For our applications, $X^v \sim v$ (or larger) and thinking of $|\gamma|
\sim \frac{1}{v} |\pa \psi|$, controlling such a term uniformly in time would require
a bound of the form $\int_{t_0}^{t_1} \frac{1}{1+t} \|\pa \psi(t)\|_{L^\infty}\, dt
\lesssim \epsilon^{1/2}$. We only expect to have such a bound with
$\pa \psi$ replaced by $\pa^2\psi$.

These bad terms can be traced back to
the terms $\pa_u (\gamma^{uu}\pa_u \psi)X^v \pa_v \psi$
and $\pa_v(\gamma^{vu}\pa_u\psi)X^v \pa_v \psi$ in \eqref{scalarcurrent}.
To handle terms of this type,
we need to proceed more carefully and the idea is to exploit the fact that
the combination $\pa_u\pa_v \psi$ is expected to be better-behaved than a
generic second-order derivative $\pa^2 \psi$. To highest order,
this combination is already
present in the second term mentioned above and after integrating by parts
it is also present in the first order term. Using the equation for
$\pa_u\pa_v\psi$, we generate additional terms involving either
$\sDelta \psi$ (which is expected to be better behaved than a generic
second-order derivative), or nonlinear terms.

This leads to a modified version of the identity
\eqref{basicnlident} for perturbations $h^{-1} = g^{-1} + \gamma$ of $g \in\{m, \mB\}$,
\begin{equation}
 \left( \pa_\mu(h^{\mu\nu}\pa_\nu \psi) + \pa_\mu P^\mu\right) X\psi
 = \pa_\mu J_{X, g, P}^\mu + K_{X, g, P}
 +\pa_\mu \mJ_{X, \gamma, P}^\mu + \mK_{X, \gamma, P},
 \label{mainmodifiedident}
\end{equation}
where the modified energy current $\mJ$ and scalar current
$\mK$ satisfy better bounds than those in \eqref{zetaJbound0}, \eqref{naiveKbd}.

This calculation is carried out in Section
\ref{modifiedproofsecm}. The quantities $\mJ$ and $\mK$ have rather complicated
expressions (see \eqref{explicitmJ}-\eqref{explicitmK} in the Minkowskian case
and \eqref{explicitmJmB}-\eqref{explicitmKmB} for the version in the central region)
and in this section we will just record the estimates for these quantities
that we will need. These estimates and formulas
are proved in
Propositions \ref{mainminkidentprop} and \ref{mainmBidentprop}.

For our applications, we will be using \eqref{mainmodifiedident}
with $\psi$ replaced with $Z^I\psi$ for a product of vector fields
$Z^I$ and in that case, $\gamma$ and $P$ will be of the form
\begin{equation}
 \gamma \sim \frac{1}{1+v} \pa \psi,
 \qquad
 P \sim \frac{1}{1+v} \pa Z^{I_1} \psi \cdot \pa Z^{I_2} \psi,
 \qquad |I_1| + |I_2| \leq |I|-1
 \label{gammaPheur}
\end{equation}
To prove our estimates, we will assume some bounds for the quantities
$\gamma$ and $P$ which are designed to capture the expectation that
they are of the form \eqref{gammaPheur} and that our bootstrap assumptions
hold; see in particular Lemmas \ref{timeintegrability-right},
\ref{timeintegrability-center} and \ref{timeintegrability-left}.

\subsubsection{Assumptions on perturbative quantities}
\label{pertsec}
We fix a metric $g \in \{m, \mB\}$ and a multiplier $X = X^n_g n + X^\ell_g \ell^g$
with $X^n_g, X^\ell_g \geq 0$. We will only consider vector fields $X$ satisfying
$X^n_g \leq X^\ell_g$. We define
\begin{equation}
 |\pa \psi|_{X,g}^2 = X^\ell_g \left(|\ell^g \psi|^2 + |\nas \psi|^2\right) + X^n_g |n\psi|^2.
 \label{Xnorm0}
\end{equation}

We assume that for $\epsilon$ sufficiently small, the perturbation
$\gamma$ and our multipliers satisfy the bound
\begin{align}
&|\gamma| \leq \epsilon\frac{1}{(1+v)(1+s)^{1/2}},
\qquad
X^\ell_g |\gamma| \leq \epsilon X^n_g.
\label{pert1}
\end{align}
Note that \eqref{pert1} implies that
\begin{equation}
 |X||\gamma||\pa \psi|^2 \leq \epsilon |\pa \psi|_{X,g}^2.
 \label{pert1useful}
\end{equation}
which will be used to handle many of our error terms. We will also assume
that the initial time $t_0$ has been chosen sufficiently large,
\begin{equation}
 \frac{1}{t_0} \leq \epsilon_0,
 \label{largestart0}
\end{equation}
which will be needed to absorb some error terms in the central region.

For most of our multipliers $X$, the first bound in \eqref{pert1} will automatically
imply the second bound there. It is only for the estimate
in the leftmost region that the last bound in \eqref{pert1} is
actually needed, but it makes proving the estimates more convenient.

\subsection{Estimates for the modified scalar and energy currents in the
exterior regions}

We now collect some estimates for the quantities $\mJ, \mK$ that will
be used in $D^R$ and $D^L$.
In $D^R$ we will need to multiply by the field $X_R$ and in
$D^L$ we will need to multiply by both $X_L$ and $X_M$, where
\begin{equation}
 X_R = (1 + |u|)^\mu \pa_t + u\pa_u + v\pa_v,
 \qquad
 X_L = uf(u)\pa_u + vf(v) \pa_v,\qquad
 X_M = ((g(r) +1))(\pa_v - \pa_u),
 \label{multiplierslocal}
\end{equation}
where $\mu > 0$ and where, with $\alpha$ as in \eqref{parameters},
\begin{equation}
 f(z) = \log z(\log \log z)^{\alpha},\qquad
 g(z) = (\log(1+z))^{1/2})f(\log(1+z)).
 \label{fgdeflocal}
\end{equation}

Note that by
contrast with \eqref{mainlinident}, which involves only a reference metric $g$,
the right hand side of the identity below is expressed in
terms of both $g = m$ and the perturbation $\gamma=h-m$.
\begin{prop}
	 \label{effectivemmmink}
	 Suppose that $\psi$ satisfies \eqref{modelwave} in either
	 $D^L$ or $D^R$ and set
	 $\gamma^{\mu\nu} = h^{\mu\nu} - m^{\mu\nu}$.
	 Let $X$ denote any of $X_L, X_M, X_R$ as defined in
	 \eqref{multiplierslocal},
	and suppose that $\gamma$ satisfies the conditions in \eqref{pert1} for some $\epsilon > 0$.
	  Then
	 \begin{equation}
	 \left( \pa_\mu(h^{\mu\nu}\pa_\nu \psi) + \pa_\mu P^\mu\right) X\psi
	 = \pa_\mu J_{X, m}^\mu + K_{X, m}
	 +\pa_\mu \mJ_{X, \gamma, P}^\mu + \mK_{X, \gamma, P},
	  \label{}
	 \end{equation}
	 where the perturbed energy current $\mJ^\mu_{X, \gamma, P}$ satisfies
	 the following bounds. With notation as
	 in \eqref{Xnorm0}, if $\zeta$ is any one-form with $|\zeta| \leq 1$,
	 if $|u| \leq v/8$ then for any $\delta > 0$,
	 \begin{multline}|\zeta(\mJ_{X, \gamma, P})|
	\lesssim
	 \delta |X^\ell_m||\evm \psi|^2
	+ \left(1 + \frac{1}{\delta}\right)|\gamma| |\pa \psi|_{X,m}^2
	+  |\zeta(X)| |\gamma| |\pa \psi|^2
	+
	|\slashed{\zeta}|^2 |\pa \psi|_{X,m}^2
	\\
  +  \left(1 + \frac{1}{\delta}\right) |X| |P|^2
  +  |X^n_m|^{1/2} |P| |\pa \psi|_{X, m}.
	  \label{zetaclose}
	 \end{multline}
	 and if $|u| \geq v/8$, then
	 \begin{equation}
	  |\zeta(\mJ_{X, \gamma, P})| \lesssim
		|X||\gamma||\pa \psi|^2 + |X||P||\pa \psi|.
	  \label{zetafar}
	 \end{equation}

The modified scalar current $\mK$ satisfies the following bounds. If $|u| \leq v/8$ then
\begin{multline}
	 |\mK_{X, \gamma, P}|\lesssim
	 \left( |\nabla \gamma| + \frac{1}{1+|u|} |\gamma|
 	+ \frac{|X^\ell_m|^{1/2}}{|X^n_m|^{1/2}} |\nabla_{\evm} \gamma|
	+ \frac{|X^\ell_m|^{1/2}}{|X^n_m|^{1/2}} |\nas \gamma|
 	\right)  |\pa \psi|_{X, m}^2
 	+ |X^n_m| |F|  |\pa \psi|_{X, m}\\
	+ \left(|\nabla P| + \frac{|P|}{1+|u|} +
	\frac{|X^\ell_m|^{1/2}}{|X^n_m|^{1/2}} |\nabla_{\evm} P|
	+|X^\ell_m| |\nas P| \right) |X^n_m|^{1/2} |\pa \psi|_{X, m}
	\\
  +
   |P| |\pa_u X^v| |\evm \psi|
	+
	|P||X|  \left( |F| + \frac{1}{1+v} |P|\right)
	 \label{modifiedKbound0}
	\end{multline}
		and in the region $|u| \geq v/8$, we instead have
	\begin{equation}
	 |\mK_{X, \gamma, P}|
	 \lesssim |\nabla \gamma | |X| |\pa \psi|^2 + |\gamma| |\pa X| |\pa \psi|^2
	 + |\nabla P| |X| |\pa \psi|
	 +
	  \left(\frac{1}{r} +\frac{1}{1+v}\right)|X| \left(|\gamma| |\pa \psi|^2+ |P| |\pa \psi|\right)
	 \label{trivialKbound0}
	\end{equation}
 \end{prop}
 \begin{proof}
	 It is straightforward to verify that each of the given multipliers
	 satisfy the condition \eqref{Xgrowth}, and so the bounds follow
	 from Proposition \ref{mainminkidentprop}.
 \end{proof}

%

\begin{remark}
 For our applications, $\zeta$ will be the one-form dual to the
 outward-pointing normal to a surface $\Sigma$. When $\Sigma = D_t$ is
 a time slice then it will suffice to bound $|\zeta(X)| \leq |X|$.
 When $\Sigma = \Gamma$ is one of our shocks, we will instead
 have a bound of the form $|\zeta(X)| \lesssim
 X^n_g + (1+v)^{-1}(1+s)^{-1/2}X_g^\ell$. That is, the ``large'' component
 $X^\ell_g$ is suppressed. This, and the smallness of $|\gamma|$
 expressed in \eqref{pert1}, will be needed to close our estimates.
\end{remark}

We will need a version of the above when $h$ is instead a perturbation of
the reciprocal of the metric $\mB$, up to terms with small coefficients that
verify a null condition.
For this we set
\begin{equation}
 \mBB^{\mu\nu} = \mB^{\mu\nu} + \frac{u}{vs} a^{\mu\nu}
 =\mB^{\mu\nu} + \gamma_a^{\mu\nu}
 \label{mBBdef}
\end{equation}
where $a^{\mu\nu} = a^{\mu\nu}(u,v,\omega)$ are smooth functions verifying
$(1+v)^k |\pa^k a|\lesssim 1$ and the null condition \eqref{intronullcondn}.

The following is an immediate consequence of Proposition \ref{mainmBidentprop}.
\begin{prop}
	 \label{effectivemmmB}
	 Suppose that $\psi$ satisfies \eqref{modelwave} and set
	 $\gamma^{\mu\nu} = h^{\mu\nu} - \mBB^{\mu\nu}$
	 with notation as in \eqref{mBBdef}.
	Fix a vector field $X = X^n_{\mB} n + X^\ell_{\mB} \ell^{\mB}$
	with $X^\ell_{\mB} = v$ and $X^n_{\mB} \gtrsim (1+s)^{-1/2}$ and
	$|\pa X|\lesssim 1$.
	Suppose that $\gamma$ satisfies the conditions in \eqref{pert1} for some $\epsilon > 0$
	and that \eqref{largestart0} holds for sufficiently small $\epsilon_0$.
	  Then
	 \begin{equation}
	 \left( \pa_\mu(h^{\mu\nu}\pa_\nu \psi) + \pa_\mu P^\mu\right) X\psi
	 = \pa_\mu J_{X, \mBB}^\mu + K_{X, \mBB}
	 +\pa_\mu \mJ_{X, \gamma, P}^\mu + \mK_{X, \gamma, P},
	  \label{}
	 \end{equation}
	 where the perturbed energy current $\mJ^\mu_{X, \gamma, P}$ satisfies
	 the following bound: if $\zeta$ is any one-form with $|\zeta| \leq 1$,
	 \begin{align}
	|\zeta(\mJ_{X, \gamma, P})|
	&\lesssim
	\delta v|\evmB \psi|^2
	+\left(1 + \frac{1}{\delta}\right)|\gamma| |\pa \psi|_{X,\mB}^2
	+ |\zeta(X)| |\gamma| |\pa \psi|^2
	+
	 \epsilon |\zeta(J_{X,\gamma_a})|
	 \\
	 &+  |\slashed{\zeta}|^2 |\pa \psi|_{X, \mB}^2
	 +\left(1 + \frac{1}{\delta}\right) v |P|^2
	 + \frac{1}{(1+s)^{1/2}} |P| |\pa \psi|,
	  \label{zetaclosemBtilde}
	 \end{align}
	 and the modified scalar current $\mK_{X, \gamma, P}$
	satisfies
	\begin{align}
		\label{mBmodifiedKboundstatement}
		|\mK_{X, \gamma, P}| &\lesssim
		\left( |\nabla \gamma| + \frac{|\gamma|}{1+s}
		+ \frac{|X^\ell_{\mB}|^{1/2}}{|X^n_{\mB}|^{1/2}}
		\left(|\nabla_{\evm} \gamma|+ |\nas \gamma|\right)
		\right)  |\pa \psi|_{X, \mB}^2
		+ \frac{1}{(1+v)^{1/4}} |F|  |\pa \psi|_{X, \mB}
		\\
		&+   \left( |\nabla P^u| + \frac{|P^u|}{1+s}
		+|X^\ell_{\mB}|^{1/2} \left(|\nabla_{\evmB} P| + |\nas P|
		 \right) \right)|X^n_{\mB}|^{1/2}|\pa \psi|_{X, \mB}\\
		&+ \epsilon\left( \frac{1}{(1+v)^{3/2}} |\pa \psi|^2 + \frac{1}{(1+v)^{1/2}} (|\evmB \psi|^2 + |\nas \psi|^2)
		\right)\\
&+
\frac{1}{(1+s)^{1/2}} \left( |\nabla P^u| + \frac{|P^u|}{1+v}\right) |\pa \psi|
+v|P| \left(|\nabla P| + \frac{|P|}{1+v} + |F|\right).
	\end{align}

 \end{prop}

We now record an analogue of Lemma \ref{integralident}. For this we introduce
the modified energy-momentum tensor
\begin{equation}
 \mQ_P^h(X, Y) = Q_P^g(X, Y)
 + h(\mJ_{X, P}, Y).
 \label{mQdef}
\end{equation}
where $Q_P^g$ is as in \eqref{QPdef}.
By Lemma \ref{divthm3}, we have
 \begin{lemma}
   \label{integralidentmodified}
   Fix a metric $h$, vector fields $X, P$ and define $\mQ^h_P$ as in \eqref{mQdef}.
  Let $D = \cup_{t_0 \leq t \leq t_1} D_t$ be a region bounded by
  two time slices, a
  (possibly empty) timelike boundary $\Gamma_{T}$  and a (possibly empty)
  spacelike boundary $\Gamma_{S}$, with $\Gamma_S$ lying to the future of $D_{t_0}$.
  Suppose that either $\{r = 0\}$ is
  not contained in $D$, or else that $\{r = 0\}$ is contained in
  $D$ and that
  with $\gamma = h^{-1}-g^{-1}$, we have
$\lim_{r\to 0} ((g^{rr}+\gamma^{rr} )X^r) = 0$. For a
spacelike surface
  $\Sigma$, let $N_h^\Sigma$ denote the future-directed
  normal vector field defined relative to the metric $h$
  and if $\Sigma$ is timelike, let $N_h^\Sigma$ denote the inward-pointing
  normal vector field defined relative to the metric $h$.
  Suppose that $\lim_{r\to 0} |\psi/r| <\infty$.
  Then the following identity holds
  \begin{multline}
   \int_{D_{t_1}}\mQ^h_P(X, N_h^{D_{t_1}})
   +\int_{\Gamma_S} \mQ^h_P(X, N_h^{\Gamma_+})
   - \int_{\Gamma_T} \mQ^h_P(X, N_h^{\Gamma_-})
   - \int_{t_0}^{t_1} \int_{D_t} \mK_{X, \gamma, P} \\
   =
   \int_{D_{t_0}} \mQ^h_P(X, N_h^{D_{t_0}})
   +
   \int_{t_0}^{t_1} \int_{D_t} -\pa_\mu( h^{\mu\nu} \pa_\nu \psi + P^\mu) X\psi
   \label{modifiedstokesident}
 \end{multline}
  where all integrals over $D_t$ are taken with respect to the measure
  $dudv d\sigma_{\S^2} = \tfrac{1}{r^2}dxdt $, the integrals over the boundary terms are taken
  with respect to the induced surface measure, and $\mK$ is as
  in the previous two results.
 \end{lemma}

\section{Formulas for the energy-momentum tensor and scalar currents}
\label{energy0}
In this section we consider a metric $h$ which is a perturbation of either
the Minkowski metric
$m$, or the metric $\mB$ defined in \eqref{mBdef}.
We collect here some basic formulas and estimates for the modified
energy-momentum tensor $\mQ^h_P$ defined in \eqref{mQdef},
the modified energy-momentum tensor
$\mQ_P^h$ defined in \eqref{mQdef}, and the linear part of the scalar current
$K_{X, g}$ defined in \eqref{KPdef}.

Each of our metrics $g$ admit spherically-symmetric null vectors
$(n, \ell^g)$ which we have normalized with $g(\evg, \eu) = -\frac{1}{2}$.
Since $g(\nabla_g \psi, X) = X\psi$ for any vector field
$X$, we have
\begin{equation}
 \nabla_g\psi = -2 (\eu\psi) \evg - 2 (\evg\psi) \eu + \nas\psi,
 \qquad
  g(\nabla_g \psi, \nabla_g \psi)
  = - 4\evg\psi \eu\psi + |\nas \psi|^2,
 \label{eik}
\end{equation}
and so the energy-momentum tensor takes the form
\begin{equation}
 Q^g(X, Y) = X\psi Y\psi +2 g(X,Y) \evg \psi \eu\psi
 - \tfrac{1}{2} g(X, Y)|\nas \psi|^2.
 \label{genEM1}
\end{equation}
If $X =X^\ell_g \evg + X^n_g n$ is spherically symmetric and
$Y = Y^\ell_g \evg + Y^n_g \eu + \slashed{Y}$ where $\slashed{Y}$
is the angular part of $Y$,
we also have
\begin{equation}
 g(X, Y) = -\tfrac{1}{2} (\Xo_g Y^n_g + \Xt_g Y^\ell_g),
 \label{ggenEM2}
\end{equation}
and so in this case
\begin{equation}
 Q^g(X, Y)
 = \Xo_g Y^\ell_g (\evg \psi)^2
 + \Xt_g Y^n_g (\eug \psi)^2
 - \frac{1}{2} g(X, Y) |\nas \psi|^2
 + X\psi \slashed{Y}\psi.
 \label{genEM2}
\end{equation}

Before proceeding, we record the following simple result.
\begin{lemma}
	\label{multiplierpointing}
	Suppose that \eqref{betaLassump}-\eqref{betaRassump} hold.
	Then $X_L$ is timelike and future-directed with respect to $m$ in
	$D^L$, $X_R$ is timelike and future-directed with respect to $\mB$
	in $D^R$, and the fields $X_T, X_C$ are future-directed
	and timelike with respect to $\mB$.
\end{lemma}
\begin{proof}
 The statements for $X_L, X_R$ are immediate. For $X_T, X_C$, we
first note that expressing the fields $X_C, X_T$ in terms of $n, \evmB$, we have
\begin{equation}
  \label{}
  X_C = sn + v\evmB, \qquad X_T = \frac{\ap}{{4}s^{1/2}} n + v\evmB,
\end{equation}
and it follows from \eqref{ggenEM2} that
\begin{align}
  \label{}
  \mB(X_T, X_T) &= -X_{T,\mB}^{\ev} X_{T,\mB}^n = -\frac{\ap}{4}\frac{v}{s^{1/2}} < 0 ,\\
  \mB(X_C, X_C) &= -X_{C,\mB}^{\ev} X_{C,\mB}^n = -vs < 0.
\end{align}

\end{proof}

\subsection{The energy-momentum tensor on the constant time slices}
The normals to the time slices $D^A_t$ are
\begin{equation}
 N_m^{D_t} = \pa_v + \pa_u = \ev^m + \eu^m
 \qquad
 N_{\mB}^{D_t} = \pa_v + \left(1 + 2\frac{u}{vs} \right) \pa_u
 = \evmB + \left(1+\frac{u}{vs}\right) \eu^{\mB},
 \label{ndtcoefs}
\end{equation}
so using \eqref{genEM2} and writing
$N_g^{D_t} = N_g^\ell \evg + N_g^n \eu^g$, we have
\begin{equation}
 Q^g(X, N_g^{D_t})
 = N_g^\ell \Xo_g (\evg\psi)^2
 + N_g^n \Xt_g (\eug\psi)^2
 - \frac{1}{2} g(X, N_g^{D_t}) |\nas \psi|^2.
 \label{generalEMdecomp}
\end{equation}
As a result,
\begin{align}
 Q^m(X, N_m^{D_t})
 &= X_m^\ell (\ev^g \psi)^2 + X_m^n  (n \psi)^2
 + {\frac{1}{4}} (X^n_m + X_m^\ell) |\nas \psi|^2,
 \\
  Q^{\mB}(X, N_{\mB}^{D_t})
  &= X_{\mB}^\ell (\ev^g \psi)^2 + X_{\mB}^n\left(1 + \tfrac{u}{vs}\right)  (n \psi)^2
  + {\frac{1}{4}} \left(X_{\mB}^n + X_{\mB}^\ell\left(1 + \tfrac{u}{vs}\right)\right) |\nas \psi|^2.
 \label{}
\end{align}
Note that in this setting $X$ is timelike when $X^\ell_g X^n_g >0$ and
future-directed when $X^\ell_g + X^n_g >0$ when $g = m$ and
$X^\ell_g + X^n_g(1 + \tfrac{u}{vs}) >0$ when $g = \mB$ so in
particular these quantities are positive definite when $X$ is timelike and
future-directed. In fact, recalling the definition \eqref{Xnorm0} from the previous section,
\begin{equation}
 |\pa \psi|_{X,g}^2 = X^\ell_g \left(|\ell^g \psi|^2 + |\nas \psi|^2\right)
 + X^n_g |n\psi|^2,
 \label{Xnormdef1}
\end{equation}
we have the bound
\begin{equation}
 Q^g(X, N_g^{D_t}) \geq C_0 |\pa \psi|_{X, g}^2
 \label{QXnorm}
\end{equation}
for a constant $C_0 > 0$,
when $X^\ell_g, X^n_g \geq 0$, for $g = m, \mB$.

By the bounds \eqref{zetaclose} and \eqref{zetafar},
this implies the following bounds for the perturbed energy-momentum tensor
$\mQ^h_P(X, N_h^{D_t})$ when $h$ is a perturbation of one of our
metrics $m, \mBB$.

\begin{lemma}
	\label{timeslicelemma}
	Suppose that either
	\begin{alignat}{2}
	 g &= m &&\quad \text{ and } X = X_L,\text{ or } X_R, \text{ or }\label{choice1}\\
	 g &= \mB &&\quad \text{ and } X = X_{T}, \text{ or } X_{C}.
	 \label{choice2}
 \end{alignat}
 Fix a metric $h$ and let $\gamma = h^{-1} - m^{-1}$ when $g = m$ and $\gamma = h^{-1} - \mBB^{-1}$
 when $g = \mB$,
 with notation as in \eqref{mBBdef}.
 There is a constant $\epsilon^\prime$ so that if
 $\gamma$ and $X$ satisfy the perturbative assumptions \eqref{pert1}
 with $\epsilon < \epsilon^\prime$ and if \eqref{largestart0}
 holds with
 $\epsilon_0 < \epsilon^\prime$,
 then with $|\pa \psi|^2_{X, g}$ defined as in \eqref{Xnormdef1}
 and the modified energy-momentum tensor $\mQ$ defined as in
 \eqref{mQdef},
 on the time slices $D_t$ we have
 \begin{equation}
  |\pa \psi|^2_{X, g} \lesssim \mQ_P^h(X, N^{D_t}_h) +  X^\ell_g |P|^2.
  \label{timesliceperturbbd}
 \end{equation}
\end{lemma}
\begin{proof}
	We first consider the Minkowskian case.
	We write
	\begin{equation}
	 \mQ_P^h(X, N^{D_t}_h) =
	 \zeta(\mJ_{X, h, P}) ={Q^m(X, N^{D_t}_m)} + \zeta(\mJ_{X, \gamma, P}),
	 \label{linearizeQ}
	\end{equation}
	where $\zeta = dt$. We just prove the bound in the region
	$|u| \leq v/8$, the other region being simpler.

	We first use that by the assumption \eqref{pert1}, $|\gamma| \leq \epsilon$.
	By the bounds \eqref{zetaclose} and \eqref{pert1useful} to get
	\begin{multline}
	 |\zeta(\mJ_{X, \gamma, P})|
	 \lesssim \delta |\pa \psi|^2_{X,m} + \delta^{-1} \epsilon |\pa \psi|^2_{X,m}
	 + |X||\gamma| |\pa \psi|^2
	 + \left(\delta^{-1}   + 1\right)X^\ell_m |P|^2 +
	 (X^n_g)^{1/2} |P| |\pa \psi|_{X,m}\\
	 \lesssim
	 \delta |\pa \psi|^2_{X,m} +
	 \left(\delta^{-1}   + 1\right)\epsilon|\pa \psi|^2_{X,m}
	 + \left(\delta^{-1}   + 1\right)X^\ell_m  |P|^2,
	 \label{}
	\end{multline}
	where we bounded $X^n_m\leq X^\ell_m$. Taking $\delta$ and then $\epsilon$
	sufficiently small, we can arrange for
	\begin{equation}
	 |\zeta(\mJ_{X, \gamma, P})| \leq \frac{1}{4} C_0 |\pa \psi|_{X,m}^2 + C X^\ell_g |P|^2,
	 \label{}
	\end{equation}
	with $C_0$ as in \eqref{QXnorm}, for a constant $C > 0$,
	 and the result now follows from
	\eqref{linearizeQ} and \eqref{QXnorm}.

	When $g$ is instead a perturbation of $\mBB$, the argument is similar
	but we use \eqref{zetaclosemBtilde} in place of \eqref{zetaclose}. We first
	write
	\begin{equation}
	 \mQ_P^h(X, N^{D_t}_h) = {Q^{\mB}(X, N^{D_t}_{\mB})} + \zeta(\mJ_{X, \gamma, P})
	 + \zeta(J_{X, \gamma_a, P}),
	 \label{linearizeQmB}
	\end{equation}
	where $\zeta = dt$, where we wrote $h = \mB + \gamma + \gamma_a$.
	Using \eqref{zetaclosemBtilde} to bound the second term on the right-hand side,
	we have
	\begin{align}
	 |\zeta(\mJ_{X, \gamma, P})|
	 &\lesssim
	 \delta v|\evmB \psi|^2
 	+\delta^{-1}\epsilon |\pa \psi|_{X,\mB}^2
 	+ |X| |\gamma| |\pa \psi|^2
 	+
 	 \epsilon |\zeta(J_{X,\gamma_a})|
 	 \\
 	 &\qquad+\left(\delta^{-1}   + 1\right) v |P|^2
 	 + \frac{1}{(1+s)^{1/2}} |P| |\pa \psi|\\
	 &\lesssim \delta |\pa \psi|^2_{X,\mB} + \left(\delta^{-1}   + 1\right) \epsilon |\pa \psi|^2_{X, \mB}
	 + \left(\delta^{-1}   + 1\right) X^\ell_{\mB} |P|^2
	 \label{}
	\end{align}
	where we used that $(1+s)^{-1/2} \lesssim X^n_{\mB}$ for both multipliers
	in this region and used \eqref{QAtimeslice} to handle the term
	$|\zeta(J_{X,\gamma_a})|$. Using \eqref{QAtimeslice} again to handle
	the last term in \eqref{linearizeQmB}, we find that
	\begin{multline}
	 |\zeta(\mJ_{X, \gamma, P})|
	 + |\zeta(J_{X, \gamma_a, P})|\\
	 \lesssim
	 \delta |\pa \psi|^2_{X,\mB} + \left(\delta^{-1}   + 1\right) \epsilon |\pa \psi|^2_{X, \mB}
	 + \left(\delta^{-1}   + 1\right) X^\ell_{\mB} |P|^2
	 +  c_0(\epsilon_0) |\pa \psi|^2_{X,\mB},
	 \label{}
	\end{multline}
	where $c_0$ is a continuous function with $c_0(0) = 0$. Taking $\epsilon_0$,
	$\delta$, and then $\epsilon$ sufficiently small, we again get the needed bound
	from \eqref{QXnorm}.
\end{proof}

\subsection{The energy-momentum tensor along the shocks}
\label{emformulageneral}

We start by recording the fact that by \eqref{genEM2} and the formulas
\eqref{normalnullframe} for the normal $N_g = N_g^{\Gamma}$ to either of the shocks
$\Gamma$, the energy-momentum tensor takes the form
\begin{multline}
 Q^g(X, N_g) = X^n_gN^n_g (n\psi)^2 + X^{\ev}_g N^{\ev}_g (\ev^g \psi)^2
 -\frac{1}{2} g(X, N_g) |\nas \psi|^2
 + X\psi \slashed{N}\psi\\
 =-X^n_g \ev^g (B-u) (n\psi)^2 + X^\ev_g (\ev^g \psi)^2-\frac{1}{2} g(X, N_g) |\nas \psi|^2
 + X\psi \slashed{N}\psi,
 \label{Qnulldecomp0}
\end{multline}
when $X = X^n_g n + X^\ell_g \ev^g$.
We note that by \eqref{signchange} and \eqref{signchange2},
if the assumptions \eqref{betaLassump}-\eqref{betaRassump} hold,
we have 
\begin{multline}
 Q^g(X, N_g^A)\\
 \sim  \frac{{\rs^A}}{2(1+v)(1+s)^{1/2}} X^n_g (n\psi)^2
+ X^\ell_g(\ev^g \psi)^2
+ \frac{1}{4}\left(X^n_g + \frac{{\rs^A}}{2(1+v)(1+s)^{1/2}} X^\ell_g\right) |\nas \psi|^2
+ X\psi \slashed{N}\psi,
\label{slemg}
\end{multline}
when $\Gamma^A$ is spacelike with respect to $g$, {where $\rs^R = \rs, \rs^L = \ls$,
the positive constants from \eqref{eq:betainitial}}. When $\Gamma^A$ is timelike
with respect to $g$, we instead have
\begin{multline}
	 Q^g(X, N_g^A)
	 \\ \sim
     -\frac{{\rs^A}}{2(1+v)(1+s)^{1/2}} X^n_g (n\psi)^2
+ X^\ell_g(\ev^g \psi)^2
+ \frac{1}{4}\left(X^n_g - \frac{{\rs^A}}{2(1+v)(1+s)^{1/2}} X^\ell_g\right) |\nas \psi|^2
+ X\psi \slashed{N}\psi.
 \label{tlemg}
\end{multline}

The formula \eqref{slemg} suggests that we should work in terms of
the quantities
\begin{equation}
    |\pa \psi|^2_{X, g, \SL} = X^{\ev}_g(\evg \psi)^2 +   \left(X_g^n + \frac{{\rs^A}}{(1+v)(1+s)^{1/2}}X^\ell_g\right) |\nas \psi|^2
    +\frac{{\rs^A}}{(1+v)(1+s)^{1/2}} X_g^n(n\psi)^2.
 \label{plusnorm}
\end{equation}
{The constants $\rs^A$ will not play an important role in what follows, and
we allow all the implicit constants in the following to depend on $\rs^A$.}

To deal with some of the upcoming perturbative quantities, it will be
helpful to record the following result.

\begin{lemma}
	Suppose that $g$ and $X$ are as in \eqref{choice1}-\eqref{choice2}.
	Fix $h$ and let $\gamma = h^{-1} - m^{-1}$ when $g = m$ and $\gamma = h^{-1} - \mBB^{-1}$
	when $g = \mB$.
	Suppose that $X, \gamma$ satisfy the perturbative assumption
	\eqref{pert1} and that the positions of the shocks \eqref{betaLassump}-\eqref{betaRassump} hold.
	Let $\mJ_{X, \gamma}$ denote the modified energy current defined in
	Proposition \ref{effectivemmmink} (when $g = m$) and \ref{effectivemmmB} (when $g = \mB$).
	Define $\zeta^{\Gamma^A}$ is as in \eqref{zetadef}
	so that $N^{\Gamma^A}_g = g^{-1} \zeta^{\Gamma^A}$.
	Then when $g = m$, for any $\delta > 0$, at the shock $\Gamma^A$ we have the bound
		\begin{multline}
		 |\zeta^{A}(\mJ_{X, \gamma, P})|
		 \lesssim
		 \left(\delta + \epsilon_2 + \epsilon + \frac{\epsilon}{\delta} \right)
		 |\pa \psi|_{X, m, \SL}^2 + \epsilon_2 \frac{|X^\ell_m|}{(1+v)^2}(1+s) |\nas \psi|^2
			\\
            +
			\left(1 + \frac{1}{\delta}\right)|X| |P|^2
			+ \frac{1}{\delta}(1+s)^{1/2}(1+v)  |X^n_{m}||P|^2.
			\label{pertgammaPshockm}
		\end{multline}
		If $g = \mB$, there is a continuous function $c_0$ with $c_0(0) = 0$
		so that for any $\delta > 0$, at the shock $\Gamma^A$
	we have the bound
	\begin{align}
	|\zeta^{A}(\mJ_{X, \gamma, P})|
&\lesssim
\left(\delta + \epsilon_2 + \epsilon + \frac{\epsilon}{\delta}  + c_0(\epsilon_0)\right)
|\pa \psi|_{X, \mB, \SL}^2 + \epsilon_2 \frac{|X^\ell_{\mB}|}{(1+v)^2}(1+s) |\nas \psi|^2
 +v |P|^2.
	 \label{pertgammaPshockmB}
	\end{align}
\end{lemma}

\begin{proof}

We start by making some preliminary estimates.  First,
if \eqref{betaLassump}-\eqref{betaRassump} hold, then
$|\nas B^A|\lesssim \ve_2 \frac{(1+s)^{1/2}}{1+v}$ near the shock (see \eqref{derivsofB})
and so, since $\slashed{N} = \frac{1}{2} \nas B\cdot \nas$, when $X^n_g, X^\ell_g > 0$
we have
\begin{multline}
 |X\psi \slashed{N}\psi|
 \lesssim
 \ve_2 \tfrac{(1+s)^{1/2}}{1+v} \left( X^n_g |n\psi| |\nas \psi|
 + X^\ell_g |\ell^g\psi| |\nas \psi|\right)\\
 \lesssim
 \ve_2 \left( \tfrac{{\rs^A}}{(1+v)(1+s)^{1/2}} X^n_g (n\psi)^2 + X^\ell_g (\ell^g \psi)^2
 +  \left( X^n_g + \tfrac{{\rs^A}}{2(1+v)(1+s)^{1/2}} X^\ell_g \right) |\nas \psi|^2\right),
 \label{angularbound}
\end{multline}
where we used $\tfrac{1+s}{1+v} \lesssim 1$.
We also record the fact that in this setting,
\begin{equation}
 |\slashed{\zeta}|\lesssim \ve_2 \frac{(1+s)^{1/2}}{1+v}.
 \label{consofang2}
\end{equation}

We also point out that under the assumptions on $X$, $\gamma$, we have
the bound
\begin{equation}
 |\gamma| |\pa\psi|_{X, g}^2  \lesssim \epsilon |\pa \psi|_{X, g, \SL}^2,
 \label{gammatoSL}
\end{equation}
{where we remind the reader that all implicit constants here and in what follows
depend on $\rs^A$.}
Indeed,
\begin{multline}
 |\gamma| |\pa\psi|_{X, g}^2
 = |\gamma| |X^n_g| |\pa \psi|^2 + |\gamma| |X^\ell| \left(|\evg \psi|^2
 + |\nas \psi|^2\right)\\
 \lesssim \epsilon \frac{1}{(1+v)(1+s)^{1/2}} |n\psi|^2
 + \epsilon |X^n| \left(|\evg \psi|^2 + |\nas \psi|^2\right)
 \lesssim \epsilon |\pa \psi|_{X, g, \SL}^2.
 \label{}
\end{multline}
We also point out the simple fact that
\begin{equation}
 \frac{1}{(1+v)(1+s)^{1/2}} |X^n_g| |\pa \psi|^2
 \lesssim |\pa \psi|_{X, g, \SL}^2,
 \label{smallcoefpapsi}
\end{equation}
which just follows from the definitions.

We now prove the bound. We recall from \eqref{zetaclose}
that when $g= m $ we have the bound
 \begin{align}
  |\zeta^{\Gamma^A}(\mJ_{X, \gamma, P})| \lesssim
	\delta |X^\ell_m||\evm \psi|^2
 + \left(1 + \frac{1}{\delta}\right)|\gamma| |\pa \psi|_{X,m}^2
 +  |\zeta(X)| |\gamma| |\pa \psi|^2
 +
 |\slashed{\zeta}|^2 |\pa \psi|_{X,m}^2
 \\
 +  \left(1 + \frac{1}{\delta}\right) |X| |P|^2
 +  |X^n_m|^{1/2} |P| |\pa \psi|_{X, m},
  \label{bdforzetagammaAm}
 \end{align}
and by \eqref{zetaclosemBtilde}, when $g= \mB$ we instead have
\begin{align}
 |\zeta^{\Gamma^A}(\mJ_{X, \gamma, P})|
 &\lesssim
 \delta v|\evmB \psi|^2
 +\left(1 + \frac{1}{\delta}\right)|\gamma| |\pa \psi|_{X,\mB}^2
 + |\zeta(X)| |\gamma| |\pa \psi|^2
 +  |\slashed{\zeta}|^2 |\pa \psi|_{X, \mB}^2\\
 &
	+\left(1 + \frac{1}{\delta}\right) v |P|^2
	+ \frac{1}{(1+s)^{1/2}} |P| |\pa \psi|+
 	\epsilon |\zeta(J_{X,\gamma_a})|.
 \label{bdforzetagammaAmB}
\end{align}
%

	We bound the first four terms in each expression in \eqref{bdforzetagammaAm}
	and \eqref{bdforzetagammaAmB}
	in the same way.
 The first term in each expression is bounded bounded by the right-hand side of \eqref{pertgammaPshockm},
 resp. \eqref{pertgammaPshockmB}.
 For the second term we use \eqref{gammatoSL},
 \begin{equation}
  \left( 1 + \frac{1}{\delta}\right) |\gamma|
	|\pa \psi|_{X, g}^2
	\lesssim
	\left( 1 + \frac{1}{\delta}\right)|\pa \psi|_{X, g, \SL}^2.
  \label{}
 \end{equation}
 To handle the third term, we note
that if the assumptions \eqref{betaLassump}-\eqref{betaRassump}
 about the positions of the shocks hold, then
\begin{equation}
    |\zeta^{\Gamma^A}(X)| \lesssim X^n_m + \frac{{\rs^A}}{(1+v)(1+s)^{1/2}} X^\ell_m,
 \label{}
\end{equation}
(see the estimates in \eqref{signchange2} and note that
$\zeta^{\Gamma^A}(X) = g(X, N_g^{\Gamma^A})$ for any metric $g$),
 and so, using \eqref{gammatoSL}, \eqref{pert1} and then
 \eqref{smallcoefpapsi}
 \begin{multline}
  |\zeta^{\Gamma^A}(X)| |\gamma| |\pa \psi|^2
	\lesssim |X^n_g| |\gamma| |\pa \psi|^2
	+ \frac{|X^\ell_g|}{(1+v)(1+s)^{1/2}} |\gamma| |\pa \psi|^2\\
	\lesssim \epsilon |\pa \psi|_{X, g, \SL}^2
	+ \frac{\epsilon}{(1+v)(1+s)^{1/2}} |X^n_g| |\pa \psi|^2
	\lesssim \epsilon |\pa \psi|_{X, g, \SL}^2.
  \label{}
 \end{multline}

	For the fourth term in \eqref{bdforzetagammaAm}-\eqref{bdforzetagammaAmB}, we use
	\eqref{consofang2} to get
	\begin{equation}
	 |\slashed{\zeta}|^2 |\pa \psi|_{X, g}^2
	 \lesssim \epsilon_2^2 \frac{1+s}{(1+v)^2}\left( |X^n_g| |\pa \psi|^2
	 + |X^\ell_g| \left( |\evg \psi|^2 + |\nas \psi|^2\right)\right)
	 \lesssim \epsilon_2^2 |\pa \psi|_{X, g, \SL}^2
	 + \epsilon_2^2\frac{|X^\ell_g|}{(1+v)^2} (1+s) |\nas \psi|^2,
	 \label{}
	\end{equation}
	as needed. When $g = m$ it remains to bound the terms on the last
	line of \eqref{bdforzetagammaAm} and for this we just bound
	\begin{multline}
	 |X^n_m|^{1/2} |P| |\pa \psi|_{X, m}
	 \lesssim \delta \frac{1}{(1+v)(1+s)^{1/2}} |\pa \psi|_{X, m}^2
	 + \frac{1}{\delta} |X^n_m| (1+v)(1+s)^{1/2}|P|^2\\
	 \lesssim \delta |\pa \psi|_{X, m, \SL}^2
	 + \frac{1}{\delta} |X^n_m| (1+v)(1+s)^{1/2}|P|^2
	 \label{}
	\end{multline}
	which gives the needed bounds.

	When $g = \mB$, to handle the contribution from the term in
	\eqref{bdforzetagammaAmB} involving $\gamma_a$, we just use \eqref{QAspacelike},
	and for the terms involving $P$ we just bound
	\begin{equation}
	 \frac{1}{(1+s)^{1/2}} |P||\pa \psi|^2
	 \lesssim \delta \frac{1}{1+v} \frac{1}{1+s} |\pa \psi|^2
	 + \frac{1}{\delta}(1+v) |P|^2
	 \lesssim \delta |\pa \psi|_{X, \mB, +}^2
	 + \frac{1}{\delta}(1+v) |P|^2,
	 \label{}
	\end{equation}
	using that $|X^n_{\mB}| \gtrsim (1+s)^{1/2}$.

\end{proof}

\subsubsection{The energy-momentum tensor on the spacelike side of the shock}
Let $(g, A) = (m, R)$ or $(\mB, L)$ so that $\Gamma^A$ is spacelike with
respect to $g$.
We recall the well-known fact that if $X$ is timelike and future-directed,
$\Sigma$ is a spacelike surface and $N^\Sigma_g$ is the future-directed normal to
$\Sigma$ then $Q^g(X, N_g^\Sigma) \geq 0$. In this setting, this positivity can
be seen easily from \eqref{slemg} and the fact that with our conventions,
$X$ is timelike and future-directed exactly when $X^n_g, X^{\ell}_g > 0$.

Note that if
\eqref{betaLassump}-\eqref{betaRassump} hold,
then by \eqref{angularbound}, provided $\ve_2$ is sufficiently small, if
 $X$ is timelike and future-directed, $|X\psi \slashed{N}\psi|
 \lesssim
 \ve_2 |\pa \psi|_{X,g}^2$, and it follows that there is a constant $C_+$ so that
\begin{equation}
 Q^g(X, N_g^{\Gamma^A}) \geq C_+ |\pa \psi|^2_{X, g, \SL} > 0.
 \label{Qspacelikelower}
\end{equation}

On the spacelike side of the shocks we will need a version of
\eqref{Qspacelikelower} where,
with notation as in \eqref{QPdef}, $Q^g$ is replaced by
$Q^h_P$ where $h$ is a perturbation of $g$.
It will be convenient to state these results separately on the spacelike
side of the right shock and on the spacelike side of the left shock.
We start with the result on the spacelike side of the right shock.
\begin{lemma}
  \label{spacelikeperturb}
	Let $X = X_R$ with notation as in Section \ref{fields} and write $X = X^n_m n + X^\ell_m\ev^m$
	%
	%
	Define $|\pa \psi|_{X,m,\SL}$ as in
	\eqref{plusnorm}.
	There is a constant $\epsilon^\prime >0$ so that if $\gamma = h^{-1} - m^{-1}$
	satisfies the perturbative assumptions \eqref{pert1} with
	$\epsilon < \epsilon^\prime$ and
	\eqref{betaRassump} holds with $\epsilon_2 < \epsilon^\prime$, then along $\Gamma^R$,
	\begin{equation}
	 |\pa \psi|^2_{X, m, \SL}
	 \lesssim
	 \mQ^h_P(X, N^\Gamma_h) + 
			\left(|X| + (1+s)^{1/2}(1+v)  |X^n_{m}|\right)|P|^2.
	 \label{mQspacelikelower}
	\end{equation}
\end{lemma}
\begin{proof}

We start by splitting $Q$ into a linear part and a perturbative part,
which we write as
\begin{equation}
 \mQ^h_P(X, N_h^{\Gamma^R}) = Q^m(X, N_m^{\Gamma^R}) + \zeta^{R}(\mJ_{X, \gamma, P}),
 \label{Qperturbshockformula}
\end{equation}
	with $\zeta^{R}$ as in \eqref{zetadef}. By \eqref{Qspacelikelower} we
	 have
	\begin{equation}
	 C_+ |\pa \psi|_{X, m, \SL}^2
	 \leq Q^m(X, N_m^{\Gamma^R})
	 \leq \mQ^h(X, N_h^{\Gamma^R}) + |\zeta^{R}(\mJ_{X, \gamma, P})|.
	 \label{plusnormboundbyQm}
	\end{equation}
    The result now follows after using the bound
    \eqref{pertgammaPshockm}, taking
    $\delta, \epsilon_2$, and then $\epsilon$ sufficiently small, and absorbing into
    the left-hand side.

\end{proof}

On the spacelike side of the left shock, we will instead use the following result.
\begin{lemma}
  \label{spacelikeperturbmB}
  Let $X = X_{T}$ or $X = X_{C}$ and write $X = X^n_{\mB} n + X^\ell_{\mB}\ev^{\mB}$
	%
	%
	Define $|\pa \psi|_{X,\mB,\SL}$ as in
	\eqref{plusnorm}.
	There is a constant $\epsilon^\prime >0$ so that if $\gamma = h^{-1} - \mBB^{-1}$
	satisfies the perturbative assumptions \eqref{pert1} with
	$\epsilon < \epsilon^\prime$, \eqref{betaLassump} holds with $\epsilon_2 < \epsilon^\prime$
	and \eqref{largestart0} holds with $\epsilon_0 < \epsilon^\prime$, then along $\Gamma^R$,
	\begin{equation}
	 |\pa \psi|^2_{X, \mB, \SL}
	 \lesssim
	 \mQ^h_P(X, N^\Gamma_h) + (1+v) |P|^2.
	 \label{mQspacelikelowermB}
	\end{equation}
\end{lemma}

\begin{proof}
	By \eqref{Qspacelikelower}, we have the bound
 \begin{equation}
  C_+ |\pa \psi|_{X, \mB, \SL}^2
	\leq Q^{\mB}(X, N_{\mB}^{\Gamma^R})
	\leq \mQ^h(X, N_h^{\Gamma^R}) + |\zeta^{\Gamma^R}(\mJ_{X, \gamma, P})|,
  \label{plusnormboundbyQmB}
 \end{equation}
 and recalling the bound \eqref{pertgammaPshockmB},
 \begin{equation}
  |\zeta^{A}(\mJ_{X, \gamma, P})|
\lesssim
\left(\delta + \epsilon_2 + \epsilon + \frac{\epsilon}{\delta}  + c_0(\epsilon_0)\right)
|\pa \psi|_{X, \mB, \SL}^2 + \epsilon_2 \frac{|X^\ell_{\mB}|}{(1+v)^2}(1+s) |\nas \psi|^2
 +v |P|^2,
  \label{}
 \end{equation}
 taking $\delta$ and then $\epsilon, \epsilon_2$ sufficiently small we
 get the result.
\end{proof}

%

\subsubsection{The energy-momentum tensor on the timelike side of the shock}
\label{emtimelike}
Let $(g, A) = (\mB, R)$ or $(m, L)$ so that $\Gamma^A$ is timelike with
respect to $g$. In this case the energy-momentum tensor
$Q^g(X, N_g)$ is no longer positive-definite, even when $X$ is
timelike and future-directed. For the purposes of this section,
what is relevant is the sign of $-Q^g(X, N_g^A)$ (see \eqref{stokesident}).
We note that if \eqref{betaLassump}-\eqref{betaRassump} hold then by
\eqref{tlemg} and
\eqref{angularbound}, provided $\ve_2$ is sufficiently small
and $X^n_g, X^\ell_g > 0$, we have
\begin{multline}
    -Q^g(X, N_g) \geq C_1 \frac{{\rs^A}}{(1+v)(1+s)^{1/2}}X^n_g (n\psi)^2 - C_2X^\ell_g (\evg \psi)^2\\
    + C_3 \left(\tfrac{{\rs^A}}{2(1+v)(1+s)^{1/2}}(1-\epsilon_2) X^\ell_g - (1+\epsilon_2)X^n_g \right) |\nas \psi|^2
 \label{basicnegativity}
\end{multline}
Note that the last term here need not be positive; for the multipliers $X_L$
and $X_T$ it winds up being positive for small $\epsilon_2$, but for the multiplier $X_C$
it is negative. Independently of this, the term involving
$\evg$ needs to be bounded and for this we will need to use the boundary
conditions. See Section \ref{bcbootstrapsection}.

We now bound the energy-momentum tensor along the timelike side of the left shock.
\begin{lemma}
  \label{timelikeperturb}
	Let $X = X_L$ and write $X = X^n_m n + X^\ell_m\ev^m$.
  There is a constant $\epsilon^\prime >0$ so that if $\gamma = h^{-1} -m^{-1}$
	and $X$ satisfy the perturbative assumptions \eqref{pert1} with
	$\epsilon < \epsilon^\prime$,
	\eqref{betaLassump} holds with $\epsilon_2 < \epsilon^\prime$,
	and \eqref{largestart} holds with $\epsilon_0 < \epsilon^\prime$,
	then along $\Gamma^L$,
  \begin{multline}
   \frac{1}{(1+v)(1+s)^{1/2}} X^n_m |n\psi|^2
	 + \frac{1}{(1+v)(1+s)^{1/2}} X^\ell_m |\nas \psi|^2
	 \\
	 \lesssim
	 -\mQ_P^h(X, N_h^L) +
	  X^\ell_m |\ell^m\psi|^2
+ \left(X^\ell_m  + (1+v)(1+s)^{1/2}X^n_m\right) |P|^2   \label{timelikelower}
   \end{multline}
\end{lemma}
\begin{proof}

	Following a nearly identical argument to the proof
	of Lemma \ref{spacelikeperturb}, but using \eqref{basicnegativity}
in place of \eqref{Qspacelikelower},
we find that for $\epsilon, \epsilon_2$ small enough,
\begin{multline}
    \frac{{\rs^A}}{(1+v)(1+s)^{1/2}} X^n_m |n\psi|^2
    + \left(\frac{{\rs^A}}{2(1+v)(1+s)^{1/2}}(1-2\epsilon_2)X^\ell_m - (1+ 2\epsilon_2)X^n_m \right) |\nas \psi|^2
 \\
 \lesssim
 -\mQ_P^h(X, N_h^L) +
	X^\ell_g |\ell^m\psi|^2
+ \left(X^\ell_g  + (1+v)(1+s)^{1/2}X^n_m\right) |P|^2.   \label{timelikelower0}
 \end{multline}
 Now we note that since $|u| \sim (\log v)^{1/2}$ along $\Gamma^L$ and
 $\alpha > 1$, $X = X_L$ satisfes
 \begin{multline}
  \frac{1}{2(1+v)(1+s)^{1/2}}(1-2\epsilon_2)X^\ell_m - (1+ 2\epsilon_2)X^n_m\\
	= \frac{1}{2(1+v)(1+s)^{1/2}} (1-2\epsilon_2)
	v \log v (\log \log v)^\alpha- (1+ 2\epsilon_2)|u| \log |u| (\log \log |u|)^\alpha
	\\ \gtrsim
	v \log v (\log \log v)^\alpha = X^\ell_m,
  \label{}
 \end{multline}
 along $\Gamma^L$. Therefore the second term on the left-hand side
 of \eqref{timelikelower0} is bounded from below by the second
 term on the left-hand side of \eqref{timelikelower} and the result follows.
\end{proof}

On the timelike side of the right shock, we will need a bound involving
$X_T$ and a bound involving $X_C$. We remind the reader that 
\begin{equation}
    X_T = v\pa_v + \left(  \frac{u}{s} + \frac{{\ap}}{{4}s^{1/2}}\right)\pa_u
    = v \evmB +\frac{{\ap}}{{4}s^{1/2}}n,
 \qquad
 X_C = v\pa_v + \left(s {+ \frac{u}{s}}\right) \pa_u = v \evmB +  sn.
 \label{localXCformula}
\end{equation}
\begin{lemma}
  \label{timelikeperturbmB}
	Let $X = X_C$ or $X_T$ and write $X = X^n_{\mB} n + X^\ell_{\mB}\ev^{\mB}$.
  There is a constant $\epsilon^\prime >0$ so that if $\gamma = h^{-1} - \mBB^{-1}$
	and $X$ satisfy the perturbative assumptions \eqref{pert1} with
	$\epsilon < \epsilon^\prime$,
	\eqref{betaRassump} holds with $\epsilon_2 < \epsilon^\prime$,
	and \eqref{largestart0} holds with $\epsilon_0 < \epsilon^\prime$, then along $\Gamma^R$
	we have the following bounds,
		   \begin{multline}
		    \frac{1}{(1+v)(1+s)^{1/2}} X^n_{T, \mB} |n\psi|^2
		 	 + \frac{1}{(1+v)(1+s)^{1/2}} X^\ell_{T, \mB} |\nas \psi|^2
		 	 \\
		 	 \lesssim
		 	 -\mQ_P^h(X_T, N_h^L) +
		 	  X^\ell_{T, \mB} |\evmB\psi|^2 + (1+v)|P|^2. \label{timelikelowermBXT}
		    \end{multline}
				and
  \begin{multline}
   \frac{1}{(1+v)(1+s)^{1/2}} X^n_{C, \mB} |n\psi|^2
	 \\
	 \lesssim
	 -\mQ_P^h(X_C, N_h^L) +
	  X^\ell_{\mB} |\evmB\psi|^2
 	 + \frac{1}{(1+v)(1+s)^{1/2}} X^\ell_{C, \mB} |\nas \psi|^2+ (1+v)|P|^2. \label{timelikelowermBXD}
   \end{multline}
\end{lemma}

\begin{remark}
	\label{differentmultiplierscentral}
 The above inequalities are why we need to use two different multipliers in the central
 region. The multiplier $X_C$ is needed to give us energies which are strong enough
 to get good decay estimates, but has the downside that the associated energy-momentum
 tensor along the timelike side of the right shock does not control
 angular derivatives and so we cannot close estimates using this multiplier
 alone. The multiplier $X_T$ has been chosen so that the associated energy-momentum
 tensor does control angular derivatives along the timelike side of
 the right shock, but it is too weak to give good decay estimates.
\end{remark}

\begin{proof}
 	Following a nearly identical proof to the proof
 	of Lemma \ref{spacelikeperturbmB}, but using \eqref{basicnegativity}
 in place of \eqref{Qspacelikelower}, we find that for $\epsilon, \epsilon_0, \epsilon_2$ small enough,
 for either $X = X_T$ or $X_C$,
 \begin{multline}
  \frac{1}{(1+v)(1+s)^{1/2}} X^n_{\mB} |n\psi|^2
  + \left(\frac{1}{2(1+v)(1+s)^{1/2}}(1-2\epsilon_2)X^\ell_{\mB} - (1+ 2\epsilon_2)X^n_{\mB} \right) |\nas \psi|^2
  \\
  \lesssim
  -\mQ_P^h(X, N_h^L) +
 	X^\ell_{\mB} |\evmB\psi|^2
 + \left(X^\ell_{\mB}  + (1+v)(1+s)^{1/2}X^n_{\mB}\right) |P|^2.   \label{timelikelower0mB}
  \end{multline}
	We now bound the coefficient of the angular deriatives in \eqref{timelikelower0mB}.
	When $X = X_T$, recalling \eqref{localXCformula}, for $\epsilon_2$
	sufficiently small we have the bound
	\begin{multline}
\frac{1}{2(1+v)(1+s)^{1/2}}(1-2\epsilon_2)X^\ell_{\mB} - (1+ 2\epsilon_2)X^n_{\mB} 
\\
=
	 	 \frac{v}{2(1+v)(1+s)^{1/2}}(1-2\epsilon_2) -
         \left({ \frac{u}{s}} + \frac{{\ap}}{s^{1/2}}\right)(1+2\epsilon_2)\\
		 \geq
		 \frac{1}{2}(1-2\epsilon_2) \frac{1}{(1+s)^{1/2}}
		 - \frac{1}{4}(1+2\epsilon_2) \frac{1}{s^{1/2}}\geq
	 	 \frac{1}{16} \frac{1}{(1+s)^{1/2}},
	 \label{}
	\end{multline}
	along $\Gamma^R$, where we used that $|u| \geq s^{1/2}$ there. The bound
	\eqref{timelikelowermBXT} follows.

	When $X = X_C$, the coefficient of the angular derivatives
	is no longer positive, and \eqref{localXCformula} instead gives
	the bound
	\begin{equation}
	 \frac{1}{2(1+v)(1+s)^{1/2}}(1-2\epsilon_2)X^\ell_{\mB} - (1+ 2\epsilon_2)X^n_{\mB} 
	 \lesssim
	 (1+s),
	 \label{}
	\end{equation}
	and \eqref{timelikelowermBXD} follows.
\end{proof}

  \subsection{The scalar currents}

We now compute $K_{X, g}$ where $g$ is either the Minkowski metric
$m$ or the metric $\mB$ from \eqref{mBdef}
and where $X = X^u\pa_u + X^v\pa_v$ is spherically-symmetric. Recall $K_{X, g}$ is given by
\begin{equation}
 K_{X, g} = \frac{1}{2} \pa_\alpha (g^{\mu\nu} X^\alpha)
 \pa_\mu \psi \pa_\nu \psi - \pa_\mu X^\alpha g^{\mu\nu}
 \pa_\nu \psi \pa_\alpha \psi
\end{equation}
 First, for both metrics $g^{uv}$ are constants and
 $g^{vv}$ vanishes, so we have
\begin{multline}
 \frac{1}{2}\pa_\alpha( g^{\mu\nu} X^\alpha)
 \pa_\mu \psi \pa_\nu \psi
 =\frac{1}{2} \left( (\pa_u X^u + \pa_vX^v) g^{uu} + X g^{uu}\right)
 (\pa_u\psi)^2
 + (\pa_uX^u + \pa_v X^v) g^{uv}\pa_u\psi \pa_v\psi\\
 + \frac{1}{2} \left( \pa_u X^u +\pa_v X^v
 - \frac{2}{r} Xr\right)|\nas \psi|^2.
\end{multline}
If the coefficients $X$ depend only on $u, v$, we also have
\begin{multline}
 \pa_\mu X^\alpha g^{\mu\nu}\pa_\nu \psi \pa_\alpha \psi
 \\
 = \left( \pa_v X^u g^{uv} + \pa_u X^u g^{uu}\right) (\pa_u\psi)^2
 + \pa_u X^vg^{uv} (\pa_v\psi)^2
 + \left( (\pa_u X^u + \pa_v X^v) g^{uv} + \pa_u X^v g^{uu}\right)
 \pa_u\psi\pa_v\psi,
\end{multline}
so subtracting these two expressions and writing
$Xr = \frac{1}{2} (X^v- X^u)$, and $r = \frac{1}{2} (v-u)$, we find
\begin{multline}
 K_{X, g}
 =
 \left(-\pa_v X^u g^{uv}
 +\frac{1}{2} (\pa_v X^v - \pa_u X^u) g^{uu} + \frac{1}{2}X g^{uu} \right) (\pa_u\psi)^2
 - \pa_u X^v g^{uv} (\pa_v\psi)^2\\
 -\pa_u X^v g^{uu} \pa_u \psi \pa_v\psi
 + \frac{1}{2} \left( \pa_u X^u +\pa_v X^v
 - 2 \frac{X^v-X^u}{v-u}\right)|\nas \psi|^2.
 \label{KXformula}
\end{multline}
For $g = m$ we have $m^{uv} = -2$ and this reads
\begin{equation}
 K_{X, m}
 =
 2\pa_v X^u (\pa_u\psi)^2
 +
  2\pa_u X^v  (\pa_v\psi)^2
 + \frac{1}{2} \left( \pa_u X^u +\pa_v X^v
 - 2 \frac{X^v-X^u}{v-u}\right)|\nas \psi|^2.
 \label{KXmformula}
\end{equation}
When $g = \mB$, we have $g^{uv} = -2, g^{uu} = -4 \frac{u}{vs}$. We have
\begin{equation}
 \frac{1}{2} (\pa_v X^v - \pa_u X^u) g^{uu} + \frac{1}{2}X g^{uu}
 = -2 \frac{u}{vs}(\pa_vX^v - \pa_u X^u)
 -2 \frac{1}{vs} X^u + 2 \left( \frac{u}{v^2s^2} + \frac{u}{v^2 s}\right) X^v,
 \label{}
\end{equation}
and it follows that
\begin{multline}
 K_{X, \mB} =2\left(
                 \left(\pa_v + \frac{u}{vs} \pa_u\right)X^u 
                 + \frac{u}{vs}\left(\frac{1}{v} X^v - \pa_vX^v\right)
                 -\frac{1}{vs}\left(X^u - \frac{u}{vs} X^v \right)
            \right)
     (\pa_u\psi)^2
 {+2} \pa_u X^v  (\pa_v\psi)^2\\
 +4 \frac{u}{vs} \pa_u X^v \pa_u \psi \pa_v\psi
 + \frac{1}{2} \left( \pa_u X^u +\pa_v X^v
 - 2 \frac{X^v-X^u}{v-u}\right)|\nas \psi|^2.
 \label{KXgBformula}
\end{multline}

{Noting that $(\pa_v + \frac{u}{vs}\pa_u) (u/s) = \evmB(u/s) = 0$,
using the formula \eqref{Xnullexplicit} to express $X$ in terms of $n, \evmB$
to re-write the coefficient of the first term here, writing 
$\pa_v X^v - X^v/v = v\pa_v(X^v/v) = v\pa_v(X^\ell_{\mB}/v)$,
the above can be re-written in the form
\begin{multline}
  \label{}
  K_{X, \mB} =2\left(
                \evmB X^n_{\mB} - \frac{1}{vs} X^n_{\mB}
                + \frac{u^2}{s^2} \pa_u\left(\frac{X^\ev_{\mB}}{v} \right)
            \right)
     (\pa_u\psi)^2
 {+2} \pa_u X^v  (\pa_v\psi)^2\\
 +4 \frac{u}{vs} \pa_u X^v \pa_u \psi \pa_v\psi
 + \frac{1}{2} \left( \pa_u X^u +\pa_v X^v
 - 2 \frac{X^v-X^u}{v-u}\right)|\nas \psi|^2.
 \label{KXgBformula2} 
\end{multline}
}

%
%

\section{The energy estimates}
\label{ensec2}

In this section we use the results of the previous two sections
to prove energy estimates for the wave equation
\begin{equation}
 \pa_\mu \left( h^{\mu\nu}_A \pa_\nu \psi
 + P^\mu \right)=  F,
 \qquad \text{ in } D^A,
 \label{ensecabstractwave}
\end{equation}
for $A = R, C, L$.
We assume that the reciprocal acoustical metrics $h_L,h_R$ are perturbations
of the Minkowski metric
and that $h_C$ is a perturbation of the metric $\mB$ defined in \eqref{mBdef},
in a sense made precise in the upcoming results.
\subsection{The energy estimates to the right of the right shock}

In this section, we consider the wave equation
\eqref{ensecabstractwave} when $h_R^{-1} = m^{-1}+ \gamma$ is a perturbation of the Minkowski metric,
\begin{equation}
  \label{modelwavenull}
  -4\pa_u\pa_v \psi + \sDelta \psi + \pa_\mu(\gamma^{\mu\nu}\pa_\nu \psi) + \pa_\mu P^\mu = F,
\end{equation}

The estimates in the region to the right of the right shock are fairly
simple and are based on the weighted energy estimates from
\cite{LindbladRodnianski2010} and \cite{DafermosRodnianski2010}.
We will use the following multiplier,
\begin{equation}
    X = X_R = w(u)(\pa_u + \pa_v) + r(\log r)^\nu \pa_v
 \label{enestrightfv}
\end{equation}
where $w$ is a function with $w(u) \geq 0$, $w'(u) \leq 0$
and $\nu \geq 0$.
In the proof of the main theorem, we will take
$w(u) = (1+|u|)^\mu$ for large $\mu$, but this particular
choice plays no role in the upcoming section. The term $r(\log r)^\nu\pa_v$ is needed
to control some of the boundary terms we will generate along the timelike side of the
right shock when we prove estimates in the central region, but this term is would not
be needed if our only goal was to
close the estimates in the rightmost region.

This field is timelike and future-directed,
\begin{equation}
    m(X_R, X_R) = -2w(u)(w(u) + r(\log r)^\nu) <0.
 \label{}
\end{equation}

We note at this point that if $\gamma$ satisfies the condition
\eqref{pert1}, we have
\begin{equation}
    X^\ell_m |\gamma| \leq \frac{\epsilon }{(1+v)(1+s)^{1/2}} (w(u) + r(\log r)^\nu)
 \leq \epsilon w(u) = X^n_m,
 \label{pert01right}
\end{equation}
where here we used that by \eqref{parameters}, $\nu \leq \mu/2+1/2$ and so
$(\log r)^{\nu}(1+s)^{-1/2} \leq (1+s)^{\nu-1/2} \leq (1+|u|)^{\mu}$ in $D^R$.
As a result, for this multiplier, the
 first bound in \eqref{pert1} implies the second one.

The energies in this region are, with notation as in
\eqref{plusnorm},
\begin{multline}
 E_{X}(t)
 = \int_{D^R_t} w(u) |\pa \psi|^2 + (w(u) + r(\log r)^\nu) \left(|\pa_v \psi|^2 + |\nas
     \psi|^2\right) \\
 +
 \int_{t_0}^t \int_{\Gamma^R_t} (w(u) + r(\log r)^\nu)(\pa_v\psi)^2 +
 \left(w(u)+ \frac{w(u)+r(\log r)^\nu}{(1+v)(1+s)^{1/2}}\right) |\nas \psi|^2
 + \frac{w(u)}{(1+v)(1+s)^{1/2}}(\pa_u\psi)^2\, dS dt\\
 \sim \int_{D^R_t} |\pa \psi|_{X, m}^2 +
 \int_{t_0}^t \int_{\Gamma^R_t} |\pa \psi|^2_{X, m, \SL}\, dS dt,
 \label{EXRXnormrln}
\end{multline}
where $ |\pa \psi|_{X, m}^2$ is defined as in \eqref{Xnorm0}
and $|\pa \psi|_{X,m,\SL}^2$ is defined as in \eqref{plusnorm}.

Since we are assuming $w'(u) \leq 0$,
it turns out that the scalar current $K_{X_R,m}$ contributes an additional
positive time-integrated term,
\begin{equation}
 S_{X}(t_1) = \int_{t_0}^{t_1} \int_{D^R_t}
 \left(-w'(u) + \frac{1}{4} (\log r)^\nu\right)\left(
2(\pa_v \psi)^2 + \frac{1}{2}|\nas \psi|^2\right)\, dt.
 \label{}
\end{equation}

Our estimates will involve the following perturbative error terms,
\begin{align}
 R_{P,X}(t_1)
 &= \int_{D_{t_0}^{R}} |X| |P|^2
 + \int_{D_{t_1}^{R}} |X| |P|^2
 + \int_{t_0}^{t_1}\int_{\Gamma^R_t}
			\left(|X| + (1+s)^{1/2}(1+v)  |X^n_{m}|\right)|P|^2\,  dS dt
 \label{RPXR}
\end{align}

\begin{prop}[Energy estimates in the rightmost region]
  \label{rightenest}
  Set $\gamma = h^{-1} - m^{-1}$. There is a constant
  $\epsilon^\prime > 0$ so that if the first perturbative assumption
	in \eqref{pert1} holds with $\epsilon < \epsilon^\prime$
	if the assumption \eqref{betaRassump} on the geometry of the right shock
    holds with $\epsilon_2 < \epsilon^\prime$, and so that 
    the assumption \eqref{largestart0} holds
    with $\epsilon_0 < \epsilon^\prime$,
	 then the following bounds hold. With $X = X_R$ as in \eqref{enestrightfv},
	 and with notation as in \eqref{RPXR},
	\begin{equation}
	 E_{X}(t_1) + S_{X}(t_1)  \lesssim E_{X}(t_0)
	 + \int_{t_0}^{t_1} \int_{D^R_t} |\mK_{X, \gamma, P}| + |F| |X\psi|\, dt
	 + R_{P,X}(t_1)
	 \label{rightenestbd}
	\end{equation}
\end{prop}

	\begin{proof}
		The modified multiplier identity \eqref{modifiedstokesident} yields
		\begin{multline}
			\int_{D^R_{t_1}} \mQ^h_P(X, N_{h}^{D^R_{t_1}})
			+ \int_{t_0}^{t_1} \int_{D^R_t}
			-K_{X, m}\, dt
			+
			\int_{t_0}^{t_1} \int_{\Gamma^R_t} \mQ^h_P(X, N_{h}^{R})
			\\
			=
			\int_{D^R_{t_0}} \mQ_P^h(X, N^{D^R_{t_0}}) + \int_{t_0}^{t_1} \int_{D^R_t}
			\mK_{X,\gamma, P}
			+ F X\psi\, dt,
			\label{stokesER}
		\end{multline}
        where $\mQ_P^h$ is defined as in \eqref{mQdef}, where the scalar current $\mK_{X,m}$
        is as in
		\eqref{KPdef} and the modified scalar current $\mK$ defined as in Proposition 
        \ref{effectivemmmink}. 
        Using Lemma \ref{timeslicelemma} to handle the energy-momentum tensor $\mQ$ on the time slices,
        Lemma \ref{spacelikeperturb} to handle $\mQ$ along the shock, and the identity
        \eqref{EXRXnormrln},
        provided $\epsilon$ is taken small enough we have
		\begin{equation}
			E_{X}(t_1)
			\lesssim
		 \int_{D^R_{t_1}} \mQ^h_P(X, N_h^{D^R_t})
		 + \int_{t_0}^{t_1} \int_{\Gamma^R_t} \mQ^h_P(X, N_h^{\Gamma^R})\, dt
		 + R_{P, X}(t_1),
		 \label{almostER}
	 \end{equation}
	 so by the energy identity \eqref{stokesER} we have the bound
		 \begin{equation}
		  E_{X}(t_1) +\int_{t_0}^{t_1} \int_{D^R_t} -K_{X, m}\, dt
			\lesssim
			E_{X}(t_0) + \int_{t_0}^{t_1} \int_{D^R_t}
			\left(\mK_{X,\gamma,  P}
			+ F X\psi\right)\, dt
			+ R_{P, X}(t_1).
		  \label{}
		 \end{equation}
	 From \eqref{KXmformula}, the scalar current is
	 \begin{multline}
	  K_{X, m} =
      2(w'(u)+\pa_u (r(\log r)^{\nu})  (\pa_v\psi)^2
      + \frac{1}{2} \left(w'(u) + \pa_v (r(\log r)^{\nu})
      - 2 \frac{r(\log r)^{\nu}}{v-u}\right)|\nas \psi|^2\\
      =  2\left(w'(u)-\frac{1}{2}(\log r)^{\nu} - 
      \frac{\nu}{2} (\log r)^{\nu-1}\right)  (\pa_v\psi)^2
      + \frac{1}{2} \left(w'(u) -\frac{1}{2}(\log r)^{\nu} + \frac{\nu}{2} (\log
      r)^{\nu-1}\right)|\nas \psi|^2.
	  \label{}
	 \end{multline}
     We now take $\epsilon_0$ so small that if the initial time $t_0$ satisfies
     \eqref{largestart0}, then $r \geq e^{2\nu}$ in $D^R_t \sim \{r \gtrsim t +
     s^{1/2}\}$
     for $t \geq t_0$, which gives the lower
     bound
     \begin{equation}
       \label{}
       -K_{X, m} \gtrsim  \left(-w'(u) + \frac{1}{4}(\log r)^{\nu}\right)
       \left((\pa_v \psi)^2 + |\nas \psi|^2\right),
    \end{equation}
     and the result follows.
	\end{proof}

\subsection{The energy estimates in the region between the shocks}

In this section, we consider the wave equation \eqref{ensecabstractwave} when
$h_C^{-1} = (\mB + \gamma_a^{-1}) + \gamma$, where $\mB$ is the metric
defined in \eqref{mBdef}, $\gamma_a$ collects some small terms verifying the null
condition and where $\gamma$ is a perturbation. This equation reads
\begin{equation}
  \label{ensecabstractwaveC}
  -4\pa_u\left(\pa_v + \frac{u}{vs} \pa_u \right) \psi + \sDelta \psi + \pa_\mu(\gamma_a^{\mu\nu}\pa_\nu \psi) + 
  \pa_\mu(\gamma^{\mu\nu} \pa_\nu \psi) + \pa_\mu P^\mu = F,
 \end{equation}
 where $ \gamma_a^{\mu\nu} = \frac{u}{vs} a^{\mu\nu}$ with  $a^{\mu\nu} \pa_\mu u \pa_\nu u= 0$.
 For some of our applications, we will need to keep track of the structure of the term $F$
 more carefully than in the other regions, and for this reason we will write the above as
 \begin{equation}
   \label{}
      \label{ensecabstractwaveC2}
  -4\pa_u\left(\pa_v + \frac{u}{vs} \pa_u \right) \psi + \sDelta \psi + \pa_\mu(\gamma_a^{\mu\nu}\pa_\nu \psi) + 
  \pa_\mu(\gamma^{\mu\nu} \pa_\nu \psi) + \pa_\mu P^\mu + F_1 =  F_2,
 \end{equation}
 where the terms in $F_2$ will be treated as error terms and where the terms in $F_1$ will
 need to be manipulated in order to close our estimates. See remark \ref{explanationofF1F2}.

The top-order and decay multipliers we use in the central region are
\begin{equation}
    X_T = \left( \frac{u}{s} + \frac{\ap}{{4}s^{1/2}}\right) \pa_u + v\pa_v,
 \qquad X_C = {\left(s +\frac{u}{s}  \right)}  \pa_u + v\pa_v.
 \label{Cvfsdef}
\end{equation}
In terms of the null vectors $\evmB,\eu$ from \eqref{nullfields},
these multipliers take the form
\begin{equation}
 X_T = 
 {\frac{\ap}{{4}s^{1/2}}}  \eu + v \evmB,
 \qquad
 X_C ={s}\eu +  v\evmB.
 \label{Cvfsdefnull}
\end{equation}

By Lemma \ref{multiplierpointing},
both $X_T, X_C$ are future-directed and timelike
with respect to $\mB$ in the region between the shocks
under our assumptions \eqref{betaLassump}-\eqref{betaRassump}.
%
As a result, by the above formulas
at the left shock the norms $|\pa \psi|_{X,+}^2$ from \eqref{plusnorm}
satisfy
\begin{align}
 |\pa \psi|_{X_T, \SL}^2 &\gtrsim \frac{1}{(1+v)(1+s)} (n\psi)^2
 + v (\ell^{\mB} \psi)^2 +  \frac{1}{(1+s)^{1/2}} |\nas \psi|^2,\\
 |\pa \psi|_{X_C, \SL}^2 &\gtrsim \frac{s^{1/2}}{1+v} (n\psi)^2
 + v  (\ell^{\mB} \psi)^2 + {s} |\nas \psi|^2,
 \label{lowerplusnormbounds}
\end{align}
{where the implicit constant in the first estimate depends on the parameter $\ls > 0$.} 


We also note at this point that the second bound in \eqref{pert1}
for $X_T, X_C$ follows from the first one,
\begin{equation}
 X_{T, \mB}^\ell |\gamma|
 = X_{C, \mB}^\ell |\gamma|
 \leq \epsilon \frac{1}{(1+s)^{1/2}} \lesssim \epsilon |X_{T, \mB}^n|
 \leq \epsilon X_{C,\mB}^n.
 \label{proofofperturbmB}
\end{equation}

The top-order energy is
\begin{multline}
 E_{X_T}(t) = \int_{D^C_t} \frac{1}{(1+s)^{1/2}} (\pa_u\psi)^2
 +v(\pa_v\psi)^2 + v |\nas\psi|^2
 +\int_{t_0}^t \int_{\Gamma_{t'}^L} |\pa \psi|^2_{X_T, \SL}\, dS dt'
 \\
  = \int_{D^C_t} |\pa \psi|_{X_T}^2
 +\int_{t_0}^t \int_{\Gamma_{t'}^L} |\pa \psi|^2_{X_T, \SL}\, dS dt'
 \label{topordercentraldef}
\end{multline}
and the lower-order energy is
\begin{multline}
 E_{X_C}(t) = \int_{D^C_t} s(\pa_u\psi)^2
 + v\left((\pa_v\psi)^2+ |\nas \psi|^2\right)
 +\int_{t_0}^t \int_{\Gamma_{t'}^L} |\pa \psi|^2_{X_{C{,+}}}\, dS dt'
 \\
 =\int_{D^C_t} |\pa \psi|_{X_C}^2
 +\int_{t_0}^t \int_{\Gamma_{t'}^L} |\pa \psi|^2_{X_{C{,+}}}\, dS dt'
 \label{}
\end{multline}
We will see that $-K_{X_T, \mB}$ is positive and this generates
an additional time-integrated term in our estimates,
\begin{equation}
 S_{X_T}(t_1) = \int_{t_0}^{t_1} \int_{D^C_t} \frac{1}{(1+v)(1+s)^{3/2}} (\pa_u\psi)^2
 +  |\nas \psi|^2.
 \label{SXTCdef}
\end{equation}
In the estimate for $E_{X_T}$ we will encounter the following positive term on the
timelike boundary,
\begin{equation}
 B_{X_T}(t_1)
 = \int_{t_0}^{t_1} \int_{\Gamma^R_t}
 \frac{1}{(1+v)(1+s)} |n\psi|^2 + \frac{1}{(1+s)^{1/2}} |\nas \psi|^2\, dS dt
 \label{}
\end{equation}
and in the estimate for $E_{X_C}$ we will encounter the following positive term on the
timelike boundary,
\begin{equation}
 B_{X_C}(t_1) =
 \int_{t_0}^{t_1} \int_{\Gamma^R_t}
 \frac{(1+s)^{1/2}}{1+v}
 |n\psi|^2\, dS dt.
\end{equation}
Note that this term does not involve angular derivatives along the timelike
side of $\Gamma^R_t$.

We will prove bounds for the energies that involve the following perturbative
error terms along the time slices and shocks,
	\begin{equation}
	 R_{X, P}(t_1) = \int_{D_{t_0}^{C}} v |P|^2
	 + \int_{D_{t_1}^{C}}v |P|^2
	 + \int_{t_0}^{t_1}\int_{\Gamma^R_t}v|P|^2\,dS dt
	 + \int_{t_0}^{t_1}\int_{\Gamma^L_t}
	 v
	 |P|^2\, dS dt.
	 \label{RPXC}
	\end{equation}
    Our estimate will also involve an error term coming from the scalar current $K_{X, \gamma_a}$
    generated by the $\gamma_a$, the linear part of the metric which verifies the null condition.
    This term will of course not cause any serious difficulties in our upcoming estimates.

    %
%

\begin{prop}[Energy estimates in the central region]
  \label{centralenest}
  Set $\gamma = h^{-1} - \mB^{-1}$. There is a constant
  $\epsilon^\prime > 0$ so that if the first perturbative assumption
	in
	\eqref{pert1} holds with $\epsilon < \epsilon^\prime$ and if
	\eqref{betaLassump}-\eqref{betaRassump} hold
	with $\epsilon_2 < \epsilon^\prime$,
	 then the following bounds hold.
  With notation as in \eqref{RPXC} and \eqref{ensecabstractwaveC2},
\begin{multline}
  E_{X_T}(t_1) + S_{X_T}(t_1)
  + B_{X_T}(t_1)
  {-C_0 \int_{t_0}^{t_1} \int_{D^C_t} F_1 X_T\psi\, dt}
  \lesssim
  E_{X_T}(t_0) + \int_{t_0}^{t_1} \int_{D^C_t}
  \left( |\mK_{X_T, \gamma, P}| {+ |K_{X_T, \gamma_a}|} + |{F_2}| |X_T\psi|\right)\, dt\\
  + \int_{t_1}^{t_2} \int_{\Gamma^R_t}
  (1+v)|\ev^{\mB}\psi|^2 \, dt
  +R_{X_T, P}(t_1),
  \label{topordercentralestimate}
\end{multline}
and
 \begin{multline}
   E_{X_C}(t_1)
   + B_{X_C}(t_1)
   \lesssim
   E_{X_C}(t_0) +
   S_{X_T}(t_1)
  {-C_0 \int_{t_0}^{t_1} \int_{D^C_t} F_1 X_C\psi\, dt}
   + \int_{t_0}^{t_1} \int_{D^C_t}\left(
   |\mK_{X_C, \gamma, P}|{ + |K_{X_C, \gamma_a}|  } +|{F_2}| |X_{C}\psi|\right)\, dt
   \\
   + \int_{t_0}^{t_1} \int_{\Gamma^R_t}
   (1+s)|\slashed{\nabla} \psi|^2  +
   (1+v) |\ev^{\mB}\psi|^2 \,dt
  + R_{X_C, P}(t_1),
   \label{lowordercentralestimate}
 \end{multline}
 where $C_0>0$.
\end{prop}

\begin{proof}
    We use the identity \eqref{modifiedstokesident} with $h^{-1} = (\mB +\gamma_a)^{-1} + \gamma$,
	where $\gamma_a$ collects linear terms verifying the null condition. This gives
  \begin{multline}
    \int_{D^C_{t_1}} \mQ^h_P(X, N_{h}^{D^C_{t_1}})
    + \int_{t_0}^{t_1} \int_{D^C_t}
    -K_{X, \mB}\, dt
    -
    \int_{t_0}^{t_1} \int_{\Gamma^R_t} \mQ^h_P(X, N_{h}^{R})
    +
    \int_{t_0}^{t_1} \int_{\Gamma^L_t} \mQ^h_P(X, N_h^L)
    \\
    =
    \int_{D^C_{t_0}} \mQ_P^h(X, N^{D^C_{t_0}}) + \int_{t_0}^{t_1} \int_{D^C_t}
    \mK_{X, \gamma, P} {+ K_{X, \gamma_a}}
    + F X\psi\, dt.
    \label{stokesEC}
  \end{multline}
  where the modified energy-momentum tensor $\mQ$ is defined in \eqref{mQdef},
	the modified scalar current $\mK$ is as in Proposition \ref{effectivemmmB}
    {and where $F = F_1 - F_2$}.
  The result now follows from the above computations and Lemmas
  \ref{timeslicelemma} (which deals with the energy-momentum tensor along
	the time slices), \ref{spacelikeperturbmB} (which deals with the energy-momentum
	tensor along the spacelike side of the left shock) and
  \ref{timelikeperturbmB} (which deals with the energy-momentum tensor
	along the timelike side of the right shock).
	Specifically, by \eqref{proofofperturbmB} the
	second perturbative
	assumption in \eqref{pert1} holds and so the conclusions
	of these Lemmas hold. As a result, we have the following bounds for
	the energy-momentum tensors on the time slices and along the shocks,
	\begin{multline}
    E_{X_T}(t_1) + B_{X_T}(t_1)
    \lesssim
   \int_{D^C_{t_1}} \mQ^h_P(X_T, N_h^{D^C_t}) -
   \int_{t_0}^{t_1} \int_{\Gamma^R_t} \mQ^h_P(X_T, N_{h}^{R})\, dS dt
   +
   \int_{t_0}^{t_1} \int_{\Gamma^L_t} \mQ^h_P(X_T, N_h^L)\, dS dt
   \\
   + \int_{t_0}^{t_1} \int_{\Gamma^R_t} (1+v) |\evmB\psi|^2\, dS dt
   +R_{X_T, P}(t_1)
   \label{almostECT}
 \end{multline}
  and
  \begin{multline}
    E_{X_C}(t_1) + B_{X_C}(t_1)
    \lesssim
   \int_{D^C_{t_1}} \mQ^h_P(X_C, N_h^{D^C_t}) -
   \int_{t_0}^{t_1} \int_{\Gamma^R_t} \mQ^h_P(X_C, N_{h}^{R})
   +
   \int_{t_0}^{t_1} \int_{\Gamma^L_t} \mQ^h_P(X_T, N_h^L)
   \\
   + \int_{t_0}^{t_1} \int_{\Gamma^R_t} (1+v) |\evmB\psi|^2
   +(1+s) |\nas \psi|^2
   +R_{X_C, P}(t_1).
   \label{almostECD}
 \end{multline}

 By the energy identity \eqref{stokesEC}, we therefore have the bounds
 \begin{multline}
	 E_{X_T}(t_1) + B_{X_T}(t_1)
	 + \int_{t_0}^{t_1} \int_{D^C_t} -K_{X_T, \mB}\, dt
     { + \int_{t_0}^{t_1} \int_{D^C_t} - F_1 X_T\psi\, dt}
	 \lesssim
	 E_{X_T}(t_0)
	 \\
		+ \int_{t_0}^{t_1}\int_{D^C_t}
        |\mK_{X_T, \gamma, P}|  { +|K_{X_T, \gamma_a}| }
        + |{F_2}| |X_T \psi|
	 + \int_{t_0}^{t_1} \int_{\Gamma^R_t} (1+v) |\evmB\psi|^2
	 +R_{X_T, P}(t_1)
	\label{almostECT1}
\end{multline}
and
\begin{multline}
    E_{X_C}(t_1) + B_{X_C}(t_1) +
    {\int_{t_0}^{t_1} \int_{D^C_t} - F_1X_C \psi\, dt}
	\lesssim
	E_{X_C}(t_0)
	\\ + \int_{t_0}^{t_1}\int_{D^C_t}
	|K_{X_C,\mBB}|
	+
	|\mK_{X_C}|{ +|K_{X_C, \gamma_a}| }
	+ |F| |X_C \psi|
	+ \int_{t_0}^{t_1} \int_{\Gamma^R_t}(1+v) |\evmB\psi|^2
	+ (1+s)|\nas \psi|^2
	+R_{X_C, P}(t_1).
 \label{almostEXD1}
\end{multline}

 We now compute the scalar currents $K_{X,\mB}$ with $X = X_T$ and
 $X = X_C$. Both our fields satisfy $X^\ell_{\mB} = X^v = v$,
 and using the earlier formula \eqref{KXgBformula2} for the scalar current, 
 in this case we have
\begin{equation}
  \label{}
  K_{X, \mB} =2\left( \evmB X^n_{\mB} - \frac{1}{vs} X^n_{\mB} \right) (\pa_u\psi)^2
 - \frac{1}{2} \left(1 + 2 \frac{u}{v-u} - \pa_u X^u - 2 \frac{X^u}{v-u}\right)|\nas \psi|^2,
 \label{} 
\end{equation}
where we wrote $\frac{v}{v-u} = 1 + \frac{u}{v-u} $. For $X = X_T = \frac{\ap}{{4}s^{1/2}} n + v\evmB$, this gives
\begin{multline}
  \label{}
  -K_{X_T, \mB} = \frac{3}{{4}} \frac{\ap}{vs} (\pa_u \psi)^2 + \frac{1}{2} |\nas \psi|^2
  + \left( 2 \frac{u}{v-u} - \pa_u X_T^u - 2 \frac{X_T^u}{v-u}\right) |\nas \psi|^2
  \\
  \gtrsim \frac{1}{(1+v)(1+s)^{3/2}} (\pa_u\psi)^2 + |\nas \psi|^2,
\end{multline}
using that $2|u|/(v-u) + |\pa_u X_T^u| + 2|X_T^u|/(v-u)\leq 1/4$, say, in $D^C_T$.
For $X = X_C$, we note that $\evmB X^n_{\mB} - \frac{1}{vs} X^n_{\mB} =  0$ and
so
\begin{equation}
  \label{}
  -K_{X_C, \mB} =  \frac{1}{2} |\nas \psi|^2
  + \left( 2 \frac{u}{v-u} - \pa_u X_C^u - 2 \frac{X_C^u}{v-u}\right) |\nas \psi|^2
  \gtrsim |\nas \psi|^2.
\end{equation}
It follows that
\begin{equation}
  \label{}
  S_{X_T}(t_1) \lesssim \int_{t_0}^{t_1} -K_{X_T, \mB}\, dt,
  \qquad
  S_{X_C}(t_1) \lesssim \int_{t_0}^{t_1} -K_{X_C, \mB}\, dt,
\end{equation}
and the result follows.

\end{proof}

\subsection{The energy estimates in the region to the left
of the left shock}

In the region to the left of the left shock it will suffice
to use the multiplier
\begin{equation}
 X = X_L = uf(u) \pa_u + vf(v) \pa_v = uf(u) n + vf(v) \evm, \quad \text{ where }
 f(z) = \log z (\log \log z)^\alpha.
 \label{Lvfsdef}
\end{equation}
for $\alpha > 1$. For our applications we will take $1 < \alpha < 3/2$
but for the below argument the upper bound is irrelevant.
It is clear that $X_L$ is timelike and future-directed with respect
to the Minkowski metric in the region to the left of the left shock
since with our conventions $u$ is positive there and $\pa_u, \pa_v$
are future-directed.


The energies are
\begin{equation}
 E_{X}(t)
 = \int_{D_t^L} u f(u) |n\psi|^2 + v f(v) |\evm \psi|^2
 + vf(v) |\nas \psi|^2.
 \label{EXfdef}
\end{equation}
This will enter our calculations
with an additional positive term on the left shock
(which is timelike with respect to $m$)
\begin{equation}
 B_{X}(t_1) = \int_{t_0}^{t_1} \int_{\Gamma^L_t}
 \frac{1}{1+v} f(s^{1/2}) |n\psi|^2
 + (1+s)^{1/2} f(s^{1/2}) |\nas \psi|^2\, dS dt
 \label{BXDleft}
\end{equation}

We will prove bounds involving the following perturbative error term,
\begin{align}
 R_{X, P}(t_1)
 &=
 \int_{D^L_{t_0}} X^\ev_m |P|^2 + \int_{D^L_{t_1}} X^\ell_m |P|^2
 + \int_{t_0}^{t_1} \int_{\Gamma^L_t} \left( X_m^\ev + (1+v)(1+s)^{1/2} X^n_m\right)
 |P|^2\,dS dt\\
 &\sim \int_{D^L_{t_0}} vf(v) |P|^2 + \int_{D^L_{t_1}} vf(v) |P|^2
  + \int_{t_0}^{t_1} \int_{\Gamma^L_t} vf(v) |P|^2\,dS dt
 \label{RPDL}
\end{align}

To ensure that the second condition in \eqref{pert1} holds, we assume
that $\gamma$ satisfies the estimate
\begin{equation}
	(1+v)(1+s)^{1/2}\frac{(\log_+s)^{\alpha-1}}{(\log \log_+s)^{\alpha}}
	 |\gamma| \leq \epsilon,
 \label{pert1left}
\end{equation}
 For our applications it is this condition that forces us to
take $\alpha < 3/2$. We note that this condition is stronger than
the first bound in \eqref{pert1}.

Then we have
\begin{prop}[Energy estimates in the leftmost region]
  \label{leftengen}
	  Set $\gamma = h^{-1} - m^{-1}$
		There is a constant
	  $\epsilon^\prime > 0$ so that if
		the
		assumption
		\eqref{pert1left} holds with $\epsilon < \epsilon^\prime$
		and \eqref{betaLassump} holds with $\epsilon_2 < \epsilon^\prime$,
		 then the following bound holds.
	  With notation as in \eqref{RPDL},
   \begin{multline}
    E_{X_L}(t_1) + B_{X_L}(t_1)
    \lesssim
    E_{X_L}(t_0) + \int_{t_0}^{t_1}\int_{D^L_{t}}
    \left( |\mK_{X_L, \gamma, P}| + |F| |X_L\psi|\right)\\
    + \int_{t_0}^{t_1}\int_{\Gamma^L_t}
		vf(v) |\evm\psi|^2
    \, dSdt
    + R_{P, X_L}(t_1).
    \label{decayleftenergybd}
  \end{multline}
\end{prop}

\begin{proof}

	If \eqref{pert1left} holds, then with $X = X_L = X^\evm \evm + X^n_m n,$
	\begin{multline}
	 X^\ell_{m} |\gamma| \leq (1+v)(1+s) (\log(1+s))^\alpha |\gamma|
	 = (1+s)^{1/2} \log(1+s)\left( (1+v)(1+s)^{1/2} (\log(1+s))^{\alpha-1} |\gamma|\right)
	 \\
	 \leq \epsilon (1+s)^{1/2} \log(1+s) (\log \log(1+s))^{\alpha}
	 \lesssim \epsilon X^n_{m},
	 \label{}
	\end{multline}
	so both bounds (and in particular the second bound) in \eqref{pert1} holds for $X = X_L$.

  Since $X_L$ satisfies
  $X_L^{ r}|_{r =0} = 0$, we can apply \eqref{integralidentmodified}
  which gives
  \begin{multline}
   \int_{D^L_{t_1}} \mQ^h_P(X_L, N^{D_{t_1}}_h) + \int_{t_0}^{t_1} \int_{D^L_t} -K_{X_L, m}
   + \int_{t_0}^{t_1} \int_{\Gamma^L_t} -\mQ^h_P(X_L, N^L_h)\\
   = \int_{D^L_{t_0}} \mQ^h_P(X_L, N_h^{D_{t_0}})
   + \int_{t_0}^{t_1} \int_{D^L_t}
   \mK_{X_L} + F X_L\psi.
   \label{usestokesleft}
 \end{multline}
 The result now follows from the above computations and Lemmas
  \ref{timeslicelemma}, \ref{spacelikeperturb} and
  \ref{timelikeperturb}. We have just shown that the hypotheses of these results
	 hold, and so we have the following bounds for the energy-momentum tensor
	 along the time slices and the left shock,
 \begin{multline}
  E_{X_L}(t_1) + B_{X_L}(t_1) \lesssim
  \int_{D^L_{t_1}} \mQ^h_P(X_L, N^{D^L_{t_1}}_h)
  -\int_{t_0}^{t_1} \int_{\Gamma^L_t}
  \mQ^h_P(X_L, N^{L}_h) \, dt\\
  +
  \int_{t_0}^{t_1} \int_{\Gamma^L_t}
  X_L^v|\evm \psi|^2
  +R_{P, X_L}(t_1).
  \label{}
\end{multline}

From the identity \eqref{usestokesleft} we therefore have
\begin{multline}
  E_{X_L}(t_1) + B_{X_L}(t_1)+\int_{t_0}^{t_1}-K_{X_L,m}
   \lesssim
   E_{X_L}(t_0)
   \\
  + \int_{t_0}^{t_1} |\mK_{X_L,\gamma, P}|
  +|F| |X_L \psi|
  +
  \int_{t_0}^{t_1} \int_{\Gamma^L_t}
  X^v_f|\ev^m \psi|^2\, dt
  +R_{P, X_L}(t_1).
 \label{almostLXD}
\end{multline}
It remains to compute the scalar current $K_{X_L,m}$.
Since $\pa_u X^v_f = \pa_v X^u_f = 0$, \eqref{KXmformula}
gives
  \begin{equation}
  -K_{X_L, m}= \frac{1}{2}\left( \frac{2}{v-u}\left(X^v_f - X_L^u\right)-\pa_u X_L^u - \pa_v X_L^v\right)
   |\nas \psi|^2.
   \label{}
  \end{equation}

	Let $\tilde f=zf(z)$. Then
	\begin{align}
	 \frac{2}{v-u}\left(X^v_f- X^u_f\right) -  \left( \pa_v X_L^v + \pa_u X^f_u\right)
	 &=\frac{2}{v-u}\int_u^v \tilde f'(z) \,dz
	 - \tilde f'(v) - \tilde f'(u)\\
	 &=\frac{1}{v-u}\int_u^v \left(2\tilde f'(z)-\tilde f'(v)-\tilde f'(u)\right) \,dz \\
   &=\frac{1}{v-u}\int_u^v  \left(\int_u^z \tilde f''(\tau)\,d\tau - \int_z^v \tilde f''(\tau)\,d\tau)\right) \,dz \\
   &=\frac{1}{v-u}\int_u^v  \left(\tilde f''(\tau) (v-\tau)- \tilde f''(\tau))(\tau-u)\right) \,d\tau
\label{}
	\end{align}
	where in the last line we interchanged the limits of integration and performed an explicit integration with respect to $z$.
	Furthermore,
	\begin{align}
	\int_u^v  \tilde f''(\tau) (v+u-2\tau) \,d\tau
	&=\int_u^{\frac{v+u}2} \tilde f''(\tau) (v+u-2\tau) \,d\tau + \int_{\frac{v+u}2}^v \tilde f''(\sigma) (v+u-2\sigma) \,d\sigma\\
	&=\int_u^{\frac{v+u}2} \tilde f''(\tau) (v+u-2\tau) \,d\tau {+} \int_{\frac{v+u}2}^{{u}} \tilde f''(v+u-\tau) (v+u-2\tau) \,d\tau
	\\
	&=\int_u^{\frac{v+u}2} \left(\tilde f''(\tau) - \tilde f''(v+u-\tau)\right) (v+u-2\tau) \,d\tau \\
	&=-\int_u^{\frac{v+u}2} \left(\int_{\tau}^{v+u-\tau} \tilde f'''(\rho)\,d\rho\right) (v+u-2\tau) \,d\tau.
	\end{align}
	To compute the sign of $\tilde f'''(z)$ we observe
	$$
	\tilde f''(z)=\frac 1z (\log\log z)^\alpha +\frac 1z O\left((\log\log z)^{\alpha-1}\right)
	$$
	where we write $f_1 = O(f_2)$ if $|f_1| \lesssim |f_2|$ and $|f_1'| \lesssim |f_2'|$.
	As a result,
	$$
	\tilde f'''(z)=-\frac 1{z^2} (\log\log z)^\alpha+\frac 1{z^2} O\left((\log\log z)^{\alpha-1}\right)\leq 0,
	$$
	and it follows that
	$$
	-K_{X_L,m}\geq 0
	$$
\end{proof}

\subsection{The Morawetz estimate}
\label{genmorsec}
In order to close our estimates in the region to the left of the left shock
we use the spacelike multiplier
\begin{equation}
 X_M = \left(g(r) + 1\right)
 (\pa_v - \pa_u),
 \qquad g(r) = (\log (1+r))^{1/2} f(\log (1 + r))
 \label{Gdef}
\end{equation}
where $f(z) = \log z (\log\log z)^\alpha$.

The reason for using this multiplier is that the scalar current
$K_{X_M, m}$ is positive-definite and this gives a time-integrated
bound for weighted derivatives.
Since $X_M$ is not timelike, after multiplying the wave equation by $X_M\psi$
and integrating by parts, the terms on the time slices are not
positive-definite, nor are those along the shock. However, $X_M$ has been chosen so that those terms can
be controlled by the energies $E_{X_L}(t)$, see \eqref{importantgbd}.

\begin{prop}
  \label{morawetzlem}
  Suppose that the assumptions \eqref{betaLassump}
  hold and that with $\gamma = h^{-1} - m^{-1}$, we have
 $\lim_{r\to 0} |\gamma^{rr}|\leq \tfrac{1}{4}$.
 There is a constant $\epsilon^\prime > 0$ so that if the bound
 \eqref{pert1left} holds with $\epsilon < \epsilon^\prime$, then with
 $g$ as in \eqref{Gdef}
 and with $E_{X_L}, B_{X_L}$ defined as in
 \eqref{EXfdef} and \eqref{BXDleft} and $R_{X_L,P}$ as in \eqref{RPDL},
 we have
\begin{multline}
 \int_{t_0}^{t_1} \int_{D^L_t}
 g'(r) \left( (\pa_u\psi)^2 + (\pa_v\psi)^2\right) +  \frac{g(r) + 1}{r} |\nas \psi|^2
 + \int_{t_0}^{t_1} (1+t)\lim_{r \to 0} \left| \frac{\psi}{r}\right|^2\, dt
 \\
 \lesssim
 E_{X_L}(t_1) + E_{X_L}(t_0)
 + B_{X_L}(t_1)
 +\int_{t_0}^{t_1} \int_{D^L_t}
 |\mK_{X_M, \gamma, P}|
 +|F| |X_M\psi|^2 \, dt
 \\
 +\int_{t_0}^{t_1} \int_{\Gamma^L_t}v f(v)  (\pa_v\psi)^2\,
 dt
 + R_{X_L, P}(t_1).
 \label{morawetzbd}
\end{multline}
\end{prop}
\begin{proof}
  From \eqref{divthm2} we have the identity
\begin{multline}
 \int_{t_0}^{t_1} \int_{D^L_t} -K_{X_M,m}
 + \int_{t_0}^{t_1} \lim_{r\to 0}X^r_M h^{rr} \left|\frac{\psi}{r}\right|^2
 =
  \int_{D^L_{t_0}} \mQ^h_P(X_M, N_h^{D^L_{t_0}})
 - \int_{D^L_{t_1}} \mQ^h_P(X_M, N_h^{D^L_{t_1}})\\
 + \int_{t_0}^{t_1}\int_{\Gamma^L_t} \,\mQ_P^h(X_M, N_h^{\Gamma})
 + \int_{t_0}^{t_1}\int_{D^L_t} F X_M\psi - \mK_{X_M, \gamma, P}.
 \label{usedivmor}
\end{multline}
We note that $X^r_M = X_Mr = 2 g(r) + 1$ with $g(0) = 0$,
and that by our assumptions, $\lim_{r\to 0}|h^{rr} - 1| = \lim_{r\to 0}|\gamma^{rr}|
\leq \frac{1}{4}$. It follows that
\begin{equation}
 \int_{t_0}^{t_1} \lim_{r \to 0}
 \left|\frac{\psi}{r}\right|^2 \lesssim
 \int_{t_0}^{t_1} \lim_{r\to 0} X^r_M h^{rr}\left|\frac{\psi}{r}\right|^2.
 \label{zeromor}
\end{equation}

We now bound the terms appearing on the right-hand side
of \eqref{usedivmor}. For the integrals over the time slices, we
first use the identity \eqref{generalEMdecomp} which gives
\begin{multline}
 |Q^m(X_M, N^{D_t^L}_m)|
 \lesssim |X_M^v| (\pa_v\psi)^2 + |X_M^u| (\pa_u\psi)^2
 + |X_M| |\nas \psi|^2
 \\
 \lesssim
 (g(r) +1 )\left( (\pa_u \psi)^2 + (\pa_v \psi)^2 + |\nas \psi|^2\right)
 \label{}
\end{multline}
By the definition of $g$ and the fact that $u \gtrsim s^{1/2}$ in $D^L_t$,
{where the implicit constant depends on the constants $\ls, \rs$ from \eqref{betaLassump}-
\eqref{betaRassump}.} 
we have the bound
\begin{align}
 &g(r)+1 = (\log (r+1))^{1/2}f((\log (r+1))^{1/2}) +1 \lesssim (1+s)^{1/2}f(s^{\frac 12}) \lesssim u f(u)
 = X_L^u,
 \label{importantgbd}
\end{align}
with $X_L$ as in \eqref{Lvfsdef}.
Since clearly $g(r) \leq X_L^v$, using \eqref{timesliceperturbbd},
 our perturbative assumptions,
and the definition of $E_{X_L}$ from \eqref{EXfdef},
we have the bound
\begin{equation}
 \int_{D^L_t} |\mQ_P^h(X_M, N_h^{D_t^L})|
 \lesssim E_{X_L}(t)
 + \int_{D^L_t} |X^v_L| |P|^2.
 \label{mortimeslice}
\end{equation}

We now deal with the integral over the shock $\Gamma^L$ appearing
in \eqref{usedivmor}.
By \eqref{timelikelower}, we have 
 \begin{multline}
  |Q^m(X_M, N_m^{L})|
	 \lesssim
	\frac{1}{(1+v)(1+s)^{1/2}} (g(r)+1) |\pa_u\psi|^2 +  (g(r) + 1)|\pa_v \psi|^2\\
	+ \left(\frac{1}{(1+v)(1+s)^{1/2}} + 1 \right)(g(r)+1) |\nas \psi|^2
	+\left(1 + (1+v)(1+s)^{1/2}\right) (g(r)+1)|P|^2,
  \label{}
 \end{multline}
 and
 after using \eqref{importantgbd} this gives
 \begin{equation}
  |Q^m(X_M, N_m^L)| \lesssim \frac{1}{(1+v)(1+s)^{1/2}} X^u_f |\pa_u\psi|^2
	+ X^u_f |\nas \psi|^2
	+ X^v_f |\pa_v\psi|^2
	+ (1+v)(1+s)^{1/2}X^u_f |P|^2.
  \label{}
 \end{equation}
 From the definition of the boundary term $B_{X_L}$ from \eqref{BXDleft}
 and using our perturbative assumptions, we therefore have
\begin{equation}
 \int_{t_0}^{t_1} \int_{\Gamma^L_t} |\mQ^h_P(X_M, N^L_m)|
 \lesssim
 B_{X_L}(t_1) + \int_{t_0}^{t_1} \int_{\Gamma^L_t} vf(v)|\pa_v \psi|^2
  + \left(|X_L^v| + (1+v)(1+s)^{1/2}|X_L^u|\right) |P|^2.
 \label{morbdy}
\end{equation}

It remains to compute the scalar current $K_{X_M, m}$.
Recall from \eqref{KXmformula} that
\begin{equation}
 -K_{X_M, m} =
 -2\pa_v X_M^u (\pa_u\psi)^2
 -
  2\pa_u X_M^v  (\pa_v\psi)^2
 - \frac{1}{2} \left( \pa_u X_M^u +\pa_v X_M^v
 - 2 \frac{X_M^v-X_M^u}{v-u}\right)|\nas \psi|^2.
 \label{KXMformula2}
\end{equation}
Since $\pa_v r = \frac{1}{2} = -\pa_u r$, we have
\begin{equation}
 \pa_v X_M^u = \pa_v X_M^v = -g'(r),
 \label{}
\end{equation}
and
\begin{equation}
 \pa_u X_M^u + \pa_v X_M^v =  g'(r)
 \qquad
 \frac{2}{v-u}(X^v_M - X^u_M) = \frac{2}{r}(g(r) + 1)
 \label{}
\end{equation}
Therefore, the coefficient of $\frac{1}{2}|\nas \psi|^2$ in \eqref{KXMformula2} is
\begin{equation}
 \frac{2}{v-u}(X^v_M - X^u_M) - \pa_v X^v_M - \pa_u X^u_M
 = \frac{g(r) + 2}{r}  +  \frac{g(r)}{r} - g'(r).
 \label{}
\end{equation}
Since $g(0) = 0$, $\lim_{r \to 0^+} r g'(r) =  0$ and $g'' < 0$, we have
\begin{equation}
	\frac{g(r)}{r} - g'(r) = \frac{1}{r} \int_0^r g'(z)\, dz - g'(r)
	=-\frac{1}{r}\int_0^r z g''(z)\, dz \geq 0,
 \label{}
\end{equation}
and so we have the lower bound
\begin{multline}
 -K_{X_M, m} = 2 g'(r) \left((\pa_u\psi)^2 + (\pa_v \psi)^2\right)
 + \left(\frac{g(r) + 2}{r}  +  \frac{g(r)}{r} - g'(r) \right) |\nas \psi|^2\\
 \geq 2 g'(r) \left((\pa_u\psi)^2 + (\pa_v \psi)^2\right)
 + \frac{g(r) + 2}{r}|\nas \psi|^2.
\end{multline}
%
%
%
%
%
Combining this with \eqref{zeromor}, \eqref{mortimeslice}, and \eqref{morbdy} gives the result.

\end{proof}

\section{The nonlinear equations and the main theorem}
\label{mainthmsec}
We start by recording the system of equations and boundary conditions
we are considering.

\subsection{The equations for the perturbations in each region}
With $H^\alpha$ defined as in \eqref{Hdefintro},
We consider the wave equation
\begin{equation}
  \pa_\alpha H^\alpha(\pa \Phi) = 0,
 \label{contpf}
\end{equation}
in regions $D^L, D^C, D^R$, subject to the boundary conditions
\begin{equation}
 [H^\alpha(\pa\Phi)] \zeta_\alpha = 0, \qquad [\Phi] = 0
 \label{rhpf}
\end{equation}
across the shocks $\Gamma^L, \Gamma^R$,
where $\zeta = \zeta_\alpha dx^\alpha$ is a one-form whose null space at each
point $(t,x)$ on the shock $\Gamma$ is the tangent space $T_{(t,x)}\Gamma$,
and where in each region $\Phi$ is a
perturbation of the model shock profile $\sigma = \frac{1}{r}\Sigma$
given in \eqref{introsigmadef},
\begin{equation}
 \Phi = \phi + \sigma = \phi + \begin{cases}\frac{u^2}{2rs} \qquad \text{ in } D^C,\\
 0, \qquad \text{ in } D^L, D^R.\end{cases}
 \label{decomp}
\end{equation}

By Lemma \ref{higherorderexterioreqns} with $|I| = 0$,
in the exterior regions $D^L, D^R$, the variables
$\psi_L = r\phi_L, \psi_R = r\phi_R$ satisfy the following quasilinear perturbation
of the Minkowskian wave equation,
\begin{equation}
	(-4\pa_u\pa_v \psi_A + \sDelta \psi_A)  + \pa_\mu(\gamma^{\mu\nu}\pa_\nu \psi_A)
	= F_A,
 \label{waveext0}
\end{equation}
for $A = L, R$,
where $\gamma = \gamma(\pa (\psi_A/r))$. By
Lemma \ref{centraleqnlowestprop}, in the region between the shocks $D^C$,
with notation as in \eqref{centraleqnlowest}-\eqref{FSigmaformula},
$\psi_C = r\phi_C$ satisfies the wave equation
\begin{equation}
 -4 \pa_u \left( \pa_v + \frac{u}{vs} \pa_v \right)\psi_C
 + \sDelta \psi_C + \pa_\mu \left(\gamma_a^{\mu\nu}\pa_\nu \psi_C\right)
 + \pa_\mu\left( \gamma^{\mu\nu}\pa_\nu \psi_C\right)
 + \pa_\mu P^\mu = F  + F_{\Sigma},
 \label{waveint0}
\end{equation}
where 
\begin{equation}
  \label{gammaadef}
  \gamma_a^{\mu\nu} =\frac{u}{vs} a^{\mu\nu}
\end{equation}
verifies the null condition \eqref{intronullcondn}
and is expected to be better-behaved than the other linear terms above.
The quantity
$F_{\Sigma}$ collects the error terms involving the model shock
profile $\Sigma = \tfrac{u^2}{2s}$ alone.

For some of our applications, we will use
that \eqref{waveext0} can be written in the form
\begin{equation}
 -4\pa_u\evm \psi_A = -\sDelta \psi_A  - \pa_\mu((1+v)^{-1}Q^\mu(\pa\psi_A, \pa \psi_A))
 + F'_A,
 \label{waveext2}
\end{equation}
where $Q^{\mu}(\pa\psi_A, \pa \psi_A) =
Q^{\mu\nu\delta}\pa_\nu\psi_A \pa_\delta \psi_A$ for
smooth functions
$Q^{\mu\nu\delta}$ satisfying the symbol condition
\eqref{minksymb}. Here, $F'_A = F_A$ up to lower-order terms
with rapidly decaying coefficients.

Similarly, we can write \eqref{waveint0} in the form
\begin{equation}
 -4 \pa_u \evmB \psi_C = - \sDelta \psi_C -
 \pa_\mu((1+v)^{-1}Q^\mu(\pa\psi_A, \pa \psi_A))
 -
 \pa_\mu\left( \frac{u}{vs} a^{\mu\nu}\pa_\nu \psi_C\right)
 + F_A'.
 \label{waveint2}
\end{equation}
Recall that $\evmB =\pa_v + \tfrac{u}{vs}\pa_u$.

\subsection{The boundary conditions along the timelike
sides of the shock}
Along the left shock, By Lemma \ref{tlbc} and
\eqref{YL0}, the Rankine-Hugoniot conditions imply the following
equation which plays the role of a boundary condition for $\psi_L$ (recall
that the left shock is spacelike with respect to the metric in the leftmost
region)
\begin{equation}
 Y_L^-(\pa \psi_L) \psi_L = Y_L^+(\pa \psi_C)\psi_C + G
 \label{introbcL}
\end{equation}
where
\begin{align}
Y_L^-(\pa \psi_L)\psi_L
&=  \pa_v \psi_L + \frac{1}{v} Q_L(\pa \psi_L, \pa \psi_L),
\label{YLdef1}
\\
Y_L^+(\pa \psi_R)\psi_C
&=\left(\pa_v + \frac{1}{vs} \pa_u \right)\psi_C
+ \frac{1}{v} Q_C(\pa \psi_C, \pa \psi_C),
 \label{YLdef2}
\end{align}
where the $Q$ are quadratic nonlinearities and the error term
$G$, which consists of quadratic terms verifying a null condition,
higher-order
nonlinearities and rapidly-decaying inhomogeneous terms,
is given explicitly in \eqref{Gstructure} and \eqref{Gexpression}.
Similarly, along the right shock, we have the following
boundary condition
\begin{equation}
 Y_R^-(\pa \psi_C) \psi_C = Y_R^+(\pa \psi_R, B^R) + G,
 \label{introbcR}
\end{equation}
with
\begin{align}
 Y_R^-(\pa \psi_C) \psi_C
 &=\left(\pa_v + \frac{1}{vs} \pa_u \right)\psi_C
+ \frac{1}{v} Q_C(\pa \psi_C, \pa \psi_C),
\label{YRdef1}
\\
Y_R^+(\pa \psi_R, B^R) \psi_R
&=
\pa_v \psi_R + \frac{1}{v} Q_R(\pa \psi_R, \pa \psi_R)
 \label{YRdef2}
\end{align}
We note that $Y_L^+\psi_C = Y_R^- \psi_C$.

\subsubsection{The higher-order wave equations}
In the regions outside the shocks, we will work in terms of the quantities
\begin{equation}
 \psi^I_A = r Z^I\phi_A,
 \label{}
\end{equation}
where $Z^I$ denotes a product of the fields in \eqref{Zdef}
and where $A = R$ in the rightmost region and $A = L$ in the leftmost region.
In the region between the shocks we will work in terms of the quantities
\begin{equation}
 \psi^{I}_C =  \ZB^I(r \phi_C),
 \label{}
\end{equation}
where $\ZB^I$ denotes a product of the fields in \eqref{ZBdef}.

In the exterior regions $D^A$ for $A = L, R$, $\psi^I_A = r Z^I \phi$ satisfies
\begin{align}
    -4\pa_u\pa_v \psi^I_A + \sDelta \psi^I_A
    + \pa_\mu(\gamma^{\mu\nu}\pa_\nu \psi^I_A)
 + \pa_\mu P^\mu_{I,A} = F_{I, A},
 \label{higherordereqnext}
\end{align}
where the quantities in the above expression are given in
Lemma \ref{higherorderexterioreqns}.
In the region between the shocks $D^C$, $\psi^I_C = Z^I_{\mB} (r\phi)$ satisfies
\begin{multline}
    -4\pa_u \left(\pa_v + \frac{u}{vs} \pa_u \right) \psi^I_C + \sDelta \psi^I_C
    +
 \pa_\mu\left( \frac{u}{vs} a^{\mu\nu}\pa_\nu \psi^I_C\right) +\pa_\mu(\gamma^{\mu\nu}\pa_\nu \psi^{I}_C)
 + \pa_\mu P^\mu_{I, C} +
 \pa_\mu P^\mu_{I, null}+  {F_{I, \mB}^1} \\= F_{I,C} + F_{\Sigma, I} + {F_{\mB, I}^2},
 \label{higherordereqcentral}
\end{multline}
where the above quantities are given in Lemma \ref{higherordereqncentral}.

Roughly speaking,
in each region
$\gamma$ behaves like $\frac{1}{1+v} \pa \psi^A$.
The nonlinear commutation errors $P_{I, A}$ behaves like a
sum of terms$\frac{1}{1+v} \pa Z^{I_1} \psi^A \cdot \pa Z^{I_2} \psi^A$
for $\max(|I_1|,|I_2|) \leq |I|-1$.
The quantities $a^{\mu\nu}$ verify the null condition
$a^{uu} = 0 $ and are expected to be better behaved
than the other linear terms in \eqref{higherordereqncentral}.
When we commute the equation with our fields, this term generates
additional errors, which are collected in the quantity $P_{I, null}$.
This current satisfies the bounds \eqref{borderlinebd0}-\eqref{borderlinebd2}.

 The quantities $F_{I,A}$ collect various nonlinear
 error terms which behave roughly like
 $\frac{1}{(1+v)^2} \pa Z^{I_1}\psi^A\cdot \pa Z^{I_2} \psi^A$ for
 $\max(|I_1|, |I_2|) \leq |I|$. {The quantities $F_{\mB, I}^1, F_{\mB, I}^2$} collect the error
 terms generated by commuting our fields with the linear part of the equation
 in the central region (note that the fields $\sBo, \sBt$ do not commute with
 $\sDelta$, and that $\sBt$ only approximately commutes with the radial
 part $\pa_u(\pa_v + \frac{u}{vs} \pa_v)$). {The quantity $F_{\mB, I}^2$ can be treated as an error term,
 but $F_{\mB, I}^2$ is slightly too large for this. However, it turns out
 that (see Lemma \ref{FImBibplemma} ), this term can indeed be handled after integration
by parts.}
The quantity $F_{\Sigma, I}$ collects the ``inhomogeneous'' error terms, which involve
 only the model shock profile $\Sigma$ and its derivatives.

\subsection{The definitions of the energies}

We fix parameters $N_L, N_C, N_R, \epsilon_L, \epsilon_{C, D},\epsilon_{C, T}, \epsilon_R, \mu,
{\nu}, \alpha$ satisfying
\begin{align}
	N_L \leq N_C - 6 &\leq N_R - 8,
	\quad N_R \geq 30,
	\quad
	\epsilon_R \leq \epsilon_{C,T}^2 \leq \epsilon_{C, D}^4 \leq \epsilon_L^6,
	\label{epsparameters}\\
                     &\nu \geq N_C \quad
                     \mu \geq \max(2\nu, 2N_C+3/2) \quad 1 < \alpha < 3/2.
 \label{parameters}
\end{align}
We remark that if we only needed to close estimates in the rightmost region, for our
arguments it would suffice to take $\mu \geq 6$. We only need to take it larger because we
need to control some error terms generated along the timelike side of the right 
shock. 
We now define the energies we will use to control the solution.

The energies in the region to the right of the right shock are
\begin{equation}
 \mathcal{E}_{N_R}^R(t)
 = \sum_{|I| \leq N_R} E_I^R(t) + S_I^R(t)
 \label{ERdef0}
\end{equation}
where the energies $E_I^R$ and time-integrated quantities $S_I^R$ are given by
\begin{align}
 E_I^R(t_1) &= \int_{D_{t_1}^R}
 \left(1 + |u|\right)^\mu|\pa \psi_R^I|^2
 +(1 +|u|^\mu+ r(\log r)^{\nu}) \left( |\pa_v \psi_R^I|^2 + |\nas \psi_R^I|^2\right)
 \\
 &\qquad
 +\int_{t_0}^{t_1} \int_{\Gamma^R_t}
  (1+|u|^\mu + r(\log r)^{\nu})|\pa_v \psi_R^I|^2
 + (1+|u|)^\mu |\nas \psi_R^I|^2
+ \frac{(1+|u|)^\mu}{(1+v)(1+s)^{1/2}} |\pa_u \psi_R^I|^2
\, dS dt\\
	 \label{ERdef}
\\
 S_I^R(t_1) &= \int_{t_0}^{t_1} \int_{D^R_t}
 \left(1+|u|^{\mu-1} + (\log r)^{\nu-1}\right) \left(|\pa_v \psi_R^I|^2 +  |\nas \psi_R^I|^2
 \right)\, dt.\label{SRdef}
\end{align}
We remind the reader that all integrals over time slices are taken with respect to
the measure $\frac{1}{r^2} dxdt = dudv d\sigma_{\S^2}$ and the integrals over the shocks
are taken with respect to the corresponding surface measure.

In the central region (see remark \ref{differentmultiplierscentral}), we will work
in terms of the quantities
\begin{equation}
 \mathcal{E}_{N_C}^C(t_1) =
 \mathcal{E}_{N_C, T}^C(t_1) +
 \mathcal{E}_{N_C-1, D}^C(t_1) +
 \mathcal{E}_{N_C-2, D}^C(t_1),
 \label{ECdef}
\end{equation}
where
\begin{align}
  \mathcal{E}_{N_C, T}^C(t_1) &= \sum_{|I| \leq N_C} E_{I, T}^C(t_1) + S_I^C(t_1)
	+ B_{I}^C(t_1)\\
	\mathcal{E}_{N_C-1, D}^C(t_1) &= \sum_{|I| = N_C-1} E_{I,D}^C(t_1)\\
  \mathcal{E}_{N_C-2, D}^C(t_1) &= \sum_{|I| \leq N_C-2} E_{I,D}^C(t_1),
 \label{}
\end{align}
where the top-order energies $E_{I, T}^C$ and the time-integrated
quantities $S_I^C$ are given by
\begin{align}
 E_{I, T}^C(t_1)
 &= \int_{D^C_{t_1}} \frac{1}{(1+s)^{1/2}} (\pa_u \psi_C^{I})^2 + (1+v) ((\evmB \psi^{I}_C)^2 + |\nas \psi^{I}_C|^2)
 \\
 &+ \int_{t_0}^{t_1} \int_{\Gamma^L_{t}}
 \frac{1}{(1+v)(1+s)} (\pa_u\psi^{I}_C)^2
 + (1+v) \left( \evmB \psi^{I}_C\right)^2
 + \frac{1}{(1+s)^{1/2}} |\nas \psi^{I}_C|^2\, dS dt
 \\
 &+ \int_{t_0}^{t_1} \int_{\Gamma^R_{t}}
\frac{1}{(1+v)(1+s)} (\pa_u\psi^{I}_C)^2 +
 \frac{1}{(1+s)^{1/2}} |\nas \psi^{I}_C|^2\, dS dt,
 \label{EtopC}
\end{align}
\begin{equation}
 S_I^C(t_1) =
 \int_{t_0}^{t_1} \int_{D^C_t} \frac{1}{(1+v)(1+s)^{3/2}} (\pa_u\psi^{I}_C)^2 + |\nas \psi^{I}_C|^2\, dt,
 \label{EtopC0}
\end{equation}
the quantities $B_{I}^C(t_1)$ are given by
\begin{equation}
 B_I^C(t_1) = \int_{t_0}^{t_1} \int_{\Gamma^R_t} (1+v) |\evmB \psi^I_C|^2\,dS dt,
 \label{}
\end{equation}
and the lower-order energies (the ``decay'' energies) $E_{I, D}^C$ are given by
\begin{align}
 E_{I, D}^C(t_1)
 &=
 \int_{D^C_{t_1}}
(1+s)(\pa_u \psi^{I})^2
+ (1+v)\left( (\evmB \psi^I_C)^2 + |\nas \psi^I_C|^2\right)\\
 &+\int_{t_0}^{t_1} \int_{\Gamma^L_{t}}
 \frac{(1+s)^{1/2}}{1+v} (\pa_u\psi^{I}_C)^2
 + (1+v)(\evmB \psi^I_C)^2
 + (1+s) |\nas \psi^I_C|^2\, dS dt\\
 &+
 \int_{t_0}^{t_1} \int_{\Gamma^R_{t}}
 \frac{(1+s)^{1/2}}{1+v} (\pa_u\psi^{I}_C)^2\, dSdt.
 \label{loworderenergydef}
\end{align}

In the left-most region the energies are
\begin{equation}
 \mathcal{E}_{N_L}^L(t_1) =
 \sum_{|I| \leq N_L} E_{I}^L(t_1) + M_I(t_1) + B_I^L(t_1)
 \label{ELdef}
\end{equation}
with
\begin{align}
  E_{I}^L(t_1)
  &= \int_{D^L_{t_1}}
  uf(u) (\pa_u \psi_L^I)^2+
  v f(v) ((\pa_v \psi_L^I)^2 + |\nas \psi_L^I|^2)
  \\
	& +
  \int_{t_0}^{t_1} \int_{\Gamma^L_{t}}
	\frac{1}{1+v} f(s^{1/2}) |\pa_u\psi_L^I|^2
  + (1+s)^{1/2} f(s^{1/2}) |\nas \psi_L^I|^2\, dS dt,
  \label{EDLdef}
\end{align}
where $f(z) = \log_+ z (\log \log_+z)^{\alpha}$,
and where the quantity $M_I$ is defined by
\begin{equation}
 M_I(t) =
 \int_{t_0}^{t_1} \int_{D^L_t}
 g'(r) \left( (\pa_u\psi_L^I)^2 + (\pa_v\psi_L^I)^2\right) + \left( \frac{g(r) + 1}{r}\right) |\nas \psi_L^I|^2
 + \int_{t_0}^{t_1} (1+t)\lim_{r \to 0} \left| \frac{\psi_L^I}{r}\right|^2\, dt
 \label{MIdef}
\end{equation}
where $g(r) = (\log (1+r))^{1/2} f(\log (1 + r))$.
Finally the quantity $B^L_I(t_1)$ is defined by
\begin{equation}
 B^L_I(t_1) = \int_{t_0}^{t_1} \int_{\Gamma^L_t} v f(v) |\pa_v \psi_L^I|^2\,dS dt.
 \label{BILdef}
\end{equation}

\subsection{The quantities that control the geometry of the shocks}
\label{geomnormsection}

At each time $t$, the shocks $\Gamma^L_t, \Gamma^R_t$ are of the form
\begin{equation}
 \Gamma^A_t = \{ x \in \mathbb{R}^3 : t - |x| = B^A(t, x)\}
\end{equation}
where each $B^A$ is defined in a neighborhood of $\Gamma^A_t$ and satisfies
$\pa_u B^A = 0$. For $Z \in \mathcal{Z}_m$ and  $ Z_{\mB} \in \mathcal{Z}_{\mB}$, we define
the tangential vector fields
\begin{equation}
 Z_T^A = Z - Z(u-B^A) \pa_u,
 \qquad
 Z_{\mB, T}^A = Z_{\mB} - Z_{\mB}(u-B^A) \pa_u
 \label{}
\end{equation}
which are tangent to $\Gamma^A$ at $\Gamma^A$. We will often just write
$Z_T$ as the shock we are considering will be clear from context. To control
the functions $B^A$ we will work in terms of the following pointwise quantities,
\begin{equation}
 |B^A|_{I, \mathcal{Z}_m} = \sum_{|J| \leq |I|}
 \frac{1}{(1+s)^{1/2}} |Z_T^J B^A|,
 \qquad
 |B^A|_{I, \mathcal{Z}_{\mB}} =\sum_{|J| \leq |I|}
 \frac{1}{(1+s)^{1/2}} |Z_{\mB, T}^J B^A|,
 \label{rescaledBnormdef}
\end{equation}
where the factor of $(1+s)^{-1/2}$ has been chosen to counter
the expected growth of the functions $B^A$.
The quantities involving $B^A$ that we will control are the following,
\begin{align}
 G_{N_L}^L(t_1) &= \sum_{|I| \leq N_L-1} \sup_{t_0 \leq t \leq t_1}
 \int_{\Gamma^L_t} |Z_T B^L|_{I, m}^2\, dS +
 \sum_{|I| \leq N_L/2+1} \sup_{t_0 \leq t \leq t_1}
 \sup_{\Gamma^L_t} |B^L|_{I, m}^2\label{GLnorm}\\
 G_{N_C}^R(t_1) &= \sum_{|I| \leq N_C-1} \sup_{t_0 \leq t \leq t_1}
 \int_{\Gamma^L_t} \frac{1}{1+s}|Z_{\mB,T} B^R|_{I, \mB}^2\, dS +
 \sum_{|I| \leq N_C/2+1} \sup_{t_0 \leq t \leq t_1}
 \sup_{\Gamma^L_t} | B^R|_{I, \mB}^2.
 \label{GRnorm}
\end{align}
We remind the reader that here, $dS$ denotes the surface measure on $\Gamma^A_t$
induced by the measure $r^{-2}dx$.

The reason we have worse control of $B^R$ at top-order than
$B^L$ is ultimately because we have worse control of the potential
$\psi^C$ at top order; see in particular
the proof of Proposition \ref{rightshockbootstrap}.

In the above, we are abusing notation slightly and denoting
\begin{equation}
 |Zq|_{I, \mathcal{Z}_m} = \sum_{Z \in \mathcal{Z}_m} |Zq|_{I, \mathcal{Z}_m},
 \qquad
 |Z_{\mB} q|_{I, \mathcal{Z}_m} = \sum_{Z \in \mathcal{Z}_{\mB}} |Z_{\mB} q|_{I, \mathcal{Z}_{\mB}}.
 \label{}
\end{equation}

The above quantities will be used to control top-order derivatives of the functions
$B^L, B^R$. Bounds for these quantities are needed in order
to handle certain error terms we encounter on the timelike
sides of the shocks, see Section \ref{bcbootstrapsection}.
These quantities have been defined so that we expect $G^L, G^R \sim 1$.

We will also need some quantities that control how far the shocks are from the
model shocks. It will be convenient in the upcoming proof to keep track
of angular derivatives separately. To this end, we define
\begin{align}
 K^R(t_1) &=
 \sup_{t_0 \leq t \leq t_1}
 \sup_{x \in  \Gamma^R_t} \left( \, \left| \frac{B^R(t,x)}{s^{1/2}} +  {p}\right|
     + (1+s)^{1/2}\left|\pa_s B^R(t,x) - \frac{1}{2s} B^R(t,x)\right|\,\right),\label{KRdef}
 \\
 \slashed{K}^R(t_1) &= \sup_{t_0 \leq t \leq t_1} \sup_{ x \in \Gamma^R_t}
 \left| \frac{\Omega B^R(t,x)}{s^{1/2}}\right| ,\label{sKRdef}\\
 K^L(t_1) &= \sup_{t_0 \leq t \leq t_1}
 \sup_{x \in  \Gamma^L_t} \left(\,\left| \frac{B^L(t,x)}{s^{1/2}} - {q}\right|
 + (1+s)^{1/2}\left|\pa_s B^L(t,x) - \frac{1}{2s} B^L(t,x)\right|
 \,\right),
 \label{KLdef}\\
 \slashed{K}^L(t_1) &= \sup_{t_0 \leq t \leq t_1} \sup_{x \in \Gamma^L_t}
 \left| \frac{\Omega B^L(t,x)}{s^{1/2}}\right|,
 \label{sKLdef}
\end{align}
which we will assume are small at $t = t_0$ and which we will prove remain small at later
times. {Here, $p$ and $q$ are positive constants bounded away from $0$ by a constant $c$ which is assume to be much bigger 
than any of the small constants $\epsilon$ appearing below.}

\subsection{Assumptions about the initial data}
\label{initialdatasec}

Our result concerns data for the shock front problem which is prescribed
at a large initial time $t_0$,
\begin{equation}
	\frac{1}{t_0} \leq \epsilon_0,
  \label{largestart}
\end{equation}
where the size of the small parameter $\epsilon_0$ will be set in the course of
the upcoming proof.
%
%
We also assume that the initial shock surfaces $\Gamma^L_{t_0}, \Gamma^R_{t_0}$ are given as
\begin{equation}
 \Gamma^L_{t_0} = \{ x \in \mathbb{R}^3 : t_0 - |x| = B^L_0 (x)\},
 \qquad
 \Gamma^R_{t_0} = \{x \in \mathbb{R}^3 : t_0 - |x| = B^R_0 (x)\},
 \label{}
\end{equation}
for functions $B^L_0, B^R_0$ defined in a neighborhood of $\Gamma^L_{t_0},
\Gamma^R_{t_0}$, respectively, and which are such that these surfaces
are sufficiently close to the model shocks{ $u = -\rs s^{1/2}$, $u = \ls s^{1/2}$
for constants $\rs, \ls > 0$}
at $t = t_0$. 
Specifically, we will assume that the following quantities are small
initially,
\begin{align}
 \label{KR0}
 \mathring{K}^R &= \sup_{x \in  \Gamma^R_{t_0}} 
 \left( \, \left| \frac{B^R_0(x)}{(\log(t_0+|x|)^{1/2}} +  {p}\right|
     + (1+\log(t_0+|x|))^{1/2}\left|\pa_s B^R_0(x) - \frac{1}{2\log(t_0+|x|)} B^R_0(x)\right|\,\right)
\\
\label{sKR0}
 \mathring{\slashed{K}}^R &=\sup_{\Gamma^R_{t_0}} \left|\frac{\Omega B_0^R}{\log(t_0 + |x|)^{1/2}}\right|,\\
    \label{KL0}
    \mathring{K}^L &= 
 \sup_{x \in  \Gamma^L_{t_0}} \left( \, \left| \frac{B^L_0(x)}{(\log(t_0+|x|)^{1/2}} -  {q}\right|
     + (1+\log(t_0+|x|))^{1/2}\left|\pa_s B^L_0(x) - \frac{1}{2\log(t_0+|x|)} B^L_0(x)\right|\,\right)
 \\
 \label{sKL0}
 \mathring{\slashed{K}}^L &= \sup_{\Gamma^L_{t_0}} \left|\frac{\Omega B_0^L}{\log(t_0 + |x|)^{1/2}}\right|
\end{align}
We will also assume that we have a bound for the following quantities
which control the regularity of the initial shocks,
\begin{align}
    \mathring{G}^R_{N_C} &= \sum_{|I| \leq N_C-1}
 \int_{\Gamma^R_{t_0}} |Z_{\mB,T} B^R_0|_{I, \mB}^2\, dS +
 \sum_{|I| \leq N_C/2+1}
 \sup_{\Gamma^R_{t_0}} |B^R_0|_{I, \mB}^2
 \\
 \mathring{G}^L_{N_L} &= \sum_{|I| \leq N_L-1}
 \int_{\Gamma^L_{t_0}} |Z_T B^L_0|_{I, m}^2\, dS +
 \sum_{|I| \leq N_L/2+1}
 \sup_{\Gamma^L_{t_0}} |B^L_0|_{I, m}^2
\end{align}

Finally, we will assume that we have control of the following norms
of the potentials initially,
\begin{align}
 \mathring{\mathcal{E}}_{N^R}^R &= \mathcal{E}_{N^R}^R(t_0)
 + \sum_{|I| \leq N^R} \int_{\Gamma^R_{t_0}} |\psi^I_{R}|^2\, dS,
 \label{initialenergyR}\\
 \mathring{\mathcal{E}}_{N^C}^C &= \mathcal{E}_{N^C}^C(t_0)
 + \sum_{|I| \leq N^C} \int_{\Gamma^R_{t_0}} |\psi^I_C|^2\, dS
 + \int_{\Gamma^L_{t_0}} |\psi^I_C|^2\, dS,\\
  \mathring{\mathcal{E}}_{N^L}^L &= \mathcal{E}_{N^L}^L(t_0)
  + \sum_{|I| \leq N^L} \int_{\Gamma^L_{t_0}} |\psi^I_{L}|^2\, dS.
 \label{initialenergyL}
\end{align}

\subsection{The statement of the main theorem}
Our main theorem, which establishes nonlinear stability of
the model shock solutions in weighted $L^2$-based norms, is the following. {We consider  the irrotational shock problem
 \eqref{contpf}-\eqref{rhpf}, derived from the compressible Euler equations \eqref{intromass}-\eqref{introeul} 
 under the assumption that the equation of state $p=P(\rho)$ satisfies $P'(1)>0, P''(1)\ne 0$ with $v=\nabla\Phi$. After appropriate 
 rescaling these equations take the form \eqref{Hexpintro}.} 
\begin{theorem}
	\label{mainthm}
 Fix parameters $N_R, N_C, N_L, \mu,\alpha$ as in \eqref{parameters} {and constants $\ls, \rs>0$ for the position of the model shocks
 as in \eqref{KL0},\eqref{KR0}}.
 There are $\epsilon_0, \epsilon_1, \epsilon_2, \epsilon_R, \epsilon_C, \epsilon_L, M_0^R, M_0^L$ with the following property.
 If the initial data is posed at $t = t_0$ where
 $t_0$ satisfies \eqref{largestart},
 and the initial data for the potential perturbations $(\phi^R_0, \phi^C_0,\phi^L_0)$
 and the shocks $(B^L(t_0), B^R(t_0))$ satisfy the bounds
 \begin{align}
  &\mathring{\mathcal{E}}^R_{N^R} \leq \epsilon_R^{3}, \qquad
	\mathring{\mathcal{E}}^C_{N^C} \leq \epsilon_{C}^{3}, \qquad
	\mathring{\mathcal{E}}^L_{N^L} \leq \epsilon_{L}^{3},
     \label{initialdatabds}
	\\ 
    &\mathring{K}^L + \mathring{K}^R \leq \epsilon_1^2,
    \qquad
  \mathring{\slashed{K}}^L + \mathring{\slashed{K}}^R \leq \epsilon_2^2,
\qquad
  \mathring{G}_{N_C}^R \leq M_0^R,
	\qquad
	\mathring{G}_{N_L}^L \leq M_0^L,
     \label{initialdatabds2}
 \end{align}
 with notation as in Section \ref{initialdatasec},
 then there is a {unique} global-in-time solution
 $(\phi_R, \phi_C, \phi_L, \Gamma^R, \Gamma^L)$ to the irrotational shock problem
{ \eqref{contpf}-\eqref{rhpf}} which corresponds to the decomposition (with  ${\Bbb I}_{D}$ denoting the indicator function of set $D$):
$$
\Phi=\sigma+\phi_L {\Bbb I}_{D^L}+\phi_C {\Bbb I}_{D^C}+\phi_R {\Bbb I}_{D^R}
$$
{with the profile $\sigma$ defined in \eqref{decomp} and 
smooth functions $\phi_L,\phi_C,\phi_R$ defined in the respective regions $D^L,D^C,D^R$ separated by the shocks $\Gamma^L,\Gamma^R$.} These quantities enjoy the following
 estimates. 
\begin{itemize}
    \item 
 There is a constant $C = C(N_L, N_C, N_R, \mu, \epsilon_0, \epsilon_1,
 \epsilon_2, M_0^L, M_0^R)$
 so that with $\mathcal{E}_R(t), \mathcal{E}_C(t), \mathcal{E}_L(t)$ defined as in \eqref{ERdef0}-\eqref{EDLdef},
 for $t \geq t_0$,
 \begin{equation}
  \mathcal{E}_L(t) \leq C \epsilon_L^2, \qquad
	\mathcal{E}_C(t) \leq C \epsilon_C^2(1+ \log \log t), \qquad
	\mathcal{E}_R(t) \leq C \epsilon_R^2.
  \label{}
 \end{equation}

\item 
 The potentials satisfy
 \begin{alignat}{2}
  | \pa Z^I \phi_R| &\leq C \frac{\epsilon_R}{(1 + r + t)(1 + \log(1+r+t))^{(1+\mu)/4}},
	\qquad \text{ in } D^R, &&\qquad |I| \leq N_R - 3\\
	|\pa \ZB^I \phi_C| &\leq C \frac{\epsilon_C}{(1 + r + t)(1 + \log (1+r+t))^{{3/4}}},
	\qquad \text{ in } D^C, &&\qquad |I| \leq N_C - 5\\
	| \pa Z^I \phi_L| &\leq C
	\frac{ \epsilon_C}{(1+r+t)(1+\log(1+r+t))^{1/2}(1+\log(1+\log(1+r+t)))^{\alpha/2}}
	\qquad \text{ in } D^L, &&\qquad |I| \leq N_L - 3.
\end{alignat}
\item 
There is a function $B^L$ defined in a neighborhood of $\Gamma^L$
and a function $B^R$ defined in a neighborhood of $\Gamma^R$ so that
$\pa_u B^A = 0$, $B^A(t_0, x) = B_0^A(x)$, and so that
the shocks $\Gamma^A = \cup_{t \geq t_0} \Gamma^A_t$ have the form
\begin{equation}
 \Gamma^R_t = \{ x \in \mathbb{R}^3 : t - |x| = -B^R(t, x)\}.
 \qquad
 \Gamma^L_t = \{ x \in \mathbb{R}^3 : t - |x| = B^L(t, x)\},
 \label{graphoversphere}
\end{equation}
The functions $B^A$ enjoy the following bounds,
\begin{align}
	\left| \frac{B^R(t,x)}{s^{1/2}} +  1\right|
	+ (1+s)^{1/2}\left|\pa_s B^R(t,x) - \frac{1}{2s} B^R(t,x)\right|
	+ \left| \frac{\Omega B^R(t,x)}{s^{1/2}}\right|
	&\leq C \epsilon_C ,&&\quad \text{ along } \cup_{t' \geq t_0} \Gamma^R_{t'}
 \label{lowestbootstrapofBRthmstatement}
\end{align}
and
\begin{align}
	\left| \frac{B^L(t,x)}{s^{1/2}} -  1\right|
	+ (1+s)^{1/2}\left|\pa_s B^L(t,x) - \frac{1}{2s} B^L(t,x)\right|
	+\left| \frac{\Omega B^L(t,x)}{s^{1/2}}\right|
	&\leq C\epsilon_L ,&&\quad \text{ along } \cup_{t' \geq t_0} \Gamma^L_{t'},
 \label{lowestbootstrapofBLthmstatement}
\end{align}
as well as the higher-order bounds
\begin{equation}
  \label{}
  G^R_{N_C} \leq \mathring{G}^R_{N_C} + C \epsilon_C,\qquad
  G^L_{N_L} \leq \mathring{G}^L_{N_L} + C \epsilon_L,
  \label{lowestbootstrapofBLthmstatementhigherorder}
\end{equation}

\end{itemize}

\end{theorem}
We can also get more 
    precise information than \eqref{lowestbootstrapofBRthmstatement}-\eqref{lowestbootstrapofBLthmstatementhigherorder}
    about the position of the shocks as $t \to \infty$. The following result is proven
    in Section \ref{asympsec}.
    \begin{theorem}[The asymptotic behavior of the shocks]
      \label{asympthm}
        Let $\Gamma^L, \Gamma^R$ denote the shocks $\Gamma^A = \cup_{t \geq t_0} \Gamma_t^A$
        constructed in the previous theorem
        and let $N_L' = N_L, N_C' = N_C - 2$.
        For all $t \geq t_0$, there are functions $r^A_t \in H^{N_A'}(\mathbb{S}^2)$
        so that
        \begin{equation}
          \label{}
          \Gamma^L_t = \{x \in \mathbb{R}^3 : r = t - (\log t)^{1/2}r^L_t(\omega) \},
          \qquad
          \Gamma^R_t = \{x \in \mathbb{R}^3 : r = t + (\log t)^{1/2}r^R_t(\omega) \},
        \end{equation}
        where $r = |x|, \omega = x/|x|$.
        Moreover, the functions $r^A$ have limits as $t \to \infty$: there are functions
        ${0<r}^A_\infty \in H^{N_A'}(\mathbb{S}^2)$ with
        \begin{equation}
          \label{}
          \lim_{t \to \infty} \| r^A_t - r^A_\infty\|_{H^{N_A'}(\mathbb{S}^2)} = 0.
        \end{equation}
        {The asymptotic behavior of the shocks and the pointwise estimates on the potentials $\phi_R,\phi_C,\phi_L$ from the previous 
        Theorem also imply the Landau law of decay along the shocks:}
        $$
        |\pa\Phi|\sim \frac 1{t(\log t)^{1/2}},\quad {\text{along}} \,\, \Gamma^L,\Gamma^R.  
        $$
    \end{theorem}

Theorem \ref{mainthm} is a consequence of the following bootstrap argument.
\begin{prop}
	\label{bootstrapprop}
 Fix the parameters $N_R, N_C, N_L, \mu,\alpha$ as in \eqref{parameters}. There is
 $\epsilon^* = \epsilon^*(N_R, N_C, N_L, \mu,\alpha)$ so that if
 \begin{equation}
  \epsilon_L \leq \epsilon_C^2 \leq \epsilon_R^4 \leq \epsilon^*,
  \label{}
 \end{equation}
 then there are $\epsilon_i^* = \epsilon_i^*(\epsilon_L, \epsilon_C, \epsilon_R)$
 for $i = 0,1,2$
 with the following property.
 If the conditions \eqref{largestart}
 and \eqref{initialdatabds}-\eqref{initialdatabds2} hold with $\epsilon_i \leq \epsilon_i^*$,
and the bounds
 \begin{align}
  \sup_{t_0 \leq t \leq t_1} \mathcal{E}^R_{N^R}(t) &\leq \epsilon_R^2,\label{rightboot}\\
	\sup_{t_0 \leq t \leq t_1} \mathcal{E}^C_{N^C ,T}(t) &\leq \epsilon_C^2,\label{centralboottop}\\
	\sup_{t_0 \leq t \leq t_1} \mathcal{E}^C_{N^C-1, D}(t) &\leq \epsilon_C^2
	(1+ \log \log (1+t_1)),
	\label{centralbootextra}\\
	\sup_{t_0 \leq t \leq t_1} \mathcal{E}^C_{N^C - 2 ,D}(t) &\leq \epsilon_C^2,\label{centralbootlower}\\
  \sup_{t_0 \leq t \leq t_1} \mathcal{E}^L_{N^L}(t) &\leq \epsilon_L^2\label{leftboot},\\
	G^L_{N_L}(t_1) &\leq M\label{leftgeom}\\
	G^R_{N_R}(t_1) &\leq M\label{rightgeom}\\
        K^R(t_1) + K^L(t_1)&\leq \epsilon_1,\\
        \slashed{K}^R(t_1) + \slashed{K}^L(t_1) &\leq \epsilon_2,
	\label{logeom}
 \end{align}
 hold for some $t_1 > t_0$, where $M = 4(M^L_0 +  M^R_0)$
 where $M^L_0, M^R_0$ are as in \eqref{initialdatabds}, then in fact
 \begin{align}
  \sup_{t_0 \leq t \leq t_1} \mathcal{E}^R_{N^R}(t) &\leq \epsilon_R^{2+1/2},\label{rightbootconclusion}\\
	\sup_{t_0 \leq t \leq t_1} \mathcal{E}^C_{N^C, T}(t) &\leq \epsilon_C^{2+1/2},\label{centralbootconclusion1}\\
	\sup_{t_0 \leq t \leq t_1} \mathcal{E}^C_{N^C - 1,D}(t) &\leq \epsilon_C^{2+1/2} (1+ \log \log (1+t_1)),
	\label{centralbootconclusion2}\\
	\sup_{t_0 \leq t \leq t_1} \mathcal{E}^C_{N^C-2, D}(t) &\leq \epsilon_C^{2+1/2},
	\label{centralbootconclusion3}\\
  \sup_{t_0 \leq t \leq t_1} \mathcal{E}^L_{N^L}(t) &\leq \epsilon_L^{2+1/2},\label{leftbootconclusion}\\
		G^L_{N_L}(t_1) &\leq M_0^L + \epsilon_L^2 \label{GLimprovement}\\
		G^R_{N_R}(t_1) &\leq M_0^R + \epsilon_C^2\label{GRimprovement}\\
        K^R(t_1) + K^L(t_1)&\leq \epsilon_1^{1 + 1/2},\label{logeomimprovementL}\\
        \slashed{K}^R(t_1) + \slashed{K}^L(t_1) &\leq \epsilon_2^{1 + 1/2}
		\label{logeomimprovement}.
 \end{align}
\end{prop}

Theorem \ref{mainthm} then follows from Proposition \ref{bootstrapprop},
 a standard continuity argument, and the local existence theory developed
 in \cite{Majda83}, \cite{MajdaThomann87}, and \cite{Metivier1990}. 

\subsection{The proof of the bootstrap proposition}
{For the sake of simplicity we assume that the constants
    $\ls,\rs$ determining the positions of the model shocks are equal to one, $\ls,\rs=1$,
but the argument below applies to any $\ls,\rs>0$, since all of the supporting material 
holds for arbitrary $\ls, \rs > 0$.} 

We start by showing that the conclusion of Proposition
\ref{bootstrapprop} follows from
some pointwise and time-integrated estimates for the potentials
and under the assumption that our shocks are close to the positions
of the model shocks $u = \pm s^{1/2}$.
In section \ref{scalarcontrol} and \ref{bcbootstrapsection} (see in particular Lemmas
\ref{abstractscalarright}-\ref{abstractscalarleft}
and Propositions \ref{lefttocenterprop} and \ref{centertorightprop}),
we show that the needed pointwise and time-integrated estimates follow
from the hypotheses of Proposition \ref{bootstrapprop}. Finally, in Propositions \ref{leftshockbootstrap}
and \ref{rightshockbootstrap}, we show how to recover the needed
assumptions on the positions of the shocks.

In the rightmost region, the result is the following.
In the upcoming Lemma \ref{abstractscalarright}, we show that
the below hypotheses on $\psi_R$ follow from the bootstrap assumptions in
Proposition \ref{bootstrapprop}. The fact that the below
hypotheses on the shock $\Gamma^R$ follow from the bootstrap
assumptions is established in Proposition \ref{rightshockbootstrap}.
\begin{prop}[The energy estimate in $D^R$]
	\label{abstractenestDR}
	There are constants $\epsilon^\prime$ and $C$ depending only on
	$N_R$ so that the following statements hold true.
	Let $\Gamma^R_t = \{ (t, x) : u = B^R(t,x)\}$
	and let
 $\psi_R(t)$ be a solution to the wave equation \eqref{waveext0} in the region
 $D^R_t$ to the right of $\Gamma^R_t$ on a time interval $[t_0, T)$.
 Suppose that $B^R$ satisfies the bounds
	\begin{alignat}{2}
		\left| \frac{B^R(t,x)}{s^{1/2}} +  1\right|
		+ (1+s)^{1/2}\left|\pa_s B^R(t,x) - \frac{1}{2s} B^R(t,x)\right|
		&\leq\ve_1,&&\quad \text{ along } \cup_{t_0 \leq t' \leq t_1} \Gamma^R_{t'}
		\label{rightshocklogeomright}
	  \\
	  |\Omega B^R(t, x)| &\leq \ve_2 (1+s)^{1/2},&&\quad \text{ along } \cup_{t_0 \leq t' \leq t_1} \Gamma^R_{t'}
	 \label{rightshocklogeomright2}
 \end{alignat}
	 for $\ve_1, \ve_2\leq \epsilon'$ and
	suppose that for some $C_0 > 0$, the following estimates hold true for
	all $|I| \leq N_R$ and all $t \leq T$, with $\mK$ as in Proposition \ref{effectivemmmink} 
    and with the higher-order current $P_{I} = P_{I, R}$ as in \eqref{higherordereqnext}
	\begin{align}
	 &|\pa \psi_R(t,x)| + \frac{1}{1+v} |\psi_R(t, x)| \leq C_0 \frac{\epsilon_R}{(1+s)^{1/2}},
	 \label{assumppwdecayR}\\
	  &\int_{t_0}^{t} \int_{D^R_t} |\mK_{X_R, \gamma, P_I}[\psi^I_R]|
		+ |F_I| |X \psi^I_R|\, dt' \leq C_0 \epsilon_R^{3},\label{needforER0}\\
			 &\int_{D^R_{t_0}} |X_R| |P_I|^2 + \int_{D^R_{t}} |X_R| |P_I|^2
	 + \int_{t_0}^{t}\int_{\Gamma^R_{t'}} (1+v)(1+s)^{1/2} |X_R||P_I|^2\, dS dt'
	 \leq C_0 \epsilon_R^{3},
	 \label{PIassump}\\
		&\mathring{\mathcal{E}}^R_{N^R} \leq \epsilon_R^{3},
	  \label{needforER}
	 \end{align}
	 where $P_I$ is as in Lemma \ref{higherorderexterioreqns}
	 and where $\psi^I_R = r Z^I \phi$.

	 Then
	 \begin{equation}
	  \mathcal{E}^R_{N^R}(t) \leq C \epsilon_R^{3}.
	  \label{conclrightboot0}
	 \end{equation}
\end{prop}

\begin{proof}
	We first show that if the given assumptions hold, the energy
	estimate from Proposition \ref{rightenest} holds.
	Under our assumptions, $r \gtrsim v$, and so
	writing $\phi_R = \frac{1}{r}\psi_R \sim \frac{1}{1+v} \psi_R$,
	and using Lemma \ref{higherorderexterioreqns}
 \begin{equation}
  |\gamma| \lesssim \frac{1}{1+v} |\pa \psi_R| + \frac{1}{(1+v)^2} |\psi_R|.
  \label{gammapsi0R}
 \end{equation}
 If the bound \eqref{assumppwdecayR} holds, then by \eqref{gammapsi0R},
 the first bound for $\gamma$ in \eqref{pert1} holds, and as a result,
 provided $\epsilon_1, \epsilon_2$ are taken sufficiently small,
 the hypotheses of Proposition \ref{rightenest} hold. As a result,
 for each $|I| \leq N_R$ and $t_1 \leq T$,
 \begin{equation}
  	 E^R_I(t_1) + S^R_I(t_1)  \lesssim
		  \int_{t_0}^{t_1} \int_{D^R_t} |\mK_{X, \gamma, P_I}[\psi_R^I]| + |F| |X\psi^I_R|\, dt
		 + \epsilon_R^{3},
  \label{enbootR0}
 \end{equation}
 where we used \eqref{PIassump} to control the term $R_{P_I, X_R}$
 from Proposition \ref{rightenest} and \eqref{needforER}
 to control the energy at $t = t_0$. The result
 now follows immediately from our assumptions.
\end{proof}

We now record an analogous statement in the central region. The statement
is slightly more complicated because we need to keep track of different energies
and some of the energies are allowed to grow in time. The proof
that the below bounds involving $\psi$ follow from the bootstrap assumptions
appears in Lemma \ref{abstractcentralscalar}.
The fact that the below
hypotheses on the shock $\Gamma^R, \Gamma^L$ follow from the bootstrap
assumptions is established in Propositions \ref{rightshockbootstrap}-\ref{leftshockbootstrap}.
\begin{prop}[The energy estimate in $D^C$]
	\label{abstractenestDC}
There are constants $\epsilon^\prime$ and $C$ depending only on
$N_C$ so that the following statements hold true.
Let $\Gamma^R_t = \{ (t, x) : u = B^R(t,x)\}$,
$\Gamma^L_t = \{ (t, x) : u = B^L(t,x)\}$,
and let
$\psi_C(t)$ be a solution to the wave equation \eqref{waveint0} in the region
$D^C_t$ lying between $\Gamma^R_t$ and $\Gamma^L_t$ on a time interval $[t_0, T)$.
Suppose that $B^R$, $B^L$ satisfy the bounds
\begin{alignat}{2}
	\left| \frac{B^R(t,x)}{s^{1/2}} +  1\right|
	+ (1+s)^{1/2}\left|\pa_s B^R(t,x) - \frac{1}{2s} B^R(t,x)\right|
	&\leq\ve_1,&&\quad \text{ along } \cup_{t_0 \leq t' \leq t_1} \Gamma^R_{t'}
	\\
	|\Omega B^R(t, x)| &\leq \ve_2 (1+s)^{1/2},&&\quad \text{ along } \cup_{t_0 \leq t' \leq t_1} \Gamma^R_{t'}
 \label{BRgeomC0}
\end{alignat}
\begin{alignat}{2}
	\left| \frac{B^L(t,x)}{s^{1/2}} -  1\right|
	+ (1+s)^{1/2}\left|\pa_s B^L(t,x) - \frac{1}{2s} B^L(t,x)\right|
	&\leq\ve_1,&&\quad \text{ along } \cup_{t_0 \leq t' \leq t_1} \Gamma^L_{t'}
	\\
	|\Omega B^L(t, x)| &\leq \ve_2 (1+s)^{1/2},&&\quad \text{ along } \cup_{t_0 \leq t' \leq t_1} \Gamma^L_{t'}
 \label{}
\end{alignat}
for $\ve_1, \ve_2\leq \epsilon'$, and further suppose that the parameter
$\epsilon_0$ from \eqref{largestart} satisfies $\epsilon_0 \leq \epsilon'$.

 Suppose that with $X_C, X_T$ defined as in section \ref{fields}, for some $C_0 > 0$
	the following estimates hold, with $\mK$ defined as in Proposition \ref{effectivemmmB} 
    and the currents $P_{I, C}, P_{I, null}$ as in \eqref{higherordereqncentral}.
    First,
	 \begin{align}
		 |\pa \psi_C| + \frac{1}{1+v} |\psi_C| \leq C_0 \epsilon_C,\label{assumppwdecayC}
     \end{align}
     Next, for all $t_1 \leq T$, writing $\psi_I^C = Z_{\mB}^I(r\phi_C)$, we
     assume that:
\begin{itemize}
    \item (Top-order assumptions) For all $|I| \leq N_C$,
     \begin{multline}
         \int_{t_0}^{t_1} \int_{D^C_t} |\mK_{X_T, \gamma, P_{I,C} + P_{I, null}}[Z_{\mB}^I\psi_C]|
          +|K_{X_T, \gamma_a}[Z_{\mB}^I \psi_C]|
		 + \left(|F_{C,I}| + |F_{\Sigma, I}| + |{F_{\mB, I}^2}|\right) |X_T \psi^I_C|\, dt
		 \\
		 \leq C_0\epsilon_C^{3} + c_0(\epsilon_0)\left(1 + \frac{1}{\delta}\right)\epsilon_C^2
            + c_0(\epsilon_0)
            + C_0 \delta  S_I^C(t_1)
         \label{XTKassumpC},
      \end{multline}
      \begin{multline}
        {\int_{t_0}^{t_1} \int_{D^C_t} -F^1_{\mB, I} X_T \psi\, dt} \leq
         c_0(\epsilon_0) \epsilon_C^2 + C_0\delta\sum_{|J|\leq |I|}
          \left(E_{J, X_T}(t_1)  + S_{I}^C(t_1)\right)\\
            +\frac{C_0}{\delta} \sum_{|J| \leq |I|-1}
          E_{J, X_T}(t_1) + C_0\sum_{|J| \leq |I|-1} S_{J}^C(t_1) + C_0\epsilon_C^3,
          \label{F1assumpbdtop}
     \end{multline}
     and, with $P_I = P_{I, C} + P_{I, null}$,
     \begin{multline}
       \label{}
         \int_{D_{t_0}^{C}} v |P_{{I}}|^2
		 	 + \int_{D_{t_1}^{C}}v |P_{{I}}|^2
		 	 + \int_{t_0}^{t_1}\int_{\Gamma^R_t}v|P_{{I}}|^2\,dt
		 	 + \int_{t_0}^{t_1}\int_{\Gamma^L_t} v |P_{{I}}|^2\, dt
			 \leq C_0 \epsilon_C^{3}, 
             \label{Pbdscentral}
     \end{multline}

 \item (Below top-order assumptions)
For all $|I| = N_C-1$,
\begin{multline}
  \label{}
    \int_{t_0}^{t_1} \int_{D^C_t} |\mK_{X_C, \gamma, P_I}[Z_{\mB}^I\psi_C]|
          +|K_{X_C, \gamma_a}[Z_{\mB}^I \psi_C]|
          + \left(|F_{C,I}| + |F_{\Sigma, I}| + |{F_{\mB, I}^2}|\right) |X_C \psi^I_C|\, dt\\
			 \leq C_0 \epsilon_C^{3} (1 + \log \log t_1) + c_0(\epsilon_0)\left(1 + \frac{1}{\delta}\right)\epsilon_C^2
			 + c_0(\epsilon_0)
			 + C_0 \delta S_I^C(t_1)
			 \\ + C_0 \epsilon_C(1+\log\log t_1) \sum_{|J| \leq |I|-1} \sup_{t_0 \leq t \leq t_1}
				(E_{D, J}^C(t))^{1/2} 
				\label{XDKassumpClog},
\end{multline}
and for all $|I| \leq N_C-2$, 
\begin{multline}
      \label{XDKassumpC}
        \int_{t_0}^{t_1} \int_{D^C_t} |\mK_{X_C, \gamma, P_I}[Z_{\mB}^I\psi_C]|
          +|K_{X_C, \gamma_a}[Z_{\mB}^I \psi_C]|
          + \left(|F_{C,I}| + |F_{\Sigma, I}| + |{F_{\mB, I}^2}|\right) |X_C \psi^I_C|\, dt\\
			\leq
			C_0 \epsilon_C^{3} + c_0(\epsilon_0)\left(1 + \frac{1}{\delta}\right)\epsilon_C^2
		  + c_0(\epsilon_0)
			+ \delta\left( \sup_{t_0 \leq t \leq t_1} E^C_{D, I}(t)
		  + S_I^C(t_1)\right),
    \end{multline}
    and, finally, for all $|I| \leq N_C-1$, we assume that
    \begin{multline}
      \label{}
        {\int_{t_0}^{t_1} \int_{D^C_t} -F^1_{\mB, I} X_C \psi\, dt} \lesssim 
             c_0(\epsilon_0) \epsilon_C^2 + C_0\delta\sum_{|J|\leq |I|}
      \left(E_{J, X_T}(t_1)  + S_{I}^C(t_1)\right)\\
            +\frac{C_0}{\delta} \sum_{|J| \leq |I|-1}
      E_{J, X_T}(t_1) + C_0\sum_{|J| \leq |I|-1} S_{J}^C(t_1) + C_0\epsilon_C^3,
          \label{F1assumpbddecay}
    \end{multline}
\end{itemize}

Suppose additionally that the initial data satisfies
\begin{equation}
  \label{}
 \mathring{\mathcal{E}}^C_{N^C} \leq \epsilon_C^{3},
          \label{initialEbdcentral}
\end{equation}

	 and suppose that we have the following
	 estimate at the right shock,
	 \begin{equation}
	\sum_{|I| \leq N_C}  \int_{t_0}^{t_1} \int_{\Gamma^R_t} (1+v) |\evmB \psi^I_C|^2\, dS dt
	+ \sum_{|I| \leq N_C - 1} \int_{t_0}^{t_1} \int_{\Gamma^R_t} (1+s) |\nas \psi^I_C|^2\, dS dt
	\leq C_0 \epsilon_C^{3}
	  \label{rightshockerror}
	 \end{equation}
	 Then
	 \begin{equation}
	  \mathcal{E}^C_{N^C, T}(t) \leq C \epsilon_C^{3},
		\qquad
	  \mathcal{E}^C_{N^C-1, D}(t) \leq C \epsilon_C^{5/2}(1+\log\log t),
		\qquad
		\mathcal{E}^C_{N^C-2, D}(t) \leq C \epsilon_C^{3}.
	  \label{EC52goal}
	 \end{equation}
\end{prop}
\begin{proof}
	As in the previous result, we start by showing that the hypotheses
	here imply that the energy estimate from Proposition \ref{centralenest} holds.
	From \eqref{gammaboundscentral0}, we have the following bound for
	$\gamma$,
	\begin{equation}
	 |\gamma| \lesssim \frac{1}{1+v} |\pa \psi_C| + \frac{1}{(1+v)^2} |\psi_C|
	 + \frac{1}{(1+v)^2}.
	 \label{}
	\end{equation}
	Provided \eqref{largestart} holds, the last term here is bounded by
	$c_0(\epsilon_0) (1+v)^{-1}(1+s)^{-1/2}$ for a continuous function
	$c_0$ with $c_0(0) = 0$. Assuming
	\eqref{assumppwdecayC} to bound the first two terms here, we have
	\begin{equation}
	 |\gamma| \lesssim \frac{\epsilon_C + c_0(\epsilon_0)}{(1+v)(1+s)^{1/2}},
	 \label{}
	\end{equation}
	and so provided $\epsilon_C, \epsilon_0$ are taken sufficiently small,
	the hypotheses of Proposition \ref{centralenest} hold true. It then follows
	from the definitions of the energies, the assumptions
	\eqref{Pbdscentral}-\eqref{initialEbdcentral},
	and the bound for the first term in \eqref{rightshockerror}, that for some $C_0' > 0$,
	we have the energy estimate
	\begin{multline}
        E_{I, T}^C(t_1) + S_{I}^C(t_1) - {C_0' \int_{t_0}^{t_1} \int_{D^C_t} F_{I, \mB}^1 X_T\psi_C^I\, dt}
        \\
        \lesssim
	 \int_{t_0}^{t_1}
	 \int_{D^C_t} |\mK_{X_T, \gamma, P_I}[\psi_C^I]|
     + {|K_{X_T, \gamma_a}[\psi_C^I]}
     +(|F_{C,I}| + |F_{\Sigma, I}| + |{F_{\mB, I}^2}|)|X_T \psi_C^I|\,dt
	 + \epsilon_C^{3}, \qquad |I| \leq N_C
	 \label{enbootC0}
	\end{multline}
	and additionally using the bound for the second term in \eqref{rightshockerror},
    we have the energy estimate
	\begin{multline}
	 E_{I, D}^C(t_1) - {C_0' \int_{t_0}^{t_1} \int_{D^C_t} F_{I, \mB}^1 X_T\psi_C^I\, dt}
     \\
 \lesssim S_{I}^C(t_1)+
	 \int_{t_0}^{t_1}
	 \int_{D^C_t} |\mK_{X_C, \gamma, P_I}[\psi_C^I]| 
     + {|K_{X_C, \gamma_a}[\psi_C^I]}
     + (|F_{C,I}| + |F_{\Sigma, I}| + |{F_{\mB, I}^2}|) |X_C \psi_C^I|\,dt
	 + \epsilon_C^{3}, \qquad |I| \leq N_C-1.
	 \label{EDClocalbd}
	\end{multline}

    {We start with the first estimate here. By the assumption \eqref{F1assumpbdtop},
        taking $\delta$ and then $\epsilon_0$ sufficiently small, the bound
        \eqref{enbootC0} implies that for $|I| \leq N_C$
    	\begin{multline}
            \sum_{|I'| = |I|} E_{I', T}^C(t_1) + S_{I'}^C(t_1) 
        \lesssim
	 \int_{t_0}^{t_1}
	 \int_{D^C_t} |\mK_{X_T, \gamma, P_I}[\psi_C^I]|
     + {|K_{X_T, \gamma_a}[\psi_C^I]}
     +(|F_{C,I}| + |F_{\Sigma, I}| + |{F_{\mB, I}^2}|)|X_T \psi_C^I|\,dt
     \\
     + \sum_{|J| \leq |I|-1} E_{J, T}^C(t_1) + S_{J}^C(t_1)
	 + \epsilon_C^{3}, 
	 \label{enbootC0prime}
	\end{multline}
    and the assumption \eqref{XTKassumpC} then implies
\begin{equation}
            \sum_{|I'| = |I|} E_{I', T}^C(t_1) + S_{I'}^C(t_1) 
        \lesssim
	      \sum_{|J| \leq |I|-1} E_{J, T}^C(t_1) + S_{J}^C(t_1)
	 + \epsilon_C^{3}
	 \label{enbootC0primeprime}
	\end{equation}
    By induction, this gives the first bound in \eqref{EC52goal} for a constant $C = C(N_C)$.

    Similarly, for $|I| \leq N_C-2$, using the bound we just proved for $S_I^C$
    and the assumptions \eqref{XDKassumpC}-\eqref{F1assumpbddecay}, 
    the energy estimate \eqref{EDClocalbd}
    implies that, after possibly taking $\epsilon_0$ smaller,
    \begin{equation}
      \label{}
    \sum_{|I'| = |I|} E_{I', D}^C(t_1)
        \lesssim
	      \sum_{|J| \leq |I|-1} E_{J, D}^C(t_1) 
	 + \epsilon_C^{3},
    \end{equation}
    and by induction this gives the third bound in \eqref{EC52goal}.

    It remains only to get the second bound in \eqref{EC52goal}, and for this we use 
    \eqref{EDClocalbd}, the bounds we just proved, and the assumption \eqref{XDKassumpClog},
    we find that, after possibly taking $\epsilon_0$ smaller still,
    \begin{equation}
      \label{}
        \sum_{|I'| = N_C-1} E_{I', D}^C(t_1) \lesssim
        \epsilon_C^{5/2}(1 + \log \log t_1),
    \end{equation}
    as needed.

    }

\end{proof}

Finally, in the leftmost region we rely on the following result.
The fact that the below bounds on $\psi_L$ follow from our bootstrap
assumptions can be found in Lemma \ref{abstractscalarleft}.
The fact that the below
hypotheses on the shock $\Gamma^L$ follow from the bootstrap
assumptions is established in Proposition \ref{leftshockbootstrap}.
\begin{prop}[The energy estimate in $D^L$]
	\label{abstractenestDL}
	There are constants $\epsilon^\prime$ and $C$ depending only on
	$N_L$ so that the following statements hold true.

Let $\Gamma^L_t = \{ (t, x) : u = B^L(t,x)\}$
and let
$\psi_L(t)$ be a solution to the wave equation \eqref{waveext0} in the region
$D^L_t$ to the left of $\Gamma^L_t$ on a time interval $[t_0, T)$.
Suppose that $B^L$ satisfies the bounds
\begin{alignat}{2}
	\left| \frac{B^L(t,x)}{s^{1/2}} -  1\right|
	+ (1+s)^{1/2}\left|\pa_s B^L(t,x) - \frac{1}{2s} B^L(t,x)\right|
	&\leq\ve_1,&&\quad \text{ along } \cup_{t_0 \leq t' \leq t_1} \Gamma^L_{t'} \label{BLgeomL}
	\\
	|\Omega B^R(t, x)| &\leq \ve_2 (1+s)^{1/2},&&\quad \text{ along } \cup_{t_0 \leq t' \leq t_1} \Gamma^L_{t'}
 \label{BLgeomL2}
\end{alignat}
 for $\ve_1, \ve_2\leq \epsilon'$.
	Suppose that for some $C_0 > 0$, the following estimates hold true for
	some $\alpha > 1$,
	all $|I| \leq N_L$ and all $t \leq T$,
	\begin{align}
	 & |\pa \psi_L(t,x)| + \frac{1}{1+v} |\psi_L(t, x)| \leq C_0 \frac{\epsilon_L}{(1+s)^{1/2}}
	  \frac{(\log \log s)^{\alpha}}{(\log s)^{\alpha -1}},
	 \label{assumppwdecayL}\\
	 & |\pa \phi_L(t,x)| \leq C_0\frac{\epsilon_L}{1+v} \frac{1}{(1+s)^3}, \qquad
	 |u| \geq s^3,\label{assumppwdecayL0}
	 \\
	  &\int_{t_0}^{t} \int_{D^L_t} |\mK_{X_L, \gamma, P_I}[\psi^I_L]|
		+ |F_I| |X_L \psi^I_L|\, dt' \leq C_0 \epsilon_L^{3},\label{needforEL0}\\
		&\int_{t_0}^{t} \int_{D^L_t} |\mK_{X_M, \gamma, P_I}[\psi^I_L]|
		+ |F_I| |X_M \psi^I_L|\, dt' \leq C_0 \epsilon_L^{3},\label{needforEL0M}\\
			 &\int_{D^R_{t_0}} |X_L| |P_I|^2 + \int_{D^L_{t_1}} |X_L| |P_I|^2
	 + \int_{t_0}^{t}\int_{\Gamma^L_{t'}} \left(|X_L^v| + (1+v)(1+s)^{1/2} |X^u_L|\right) |P_I|^2\, dS dt'
	 \leq C_0 \epsilon_L^{3},
	 \label{PIassumpL}\\
		&\mathring{\mathcal{E}}^L_{N^L}\leq \epsilon_L^{3},
	  \label{needforEL}
	 \end{align}
	 where $P_I$ is as in Lemma \ref{higherorderexterioreqns}
	 and where $\psi^I_L = r Z^I\phi$.
	 Suppose moreover that we have the following bound at the left shock,
	 \begin{equation}
	  \sum_{|I| \leq N_L} \int_{t_0}^{t_1} \int_{\Gamma^L_t} v f(v) |\pa_v \psi^I_L|^2\, dS dt
		\leq C_0\epsilon_L^{3}.
	  \label{bdyerrorleftboot}
	 \end{equation}
	 Then
	 \begin{equation}
	  \mathcal{E}^L_{N^L}(t) \leq C \epsilon_L^{3}.
	  \label{ELgoal}
	 \end{equation}
\end{prop}
Since $\alpha > 1$, note that this requires a
stronger pointwise estimate for the potential than the previous two results.

\begin{proof}
	From Lemma \ref{higherorderexterioreqns}, we have the following bounds,
	\begin{equation}
	 |\gamma| \lesssim \frac{1}{r} |\pa \psi_L| + \frac{1}{r^2} |\psi_L|^2,\qquad
	 |\gamma| \lesssim |\pa \phi_L|,
	 \label{}
	\end{equation}
	where recall $\phi_L = r^{-1} \psi_L$. Using \eqref{assumppwdecayL} in
	the region $D^L_t \cap \{|u| \leq s^3\}$ and
	\eqref{assumppwdecayL0} when $|u| \geq s^3$, we therefore have the bound
	\begin{equation}
	 |\gamma| \lesssim \frac{\epsilon_L}{(1+v)(1+s)^{1/2}}
	  \frac{(\log \log s)^{\alpha}}{(\log s)^{\alpha -1}}
	 \label{pert1leftused}
	\end{equation}
	everywhere in $D^L_t$, and so provided $\epsilon_L$ is
	taken sufficiently small, the hypothesis \eqref{pert1left}
  of Proposition
	\ref{leftengen} holds true. It then follows from our assumptions and the
	definitions of the energies that
	\begin{equation}
	 E_{I}^L(t) \lesssim \int_{t_0}^{t_1} \int_{D^L_t} |\mK_{X_L, \gamma, P_I}[\psi_L^I]|
	 + |F_I| |X_L \psi_L^I|\, dt
	 + \epsilon_L^{3},
	 \label{enbootL0}
	\end{equation}
	using \eqref{bdyerrorleftboot} to handle the boundary term on the right-hand
	side of the identity \eqref{decayleftenergybd}, and the bound for
	$\sum_{|I| \leq N_L} E_I^L(t)$ follows from \eqref{needforEL0}. The bound
	for the (Morawetz) energy $\sum_{|I| \leq N_L} M_I(t)$ from
	\eqref{EDLdef}-\eqref{MIdef} follows in the same way after using Proposition
	\ref{morawetzlem} in place of Proposition \ref{leftengen} and the
	bound \eqref{needforEL0M} in place of \eqref{needforEL0}.
\end{proof}

\begin{proof}[The proof of Proposition \ref{bootstrapprop}]

	Proposition \ref{bootstrapprop} follows
	from Propositions \ref{abstractenestDR}-\ref{abstractenestDL}, provided
	we can show that the hypotheses of Proposition \ref{bootstrapprop}
	imply the hypotheses of these results. In this section we map out how
	this follows from the results in the upcoming sections
	\ref{consofbootsec}-\ref{bdsforbdfsec}.
		We start with Proposition \ref{abstractenestDR}, which controls
		the solution in the rightmost region.

		\begin{lemma}[The estimates in $D^R$]
		 Under the hypotheses of Proposition \ref{bootstrapprop},
		 the bounds \eqref{rightshocklogeomright}-\eqref{needforER}
		 hold true. In particular, under the hypotheses of Proposition \ref{bootstrapprop},
		 the bound \eqref{rightbootconclusion} holds.
		\end{lemma}
\begin{proof}
	 First, by Proposition \ref{rightshockbootstrap}, under the hypotheses of
	 Proposition \ref{bootstrapprop}, the bounds
	 \eqref{rightshocklogeomright}-\eqref{rightshocklogeomright2}
	 for the right shock $\Gamma^R$ hold with $\ve_1 =
	 \ve_2 = \epsilon_C$.
	 Next, by the pointwise bound \eqref{pwright} for $\psi_R$
	  from Lemma \ref{basicpwdecay},
	since $\mu > 1$
	the pointwise bound \eqref{assumppwdecayR} holds, and by Lemma
	\ref{abstractscalarright}, the bounds \eqref{needforER0}
	for the scalar current and for the terms $F_I$ hold.
	Finally, by the bound \eqref{RPXboundR}, the bound \eqref{PIassump}
	for the lower-order currents $P_I$ along
	the time slices and the right shock holds. Therefore, the bound
	\eqref{conclrightboot0} holds, and after taking $\ve_R$ small enough
	that $C \ve_R \leq \ve_R^{1/2}$, we get the needed bound
	\eqref{rightbootconclusion}.
\end{proof}

	We now show how the hypotheses of Proposition \ref{abstractenestDC},
	which gives bounds for the solution in the central region,
	follow from the hypotheses of our bootstrap Proposition \ref{abstractenestDC}.
	\begin{lemma}[The estimates in $D^C$]
	 Under the hypotheses of Proposition \ref{bootstrapprop},
	 the bounds \eqref{BRgeomC0}-\eqref{rightshockerror} hold true.
	 In particular, under the hypotheses of Proposition \ref{bootstrapprop},
	 the bounds \eqref{centralbootconclusion1}-\eqref{centralbootconclusion3} hold.
	\end{lemma}

	\begin{proof}
	 	By Lemma \ref{rightshockbootstrap}, under the hypotheses of
 	 Proposition \ref{bootstrapprop} the bounds
 	 \eqref{rightshocklogeomright} for the right shock $\Gamma^R$ hold with $\ve_1 =
 	 \ve_2 = \epsilon_C$ and the bounds \eqref{rightshocklogeomright2}
	 for the left shock $\Gamma^L$ hold with
	 $\ve_1 = \ve_2 = \epsilon_L$.
	 Next, by the pointwise bound \eqref{pwcentral}
	 for $\psi_C$ from Lemma \ref{basicpwdecay},
     the assumption \eqref{assumppwdecayC} holds.
     By Lemmas {\ref{FImBibplemma}} and \ref{abstractcentralscalar}, if we take
     $\epsilon_0$ sufficiently small, the bounds in \eqref{XTKassumpC}-{\eqref{F1assumpbddecay}}
		for the scalar currents and the inhomogeneous terms $F_I$
	 	hold, and by \eqref{RPXboundC} and \eqref{extraPIslicebd}, the bounds \eqref{Pbdscentral}
		for the lower-order currents hold.

		It remains to handle the boundary terms along the right shock from
		 \eqref{rightshockerror}. By Proposition
	 	\ref{centertorightprop}, the first term there is bounded by the right-hand side of
	 	\eqref{rightshockerror} if \eqref{rightboot} holds and $\epsilon_0$
	 	is taken sufficiently small. Using the bound \eqref{rightangularbound}
		for angular derivatives of $\psi_C$ along the right shock
		and the bound we just proved for the derivatives $\evmB$ of $\psi_C$ along
		the right shock,
	 	the second term on the left-hand side of \eqref{rightshockerror} is
	 	also bounded by the right-hand side of \eqref{rightshockerror}.

		As a result, the bounds \eqref{EC52goal}
		for the energies in the central region hold under the hypotheses of
		the bootstrap proposition \ref{bootstrapprop}, and
		taking $\ve_C$ smaller if needed we therefore get
		\eqref{centralbootconclusion1}-\eqref{centralbootconclusion3}.
	\end{proof}

	Next, we show how the hypotheses of Proposition \ref{abstractenestDL}, which handles
	the bounds in the leftmost region, follow from the bootstrap proposition.
	\begin{lemma}[The estimates in $D^L$]
Under the hypotheses of Proposition \ref{bootstrapprop},
the bounds \eqref{BLgeomL}-\eqref{bdyerrorleftboot} hold true.
In particular, under the hypotheses of Proposition \ref{bootstrapprop},
the bound \eqref{leftbootconclusion} holds.
	\end{lemma}
\begin{proof}
 By Proposition \ref{leftshockbootstrap}, the bounds
 \eqref{BLgeomL}-\eqref{BLgeomL2} for the left shock $\Gamma^L$ hold with
 $\ve_1 = \ve_2 = \epsilon_L$ under the hypotheses of Proposition \ref{bootstrapprop}.
 Next, by the pointwise bound \eqref{pwleft} for $\psi_L$ from Lemma \ref{basicpwdecay},
 the bound \eqref{assumppwdecayL} holds. For the bound \eqref{assumppwdecayL0}
 for $\phi_L = \tfrac{1}{r} \psi_L$,
 we instead use the pointwise bound \eqref{pwleft2}.
 By Lemma \ref{abstractscalarleft},
 the bounds \eqref{needforEL0}-\eqref{needforEL0M}
 for the scalar currents $\mK_{X_L, \gamma, P_I}$
 and $\mK_{X_M, \gamma, P_I}$ and for the quantity $F_I$ hold under our hypotheses,
 and by the estimate \eqref{RPXboundL} from Lemma \ref{timeintegrability-left},
 the bounds for the lower-order currents $P_I$ along
 the time slices and the left shock hold.

 Finally, to get the bound \eqref{bdyerrorleftboot} for $\evm \psi_L$ along
 the shock, we use Proposition \ref{centertorightprop}. Combining the above,
	the bound \eqref{ELgoal} holds under the hypotheses of Proposition \ref{bootstrapprop},
	and taking $\epsilon_L$ smaller if needed we get \eqref{leftbootconclusion}.
\end{proof}

To conclude the proof of Proposition \ref{bootstrapprop}, we need to show
how the improved estimates \eqref{GLimprovement}-\eqref{logeomimprovement} 
follow from our assumptions. These bounds are all direct consequences of
Propositions \ref{rightshockbootstrap}- \ref{leftshockbootstrap}
after taking $\epsilon_0$ smaller, if needed.
\end{proof}

It remains to prove the above-mentioned results, which control the scalar
currents, boundary terms along the timelike sides of the shock,
and give pointwise decay estimates for the solution.
The goal of the next three sections is to prove these bounds,
under the hypotheses of Proposition \ref{bootstrapprop}.

\section{Basic consequences of the bootstrap assumptions}
\label{consofbootsec}

We collect here some simple consequences of the bootstrap proposition
\ref{bootstrapprop}.
In the next section, we will use the estimates from this section to bound the scalar
currents and inhomogeneous terms in each region,
as well as the error terms along the timelike sides of the shocks.
\subsection{Pointwise estimates}
\label{pwdecaysubsection}
We start by recording the pointwise decay estimates.
\begin{lemma}[Pointwise decay estimates]
	\label{basicpwdecay}
	Under the hypotheses of the bootstrap proposition \ref{bootstrapprop},
	provided
	the quantities $\epsilon_0, \epsilon_R, \epsilon_C, \epsilon_L$
	are taken sufficiently small, we have the following estimates.
	\begin{alignat}{2}
     &\sum_{|I| \leq N_R-3} (1+|u|)^{\mu/2} |\pa \psi^I_R| 
     + (1+|u|^\mu + r(\log r)^\nu)^{1/2}\left( |\pa_v\psi^I_R|
     + |\nas \psi^I_R|\right)
	\lesssim
	 \epsilon_R, &&\qquad \text{ in } D^R_t\label{pwright}\\
	 &\sum_{|I| \leq N_C-5} (1 + \log t)^{1/4} \left((1+s)^{1/2} |\pa \psi^I_C|
	 + (1+v)^{1/2} \left(|\evmB\psi^I_C| +  |\nas  \psi^I_C|\right)\right)
	 \lesssim
	 \epsilon_C,&&\qquad \text{ in } D^C_t\label{pwcentral}\\
	 &\sum_{|I| \leq N_L-3} (1+ |u|)(\log |u|)^{1/2} (\log \log |u|)^{\alpha/2}
	 |\pa \psi_L^I|\\
	 &\qquad\qquad+ \sum_{|I| \leq N_L-3}(1+v)^{1/2}(1+s)^{1/2}(\log s)^{\alpha/2} \left(|\pa_v \psi_L^I|
	 +|\nas \psi_R^I|\right)
	 \lesssim \epsilon_L, &&\qquad \text{ in } D^L_t
	  \label{pwleft}
\\
	 &\sum_{|I| \leq N_L-3}(1+ v) (1+s)^3 |\pa Z^I \phi_L|
	 \lesssim \epsilon_L, &&\qquad \text{ in } D^L_t \cap \{ |u| \geq s^3\}
	 \label{pwleft2}
 \end{alignat}
	where recall $\psi_L = r \phi_L$.

	We also have the bounds
	\begin{alignat}{2}
	 \sum_{|I| \leq N_R - 3} |\psi^I_R|
	 &\lesssim
	 %
	 \frac{1}{(1 + \log t)^{(\mu-1)/4}}
	 \epsilon_R,
	 &&\qquad \text{ in } D^R_t,
	 \label{pwRhardy}\\
	 \sum_{|I| \leq N_C - 5} |\psi^I_C| &\lesssim
	 (1+\log t)^{1/2}\epsilon_C,&&\qquad \text{ in } D^C_t,
	 \label{pwChardy}\\
	 \sum_{|I| \leq N_L-3} |\psi^I_L|
	 &\lesssim
	 (1 + \log t)^2 \epsilon_L
	 &&\qquad \text{ in } D^R_t \cap\{|u| \leq s^3\}.
	 \label{pwLhardy}
 \end{alignat}
\end{lemma}

\begin{proof}
	The bounds \eqref{pwright}-\eqref{pwleft} are immediate consequences of the definitions
 of the energies, the definitions of the domains
 $D^L, D^C, D^R$, and the Klainerman-Sobolev inequalities
 from Section \ref{klainermansobolev}. For the bound in the central region,
 we additionally use the fact that $|Z^I q| \lesssim \sum_{|I'| \leq |I|}
 |Z_{\mB}^{I'} q|$.
 The bounds \eqref{pwRhardy}-\eqref{pwLhardy}
 follow after using the upcoming Lemma \ref{homoglem} to control
 the relevant $L^2$-based norms. To prove \eqref{pwleft2}, we use
 the standard Klainerman-Sobolev inequality to get
 \begin{align}
  (1+v)(1+|u|)^{1/2} |\pa Z^I \phi_L|
	&\lesssim \sum_{|J| \leq |I| + 3}
	\left(\int_{D^L_t \cap\{|u| \geq s^3\}}
	 |\pa Z^J \phi_L|^2\, r^2 dr dS(\omega)\right)^{1/2}
	 \\
	 &\lesssim
	 \sum_{|J| \leq |I| + 3}
	 \left(\int_{D^L_t\cap\{|u| \geq s^3\}}
	 {|\pa \psi^J_L|^2 + \frac{1}{r^2 }|\psi^J_L|^2}\, dr dS(\omega)\right)^{1/2}\\
	 &\lesssim
	 \sum_{|J| \leq |I| + 3}
	 \left(\int_{D^L_t\cap\{|u| \geq s^3\}}
	 {|\pa \psi^J_L|^2} \, dr dS(\omega)\right)^{1/2}\\
	 &\lesssim (1 + \log t)^{-{3/2}} \epsilon_L,
  \label{}
 \end{align}
 {for $|I| \leq N_C - 3$,
 where in the second-last step we used the Hardy inequality
 \eqref{rhardyL} and the fact that $\psi^J_L = r Z^J \phi_L$ vanishes
 at $r = 0$. In the last step we used
 that the energy in $D^L_t$ controls
 $\| |u|^{1/2} \pa Z^J \psi_L\|_{L^2(D^L_t)}$ and that
 we are just considering the region $|u| \gtrsim (\log t)^3$.}
\end{proof}

We record some $L^2$-based bounds for homogeneous quantities. In each region
the idea is just to integrate
to one of the shocks and use bounds for the energies to control
the resulting boundary terms.
\begin{lemma}
	\label{homoglem}
	Under the hypotheses of Proposition \ref{bootstrapprop}, we have the bounds
	\begin{align}
	 \sum_{|I| \leq N_R-1}\| \psi^I_R\|_{L^2(D^R_t)}
	 &\lesssim
	 {(1 + \log t)^{1/2} (1 + \log t)^{-\mu/4} \epsilon_R}
	 \label{hardyR}
\\
	 \sum_{|I|\leq N_C-2}
	 \|\psi^I_C\|_{L^2(D^C_t)}
	 &\lesssim
	 (1+\log t)^{1/4}(\mathring{\mathcal{E}}_{N_C}^C)^{1/2}
	  +  (1 + \log t)^{3/4}\epsilon_C
	 \label{hardyC}\\
	 \sum_{|I| \leq N_L-1}\|\psi^I_L\|_{L^2(D^L_t \cap \{|u|\leq s^3\}}
	 &\lesssim
	 (1+\log t)^2
	 \left( (\mathring{\mathcal{E}}_{N_L}^L)^{1/2} +  \epsilon_L\right).
		\label{hardyL}
 \end{align}
\end{lemma}
\begin{remark}
 The precise powers of $\log t$ appearing in \eqref{hardyC}-\eqref{hardyL}
 are largely irrelevant for our estimates, since the quantities
 on the left-hand sides of \eqref{hardyC}-\eqref{hardyL} will always
 enter into our estimates with an additional power of $t^{-1}$ which
 can be used to absorb these slowly-growing factors.
\end{remark}
\begin{proof}
	By the Hardy-type inequality \eqref{hardyRapp}, we have
	\begin{equation}
	 (1 + \log t)^{\mu/4} \|  \psi_R^I \|_{L^2(D^R_t)} \lesssim
	 (1 + \log t)^{1/2} \| (1 + r-t)^{\mu/2} \pa \psi_R^I \|_{L^2(D^R_t)}
	 \lesssim
	 (1 + \log t)^{1/2} \epsilon_R,
	 \label{}
	\end{equation}
	which is \eqref{hardyR}.

	To get \eqref{hardyC}, we use \eqref{hardyCapp}
	with $q = \psi_C^I$,
	\begin{align}
	 \|\psi_C^I\|_{L^2(D^C_t)} &\lesssim (\log t)^{1/4} \|\psi_C^I\|_{L^2(\Gamma^L_{t_0})}
 	+ (\log t)^{3/4} \left(\int_{t_0}^{t} \int_{\Gamma^L_{t'}} v |\pa_v \psi_C^I|^2
 	+ \frac{1}{vs} |\pa_u \psi_C^I|^2\, dS dt'\right)^{1/2}\\
     &\qquad+ (\log t)^{1/2} \|\pa \psi_C^I\|_{L^2(D^C_t)},
	 \label{usehardyCpf}
	\end{align}
	and noting that for $|I| = N_C$, we only have a uniform bound for
	$(\log t)^{-1/2} \|\pa \psi_C^I\|_{L^2(D^C_t)}$ and that energy
	controls the boundary term here, the result follows.
	The bound \eqref{hardyL} follows in the same way, but
	using \eqref{rhardyL0}; note that it is here that we needed
	to assume the bound for the quantity $B_I$ in the definition
	of the energy \eqref{ELdef}.
\end{proof}

We now move onto the time-integrated estimates.

\subsection{Time-integrated estimates for the potentials}
\label{timeintsec}
\begin{lemma}[Time-integrated estimates in the rightmost region]
		Under the hypotheses of Proposition \ref{bootstrapprop}, we have
		\begin{multline}
	\sum_{|I| \leq N_R-5}
	\int_{t_0}^{t_1} \frac{1}{1+t}
	\left(\| \pa^2 \psi^I_R\|_{L^\infty(D^R_t)} +
    \left\| \frac{1}{1+|u|}\pa \psi^I_R\right\|_{L^\infty(D^R_t)}\right)\, dt\\
    + {\sum_{|I| \leq N_R-5}\int_{t_0}^{t_1} \frac{1}{1+t} 
        \left\| \left(1 + \frac{r^{1/2}(\log r)^{\nu/2}}{(1+|u|)^{\mu/2}}\right)\pa_v\pa
        \psi^I_R\right\|_{L^\infty(D^R_t)} + \left\| \left(1 + \frac{r^{1/2}(\log
        r)^{\nu/2}}{(1+|u|)^{\mu/2}}\right)\nas \pa
\psi^I_R\right\|_{L^\infty(D^R_t)}\, dt}\\
	+ \sum_{|I| \leq N_R-5}\int_{t_0}^{t_1}\frac{1}{(1+t)^2} \|\psi^I_R\|_{L^\infty(D^R_t)} \, dt \lesssim 
    {\epsilon_R}\label{rightint1}
\end{multline}
\end{lemma}
\begin{proof}
	Bounding $|\pa^2 \psi|\lesssim \sum_{Z \in \mZ} |\pa Z \psi|$
and $|u| \gtrsim (1+\log t)^{1/2}$ in $D^R$,
the first bound in \eqref{pwright} gives
\begin{equation}
 |\pa^2 \psi^I_R| + (1+|u|) |\pa \psi^I_R| \lesssim
 (1+|u|)^{-\mu/2+1} \epsilon_R
 \lesssim
  (1+\log t)^{-(\mu-2)/4}\epsilon_R,
 \label{}
\end{equation}
and since we chose $\mu > 6$ in \eqref{parameters}, the first two bounds in \eqref{rightint1} follow immediately.
Note that a slightly better bound is possible for just $|\pa^2\psi_R^I|$
but this will not be needed. For the bounds on the second line of 
\eqref{rightint1}, we just use the bound $(1+v)(|\pa_v q | + |\nas q|)
\lesssim \sum_{Z \in \mathcal{Z}} |Z q|)$ and then estimate
\begin{equation}
  \label{}
  \frac{1}{1+t} \left(1 + 
  \frac{r^{1/2}(\log r)^{\nu/2}}{(1+|u|)^{\mu/2}}\right)|(\pa_v,\nas)\pa\psi_R^I|
  \lesssim
  \frac{1}{1+t} \left(\frac{1}{1+v} + \frac{(\log r)^{\nu/2}}{(1+|u|)^{\mu/2} (1+v)^{1/2}}\right)
  \sum_{|J| \leq 1} |\pa Z^J \psi_R^I|.
\end{equation}
Bounding $(\log r)^{\nu/2}(1+|u|)^{-\mu/2} (1+v)^{-1/2} \lesssim (1+t)^{-1/4}$, say,
gives the bound for the terms on the second line of \eqref{rightint1}, and
the remaining bound follows directly from \eqref{pwRhardy}.
\end{proof}

\begin{lemma}[Time-integrated estimates in the central region]
		Under the hypotheses of Proposition \ref{bootstrapprop}, 
\begin{equation}
	\sum_{|I|\leq N_C-6}	 \int_{t_0}^{t_1}
	\frac{1}{1+t} \left( \| \pa^2 \psi^I_C\|_{L^\infty(D^C_t)}
    + \left\| \frac{1}{1+s} \pa \psi^I_C\right\|_{L^\infty(D^C_t)}\right)
	\,dt \lesssim \epsilon_C,
			 \label{centralint1}
		 \end{equation}
		 \begin{multline}
		\sum_{|I|\leq N_C-6}	 \int_{t_0}^{t_1}
		\frac{1}{1+t}
		 \| (1+s)^{1/2} (1+v)^{1/2} \pa_v \pa \psi^I_C\|_{L^\infty(D^C_t)}\, dt
		 \\
		 + \sum_{|I| \leq N_C - 6} \int_{t_0}^{t_1}
		 \frac{1}{1+t}
		 \| (1+s)^{1/2} (1+v)^{1/2} \nas \pa\psi^I_C\|_{L^\infty(D^C_t)}
		\,dt \lesssim \epsilon_C
			 \label{centralint2}
		\end{multline}
\end{lemma}
\begin{proof}

To prove \eqref{centralint1},
we use that $(1+s)|\pa q| \lesssim \sum_{Z_{\mB} \in \mZB} |Z_{\mB} q|$
and so
\begin{equation}
 |\pa^2 Z_{\mB}^I \psi_C| + (1+s)^{-1} |\pa Z_{\mB}^I \psi_C|
 \lesssim \sum_{|J| \leq |I|+1} (1+\log t)^{-1} |\pa Z_{\mB}^J\psi_C|
 \lesssim (1+\log t)^{-3/2} \epsilon_C,
 \label{}
\end{equation}
which gives \eqref{centralint1}.

To get \eqref{centralint2}, we bound
$(1+v)(|\pa_v q| + |\nas q|)\lesssim
\sum_{Z_{\mB} \in \mathcal{Z}_{\mB}} |Z_{\mB} q|$ which gives
\begin{equation}
	(1+s)^{1/2} (1+v)^{1/2} \left( |\pa_v \pa Z_{\mB}^I \psi_C| + |\nas \pa Z_{\mB}^I\psi_C|\right)
	\lesssim\frac{(1+s)^{1/2}}{(1+v)^{1/2}} \sum_{|J| \leq |I|+1}
	|\pa Z_{\mB}^J \psi_C|
	\lesssim \frac{\epsilon_C }{(1+t)^{1/2}},
 \label{}
\end{equation}
by \eqref{pwcentral}, and \eqref{centralint2} follows.
\end{proof}

\begin{lemma}[Time-integrated estimates in the left region]
	\label{timeintpsileft}
			Under the hypotheses of Proposition \ref{bootstrapprop},
		\begin{equation}
	\sum_{|I| \leq N_L-4}
	\int_{t_0}^{t_1} \frac{1}{1+t}
	\left(\| \pa^2 \psi^I_C\|_{L^\infty(D^L_t)} +
    \left\| \frac{1}{1+|u|}\pa \psi^I_L\right\|_{L^\infty(D^L_t)}\right)\, dt \lesssim \epsilon_L \label{leftint1}
\end{equation}
\begin{multline}
	 \sum_{|I| \leq N_L-4}
	\int_{t_0}^{t_1}
	\frac{1}{1+t}\left\|\frac{(1+v)^{1/2} f(v)^{1/2}}{(1+|u|)^{1/2} f(|u|)^{1/2}}\pa_v \pa \psi^I_L\right\|_{L^\infty(D^L_t)}\,dt\\
	+
	\sum_{|I| \leq N_L-4}
 \int_{t_0}^{t_1}
 \frac{1}{1+t}\left\|\frac{(1+v)^{1/2} f(v)^{1/2}}{(1+|u|)^{1/2} f(|u|)^{1/2}}\nas \pa\psi^I_L\right\|_{L^\infty(D^L_t)}
	\, dt
		\lesssim \epsilon_L,
		\label{leftint2}
\end{multline}
and
\begin{equation}
 \sum_{|I| \leq N_L-4}
 \int_{t_0}^{t_1}
 \|\pa^2 Z^I\phi_L  \|_{L^\infty(D^L_t \cap \{|u| \geq s^3\})}\, dt
 \lesssim \epsilon_L
 \label{leftint3}
\end{equation}
where recall $\psi^I_L = r Z^I\phi_L$.
\end{lemma}

\begin{proof}
We start with the bound \eqref{leftint1} which is where we will need the bound
for the quantity $M_I$ defined in \eqref{MIdef}.
The needed estimate in the region $|r-t| \geq t/8$, say, follows directly
from the bound \eqref{pwleft} since that implies
\begin{equation}
	\sum_{|I| \leq N_L-4}
\int_{t_0}^{t_1} \frac{1}{1+t}\sup_{D^L_t\cap \{|u|\geq t/8\}}
|\pa^2 \psi^I_L|\, dt
\lesssim
\epsilon_L\int_{t_0}^{t_1} \frac{1}{(1+t)^2}\, dt ,
 \label{}
\end{equation}
which is more than sufficient, with a similar
estimate for $(1+|u|)^{-1}|\pa \psi^I_L|$. We therefore focus only on the region $|r-t| \leq t/8$.

Now, the bootstrap assumption
\eqref{leftboot} and the definition of the energy
\eqref{ELdef}-\eqref{MIdef} give, in particular
\begin{equation}
	\sum_{|I| \leq N_L}	\int_{t_0}^{t_1} \int_{D^L_t} g'(r) |\pa Z^I \psi_L|^2\,dt
	\lesssim \epsilon_L^2,
 \label{localleftboot0}
\end{equation}
 where recall
\begin{equation}
 g(r) = (\log(1+r))^{1/2} (\log \log (1+r)) (\log \log \log (1+r) )^{\alpha}.
 \label{}
\end{equation}
In \eqref{localleftboot0} we used that $g(r)/r \gtrsim g'(r)$.
This function satisfies
\begin{equation}
 g'(r) \gtrsim \frac{1}{1+r} \frac{1}{\log (1+r))^{1/2}} (\log \log (1+r))
  (\log \log \log (1+r))^{\alpha}.
 \label{}
\end{equation}
We therefore have the following bound
\begin{equation}
 \frac{1}{1+t} \frac{1}{(1+ \log t)^{1/2}} W_\alpha(t)
 \lesssim g'(r), \qquad W_\alpha(t) = (\log \log t)(\log\log\log t)^{\alpha},
 \qquad |r-t| \leq t/8.
 \label{boundtheweights}
\end{equation}
and in particular, if we define
\begin{equation}
 m_I(t) = \frac{1}{1+t}\frac{1}{(\log t)^{1/2}} W_\alpha(t) \|\pa Z^I \psi_L\|_{L^2(D^L_t
 \cap \{|u| \leq t/8\})}^2,
 \label{morimp}
\end{equation}
we have
\begin{equation}
 \sum_{|I| \leq N_L} \int_{t_0}^{t_1} m_I(t)\, dt \lesssim \epsilon_L^2.
 \label{morimp1}
\end{equation}

We also note that the weight $W_\alpha$ satisfies the following property,
\begin{equation}
 \int_{t_0}^{t_1} \frac{1}{1+t} \frac{1}{\log t} \frac{1}{W_\alpha(t)} \lesssim 1,
 \label{l1weight}
\end{equation}
since $\alpha > 1$.

We now prove the bound.
By the Klainerman-Sobolev inequality, in the region
$|u| \leq t/8$ we have the bound
\begin{equation}
 \sum_{|I| \leq N_L - 4} (\log t)^{3/4} |\pa^2 Z^I \psi_L|
 \lesssim \sum_{|I|\leq N_L-4} (1+|u|)^{3/2} |\pa^2 Z^I\psi_L|
 \lesssim \sum_{|I| \leq N_L-1} \| \pa Z^I \psi_L\|_{L^2(D^L_t\cap\{|u| \leq t/8\})},
 \label{}
\end{equation}
since $|u| \gtrsim (\log t)^{1/2}$ in $D^L_t$.
In particular,
\begin{align}
 \sum_{|I| \leq N_L-4} \int_{t_0}^{t_1} \frac{1}{1+t} \|\pa^2 Z^I \psi_L(t)&\|_{L^\infty(D^C_t\cap\{|u| \leq t/8\})}
 \\
 &\lesssim
 \sum_{|I| \leq N_L-1} \int_{t_0}^{t_1} \frac{1}{1+t} \frac{1}{(\log t)^{3/4}}
  \| \pa Z^I \psi_L\|_{L^2(D^L_t\cap\{|u| \leq t/8\})}\\
	&= \sum_{|I| \leq N_L-1} \int_{t_0}^{t_1} 
    \frac{1}{(1+t)^{1/2}} \frac{1}{(\log t)^{1/2}}
	\frac{1}{W_\alpha(t)^{1/2}} m_I(t)^{1/2}\, dt
	\\
	&\lesssim \left( \int_{t_0}^{t_1} \frac{1}{1+t} \frac{1}{\log t}
	\frac{1}{W_\alpha(t)}\, dt\right)^{1/2}
	\left( \int_{t_0}^{t_1} m_I(t)\, dt\right)^{1/2}
	\\
	&\lesssim \epsilon_L,
 \label{}
\end{align}
by \eqref{morimp1}-\eqref{l1weight},
as needed.

To get \eqref{leftint2} we just bound
$\frac{(1+v)^{1/2} f(v)}{(1 +|u|)^{1/2} f(|u|)} \lesssim (1+v)^{1/2} \log v
\lesssim (1+v)^{3/4}$, say,
 and then bound $(1+v)(|\pa_v q| + |\nas q|)\lesssim \sum_{Z \in \mathcal{Z}_m}
 |Z q|$, which gives
 \begin{equation}
	 \frac{(1+v)^{1/2} f(v)}{(1 +|u|)^{1/2} f(|u|)} \left(
	 |\pa_v \pa Z^I \psi_L| + |\nas \pa Z^I \psi_L|\right)
	 \lesssim
	 \frac{1}{(1+t)^{1/4}} \sum_{|J| \leq |I| +1} |\pa Z^J \psi_L|,
  \label{}
 \end{equation}
 and the needed bound then follows from \eqref{pwleft}.
 Finally, \eqref{leftint3} follows directly from \eqref{pwleft2}.

\end{proof}

\subsection{Estimates for quantities along the shocks}
\label{boundsalongshocks}
We will also need to record some estimates for quantities that
we control at the boundary. Apart from \eqref{rightbdytrivialbds},
these are all immediate consequences of the definitions of the energies
and the bootstrap assumptions, but it is convenient to record
these explicitly.

\begin{lemma}
 Under the hypotheses of Proposition \ref{bootstrapprop}, we have
 the following bounds.
 \begin{multline}
  \sum_{|I| \leq N_R}
  \int_{t_0}^{t_1}\int_{\Gamma^R_t} 
  (1+v)^{-1}(1+s)^{(\mu-1)/2}
	|\pa \psi^I_R|^2
	 dS dt \\
  +\sum_{|I| \leq N_R}
  \int_{t_0}^{t_1}\int_{\Gamma^R_t} 
     (1+v)(1+ s^{{\nu}})|\pa_v \psi^I_R|^2
    + (1+s)^{\mu/2}|\nas \psi^I_R|^2\, dS dt
     \lesssim
	\epsilon_R^2.
  \label{rightrivialbdybds}
 \end{multline}
\end{lemma}

\begin{lemma}
		\label{bdsforpsiCalongshock}
	Under the hypotheses of Proposition \ref{bootstrapprop}, we have
	the following bounds.
	\begin{align}
	 	&\sum_{|I| \leq N_C} \int_{t_0}^{t_1} \int_{\Gamma^L_t} v |\evmB \psi_C^I|^2
		+ (1+s)^{-1/2} |\nas \psi_C^I|^2 + (1+v)^{-1}(1+s)^{-1} |\pa \psi_C^I|^2\, dS dt\\
		&+ \sum_{|I| \leq N_C-1} \frac{1}{1+\log \log t_1}
		\int_{t_0}^{t_1} \int_{\Gamma^L_t} v |\evmB \psi_C^I|^2
		+ (1+s)|\nas \psi_C^I|^2 + (1+v)^{-1}(1+s)^{1/2} |\pa \psi_C^I|^2\, dS dt\\
		&+ \sum_{|I| \leq N_C-2}
		\int_{t_0}^{t_1} \int_{\Gamma^L_t} v |\evmB \psi_C^I|^2
		+ (1+s)|\nas \psi_C^I|^2 + (1+v)^{-1}(1+s)^{1/2} |\pa \psi_C^I|^2\, dS dt
		\lesssim \epsilon_C^2,
	 \label{leftbdytrivialbds}
	\end{align}
	and
	\begin{align}
	 &\sum_{|I| \leq N_C} \int_{t_0}^{t_1} \int_{\Gamma^R_t}
	  (1+s)^{-1/2} |\nas \psi_C^I|^2 + (1+v)^{-1}(1+s)^{-1} |\pa \psi_C^I|^2
		+ v |\evmB \psi_C^I|^2\, dS dt\\
	 &+ \sum_{|I| \leq N_C-1} \frac{1}{1+\log \log t_1}
	 \int_{t_0}^{t_1} \int_{\Gamma^R_t}
	(1+v)^{-1}(1+s)^{1/2} |\pa \psi_C^I|^2\, dS dt\\
	 &+ \sum_{|I| \leq N_C-2}
	 \int_{t_0}^{t_1} \int_{\Gamma^R_t}  (1+v)^{-1}(1+s)^{1/2} |\pa \psi_C^I|^2\, dS dt
	 \lesssim \epsilon_C^2,
	 \label{rightbdytrivialbds}
	\end{align}
	and finally, there is a continuous function $c_0$ with $c_0(0) = 0$
	so that
	\begin{align}
	 \sum_{|I| \leq N_C-1} \int_{t_0}^{t_1} \int_{\Gamma^R_t} (1+s) |\nas \psi_C^I|^2\, dS dt
	 \lesssim c_0(\epsilon_0) \epsilon_C^2 + \sum_{|I| \leq N_C}\int_{t_0}^{t_1} \int_{\Gamma^R_t} v |\evmB \psi_C^I|^2\, dS dt.
	 \label{rightangularbound}
	\end{align}
\end{lemma}
\begin{proof}
 The bounds in \eqref{leftbdytrivialbds} and \eqref{rightbdytrivialbds}
 follow direcly from the definition of the energies and the bootstrap
 assumptions \eqref{centralboottop}-\eqref{centralbootlower}. To get the
 bound \eqref{rightangularbound}, we bound $|\nas \psi_C^I|^2 \lesssim (1+v)^{-2} |\Omega \psi_C^I|$
 and then use Lemma \ref{controlangularhardyright} with
 $q = \Omega \psi_C^I$,
 \begin{equation}
  \int_{t_0}^{t_1} \int_{\Gamma^R_t} \frac{s}{v^2} |\Omega \psi_C^I|^2\, dS dt
	\lesssim \frac{1}{1+t_0} \int_{\Gamma^R_{t_0}} |\Omega \psi_C^I|^2\, dS
	+ c_0(\epsilon_0)\int_{t_0}^{t_1} \int_{\Gamma^R_t} v |\evmB \Omega \psi_C^I|^2 + \frac{1}{vs} |\pa_u \Omega \psi_C^I|^2\, dS dt,
 \end{equation}
 which gives the result after using the bound \eqref{rightbdytrivialbds} to control
 the last term here.
\end{proof}

We also record some bounds for derivatives along the timelike (left) side
of the left shock. These follow immediately from the definition of the energy
from \eqref{EDLdef} and the bootstrap assumption \eqref{leftboot}.
\begin{lemma}
 Under the hypotheses of Proposition \ref{bootstrapprop}, we have
 \begin{multline}
  \sum_{|I| \leq N_L}
	\int_{t_0}^{t_1} \int_{\Gamma^L_t}
	\frac{1}{1+v} \log (1+s) (\log \log (1+s))^\alpha |\pa_u\psi_L^I|^2\, dS dt\\
  + \sum_{|I| \leq N_L}
	\int_{t_0}^{t_1} \int_{\Gamma^L_t} (1+s)^{1/2}\log (1+s) (\log \log (1+s))^\alpha |\nas \psi_L^I|^2
	\\
	+ \sum_{|I| \leq N_L}
	\int_{t_0}^{t_1} \int_{\Gamma^L_t} (1+v)(1+s)(1+\log s)^\alpha |\pa_v \psi_L^I|^2\,dS dt
	\lesssim \epsilon_L^2
  \label{controlofpsiLL}
 \end{multline}
\end{lemma}

\section{Estimates for the scalar currents}
The goal of this section is to prove that our bootstrap assumptions
imply the estimates for the scalar currents
$\widetilde{K}$ that we assumed in Propositions \ref{abstractenestDR}-\ref{abstractenestDL}.
As a first step,
we show how the bounds from the previous section give us control of the quantities
$\gamma_A, P_I$ and $F_{I, A}$ appearing in \eqref{higherordereqnext}-\eqref{higherordereqcentral}. These rely
on the estimates from Section \ref{higherorderequations}.
We point out at this point that by Lemma \ref{basicpwdecay}, for each $A = L, C,R$,
since each $N_A \geq 30$, we have $N_A -5 \geq N_A/2+1$, and so
\begin{equation}
	\sum_{|I| \leq N_A/2+1} |\pa Z^I_A \psi^A(t,x)| \lesssim \frac{\epsilon_A}{(1 + \log t)^{1/2}},
	\qquad \text{ in } D^A_t,
 \label{perturbativeprf}
\end{equation}
where $Z^I_L, Z^I_R$ denote products of the Minkowski fields and
$Z^I_C$ denote products of the fields from $\mZB$. For most of the upcoming
estimates the bound \eqref{perturbativeprf} will be all that is needed.

\subsection{Control of the metric perturbation $\gamma$, the currents $P$
and the inhomogeneous terms}
We now use the bounds for the previous sections to bound various quantities that
appear in the scalar currents that we will need to estimate in the next section
(see \eqref{modifiedKbound0boot}-\eqref{mBmodifiedKboundstatement2}).

We start with the estimates in the rightmost region. 
\begin{lemma}
	\label{timeintegrability-right}
	Let $X = X_R$ be defined as in Section
	\ref{fields} and write $X = X^n\pa_u + X^\ell \pa_v$.
	If the hypotheses of Proposition \ref{bootstrapprop} hold, then the quantities
 $\gamma, P_{I, R}, F_{I, R}$ appearing in \eqref{higherordereqnext}
 satisfy the following estimates. Writing $\gamma = \gamma[\psi_R]$,
 \begin{equation}
  |\gamma| \lesssim \frac{\epsilon_R}{(1+v)(1+s)^{1/2}},
  \label{pert0right}
 \end{equation}
 as well as the following time-integrated bound,
 \begin{multline}
  \int_{t_0}^{t_1} 
      \|\nabla \gamma\|_{L^\infty(D^R_t)}
      +
      \left\|(1+|u|)^{-1}\gamma\right\|_{L^\infty(D_t^R)}\, dt\\
  + \int_{t_0}^{t_1}
      \left\|\frac{|X^\ell_m|^{1/2}}{|X^n_m|^{1/2}}\pa_v \gamma\right\|_{L^\infty(D^R_t)}
  +
      \left\|\frac{|X^\ell_m|^{1/2}}{|X^n_m|^{1/2}}\nas\gamma\right\|_{L^\infty(D^R_t)}
	\, dt
	\lesssim \epsilon_R,
  \label{intgammaR}
 \end{multline}
 and with $P_I = P_{I,R}[\psi_R^I]$ and $F_I = F_{I,R}[\psi_R^I]$,
 \begin{multline}
	 \int_{t_0}^{t_1}
     \| |X^n_m|^{1/2} \nabla P_I\|_{L^2(D^R_t)}
     +
     \| (1+|u|)^{-1}|X^n_m|^{1/2} P_I\|_{L^2(D^R_t)}\, dt\\
     +
     \int_{t_0}^{t_1}
     \left\| |X^\ell_m|^{1/2} \pa_v P_I\right\|_{L^2(D^R_t)} 
 + \left\| |X^\ell_m|^{1/2}\nas P_I\right\|_{L^2(D^R_t)}
	 +
	  \| |X|^{1/2} F_I\|_{L^2(D^R_t)}
	 \, dt\lesssim \epsilon_R^2.
  \label{intPFR}
 \end{multline}

 We also have
 \begin{equation}
  \sup_{t_0 \leq t \leq t_1}
	\int_{D^R_t} |X_R| |P_I|^2 
    + \int_{t_0}^{t_1} \int_{\Gamma^R_t} 
    \left(|X_R| + (1+s)^{1/2}(1+v)  |X^n_{R}|\right)|P_I|^2\, dS dt
	\lesssim \epsilon_R^{3}, \qquad |I| \leq N_R.
  \label{RPXboundR}
 \end{equation}
\end{lemma}

\begin{proof}

	\textit{Part 1: Estimates for $\gamma$}
	To prove the first bound, we use Lemma \ref{higherorderexterioreqns}.
	Since
 $r \gtrsim v$ in $D^L_t$, writing $\phi = \frac{1}{r}\psi \sim \frac{1}{1+v} \psi$,
 this gives
 \begin{equation}
  |\gamma| \lesssim \frac{1}{1+v} |\pa \psi_R| + \frac{1}{(1+v)^2} |\psi_R|.
  \label{gammabdrightpfgammaint}
 \end{equation}
 By
	\eqref{pwright} and \eqref{perturbativeprf}, we therefore have
	\begin{equation}
		(1+v)(1+s)^{1/2} |\gamma|\lesssim
		(1+t)(1+\log t)^{1/2} |\gamma|
		\lesssim (1+\log t)^{1/2} |\pa \psi_R| + \frac{(1+\log t)^{1/2}}{1+t} |\psi_R|
		\lesssim \epsilon_R,
	 \label{}
	\end{equation}
which gives \eqref{pert0right}.

To prove the first two bounds in \eqref{intgammaR}, we use Lemma \ref{higherorderexterioreqns}
again and argue as above to get
\begin{equation}
 |\nabla \gamma| \lesssim \frac{1}{1+v} |\nabla \pa \psi_R| +
 \frac{1}{(1+v)^2} |\pa \psi_R| + \frac{1}{(1+v)^3} |\psi_R|.
 \label{nablaygamma}
\end{equation}

In particular,
 \begin{multline}
  \int_{t_0}^{T} \|\nabla \gamma\|_{L^\infty(D^R_t)}
  + \|(1 +|u|)^{-1} \gamma\|_{L^\infty(D^R_t)}\, dt\\
	\lesssim
	\int_{t_0}^{t_1} \frac{1}{1+t}
	\left(\|\pa^2 \psi_R\|_{L^\infty(D^R_t)}
    + \| (1+|u|)^{-1}\pa \psi_R\|_{L^\infty(D^R_t)}\right)+ \frac{1}{(1+t)^2} \|(1+|u|) \psi_R\|_{L^\infty(D^R_t)}\, dt
	\lesssim \epsilon_L,
  \label{nablaygamma2}
 \end{multline}
 using the time-integrated bounds \eqref{rightint1} for the first two terms and the bound
 \eqref{pwRhardy} for the last term (recall $\mu \geq 6$). Similarly,
 using the bounds
\begin{equation}
  \label{}
    |\nabla_v \gamma| + |\nas \gamma|
    \lesssim
    \frac{1}{1+v} (|\nabla_v\pa \psi_R|+|\nas\pa \psi_R|)
    + \frac{1}{(1+v)^2} |\pa \psi_R| + \frac{1}{(1+v)^3} |\psi_R|,
\end{equation}
and the time-integrated bounds \eqref{rightint1} again, we get the bound
for the terms on the second line of \eqref{intgammaR}.
 
	\textit{Part 2: Estimates for $P_I$ and $F_I$}

    By the pointwise bounds \eqref{pwright} and \eqref{pwRhardy} combined with the
    fact that $r \gtrsim t$ in $D^R_t$, we clearly have $| \pa Z^J (\psi_R/r)| \lesssim 1$
    in $D^R_t$ and so the bounds \eqref{Pbounds}-\eqref{Pbounds3} from
    Lemma \ref{higherorderexterioreqns} hold. Writing $Z^J \phi_R = r^{-1} \psi_R^J$
    we find from \eqref{Pbounds} that
    \begin{multline}
     |P_I| \lesssim \sum_{\substack{|I_1| + |I_2| \leq |I|,\\ |I_1|, |I_2| \leq |I|-1}}
     \left(
     \frac{1}{1+v} |\pa \psi_R^{I_1}| |\pa \psi_R^{I_2}| + \frac{1}{(1+v)^3}
     |\psi_R^{I_1} | |\psi_R^{I_2}|\right)\\
     \lesssim \sum_{|K| \leq |I|/2+1} \sum_{|J| \leq |I|-1}
     \left(
     \frac{1}{1+v} |\pa \psi_R^{K}| |\pa \psi_R^{J}| + \frac{1}{(1+v)^3}
     |\psi_R^{K} | |\psi_R^{J}|\right).
     \label{pfofbasicPIbdR}
    \end{multline}
    Similarly, it follows from \eqref{Pbounds2} and the bound 
    $|\nabla q|\lesssim (1+|u|)^{-1} \sum_{Z \in
    \mathcal{Z}}|Zq|$ that
    \begin{equation}
      \label{}
        |\nabla P_I|
        \lesssim \frac{1}{1+|u|}\sum_{|K| \leq |I|/2+1} \sum_{|J| \leq |I|}
     \left(
     \frac{1}{1+v} |\pa \psi_R^{K}| |\pa \psi_R^{J}| + \frac{1}{(1+v)^3}
     |\psi_R^{K} | |\psi_R^{J}|\right).
    \end{equation}
    As a result, since the bootstrap assumption
    on the energy in this region gives $\| |X^n_m|^{1/2} \pa \psi_R^J\|_{L^2(D^R_t)}
    \lesssim \epsilon_R$, we have
    \begin{multline}
      \label{}
      \int_{t_0}^{t_1} \| |X^n_m|^{1/2} \nabla P_I\|_{L^2(D^R_t)}
      + \||X^n_m|^{1/2}(1+|u|)^{-1}P_I\|_{L^2(D^R_t)}\, dt
      \\
      \lesssim \epsilon_R  \sum_{|K| \leq |I|/2+1}\int_{t_0}^{t_1}
      \frac{1}{1+t}\| (1+|u|)^{-1}\pa \psi_R^K\|_{L^\infty(D^R_t)}\, dt
      \lesssim \epsilon_R^2,
      \label{}
    \end{multline}
    by \eqref{rightint1}. The bounds for the terms on the second line
    of \eqref{intPFR} can be handled in a similar way and we skip them.

To get the bound for $F_I$, we write \eqref{Fbounds} in terms of $\psi_R^J$ and
use that $r \gtrsim t$ in
	$D^R_t$ again to find
	\begin{equation}
	 |F_I| \lesssim \sum_{|K| \leq |I|/2+1} \sum_{|J| \leq |I|}
	 \left( \frac{1}{(1+v)^2} |\pa \psi_R^{K}| |\pa \psi_R^J|
	 + \frac{1}{(1+v)^4} |\psi_R^{K}| |\psi_R^J|\right),
	 \label{}
	\end{equation}
	which is similar to \eqref{pfofbasicPIbdR} but with an additional factor of $(1+t)^{-1}$.
    The bound for $F_I$ then follows in the same way as the above bound for $P_I$.

Finally, we prove the bounds in \eqref{RPXboundR}. The bound
\eqref{pfofbasicPIbdR} gives
\begin{multline}
 \int_{D^R_t} |X_R| |P_I|^2
 \lesssim
 \frac{1}{(1+t)^2} \sum_{|K| \leq |I|/2+1} \sum_{|J| \leq |I|}
 \int_{D^R_t} |\pa \psi_R^K|^2 |\pa \psi_R^J|_{X_R, m}^2\\
 + \frac{1}{(1+t)^6}
 \sum_{|K| \leq |I|/2+1} \sum_{|J| \leq |I|} \int_{D^R_t}
 |X_R||\psi_R^K|^2 |\psi_R^J|^2
 \lesssim \epsilon_R^4,
 \label{}
\end{multline}
As for the term in \eqref{RPXboundR} along the shock,
    we first bound
    \begin{equation}
      \label{}
      |X_R| + (1+s)^{1/2}(1+v)|X_R^n|
      \lesssim 1 + r(\log r)^\nu + (1+s)^{1/2} (1+v) (1+|u|)^{\mu}
      \lesssim (1 + v)(1 + |u|)^{\mu},
    \end{equation}
    along the shock, where we used that by our choices of $\mu, \nu$ 
    in \eqref{parameters},
    $\mu \geq 2\nu$ and that $|u|\gtrsim s^{1/2}$ along the shock.
    By \eqref{pfofbasicPIbdR} we therefore have
\begin{align}
	\int_{t_0}^{t_1} &\int_{\Gamma^R_t}
    \left(|X_R| + (1+s)^{1/2}(1+v)  |X^n_{R}|\right)|P_I|^2\, dS dt\\
	&\lesssim
	\sum_{|K| \leq |I|/2+1}\sum_{|J|\leq |I|}
	\int_{t_0}^{t_1} \int_{\Gamma^R_t}
    \frac{(1+|u|)^{\mu}}{1+v} |\pa \psi^{K}_R|^2 |\pa \psi^{J}_R|^2\,
	dS dt\\
	&\qquad+ 	\sum_{|K| \leq |I|/2+1}\sum_{|J|\leq |I|}
		\int_{t_0}^{t_1} \int_{\Gamma^R_t}
        \frac{(1+|u|)^{\mu}}{(1+v)^3}| \psi^{K}_R|^2 |\psi^{J}_R|^2\, dS dt\,\\
	&\lesssim
	\epsilon_R^2 \sum_{|K| \leq |I|/2+1}\sum_{|J|\leq |I|}
	\int_{t_0}^{t_1} \int_{\Gamma^R_t}
	\frac{X_R^n}{(1+v)(1+s)^{1/2}} |\pa \psi^{J}_R|^2\,
	dS dt\\
	&\qquad+ 	\sum_{|K| \leq |I|/2+1}\sum_{|J|\leq |I|}
		\int_{t_0}^{t_1} \int_{\Gamma^R_t}
		\frac{(1+s)^{(\mu+1)/2}}{(1+v)^3} | \psi^{K}_R|^2 |\psi^{J}_R|^2\, dS dt\\
		&\lesssim
		\epsilon_R^4
 \label{}
\end{align}
where used the weak decay estimate \eqref{perturbativeprf} 
the control
we have over the boundary term in the energy \eqref{ERdef}, along with the Hardy inequality
\eqref{controlangularderivs} to control the terms on the last line.

\end{proof}

We now move onto the estimates in the central region. Recall from
\eqref{higherordereqcentral} that we need to
handle the current $\widetilde{P}_{I,C} = P_{I,C} + P_{I, null}$
where $P_{I, C}$ collects the error terms coming from commuting our vector
fields with the nonlinear terms, and some lower-order and rapidly-decaying 
terms coming from commuting with the linear part of the equation. The current
$P_{I, null}$ collects the most dangerous commutation errors generated by commuting
with the linear term statisfying the null condition.
In the next lemma we control some quantities involving
the quantities $\gamma$ and $\widetilde{P}_{I}$. 
Note that in the first line of
\eqref{centralPIbounds} below and in \eqref{RPXboundC}, we are only estimating $P_{I,C}$ and
not the full current $\widetilde{P}_{I, C}$. We postpone handling the relevant
bounds for the linear errors $P_{I, null}$ until
Lemma \ref{timeintegrability-center-linear}. {Also, it turns out that
    the term $F_{\mB, I}^1$ from \eqref{higherordereqcentral} is (slightly) too large to be treated
as an error term; we postpone handling this term until Lemma \ref{FImBibplemma}. }
\begin{lemma}
	\label{timeintegrability-center}
	Let $X = X_{C}$ or $X = X_{ T}$ with notation as
	in Section \ref{fields} and write
	$X = X^n \pa_u + X^\ell \evmB$.
	There is $\epsilon_0^*$ so that if the hypotheses of Theorem \ref{mainthm} hold
	with $\epsilon_0 < \epsilon_0^*$,
	we have the following bounds. First, for $|I| \leq N_C$, the quantity
 $\gamma$ appearing in \eqref{higherordereqcentral}
 satisfies the following estimates.
 \begin{equation}
  |\gamma| \lesssim \frac{\epsilon_C}{(1+v)(1+s)^{1/2}},
	\qquad
	|X^\ell| |\gamma| \lesssim \epsilon_C |X^n|,
  \label{pert0center}
 \end{equation}
 \begin{multline}
  \int_{t_0}^{t_1} \|\nabla \gamma\|_{L^\infty(D^C_t)}
	+ \left\|\frac{1}{1 + s} \gamma\right\|_{L^\infty(D_t^C)}\, dt
	\\
	+
	\int_{t_0}^{t_1} \left\|(1+v)^{1/2} (1+s)^{1/2} \nabla_{\evmB} \gamma\right\|_{L^\infty(D_t^C)}
	+ \left\|(1+v)^{1/2} (1+s)^{1/2} \nas \gamma\right\|_{L^\infty(D_t^C)}
	 \, dt
	\lesssim \epsilon_C.
  \label{intgammaC}
 \end{multline}
 { The currents $P_{I,C }, P_{I, null}$ from \eqref{higherordereqcentral}}
 satisfy the following estimates, 
 \begin{multline}
	 \sum_{|I| \leq N_C}\int_{t_0}^{t_1}\| |X^n_{\mB}|^{1/2} \nabla P_{I,C}\|_{L^2(D_t^C)}
	 + \| (1+s)^{-1}|X^n_{\mB}|^{1/2}  P_{I,C}\|_{L^2(D_t^C)}\, dt\\
     +\sum_{|I| \leq N_C}\int_{t_0}^{t_1} \| (1+v)^{-1/2} \nabla_{\evmB} (P_{I,C} + P_{I, null})\|_{L^2(D_t^C)}
     +  \| (1+v)^{-1/2} \nas (P_{I,C} + P_{I, null}) \|_{L^2(D_t^C)}
	 \,dt\\
	 \lesssim \epsilon_C^2 + c_0(\epsilon_0) \epsilon_C,
  \label{centralPIbounds}
 \end{multline}
as well as the bounds
  \begin{align}
      \sup_{t_0 \leq t \leq t_1} \int_{D^C_t} v |P_{I,C}|^2 +
      \int_{t_0}^{t_1} \int_{\Gamma^L_t} v|P_{I,C}|^2\, dS dt
      +\int_{t_0}^{t_1} \int_{\Gamma^R_t} v|P_{I,C}|^2\, dS dt
	&\lesssim
	\epsilon_C^{3}, \qquad |I| \leq N_C.
    \label{RPXboundC}
  \end{align}

 The quantities on the right-hand side of \eqref{higherordereqcentral}
 satisfies the following estimates. The remainders $F_{C, I}$ and $F_{\Sigma, I}$
 satisfy
	\begin{equation}
	 \int_{t_0}^{t_1} \| |X^\ell_{\mB}|^{1/2} F_{C, I}\|_{L^2(D^C_t)}
	 + \| |X^\ell_{\mB}|^{1/2} F_{\Sigma, I}\|_{L^2(D^C_t)}^2\, dt
	 \lesssim c_0(\epsilon_0) + (c_0(\epsilon_0) + \epsilon_C)\epsilon_C^2,
	 \label{FICl2bds}
	\end{equation}
    {while, for any $\delta > 0$, $F_{\mB, I}^2$ satisfies
    \begin{equation}
      \label{FICl3bds}
      \int_{t_0}^{t_1} \int_{D^C_t} |F_{\mB, I}^2| |X\psi_C^I|\, dt
      \lesssim \delta\sum_{|J| \leq |I|}  S_J^C(t_1) + \left(\frac{1}{\delta} +1\right)c_0(\epsilon_0)
      \epsilon_C^2.
    \end{equation}

    }
%
%
%
%

\end{lemma}

\begin{proof}

	\textit{Part 1: Estimates for $\gamma$}

	We start by noting that the first bound in \eqref{pert0center} implies
	the second one, because if the first bound holds,
	\begin{equation}
	 |X^\ell| |\gamma| \lesssim \frac{\epsilon_C}{(1+s)^{1/2}} \lesssim \epsilon_C |X^n|,
	 \label{}
	\end{equation}
	for both $X = X_{ C}, X_{ T}$. To prove the first bound in \eqref{pert0center},
	we use \eqref{gammaboundscentral0}-\eqref{gammaboundscentral} from Lemma \ref{higherordereqncentral}, which give
	\begin{equation}
	 |\gamma|\lesssim \frac{1}{1+v} |\pa \psi_C| + \frac{1}{(1+v)^2},
	 \label{}
	\end{equation}
	so by \eqref{pwcentral}, we have
	\begin{equation}
	(1+v)(1+s)^{1/2} |\gamma| \lesssim (1+s)^{1/2} |\pa \psi_C| + \frac{(1+s)^{1/2}}{1+v}
	\lesssim \epsilon_C + c_0(\epsilon_0).
	 \label{}
	\end{equation}
	where $c_0(0) = 0$. Taking $\epsilon_0$ small
	enough that $c_0(\epsilon_0) \leq \epsilon_C$ gives the first
	bound in \eqref{pert0center}.

	We now prove the time-integrated bounds.
	For this we use
	the bound in \eqref{gammaboundscentral} which gives
	\begin{multline}
	 (1+s)^{1/2}(1+v) \left((1+s) |\pa_u \gamma| +|\gamma|  + (1+v) |\pa_v \gamma| +(1+v) |\nas \gamma|\right)
	 \\
	 \lesssim \sum_{|I|\leq 1} (1+s)^{1/2} |\pa Z_{\mB}^I \psi_C|
	 + \frac{(1+s)^{1/2}}{(1+v)^2}
	 \lesssim \epsilon_C + c_0(\epsilon_0)
	 \label{}
	\end{multline}
	where $c_0(0) = 0$
	and where we used the pointwise bound
	\eqref{pwcentral}. In particular this gives
	\begin{multline}
	 |\nabla \gamma| + \frac{1}{1+s} |\gamma|
	 +(1+v)^{1/2}(1+s)^{1/2} |\nabla_{\evmB} \gamma| + (1+v)^{1/2}(1+s)^{1/2} |\nas \gamma|
	 \\
	 \lesssim \frac{1}{1+t} \frac{1}{(1 + \log t)^{3/2}}
	 \left(\epsilon_C + c_0(\epsilon_0)\right),
	 \label{}
	\end{multline}
	which gives \eqref{intgammaC} after taking $\epsilon_0$ smaller if needed.
	Here we bounded $|\nabla_{\evmB}\gamma| \lesssim |\pa_v \gamma| + \frac{1}{vs^{1/2}} |\pa_u \gamma|
	+ \frac{1}{v} |\gamma|$.

    \textit{Part 2: Estimates for the currents $P_{I, C}, P_{I, null}$}
	 We start with the bound
	\begin{multline}
        \int_{t_0}^{t_1} \| |X^n_{\mB}|^{1/2} \nabla P_{I,C} \|_{L^2(D^C_t)}
        + \||X^n_{\mB}|^{1/2} (1+s)^{-1} P_{I,C}\|_{L^2(D^C_t)}\, dt
	 \\
     + \int_{t_0}^{t_1} \| |X^\ell_{\mB}|^{1/2} \nabla_{\evmB} P_{I,C} \|_{L^2(D^C_t)}
     + \||X^\ell_{\mB}|^{1/2} \nas P_{I,C}\|_{L^2(D^C_t)}\, dt\\
	 \lesssim
	 \sum_{|J| \leq 1}
	 \int_{t_0}^{t_1} \frac{1}{1+t} \frac{1}{(1+ \log t)^{3/2}}
	 \|(1+v)(1+s)^{1/2} |X^n_{\mB}|^{1/2} Z_{\mB} P_{I,C}\|_{L^2(D^C_t)}\, dt,
	 \label{pi2bdstart}
	\end{multline}	where we used $(1+s) |\pa q| +(1+v) |\pa_v q| + (1+v) |\nas q|
		\lesssim |Z_{\mB} q|$ and that $(1+v)^{-1/2} \lesssim (1+s)^{1/4}
		\lesssim |X^n_{\mB}|^{1/2}$
		for both our multipliers.

        Using the estimate \eqref{ffspik} to control $Z_{\mB} P_{I,C}$, we have
	\begin{multline}
        (1 + v) (1+s)^{1/2} |X^n_{\mB}|^{1/2} |Z_{\mB} P_{I,C}|\\
	 \lesssim
	 \sum_{\substack{|I_1| + |I_2| \leq |I| + 1,\\ |I_1|, |I_2| \leq |I|}}
	 (1+s)^{1/2} |\pa \psi^{I_1}_C| \left(|X^n_{\mB}|^{1/2} |\pa \psi^{I_2}_C|\right)
	+ \sum_{|J| \leq |I|}  \frac{1}{(1+s)^{1/2}}
	\left(|X^n_{\mB}|^{1/2} |\pa \psi^J_C|\right)\\
	\lesssim
	\left(\epsilon_C + c_0(\epsilon_0)\right)
	\sum_{|J| \leq |I|} |X^n_{\mB}|^{1/2} |\pa \psi^J_C|,
	 \label{pwPIbdDC}
	\end{multline}
	using the bound \eqref{perturbativeprf} and bounding
	$(1+s)^{-1/2} \leq c_0(\epsilon_0)$. Since
	$\||X^n_{\mB}|^{1/2} \pa\psi^J_C\|_{L^2(D^C_t)}
	\lesssim \epsilon_C (1+ \log\log t)$ if $|J| \leq N_C-1$ and $X = X_C$
	or $|J| \leq N_C$ and $X = X_T$ (the factor $\log \log t$ is only needed
	for the case $|J| = N_C-1, X = X_C$), inserting \eqref{pwPIbdDC} into
	\eqref{pi2bdstart} we find
	\begin{equation}
	 \int_{t_0}^{t_1} \frac{1}{1+t} \frac{1}{(1+\log t)^{3/2}}
     \|(1+v)(1+s)^{1/2} |X^n_{\mB}|^{1/2} Z_{\mB} P_{I,C}\|_{L^2(D^C_t)}\, dt
	 \lesssim (\epsilon_C + c_0(\epsilon_0))\epsilon_C,
	 \label{}
	\end{equation}
	which is bounded by the right-hand side of \eqref{centralPIbounds}.
    The bounds for the contribution from $P_{I, null}$, which
    only involve the derivatives $\evmB$ and $\nas$, into \eqref{centralPIbounds}
    follow easily from \eqref{borderlinebd2}.

	The bound \eqref{RPXboundC} follows in a straightforward way
	from the pointwise bound \eqref{PIkbdcentral}. We omit the proof.

	 \textit{Part 3: Bounds for $F_{C, I}$,  $F_{\Sigma, I}$, $F_{\mB, I}^2$}

	We now move onto controlling the remainder terms on the right-hand
	side of \eqref{higherordereqcentral}.
	 We recall from Lemma \ref{higherordereqncentral} that these quantities
	 satisfy the following bounds. First, $F_{C,I}$ collects various nonlinear error
	  terms and satisfies
	\begin{multline}
   |F_{C,I}| \lesssim
 	\frac{1}{(1+v)^2}
 	\sum_{|I_1| + |I_2| \leq |I|} |\pa \psi_C^{I_1} | |\pa \psi_C^{I_2}|
 	+ \frac{1}{(1+v)^4}
 	\sum_{|I_1| + |I_2| \leq |I|} |\psi_C^{I_1} | | \psi_C^{I_2}|\\
 	+ \frac{1}{(1+v)^2} \sum_{|J| \leq |I|} |\pa \psi_C^J|
 	+ \frac{1}{(1+v)^2(1+s)} \sum_{|J| \leq |I|} |\psi_C^J|.
   \label{FInonlins9}
  \end{multline}
	The remainder
   $F_{\Sigma, I}$ collects the error terms involving the model
  profile $\Sigma$ alone and satisfies
  \begin{equation}
   |F_{\Sigma, I}| \lesssim \frac{1}{(1+v)^2},
   \label{FSigmabds9}
  \end{equation}
	Finally, $F_{\mB, I}$ collects the error terms that we generated when we
	commuted the angular Laplacian with our fields.
   With $Z_{\mB}^I = X^k \Omega^K$, where $X^k$
  denotes an arbitrary $k$-fold product of the fields
  $X \in \{X_1, X_2\}$, it satisfies
  \begin{equation}
 	 |F_{\mB, I}| \lesssim \frac{1}{1+v} \sum_{j \leq k-1} \sum_{|J| \leq |K|+1}
 	 |\nas X^j \Omega^J \psi_C|
   \label{fsdeltabd}
  \end{equation}

We start by proving the bound for $F_{C,I}$ in \eqref{FICl2bds}. The contribution
from the terms on the first line of \eqref{FInonlins9} is straightforward to
handle so we skip it. For the contribution from the terms on the second line, we bound
\begin{multline}
	\sum_{|J| \leq |I|}
 \int_{t_0}^{t_1}
 \left(\int_{D^C_t} \frac{1}{(1+v)^4}|X^\ell| |\pa \psi_C^{J}|^2\right)^{1/2}
 + \sum_{|J| \leq |I|}\int_{t_0}^{t_1}
 \left(\int_{D^C_t} \frac{1}{(1+v)^4(1+s)^2}
 |X^\ell| | \psi_C^{J}|^2\right)^{1/2}\\
 \lesssim
 \sum_{|J| \leq |I|}
 \int_{t_0}^{t_1} \frac{1}{(1+t)^{3/2}}
 \left( \int_{D^C_t} |\pa \psi_C^J|^2 + \frac{1}{(1+s)^2}
 |\psi_C^J|^2\right)^{1/2}\, dt,
 \label{}
\end{multline}
and this is easily bounded by the right-hand side of
\eqref{FInonlins9} after using the bootstrap assumptions
 \eqref{centralboottop}-\eqref{centralbootlower} for the energies
 \eqref{EtopC}-\eqref{loworderenergydef}
and additionally using \eqref{hardyC} to control $\|\psi^J_C\|_{L^2(D^C_t)}$.

For the remainder $F_{I,\Sigma}$ we just use that $Vol(D^C_t) \lesssim
s^{1/2} \lesssim (1+v)^{1/5}$ (recall that we are using the measure $r^{-2}dx$),
 and use \eqref{FSigmabds9} to bound
\begin{equation}
 \int_{t_0}^{t_1} \left(\int_{D^C_t} |X^\ell| |F_{\Sigma, I}|^2\right)^{1/2}\, dt
 \lesssim
 \int_{t_0}^{t_1} \left(\int_{D^C_t} \frac{1}{(1+v)^3}\right)^{1/2}\, dt
 \lesssim
 \int_{t_0}^{t_1} \frac{1}{(1+t)^{5/4}}\, dt
 \lesssim c_0(\epsilon_0),
 \label{}
\end{equation}
which completes the proof of \eqref{FICl2bds}.

{We now control the contribution from $F_{I, \mB}^2$. By \eqref{nonlinearfmBI2},
\begin{equation}
  \label{fmBIbdbootstrappf}
  |F_{\mB, I}^2| \lesssim
  \frac{1+s}{(1+v)^2} \sum_{|J| \leq |I|} |\nas \psi^J_C| + \frac{1}{(1+v)^2} \sum_{|J| \leq |I|-1}
  |\Omega \psi^J_C|.
\end{equation}
Bounding $|X\psi^I_C| \lesssim (1+v)|\evmB \psi^I_C| + (1+s) |\pa \psi^I_C|$ for 
either of our multipliers $X = X_C, X_T$, for $|J| \leq |I|$ we have
\begin{multline}
  \label{}
  \int_{t_0}^{t_1} \int_{D^C_t} \frac{1+s}{(1+v)^2} |\nas \psi^J_C|  |X\psi_C^I|\, dt
  \\
  \lesssim
  \int_{t_0}^{t_1} \frac{1+\log t}{(1+t)^{3/2}} 
  \left(\int_{D^C_t} |\nas \psi_C^J| |v^{1/2}\evmB \psi_C^I|\right)\,dt
  + \int_{t_0}^{t_1} \frac{(1+\log t)^{3/2}}{(1+t)^2} 
  \left(\int_{D^C_t} |\nas\psi_C^J||X^n|^{1/2}|\pa \psi_C^I|\right)\, dt
  \\
  \lesssim \delta \int_{t_0}^{t_1} \int_{D^C_t} |\nas \psi_C^J|^2\, dt
   + \frac{1}{\delta} c_0(\epsilon_0) \sup_{t_0 \leq t \leq t_1} \int_{D^C_t} |\pa \psi_C^J|^2_{X,\mB}
      \lesssim \delta S_{J}^C(t_1) + \frac{c_0(\epsilon_0)}{\delta} \epsilon_C^2.
\end{multline}

As for the second term in \eqref{fmBIbdbootstrappf}, we use the Poincar\'{e}-type inequality 
\eqref{hardyCapp} combined with our bootstrap assumptions to bound
\begin{multline}
  \label{}
  \int_{D^C_t} |\Omega \psi_C^J|^2 \\
  \lesssim (\log t)^{1/2}\int_{D^C_{t_0}} |\Omega \psi_C^J|^2
  + (\log t)^{3/2}\int_{t_0}^{t_1} \int_{\Gamma^L_{t'}} v |\evmB \Omega \psi_C^J|^2 
  + \frac{1}{vs} |\pa_u \psi_C^J|^2\, dS dt
  + \log t \int_{D^C_t} |\pa \psi_C^J|^2
  \\
  \lesssim (\log t)^{3/2}\epsilon_C^2,
\end{multline}
for $|J| \leq |I|-1 \leq N_C$. It easily follows that
\begin{multline}
  \label{}
  \int_{t_0}^{t_1} \int_{D^C_t} \frac{1}{(1+v)^2} |\Omega \psi_C^J| |X\psi_C^I|\,dt
  \lesssim
  \int_{t_0}^{t_1} \int_{D^C_t} \frac{1}{(1+v)^{3/2}} |\Omega \psi_C^J|
  |\pa \psi_C^I|_{X, \mB}
  \, dt
  \lesssim
 \epsilon_C^2 \int_{t_0}^{t_1} \frac{(1+\log t)^{3/2}}{(1+t)^{3/2}} \, dt
 \\
 \lesssim
  c_0(\epsilon_0) \epsilon_C^2.
\end{multline}

}  
\end{proof}
We now handle the contribution from the component $P^u_{I, null}$, which is
responsible for the double-logarithmic growth in some of our estimates.
\begin{lemma}
	\label{timeintegrability-center-linear}
	Let $X = X_{C}$ or $X = X_{T}$ with notation as
	in Section \ref{fields} and write
	$X = X^n \pa_u + X^\ell \evmB$.
	There is $\epsilon_0^*$ so that if the hypotheses of Theorem \ref{mainthm} hold
	with $\epsilon_0 < \epsilon_0^*$, the $u$-component of $P_{I, null}$,
	defined in Lemma \ref{higherordereqncentral} satisfies
	the following bounds.
	\begin{itemize}
	 \item
	 For $|I| \leq N_C$ and any $\delta > 0$,
	\begin{equation}
	 \int_{t_0}^{t_1} \int_{D^C_t}
	 |X^n_{T, \mB}|^{1/2} \left( |\nabla P^u_{I,null}| + (1+s)^{-1} |P^u_{I,null}|\right)
	 |\pa \psi^I_C|_{X_T, \mB}\, dt
	 \lesssim \delta S_I^C(t_1) + \frac{1}{\delta} c_0(\epsilon_0) \epsilon_C^2,
	 \label{extraPIlinbd0}
	\end{equation}
where $S_I^C$ is the spacetime integral defined in \eqref{EtopC0}.

\item
For $|I| \leq N_C-1$ and any $\delta >0$,
\begin{multline}
 \int_{t_0}^{t_1} \int_{D^C_t} |X^n_{D, \mB}|^{1/2}
 \left( |\nabla P^u_{I,null}| + (1+s)^{-1}|P^u_{I,null}|\right)
 |\pa \psi^I_C|_{X_C, \mB}\, dt\\
 \lesssim
 \delta E_{I, D}^C(t_1) 
 + \frac{1}{\delta} \sum_{|K| \leq |I|-1} E_{K, D}^C(t_1)
+ \left(\delta + \frac{1}{\delta}\right) c_0(\epsilon_0) \epsilon_C^2
 +  \epsilon_C
 (1+\log\log t_1)
 \sum_{|J| \leq |I|-1} \sup_{t_0 \leq t \leq t_1}
 \left(E_{J, D}^C(t)\right)^{1/2},
 \label{extraPIlinbd}
\end{multline}
\item
for $|I| \leq N_C-2$, and any $\delta >0$,
\begin{multline}
 \int_{t_0}^{t_1} \int_{D^C_t}
 |X^n_{D, \mB}|^{1/2}
 \left( |\nabla P^u_{I, null}| + (1+s)^{-1}|P^u_{I, null}|\right)
 |\pa \psi^I_C|_{X_C, \mB}\, dt
 \\
 \lesssim 
 \delta E_{I, D}^C(t_1) 
 + \frac{1}{\delta} \sum_{|K| \leq |I|-1} E_{K, D}^C(t_1)
+ \left(\delta + \frac{1}{\delta}\right) c_0(\epsilon_0) \epsilon_C^2
 \label{extraPIlinbd1}
\end{multline}

	\end{itemize}

We also have
\begin{multline}
 \sup_{t_0 \leq t \leq t_1}\int_{D^C_t} v |P_{I, null}|^2
 + \int_{t_0}^{t_1} \int_{\Gamma^L_t} v|P_{I, null}|^2\, dS dt
 + \int_{t_0}^{t_1} \int_{\Gamma^R_t} v|P_{I, null}|^2\, dS dt\\
 \lesssim
 c_0(\epsilon_0) \epsilon_C^2 + \sum_{|J| \leq |I|-1}
 \sup_{t_0 \leq t \leq t_1} E_{X_C, J}^C(t).
 \label{extraPIslicebd}
\end{multline}

\end{lemma}

\begin{proof}

	We recall that for our
	multipliers $X \in \{X_{C}, X_{T}\}$,
	by definition
	\begin{equation}
	 |\pa q|_{X, \mB}
	= |X^n_{\mB}|^{1/2} |\pa q|+ v^{1/2} (|\evmB q| + |\nas q|).
	 \label{}
	\end{equation}
	We also recall that from Lemma \ref{higherordereqncentral},
	the components $P^u_{I, null}$ enjoy the following estimates,
	\begin{align}
|P_{I, null}^u|
&\lesssim
\frac{1}{1+v}
\sum_{|J| \leq |I|-1}\left(
\frac{1}{(1+s)^{1/2}} |\pa \psi_C^J| + |\opa \psi_C^J|\right)
+  \frac{1}{1+v}\sum_{|J| \leq |I|-2} |\pa \psi_C^J|,
\label{borderlinebd0pf9}
\end{align}
\begin{multline}
(1+s)|\nabla P_{I, null}^u| + (1+v)|\pa_v P_{I, a}^u|
+ |\Omega P_{I, null}^u| \\
\lesssim
\frac{1}{1+v}\sum_{|J| \leq |I|}\left(
\frac{1}{(1+s)^{1/2}} |\pa \psi_C^J| + |\opa \psi_C^J|\right)
+ \frac{1}{1 + v}\sum_{|J| \leq |I|-1} |\pa \psi_C^J|,
\label{ffsderivbd0pf9}
\end{multline}
where $\opa = (\nas, \evmB)$. Note the above imply
\begin{align}
  \label{}
   |P_{I, null}^u|
&\lesssim
\frac{1}{(1+v)^{3/2}} \sum_{|J| \leq |I|-1} |\pa \psi_C^J|_{X, \mB}
+
\frac{1}{(1+v)(1+s)^{1/2}}
\sum_{|J| \leq |I|-1} |\pa \psi_C^J| 
+  \frac{1}{1+v}\sum_{|J| \leq |I|-2} |\pa \psi_C^J|,
\end{align}
and
\begin{align}
(1+s)|\nabla P_{I, null}^u| + (1+v)|\pa_v P_{I, a}^u|
+ |\Omega P_{I, null}^u|
&\lesssim \frac{1}{(1+v)^{3/2}} \sum_{|J| \leq |I|} |\pa \psi_C^J|_{X, \mB}\\
&\quad+
\frac{1}{(1+v)(1+s)^{1/2}}
\sum_{|J| \leq |I|-1} |\pa \psi_C^J| 
+  \frac{1}{1+v}\sum_{|J| \leq |I|-1} |\pa \psi_C^J|.
\end{align}

By these estimates, for either multiplier $X$ we have
\begin{align}
  \label{PInullintegrated}
  \int_{t_0}^{t_1}\int_{D^C_t} |X^n_{\mB}|^{1/2} &\left(|\nabla P^u_{I,null}| 
  + (1+s)^{-1} |P^u_{I, null}| \right) |\pa \psi_C^I|_{X, \mB}\, dt
  \\
    &\lesssim
        \sum_{|J| \leq |I|}
        \int_{t_0}^{t_1} \int_{D^C_t} 
        \frac{|X^n_{\mB}|^{1/2}}{1+s} \frac{1}{(1+v)^{3/2}}|\pa \psi_C^J|_{X, \mB}^2\, dt
        \\
    &+
        \sum_{|J| \leq |I|}
        \int_{t_0}^{t_1} \int_{D^C_t} \frac{1}{(1+v)(1+s)^{3/2}} 
        |X^n_{\mB}|^{1/2}|\pa \psi_C^J| |\pa \psi_C^I|_{X, \mB}\, dt
        \\
    &+
        \sum_{|J| \leq |I|-1} 
        \int_{t_0}^{t_1} \int_{D^C_t} \frac{1}{(1+v)(1+s)} |X^n_{\mB}|^{1/2}
        |\pa \psi_C^J| |\pa \psi_C^I|_{X, \mB}\, dt.
\end{align}
By definition $|X^n_{\mB}|^{1/2}|\pa \psi_C^J|\lesssim |\pa \psi_C^J|_{X, \mB}$,
and since both our multipliers satisfy $|X^n_{\mB}| \lesssim 1+s$, 
for the terms on the second and third lines we can just bound
\begin{multline}
  \label{}
  \sum_{|J| \leq |I|}
  \int_{t_0}^{t_1} \int_{D^C_t} 
        \frac{|X^n_{\mB}|^{1/2}}{1+s} \frac{1}{(1+v)^{3/2}}|\pa \psi_C^J|_{X, \mB}^2
        +\frac{1}{(1+v)(1+s)^{3/2}} 
        |X^n_{\mB}|^{1/2}|\pa \psi_C^J| |\pa \psi_C^I|_{X, \mB} \, dt
        \\
    \lesssim
    \left(\int_{t_0}^{t_1} \frac{1}{(1+t)(1 + \log t)^{9/8}}\,dt\right) \frac{1}{1+ \log \log t_1} \sup_{t_0\leq t \leq t_1} 
    \left(
 \sum_{|J| \leq |I|}
    \int_{D^C_t} |\pa \psi_C^J|^2_{X, \mB}\right)
    \lesssim c_0(\epsilon_0)\epsilon_C^2,
\end{multline}
where the double-logarithmic factor is only needed in the case $X = X_C$ and $|I| = N_C - 1$,
and where we bounded $(1+s)^{-3/2}|X^n_{\mB}|^{1/2}(1+\log \log t)
\lesssim (1+\log t)^{-3/2}(1+\log t)^{1/4}(1+\log \log t) \lesssim (1+\log t)^{-9/8}$.

It remains to control the terms on the last line of \eqref{PInullintegrated}. 
This is straightforward when $X = X_T$, since then we have $|X^n_{T, \mB}|^{1/2}
\lesssim (1+s)^{-1/4}$ and in that case
\begin{multline}
  \label{}
    \sum_{|J| \leq |I|-1} 
        \int_{t_0}^{t_1} \int_{D^C_t} \frac{1}{(1+v)(1+s)} |X^n_{T,\mB}|^{1/2}
        |\pa \psi_C^J| |\pa \psi_C^I|_{X_T, \mB}\, dt
        \\
        \lesssim
            \int_{t_0}^{t_1} \int_{D^C_t} \frac{1}{(1+v)(1+s)^{3/2}}
            \left(\sum_{|J| \leq |I|-1} |\pa \psi_C^J|_{X_C, \mB}\right)|\pa \psi_C^I|_{X_T, \mB}^2\, dt
            \lesssim c_0(\epsilon_0) \epsilon_C^2,
\end{multline}
after, similarly to the above, bounding $\int_{t_0}^{t_1} \frac{1 +\log\log t}{(1+t)(1+ \log t)^{3/2}} 
\lesssim c_0(\epsilon_0)$.

The argument is more complicated when $X = X_C$, because we cannot afford to directly
use the bootstrap assumptions to handle this term, as that would lead to a bound
of size $\epsilon_C^2 (1+ \log \log t)^2$, which is too large for our purposes.
We are going to instead prove the improved estimate: for any $\delta > 0$
and $|J| \leq |I| \leq N_C-1$,
\begin{multline}
 \sum_{|J| \leq |I|-1}
 \int_{t_0}^{t_1} \int_{D^C_t} \frac{1}{(1+v)(1+s)} 
 |X^n_{C, \mB}|^{1/2} |\pa \psi^J_C| |\pa \psi^I_C|_{X_C, \mB}\, dt\\
 \lesssim \delta E_{I, D}^C(t_1) 
 + \frac{1}{\delta} \sum_{|K| \leq |I|-1} E_{K, D}^C(t_1)
+ \left(\delta + \frac{1}{\delta}\right) c_0(\epsilon_0) \epsilon_C^2
 +
 \epsilon_C (1 + \log \log t_1)^{a_I} \sum_{|K| \leq |I|-1}
 \sup_{t_0 \leq t \leq t_1} (E_{K, D}^C(t))^{1/2},
 \label{extraPIlinbdleftover}
\end{multline}
with $a_I = 1$ when $|I| = N_C-1$ and $a_I = 0$ otherwise.

The idea is to exploit the fact that since we only need
to consider $|J| \leq |I|-1 \leq N_C -2$, we can afford to integrate to
the shock using Lemma \ref{hardyClemma}. Since the domain has width $\sim s^{1/2}$ and since we control
$s \pa_u$ applied to the solution, the interior term we generate
is easily handled. It turns out that the boundary term this generates
is exactly of the form controlled by our energy, which allows us to close the estimate.
When $|I| = N_C-1$ there is the added complication that we cannot afford
to integrate in both factors because the bounds we have for the
top-order energies $E^C_{T}$ are not strong enough to control
the resulting quantities (recall that $E^C_{T, I}
\gtrsim (1+\log t)^{-1/2}\|\pa \psi_C^I\|_{L^2(D^C_t)}^2$).

By the bound \eqref{hardyCapp0} from Lemma \ref{hardyClemma} and the fact that
$(1 + \log t)^{1/2}|\pa q|\lesssim
(1+ \log t)^{-1/2} \sum_{Z_{\mB} \in \mZB} |Z_{\mB} q|$,
we have the bound
\begin{equation}
  \label{}
  \int_{D^C_t} |q|^2 \lesssim \int_{\Gamma^L_t} (1 + \log t)^{1/2} |q|^2\, dS
  + \frac{1}{1+ \log t} \sum_{Z_{\mB} \in \mZB}  \int_{D^C_t} |Z_{\mB} q|^2,
\end{equation}
and in particular, since $E_{K, D}(t) \gtrsim (1+\log t)\int_{D^C_t} |\pa \psi_C^K|^2$,
	\begin{equation}
		\int_{D_t^C} |\pa \psi_C^J|^2
	  \lesssim \int_{\Gamma^L_t} s^{1/2} |\pa \psi_C^J|^2\, dS
		+ \frac{1}{(1 + \log t)^{2}}\sum_{|K| \leq |J|+1}  E^C_{K, D}(t),
	 \label{ibpdC}
	\end{equation}
    as well as the similar estimate
    \begin{equation}
      \label{ibpdC2}
      \int_{D^C_t} |\pa \psi_C^J|_{X, \mB}^2
      \lesssim
      \int_{\Gamma^L_t} s^{1/2} |\pa \psi_C^J|_{X,\mB}^2\ dS
      + \frac{1}{(1+\log t)^2}\sum_{|K| \leq |J|+1} E^C_{K, D}(t), 
    \end{equation}

We can now prove \eqref{extraPIlinbdleftover}.
 We will need to handle the two cases
$|I| = N_C - 1$ and $|I| \leq N_C-2$ separately,
with the first of these being slightly more involved and 
responsible for the (slow) growth of our energies.

\emph{The proof of \eqref{extraPIlinbdleftover} when $|I| = N_C-1$}

For $|J| \leq |I|-1$, since $|X^n_{C, \mB}|\lesssim 1+\log t$, we have 
\begin{align}
    \int_{t_0}^{t_1} \int_{D^C_t}&\frac{1}{(1+v)(1+s)}|X^n_{C,\mB}|^{1/2}
    |\pa \psi_C^J| |\pa \psi_C^I|_{X, \mB}\,
 dt\\
 &\lesssim
 \int_{t_0}^{t_1} \frac{1}{1+t} \frac{1}{(1 + \log t)^{1/2}}
 \left( \int_{D^C_t} |\pa \psi_C^J|^2\right)^{1/2} \left(E^C_{I, D}(t)\right)^{1/2}\, dt\\
 &\lesssim
 \int_{t_0}^{t_1} \frac{1}{1+t} \frac{1}{(1 + \log t)^{1/2}}
 \left( \int_{\Gamma^L_t} s^{1/2} |\pa \psi_C^J|^2\, dS
 + \frac{1}{(1+ \log t)^2} \sum_{|K| \leq |J|+1} E_{K, D}^C(t)\right)^{1/2}
  (E^C_{I,D}(t))^{1/2}\, dt,
 \label{gettinglogs}
\end{align}
by \eqref{ibpdC}.
For the second term here, we just note that by
\eqref{centralbootextra} we have
\begin{equation}
 \int_{t_0}^{t_1} \frac{1}{1+t} \frac{1}{(1 + \log t)^{3/2}}
 \sum_{|K| \leq |J| + 1} E_{K, D}^C(t)^{1/2} E_{I, D}^C(t)^{1/2}\,dt
 \lesssim
 \epsilon_C
 \int_{t_0}^{t_1} \frac{1}{1+t} \frac{1 + \log \log t}{(1 + \log t)^{3/2}}\, dt
 \lesssim c_0(\epsilon_0)\epsilon_C.
 \label{}
\end{equation}
For the first term in \eqref{gettinglogs}, we bound
\begin{multline}
 \int_{t_0}^{t_1} \frac{1}{1+t} \frac{1}{(1+ \log t)^{1/2}}
 \left(\int_{\Gamma^L_t} s^{1/2} |\pa \psi^J_C|^2\, dS\right)^{1/2}
 (E^C_{I, D}(t))^{1/2}\, dt\\
 \lesssim
 \epsilon_C(1 + \log \log t_1)^{1/2}
 \left( \int_{t_0}^{t_1} \frac{1}{1+t} \frac{1}{1+\log t}\right)^{1/2}
 \left(\int_{t_0}^{t_1}
  \int_{\Gamma^L_t} \frac{s^{1/2}}{v} |\pa \psi_C^J|^2\,
	dS\, dt\right)^{1/2}\\
	\lesssim
	\epsilon_C (\log\log t_1)
	 \sum_{|K| \leq |I| - 1} \sup_{t_0 \leq t \leq t_1}
	E_{K, D}^C(t)^{1/2},
 \label{}
\end{multline}
and combining this with the previous inequality we get
\eqref{extraPIlinbdleftover}.
We remark that it is to handle this term that we needed to allow
the norms $E_{I, D}^C(t)$ to grow slowly when $|I| = N_C - 1$.

\textit{The proof of \eqref{extraPIlinbdleftover} when $|I| \leq N_C - 2$}

The argument in this case is similar but a bit simpler, because we can afford
to integrate as in \eqref{ibpdC} in both factors. 
For $|I| \leq N_C-2$ and $|J| \leq |I|-1$, 
 we first bound
$|X^n_{C,\mB}|^{1/2} |\pa \psi_C^J| |\pa \psi_C^I|_{X, \mB}
\leq |\pa\psi_C^J| |\pa \psi_C^I|_{X, \mB}$,
so for any $\delta > 0$ we have
\begin{multline}
  \label{}
  \int_{t_0}^{t_1} \int_{D^C_t} \frac{1}{(1+v)(1+s)}  |X^n_{C,\mB}|^{1/2}
  |\pa \psi_C^J| |\pa\psi_C^I|_{X, \mB}\, dt
  \leq
  \int_{t_0}^{t_1} \int_{D^C_t} \frac{1}{(1+v)(1+s)}  |\pa \psi_C^J|_{X, \mB} |\pa\psi_C^I|_{X, \mB}
  \, dt
  \\
  \lesssim
  \delta\int_{t_0}^{t_1}
\int_{D^C_t}\frac{1}{(1+v)(1+s)} |\pa \psi_C^I|_{X, \mB}^2\, dt
+ \frac{1}{\delta}  \int_{t_0}^{t_1} \int_{D^C_t}\frac{1}{(1+v)(1+s)}  |\pa \psi_C^J|_{X, \mB}^2\, dt.
\end{multline}
Using \eqref{ibpdC2}, we find
\begin{multline}
  \label{}
  \int_{t_0}^{t_1} \int_{D^C_t} \frac{1}{(1+v)(1+s)}  |\pa \psi_C^K|_{X, \mB}^2\, dt
  \,dt
  \\
  \lesssim 
   \int_{t_0}^{t_1} \int_{\Gamma^L_t} \frac{1}{(1+v)(1+s)^{1/2}} 
  |\pa \psi_C^K|_{X, \mB}^2\, dS dt
  + \sum_{|K'| \leq |K|+1} \int_{t_0}^{t_1} \frac{1}{(1+t)(1+\log t)^3} E_{K', D}^C(t)\, dt
  \\
  \lesssim  E_{K, D}^C(t_1) + 
  c_0(\epsilon_0)\epsilon_C^2,
\end{multline}
for $|K| \leq N_C-2$. From the above bounds we find
\begin{multline}
  \label{}
    \int_{t_0}^{t_1} \int_{D^C_t} \frac{1}{(1+v)(1+s)}
    |X^n_{C,\mB}|^{1/2} |\pa \psi_C^J| |\pa\psi_C^I|_{X, \mB}\, dt
    \\
    \lesssim \delta E_{I, D}^C(t_1) + \frac{1}{\delta}\sum_{|J| \leq |I|-1} E_{J, D}(t_1)
    + \left(\delta + \frac{1}{\delta}  \right) c_0(\epsilon_0)\epsilon_C^2, 
\end{multline}
which concludes the proof of \eqref{extraPIlinbdleftover}.

It remains only to prove 
\eqref{extraPIslicebd}.
This follows after using the simple estimate
\begin{equation}
 |P_{I, null}| \lesssim \frac{1}{1+v} \sum_{|J|\leq |I|-1} |\pa \psi^J_C|,
 \label{}
\end{equation}
which follows directly from the estimate
\eqref{borderlinebd2} from Lemma \ref{higherordereqncentral},
and then bounding
\begin{equation}
 \int_{D^C_t} (1+v)|P_{I, null}|^2
 \lesssim
 \frac{1}{1+t}
 \sum_{|J| \leq |I|}
 \int_{D^C_t} |\pa \psi^J_C|^2\, dt
 \lesssim c_0(\epsilon_0) \epsilon_C^2,
 \label{}
\end{equation}
and, for $A = L, R$,
\begin{multline}
 \int_{t_0}^{t_1} \int_{\Gamma^A_t}
 (1+v)|P_{I,null}|^2\,dS dt
 \lesssim
 \sum_{|J| \leq |I|-1}
 \int_{t_0}^{t_1} \int_{\Gamma^A_t}
 \frac{1}{1+v}|\pa \psi_C^J|^2\,dS dt\\
 \lesssim
 \frac{1}{(1 + \log t_0)^{1/2}}
 \int_{t_0}^{t_1} \int_{\Gamma^A_t}
 \frac{(1+s)^{1/2}}{1+v} |\pa \psi_C^J|^2\, dS dt
 \lesssim c_0(\epsilon_0) \epsilon_C^2.
 \label{}
\end{multline}
Here, we used \eqref{rightangularbound} to control the angular derivatives
along the shock.

\end{proof}

{We now control the term $F_{\mB, I}^1$.
\begin{lemma}
    \label{FImBibplemma}
    Under the hypotheses of Theorem \ref{mainthm}, for either $X = X_T$ or $X = X_C$ and $|I|\leq N_C$,
    for any $\delta > 0$, we have
    \begin{multline}
      \label{FImBibpargument}
      -\int_{t_0}^{t_1} \int_{D^C_t} F_{I, \mB}^1 \, dt
      \\
      \lesssim
      c_0(\epsilon_0) \epsilon_C^2 + \delta\sum_{|J|\leq |I|}
      \left(E_{J, X_T}(t_1)  + S_{I}^C(t_1)\right)
      +\frac{1}{\delta} \sum_{|J| \leq |I|-1}
      E_{J, X_T}(t_1) + \left(1+\frac{1}{\delta}  \right) \sum_{|J| \leq |I|-1} S_{J}^C(t_1) + \epsilon_C^3.
    \end{multline}
\end{lemma}

\begin{proof}
    
    By the definition \eqref{nonlinearfmBI1} of $F_{\mB, I}^1$,
\begin{multline}
  \label{adhoc00}
  -\int_{t_0}^{t_1}\int_{D^C_t} F_{\mB, I}^1 X \psi^I \, dt
  = \sum_{|J_1| + |J_2| \leq |I|-1}
  \int_{t_0}^{t_1} \int_{D^C_t} a^I_{J_1J_2} \sDelta Z_{\mB}^J \psi_C
  (v\pa_v Z_{\mB}^{J_1} (v\pa_v)Z_{\mB}^{J_2}\psi_C)\, dt
  \\
  +
  \sum_{|J_1| + |J_2| \leq |I|-1}
  \int_{t_0}^{t_1} \int_{D^C_t} a^I_{J_1J_2} \sDelta Z_{\mB}^J \psi_C
  (X^u\pa_u Z_{\mB}^{J_1}(v\pa_v) Z_{\mB}^{J_2}\psi_C)\, dt,
\end{multline}
where $a^I_{J_1J_2} = 1$ if $Z_{\mB}^I = Z_{\mB}^{J_1} (v\pa_v) Z_{\mB}^{J_2}$ and $a^I_{J_1J_2} = 0$
otherwise, and where $Z_{\mB}^J = Z_{\mB}^{J_1}Z_{\mB}^{J_2}$.

We start by dealing with the first term in \eqref{adhoc00}. 
For this, write
\begin{multline}
  \label{adhoc0}
    \sDelta Z_{\mB}^J \psi_C
    (v\pa_v Z_{\mB}^{J_1} (v\pa_v)Z_{\mB}^{J_2}\psi_C) = 
   a^I_{J_1J_2} \sDelta Z_{\mB}^J \psi_C (v\pa_v)^2Z_{\mB}^J\psi_C
  \\
  +  a^I_{J_1J_2} \sDelta Z_{\mB}^J\psi_C 
  v\pa_v \left([Z_{\mB}^{J_1}, v\pa_v] Z_{\mB}^{J_2}\psi_C\right).
\end{multline}
The first term here needs to be handled carefully.
We write
\begin{align}
  \label{adhoc}
  \sDelta Z^J_{\mB} \psi_C (v\pa_v)^2 Z_{\mB}^J \psi_C
      &= \nas \cdot \left( \nas Z^J_{\mB} \psi_C (v\pa_v)^2 Z_{\mB}^J\psi_C\right)
      - \nas Z^J_{\mB}\psi_C \cdot v\pa_v \nas (v\pa_v Z_{\mB}^J \psi_C)\\
      &\qquad +\nas Z_{\mB}^J\psi_C \cdot [v\pa_v,\nas] v\pa_v Z_{\mB}^J \psi_C\\
      &= \nas \cdot \left(v \nas Z^J_{\mB} \psi_C \pa_v (v\pa_vZ_{\mB}^J \psi_C) \right) 
      - \pa_v \left(v \nas Z^J_{\mB} \psi_C \cdot \nas (v\pa_vZ_{\mB}^J \psi_C)  \right) \\
      &\qquad
      + \nas Z_{\mB}^J \psi_C\cdot \nas (v\pa_vZ_{\mB}^J\psi_C)
      + |\nas (v\pa_v Z_{\mB}^J \psi_C)|^2\\
      &\qquad+ \nas Z_{\mB}^J\psi_C \cdot [v\pa_v, \nas] v\pa_v Z_{\mB}^J \psi_C 
      + [v\pa_v, \nas]Z^J_{\mB} \psi_C \cdot \nas (v\pa_v Z_{\mB}^J \psi_C)
\end{align}

We start by handling the spacetime integrals of the terms on the first line.
By Stokes' theorem,
\begin{multline}
  \label{}
  \int_{t_0}^{t_1} \int_{D^C_t} \nas \cdot \left( \nas Z^J_{\mB} \psi_C (v\pa_v)^2Z_{\mB}^J \psi_C \right) \, dt
  \\
  = \int_{t_0}^{t_1} \int_{\Gamma^L_t} v \nas B^L \cdot \nas Z^J_{\mB} \psi_C \pa_v (v\pa_vZ_{\mB}^J\psi_C)\,
  dSdt
  - \int_{t_0}^{t_1} \int_{\Gamma^R_t} v \nas B^R \cdot \nas Z^J_{\mB} \psi_C \pa_v (v\pa_vZ_{\mB}^J\psi_C)\,
  dSdt.
\end{multline}
Since $|v\nas B^A| \lesssim s^{1/2}$, we have
\begin{multline}
  \label{}
  \int_{t_0}^{t_1} \int_{\Gamma^A_t} |v \nas B^A \cdot \nas Z^J_{\mB} \psi_C \pa_v (v\pa_vZ_{\mB}^J\psi_C)|\,dSdt
  \\
  \lesssim
  \int_{t_0}^{t_1} \int_{\Gamma^A_t} \frac{(1+s)^{1/2}}{(1+v)^{1/2}} 
  |\nas Z^J_{\mB} \psi_C| |v^{1/2} \pa_v (v\pa_vZ_{\mB}^J\psi_C)|\,dSdt
  \lesssim c_0(\epsilon_0)\epsilon_C^2,
\end{multline}
using the bounds \eqref{leftbdytrivialbds}-\eqref{rightbdytrivialbds} for the boundary terms
in the energies.

Also by Stokes' theorem,
\begin{align}
  \label{adhocstokes2}
  \int_{t_0}^{t_1} &\int_{D_t^C} \pa_v
  \big(v \nas Z^J_{\mB} \psi_C \cdot \nas (v\pa_v)Z_{\mB}^J \psi_C  \big)\, dt
  \\
  &= \int_{t_0}^{t_1} \int_{\Gamma^L_t} v\pa_v B^L \nas Z^J_{\mB} \psi_C \cdot \nas (v\pa_v)Z_{\mB}^J \psi_C
  \, dS dt
  - \int_{t_0}^{t_1} \int_{\Gamma^R_t} v\pa_v B^R \nas Z^J_{\mB} \psi_C \cdot \nas (v\pa_v)Z_{\mB}^J \psi_C
  \, dS dt
  \\
  &+ \int_{D^C_{t_1}} v\nas Z^J_{\mB} \psi_C \cdot \nas (v\pa_vZ_{\mB}^J \psi_C)
  - \int_{D^C_{t_0}} v\nas Z^J_{\mB} \psi_C \cdot \nas (v\pa_vZ_{\mB}^J \psi_C).
\end{align}
For the terms along the shocks, we use that $|v\pa_v B^A| \lesssim s^{-1/2}$ and write
$v\pa_v Z_{\mB}^J = Z_{\mB}^{J'}$, which gives
\begin{align}
  \label{}
  \left|\int_{t_0}^{t_1} \int_{\Gamma^A_t} v\pa_v B^A \nas Z^J_{\mB} \psi_C \cdot \nas (v\pa_vZ_{\mB}^J \psi_C)
  \, dS dt
 \right|
 &\lesssim \int_{t_0}^{t_1} \int_{\Gamma^A_t} \frac{1}{(1+s)^{1/2}}
 |\nas Z^J_{\mB}\psi_C| |\nas Z_{\mB}^{J'}\psi_C|\, dS dt\\
 &\lesssim (E_{J, T}^C(t))^{1/2} (E_{J', T}^C(t))^{1/2}
 \lesssim \frac{1}{\delta} E_{J, T}^C(t_1) + \delta E_{J', T}^C(t_1)
\end{align}
for arbitrary $\delta > 0$ (recall here $|J| \leq |I|-1$ and $|J'| = |I|$). 
For the terms in \eqref{adhocstokes2} along the time slices, we just bound
\begin{equation}
  \label{}
  \int_{D^C_t} v |\nas Z^J_{\mB} \psi_C||\nas (v\pa_v Z_{\mB}^J \psi_C)|
  \lesssim (E_{J, T}^C(t))^{1/2} (E_{J', T}^C(t))^{1/2}
  \lesssim \frac{1}{\delta} E_{J, T}^C(t) + \delta E_{J', T}^C(t).
\end{equation}

For the first two terms on the last line of \eqref{adhoc} we just note that, again writing
$v\pa_v Z_{\mB}^J = Z_{\mB}^{J'}$,
\begin{align}
  \label{crucialIBP}
  \int_{t_0}^{t_1} \int_{D^C_t} \nas Z_{\mB}^J \psi_C\cdot \nas &(v\pa_vZ_{\mB}^J\psi_C)
      + |\nas (v\pa_v Z_{\mB}^J \psi_C)|^2\, dt
      \\
      &\geq S_{X_T}[ Z^{J'}_{\mB}\psi_C] - S_{X_T}[ Z^{J'}_{\mB}\psi]^{1/2}S_{X_T}[Z^J_{\mB} \psi_C]^{1/2},
      \\
      &\geq \frac{1}{2} S_{X_T}[ Z^{J'}_{\mB}\psi_C] - \frac{1}{2}  S_{X_T}[Z^J_{\mB}\psi_C]
      \\
      &\geq -\frac{1}{2} S_{X_T}[Z^J_{\mB}\psi_C],
\end{align}
which is the crucial step. This last term is of the correct form since $|J| \leq |I|-1$.

To deal with the terms from \eqref{adhoc} involving
$[v\pa_v, \nas]$, we use \eqref{sdeltaformulaOmega} to write
\begin{equation}
  \label{}
  [v\pa_v, \nas_i] = [v\pa_v, \frac{\omega^j}{r}\Omega_{ij} ] = -\frac{1}{2} \frac{v}{r^2}\omega^j \Omega_{ij} 
  =-\frac{1}{2} \frac{v}{r} \nas_i,
\end{equation}
and since $|J| \leq |I|-1$, we have
\begin{multline}
  \label{}
  \int_{t_0}^{t_1} \int_{D^C_t} |\nas Z_{\mB}^J \psi_C \cdot [v\pa_v, \nas]v\pa_v Z_{\mB}^J\psi_C|\, dt
  \lesssim 
  \int_{t_0}^{t_1} \int_{D^C_t} |\nas Z_{\mB}^J \psi_C| |\nas (v\pa_v Z_{\mB}^J\psi_C)|\, dt
  \\
  \lesssim
  \frac{1}{\delta} \sum_{|K| \leq |I|-1} S_K^C(t_1) + \delta \sum_{|K| \leq |I|} S_{K}^C(t_1),
\end{multline}
and similarly
\begin{equation}
  \label{}
  \int_{t_0}^{t_1} \int_{D^C_t} |[v\pa_v, \nas] Z_{\mB}^J \psi_C \cdot \nas (v\pa_v Z_{\mB}^J\psi_C)|\, dt
  \lesssim 
  \frac{1}{\delta} \sum_{|K| \leq |I|-1} S_K^C(t_1) + \delta \sum_{|K| \leq |I|} S_{K}^C(t_1),
\end{equation}

It remains only to deal with the second term in \eqref{adhoc00} and
the second term in \eqref{adhoc0}. For the former, we just bound, for any $\delta > 0$,
\begin{align}
  \label{}
  \int_{t_0}^{t_1} \int_{D^C_t} |\nas^2 Z_{\mB}^J\psi_C|& |X^u| |\pa Z_{\mB}^{I} \psi_C|\, dt
    \\
    &\lesssim
   \delta \int_{t_0}^{t_1} \int_{D^C_t} |\nas \Omega Z_{\mB}^J \psi_C|^2\, dt
    + \frac{1}{\delta} \int_{t_0}^{t_1}\frac{1}{(1+t)^2}  
    \left(\int_{D^C_t} |X^u| |\pa Z_{\mB}^I \psi_C|^2\right)\, dt
    \\
    &\lesssim \delta S_{X_T}[Z_{\mB}^J\psi_C] + \frac{1}{\delta} c_0(\epsilon_0) \epsilon_C^2.
\end{align}

As for the latter term, 
using that $[v\pa_v, s\pa_u] = \pa_u$ and that otherwise $[v\pa_v, Z_{\mB}] = 0$,
we have $|v\pa_v [Z^{J_1}_{\mB}, v\pa_v]Z_{\mB}^{J_2} \psi_C| \lesssim  \sum_{|K|\leq |J_1| + |J_2|}
|v\pa_v\pa Z_{\mB}^K\psi_C| \lesssim
\sum_{|K|\leq |J_1| + |J_2|+1}
|\pa Z_{\mB}^K\psi_C|$, so that, since $|J_1| + |J_2| \leq |I|-1$,
\begin{align}
  \label{}
\int_{t_0}^{t_1} \int_{D^C_t}
|\nas^2 Z_{\mB}^J \psi_C| &|v\pa_v Z_{\mB}^{J_1} [Z_{\mB}^{J_2}, v\pa_v] \psi_C|\, dt
\\
&\lesssim
\sum_{|K| \leq |I|} \int_{t_0}^{t_1} \int_{D^C_t} 
\frac{1}{1+v} |\nas \Omega Z_{\mB}^J \psi_C| |\pa Z_{\mB}^K \psi_C|\, dt
\\
&\lesssim 
\delta \int_{t_0}^{t_1}\int_{D^C_t} |\nas \Omega Z_{\mB}^K \psi_C|^2\, dt
+ \frac{1}{\delta} \int_{t_0}^{t_1} \frac{1}{(1+t)^2} \left(\int_{D^C_t} |\pa Z_{\mB}^K \psi_C|\right)\, dt
\\
&\lesssim
\delta S_{X_T}[Z_{\mB}^J \psi_C] + \frac{1}{\delta} c_0(\epsilon_0) \epsilon_C^2,
\end{align}
as needed.

\end{proof}
}

We now prove the corresponding bounds in the leftmost region.
\begin{lemma}
	\label{timeintegrability-left}
	Define $X = X_L$ or $X = X_M$ as in Section \ref{fields} and write
	$X = X^n \pa_u + X^\ell \pa_v$.
 If the hypotheses of Proposition \ref{bootstrapprop} hold,
 the quantities
 $\gamma, P_I, F_{I}$ appearing in \eqref{higherorderexterioreqn}
 satisfy the following estimates.
 \begin{equation}
  |\gamma| \lesssim \frac{\epsilon_L}{(1+v)(1+s)^{1/2}} \frac{(\log \log s)^\alpha}{(\log s)^{\alpha - 1}},
	\qquad
	|X_m^\ell| |\gamma| \lesssim \epsilon_L |X^n_m|,
  \label{pert0left}
 \end{equation}
 \begin{multline}
  \int_{t_0}^{T} \|\nabla \gamma\|_{L^\infty(D^L_t\cap\{|u| \leq v/8\})}
	+ \left\|\frac{1}{1 + |u|}
	\gamma\right\|_{L^\infty(D_t^L\cap\{|u| \leq v/8\})}\,dt\\
	+
	\int_{t_0}^{t_1} \left\|\left(1 + \frac{|X_m^\ell|^{1/2}}{|X^n|^{1/2}}\right) \nabla_v \gamma\right\|_{L^\infty(D_t^L\cap\{|u| \leq v/8\})}
	+ \left\|\left(1 + \frac{|X_m^\ell|^{1/2}}{|X_m^n|^{1/2}}\right)\nas \gamma\right\|_{L^\infty(D_t^L\cap\{|u| \leq v/8\})}
	 \, dt
	\lesssim \epsilon_L,
  \label{intgammaL}
 \end{multline}
 and
 \begin{multline}
	 \int_{t_0}^{t_1} \| |X^n_m|^{1/2} \nabla P_I\|_{L^2(D^L_t\cap\{|u| \leq v/8\})}
	 + \| (1+|u|)^{-1} |X^n_m|^{1/2} P_I\|_{L^2(D^L_t\cap \{|u| \leq v/8\})}
	 \, dt\\
	 +\int_{t_0}^{t_1} \||X^\ell_m|^{1/2} \nabla_v P_I\|_{L^2(D^L_t\cap\{|u| \leq v/8\})}
	 + \||X^\ell_m|^{1/2} \nas P_I\|_{L^2(D^L_t\cap\{|u| \leq v/8\})} + \| |X^\ell_m|^{1/2} F_I\|_{L^2(D^L_t)}
	 \, dt\lesssim \epsilon_L^2,
  \label{intPFL}
 \end{multline}
 as well as the following bounds in the region $|u| \geq v/8$,
 \begin{align}
  \int_{t_0}^{t_1} \|\mathcal{L}_X \gamma\|_{L^\infty(D^L_t\cap \{|u| \geq v/8\})}
	\, dt
	&\lesssim \epsilon_L,\label{intgammaclosetoorigin}\\
	\int_{t_0}^{t_1} \|\mathcal{L}_X P_I\|_{L^2(D^L_t\cap \{|u| \geq v/8\})}
	\, dt
	&\lesssim \epsilon_L^2.
  \label{intPclosetoorigin}
 \end{align}

 We also have
 \begin{equation}
  \sup_{t_0\leq t \leq t_1}
	\int_{D^L_t} |X^\ell_m| |P_I|^2 + \int_{t_0}^{t_1} \left(|X^\ell_m|  + (1+v)(1+s)^{1/2} |X^n_m|\right) |P_I|^2\, dS dt
	\lesssim \epsilon_L^{3}.
  \label{RPXboundL}
 \end{equation}
\end{lemma}
\begin{proof}
	We will need to argue slightly differently in the three regions
	\begin{equation}
	 D_{t,1}^L = D_t^L \cap\{ |u| \leq s^3\},
	 \qquad
	 D_{t,2}^L = D_t^L \cap \{s^3 \leq |u| \leq v/8\},
	 \qquad
	 D_{t, 3}^L = D_t^L \cap \{|u| \geq v/8\}.
	 \label{}
	\end{equation}

	We first consider the bounds in
	$D_{t,1}^L$. As in
	\eqref{gammabdrightpfgammaint}, we have the bounds
	\begin{align}
	 |\gamma| &\lesssim \frac{1}{1+v} |\pa \psi_L| + \frac{1}{(1+v)^2} |\psi_L|,
	 \label{gammapsiL0}
	 \\
	 |\nabla_Y \gamma| &\lesssim \frac{1}{1+v} |\nabla_Y \pa \psi_L|
	 + \frac{1}{(1+v)^2}|Y| |\pa \psi_L|  + \frac{1}{(1+v)^3} |Y| |\psi_L|,
	 \label{gammapsiL1}
	\end{align}
	where we used that $r \geq \tfrac{1}{4}v$, say, in this region.

		The bounds \eqref{pert0left} in the region $D_{t, 1}^L$
		then follow from \eqref{gammapsiL0} and \eqref{pwleft}, with the crucial
		observation being that
		\begin{multline}
		 \frac{|X^\ell_m|}{1+v} |\pa \psi_L|
		 \lesssim
		 f(v) |\pa \psi_L|
		 \lesssim
		 \epsilon_L\frac{f(v)}{|u| f(u)^{1/2}} \lesssim
		 \epsilon_L\frac{s (\log s)^\alpha}{s^{1/2} (\log s^{1/2} (\log \log s)^{\alpha/2}}
		 \lesssim \epsilon_Ls^{1/2} \frac{(\log s)^{\alpha-1/2}}{ (\log \log s)^{\alpha/2}}\\
		 \lesssim \epsilon_L s^{1/2} \log s (\log \log s)^\alpha
		 \lesssim \epsilon_L |u| f(u),
		 \label{crucialpowerchoice}
		\end{multline}
		where in the second-last step we used that $\alpha < 3/2$ and in the last
		step we used the lower bound for $|u|$ in $D^L_t$. To control
		the last term in \eqref{gammapsiL0}, we used
	the bound \eqref{pwLhardy}.

	As in Lemmas \ref{nablaygamma}-\eqref{nablaygamma2}, the time-integrated bounds for $\gamma$
	and its derivatives in $D^L_{t, 1}$ follow from \eqref{gammapsiL0}-
	\eqref{gammapsiL1} and the time-integrated bounds
	in Lemma \ref{timeintpsileft}, but using
	\eqref{pwLhardy} and the fact that we have a bound for $B_{I}(t)$,
	in place of the bound \eqref{pwRhardy}
	that we used in the right region. The quantities
	$P_I$ and $F_I$ in this region can be handled using similar arguments
	and so we skip them.

	It remains only to prove the needed bounds in the region $D^L_{t,2}$
	and $D^L_{t,3}$. Here the estimates are less delicate because of the lower
	bound for $|u|$. In this region, we work in terms of $\phi$ and recall
	from
	\eqref{gammabdsexteriorapp},
	\begin{equation}
		|\gamma| \lesssim |\pa \phi_L|,
		\qquad
		|\nabla_X \gamma| \lesssim |\nabla_X \pa \phi_L|,
		\qquad
		|\mathcal{L}_X \gamma| \lesssim |\mathcal{L}_X \pa \phi_L|.
	\end{equation}
	By the pointwise bound \eqref{pwleft2}, the bounds
	in \eqref{pert0left} clearly hold in $D^L_{t, 2}$
	and $D^L_{t,3}$.
	To get the time-integrated bound for derivatives of $\gamma$ in
	$D^L_{t,2}$, we just bound
	\begin{equation}
	 |\nabla \pa \phi|\lesssim |\pa^2 \phi| + \frac{1}{r} |\pa \phi|
	 \lesssim |\pa^2 \phi| + \frac{1}{1+v} |\pa \phi|,
	 \qquad |u| \leq v/8
	 \label{}
	\end{equation}
	using that the Christoffel symbols in our coordinate system
	satisfy $|\Gamma| \lesssim \frac{1}{r}$. By \eqref{pwleft2}
	this gives
	\begin{equation}
	 |\nabla \pa \phi| \lesssim \frac{\epsilon_L}{(1+v)(1+s)^6},
	 \qquad |u| \geq s^3,
	 \label{}
	\end{equation}
	which is more than enough to get the first bound in
	\eqref{intgammaL}. The bound
	for $(1+|u|)^{-1}\gamma$ is identical, and the bound
	for the terms on the second line of \eqref{intgammaL}
	is easier since we can just bound $|X^\ell|^{1/2}
	(|\pa_v q| + |\nas q|)
	\lesssim (1+v)^{-1/4} \sum_{Z \in \mathcal{Z}} |Z q|$.

	To get the time-integrated
	bound for $\mathcal{L}_X\gamma$, which is needed
	in $D^L_{t,3}$, we bound
	\begin{equation}
	 |\mathcal{L}_X \pa \phi|
	 \lesssim
	 \sum_{\nu \in \{u,v,\theta^1,\theta^2\}}
	 |X^\mu \pa_\mu \pa_\nu \phi|
	 + |\pa_\nu X^\mu \pa_\mu \phi|.
	 \label{liederivativeofpaphi}
	\end{equation}
	For the first term, we bound
	\begin{equation}
	 |X^\mu\pa_\mu \pa_\nu \phi|
	 \lesssim |X^v| |\pa_v \pa \phi| + |X^u| |\pa_u \pa \phi|
	 \lesssim f(v)\sum_{Z \in \mathcal{Z}} |\pa Z \phi|,
	 \label{}
	\end{equation}
	where we used that for either multiplier we have $|X^v|/v + |X^u|/|u|
	\lesssim f(v)$. By the pointwise bound
	\eqref{pwleft2} and the definition of $f$ from
	\eqref{fgdefintro} in the region $D^L_{t,1}$ we have
	\begin{equation}
	 f(v)|\pa Z \phi|
	 \lesssim
	 \frac{s (\log s)^\alpha}{(1+v)(1+s)^3} \epsilon_L
	 \lesssim
	 \frac{1}{1+t} \frac{1}{(1 + \log t)^2}\epsilon_L,
	 \label{fpaphibd}
	\end{equation}
	which is time-integrable. For the second term in
	\eqref{liederivativeofpaphi}, we just use that $v|f'(v)|
	+f(u) + |u f'(u)| \lesssim f(v)$ and then
	\begin{equation}
	 \sum_{\nu \in \{u,v,\theta^1,\theta^2\}}
	 |\pa_\nu X^\mu \pa_\mu \phi|
	 \lesssim f(v) |\pa \phi|,
	 \label{}
	\end{equation}
	which can be bounded just as in \eqref{fpaphibd}.
	The bound for $|\nabla_X \gamma|$ in the region $D^L_{t,2}$ can be
	proven in a nearly identical way.

	The needed bounds for $P_I$ and $F_I$ can be proven
	using the same arguments we have now used many times,
	after using the bounds
	from Lemma \ref{higherorderexterioreqns}.
\end{proof}

\subsection{ Control of the scalar currents}
\label{scalarcontrol}
We now use the results of the previous section to control the scalar currents
$\mK_{X, \gamma, P}$ which appear
 in the energy estimates \eqref{enbootR0}, \eqref{enbootC0}, \eqref{EDClocalbd}
 and \eqref{enbootL0}.

By Proposition \ref{effectivemmmink}, in the regions $D^L$ and $D^R$, when
 $|u| \leq v/8$ for our multipliers $X = X_L, X_M, X_R$, this quantity satisfies
\begin{multline}
	 |\mK_{X, \gamma, P}[\psi]|\lesssim
	 \left( |\nabla \gamma| + \frac{1}{1+|u|} |\gamma|
 	+ \frac{|X^\ell_m|^{1/2}}{|X^n_m|^{1/2}} |\nabla_{\evm} \gamma|
	+ \frac{|X^\ell_m|^{1/2}}{|X^n_m|^{1/2}} |\nas \gamma|
 	\right)  |\pa \psi|_{X, m}^2
 	+ |X^n_m| |F|  |\pa \psi|_{X, m}\\
	+ \left(|\nabla P| + \frac{1}{1+|u|}|P| +
	\frac{|X^\ell_m|^{1/2}}{|X^n_m|^{1/2}} |\nabla_{\evm} P|
	+|X^\ell_m| |\nas P| \right) |X^n_m|^{1/2} |\pa \psi|_{X, m}
	\\
  +
   |P| |\pa_u X^v| |\evm \psi|
	+
	|P||X|  \left( |F| + \frac{1}{1+v} |P|\right)
	 \label{modifiedKbound0boot}
	\end{multline}
		and in the region $|u| \geq v/8$,
		we instead have the bound
		\begin{equation}
		 |\mK_{X, \gamma, P}[\psi]|
		 \lesssim |\mathcal{L}_X \gamma |  |\pa \psi|^2 + |\gamma| |\pa X| |\pa \psi|^2
		 + |\mathcal{L}_X P||\pa \psi|
		 +
		 \frac{1}{1+v}|X| \left(|\gamma| |\pa \psi|^2+ |P| |\pa \psi|\right).
		 \label{trivialKbound0bootlie}
		\end{equation}

	In the central region $D^C_t$ by Propositition \ref{effectivemmmB}, we have
	\begin{align}
		\label{mBmodifiedKboundstatement2}
		|\mK_{X, \gamma, P}[\psi]| &\lesssim
		\left( |\nabla \gamma| + \frac{|\gamma|}{1+s}
		+ \frac{|X^\ell_{\mB}|^{1/2}}{|X^n_{\mB}|^{1/2}}
		\left(|\nabla_{\evmB} \gamma|+ |\nas \gamma|\right)
		\right)  |\pa \psi|_{X, \mB}^2
		+ \frac{1}{(1+v)^{1/4}} |F|  |\pa \psi|_{X, \mB}\\
		&+   \left( |\nabla P^u| + \frac{|P^u|}{1+s}
		+|X^\ell_{\mB}|^{1/2} \left(|\nabla_{\evmB} P| + |\nas P|
		 \right) \right)|X^n_{\mB}|^{1/2}|\pa \psi|_{X, \mB}\\
		&+ \epsilon_C\left( \frac{1}{(1+v)^{3/2}} |\pa \psi|^2 + \frac{1}{(1+v)^{1/2}} (|\evmB \psi|^2 + |\nas \psi|^2)
		\right)\\
&+
\frac{1}{(1+s)^{1/2}} \left( |\nabla P^u| + \frac{|P^u|}{1+v}\right) |\pa \psi|
+v|P| \left(|\nabla P| + \frac{|P|}{1+v} + |F|\right).
	\end{align}
    {    We also note that by \eqref{nulllemmascalarcurrentbd}, we have the following
    bound for the scalar current $K_{\gamma_a, X}$ generated by the linear term \eqref{gammaadef}
    satisfying the null condition,
    \begin{equation}
      \label{linscalarcurrentbd}
      |K_{\gamma_a, X}[\psi]|
      \lesssim
      \frac{1}{(1+v)^{3/2}} |\pa \psi|^2 + \frac{1}{(1+v)^{1/2}} |\pa_v \psi|^2 
      + \frac{1}{(1+v)^{1/2}} |\nas \psi|^2.
    \end{equation}
    
}

	%

	We now record the needed $L^1_tL^1_x$ bounds for these quantities in
	each region. The main ingredients needed for these bounds are the
	time-integrated bounds in Section \ref{timeintsec}.

As a result of Lemma \ref{timeintegrability-right},
we have the following bound in the right-most region.
\begin{lemma}[Estimates for the scalar currents in the rightmost region]
	\label{abstractscalarright}
	Under the hypotheses of Proposition \ref{bootstrapprop}, with $X_R$ defined as in
	Section
	\eqref{fields}, we have the following bound,
	\begin{equation}
	 \sum_{|I| \leq N_R} \int_{t_0}^{t_1} \int_{D^R_t} |\mK_{X_R, \gamma, P_I}[Z^I\psi_R]|
	 + |F_I| |X_R \psi^I_R|\,dt
	 \lesssim \epsilon_R^{3}.
	 \label{needinrightforK}
	\end{equation}
\end{lemma}
\begin{proof}
    We just prove the bound in the region $|u| \leq v/8$, as the bound
    in the region $|u|\geq v/8$ is simpler in light of the strong decay estimates
    \eqref{pwright} for $\psi$ in that region. For the multiplier 
    $X = X_R$ we have the bound $\frac{|X^\ell_m|^{1/2}}{|X^n_m|^{1/2}}
    \lesssim 1 + \frac{r^{1/2}}{(1+|u|)^{\mu/2}}$, and so by
    Lemma \ref{timeintegrability-right}, we have
    \begin{align}
      \label{}
      \int_{t_0}^{t_1} \int_{D^R_t \cap \{|u| \leq v/8\}}&
      \left( |\nabla \gamma| + \frac{1}{1+|u|} |\gamma|
 	+ \frac{|X^\ell_m|^{1/2}}{|X^n_m|^{1/2}} |\nabla_{\evm} \gamma|
	+ \frac{|X^\ell_m|^{1/2}}{|X^n_m|^{1/2}} |\nas \gamma|
 	\right)  |\pa \psi|_{X, m}^2\, dt\\
     &\lesssim
    \int_{t_0}^{t_1} \left(\|\nabla \gamma\|_{L^\infty(D^R_t)}
    + \|(1+|u|)^{-1} \gamma\|_{L^\infty(D^R_t)}\right) E^R_I(t)\, dt\\
    &+\int_{t_0}^{t_1}
    \left(\left\|\left(1 + \frac{r^{1/2}}{(1+|u|)^{\mu/2}}\right)\pa_v \gamma\right\|_{L^\infty(D^R_t)}
  +
      \left\|\left(1 +
      \frac{r^{1/2}}{(1+|u|)^{\mu/2}}\right)\nas\gamma\right\|_{L^\infty(D^R_t)}\right)
      E^R_I(t)
\,dt \lesssim \epsilon_R^3,
    \end{align}
    using Lemma \ref{timeintegrability-right}. By the same result, we have
    \begin{multline}
      \label{}
      \int_{t_0}^{t_1}\int_{D^R_t \cap \{|u| \leq v/8\}} \left(|\nabla P_I| 
          + (1+|u|)^{-1}|P_I| +
	\frac{|X^\ell_m|^{1/2}}{|X^n_m|^{1/2}} |\nabla_{\evm} P|
	+|X^\ell_m| |\nas P| \right) |X^n_m|^{1/2} |\pa \psi|_{X, m}\,dt
    \\
    \lesssim
    \int_{t_0}^{t_1} \left(\||X^n_m|^{1/2} \nabla P_I\|_{L^2(D^R_t)}
    + \|(1+|u|)^{-1} |X^n_m|^{1/2} P_I\|_{L^2(D^R_t)}\right)E_I^R(t)\, dt
    \\
    +\int_{t_0}^{t_1}
 \left(   \| |X^\ell_m|^{1/2} \nabla_{\evm} P_I\|_{L^2(D^R_t)}
    + \| |X^\ell_m|^{1/2} \nas P_I\|_{L^2(D^R_t)}\right) E_I^R(t)\, dt
    \lesssim \epsilon_R^3.
    \end{multline}
    By our choice of $\mu, \nu$ in \eqref{parameters}, in $D^R_t$ we have
    the bound $|\pa_u X^v_m| \lesssim 1+ |u|^{\mu-1} + (\log r)^{\nu-1}
    \lesssim 1+ |u|^{\mu-1}$ in $D^R_t$ and so we also have
    \begin{equation}
        \int_{t_0}^{t_1} \int_{D^R_t} |P_I| |\pa_u X^v| |\evm \psi_R^I|\, dt
        \lesssim
        \int_{t_0}^{t_1} \int_{D^R_t} \frac{1 + |u|^{\mu-1}}{r^{1/2}} |P_I|
        |\pa \psi_R^I|_{X_R, m}\, dt
        \lesssim \epsilon_R^3,
    \end{equation}
    using Lemma \ref{timeintegrability-right} to bound the contribution
    from $P_I$.
    To complete the bounds in the region $|u|\leq v/8$,
    it remains only to bound the last term in \eqref{modifiedKbound0boot} and 
    \eqref{needinrightforK}, and the bounds for these quantities follow easily
    from the estimates \eqref{intPFR} \eqref{pfofbasicPIbdR}. 

\end{proof}

In the central region, the analogous result is the following. 

\begin{lemma}[Estimates for the scalar currents in the central region]
	\label{abstractcentralscalar}
	Under the hypotheses of Proposition \ref{bootstrapprop},
	there is a continuous function $c_0$ with $
	c_0(0) = 0$ so that for any $\delta > 0$ we have the following estimates.

	 If $|I| \leq N_C$,
	 \begin{multline}
         \int_{t_0}^{t_1} \int_{D^C_t} |\mK_{X_T, \gamma, P_I+ P_{I, null}}[Z_{\mB}^I\psi_C]|
         + | K_{X_T, \gamma_a}[Z_{\mB}^I\psi_C]|
         + \left(|F_{C,I}| + |F_{\Sigma, I}| + |{F_{\mB, I}^2}|\right) |X_T \psi^I_C|\,
         dt
         \\
         \lesssim
         \epsilon_C^{3} + c_0(\epsilon_0)\left(1 + \frac{1}{\delta}\right)\epsilon_C^2
         + c_0(\epsilon_0)
         + \delta \sum_{|J| \leq |I|}\left(\sup_{t_0 \leq t \leq t_1} E_{J, T}^C(t) +  S_J^C(t_1)\right)
         \label{scCtop}
\end{multline}

If $|I| \leq N_C-1$,
\begin{multline}
 \int_{t_0}^{t_1} \int_{D^C_t} |\mK_{X_C, \gamma, P_I+ P_{I, null}}[Z_{\mB}^I\psi_C]|
         + | K_{X_C, \gamma_a}[Z_{\mB}^I\psi_C]|
         + \left(|F_{C,I}| + |F_{\Sigma, I}| + |{F_{\mB, I}^2}|\right) |X_C \psi^I_C|\,dt
 \\
 \lesssim \epsilon_C^{3} (1 + \log \log t_1) + c_0(\epsilon_0)\left(1 + \frac{1}{\delta}\right)\epsilon_C^2
 + c_0(\epsilon_0)
 + \delta S_I^C(t_1)\\
	+ \epsilon_C(1+\log\log t_1) \sum_{|J| \leq |I|-1} \sup_{t_0 \leq t \leq t_1}
    (E_{J, D}^C(t))^{1/2}, 
	\label{scCsubtop}
\end{multline}
and if $|I| \leq N_C-2$,
\begin{multline}
 \int_{t_0}^{t_1} \int_{D^C_t} |\mK_{X_C, \gamma, P_I+ P_{I, null}}[Z_{\mB}^I\psi_C]|
         + | K_{X_C, \gamma_a}[Z_{\mB}^I\psi_C]|
         + \left(|F_{C,I}| + |F_{\Sigma, I}| + |{F_{\mB, I}^2}|\right) |X_C \psi^I_C|\,dt
 \\
 \lesssim
 \epsilon_C^{3} + c_0(\epsilon_0)\left(1 + \frac{1}{\delta}\right)\epsilon_C^2
 + c_0(\epsilon_0)
 + \delta\sum_{|J| \leq |I|} S_J^C(t_1).
 \label{scCdecay}
\end{multline}

\end{lemma}
\begin{proof}
	We use \eqref{mBmodifiedKboundstatement2}
	and Lemma \ref{timeintegrability-center}, \ref{timeintegrability-center-linear}.
	First, by Lemma \eqref{timeintegrability-center}, regardless of the
	multiplier $X$ we bound
	\begin{align}
	 \int_{t_0}^{t_1}
	 \int_{D^C_t}&
	 \left(|\nabla \gamma|  + \frac{1}{1+s} |\gamma| + (1+v)^{1/2}(1+s)^{1/2}
	|\nabla_{\evmB} \gamma| + (1+v)^{1/2}(1+s)^{1/2}|\nas \gamma|\right)|\pa Z_{\mB}^I \psi|_{X, \mB}^2\,dt
	\nonumber\\
	&\lesssim
	\int_{t_0}^{t_1} \|\nabla \gamma\|_{L^\infty(D^C_t)}
	+ \left\|\frac{1}{1 + s} \gamma\right\|_{L^\infty(D_t^C)}E_{X, I}(t)\, dt
	\\
	&+
	\int_{t_0}^{t_1} \left\|(1+v)^{1/2} (1+s)^{1/2} \nabla_{\evmB} \gamma\right\|_{L^\infty(D_t^C)}
	+ \left\|(1+v)^{1/2} (1+s)^{1/2} \nas \gamma\right\|_{L^\infty(D_t^C)}E_{X, I}(t)
	 \, dt,
	 \\
	 &\lesssim \epsilon_C \sup_{t_0 \leq t \leq t_1} E_{X, I}(t),
	 \label{}
	\end{align}
	with the notation $E_{X_{T}, I} = E_{T, I}^C$ and $E_{X_{ C}, I} = E_{D, I}^C$.
	By the bootstrap assumptions \eqref{centralboottop}-\eqref{centralbootextra} for the energies, this is bounded by
	\eqref{scCtop} when $X = X_{T}$ and $|I| \leq N_C$, by \eqref{scCsubtop} when
	$X = X_{ C}$ and $|I| \leq N_C-1$, and by \eqref{scCdecay} when
	$X = X_{ C}$ and $|I| \leq N_C-2$.

	We now control the terms on the second line of
	 \eqref{mBmodifiedKboundstatement2}.
	Regardless of which multiplier
	we use, by \eqref{centralPIbounds} we have
	\begin{align}
	 \int_{t_0}^{t_1} \int_{D^C_t} &|X^n_{\mB}|^{1/2} \left( |\nabla P_{I}| +
	 \frac{1}{1+s} |P_{I}|\right) |\pa Z_{\mB}^I \psi_C|_{X, \mB}\, dt\\
	 &+ \int_{t_0}^{t_1} \int_{D^C_t} |X^\ell_{\mB}|^{1/2}
	 \left( |\nabla_{\evmB} P_I| + |\nas P_I|\right)|\pa Z_{\mB}^I\psi_C|_{X, \mB}\,dt
	 \\
	 &\lesssim
	 \int_{t_0}^{t_1}
	 \| |X^n_{\mB}|^{1/2} \nabla P_{I}\|_{L^2(D^C_t)}
	 + \||X^n_{\mB}|^{1/2}(1+s)^{-1}  P_{I}\|_{L^2(D^C_t)} E_{X, I}(t)^{1/2}\,dt\\
	 &+
	 \int_{t_0}^{t_1}
	 \left( \| (1+v)^{1/2} \nabla_{\evmB} P_{I}\|_{L^2(D^C_t)} +
	 \| (1+v)^{1/2} \nas P_{I}\|_{L^2(D^C_t)}\right) E_{X, I}(t)^{1/2}\, dt\\
	 &\lesssim (\epsilon_C^2 + c_0(\epsilon_0)\epsilon_C) \sup_{t_0 \leq t \leq t_1} E_{X, I}(t)^{1/2},
	 \label{}
	\end{align}
	which is bounded by the right-hand side of \eqref{scCtop}-\eqref{scCdecay}.
	To control the $L^1_tL^1_x$ norm of the terms on the second line of
	\eqref{mBmodifiedKboundstatement2},
	it remains to control the contribution from $P_{I, null}^u$,
	and the needed bound follows directly from
	Lemma \ref{timeintegrability-center-linear}.

	For the terms on the third line of \eqref{mBmodifiedKboundstatement2} we just bound
	\begin{multline}
	 \int_{t_0}^{t_1} \int_{D^C_t}\left( \frac{1}{(1+v)^{3/2}} |\pa Z_{\mB}^I\psi_C|^2
	 + \frac{1}{(1+v)^{1/2}} (|\evmB Z_{\mB}^I \psi_C|^2 + |\nas Z_{\mB}^I\psi_C|^2)
	 \right)\, dt\\
	 \lesssim
	 \int_{t_0}^{t_1} \frac{1}{(1+t)^{5/4}}
	 \int_{D^C_t} |X^n_{\mB}||\pa Z_{\mB}^I\psi_C|^2  + |X^\ell_{\mB}| (|\evmB Z_{\mB}^I \psi_C|^2 + |\nas Z_{\mB}^I\psi_C|^2)\, dt
	 \lesssim c_0(\epsilon_0) \epsilon_C^2.
	 \label{}
	\end{multline}
	To finish the bounds for the scalar current $\mK$,
	it remains to control the terms on the last line of \eqref{mBmodifiedKboundstatement2}.
	These terms are easier to handle than the above
	after using the pointwise estimates from
	Lemma \ref{higherordereqncentral}, noting in particular
	that $(1+s)^{1/2} \leq |X^n_{\mB}|^{1/2}$ for both estimates, and we skip them.
    The bounds for the linear scalar currents $K_{X,\gamma_a}$ follow easily from
    the bound \eqref{linscalarcurrentbd} and the definitions of our energies,
    \begin{multline}
      \label{}
      \int_{t_0}^{t_1} \int_{D^C_t} \frac{1}{(1+v)^{3/2}} |\pa_u \psi_C^I|^2
      + \frac{1}{(1+v)^{1/2}}\left( |\pa_v \psi_C^I|^2 + |\nas \psi_C^I|^2\right)\,dt
          \\
          \lesssim
          \int_{t_0}^{t_1} \frac{1}{(1+t)^{5/4}} \int_{D^C_t} X^n_{\mB}
          |\pa_u \psi_C^I|^2
          + X^\ell_{\mB} \left(|\evmB \psi_C^I|^2 +|\nas \psi_C^I|^2  \right) \, dt
          \lesssim c_0(\epsilon_0)\sup_{t_0 \leq t \leq t_1} E_{X, I}(t),
      \end{multline}
      where we used that $|X^n_{\mB}|\gtrsim (1+s)^{-1/2}$ and $X^\ell_{\mB} = v$
      for both multipliers.

	The needed bounds for the remainder terms $F_{C, I},
    F_{\Sigma, I}, {F_{\mB, I}^2}$ follow directly
	from the bounds \eqref{FICl2bds}-\eqref{FICl3bds}.

\end{proof}

Finally, in the leftmost region we need the following result.
\begin{lemma}[Estimates for the scalar currents in the leftmost region]
	\label{abstractscalarleft}
	Under the hypotheses of Proposition \ref{bootstrapprop}, with $X_L$ and $X_M$
	defined as in Section \ref{fields}, we have
	\begin{equation}
	\sum_{|I| \leq N_L} \int_{t_0}^{t_1} \int_{D^L_t} |\mK_{X_L, \gamma, P_I}[\psi^I_L]|
	 +|\mK_{X_M, \gamma, P_I}[\psi^I_L]|\, dt
	 + \int_{t_0}^{t_1} \int_{D^L_t} |F_I| |X_L \psi^I_L| + |F_I| |X_M \psi^I_L|\, dt
	 \lesssim\epsilon_L^{3}.
	 \label{}
	\end{equation}

\end{lemma}
\begin{proof}
	The proof follows in the same way as the above results, but using
	the pointwise estimates \eqref{modifiedKbound0boot}-\eqref{trivialKbound0bootlie}
	for the scalar current and
	Lemma \ref{timeintegrability-left} for the needed time-integrated bounds.
\end{proof}

\section{The higher-order boundary conditions}
\label{bcbootstrapsection}
The goal of this section is to prove that under the hypotheses of
Proposition \ref{bootstrapprop}, the bounds in \eqref{rightshockerror}
and \eqref{bdyerrorleftboot} for the derivatives $\ev \psi$ along the timelike
sides of the shocks hold. Specifically we will be proving bounds for the quantities
\begin{equation}
  \label{BIdef2}
 B^L_I(t_1) = \int_{t_0}^{t_1} \int_{\Gamma^L_t} v f(v) |\evm \psi_L^I|^2\,dS dt,
  \qquad
  B^C_I(t_1) = \int_{t_0}^{t_1}\int_{\Gamma^R_t}(1+v) |\evmB \psi^I_C|^2\,dS dt.
\end{equation}
We remind the reader at this point that by the definitions of the energies
in \eqref{ECdef} and \eqref{ELdef},
our bootstrap assumptions \eqref{centralboottop}-\eqref{leftboot}
imply the bounds
\begin{equation}
  \label{BIbds}
  \sum_{|I| \leq N_L} B^L_I(t_1) \lesssim \epsilon_L^2,\qquad
  \sum_{|I| \leq N_C} B^C_I(t_1) \lesssim \epsilon_C^2.
\end{equation}

At the left shock, the result is the following.
\begin{prop}
 \label{lefttocenterprop}
 There is $\epsilon_0^* > 0$ with the following property.
 If the hypotheses of Proposition \ref{bootstrapprop}
 hold with $\epsilon_0 < \epsilon_0^*$, then
 \begin{equation}
     \sum_{|I|\leq N_L} B^L_I(t_1) \lesssim \epsilon_L^3.
  \label{}
 \end{equation}
\end{prop}

The analogous result at the right shock is the following.
\begin{prop}
 \label{centertorightprop}
 There is $\epsilon_0^* > 0$ with the following property.
 If the hypotheses of Proposition \ref{bootstrapprop}
 hold with $\epsilon_0 < \epsilon_0^*$, then writing
 $X = X^{\ell} \evmB + X^nn$,
 with $X = X^\ell_{C}$ or $X = X^{\ell}_{T}$, we have
 \begin{equation}
     \sum_{|I|\leq N_C} B^C_I(t_1)
    \lesssim \epsilon_C^3.
  \label{centertorightpropbd}
 \end{equation}
\end{prop}

The idea behind these estimates is the following.
Let $(\ev_-, \ev_+)
=(\evm, \evmB)$ at the left shock and $(\evmB, \evm)$ at the right shock.
For both estimates,
the basic ingredient needed is a bound for a quantity of the form
$\ev_- \psi^I_-$ where $\psi^I_-$ denotes
some collection of vector fields $Z^I$ applied to $\psi_-$, the potential
along the timelike side of the shock. When no vector
fields are present, the jump conditions at each shock take the form
\begin{equation}
 \ev_- \psi_- + \frac{1}{1+v} Q(\pa \psi_-, \pa \psi_-)
 = \ev_+\psi_+ + \frac{1}{1+v} Q(\pa \psi_+,\pa \psi_+),
 \label{schembc}
\end{equation}
where $Q$ is a quadratic nonlinearity and where we are omitting lower-order terms
(see \eqref{introbcL} and \eqref{introbcR}).
This expresses $\ev_- \psi_-$ in terms of the ``boundary data'' $\ev_+\psi_+$  and
nonlinear terms.
The weights $X^\ell$ on the timelike sides of the shock are such that the
contribution from the nonlinear terms can be handled, and using
our energy estimates on the spacelike side of the shock the contribution
from the boundary data $\ev_+ \psi_+$ can also be handled.

At higher order,
the calculation is more involved because we do not directly have an equation
for $\ev_- \psi^I_-$ in terms of the higher-order boundary data
$\ev_+\psi^I_+$ since the vector fields
we consider are transverse to the shock. We therefore need to replace
the vector fields $Z^I$ with a product of vector fields $Z_T^I$
which are tangent to the shock. For this we first
need to commute the fields $Z^I$ with the derivatives
$\ev_-$ which generates lower-order terms.
We then replace the fields $Z^I$ with the $Z_T^I$,
which generates error terms which involve
high-order derivatives of the boundary-definining functions $B^L, B^R$,
which is where we need the bounds \eqref{leftgeom}-\eqref{rightgeom}
for the geometry of the shocks. We can then bound $Z_T^I \ev_- \psi_-$
by applying tangential fields to \eqref{schembc} and replacing the fields
$Z_T^I$ with the usual fields $Z^I$ we can bound the quantities
on the right-hand side of \eqref{schembc} by our energies. This
is slightly cumbersome because we are using different
families of vector fields in each region, and this is ultimately why
we need to take the parameter $\mu$ from \eqref{schembc} large.

There is an additional difficulty at the left shock, which is that
the weight $X^\ell$ we use on the timelike side is large relative
to the weights we use on the spacelike side, and so to deal with the
contribution from the error term $X^\ell | \ev_+ \psi_+^I|^2$
we need to integrate to the right shock. This generates a bulk
term and a boundary term. The bulk term can be handled since the main term we generate in
this way is of the form $n \ev_+\psi_+$ and we have an equation for
this quantity. The boundary term can be handled by using the above
strategy to control
the resulting term on the timelike side of the right shock in
terms of the data on the spacelike side and nonlinear terms.

The estimates for replacing the $Z^I$ with the $Z^I_T$ are
the content of Section \ref{transitionsection}.
In sections \ref{pflefttocenterprop} and \ref{pfcentertorightprop}
we reduce the proofs of Propositions
\ref{lefttocenterprop}-\ref{centertorightprop} to a sequence of lemmas
which handle the nonlinear terms, the various error terms we generate
when commuting the fields $Z^I$ with the derivatives $\ev_-$, and which
give the needed bounds for higher-order derivatives
of the boundary data.

\subsection{Estimates for derivatives in terms of
tangential derivatives}
\label{transitionsection}
Let $\mathcal{Z}$ denote a collection of vector fields. In what follows
we will take either $\mathcal{Z} = \mathcal{Z}_m$ or $\mZB$, with
notation as in Section \ref{fields}.
It will be helpful to enlarge the collection $\mathcal{Z}$ and to write
$\widehat{Z} \in \widehat{\mathcal{Z}} = \mathcal{Z} \cup \{n\}$.
We will write $\widehat{Z}^K$ for a $|K|$-fold product of the fields in
$\widehat{Z}$. Specifically, if $\mathcal{Z} = \{Z_1\cdots Z_m\}$
we let $K = (K_1,..., K_r)$ where each $K_j \in \{e_1,\dots e_{m+1}\}$ where
$e_j$ denotes the standard basis of $\mathbb{R}^{m+1}$. If $K_j = e_p$
we then define $\widehat{Z}^{K_j} = \widehat{Z}_p$ with
$\widehat{Z}_p = Z_p$ when $p \leq m$ and $\widehat{Z}_{p} = n$ when $p  =m+1$.
Finally we set
$\widehat{Z}^K = \widehat{Z}^{K_1}\cdots \widehat{Z}^{K_r}$.

Fix a function $\xi:\mathbb{R}^4\to\mathbb{R}$ with $d\xi \not =0$
and which satisfies $n \xi = 1$. Given a vector field $Z$,
we define $Z_T  = Z - Z\xi n$, so that $Z_T \xi = d\xi(Z_T) = 0$. In particular
$Z_T$ is tangent to the set $\{\xi = 0\}$.
Let $Z_T^I$ denote a product of the fields $Z_T$
for $Z \in \mathcal{Z}$.
The following basic result then relates the
tangential fields $Z_T$ to the fields $Z$.
\begin{lemma}
	\label{decomplem}
	Let $\mathcal{Z}$ denote any collection of vector fields and
	define $Z^I, \widehat{Z}^I$ as in the above paragraph. Then we have
 \begin{equation}
  |Z^I q - Z_T^I q - (Z^I\xi) n q| \lesssim
	\sum_{r \geq 1} \sum_{\substack{|I_1| + \cdots |I_r| + |I_{r+1}| \leq |I|+1,
                                    \\ |I_{k}| \geq 1, |I_{r+1}|\ge 2}}
	|\widehat{Z}^{I_1}\xi| \cdots |\widehat{Z}^{I_r}\xi|
	 |\widehat{Z}^{I_{r+1}} q|,
  \label{higherorderdecomp}
 \end{equation}
 where there are $r$ factors of $n$ present in
 the collection $\widehat{Z}^{I_1},...,\widehat{Z}^{I_{r+1}}$ and at least
 one factor of $n$ in $\widehat{Z}^{I_{r+1}}$.
\end{lemma}
\begin{remark}
 Note that on the right-hand side of \eqref{higherorderdecomp}, there are no
 more than $|I|-1$ of the $\widehat{Z}$ derivatives landing on $\xi$ by the last condition
 in the sum and no more than $|I|$ of the $\widehat{Z}$ derivatives of $q$ since
 each $|I_k| \geq 1$. We also note that the reason there are $r$ factors
 of $n$ present in $\widehat{Z}^{I_{r+1}}$ is that this bound
 follows from repeatedly applying the definition $Z_T = Z - (Z\xi) n$,
 and every time we use this formula on derivatives of $q$, the number of
 $n$ derivatives present and number of factors of $Z\xi$ both
 increase by one. This counting is important because in our applications we expect
 $Z\xi \sim |u|$ and so we need to gain a power of $|u|^{-1}$ for each
 factor of $Z\xi$ (and thus factor of $n$) we encounter. The vector fields we consider
 are such that schematically, $n \sim \frac{1}{|u|} Z$ or better (see e.g \eqref{basicmink}),
 and so we gain (at least) one power of $u$ for each factor.
\end{remark}

\begin{proof}
 When $|I| = 1$ this is just the definition $Z_T q = Zq - (Z\xi) n q$.
 If the result holds for all $I$ with $|I| \leq m$ for
 some $m \geq 1$, we fix
 a multi-index $I$ with $|I| = m +1$ and write $Z^I = Z^J Z$ and then,
 writing $Zq = Z_Tq + (Z\xi) nq$,
 we have
 \begin{align}
  Z^J Z q
	&=
	Z_T^J Zq + (Z^J - Z_T^J)Zq\\
	&=
	Z_T^I q + Z^J_T(Z \xi n q) + (Z^J - Z_T^J)Zq\\
	&=
	Z_T^Iq + (Z^I\xi) nq \\
	&\quad+ \sum_{\substack{|J_1| + |J_2| \leq |J|,\\ |J_1| \leq |J|-1}}
	c^J_{J_1 J_2} (Z_T^{J_1} Z\xi)(  Z^{J_2}_Tnq)+ ((Z_T^J-Z^J)Z \xi)nq + (Z^J - Z_T^J)Zq,
  \label{zminusztagain}
\end{align}
 for constants $c^J_{J_1J_2}$. It remains to show that the terms on the last
 line here are of the appropriate form.

 By the inductive assumption, whenever $|L| \leq |I|-1$
 for any $q'$ we have the bound
 \begin{equation}
  |(Z^L - Z_T^L)q'|
	\lesssim
	\sum_{s \geq 1}
	\sum_{\substack{|L_1| + \cdots + |L_{s+1}| \leq |L|+1,\\|L_i| \geq 1}}
	|\widehat{Z}^{L_1} \xi|
	\cdots
	|\widehat{Z}^{L_s} \xi|
	|\widehat{Z}^{L_{s+1}} q'|,
  \label{}
 \end{equation}
where there are $s$ factors of $n$ present in the collection $\widehat{Z}^{L_1},
 \cdots \widehat{Z}^{L_{s+1}}$ and at least one factor in $\widehat{Z}^{L_{s+1}}$.
Applying this with $q'$ replaced by $Zq$ we find that
\begin{equation}
 |(Z^J - Z_T^J)Zq|
 \lesssim
 \sum_{s \geq 1}
 \sum_{\substack{|L_1| + \cdots + |L_{s+1}| \leq |L|+1,\\|L_i| \geq 1}}
 |\widehat{Z}^{L_1} \xi|
 \cdots
 |\widehat{Z}^{L_s} \xi|
 |\widehat{Z}^{L_{s+1}}Z q|,
 \label{}
\end{equation}
which is the correct form. Applying this with $q'$ replaced by $Z\xi$, we also have
that
\begin{equation}
 |((Z_T^J-Z^J)Z \xi)nq|
 \lesssim
 \sum_{s \geq 1}
 \sum_{\substack{|L_1| + \cdots + |L_{s+1}| \leq |L|+1,\\|L_i| \geq 1}}
 |\widehat{Z}^{L_1} \xi|
 \cdots
 |\widehat{Z}^{L_s} \xi|
 |\widehat{Z}^{L_{s+1}}Z \xi| |nq|,
 \label{indas}
\end{equation}
where there are now $s+1$ factors of $\xi$ present in the sums and $s$ factors
of $n$ present in the $\widehat{Z}^{L_k}$ along with one additional one in the last factor,
so this is also of the correct form.
In the same way, to handle
the terms in the sum on the last line of \eqref{zminusztagain}
we apply \eqref{indas} to $q' = Z\xi$ and $q' = nq$ and note that there
is already one factor of $n$ present in each product there, and the result follows.
 
\end{proof}

From now on, we take $\xi = u - B^A$.
In the next section we collect some estimates for the quantities
appearing in \eqref{higherorderdecomp} when $Z^I$ denotes a product
of Minkowski fields and when $Z^I$ denotes a product
of the fields from $\mathcal{Z}_{\mB}$. At the left shock, the main
result we need is Lemma \ref{minkdecomplemma} and at the right shock, the
main result is Lemma \ref{mBtangentialdifference}.

\subsubsection{Estimates for $Z^I - Z_T^I$ in the Minkowskian case}

We start with the following simple result.
\begin{lemma}
	\label{takehatoff}
 Let $\mathcal{Z} = \mathcal{Z}_m$ denote the Minkowskian fields. Fix a multi-index
 $J$ and suppose that
 \begin{equation}
  \sum_{|L| \leq |J|/2+1} \frac{|Z_T^L B|}{1+|u|} \leq M.
  \label{}
 \end{equation}
 Let
 $\widehat{Z}^J$ be as in the paragraph
 before Lemma \ref{decomplem}. If there are $j$ factors of $n$ present in $\widehat{Z}^J$
 then with $\xi = u - B$,
 \begin{equation}
  (1+|u|)^j|\widehat{Z}^J \xi| \leq C(M) (1+|u|)
	\left(1 + \sum_{|K| \leq |J|} \frac{|Z_T^K B|}{1+|u|}\right),
	\qquad \text{ if } t/2 \leq r \leq 3t/2, t\geq 1.
  \label{xibound}
 \end{equation}
 \end{lemma}

\begin{proof}
	By definition, $\widehat{Z}^J$ is a product of the form
	$n^{j_1} Z^{J_1}\cdots n^{j_k}Z^{J_k}$ where $\sum j_i + |J_i| = |J|$
	and where we recall that $n = \pa_u$. The idea in what follows is to
	first use basic properties of the Minkowski fields $\mathcal{Z}_m$ to re-write
	the vector fields $n$ in terms of powers of $(1+|u|)^{-1}$ and the Minkowski fields,
	and then to use \eqref{higherorderdecomp} and the fact that $nB = 0$ to
	re-write quantities of the form $Z^K B$ in terms of tangential derivatives.

	For this, it will be helpful to recall some simple and well-known properties of the vector fields
	$Z \in \mathcal{Z}_m$, which follow immediately
	from the formulas \eqref{inhomognull}-\eqref{homognull}. First, there are functions $a^Z, a^Z_u, a^Z_v, \slashed{a}^Z$
	satisfying the (Minkowskian) symbol condition
	\begin{equation}
	 |Z^J a|\leq C_J
	 \label{anothersymb}
	\end{equation}
	for constants $C_J$ so that we can write
	\begin{equation}
	 n = \frac{1}{1+|u|} \sum_{Z \in \mathcal{Z}_m} a^Z Z,
	 \qquad
	 Z = a^Z_u(1+|u|)\pa_u + a^Z_v(1+v)\pa_v + \slashed{a}^Z(1+r)\cdot \nas,
	 \quad \text{ if } t/2\leq r \leq 3t/2, t\geq 1.
	 \label{basicmink}
	\end{equation}
	We will also use that there are constants $c_{ZZ'}^{Z''}$ so that
	\begin{equation}
	 [Z, Z'] = \sum_{Z''\in \mathcal{Z}_m} c_{ZZ'}^{Z''} Z''.
	 \label{comm}
	\end{equation}

	We now prove the bound \eqref{xibound}. To start, we claim that
	if there are $j$ factors of $n$ present in $\widehat{Z}^J$, then
	\begin{equation}
	 (1+ |u|)^j |\widehat{Z}^J \xi|
	 \lesssim 1+ |u| + \sum_{|J'| \leq |J| - \widetilde{j}} |Z^{J'} B|,
	 \qquad
	 \widetilde{j} = \begin{cases}1,\quad j \geq 1.
 \\0,\quad j = 0\end{cases},
	 \label{jJxiclaim}
	\end{equation}
	Recalling $\xi = u - B$ and using \eqref{basicmink} it is enough
	to prove this bound with $\xi$ replaced by $B$. When $j = 0$ there
	is nothing to prove since then $\widehat{Z}^J = Z^J$. If $j \geq 1$, we
	write $\widehat{Z}^J = n^{j_1} Z^{J_1}\cdots n^{j_r} Z^{J_r}$ where without
	loss of generality $j_r \geq 1$. Using the first identity in
	\eqref{basicmink} to convert
	$n$ derivatives into $Z$ derivatives,
  $\widehat{Z}^J B$ can be written as a sum of terms of the
	form
	\begin{equation}
	 \frac{a}{(1+|u|)^{j_1+\cdots j_{r-1} + j_{r} - 1}} Z^{J'}n Z^{J''} B
 	= \frac{a}{(1+|u|)^{j_1+\cdots j_{r-1} + j_{r} - 1}} Z^{J'}[n, Z^{J''}] B
	 \label{ffs20}
	\end{equation}
	where $a$ satisfies the symbol condition \eqref{anothersymb} and
	where $|J'| + |J''| \leq |J| - 1$. Here we used that $nB \equiv 0$.
	To handle the commutator, we just use \eqref{basicmink} to express
	$n$ in terms of the fields $Z$ and then use the algebra property
	\eqref{comm}. This gives $|Z^{J'}[n, Z^{J''}] B| \lesssim
	(1+|u|)^{-1} \sum_{|J'''| \leq |J'| + |J''|} |Z^{J'''} B|$,
	and the claim \eqref{jJxiclaim} follows.

	Having proven \eqref{jJxiclaim}, to conclude the proof of
	\eqref{xibound} it remains to convert the
	$Z$ derivatives into $Z_T$ derivatives. For this, we use the bound
	\eqref{higherorderdecomp} and the fact that $nB = 0$ to get
	\begin{equation}
	 |Z^JB|\lesssim |Z_T^J B| + \sum_{r \geq 1}
	 \sum_{\substack{|J_1| +\cdots |J_{r+1}| = |J| +1,\\|J_k| \geq 1, |J_{r+1}|\geq 2}} |\widehat{Z}^{J_1} \xi|\cdots
	 |\widehat{Z}^{J_r} \xi||\widehat{Z}^{J_{r+1}} B|,
	 \label{ffs21}
	\end{equation}
	where there are $r$ factors of $n$ present in the collection
	$\widehat{Z}^{I_1},\dots \widehat{Z}^{I_{r+1}}$ and at least one factor
	of $n$ present in $\widehat{Z}^{J_{r+1}}$. Using the bound \eqref{jJxiclaim},
	we find that
	\begin{equation}
	  |Z^JB|\lesssim |Z_T^J B|
		+ \sum_{\substack{|J_1| +\cdots |J_{r+1}| = |J| +1,\\|J_k| \geq 1, |J|-1 \geq |J_{r+1}|\geq 2}}
		(1+|u|)^{-r}(1 +|u| + |{Z}^{J_1} B|)\cdots
 	 (1+ |u| + |{Z}^{J_r} B|) |{Z}^{J_{r+1}} B|,
	 \label{}
	\end{equation}
	where the fact that $|J_{r+1}| \leq |J| - 1$ follows from the fact that in
	\eqref{ffs21}, $|J_{r+1}| \leq |J|$ and that we are using \eqref{jJxiclaim}
	with $j = 1$. Since we also have $|J_k| \leq |J|-1$ for all $k = 1,...,r $
	in the sum, the bound \eqref{xibound} now follows from induction.
\end{proof}

As a result, we have the following bound.
\begin{lemma}
	\label{minkdecomplemma}
Under the hypotheses of Lemma \ref{takehatoff}, we have
 \begin{multline}
  |Z^I  q - Z_T^I  q|
	\leq
	C(M)(1+|u|) \sum_{|J| \leq |I|-1}  \left(|Z^Jnq| + (1+|u|)^{-1} |Z^J q|\right)\\
	+C(M)(1+|u|)|B|_{I, \mathcal{Z}_m} \sum_{|K| \leq |I|/2+1}
	\left(|Z^K n q| + (1+|u|)^{-1} |Z^K q|\right).
  \label{minkdecomplemmabd}
\end{multline}
where $|B|_{I, \mathcal{Z}_m}$ is defined as in \eqref{rescaledBnormdef},
and where the term $|Z^J q|$ is not present when $|I| = 1$.
\end{lemma}
\begin{remark}
 We will be applying this with $q$ replaced by $\evm q$ plus
 nonlinear terms and in that case
 we expect the quantity $Z^J n q$ to be well-behaved.
\end{remark}
\begin{remark}
 We also note that if we have the bound
 $|B|_{I,\mathcal{Z}_m} \leq M$, then \eqref{minkdecomplemmabd}
 implies that
 \begin{equation}
  |Z_T^I q| \lesssim C(M)\sum_{|J| \leq |I|} |Z^J q|, \qquad
	|Z^I q| \lesssim C(M) \sum_{|J| \leq |I|} |Z_T^J q|,
  \label{lowordertransfer}
 \end{equation}
 which follows after using that $(1+|u|) |n q'| \lesssim |Z q'|$ and standard
 properties of the fields $Z$. More generally, we have
 \begin{align}
  |Z_T^I q| &\lesssim C(M)\sum_{|J| \leq |I|} |Z^J q| +
	C(M) |B|_{I, \mathcal{Z}_{m}} \sum_{|K| \leq |I|/2+1} |Z^K q|\label{eztransfer0},
	\\
	|Z^I q| &\lesssim C(M)\sum_{|J| \leq |I|} |Z_T^J q| +
	C(M) |B|_{I, \mathcal{Z}_{m}} \sum_{|K| \leq |I|/2+1} |Z_T^K q|,
  \label{eztransfer}
 \end{align}
 if $\sum_{|I'| \leq |I|/2+1} |B|_{I',\mathcal{Z}_m} \leq M$.
%
\end{remark}

\begin{proof}
By \eqref{higherorderdecomp} we have
\begin{equation}
 |Z^I q - Z_T^I  q|
 \lesssim |Z^I\xi| |n q|
 + \sum_{r \geq 1} \sum_{\substack{|I_1| + \cdots + |I_{r+1}| \leq |I|+1, \\|I_{r+1}| \geq 2}}
 |\widehat{Z}^{I_1}\xi|\cdots |\widehat{Z}^{I_r} \xi| |\widehat{Z}^{I_{r+1}}  q|,
 \label{usedhigherorder}
\end{equation}
where there are $r$ factors of $n$ present in the collection
$\widehat{Z}^{I_1},...\widehat{Z}^{I_{r+1}}$ with at least one factor of $n$
in $\widehat{Z}^{I_{r+1}}$. We now
re-write the last factor in terms of the vector fields
$Z \in \mathcal{Z}_m$, the quantity $n q$, and lower-order terms.

 We claim that if there are $j\geq 1$ factors of $n$ present in $\widehat{Z}^J$
 then
 \begin{equation}
  (1+ |u|)^{j-1} |\widehat{Z}^J q| \lesssim
	\sum_{|J'| \leq |J|-1} |Z^{J'} n q| + (1+|u|)^{-1}|Z^{J'} q|.
  \label{widehatnl}
 \end{equation}
 This follows in a similar way to how we proved \eqref{jJxiclaim}.
 Since $\widehat{Z}^J = n^{j_1} Z^{J_1} \cdots n^{j_k} Z^{J_k}$ with
 $\sum_{s=1}^k j_s = j$, we just use \eqref{basicmink} to re-write $j-1$  factors
 of $n$ in terms of the fields $(1+|u|)^{-1} Z$ and then repeatedly use \eqref{comm} to bound
 \begin{equation}
  (1+|u|)^{j-1} |\widehat{Z}^J q|
	\lesssim \sum_{|K| \leq |J|} |\widehat{Z}^K q|,
  \label{}
 \end{equation}
 where the sum is over multi-indices $K$ satisfying the condition that there is exactly one factor of $n$ present in $\widehat{Z}^K$. Now we
 write
 \begin{equation}
  \widehat{Z}^K = Z^{K_1} n Z^{K_2} = Z^{K_1}Z^{K_2} + Z^{K_1} [n,Z^{K_2}] ,
	\qquad
	|K_1| + |K_2| = |K| - 1,
  \label{}
 \end{equation}
 and again use \eqref{basicmink}-\eqref{comm} to bound
 \begin{equation}
  |Z^{K_1} [n, Z^{K_2}]q|
	\lesssim (1+|u|)^{-1} \sum_{|K'| \leq |K|-2} |Z^{K'} q|.
  \label{}
 \end{equation}
 Combining the above, we get \eqref{widehatnl}.

 By \eqref{widehatnl}, have
\begin{align}
 \sum_{r \geq 1}& \sum_{\substack{|I_1| + \cdots |I_{r+1}| \leq |I|+1,\\ |I_{r+1}| \geq 2}}
 |\widehat{Z}^{I_1}\xi| \cdots |\widehat{Z}^{I_{r}}\xi|
 |\widehat{Z}^{I_{r+1}} q|\\
 &\lesssim
  \sum_{r \geq 1}\sum_{\substack{|I_1| + \cdots |I_{r+1}| \leq |I|+1,\\|I_{r+1}| \geq 2}}
	(1+ |u|)^{-r+1}
  |\widehat{Z}^{I_1}\xi| \cdots |\widehat{Z}^{I_{r}}\xi|
	\left(\sum_{|I'| \leq |I_{r+1}| -1}
  \left(|{Z}^{I'} nq| + (1+|u|)^{-1} |Z^{I'} q|\right)\right)\\
	& \leq
	C(M)(1+|u|) \sum_{|J| \leq |I|-1} |{Z}^{J} nq| + (1+|u|)^{-1} |Z^{J} q|\\
	&\qquad+ C(M)(1+|u|) |B|_{I, \mathcal{Z}} \sum_{|K| \leq |I|/2+1} |{Z}^{K} nq| + (1+|u|)^{-1} |Z^{K} q|
 \label{}
\end{align}
 where we used Lemma \ref{takehatoff} to handle the contributions
 from $\xi$.
\end{proof}

\subsubsection{Estimates for $Z^I - Z_T^I$ when $\mathcal{Z} = \mathcal{Z}_{\mB}$}
We now want a result analogous to Lemma \ref{minkdecomplemma}. This is somewhat
 simpler than the
result in the previous section because $n$ commutes with all the fields in
$\mathcal{Z}_{\mB}$.

The first step is the following. 
\begin{lemma}
	\label{loBmBbdlem}
	Fix a multi-index $J$ and suppose that
	\begin{equation}
   \sum_{|L| \leq |J|/2+1} \frac{|Z_{\mB, T}^L B|}{1+|u|} \leq M.
   \label{loBmBbd}
  \end{equation}
	With $\xi = u -B$, we have
 \begin{equation}
  |\widehat{Z}_{\mB}^J \xi| \leq C(M)\left(1 + s + \sum_{|J'| \leq |J|}
	|Z_{\mB, T}^{J'} B|\right)
  \label{mbdecomp0}
 \end{equation}
\end{lemma}
\begin{proof}
 First, since $[n, Z_{\mB}] = 0$ for all $Z_{\mB} \in \mathcal{Z}_{\mB}$, it
 enough to prove this bound when $Z_{\mB}^J = Z_{\mB}^K n^j$ for $|K| + j = |J|$.
 Since $n(u-B) = 1$, if $j \geq 1$ we clearly have $|Z_{\mB}^K n^j\xi|
 \lesssim 1$ (in fact if $j \geq 1$ this is only nonzero when $|K| = 0$)
 so we have the simple bound
 $|\widehat{Z}_{\mB}^J\xi|\lesssim 1 + \sum_{|J'|\leq|J|} |Z_{\mB}^J \xi|$, and
 so it is enough to bound $|Z_{\mB}^J\xi|$. We clearly have
 $|Z_{\mB}^J u|\lesssim (1+s)$ and so $|Z_{\mB}^J \xi| \lesssim
 1+s + |Z_{\mB}^JB|$.
 It remains to handle this last term. For this, we use
 \eqref{higherorderdecomp}, which,
 in light of what we have just proved, gives
 \begin{equation}
  |Z^J_{\mB} B - Z_{T, \mB}^J B|
	\lesssim
	1 + \sum_{r\geq 1}
	 \sum_{\substack{|J_1| + \cdots + |J_{r+1}| \leq |J|+1,\\
	|J_{r+1}| \geq 2}}
	(1+s + |\widehat{Z}^{J_1}_{\mB} B|)\cdots (1+s + |\widehat{Z}^{J_r}_{\mB} B|) |\widehat{Z}^{J_{r+1}}_{\mB} B|,
  \label{}
 \end{equation}
 where there are $r$ factors of $n$ present in the collection
 $\widehat{Z}^{J_1},\dots\widehat{Z}^{J_{r+1}}_{\mB}$
 and at least one factor of $n$ present in $\widehat{Z}^{J_{r+1}}$.
 Again using that $[n, Z_{\mB}] = 0$ and that $nB = 0$, it follows that
 the right-hand side is zero, and the result now follows.
\end{proof}
Recalling that $|B|_{I, \mB} = (1+s)^{-1/2} \sum_{|J| \leq |I|}
|Z_{\mB, T}^J B|$, we have the following analogue of
Lemma \ref{minkdecomplemma}.
\begin{lemma}
	\label{mBtangentialdifference}
	Under the hypotheses of Lemma \ref{loBmBbdlem}, we have
 \begin{multline}
  |Z^I_{\mB} q - Z_{\mB, T}^I q|
	\leq
	C(M) (1+s)\sum_{|J| \leq |I|-1} |Z_{\mB}^J nq|\\
	+C(M)(1+s) \left(1 + (1+s)^{-1/2}|B|_{I, \mZB} \right)
	\sum_{|L| \leq |I|/2+1} |Z_{\mB}^L nq|.
  \label{mBtangentialdifferencebd}
\end{multline}
\end{lemma}
\begin{remark}
 For some of our applications, it is better to write the above in the forms
  \begin{align}
   |Z^I_{\mB} q|
 	&\leq
	C(M) \sum_{|J| \leq |I|} |Z_{\mB, T}^J q|
	+ C(M) \left(1 + (1+s)^{-1/2}|B|_{I, \mZB} \right)
 	\sum_{|L| \leq |I|/2+2} |Z_{\mB, T}^L q|,
   \label{mBtangentialdifferencermk}\\
	 |Z^I_{\mB, T} q|
	 &\leq
	 	C(M) \sum_{|J| \leq |I|} |Z_{\mB}^J q|
	 	+ C(M) \left(1 + (1+s)^{-1/2}|B|_{I, \mZB} \right)
	  	\sum_{|L| \leq |I|/2+2} |Z_{\mB}^L q|
			\label{mBtangentialdifferencermk2}
 \end{align}
	which follow from \eqref{mBtangentialdifferencebd}
	and induction
	since $\mathcal{Z}_{\mB}$ includes
	$\sBo = s \pa_u = s n$ so $(1+s)|nq|\lesssim
	\sum_{Z_{\mB} \in \mathcal{Z}_{\mB}} |Z_{\mB} q|$.
\end{remark}
\begin{proof}
 By \eqref{higherorderdecomp} and \eqref{mbdecomp0}
 \begin{align}
  |Z^I_{\mB}q - &Z_{\mB,T}^I q|\\
	&\lesssim
	|Z_{\mB}^I \xi| |n q| +
	\sum_{r\geq 1}
	\sum_{\substack{|I_1| + \cdots |I_{r+1}| \leq |I|+1,\\ |I_{r+1}|\geq 2} }
	|\widehat{Z}_{\mB}^{I_1} \xi|\cdots |\widehat{Z}_{\mB}^{I_{r}} \xi| |\widehat{Z}_{\mB}^{I_{r+1}} q|\\
	&\leq
	C(M)(1+s) (1+s^{-\frac 12})|B|_{I, \mZB} |n q|
	\\
	&\qquad+ C(M) \sum_{r\geq 1}
	\sum_{\substack{|I_1| + \cdots |I_{r+1}| \leq |I|+1,\\ |I_{r+1}|\geq 2} }
	(1+s)^{r} (1+s^{-\frac 12}|B|_{I_1, \mZB})\cdots (1+s^{-\frac 12}|B|_{I_r, \mZB}|)|\widehat{Z}_{\mB}^{I_{r+1}} q|,
  \label{}
 \end{align}
 where there are $r$ factors of $n$ present in the collection $\widehat{Z}_{\mB}^{I_{1}},...\widehat{Z}_{\mB}^{I_{r+1}}$
 with at least one present in $\widehat{Z}_{\mB}^{I_{r+1}}$.
 Since $[n, Z_{\mB}] = 0$, we have
 $\widehat{Z}_{\mB}^{I_{r+1}} q = n^r Z_{\mB}^{K}$ where $|K| +r = |I_{r+1}|$
 and since $n = \frac{1}{s} \sBo$ we have
 $|\widehat{Z}_{\mB}^{I_{r+1}} q| \lesssim (1+s)^{-r+1} \sum_{|K|\leq |I_{r+1}|}
 |Z_{\mB}^K nq|$.
\end{proof}

\subsection{Proof of Proposition \ref{lefttocenterprop}}
\label{pflefttocenterprop}

The result is a consequence of the upcoming Lemmas \ref{linearcommjumpL},
\ref{transferlemL} and Proposition \ref{UpsilonLplusbd}, as follows.
Define the quantities
\begin{equation}
 \Upsilon_{I, {L}}^+(t_1) = \int_{t_0}^{t_1} \int_{\Gamma^L_t}
 X^\ell | Z_T^I Y_L^+ \psi_C|^2\, dS dt,
 \label{Upsilondef}
\end{equation}
where $Y_L^+$ is as in \eqref{YLdef2}. 

Combining Lemmas \ref{linearcommjumpL} and \ref{transferlemL}, for $|I| \leq N_L$, under our
hypotheses we have the bounds
\begin{equation}
 B^L_I(t_1) \lesssim 
 \Upsilon_{I, L}^+(t_1)
 + \sum_{|J| \leq |I|-1} B^L_J(t_1)
 + (c_0(\epsilon_0) + \epsilon_L^2)\epsilon_L^2 + c_0(\epsilon_0),
 \label{upsilonclaim}
\end{equation}
where $c_0$ is a continuous function with $c_0(0) = 0$.
Taking $\epsilon_0$ sufficiently small and using induction
we get
\begin{equation}
 B^L_I(t_1)\lesssim \Upsilon_{I, L}^+(t_1) + \epsilon_L^3,
 \label{}
\end{equation}
and the result now follows from the upcoming Proposition \ref{UpsilonLplusbd}
and the fact that $\epsilon_C^2 \lesssim \epsilon_L^4$
by \eqref{epsparameters}.
\hfill \qedsymbol

In the remainder of this section we prove Lemmas \ref{linearcommjumpL},
 \ref{transferlemL}, and Proposition \ref{UpsilonLplusbd}.
\subsubsection{Supporting lemmas for the proof of Proposition
\ref{lefttocenterprop}}

We start with a product estimate that we will use the handle the nonlinear
terms we encounter. For this result it is important that we take $\alpha < 3/2$
in the definitions of the vector fields $X_L, X_M$.
\begin{lemma}
	\label{generalproduct}
	Let $Q(\pa \psi, \pa \psi) = Q^{\alpha\beta}\pa_\alpha \psi \pa_\beta \psi$
	be a quadratic nonlinearity where the coefficients $Q^{\alpha\beta}$ are
	smooth functions satisfying the symbol-type condition \eqref{minksymb}.
	With $X = X_L$ or $X = X_M$, writing $X = X^\ell \pa_v + X^n \pa_u$,
	under the hypotheses of Proposition \ref{lefttocenterprop} we have
	\begin{equation}
	 \int_{t_0}^{t_1} \int_{\Gamma^L_t} |X^\ell| |Z^I\left( (1+v)^{-1} Q(\pa \psi_L, \pa \psi_L)\right)|^2\, dS dt
	 \lesssim 
     \epsilon_L^4.
	 \label{nonlinframebound}
	\end{equation}
\end{lemma}
\begin{proof}

	We first claim that under our hypotheses and by our choice of the field $X_L$
    we have
	\begin{equation}
	 (1+v)^{-1} |X^\ell_L| |\pa Z^K \psi_L|\lesssim \epsilon_L |X^n_L|,
	 \qquad |K| \leq N_L/2+1.
	 \label{leftvfpwdecay}
	\end{equation}
   Indeed, by the pointwise estimates from Lemma \ref{basicpwdecay},
  we have
  \begin{align}
   \frac{|X^\ell|}{1+v} |\pa Z^K \psi_L|
	 &\lesssim \epsilon_L\frac{|X_L^\ell|}{1+v} \frac{1}{(1+|u|)^{1/2}} \frac{1}{|X^n|^{1/2}}\\
   &\lesssim
   \epsilon_L \frac{(1+s) (\log s)^\alpha}{(1+s)^{1/2} (\log s)^{1/2}
   (\log \log s)^{\alpha/2}}
   \\
   &\lesssim \epsilon_L\frac{(1+s)^{1/2} (\log s)^{\alpha-1/2}}{(\log \log s)^{\alpha/2}}\\
   &\lesssim
	 \epsilon_L
   (1+s)^{1/2} (\log s)(\log\log s)^{\alpha}, \qquad |K| \leq N_L/2+1,
  \end{align}
  since $\alpha < 3/2$, and this is bounded by the right-hand side of \eqref{leftvfpwdecay}.

In particular, \eqref{leftvfpwdecay} and the pointwise decay bound \eqref{pwleft} imply
  \begin{equation}
   \frac{|X^\ell_L|}{(1+v)^2} |\pa Z^K \psi_L|^2|n Z^J \psi_L|^2
 	\lesssim \epsilon_L^2 \frac{|X^n_L|}{(1+v)(1+s)^{1/2}}|n Z^J \psi_L|^2,
   \label{}
  \end{equation}
  and so bounding $|\pa q|\lesssim |\evm q| + |n q| + |\nas q|$
  and using the simpler estimates
  \begin{alignat}{2}
   \frac{|X^\ell|}{(1+v)^2} |\pa Z^K \psi_L|^2 |\pa_v Z^J\psi_L|^2
   &\lesssim \epsilon_L^2 |X^\ell| |\evm Z^J \psi|^2, &&\qquad |K| \leq N_L/2+1,
   \\
   \frac{|X_L^\ell|}{(1+v)^2} |\pa Z^K \psi_L|^2 |\nas Z^J\psi_L|^2
   &\lesssim \epsilon_L^2 \frac{|X_L^\ell|}{(1+v)(1+s)^{1/2}} |\nas Z^J \psi_L|^2,
   &&\qquad |K| \leq N_L/2+1
   \label{}
  \end{alignat}
  we therefore have
  \begin{multline}
  	\sum_{|K| \leq N_L/2+1} \sum_{|J| \leq N_L}
  	\int_{t_0}^{t_1}\int_{\Gamma^L_t}
   \frac{|X_L^\ell|}{(1+v)^2} |\pa Z^K \psi_L|^2 |\pa Z^J \psi_L|^2\\
   \lesssim \epsilon_L^2
   \sum_{|J| \leq N_L}\int_{t_0}^{t_1}\int_{\Gamma^L_t}\left(|X_L^\ell||\pa_v Z^J\psi_L|^2 +
       \frac{|X_L^n|}{(1+v)(1+s)^{1/2}}
   |n Z^J \psi_L|^2 + \frac{|X_L^\ell|}{(1+v)(1+s)^{1/2}} |\nas Z^J\psi_L|^2\right)\\
   \lesssim
   \epsilon_L^2
   \sum_{|J| \leq N_L}\int_{t_0}^{t_1}\int_{\Gamma^L_t} |X_L^\ell| |\pa_v Z^J\psi_L|^2\, dS dt
   + \epsilon_L^4.
   \label{nonlinframebound0}
  \end{multline}
  which gives \eqref{nonlinframebound} after bounding
	$(1+v)|Z^I (1+v)^{-1} Q(\pa \psi_L, \pa \psi_L)|
	\lesssim  \sum_{|J| \leq |I|}\sum_{|K| \leq |I|/2+1} |\pa Z^J\psi_L||\pa Z^K\psi_L|$
    and using the bootstrap assumption \eqref{BIbds}.
\end{proof}

The first step in the proof of Proposition \ref{lefttocenterprop} is to commute $\evm$ with $Z^I$
and write the result in terms of the nonlinear boundary operator $Y_L^-$
(recall the definition \eqref{YLdef1}).
\begin{lemma}
	\label{linearcommjumpL}
    Under the hypotheses of Proposition \ref{lefttocenterprop}, we have
\begin{multline}
    \int_{t_0}^{t_1} \int_{\Gamma^L_t} v f(v) |\pa_v \psi_L^I|^2\,dS
         dt
	\lesssim \int_{t_0}^{t_1} \int_{\Gamma^L_t} vf(v)
	|Z^I Y_L^-\psi_L|^2\, dS dt
    + (c_0(\epsilon_0) +\epsilon_L^2) \epsilon_L^2
    \\ + \sum_{|K| \leq |I|-1} \int_{t_0}^{t_1}\int_{\Gamma^L_t} vf(v) |\pa_v \psi_L^J|^2\, dSdt.
	\label{linearcommjumpLbd}
\end{multline}
\end{lemma}
\begin{proof}
 We just prove the result with $\psi_L^I = r Z^I \phi_L$ replaced with
 $Z^I (r\phi_L)$, the difference being straightforward to handle using
 arguments we have by now used many times.

Recalling the definition of $Y_L^-$ from
\eqref{YLdef1}, along $\Gamma^L$ we have
\begin{align}
 |\evm Z^I \psi_L| &\lesssim |Z^I\evm\psi_L| + |[Z^I, \evm] \psi_L|\\
&\lesssim
|Z^IY_L^-\psi_L| + |Z^I \left((1+v)^{-1} Q(\pa \psi_L, \pa \psi_L)\right)|+ |[Z^I, \evm] \psi_L| \\
&\lesssim
|Z^IY^L_-\psi_L|+ |Z^I \left((1+v)^{-1} Q(\pa \psi_L, \pa \psi_L)\right)|\\
&+ \sum_{|J| \leq |I|-1} |\evm Z^J \psi_L| + \frac{1}{(1+v)^{3/2}} |n Z^J \psi_L|
+ \frac{1}{(1+v)^{3/4}} |\nas Z^J \psi_L|
 \label{boundaryerrorgoalleftminusone}
\end{align}
where we used \eqref{ellZcomm} and the fact that
$|u| \lesssim s^{1/2}$ along $\Gamma^L$ to bound $|[Z^I, \evm]\psi_L|$. Using Lemma
\ref{generalproduct} to control the
quadratic term here, it remains only to handle the contribution from
the last line of \eqref{boundaryerrorgoalleftminusone}.

The contribution from the first term there is bounded by the last term in 
\eqref{linearcommjumpLbd}.
For the other two terms,
we note that along $\Gamma^L$, we have
\begin{align}
 \frac{1}{(1+v)^3} vf(v) &
 \lesssim c_0(\epsilon_0)\frac{1}{1+v} \log (1+s) (\log \log (1+s))^\alpha,
 \label{ntermbdez}\\
 \frac{1}{(1+v)^{3/2}}vf(v) &\lesssim c_0(\epsilon_0)(1+s)^{1/2}\log (1+s) (\log \log (1+s))^\alpha,
 \label{angletermbdez}
\end{align}
for a continuous function $c_0(\epsilon_0)$ with $c_0(0) = 0$,
where recall $v \gtrsim \frac{1}{\epsilon_0}$ along $\Gamma^L$.
As a result, using \eqref{controlofpsiLL} for $|J| \leq N_L$ we have the bounds
\begin{multline}
 \int_{t_0}^{t_1} \int_{\Gamma^L_t} vf(v)\frac{1}{(1+v)^3} |nZ^J \psi_L|^2\, dSdt\\
 \lesssim
 c_0(\epsilon_0) \int_{t_0}^{t_1} \int_{\Gamma^L_t} \frac{1}{1+v} \log (1+s) (\log \log (1+s))^\alpha |nZ^J \psi_L|^2\, dSdt
 \lesssim c_0(\epsilon_0) \epsilon_L^2,
\end{multline}
and
\begin{multline}
  \int_{t_0}^{t_1} \int_{\Gamma^L_t} vf(v) \frac{1}{(1+v)^{3/2}} |\nas Z^J \psi_L|^2\, dSdt
	\\
	\lesssim
	 c_0(\epsilon_0)\int_{t_0}^{t_1} \int_{\Gamma^L_t} (1+s)^{1/2}\log (1+s) (\log \log (1+s))^\alpha |\nas Z^J \psi_L|^2
	\lesssim  c_0(\epsilon_0) \epsilon_L^2,
 \label{ezbdcons}
\end{multline}
as needed.

\end{proof}

We now want to handle the nonlinear boundary operator $Y_L^-$ appearing on the right-hand side of
\eqref{linearcommjumpLbd}. For this we use Lemma \ref{minkdecomplemma} to replace
the fields $Z^I$ with the tangential fields $Z_T^I$. This generates
error terms involving the function $B^L$ which defines the boundary and which
are bounded using our assumptions on the geometry of the shocks.
\begin{lemma}
		\label{transferlemL}
		With $B^L_K$ defined as in \eqref{BIdef2}, under
		the hypotheses of Proposition \ref{lefttocenterprop}, for
		$|I| \leq N_L$ we have
		\begin{equation}
		 \int_{t_0}^{t_1} \int_{\Gamma^L_t} X^\ell
	 	|Z^I Y_L^-\psi_L|^2\, dS dt
		\lesssim
		\Upsilon_L^+(t_1)
		+
		(c_0(\epsilon_0) +\epsilon_L^2)\epsilon_L^2
        + \sum_{|K| \leq |I|-1} B^L_{K}(t_1) + c_0(\epsilon_0),
		 \label{transferlemLbd}
		\end{equation}
		for a continuous function $c_0$ with $c_0(0) = 0$.
\end{lemma}
\begin{proof}
 We first use Lemma \ref{minkdecomplemma} to convert
 the fields $Z$ into tangential fields $Z_T$,
 \begin{multline}
  |Z^I Y_L^-\psi_L|
	\lesssim
	|Z_T^I Y_L^-\psi_L|
	+ C(M)(1 +|u|)
	\sum_{|J| \leq |I|-1}
	\left(
	|Z^J n Y_L^-\psi_L| + (1+|u|)^{-1} |Z^J Y_L^- \psi_L|\right)\\
	+ C(M)|B|_{I, \mathcal{Z}_{m}}
	(1+|u|)\sum_{|K| \leq |I|/2+1}
	\left(|Z^K n Y_L^- \psi_L| + (1+|u|)^{-1}|Z^K Y_L^- \psi_L|\right).
 \end{multline}
 Since the fields $Z_T$ are tangent to the shock,
 by the boundary condition \eqref{introbcL} we have
 $$Z_T^I Y_L^-\psi_L = Z_T^I Y_L^+\psi_C + Z_T^IG,\qquad \text{ at } \Gamma^L$$ and so recalling
 the definition of $\Upsilon_L^+$ from \eqref{Upsilondef}, to
 conclude it is enough to prove that for $|I| \leq N_L$ we have
 the following estimates,
\begin{equation}
  \int_{t_0}^{t_1} \int_{\Gamma^L_t} X^\ell(1+|u|)^2|Z^J n Y_L^-\psi_L|^2\, dSdt
  \lesssim 
    \sum_{|K| \leq |I|-1} B^L_K(t_1)+
     (c_0(\epsilon_0) + \epsilon_L^2)\epsilon_L^2 \qquad |J| \leq |I|-1
  \label{bdytermleft1}
\end{equation}
\begin{equation}
 \int_{t_0}^{t_1} \int_{\Gamma^L_t} X^\ell|Z^J Y_L^-\psi_L|^2\, dSdt
 \lesssim
    \sum_{|K| \leq |I|-1} B^L_K(t_1)+
    (c_0(\epsilon_0) + \epsilon_L^2)\epsilon_L^2 \qquad |J| \leq |I|-1
 \label{bdytermleft2}
\end{equation}
\begin{multline}
 \int_{t_0}^{t_1} \int_{\Gamma^L_t} X^\ell
 |B^L|_{I, \mathcal{Z}_{m}}^2
 \left((1+|u|)^2|Z^K n Y_L^- \psi_L|^2 + |Z^K Y_L^- \psi_L|^2\right)
 \\ \lesssim 
    \sum_{|K| \leq |I|-1} B^L_K(t_1)+
     (c_0(\epsilon_0) + \epsilon_L^2)\epsilon_L^2
 \qquad |K| \leq |I|/2+1,
 \label{bdytermleft3}
\end{multline}
\begin{align}
 \int_{t_0}^{t_1} \int_{\Gamma^L_t} X^\ell |Z_T^I G|^2\, dS dt
 &\lesssim  
 \sum_{|K| \leq |I|-1} B^L_K(t_1)+
 (c_0(\epsilon_0) + \epsilon_L^2)\epsilon_L^2 + c_0(\epsilon_0),
 \label{errorbdytermleft}
\end{align}
where the implicit constants depend on $M$.

To prove \eqref{bdytermleft1}, we recall the definition of $Y_L^-$ from \eqref{YLdef1}
 and use that $(1+|u|) |\pa q| \lesssim \sum_{Z \in \mathcal{Z}_m}|Z q|$, which gives
	\begin{equation}
		(1+|u|)|Z^Jn Y_L^- \psi_L|
		\lesssim
		(1+|u|) |Z^J n \evm \psi_L| + \sum_{|J'| \leq |J| +1}
		 |Z^{J'} \left((1+v)^{-1} Q(\pa \psi_L, \pa \psi_L)\right)|.
	 \label{}
	\end{equation}
	By Lemma \ref{generalproduct}, for $|J| \leq N_L-1$,
	the contribution from the nonlinear term here
	is bounded by the right-hand side of \eqref{transferlemLbd}. For the first
	term we use the equation \eqref{waveext2} for $\psi_L$ and bound
	\begin{equation}
	 (1+|u|) |Z^J n \evm \psi_L|
	 \lesssim (1+|u|) |Z^J \sDelta \psi_L| + \sum_{|J'| \leq |J|+1}
	 |Z^{J'} \left((1+v)^{-1} Q(\pa \psi_L, \pa \psi_L)\right)|
	 + (1+|u|)|Z^J F'|,
	 \label{useeqnbdy0}
	\end{equation}
	where we again used that $(1+|u) |\pa q| \lesssim \sum_{Z \in \mathcal{Z}_m} |Zq|$.
	Just as above, the contribution from the quadratic term here
	is bounded by the right-hand side of \eqref{transferlemLbd}. The contribution
	from $F'$ is simpler to deal with so we skip it.
To handle the first term on the right-hand side of \eqref{useeqnbdy0},
we write $\sDelta = \frac{1}{r^2} \Omega^2$ and bound
\begin{multline}
(1+|u|)|Z^J \sDelta \psi_L|
\lesssim
\sum_{|J'| \leq |J|+2} \frac{1+|u|}{(1+v)^2}| Z^{J'}\psi_L|\\
\lesssim
\sum_{|J'| \leq |J|+1} \frac{1+|u|}{1+v}
\left(|\evm Z^{J'}\psi_L| + |\nas Z^{J'}\psi_L|\right)
+ \frac{(1+|u|)^2}{(1+v)^2}|n Z^{J'}\psi_L|.
\label{dealingwithsDelta0}
\end{multline}
 Now we bound
 $(1+|u|)(1+v)^{-1} \lesssim (1+v)^{-3/4}$ and $(1+|u|)^2(1+v)^{-2} \lesssim
(1+v)^{-3/2}$.
We now recall that $|J| \leq |I|-1$ and argue as in \eqref{ntermbdez}-\eqref{ezbdcons} to get
\eqref{bdytermleft1} after additionally bounding $(1+|u|)(1+v)^{-1} \lesssim c_0(\ve_0)$
to handle the first term here.

To prove \eqref{bdytermleft2}, we just note that
 $\sum_{|K| \leq |I|-1}B^L_K(t_1)$ appears on the
right-hand side and use the product estimate
from Lemma \ref{generalproduct}, the definition of $Y_L^-$,
along with the argument of Lemma \ref{linearcommjumpL} applied in reverse to deal with $[\ell^m,Z^k]$ (see \eqref{boundaryerrorgoalleftminusone}).  

To prove \eqref{bdytermleft3}, we bound
\begin{multline}
 \int_{t_0}^{t_1} \int_{\Gamma^L_t} X^\ell
 |B|_{I, \mathcal{Z}_{m}}^2
 \left((1+|u|)^2|Z^K n Y_L^- \psi_L|^2 + |Z^K Y_L^- \psi_L|^2\right)\, dS dt
 \\
 \lesssim
 \left(
 \int_{t_0}^{t_1} \sup_{\Gamma^L_t}|X^\ell|
 \left[(1+|u|)^2 |Z^K n Y_-^L\psi_L|^2 + |Z^KY^L_-\psi_L|^2\, dt\right]\right)
 \left( \sup_{t_0 \leq t\leq t_1} \int_{\Gamma^L_t} |B|_{I, \mathcal{Z}_m}^2\,dS\right)\\
 \lesssim C(M)\int_{t_0}^{t_1} \sup_{\Gamma^L_t}|X^\ell|
 \left[(1+|u|)^2 |Z^K n Y_-^L\psi_L|^2 + |Z^KY^L_-\psi_L|^2\, dt\right].
 \label{highnormbL1}
\end{multline}
To handle this last term, we just use Sobolev embedding on $\Gamma^L_t \sim \mathbb{S}^2$
and the bounds \eqref{bdytermleft1}-\eqref{bdytermleft2} that we just proved. Specifically, we use
that the fields $\Omega_T$ span the tangent space to $\Gamma^L_t$ at each
point and bound
$\sup_{\Gamma^L_t} |q| \lesssim \sum_{|R| \leq 2} \|\Omega_T^R q\|_{L^2(\Gamma^L_t)}$
and using \eqref{lowordertransfer} to bound this in terms of vector fields
applied to $q$ gives $\sup_{\Gamma^L_t} |q| \lesssim C(M) \sum_{|R| \leq 2} |Z^R q|$.
Applying this with $q = Z^K n Y_-^L\psi_L$ and $q = Z^K Y_-^L\psi_L$
where $|K| \leq N_L/2+1$ and applying
\eqref{bdytermleft1}-\eqref{bdytermleft2} gives the result.

It remains to prove the bound \eqref{errorbdytermleft} for the remainder
 term $G$, which is given explicitly in \eqref{Gexpression}. The terms in
 \eqref{alternateGexpression} and the last term in
 \eqref{Gexpression} are all straightforward to handle using
 similar arguments to the ones we have encountered many times by now. Note that
 the term $c_0(\epsilon_0)$ on the right-hand side of \eqref{errorbdytermleft}
is needed to control the quantities involving $\Sigma$ in \eqref{alternateGexpression}.
 We just
 show how to deal with the second term in \eqref{Gexpression}, and we will
 prove
 \begin{equation}
  \int_{t_0}^{t_1} \int_{\Gamma^L_t} X^\ell \left| Z_T^I \frac{s}{u} [\nas \psi]^2\right|\, dS dt
	\lesssim
  \sum_{|K| \leq |I|-1} B^L_K(t_1)+
 (c_0(\epsilon_0) + \epsilon_L^2)\epsilon_L^2.
  \label{}
 \end{equation}
 Arguing as above to replace $Z_T$ with $Z$, bounding
 $\frac{s}{u} \lesssim s^{1/2}$ and performing straightforward
 estimates, the main ingredients needed for this are the bounds
 \begin{align}
  \sum_{|J| \leq |I|}
	\sum_{|K| \leq |I|/2+1}
	\int_{t_0}^{t_1} \int_{\Gamma^L_t} (1+s)X^\ell
	\left( |\nas \psi^J_L|^2 |\nas \psi^K_L|^2\right)\, dS dt
	&\lesssim \epsilon_L^2,\label{GLL0}\\
	\sum_{|J| \leq |I|}
	\sum_{|K| \leq |I|/2+1}
	\int_{t_0}^{t_1} \int_{\Gamma^L_t} (1+s)X^\ell
	\left( |\nas Z^J \psi_C|^2 |\nas Z^K \psi_C|^2\right)\, dS dt
	&\lesssim \epsilon_L^2,
  \label{GLL1}
 \end{align}
 which we now prove.
 For both estimates we handle the lower-order terms by bounding
 $|\nas q|\lesssim (1+v)^{-1} |\Omega q|$,
 \begin{equation}
  |\nas \psi^K_L|^2 \lesssim
	\frac{1}{(1+v)^2} |\Omega \psi^K_L|^2,
	\qquad
	|\nas Z^K \psi_C|^2 \lesssim
	\frac{1}{(1+v)^2} |\Omega Z^K \psi_C|^2.
  \label{psiCK}
 \end{equation}
 By the Poincare inequality
 \eqref{sobolevpoinL} from Lemma \ref{sobolevshock2} and Sobolev embedding,
	for $|K| \leq |I|/2+1$ we have
	\begin{equation}
	 |\Omega \psi^K_L|^2 \lesssim \epsilon_L^2,
	 \label{}
	\end{equation}
	at the shock,
and the left-hand side of \eqref{GLL0} is then bounded by
 \begin{align}
  \sum_{|J| \leq |I|}
	\sum_{|K| \leq |I|/2+1}
	\int_{t_0}^{t_1} \int_{\Gamma^L_t} (1+s)X^\ell&
	\left( |\nas \psi^J_L|^2 |\nas \psi^K_L|^2\right)\, dS dt\\
	&\lesssim\epsilon_L^2
	\sum_{|J| \leq |I|}
	\int_{t_0}^{t_1} \int_{\Gamma^L_t} \frac{(1+s)X^\ell}{(1+v)^{2}} |\nas \psi^J_L|^2 \, dS dt
	\\
	&\lesssim
	\epsilon_L^2
	\sum_{|J| \leq |I|}
	\int_{t_0}^{t_1} \int_{\Gamma^L_t} \frac{(1+s)^2 (\log s)^\alpha}{1+v} |\nas \psi^J_L|^2 \, dS dt
	\lesssim \epsilon_L^4,
  \label{}
 \end{align}
 by the bound \eqref{controlofpsiLL} for the energies
 on the timelike side of the left shock, since $|I| \leq N_L$.

 For \eqref{GLL1}, since we are below the top-order number of derivatives
 of $\psi_C$ we can in fact use \eqref{psiCK} in each factor of $\psi_C$ which
 gives
 \begin{multline}
  \sum_{|J| \leq |I|}
	\sum_{|K| \leq |I|/2+1}
	\int_{t_0}^{t_1} \int_{\Gamma^L_t} (1+s)X^\ell
	\left( |\nas Z^J \psi_C|^2 |\nas Z^K \psi_C|^2\right)\, dS dt
	\\
	\lesssim
	\sum_{|J| \leq |I|}
	\int_{t_0}^{t_1} \int_{\Gamma^L_t} \frac{(1+s)X^\ell}{(1+v)^4}
	|\Omega Z^J \psi_C|^4\, dS dt
	\lesssim \epsilon_C^4,
  \label{}
 \end{multline}
 where we used that $\frac{(1+s)X^\ell}{(1+v)^4} \lesssim \frac{s}{v^2}$,
the Hardy inequality \eqref{controlangularhardyright}, and the bound
\eqref{leftbdytrivialbds} for the energy at the right shock for $|I| \leq N_L\leq N_C-3$.
Since we have taken $\epsilon_C \leq \epsilon_L$ this gives \eqref{GLL1}.
\end{proof}

To complete the proof of Proposition \ref{lefttocenterprop},
we need to prove the bound for the quantity $\Upsilon_{I, L}^+$.

\subsubsection{Control of $\Upsilon_{I, L}^+$}
If it was not for the large
(relative to the weights we use in the central region) weight $|X^\ell|$,
it would be straightforward to control $\Upsilon_{I, L}^+$
by using the bounds for $\psi_C$
in Lemma \ref{bdsforpsiCalongshock}. Instead, we are going to handle
$\Upsilon_{I, L}^+$ by integrating to the right shock. Schematically,
this amounts to trading a factor of $|u| \sim (\log t)^{1/2}$
(the width of $D^C_t$) for a
derivative of (vector fields applied to) $Y_L^+\psi_C$, see \eqref{ftc0}.
We can afford this trade in the central region,
ultimately because
we control the vector field $ \sBo = s\pa_u\sim (\log t) \pa_u$ applied to $\psi_C$.
As a result, this trade in fact gains us a factor of
$(\log t)^{1/2}$, which is enough to close our estimates.
\begin{prop}
 \label{UpsilonLplusbd}
 Under the hypotheses of Proposition \ref{lefttocenterprop}, provided $\epsilon_0$
 is taken sufficiently small, we have
 \begin{equation}
     \sum_{|I| \leq N_L} \Upsilon_{I, L}^+(t_1) \lesssim \epsilon_C^2
  \label{UpsilonepsilonCbd}
 \end{equation}
\end{prop}
The proof relies on the upcoming Lemmas \ref{UpsilonLplusbulk}
and \ref{UpsilonLplusbdy}.

\begin{proof}
	We start by recalling that
quantities in the definition of $\Upsilon_{I, L}^+$ involve
tangential vector fields $Z_T$. For our purposes it is simpler
if we replace these with the usual vector fields $Z$. We therefore
use Lemma \ref{minkdecomplemma} and the fact that $(1+|u|) |\pa q|
\lesssim \sum_{Z \in \mathcal{Z}_m}$ to bound
\begin{multline}
 \Upsilon_{I, L}^+
 \leq C(M)
 \sum_{|J| \leq |I|}
 \int_{t_0}^{t_1} \int_{\Gamma^L_t} |X^\ell| |Z^J Y_L^+\psi_C|^2\, dS dt
 \\
 +C(M)
 \sum_{|K| \leq |I|/2+1}
 \int_{t_0}^{t_1} \int_{\Gamma^L_t} |X^\ell| |B^L|_{I, m}|Z^K Y_L^+\psi_C|^2\, dS dt.
 \label{Upsilonbd0L}
\end{multline}

We claim that this implies the following bound,
\begin{multline}
 \Upsilon_{I, L}^+(t_1)
 \leq C(M) \sum_{|J| \leq |I|} \int_{t_0}^{t_1} \int_{D^C_t}(1+t)(1+\log t)^{3/2}(1+\log\log t)^{\alpha}
 |\pa_r Z^J Y_L^+\psi_C|^2\ dt\\
 + C(M)\sum_{|J| \leq |I|}\int_{t_0}^{t_1}
 \int_{\Gamma^R_t} (1+t)(1+\log t)(1+\log\log t)^{\alpha} |Z^J Y_L^+\psi_C|^2\, dS dt.
 \label{Upsilonclaim0}
\end{multline}
In the upcoming Lemmas \ref{UpsilonLplusbulk} and \ref{UpsilonLplusbdy},
we bound the right-hand side here by the right-hand side of \eqref{UpsilonepsilonCbd}.

The idea behind \eqref{Upsilonclaim0} is to control the quantities $q^J = Z^J Y_L^+\psi_C|_{\Gamma^L_t}$
by integrating along the ray $x/|x| = \omega$ at fixed time to the right shock,
using the bounds for $\psi_C$ in the central region to handle the interior
term this generates and the boundary condition at the right shock to handle
the boundary term this generates.

In particular, for $A = L, R$, we let $r_A(t', \omega)$ denote the value of $|x|$ of the point lying
at the intersection of the sets $\{t = t'\}$, $\{x/|x| = \omega\}$
and $\Gamma^A$. That is, $r_L$ is defined by the property that
$t - r_L(t,\omega) = B^L(t, r_L(t,\omega)\omega)$
and similarly for $r_R$.
For any function $q$ defined in $D^C$, fixing $t, \omega$ and integrating
from $r = r_L(t, \omega)$ to $r = r_R(t,\omega)$ we find
\begin{multline}
 |q(t, r_L(t,\omega)\omega)| \leq |q(t, r_R(t,\omega)\omega)|
	+ \int_{r_L(t,\omega)}^{r_R(t,\omega)} |(\pa_r q)(r')|\, dr'\\
	\leq |q(t, r_R(t,\omega)\omega)| +
	|r_L(t,\omega) - r_R(t,\omega)|^{1/2} \left( \int_{r_{L}(t,\omega)}^{r_R(t,\omega)}
	|\pa_r q|^2\, dr\right)^{1/2}\\
	\lesssim |q(t, r_R(t,\omega)\omega)|
	+ (\log t)^{1/4} \left( \int_{r_{L}(t,\omega)}^{r_R(t,\omega)}
	|\pa_r q|^2\, dr\right)^{1/2},
 \label{}
\end{multline}
and in particular we have the bound
\begin{equation}
 \int_{\Gamma^L_t} |q|^2\, dS
 \lesssim
 \int_{\mathbb{S}^2} |q(t, r_L(t,\omega)\omega)|^2\, dS(\omega)
 \lesssim
 \int_{\Gamma^R_t} |q|^2\, dS
 + (\log t)^{1/2} \int_{D^C_t} |\pa_r q|^2
 \label{ftc0}
\end{equation}
using that by \eqref{volumeformula} the surface measures $dS$ and
$dS(\omega)$ are equivalent.
Applying this to $q = q^J$ as in the above paragraph and integrating in
$t$, we find that the first term in \eqref{Upsilonbd0L} is bounded by
\begin{align}
 \int_{t_0}^{t_1} \int_{\Gamma^L_t}|X^\ell| |Z^J Y_L^+\psi_C|^2\, dS dt
 &\lesssim \int_{t_0}^{t_1}
 \int_{\Gamma^L_t} (1+t)(1+\log t)(1+\log\log t)^{\alpha} |Z^J Y_L^+\psi_C|^2\, dS dt
 \\
 &\lesssim
 \int_{t_0}^{t_1}
 \int_{\Gamma^R_t} (1+t)(1+\log t)(1+\log\log t)^{\alpha} |Z^J Y_L^+\psi_C|^2\, dS dt \label{parZJYL0}
\\
 &+ \int_{t_0}^{t_1} \int_{D^C_t}  (1+t)(1+\log t)^{3/2}(1+\log\log t)^{\alpha}
 |\pa_r Z^J Y_L^+\psi_C|^2\ dt.
 \label{parZJYL}
\end{align}

We now perform a similar manipulation to the second term in \eqref{Upsilonbd0L}.
We first bound
\begin{equation}
 \int_{t_0}^{t_1}\int_{\Gamma^L_t}
 |X^\ell| |B|_{I, m}^2 |Z^K Y_L^+ \psi_C|^2\, dS dt
 \lesssim
 \left(\sup_{t_0 \leq t \leq t_1} \int_{\Gamma^L_t} |B^L|_{I, m}^2\,dS\right)
 \left(\int_{t_0}^{t_1} \sup_{\Gamma^L_t} |X^\ell| |Z^K Y_L^+\psi_C|^2\, dt\right).
 \label{}
\end{equation}
Using Sobolev embedding on $\Gamma^L_t$ as in the proof of the last result,
 using the bound \eqref{leftgeom} for the high-order derivatives
 of $B^L$, and using \eqref{parZJYL}
 again, after returning to \eqref{Upsilonbd0L} we have the claim
 \eqref{Upsilonclaim0}.
%
%
\end{proof}

We now prove the needed lemmas. First, we get control over
certain time-integrated weighted norms of $\pa Z^J Y_L^+\psi_C$
in $D^C$. The idea
behind this estimate is that as usual
the most dangerous term is when $\pa = \pa_u$. Since
$Y_L^+ =\evmB$ up to nonlinear terms, and
since the equation \eqref{waveint0} in the central region
expresses $\pa_u \evmB \psi_C$
in terms of $\sDelta \psi_C$ and nonlinear terms, this term
can be handled.
\begin{lemma}
 \label{UpsilonLplusbulk}
 Under the hypotheses of Proposition \ref{lefttocenterprop},
 provided $\epsilon_0$ is taken sufficiently small, for
 $|J| \leq N_L\leq N_C - 4$, we have
	\begin{equation}
	 \int_{t_0}^{t_1} \int_{D^C_t}  (1+t)(1+\log t)^{3/2}(1+\log\log t)^{\alpha}
	 |\pa_r Z^J Y_L^+\psi_C|^2\ dt
	 \lesssim \epsilon_C^2.
	 \label{parZJYLlemstatement}
	\end{equation}
\end{lemma}

\begin{proof}
We recall the definition of $Y_L^+$ from \eqref{YLdef2} and bound
	\begin{align}
	 |\pa_r Z^J Y_L^+\psi_C|
	 &\lesssim |\pa_r Z^J \evmB \psi_C| + |\pa Z^J \left( (1+v)^{-1} Q(\pa \psi_C, \pa \psi_C)\right)|.
	 \label{parZJYL2}
 \end{align}
 To handle the contribution from the nonlinear term here,
 we use Lemma \ref{ZmBtoZlem} and the fact that we have $(1+s)|\pa q|\lesssim
 \sum_{Z_{\mB}\in \mZB} |Z_{\mB} q|$ to get that for $|J| \leq N_C-3$,
 \begin{multline}
|\pa Z^J \left( (1+v)^{-1} Q(\pa \psi_C, \pa \psi_C)\right)|
\lesssim
 \frac{1}{1+v} \frac{1}{1+s}\sum_{|J'| \leq |J|+1} \sum_{|K| \leq |J|/2+1}
|\pa Z_{\mB}^{J'}\psi_C| |\pa Z_{\mB}^{K} \psi_C|\\
\lesssim\epsilon_C
\frac{1}{1+v} \frac{1}{(1+s)^{3/2}}  \sum_{|J'| \leq |J|+1}
|\pa Z_{\mB}^{J'}\psi_C|,
  \label{}
 \end{multline}
 by the pointwise bound \eqref{pwcentral}. Therefore, the contribution
 from this term into \eqref{parZJYLlemstatement} is bounded by
 \begin{multline}
  \int_{t_0}^{t_1} \int_{D^C_t}  (1+t)(1+\log t)^{3/2}(1+\log\log t)^{\alpha}
	|\pa Z^J \left( (1+v)^{-1} Q(\pa \psi_C, \pa \psi_C)\right)|^2\ dt\\
	\lesssim \epsilon_C^2\sum_{|J'| \leq |J|+1}
	  \int_{t_0}^{t_1} \int_{D^C_t} \frac{1}{1+t} \frac{1}{(1+\log t)^{3/2}}(1+\log\log t)^{\alpha}
		|\pa Z_{\mB}^{J'}\psi_C|^2\, dt
		\lesssim \epsilon_C^4,
  \label{}
 \end{multline}
 using the bound for the lower-order energy \eqref{loworderenergydef}
 and the fact that $1/((1+t) (1+\log (1+t))^{3/2})$ is time-integrable.

 We now handle the contribution from the first term in \eqref{parZJYL2}.
 Recalling that $\pa_r = \pa_v - \pa_u$, we first bound the contribution
 from $\pa_v$ as follows,
 \begin{multline}
  \int_{t_0}^{t_1} \int_{D^C_t}
	(1+t)(1+\log t)^{3/2}(1+\log\log t)^{\alpha} |\pa_v Z^J \evmB \psi_C|^2\, dt\\
	\lesssim
	\int_{t_0}^{t_1} \int_{D^C_t}
	\left(\frac{1}{(1+t)^2}(1+\log t)^{3/2}(1+\log\log t)^{\alpha} \right)(1+v)|Z Z^J \evmB \psi_C|^2\, dt
	\lesssim
	\epsilon_C^2,
  \label{cmon00}
 \end{multline}
 after using the bound for the energy \eqref{EtopC},
 \eqref{ZtoZmBbd0} to convert $Z$ derivatives into $Z_{\mB}$ derivatives, and
 the bound \eqref{evmbZcomm} to handle the commutators between
 the $Z_{\mB}$ and $\evmB$.

 It remains to prove the analogous bound for $\pa_u$, namely
 \begin{equation}
  \int_{t_0}^{t_1} \int_{D^C_t} (1+t)(1+\log t)^{3/2}(1+\log\log t)^{\alpha}|\pa_u Z^J \evmB \psi_C|^2\, dt
	\lesssim \epsilon_C^2.
  \label{paucontrib}
 \end{equation}
 By the bound \eqref{nZcomm} for the commutator $[\pa_u, Z^J]$, we have
 \begin{multline}
  |\pa_u Z^J \evmB \psi_C|
	\lesssim |Z^J \pa_u \evmB\psi_C|
	+ \sum_{|J'| \leq |J|-1}
	|\pa_u Z^{J'}\evmB \psi_C|
	+ |\nas Z^{J'}\evmB \psi_C|
	+ |\evm Z^{J'} \evmB \psi_C|\\
	\lesssim
	|Z^J \pa_u \evmB\psi_C|
	+ \sum_{|J'| \leq |J|-1}
	|\pa_u Z^{J'}\evmB \psi_C|
	+ \frac{1}{1+v} |Z Z^{J'}\evmB \psi_C|.
  \label{cmon0}
 \end{multline}
 To handle the first term here, we use the equation
 \eqref{waveint0} which gives
 \begin{align}
  |Z^J &\pa_u\evmB \psi_C|\\
	&\lesssim |Z^J \sDelta \psi_C|
	+| Z^J \pa_\mu(\tfrac{u}{vs} a^{\mu\nu}\pa_\nu \psi_C)|
	+|Z^J \pa_\mu(\gamma^{\mu\nu}\pa_\nu \psi_C)|
	+ |Z^J \pa_\mu P^\mu|
	+ |Z^J F| + |Z^J F_\Sigma|\\
	&\lesssim\sum_{|J'|\leq |J|}\left(
	\frac{1}{(1+v)^2}
	 |Z^{J'} \Omega^2 \psi_C|
	 + \frac{1}{(1+v)(1+s)^{1/2}}
	 |\pa^2 Z^{J'} \psi_C|
	 + \frac{1}{(1+v)(1+s)}
	 |\pa Z^{J'}\psi_C|\right)\label{cmon000}
	 \\
	 &+|Z^J \pa_\mu(\gamma^{\mu\nu}\pa_\nu \psi_C)|
 	+ |Z^J \pa_\mu P^\mu|
 	+ |Z^J F| + |Z^J F_\Sigma|.
  \label{cmon}
 \end{align}
 In getting the above bound we used that $|\pa \tfrac{u}{vs}|\lesssim \tfrac{1}{vs}$,
 that $Z \frac{u}{vs} = c_Z \frac{u}{vs}$ for smooth functions $c_Z$ satisfying
 the symbol condition \eqref{strongsymbol}
 and ignored the structure of the coefficients $a^{\mu\nu}$ which is not important
 for this part of the argument.
 We just show how to handle the terms on the first line of \eqref{cmon000}, the terms on the
 second line being very similar after noting that $\gamma$ behaves like
 $(1+v)^{-1}\pa \psi_C$ and that we have the pointwise bound $|\pa \psi_C|
 \lesssim (1+s)^{-1/2}$, and that the terms $\pa_\mu P^\mu, F, F_{\Sigma}$ are better-behaved
 (recall that these quantities are given explicitly in
 \eqref{centraleqnlowest}-\eqref{FSigmaformula}).

The contribution from the third term on the first line of \eqref{cmon000} into
\eqref{paucontrib} is actually the most dangerous one,
and it is bounded by
\begin{align}
	\int_{t_0}^{t_1} \int_{D^C_t} (1+t)(1+\log t)^{3/2}&(1 + \log \log t)^\alpha
	\left( \frac{1}{1+v} \frac{1}{1+s} |\pa Z^{J'}\psi_C|\right)^2\,dt\label{maincmon000pt}\\
 &\lesssim
 \int_{t_0}^{t_1} \int_{D^C_t}
 \frac{1}{1+t} \frac{1}{(1 + \log t)^{1/2}} (1 + \log \log t)^\alpha|\pa Z^{J'} \psi_C|^2\, dt\\
 &\lesssim
 \int_{t_0}^{t_1} \int_{D^C_t}
 \frac{1}{1+t} \frac{1}{(1 + \log t)^{3/2}}(1 + \log \log t)^\alpha \left( s|\pa Z^{J'} \psi_C|^2\right)\, dt\\
 &\lesssim
 \sum_{|J''| \leq |J'|}
 \int_{t_0}^{t_1} \int_{D^C_t}
 \frac{1}{1+t} \frac{1}{(1 + \log t)^{5/4}} \left(s|\pa Z_{\mB}^{J''} \psi_C|^2\right)\, dt\\
 &\lesssim \epsilon_C^2,
 \label{}
\end{align}
for $|J'| \leq N_C -3$. Here, we have used \eqref{ZtoZmBbd0} to convert
the fields $Z$ into the fields $Z_{\mB}$,
the definition of the energy in the central region, the fact that
$1/((1+t)(1+\log t)^{5/4})$ is time-integrable, and bounded
$(\log \log t)^\alpha \lesssim (\log t)^{1/4}$. As for the second term
on the first line of \eqref{cmon000}, its contribution is bounded by
\begin{align}
 \int_{t_0}^{t_1} \int_{D^C_t} \frac{1}{1+t} &(1 + \log t)^{1/2}(1+\log\log t)^\alpha
 |\pa^2 Z^{J'} \psi_C|^2\, dt\\
 &\lesssim \int_{t_0}^{t_1} \int_{D^C_t} \frac{1}{1+t} (1 + \log t)^{1/2}(1+\log\log t)^\alpha
 |\pa^2 Z^{J'} \psi_C|^2\, dt\\
 &\lesssim \int_{t_0}^{t_1} \int_{D^C_t} \frac{1}{1+t} \frac{1}{(1+\log t)^{5/4}}
 (|\pa \sBo  Z^{J'} \psi_C|^2 + |\pa Z^{J'} \psi_C|^2) \, dt\\
 &\lesssim \epsilon_C^2,
 \label{}
\end{align}
recalling the definition $\sBo = s\pa_u$ and then arguing as above to bound
this quantity by the energy in the central region.

Finally, to deal with the first term in \eqref{cmon000} we use
the identity \eqref{Zintonull}
to bound
\begin{equation}
 \frac{1}{(1+v)^2} |Z^{J'} \Omega^2 \psi_C|
 \lesssim \frac{1}{1+v} \sum_{|J''| \leq |J'|+1}
\left(
 |\pa_v Z^{J''} \psi_C| + |\nas Z^{J''}\psi_C| + \frac{(1+|u|)}{(1+v)}|\pa_u Z^{J''}\psi_C|\right),
 \label{maincmonpt0003}
\end{equation}
and the contribution from these terms into \eqref{paucontrib} are
easily handled.

Combining the above estimates we have 
\begin{multline}
 \int_{t_0}^{t_1} \int_{D^C_t} (1+t)(1+\log t)^{3/2}(1+\log\log t)^{\alpha}|\pa_u Z^J \evmB \psi_C|^2\, dt
 \\
 \lesssim
 \epsilon_C^2 +\sum_{|J'| \leq |J|-1}
 \int_{t_0}^{t_1} \int_{D^C_t} (1+t)(1+\log t)^{3/2}(1+\log\log t)^{\alpha}|\pa_u Z^{J'} \evmB \psi_C|^2\, dt
 \\
 +\sum_{|J'| \leq |J|}
 \int_{t_0}^{t_1} \int_{D^C_t} \frac{1}{1+t} (1+\log t)^{3/2}(1+\log\log t)^{\alpha}|Z^{J'} \evmB \psi_C|^2\, dt.
 \label{cmon1}
\end{multline}
For the last term here, we bound
\begin{multline}
 \int_{t_0}^{t_1} \int_{D^C_t} \frac{1}{1+t} (1+\log t)^{3/2}(1+\log\log t)^{\alpha}|Z^{J'} \evmB \psi_C|^2\, dt\\
 \lesssim
 \int_{t_0}^{t_1} \int_{D^C_t} \left(\frac{1}{(1+t)^2} (1+\log t)^{3/2}(1+\log\log t)^{\alpha}\right)
 (1+v) |Z^{J'}\evmB \psi_C|^2\, dt
 \lesssim \epsilon_C^2,
 \label{}
\end{multline}
as in \eqref{cmon00}.
The result now follows from this bound, \eqref{cmon1}, and induction.

\end{proof}

We now control the boundary term from \eqref{parZJYL0}.
\begin{lemma}
	\label{UpsilonLplusbdy}
	Under the hypotheses of Proposition \ref{lefttocenterprop},
	for $|J| \leq N_L \leq N_C - 4$, provided $\epsilon_0$ is taken
	sufficiently small, we have
	\begin{equation}
		\int_{t_0}^{t_1}
	  \int_{\Gamma^R_t} (1+t)(1+\log t)(1+\log\log t)^{\alpha} |Z^J Y_L^+\psi_C|^2\, dS dt
		\lesssim \epsilon_C^2
	 \label{UpsilonLplusbdybd}
	\end{equation}
\end{lemma}
\begin{proof}
 We start by using Lemma \ref{minkdecomplemma} to convert the vector fields
 $Z$ into tangential vector fields $Z_T$ at the right shock, which gives
  \begin{multline}
  |Z^J Y_L^+\psi_C|
	\leq \sum_{|J'| \leq |J|}
	|Z_T^{J'} Y_L^+\psi_C|
	+ C(M)\sum_{|J'| \leq |J|-1} (1+|u|)|Z^{J'} n Y_L^+ \psi_C| + |Z^{J'} Y_L^+ \psi_C|
	\\
    + C(M)|B^{R}|_{J,\mB} \sum_{|K| \leq |J|/2+1}
	 (1+|u|)|Z^{K} n Y_L^+ \psi_C| + |Z^{K} Y_L^+ \psi_C|,
  \label{eq:Ya}
 \end{multline}
 where we used \eqref{ZtoZmBbd0} to bound
 the norm of $B^R$
 Since the vector fields $Z_T^J$ are tangent to $\Gamma^R$, at $\Gamma^R$
 by the boundary condition \eqref{introbcR} we have
 \begin{equation}
  |Z_T^J Y_L^+\psi_C|
	=|Z_T^J Y_R^-\psi_C|
	\lesssim |Z_T^J Y_R^+ \psi_R| + |Z_T^J G|.
  \label{}
 \end{equation}
 We now use Lemma \ref{minkdecomplemma} again to convert the tangential fields
 $Z_T^I$ into the usual fields $Z^I$ and bound the first term here by
 \begin{multline}
  |Z_T^J Y_R^+ \psi_R|
	\leq
	\sum_{|J'| \leq |J|}
	|Z^J Y_R^+ \psi_R|
	+ C(M) \sum_{|J'| \leq |J|-1}
	\left((1+|u|) |n Z^{J'} Y_R^+\psi_R| + |Z^{J'} Y_R^+\psi_R|\right)
	\\
	+
	C(M)|B^R|_{J, m}\sum_{|K| \leq |J|/2+1}	\left((1+|u|) |n Z^{J'} Y_R^+\psi_R| + |Z^{J'} Y_R^+\psi_R|\right)
  \label{anothertransfer}
 \end{multline}
 Thanks to the large weights appearing in the bounds \eqref{pwright} and \eqref{rightbdytrivialbds}
 for the potential in the rightmost region,
 it is straightforward to deal with these terms using arguments very similar
 but considerably simpler than ones we have encountered many times by now,
 and the result is that
 \begin{equation}
  \int_{t_0}^{t_1} \int_{\Gamma^R_t} |X^\ell| |Z_T^J Y_R^+ \psi_R|^2\,dS dt \lesssim \epsilon_R^2 \leq \epsilon_C^4,
  \label{}
 \end{equation}
 by \eqref{epsparameters}.
 See also Lemma \ref{UpsilonLplusbd} for a nearly identical but
 slightly more delicate estimate. 
The contribution from the nonlinear error terms $G^L$ can be handled
    as in the proof of \eqref{errorbdytermleft}.

 We now handle the terms involving $n Y_{L}^+ \psi_C$ \eqref{eq:Ya}.
 We claim that
 \begin{equation}
  \int_{t_0}^{t_1} \int_{\Gamma^R_t} (1+ t)(1 + \log t)(1+ \log \log t)^\alpha
	\left((1+|u|) |n Z^{J'} Y_{L}^+\psi_C|\right)^2\, dS dt \lesssim \epsilon_C^2.
  \label{mainUpsilonRL}
 \end{equation}
 In the same way that \eqref{paucontrib} implied the previous result, this bound
 implies \eqref{UpsilonLplusbdybd}.
 To prove this, we use
 \eqref{cmon000} again. As in the proof of the previous result, we will just
 consider the terms on the first line of \eqref{cmon000}, the remaining terms
 being simpler.

 For the third term on the first line of \eqref{cmon000}, using that $|u| \sim s^{1/2}$
 along $\Gamma^R$, we have
 \begin{align}
  \int_{t_0}^{t_1} \int_{\Gamma^R_t} (1+ t)(1 + \log t)(1+ \log \log t)^\alpha
	&\left( \frac{(1+|u|) }{(1+v)(1+s)} |\pa Z^{J'} \psi_C|\right)^2\, dS dt\\
	&\lesssim
	\int_{t_0}^{t_1} \int_{\Gamma^R_t} \frac{1}{1+t} 	(1+ \log \log t)^\alpha |\pa Z^{J'} \psi_C|^2\, dS dt\\
	&\lesssim\int_{t_0}^{t_1} \int_{\Gamma^R_t} \frac{(1+s)^{1/2}}{1+v}
	 |\pa Z^{J'} \psi_C|^2\, dS dt\\
	 &\lesssim \epsilon_C^2,
  \label{}
 \end{align}
 in light of \eqref{rightbdytrivialbds}
 and after using the bound \eqref{ZtoZmBbd0}
 to relate norms involving the $Z$ fields to those involving
the $Z_{\mB}$ fields.
 Similarly, for the second term in \eqref{cmon000} we have
\begin{align}
 \int_{t_0}^{t_1} \int_{\Gamma^R_t} (1+ t)&(1 + \log t)(1+ \log \log t)^\alpha
 \left( \frac{(1+|u|) }{(1+v)(1+s)^{1/2}} |\pa^2 Z^{J'} \psi_C|\right)^2\, dS dt\\
 &\lesssim
 \int_{t_0}^{t_1} \int_{\Gamma^R_t} \frac{1}{1+t} (1+ \log t)
 (1+ \log \log t)^\alpha |\pa^2 Z^{J'} \psi_C|^2\, dS dt\\
 &\lesssim \int_{t_0}^{t_1} \int_{\Gamma^R_t} \frac{1}{1+t} \frac{1}{1+\log t}
 (1+ \log \log t)^\alpha
 (|\pa \sBo  Z^{J'} \psi_C|^2 + |\pa Z^{J'} \psi_C|^2) \, dt\\
 &\lesssim \epsilon_C^2.
 \label{}
\end{align}
Finally, using \eqref{maincmonpt0003} again it is straightforward to handle
the contribution from the first term in \eqref{cmon000} into \eqref{mainUpsilonRL}.
It remains to handle the lower-order terms (the last term on the
    first line and those on the second line) from \eqref{eq:Ya}).
    The terms $Z^{J'}Y_L^+\psi_C$ for $|J'|\leq |J|$ can be handled using induction.
    As for the lower-order term involving $nY_L^+$, we just
    use the equation \eqref{cmon000} again to express
$nY_L^+ \psi_C$ in terms of nonlinear terms and linear terms
involving $\sDelta$ .

\end{proof}

\subsection{Proof of Proposition \ref{centertorightprop}}
\label{pfcentertorightprop}

We use a similar strategy as in the previous section. The bounds
in this section
are simpler because in this section we only need to consider
the weight
$X^\ell = 1+v$ which is smaller than the weight we needed to consider
in the previous section and there is more room
in all of our estimates. On the other hand, the estimates are somewhat more
cumbersome because we have worse control over the solution at top order
than we did in the central region (recall the definition of the energies
\eqref{ECdef}).

The bound \eqref{centertorightpropbd} is a consequence of the upcoming Lemmas \ref{linearcommjumpR},
\ref{transferlemR} and \ref{UpsilonRplusbd}, as follows.
Define the quantities
\begin{equation}
 \Upsilon_{I, R}^+(t_1) = \int_{t_0}^{t_1} \int_{\Gamma^R_t}
 (1+v) | Z_{\mB, T}^I Y_L^+ \psi_R|^2\, dS dt,
 \label{UpsilondefR}
\end{equation}
where $Y_R^+$ is as in \eqref{YRdef2}.

By Lemmas \ref{linearcommjumpR} and \ref{transferlemR},
for $|I| \leq N_C$, we have the bound
\begin{equation}
    B^C_I(t_1) \lesssim
  \Upsilon_{I, R}^+(t_1)
  + \sum_{|K| \leq |I|-1} B^C_K(t_1)
  + (c_0(\epsilon_0) + \epsilon_C^2)\epsilon_C^2 + c_0(\epsilon_0),
 \label{}
\end{equation}
where $c_0$ is a continuous function with $c_0(0) = 0$,
so taking $\epsilon_0$ sufficiently small and using induction we find
\begin{equation}
 B^C_I(t_1) \lesssim \Upsilon_{I, R}^+(t_1) + \epsilon_C^3,
 \label{}
\end{equation}
and the result follows after using Lemma \ref{UpsilonRplusbd} to control
the quantities $\Upsilon_{I, R}^+$. \hfill \qedsymbol

 In the rest of this section, we prove the cited Lemmas
 \ref{linearcommjumpR}-\ref{UpsilonRplusbd}.

\subsubsection{Supporting lemmas for the proof of Proposition \ref{centertorightprop}}
As in the last section, we start by recording a product estimate that we will
use to handle the nonlinear terms we encounter. Fortunately this bound
is less delicate than Lemma \ref{generalproduct}.

\begin{lemma}
	\label{generalquadraticright}
 Let $Q(\pa\psi, \pa \psi) = Q^{\alpha\beta}\pa_\alpha \psi\pa_\beta\psi$
 be a quadratic nonlinearity where the coefficients $Q^{\alpha\beta}(\omega)$
 are smooth functions satisfying the symbol condition \eqref{strongsymbol}.
 Under the hypotheses of Proposition \ref{centertorightprop}, we have
 \begin{equation}
  \int_{t_0}^{t_1} \int_{\Gamma^R_t} (1+v) |Z_{\mB}^I \left( (1+v)^{-1}Q(\pa \psi_C,
	\pa \psi_C)\right)|^2\, dS dt
	\lesssim
    \epsilon_C^4.
  \label{}
 \end{equation}
\end{lemma}
\begin{proof}
We follow the same strategy as in the proof of Lemma \ref{generalproduct}. We
first note that by the decay estimates \eqref{pwcentral} we have
$|\pa \psi^K_C| \lesssim \epsilon_C (1+s)^{-1/2}$ for $|K| \leq N_C/2+1$,
and since
both of the multipliers $X_{C}$ and $X_{T}$ in the central region 
satisfy the bounds $|X_{\mB}^n| \gtrsim (1+s)^{-1/2}$ along the shocks we have in particular
\begin{equation}
 |\pa \psi^K_C| \lesssim \epsilon_C |X^n|, \qquad |K| \leq N_C/2+1.
 \label{}
\end{equation}
As a result, using $|\pa \psi^K_C|\lesssim \epsilon_C(1+s)^{-1/2}$ again
we find
\begin{equation}
 \frac{1}{1+v} |\pa \psi^K_C|^2 |n \psi^J_C|^2
 \lesssim \epsilon_C^2 \frac{1}{(1+v)(1+s)^{1/2}}|X^n| |n\psi^J_C|^2.
 \label{}
\end{equation}
Bounding $|\pa q| \lesssim |\evmB q| + |n q| + |\nas q|$ and using the simpler
estimates
\begin{align}
 \frac{1}{1+v} |\pa \psi_C^K|^2 |\evmB \psi_C^J|^2 \lesssim
 \epsilon_C^2 |\evmB \psi_C^J|^2,\qquad |K| \leq N_C/2+1\\
  \frac{1}{1+v} |\pa \psi_C^K|^2 |\nas \psi_C^J|^2 \lesssim
  \epsilon_C^2 \frac{1}{(1+s)^{1/2}}|\nas \psi_C^J|^2,\qquad |K| \leq N_C/2+1,
 \label{}
\end{align}
we therefore have the needed bound,
\begin{multline}
 \sum_{|K| \leq N_C/2+1} \sum_{|J| \leq N_C} \int_{t_0}^{t_1}\int_{\Gamma^R_t}
 \frac{1}{1+v} |\pa \psi_C^K|^2 |\pa \psi_C^I|^2\, dS dt
 \\
 \lesssim
 \epsilon_C^2 \sum_{|J| \leq N_C} \int_{t_0}^{t_1} \int_{\Gamma^R_t}\left(
 (1+v) |\evmB \psi_C^J|^2 + \frac{|X^n|}{(1+v)(1+s)^{1/2}} |n \psi_C^J|^2
 + \frac{1}{(1+s)^{1/2}} |\nas \psi_C^J|^2\right)\,dS dt
 \\
 \lesssim
 \epsilon_C^{2} \sum_{|J| \leq N_C} \int_{t_0}^{t_1} \int_{\Gamma^R_t}
 (1+v) |\evmB \psi_C^J|^2\,dS dt + \epsilon_C^4
 \lesssim \epsilon_C^4.
 \label{}
\end{multline}
\end{proof}

We now prove the analogue of Lemma \ref{linearcommjumpL}, where
we commute $\evmB$ with $Z_{\mB}^I$.
\begin{lemma}
	\label{linearcommjumpR}
	Under the hypotheses of Proposition \ref{centertorightprop}, we have
	\begin{equation}
        B^C_I(t_1)
	 \lesssim
	 \int_{t_0}^{t_1} \int_{\Gamma^R_t} (1+v) |Z^I_{\mB} Y_R^-\psi_C|^2\,dS dt
	 + \sum_{|K| \leq |I|-1} B^C_K(t_1)
    + (c_0(\epsilon_0)
     + \epsilon_C^2)\epsilon_C^2 + c_0(\epsilon_0)
	 \label{linearcommjumpRbd}
	\end{equation}
	for a continuous function $c_0$ with $c_0(0) = 0$.
\end{lemma}
\begin{proof}
 Recalling the definition of $Y_{R}^-$ from \eqref{YRdef1}, along $\Gamma^R$ we
 have
 \begin{align}
  |\evmB Z^I_{\mB} \psi_C|
	&\lesssim
	|Z_{\mB}^I \evmB \psi_C| + |[Z_{\mB}^I, \evmB] \psi_C|\\
	&\lesssim
	|Z_{\mB}^I Y_L^-\psi_C| + |Z^I_{\mB} \left((1+v)^{-1}Q(\pa \psi_C, \pa \psi_C)\right)|
	+ |[Z_{\mB}^I, \evmB] \psi_C|\\
	&\lesssim
	|Z_{\mB}^I Y_L^-\psi_C| + |Z^I_{\mB} \left((1+v)^{-1}Q(\pa \psi_C, \pa \psi_C)\right)|
	\\
	&\qquad+
	\sum_{|J| \leq |I|-1}
	|\evmB Z^J_{\mB} \psi_C|
	+ \frac{1}{(1+v)(1+s)} |\pa Z^J_{\mB} \psi_C|.
  \label{}
 \end{align}
 By Lemma \ref{generalquadraticright} and the definition of $B^C_K$,
 the contribution from the terms on the first line here are bounded by the
 right-hand side of \eqref{linearcommjumpRbd}, and the terms on the second
 line are easily handled after bounding
 \begin{multline}
  \int_{t_0}^{t_1} \int_{\Gamma^R_t} (1+v) \left( \frac{1}{(1+v)(1+s)} |\pa Z^J_{\mB} \psi_C|\right)^2\, dSdt
	\\
	\lesssim
  \int_{t_0}^{t_1} \int_{\Gamma^R_t} \frac{1}{1+v} \frac{1}{(1+s)^2}
	|\pa Z^J_{\mB} \psi_C|^2\, dSdt
	\lesssim c_0(\epsilon_0) \epsilon_C^2,
  \label{}
 \end{multline}
 recalling the bounds from Lemma \ref{bdsforpsiCalongshock} for $\psi_C$ along
 the right shock.
\end{proof}

We now prove the analogue of Lemma \ref{transferlemL}.
\begin{lemma}
	\label{transferlemR}
 Under the hypotheses of Proposition \ref{centertorightprop}, we have
 \begin{equation}
  \int_{t_0}^{t_1} \int_{\Gamma^R_t}
	(1+v) |Z_{\mB}^I Y_R^-\psi_C|^2\, dS dt
	\lesssim \Upsilon_R^+(t_1) +
    \sum_{|K| \leq |I|-1} B_K^C(t_1)
    + (c_0(\epsilon_0) + \epsilon_C^2)\epsilon_C^2 + c_0(\epsilon_0)
  \label{transferlemRbd}
 \end{equation}
 for a continuous function $c_0$ with $c_0(0) = 0$.
\end{lemma}
\begin{proof}
 By Lemma \ref{mBtangentialdifference}, we have the bound
 \begin{multline}
|Z_{\mB}^I Y_R^-\psi_C|
\lesssim
|Z_{\mB, T}^I Y_R^- \psi_C|
+ C(M)(1+s)\sum_{|J| \leq |I|-1} |Z_{\mB}^J n Y_R^- \psi_C|\\
+ C(M)(1+s)(1 + (1+s)^{-1/2} |B^R|_{I, \mB})\sum_{|K| \leq |I|/2+1}
|Z^K_{\mB} nY_R^-\psi_C|.
  \label{}
 \end{multline}
 Since the fields $Z_{\mB, T}$ are tangent to the shock, by
 the boundary condition  \eqref{introbcR} we have
 $Z_{\mB, T}^I Y_{R}^- \psi_C = Z_{\mB, T}^I Y_{R}^+ \psi_R
 + G$ at the shock, so to conclude it is enough to prove that
 for $|I| \leq N_C$, we have the following estimates,
 \begin{equation}
  \int_{t_0}^{t_1} \int_{\Gamma^R_t} (1+v)(1+s)^2
	|Z^J_{\mB} n Y_R^-\psi_C|^2\, dS dt
	\lesssim 
    \sum_{|K| \leq |I|-1} B^C_K(t_1)
	(c_0(\epsilon_0) + \epsilon_C^2)\epsilon_C^2
	 \qquad |J| \leq |I|-1,
  \label{bdytermright1}
 \end{equation}
 \begin{multline}
  \int_{t_0}^{t_1} \int_{\Gamma^R_t}
	(1+v) |B^R|_{I, \mB}^2 (1+s) | Z^K_{\mB} n Y_R^-\psi_C|^2\, dS dt
	\\
	\lesssim 
    \sum_{|K| \leq |I|-1} B^C_K(t_1)+
    (c_0(\epsilon_0) + \epsilon_C^2)\epsilon_C^2 \qquad |K| \leq |I|/2+1,
  \label{bdytermright2}
 \end{multline}
 \begin{equation}
  \int_{t_0}^{t_1} \int_{\Gamma^R_t}
	(1+v) |Z_{\mB, T}^I G|\, dS dt
	\lesssim
    \sum_{|K| \leq |I|-1}  B^C_K(t_1)+ 
    (c_0(\epsilon_0) + \epsilon_C^2)\epsilon_C^2 + c_0(\epsilon_0) \qquad |K| \leq |I|/2+1.
  \label{}
 \end{equation}
 As in the proof of Lemma \ref{transferlemL} we just prove the first two
 bounds here, the bound for the remainder term $G$ being similar.

 To prove \eqref{bdytermright1}, we recall the definition of $Y_R^-$ from
 \eqref{YR0} and use that
 $(1+s)|\pa q|\lesssim \sum_{Z_{\mB} \in \mathbb{Z}_{\mB}} |Z_{\mB} q|$ to gain
 an extra power of $s$ and bound
 \begin{equation}
  (1+s) |Z^J_{\mB} n Y_R^- \psi_C|
	\lesssim
	(1+s) |Z^J_{\mB} n \evmB \psi_C|
	+  \sum_{|J'| \leq |J|+1}
	|Z_{\mB}^{J'} \left( (1+v)^{-1} Q(\pa \psi_C, \pa \psi_C)\right)|.
  \label{}
 \end{equation}
 By \eqref{generalquadraticright}, the term contributed from the second
 term here into \eqref{bdytermright1} satisfies the needed bound. For the
 first term here, we use the equation \eqref{waveint2} to bound
 \begin{multline}
  (1+s)|Z^J_{\mB} n \evmB \psi_C|
	\lesssim (1+s) |Z_{\mB}^J \sDelta \psi_C|
	+ (1+s)|Z_{\mB}^J \left( \pa_\mu (\tfrac{u}{vs} a^{\mu\nu}\pa_\nu \psi_C)\right)|
	\\
	+\sum_{|J'| \leq |J|+1}
	|Z_{\mB}^{J'} \left((1+v)^{-1} Q(\pa \psi_C, \pa \psi_C)\right)|
	+ (1+s)|Z^J F'|.
  \label{useeqnbdy0R}
 \end{multline}
 The quadratic term here can be handled by Lemma \ref{generalquadraticright},
 the first term here can be handled by writing $\sDelta = \frac{1}{r^2} \Omega^2$
 and making straightforward estimates, and the last term as usual is easier to handle
 then either of these terms. We will therefore just prove the bound for
 the term involving $a$.
 Unlike in \eqref{cmon000} where we did not need to worry about the
 structure of this term, here we will need to use the fact
 that that term verifies the null condition \eqref{intronullcondn0}.
 This is because this term appears linearly here and we have
 very weak control over the solution at top order along
 the shocks, whereas in \eqref{cmon000}
 we could afford to treat this term as an error term because we only needed
 to consider lower-order derivatives and because we took $\epsilon_C \leq
 \epsilon_L^2$.

 Noting that $a^{\mu\nu}\pa_\mu\pa_\nu \psi_C = a^{\mu\nu}\opa_\mu\pa_\nu \psi_C$,
 we bound
 \begin{equation}
  (1+s) |Z_{\mB}^I  \left( \tfrac{u}{vs} a^{\mu\nu}\pa_\mu \pa_\nu \psi_C)\right)|
	\lesssim
	\frac{1+s}{1+v}\sum_{|J| \leq |I|+1} |\opa Z_{\mB}^I \psi_C|
	 + \frac{1+s}{(1+v)^2}
	\sum_{|J| \leq |I|+1} |\pa Z_{\mB}^I \psi_C|,
  \label{}
 \end{equation}
 after using \eqref{comm1} to commute our fields with
 $\opa\in\{\pa_v, \nas\}$. 
 As for the term where the derivative falls on the factor $u/vs$,
 thanks to the null condition \eqref{intronullcondn0}, we can write
 \begin{align}
  a^{\mu\nu}\left(\pa_\mu \frac{u}{vs}\right) \pa_\nu \psi_C &=
	a_1^{\mu\nu} \left(\opa_\mu \frac{u}{vs}\right) \pa_\nu \psi_C
	+ a_2^{\mu\nu} \left(\pa_\mu \frac{u}{vs}\right) \opa_\nu \psi_C\\
	&= \frac{1+|u|}{(1+v)^2(1+s)}b_1^{\mu} \pa_\mu \psi_C
	+ \frac{1}{(1+v)(1+s)} b_2^{\mu} \opa_\mu \psi_C,
  \label{}
 \end{align}
 where the coefficients above satisfy the symbol condition
 \eqref{strongsymbol}. Since $Z_{\mB}^J |u| \lesssim (1+s)$
 for any $J$
 it follows from this observation and the fact that the $a^{\mu\nu}$
 satisfy \eqref{strongsymbol} that
 \begin{equation}
  (1+s) |Z_{\mB}^I  \left(   a^{\mu\nu}\pa_\mu(\tfrac{u}{vs})\pa_\nu \psi_C)\right)|
	\lesssim
	\frac{1+s}{(1+v)^2} \sum_{|I'| \leq |I|}
	|\pa Z_{\mB}^{I'}\psi_C|
	+ \frac{1}{1+v} \sum_{|I'| \leq |I|}
	|\opa Z_{\mB}^{I'} \psi_C|.
  \label{}
 \end{equation}
 We note at this point that if we had not made use of the structure of $a$,
 for the last term we would only have $(1+v)^{-1}|\pa Z_{\mB}^{I'}\psi_C|$,
 and the contribution from this term into \eqref{bdytermright1} would be
 too large for us to handle when $|I| = N_C$ since at top-order we
 can only hope to control
 $\int_{t_0}^{t_1} \int_{\Gamma^R_t} (1+v)^{-1}(1+s)^{-1} |\pa Z_{\mB}^{I'}\psi_C|^2$.
Combining the last two bounds and using that $a^{\mu\nu}$ satisfy the
 symbol condition \eqref{strongsymbol}, we find
 \begin{multline}
  \int_{t_0}^{t_1} \int_{\Gamma^R_t} (1+v)(1+s)^{2}
	 |Z_{\mB}^I  \pa_\mu\left(\tfrac{u}{vs}a^{\mu\nu}\pa_\nu \psi_C)\right)|^2\,dS
	 dt
	 \\
	 \lesssim
  \sum_{|J| \leq |I|+1}
	\int_{t_0}^{t_1} \int_{\Gamma^R_t}
	\left(\frac{(1+s)^2}{1+v} (|\pa_v Z_{\mB}^J \psi_C|^2
	+ |\nas Z_{\mB}^J\psi_C|^2)
	+ \frac{(1+s)^2}{(1+v)^3} |\pa Z_{\mB}^J\psi_C|^2\right)\, dS dt\\
	\lesssim c_0(\epsilon_0) \epsilon_C^2,
  \label{}
 \end{multline}
 which is of the correct form for \eqref{bdytermright1}
 for $|I| \leq N_C$.

 We now move on to proving \eqref{bdytermright2}. For this we bound
 \begin{multline}
  \int_{t_0}^{t_1} \int_{\Gamma^R_t}
	(1+v) |B^R|_{I, \mB}^2 (1+s) | Z^K_{\mB} n Y_R^-\psi_C|^2\, dS dt\\
	\lesssim
	\left(\sup_{t_0 \leq t \leq t_1} \int_{\Gamma^R_t} \frac{|B^R|_{I, \mB}^2}{1+s}\, dS\right)
	\left(\int_{t_0}^{t_1} \sup_{\Gamma^R_t} (1+v)(1+s)^{2}| Z^K_{\mB} n Y_R^-\psi_C|^2\right),
  \label{highnormbR1}
 \end{multline}
 and using Sobolev embedding on $\Gamma^R_t$ to handle the second factor as in
 the proof of Lemma \ref{transferlemL} and using the above bound
 \eqref{bdytermright1} again and the bound \eqref{rightgeom}
 for $G^R_{N_C}$ (defined in\eqref{GRnorm}), we get the result.
 Note that here we are able to close the estimate even though we only
 control a relatively weak norm of $B^R$ at top order in light of
the strong decay estimates we have for $n Y_R^-\psi_C$ at low order.
\end{proof}

We now prove the analogue of Proposition \ref{UpsilonLplusbd}. This is
fortunately much simpler than that result since we have extremely good
control over the solution on the spacelike side of the right shock.
\begin{lemma}
	\label{UpsilonRplusbd}
	With $\Upsilon_{I, R}^+$ defined as in \eqref{UpsilondefR}, under the
  hypotheses of Proposition \ref{centertorightprop}, we have
  \begin{equation}
      \sum_{|I|\leq N_C} \Upsilon_{I, R}^+(t_1) \lesssim \epsilon_R^2.
   \label{upsilonLplusbdbd}
  \end{equation}
\end{lemma}
\begin{proof}
	We will need to replace the vector fields $Z_{\mB}$ with the Minkowskian
	fields $Z$. For this, we start with the observation that
	the fields $Z_{\mB}$
	and the fields $Z$ satisfy
	\begin{equation}
	 Z_{\mB} = \sum_{Z \in \mathcal{Z}_m} c_{Z_{\mB}}^Z \frac{1+s}{1+|u|} Z + d_{Z_{\mB}}^Z Z
	 \label{}
	\end{equation}
	where the coefficients satisfy the symbol condition \eqref{strongsymbol}.
  This follows easily from the well-known identity \eqref{nullintoZ}
	which expresses $\pa_u, \pa_v$ in terms of the Minkowskian fields. Repeatedly
	applying this formula and using basic properties of the fields $Z$
	gives the bound
	\begin{equation}
	 |Z_{\mB}^I q| \lesssim (1+s)^{|I|/2} \sum_{|J| \leq |I|} |Z^J q|,
     \qquad \text{ along } \Gamma^R.
	 \label{ZmBtoZweight}
	\end{equation}

    We now prove \eqref{upsilonLplusbdbd}.
  We start by using Lemma \ref{mBtangentialdifference} to convert
	the tangential fields to the fields $Z_{\mB}$ at
    the right shock, which gives 
	\begin{align}
	 |Z_{\mB, T}^I Y^+_R \psi_R|
	 &\leq |Z_{\mB}^I Y^+_R\psi_R| +
	 C(M)(1+s)
	 \sum_{|J| \leq |I|-1}|Z_{\mB}^J n Y^+_R \psi_R|
	 \\
&\qquad+ C(M)(1+s) \left(1 + (1+s)^{-1/2}|B^R|_{I, \mZB} \right)
\sum_{|L| \leq |I|/2+2} |Z_{\mB}^L n Y^+_R \psi_R|\\
&
\leq 
C (1+s)^{|I|/2}\sum_{|J| \leq |I|}  |Z^J Y^+_R\psi_R| +
C(M) (1+s)^{|I|/2+1} \sum_{|J| \leq |I|-1}|Z^J nY^+_R \psi_R|\\
&\qquad+ C(M) (1+s)^{|I|/2+1} \left(1 + (1+s)^{-1/2}|B^R|_{I, \mZB} \right)
\sum_{|L| \leq |I|/2+2} |Z^Ln Y^+_R \psi_R|,
	 \label{upsilonpw}
	\end{align}
	where we used \eqref{ZmBtoZweight} in the second step.
	We now handle these terms in the usual way. Recalling the definition
	of $Y^+_R$ from \eqref{YRdef2}, we first bound
    \begin{equation}
     |Z^J Y^+_R\psi_R|
         \lesssim |\evm Z^J\psi_R| +|[\evm, Z^J] \psi_-^R| +  |Z^J\left((1+v)^{-1} Q(\pa \psi_R, \pa \psi_R)\right)|
      \label{ZJYplusR}
    \end{equation}
    Inserting this into the right-hand side of \eqref{upsilonpw}, for $|I| \leq N_C$, the contribution from
    the first term here into
    $\Upsilon_{I, R}^+$ is bounded by
    \begin{equation}
      \label{}
      \sum_{|J| \leq |I|}\int_{t_0}^{t_1} \int_{\Gamma^R_t} (1+v)(1+s)^{|I|}|\evm Z^J\psi_R|^2\,dS dt
      \lesssim
      \sum_{|J| \leq |I|} \int_{t_0}^{t_1} \int_{\Gamma^R_t} r(\log r)^{\nu} |\evm Z^J\psi_R|^2\,dS dt
      \lesssim \epsilon_R^2,
    \end{equation}
    where we used that $r \sim v$ along $\Gamma^R$,
    the bound \eqref{rightboot} for the boundary term in the definition
    of the energy $\mathcal{E}_R$ in \eqref{ERdef}, and the fact that by our choice of parameters
    \eqref{parameters}, $|I| \leq N_C \leq \nu$. 

    To handle the contribution from the nonlinear term in \eqref{ZJYplusR} into 
    $\Upsilon_{I, R}^+$, we bound
    \begin{multline}
	 \int_{t_0}^{t_1} \int_{\Gamma^R_t} (1+v)(1+s)^{|I|+2}
	 |Z^J \left( (1+v)^{-1} Q(\pa \psi_R, \pa \psi_R)\right)|^2 \, dS dt
	 \\
	 \lesssim \sum_{|J'| \leq |J|}
	 \sum_{|K| \leq |J|/2+1}\int_{t_0}^{t_1} \int_{\Gamma^R_t}
	 \frac{1}{1+v} (1+s)^{N_C+2} |\pa Z^{J'} \psi_R|^2 |\pa Z^K\psi_R|^2\,dS dt
	 \\
	 \lesssim \sum_{|J'| \leq |J|}
	 \int_{t_0}^{t_1} \int_{\Gamma^R_t}\frac{1}{1+v} (1+s)^{N_C+2-\mu}
	 |\pa Z^{J'} \psi_R|^2 \,dS dt
	 \lesssim \epsilon_R^2,
	 \label{lastQbd}
	\end{multline}
    where we used that $N_C+2-\mu \leq \mu-1/2$ by our choice of $\mu$ 
    in \eqref{parameters}. 
    The contribution from $[Z^J, \evm]$ from \eqref{ZJYplusR} into our estimates
    is straightforward to handle using \eqref{ellZcomm} so we skip it.

    It remains only to handle the terms involving $nY_R+$ from \eqref{upsilonpw}.
    We just show how to handle the term on the first line of \eqref{upsilonpw}
    since the term on the second line can be handled using the same idea.
    For both of these terms, the idea is to write $nY^+_R \psi_R = n\evm\psi_R 
    +  (1+v)^{-1} n(Q(\pa \psi_R, \pa \psi_R))$, and to handle the first term
    by using the equation \eqref{waveext2} for $\psi_R$.
    This gives
    \begin{align}
	\int_{t_0}^{t_1} \int_{\Gamma^R_t}
	(1+v)(1+s)^{|I|+2}	& | Z^J n \ev \psi_R|^2\,dS dt
	\\
	&\quad\lesssim
	\int_{t_0}^{t_1} \int_{\Gamma^R_t}
	(1+v)(1+s)^{|I|+2} |Z^J \sDelta \psi_R|^2\, dS dt\\
	&\quad +\int_{t_0}^{t_1} \int_{\Gamma^R_t} (1+v)(1+s)^{|I|+2}
	|Z^J \left( (1+v)^{-1} Q(\pa \psi_R, \pa \psi_R)\right)|^2 \, dS dt\\
	 &\quad+\int_{t_0}^{t_1} \int_{\Gamma^R_t} (1+v)(1+s)^{|I|+2}
	|Z^J F|^2 \, dS dt.
	 \label{}
	\end{align}
	As usual we skip the bounds for the last term. The second term here is bounded
	exactly as in \eqref{lastQbd}. The first term can be handled after writing
	$\sDelta  = \frac{1}{r} \Omega \nas $ and using straightforward estimates
	along with the bounds \eqref{rightrivialbdybds} along the spacelike side of the right shock.
    The nonlinear term contributed by using the above formula for $nY^+_R$
    can be handled exactly as in \eqref{lastQbd}.

\end{proof}

\section{The transport equation for the boundary-defining function}
\label{bdsforbdfsec}

In the last three sections, we showed that provided the shocks $\Gamma^L, \Gamma^R$
were close to the model shocks (in the sense that \eqref{logeom} holds, with
$K^R, K^L$ as in \eqref{KRdef}-\eqref{KLdef}), and provided that we have bounds
for high-order derivatives of the boundary-defining functions $B^L, B^R$
(namely, the bounds \eqref{leftgeom}-\eqref{rightgeom}),
we can improve the bounds from our bootstrap assumptions
\eqref{rightboot}-\eqref{leftboot} for the potentials $\psi_L, \psi_C, \psi_R$.
The goal of this section is to show that we can improve the bounds \eqref{logeom} and
\eqref{leftgeom}-\eqref{rightgeom} describing the positions of the shocks.
This is done in the upcoming Propositions \ref{leftshockbootstrap} and \ref{rightshockbootstrap}.

\begin{prop}[Improved estimates for the geometry of the left shock]
  \label{leftshockbootstrap}
  Under the hypotheses of Proposition \ref{bootstrapprop},
  there is a continuous function $c_0$ with $c_0(0) = 0$
  so that the function
  $B^L$ which defines the left shock satisfies the pointwise estimates	
  \begin{align}
	 	\left| \frac{B^L(t,x)}{s^{1/2}} -  1\right|
 		+ (1+s)^{1/2}\left|\pa_s B^L(t,x) - \frac{1}{2s} B^L(t,x)\right|
        &\leq \mathring{K}^L + c_0(\epsilon_0) \epsilon_L,
        \label{lowestbootstrapofBL}\\
		 \left| \frac{\Omega B^L(t,x)}{s^{1/2}}\right|
        &\leq \slashed{\mathring{K}}^L+ c_0(\epsilon_0) \epsilon_L ,
	 \label{lowestbootstrapofsBL}
	\end{align}
	along $\cup_{t_0 \leq t' \leq t_1} \Gamma^L_{t'}$,
    where $\mathring{K}^L, \slashed{\mathring{K}}^L$
    are the norms of the initial data defined in \eqref{KL0}-\eqref{sKL0},
    We also have the integrated estimates
    \begin{equation}
      \label{improvedbdleftshock}
         \sum_{|I| \leq N_L, |I|\geq 1} \sup_{t_0 \leq t \leq t_1}
         \int_{\Gamma^L_t} \frac{1}{1+s} |Z_T^I B^L|^2\, dS +
         \sum_{|I| \leq N_L/2+1} \sup_{t_0 \leq t \leq t_1}
         \sup_{\Gamma^L_t} \frac{1}{1+s} |Z_T^I B^L|^2
         \leq \mathring{G}^L_{N_L} +  c_0(\epsilon_0)\epsilon_L.
    \end{equation}
    In particular, if $\epsilon_0,\epsilon_1, \epsilon_2$ are taken sufficiently small, with
    $K^L$ defined as in \eqref{KLdef},
    $\slashed{K}^L$ defined as in \eqref{sKLdef},
    and $G^L$ defined as in
    \eqref{GLnorm}, we have the bounds
    \begin{equation}
      \label{}
      K^L(t_1) \leq \epsilon_1^{3/2},
      \qquad
      \slashed{K}^L(t_1) \leq \epsilon_2^{3/2}
      \qquad
      G^L(t_1) \leq M_0^L + \epsilon_L^2,
    \end{equation}
    with $M_0^L$ defined as in \eqref{initialdatabds}.
\end{prop}

The analogous result at the right shock is the following.
\begin{prop}[Improved estimates for the geometry of the right shock]
  \label{rightshockbootstrap}
  Under the hypotheses of Proposition \ref{bootstrapprop},
  there is a continuous function $c_0$ with $c_0(0) = 0$
  so that the function
  $B^R$ which defines the right shock satisfies the pointwise estimates	
  \begin{align}
	 	\left| \frac{B^R(t,x)}{s^{1/2}} +  1\right|
 		+ (1+s)^{1/2}\left|\pa_s B^R(t,x) - \frac{1}{2s} B^R(t,x)\right|
        &\leq \mathring{K}^R + c_0(\epsilon_0) \epsilon_C,\\
		 \left| \frac{\Omega B^R(t,x)}{s^{1/2}}\right|
        &\leq \slashed{\mathring{K}}^R + c_0(\epsilon_0) \epsilon_C ,
	 \label{lowestbootstrapofBR}
	\end{align}
	along $\cup_{t_0 \leq t' \leq t_1} \Gamma^R_{t'}$,
    where $\mathring{K}^R, \slashed{\mathring{K}}^R$
    are the norms of the initial data defined in \eqref{KR0}-\eqref{sKR0},
    We also have the integrated estimates
    \begin{equation}
      \label{improvedbdrightshock}
         \sum_{|I| \leq N_L-2, |I|\geq 1} \sup_{t_0 \leq t \leq t_1}
         \int_{\Gamma^R_t} \frac{1}{1+s} |Z_{\mB,T}^I B^R|^2\, dS +
         \sum_{|I| \leq N_R/2+1} \sup_{t_0 \leq t \leq t_1}
         \sup_{\Gamma^R_t} \frac{1}{1+s} |Z_{\mB, T}^I B^R|^2
         \leq \mathring{G}^R_{N_R} + c_0(\epsilon_0)\epsilon_C,
    \end{equation}
    as well as
    \begin{equation}
      \label{}
         \sum_{|I| \leq N_L, |I|\geq 1} \sup_{t_0 \leq t \leq t_1}
         \int_{\Gamma^R_t} \frac{1}{(1+s)^2} |Z_{\mB,T}^I B^R|^2\, dS
         \leq \mathring{G}^R_{N_R} + c_0(\epsilon_0)\epsilon_C.
    \end{equation}
    
    In particular, if $\epsilon_0,\epsilon_1, \epsilon_2$ are taken sufficiently small, with
    $K^R$ defined as in \eqref{KRdef},
    $\slashed{K}^R$ defined as in \eqref{sKRdef},
    and $G^R$ defined as in
    \eqref{GRnorm}, we have the bounds
    \begin{equation}
      \label{}
      K^R(t_1) \leq \epsilon_1^{3/2},
      \qquad
      \slashed{K}^R(t_1) \leq \epsilon_2^{3/2}
      \qquad
      G^R(t_1) \leq M_0^R + \epsilon_C^2,
    \end{equation}
    with $M_0^R$ defined as in \eqref{initialdatabds}.
\end{prop}
The above results rely on the fact that
$B^A$ satisfy the following transport equation, derived in
Lemma \ref{derivofbetaeqn},
\begin{equation}
    \pa_s B^A - \frac{1}{2s} B^A = -\frac{1}{2}[\pa_u \psi] + s^{1/2}F_A,
    \qquad \text{ at } \Gamma^A
 \label{localBAeqntr}
\end{equation}
Here,
$[q]$ denotes the jump in $q$ across
$\Gamma^A$, and the quantities $F_A$, which consist
of nonlinear error terms, are given in Lemma
\ref{derivofbetaeqn} (see Remark \ref{evoluteqnrmk}).

For the upcoming calculations, it will be convenient to work in
terms of a rescaling of $B^A$ restricted to the shock. 
    Specifically, for $(t, x) \in \Gamma^A$, with $s = \log(t+|x|)$
    and $\omega = x/|x|$, we define
    $\bbeta^A_s(\omega) = B^A(t, x)s^{-1/2}$.
    Writing $\frac{d}{ds} = \pa_s|_{u = B^A(t, x), \omega = const.}$, in terms
     of $\bbeta^A$,
     the transport equation \eqref{localBAeqntr} reads
    \begin{equation}
      \label{}
      \frac{d}{ds} \bbeta^A_s(\omega) = -\frac{1}{2s^{1/2}} [\pa_u\psi](s,\omega) + F_A(s,\omega),
    \end{equation}
    with the understanding that the quantities on the right-hand side
    are evaluated at the point $(t, x)$ on $\Gamma^A$ with $x/|x| = \omega$
    and $\log(t+|x|) = s$.
    To get higher-order estimates for the shock-defining functions
    $B^A$, we are going to differentiate this equation along the shock.
    For this it is convenient to work in terms of the operators
    \begin{equation}
      \label{}
      \tau_A = \pa_s\big|_{u = B^A, \omega = const.} = \pa_s + \pa_s B^A \pa_u,
      \qquad
      \Omega_A = \Omega + \Omega B^A \pa_u,
    \end{equation}
    which are tangent to the shock. If $m \geq 0$ is an integer
    and $J$ is a multi-index, then
    since $\tau_A$ and $\Omega_A$ commute with $\frac{d}{ds} = \tau^A$,
    writing $\bbeta_{s}^{A, m, J}(\omega) = \tau_A^m \Omega_A^J\bbeta_s(\omega)$,
    \begin{equation}
      \label{}
      \frac{d}{ds} \left(\bbeta_s^{A, m, J}(\omega)\right)
      = -\frac{1}{2}\tau_A^m \Omega_A^J \left( \frac{1}{s^{1/2}}[\pa_u\psi](s,\omega)\right) 
      + \tau_A^m\Omega_A^J F_A(s,\omega),
      \qquad \text{ at }\Gamma^A.
    \end{equation}
    For each fixed $\omega \in \mathbb{S}^2$, we integrate this expression between 
    any two values of $s_0, s_1$ of $s$ on the shock $\Gamma^A$ to get
    \begin{equation}
      \label{timeintegratedbeta}
      |\bbeta_{s_1}^{A, m, J}(\omega) - \bbeta_{s_0}^{A, m, J}(\omega)|
      \lesssim\int_{s_0}^{s_1}
      \left| \tau_A^m\Omega_A^J \left(s^{-1/2}[\pa_u\psi](s, \omega) \right) \right|
      +
      \left| \tau_A^m \Omega_A^J F_A(s, \omega)\right|\, ds,
    \end{equation}

    We now let $s^A(t', \omega)$ denote the value of $s = \log(t+|x|)$ lying at the
    intersection of the sets $\{t = t'\}$, $\{x/|x| = \omega\}$ and $\Gamma^A$.
    Taking $s_0 = s^A(t_0, \omega)$ and $s_1 = s^A(t_1, \omega)$
    in \eqref{timeintegratedbeta}, we have
    \begin{equation}
      \label{startoftimeintegrated}
      |\bbeta_{s^A(t_1, \omega)}^{A, m ,J}(\omega)
      - \bbeta_{s^A(t_0,\omega)}^{A, m, J}(\omega)|
      \lesssim\int_{s^A(t_0, \omega)}^{s^A(t_1, \omega)}
      \left| \tau_A^m\Omega_A^J \left(s^{-1/2}[\pa_u\psi](s, \omega) \right) \right|
      +
      \left| \tau_A^m \Omega_A^J F_A(s, \omega)\right|\, ds.
    \end{equation}

    Take $\alpha$ as in \eqref{parameters} so that in particular $\alpha > 1$
    and set $h(s) = \log s (\log \log s)^\alpha$. For any $\omega \in \mathbb{S}^2$,
    the above gives
    \begin{multline}
      \label{betadifferenceestimate}
      |\bbeta_{s^A(t_1, \omega)}^{A, m, J}(\omega)
      - \bbeta_{s^A(t_0,\omega)}^{A, m, J}(\omega)|^2
      \\
      \lesssim
     \left(\int_{s^A(t_0, \omega)}^{s^A(t_1, \omega)} \frac{1}{1+s} \frac{1}{h(s)} \,ds\right)
     \left(\int_{s^A(t_0,\omega)}^{s^A(t_1, \omega)} (1+s)h(s)  
         \left| \tau_A^m\Omega_A^J \left(s^{-1/2}[\pa_u\psi](s, \omega) \right) \right|^2
         \,ds
         \right)
         \\
         +
     \left(\int_{s^A(t_0, \omega)}^{s^A(t_1, \omega)} \frac{1}{1+s} \frac{1}{h(s)} \,ds\right)
     \left(\int_{s^A(t_0, \omega)}^{s^A(t_1, \omega)}  (1+s)h(s) |\tau_A^m\Omega_A^J F_A|^2\,ds\right)
     \\
     \lesssim c_0(\epsilon_0) 
     \int_{s^A(t_0, \omega)}^{s^A(t_1, \omega)}  
     (1+s)h(s)  \left| \tau_A^m\Omega_A^J \left(s^{-1/2}[\pa_u\psi](s, \omega) \right) \right|^2
    + h(s)|\tau_A^m\Omega_A^J F_A|^2 \, ds.
    \end{multline}
    If we integrate this expression over $\mathbb{S}^2$ and use that
    $\int_{\mathbb{S}^2} \int_{s^A(t_0, \omega)}^{s^A(t_1, \omega)}
    Q(s,\omega) dsdS(\omega) \sim \int_{t_0}^{t_1} \int_{\Gamma^A_t} \frac{1}{v} q(t,x)\, dS dt$
    where $Q(s, \omega) = q|_{u = \bbeta_s(\omega)}$, we further find
    \begin{multline}
      \label{integratedbetaeqn}
      \int_{\mathbb{S}^2}
      |\bbeta_{s^A(t_1, \omega)}^{A, m, J}(\omega)
      - \bbeta_{s^A(t_0,\omega)}^{A, m, J}(\omega)|^2\, dS(\omega)
      \\
      \lesssim
      c_0(\epsilon_0) 
      \int_{t_0}^{t_1} \int_{\Gamma^A_t}
      \frac{(1+s)h(s)}{1+v}  \left| \tau_A^m\Omega_A^J \left(s^{-1/2}[\pa_u\psi](s, \omega) \right) \right|^2
      + \frac{h(s)}{1+v}|\tau_A^m\Omega_A^J F_A|^2 \, dS dt.
    \end{multline}
    We will use the above bound at the left shock with $m + |J| \leq N_L$ and at
    the right shock with $m + |J| \leq N_C - 2$,
    and the function $h$ has been chosen so that 
    in these cases, the above is bounded by our a priori assumptions (see Lemma \ref{lem:bdrhsbetaeqn}). For $m + |J| \geq N_C - 1$, we cannot
    easily control the above quantity because
    we have weaker control over top-order derivatives
    of $\psi_C$ at the shocks. 
    To handle this case, we instead return to \eqref{startoftimeintegrated}
    and bound
    \begin{multline}
      \label{}
      \left(\int_{s^R(t_0, \omega)}^{s^R(t_1, \omega)}
      \left| \tau_R^m\Omega_R^J \left(s^{-1/2}[\pa_u\psi](s, \omega) \right) \right|
      +
      \left| \tau_R^m \Omega_R^J F_A(s, \omega)\right|\, ds\right)^2
      \\
      \lesssim
      (s^R(t_1, \omega) - s^R(t_0, \omega))
      \int_{s^R(t_0, \omega)}^{s^R(t_1, \omega)}
      \left| \tau_R^m\Omega_R^J \left(s^{-1/2}[\pa_u\psi](s, \omega) \right) \right|^2
      +
      \left| \tau_R^m \Omega_R^J F_A(s, \omega)\right|^2\, ds.
    \end{multline}
    If we use this bound in \eqref{startoftimeintegrated} and integrate over the sphere, we find
    \begin{multline}
      \label{integratedbetaeqnhigh}
      \int_{\mathbb{S}^2} \frac{1}{s^R(t_1, \omega) - s^R(t_0, \omega)}
      |\bbeta_{s^A(t_1, \omega)}^{A, m, J}(\omega)
      -  \bbeta_{s^A(t_0,\omega)}^{A, m, J}(\omega)|^2\, dS(\omega)
      \\
      \lesssim \int_{t_0}^{t_1} \int_{\Gamma^R_t} \frac{1}{1+v}
      \left| \tau_R^m\Omega_R^J \left(s^{-1/2}[\pa_u\psi] \right) \right|^2
      +
      \left| \tau_R^m \Omega_R^J F_A\right|^2\, dS dt.
    \end{multline}

    We now show how our bootstrap assumptions imply bounds
    for the quantities on the right-hand sides of \eqref{integratedbetaeqn}
    and \eqref{integratedbetaeqnhigh}.
    \begin{lemma}
        \label{lem:bdrhsbetaeqn}
      Under the hypotheses of Proposition \ref{bootstrapprop}, with $h(s) = \log s (\log \log s)^\alpha$,
      for $m + |J| \leq N_L$ we have the following bound at the left shock,
      \begin{equation}
        \label{leftshockbdrhs}
    \int_{t_0}^{t_1} \int_{\Gamma^L_t}
      \frac{(1+s)h(s)}{1+v}  \left| \tau_L^m\Omega_L^J \left(s^{-1/2}[\pa_u\psi](s, \omega) \right) \right|^2
      + \frac{h(s)}{1+v}|\tau_L^m\Omega_L^J F_L|^2 \, dS dt
      \lesssim \epsilon_L^2
      \end{equation}
      For $m + |J| \leq N_C-2$, we also have the following bound at the right shock
      \begin{equation}
        \label{rightshockbdrhs}
          \int_{t_0}^{t_1} \int_{\Gamma^R_t}
      \frac{(1+s)h(s)}{1+v}  \left| \tau_R^m\Omega_R^J \left(s^{-1/2}[\pa_u\psi](s, \omega) \right) \right|^2
      + \frac{h(s)}{1+v}|\tau_R^m\Omega_R^J F_R|^2 \, dS dt
      \lesssim \epsilon_C^2.
  \end{equation}
      Finally, for $N_C \geq m+|J| \leq N_C-1$, we have the following
      bound at the right shock,
      \begin{equation}
        \label{centralshockbdrhs}
        \int_{t_0}^{t_1} \int_{\Gamma^R_t} \frac{1}{1+v}
      \left| \tau_R^m\Omega_R^J \left(s^{-1/2}[\pa_u\psi] \right) \right|^2
      +
      \frac{1}{1+v}
      \left| \tau_R^m \Omega_R^J F_A\right|^2\, dS dt
      \lesssim
      \epsilon_C^2.
      \end{equation}
    \end{lemma}
    \begin{proof}
        We start by relating the operators $\tau_A, \Omega_A$ to the
        tangential fields we used in earlier sections.
    We abuse notation slightly and use the notation $Z_T, Z_{T, \mB}$ 
    to denote the fields $Z - Z(u-B^A) n$ and $Z_{\mB} - Z_{\mB} (u-B^A) n$ at
    either shock. With this notation, we have
    the following identities,
    \begin{equation}
      \label{}
      \tau_A = (\sBt)_T = \sum_{Z \in \mathcal{Z}_m} a_A^Z
      Z_T, \quad \Omega_A = \Omega_T,
    \end{equation}
    for coefficients $a_A^Z$ satisfying the symbol condition
    \eqref{minksymb},
    and in particular, 
    \begin{align}
      \label{yaatb}
      |\tau_A^m \Omega_A^J q| &\lesssim \sum_{|I| \leq m + |J|} |Z_{\mB, T}^I q|,
      \\
      |\tau_A^m\Omega_A^J q|
                          &\leq C(M)
  \sum_{|I| \leq m + |J|} |Z_T^I q|
  + C(M)\sum_{|I| \leq m+|J|} |B^A|_{I, \mathcal{Z}_m}
  \sum_{|K| \leq (m+|J|)/2+1} |Z^K_T q|.
    \end{align}
    In getting the second bound, we used that by 
    \eqref{eztransfer0} and the symbol condition
    \eqref{minksymb} we have
    \begin{equation}
      \label{}
      |Z^I_T a_Z| \leq C(M) \sum_{|I'| \leq |I|} |Z^{I'} a_Z|
      + C(M) |B^A|_{I, \mathcal{Z}_m} \sum_{|J| \leq |I|/2+1} |Z^J a_Z|
      \leq C'(M)(1 + |B^A|_{I, \mathcal{Z}_m})
    \end{equation}
    for a constant $C'(M)$.
   At either shock $\Gamma^A$, we therefore have
    \begin{multline}
      \label{}
         \left| \tau_A^m\Omega_A^J \left(s^{-1/2}[\pa_u\psi](s, \omega) \right) \right|
         \lesssim
         \frac{1}{(1+s)^{1/2}} \sum_{|I| \leq m + |J|}\left( |Z_{\mB, T}^I \pa_u \psi_C|
             + |Z_T^I \pa_u \psi_A| \right)\\
             + C(M)\sum_{|I| \leq m + |J|} |B^A|_{I, \mathcal{Z}_m} \sum_{|K|\leq (m+|J|)/2+1}
             |Z_T^K \pa_u \psi_A|
    \end{multline}
    Using \eqref{eztransfer0} and \eqref{mBtangentialdifferencermk2} 
    to convert the tangential fields
    $Z_{\mB, T}$ and $Z_T$ into $Z_{\mB}$ and $Z$ fields and commuting
    with $\pa_u$, this gives
    \begin{multline}
      \label{pwbdatshockforbeta}
         \left| \tau_A^m\Omega_A^J \left(s^{-1/2}[\pa_u\psi](s, \omega) \right) \right|
         \lesssim
         \frac{1}{(1+s)^{1/2}} C(M)
         \sum_{|I| \leq m + |J|}\left( |\pa_u Z_{\mB}^I \psi_C|
             + |\pa_u Z^I \pa_u \psi_A| \right)\\
             + C(M)\sum_{|I|\leq m+|J|}
             \left(|B^A|_{I, \mathcal{Z}_m} + |B^A|_{I, \mZB}\right)
             \left(\sum_{|K| \leq  (m+|J|)/2+1} \left( |\pa_u Z_{\mB}^K \psi_C|
             + |\pa_u Z^K \pa_u \psi_A| \right)\right) 
     \end{multline}
    where the terms on the last line are not present if $m+|J|\leq N_A/2$. 
    We therefore have the bound
    \begin{multline}
      \label{}
      \int_{t_0}^{t_1} \int_{\Gamma^A_t} \frac{(1+s)h(s)}{1+v}
         \left| \tau_A^m\Omega_A^J \left(s^{-1/2}[\pa_u\psi](s, \omega) \right) \right|^2
         \\
      \lesssim C(M) \sum_{|I| \leq m+|J|}
      \int_{t_0}^{t_1} \int_{\Gamma^A_t} \frac{h(s)}{1+v}
      \left(|\pa_u Z_{\mB}^I \psi_C|^2
      + |\pa_u Z^I\psi_A|^2\right)\, dS dt
      \\
      + C(M) 
      \left(\sum_{|I| \leq m+|J|}
      \sup_{t_0 \leq t \leq t_1}\int_{\Gamma^A_t}
      |B^A|_{I, \mB}^2+|B^A|_{I, \mZ}^2  \, dS\right)
      \sum_{|K| \leq (m+|J|)/2}
      \int_{t_0}^{t_1} \sup_{\Gamma_t^A} 
       \frac{h(s)}{1+v}\left(|\pa_u Z_{\mB}^K \psi_C|^2
      + |\pa_u Z^K\psi_A|^2\right) \, dt.
    \end{multline}
    If $A = L$ and $m+|J| \leq N_L$ or $A = R$ and $m+|J| \leq N_C-2$,
    to handle the terms on the first line, we use the $L^2$ bounds from \eqref{rightrivialbdybds},
    the last line of \eqref{rightbdytrivialbds}
    and \eqref{controlofpsiLL}
    for $\psi_R,\psi_C$ and $\psi_L$
    (noting that the weight $h(s)/(1+v)$ is
    dominated by all of the weights in those
    estimates)
    along the shock which gives
    \begin{equation}
      \label{}
       \sum_{|I| \leq m+|J|}
      \int_{t_0}^{t_1} \int_{\Gamma^A_t} \frac{h(s)}{1+v}
      \left(|\pa_u Z_{\mB}^I \psi_C|^2
      + |\pa_u Z^I\psi_A|^2\right)\, dS dt
      \lesssim \epsilon_A + \epsilon_C
    \end{equation}
    which is as needed since $\epsilon_R \leq \epsilon_C \leq \epsilon_L$. The terms
    on the second line are handled in the same way but using instead the pointwise
    bounds \eqref{pwright}, \eqref{pwcentral}, 
    \eqref{pwleft} for $\psi_R$,
    $\psi_C$, and $\psi_L$, and the $L^2$ bounds
    \eqref{leftgeom}-\eqref{rightgeom} 
    for $B^A$ along with the fact that $|B^A|_{I, \mZB}\sim |B^A|_{I, \mathcal{Z}_m}$.
    This gives the first bound in
    \eqref{leftshockbdrhs} and
    \eqref{rightshockbdrhs}. 
    To prove the first bound in
    \eqref{centralshockbdrhs},
    we argue in nearly the same way, with the only
    difference being that we use the weaker
    estimate on the first line of \eqref{rightbdytrivialbds} in place of the estimate on the last
    line there to handle the highest-order
    derivatives.

    The bounds for the remainder terms $F_A$ can be handled in the same way using
    the explicit formula \eqref{betaevolF2}, and we omit the proof. 
        \end{proof}

We now give the proof of Propositions
    \ref{leftshockbootstrap} and \ref{rightshockbootstrap}.
    \begin{proof}[Proof of Proposition \ref{leftshockbootstrap}]
        By \eqref{integratedbetaeqn} and 
        Lemma \ref{lem:bdrhsbetaeqn}, 
        for $m + |J| \leq N_L$, we have
        the bound
        \begin{equation}
          \label{intoverS2betabdL}
          \int_{\mathbb{S}^2}
            \left|\tau_L^m \Omega_L^J
            \bbeta^L_{s^L(t_1, \omega)}(\omega) -
            \tau_L^m
            \Omega_L^J \bbeta^L_{s^L(t_0, \omega)}
            (\omega)\right|^2\, dS(\omega)
            \leq  c_0(\epsilon_0) \epsilon_L^2.
        \end{equation}
        Applying this with $m = 0$ and summing over
        $|J| \leq 4$, by Sobolev embedding this
        implies
        \begin{equation}
          \label{}
          \sup_{\omega \in \mathbb{S}^2}
          \left|\bbeta_{s^L(t_1, \omega)}(\omega) - \bbeta_{s^L(t_0, \omega)}(\omega)\right|
          +
          \left|\Omega_L \bbeta_{s^L(t_1, \omega)}(\omega) - \Omega_L \bbeta_{s^L(t_0, \omega)}(\omega)\right|
          \leq c_0(\epsilon_0)\epsilon_L.
        \end{equation}
        Recalling the definition
        $\bbeta_s^L(\omega) = s^{-1/2}B^L(t, x)$ and that
        $\Omega_LB^L = \Omega_TB^L = \Omega B^L$ since $n B^L = 0$, this gives
        the bounds
        \begin{equation}
          \label{}
          \left|\frac{B^L(t,x)}{s^{1/2}} - \frac{B^L_0(x)}{s^L(t_0, \omega)^{1/2}} 
          \right|
          + \left| \frac{\Omega B^L(t,x)}{s^{1/2}} - \frac{\Omega B^L_0(x)}{s^L(t_0,\omega)^{1/2}}
          \right|
          \leq c_0(\epsilon_0) \epsilon_L^2, \qquad \text{ at } \Gamma^L,
        \end{equation}
        and in light of the definition
        of the norms $\mathring{K}^L, \slashed{K}^L$
        (see \eqref{KL0}-\eqref{sKL0}) of $B^L_0$, this gives the first bound in
        \eqref{lowestbootstrapofBL} and \eqref{lowestbootstrapofsBL}. To get
        the second bound in \eqref{lowestbootstrapofBL} we just use the transport
        equation
        \eqref{localBAeqntr} along with the pointwise decay estimates
        from Lemma \ref{basicpwdecay}. The higher-order estimate \eqref{improvedbdleftshock}
        follows in the same way, after additionally using \eqref{yaatb} to relate
        the operators $\tau_A, \Omega_A$ to the fields $Z_{T}$ in the definition
        of $|\cdot|_{I, \mathcal{Z}_m}$.

    \end{proof}

    \begin{proof}[Proof of Proposition \ref{rightshockbootstrap}]
        The bounds for $|I| \leq N_C-2$ are proven in exactly the same
        way was the bounds from the proof of Proposition \ref{leftshockbootstrap}.
        For $N_C-1 \leq |I| \leq N_C$, the only difference is that we use
        the fact that by \eqref{centralshockbdrhs} and \eqref{integratedbetaeqnhigh}
        we have the bound
        \begin{equation}
          \label{}
             \int_{\mathbb{S}^2}
             \frac{1}{s^R(t_1, \omega) - s^R(t_0, \omega)}
            \left|\tau_L^m \Omega_L^J
            \bbeta^L_{s^L(t_1, \omega)}(\omega) -
            \tau_L^m
            \Omega_L^J \bbeta^L_{s^L(t_0, \omega)}
            (\omega)\right|^2\, dS(\omega)
            \leq c_0(\epsilon_0) \epsilon_c^2
        \end{equation}
        in place of \eqref{intoverS2betabdL}.
    \end{proof}

\subsection{The asymptotic behavior of the shocks}
\label{asympsec}
\begin{proof}[Proof of Theorem \ref{asympthm}]

 Let $t_1, t_2,...$ be any sequence of times with $t \in \mathbb{R}_{> 0}$.
 We will show that $\log(t_j + r^A(t_j,\omega))^{-1/2}(t- r^A(t, \omega))$ form a Cauchy sequence
 in $H^{M_A}(\mathbb{S}^2)$.
 Let $s_j^A(\omega)$ denote the value of $s = \log(t+|x|)$ lying
 at the intersection of the sets $\{t = t_j\}$, $\{x/|x| = \omega\}$
 and $\Gamma^A$. By \eqref{timeintegratedbeta} with $m = 0$,
 abbreviating $\bbeta_{s}^{A, J} = \bbeta_s^{A, 0, J}$, we have the bound
 \begin{equation}
   \label{}
   |\bbeta_{s_j^A(\omega) }^{A, J}(\omega) - \bbeta_{s^A_\ell(\omega)}^{A, J}(\omega)|
   \lesssim
   \int_{s^A_\ell(\omega)}^{s^A_j(\omega)}
   \left| \tau_A^m\Omega_A^J \left(s^{-1/2}[\pa_u\psi](s, \omega) \right) \right|
      +
      \left| \tau_A^m \Omega_A^J F_A(s, \omega)\right|\, ds, \end{equation}
 Following exactly the same steps that lead to \eqref{integratedbetaeqn}
 and then using Lemma \ref{lem:bdrhsbetaeqn}, we find
 \begin{equation}
   \label{}
   \int_{\mathbb{S}^2} 
   |\bbeta_{s_j^A(\omega) }^{A, J}(\omega) - \bbeta_{s^A_\ell(\omega)}^{A, J}(\omega)|^2\, dS(\omega)
   \lesssim 
   c_0(\epsilon_\ell)(\epsilon_A + \epsilon_C),
 \end{equation}
 where $c_0$ is a continuous function with $c_0(0) = 0 $ and where 
 $\epsilon_\ell = 1/\left(\sup_{\omega \in \mathbb{S}^2}s^A_\ell(\omega)  \right)$.
 It follows that the functions $\{\bbeta^{A}_{s_j^A(\cdot)}(\cdot)\}_{j = 1}^\infty$
 form a Cauchy sequence in $H^{M_A}$. As
 a result, 
 $\widetilde{\beta}^A_{s^A(t,\omega)}(\omega) = 
 \log(t + r_A(t,\omega)) B^A(t, r^A(t,\omega)\omega)$ has a limit in
 $H^{M_A}(\mathbb{S}^2)$ as $t \to \infty$.
 The theorem now follows.
\end{proof}

\appendix

\section{Basic properties of the vector fields in $\Z$ and $\mZB$}
\subsection{Commutators with $\mathcal{Z}$}

The vector fields $Z$ from $\Z$ satisfy
\begin{equation}
 Z\Box q - \Box Z q = c_Z \Box q,
 \label{Zcommbox}
\end{equation}
where $\Box = -\pa_t^2 + \Delta$ is the Minkowskian wave operator,
and $c_S = -2$ and otherwise  $c_Z = 0.$ With
$\widetilde{Z} = Z - c_Z$,
\begin{equation}
 \widetilde{Z} \Box q = \Box Z q.
 \label{widehatbox}
\end{equation}

Moreover there are constants $c_{Z\alpha}^\beta$ so that each $Z$
in $\mathcal{Z}$ satisfies
\begin{equation}
 [\widetilde{Z}, \pa_\alpha] = c_{Z\alpha}^\beta \pa_\beta,
 \label{Zcommpa}
\end{equation}
where here $\pa_\alpha, \pa_\beta$ denote derivatives taken with respect
to the standard rectangular coordinate system.

We will need a higher-order version of this identity.
\begin{lemma}
  \label{divergencestructure}
 There are constants $c_{J\beta}^{I\alpha}$ so that
 for any vector field $\gamma = \gamma^\alpha\pa_\alpha$,
\begin{equation}
 \widetilde{Z}^I \pa_\alpha \gamma^\alpha
 = \pa_\alpha\gamma_I^\alpha,
 \label{gammaI0}
\end{equation}
where
\begin{equation}
 \gamma_I^\alpha = \widetilde{Z}^I \gamma^\alpha
 + \sum_{|J| \leq |I|-1} c_{J\beta}^{I\alpha} \widetilde{Z}^J\gamma^\beta.
 \label{gammaI}
\end{equation}

\end{lemma}
\begin{proof}
  For $|I| = 1$ the identity \eqref{gammaI} is just the fact that
  \begin{equation}
   \widetilde{Z} \pa_\alpha \gamma^\alpha
   = \pa_\alpha \gamma_Z^{ \alpha},
   \end{equation}
   where
   \begin{equation}
   \gamma_Z^{ \alpha} =
   \widetilde{Z} \gamma^\alpha +c^\alpha_{\beta Z} \gamma^\beta
   \label{lowestgammacomm}
  \end{equation}
  with the constants $c^\alpha_{\beta Z}$ as in \eqref{Zcommpa}.
  For $|I| > 1$  the identity \eqref{gammaI} follows
  from induction after writing
  \begin{equation}
   \widetilde{Z}^J \widetilde{Z}\pa_\alpha \gamma^\alpha
   = \widetilde{Z}^J \pa_\alpha \gamma_Z^{\alpha}
   = \pa_\alpha \left(\widetilde{Z}^J\gamma_Z^{\alpha}
   +\sum_{|K| \leq |J|-1} c^{\alpha J}_{\beta K}\widetilde{Z}^K \gamma_Z^{\beta}\right),
   \label{}
  \end{equation}
 for constants $c^{\alpha J}_{\beta K}$.
\end{proof}

We also record the following result, which follows directly
from the previous result, the product rule, the identity \eqref{widehatbox},
and the fact that $(1+|u|)|\pa q| + (1+v)|\pa_v q|
+ (1+v)|\nas q| \lesssim \sum_{Z \in \mathcal{Z}} |Zq|$.
\begin{lemma}[The commutation currents in the exterior]
	\label{minkcommcurrentlem}
	Suppose that $\gamma = \gamma^{\alpha\beta\beta'}$ satisfy $(1+v)|Z^J \gamma| \leq C(|J|)$
	for any $|J|$, where all quantities are expressed in the usual
	rectangular coordinate system.
	 Then
	\begin{equation}
	  \widetilde{Z}^I
	  \left(\Box q + \pa_\alpha(\gamma^{\alpha\beta\beta'}\pa_\beta q
		\pa_{\beta'} q)\right)
	=
		\left(\Box Z^Iq +
 	 \pa_\alpha(\check{\gamma}^{\alpha\beta\beta'}\pa_{\beta}q \pa_{\beta'} Z^Iq ) \right)
	 + \pa_\alpha P_{I}^\alpha
	 \label{}
	\end{equation}
	with $\check{\gamma}^{\alpha\beta\beta'} = \gamma^{\alpha\beta\beta'} + \gamma^{\alpha\beta'\beta}$
	and
	where the commutation current $P_I$ satisfies the bound
	\begin{equation}
	 |P_I| + (1+|u|) |\pa_u P_I| + (1+v)|\pa_v P_I| +
	 (1+v) |\nas P_I| \lesssim
	 \frac{1}{1+v} \sum_{|I_1| + |I_2| \leq |I|}|\pa Z^{I_1} q|
	 |\pa Z^{I_2} q|
	 \label{}
	\end{equation}
\end{lemma}

To handle the boundary terms we encounter along the left shock,
we will need bounds involving $[\evm, Z^I]$ and $[n, Z^I]$. To get bounds
for these quantities we will repeatedly make use of the following simple
identities
\begin{equation}
 \pa_t = \pa_v  + \pa_u,
 \qquad
 \pa_i = \omega_i (\pa_v - \pa_u) + \slashed{\pa}_i
 = \omega_i (\pa_v - \pa_u)
 +\frac{\omega^j}{r}\Omega_{ij},
 \label{inhomognull}
\end{equation}
\begin{equation}
 S = u\pa_u + v\pa_v,
 \qquad
 \Omega_{0i} = \omega_i (v\pa_v - u \pa_u) + \frac{t}{r} \omega^j\Omega_{ij}
 = \omega_i (v\pa_v - u \pa_u) + \left(1 + \frac{u}{r}\right) \omega^j\Omega_{ij}
 \label{homognull}
\end{equation}
We will also use the facts that
\begin{equation}
 \pa_u \omega_i = \pa_v \omega_i = 0,
 \qquad
 \pa_u \Omega_{ij} -  \Omega_{ij}\pa_u
 =
 \pa_v \Omega_{ij} - \Omega_{ij} \pa_v = 0,
 \qquad
 \Omega_{ij} r = \Omega_{ij}u = \Omega_{ij} v = 0,
 \label{anglecomm}
\end{equation}
and that the collection of angular momentum operators $\Omega_{ij}$ are closed
under commutation,
\begin{equation}
 \Omega_{ij} \Omega_{k\ell} - \Omega_{k\ell} \Omega_{ij}
 = c_{ij k\ell}^{mn}\Omega_{mn}
 \label{liealgebra}
\end{equation}
for constants $c$. From the above identities, we also have the
well-known fact that each $Z$ can be written in the form
\begin{equation}
 Z = (a_Z + a_Z' u) \pa_u + b_Z (1+v) \pa_v + c_{Z} \Omega,
 \label{Zintonull}
\end{equation}
where the coefficients $a_Z, a_Z', b_Z, c_Z$ satisfy the symbol-type condition
\begin{equation}
 |Z^J a| \lesssim 1
 \label{minksymb}
\end{equation}
if $t/2 \leq r \leq 3t/2$, $t \geq 1$, say.
We also record the well-known fact
that we can express
\begin{equation}
 \pa_u = \sum_{Z \in \mathcal{Z}}
 \frac{1}{1+|u|} a_{u}^Z Z,
 \qquad
 \pa_v = \sum_{Z \in \mathcal{Z}}
 \frac{1}{1+v} a_{v}^Z Z
 \label{nullintoZ}
\end{equation}
for coefficients satisfying \eqref{minksymb}. In fact,
\begin{equation}
 \pa_u = \frac{1}{2u} \left(S - \omega^i\Omega_{0i}\right),
 \qquad
 \pa_v = \frac{1}{2v} \left( S + \omega^i\Omega_{0i}\right).
 \label{pointofuv}
\end{equation}

At this point we also record the identities
\begin{equation}
 \nas_i = \frac{\omega^j}{r} \Omega_{ji},
 \qquad
 \sDelta = \frac{\omega^j \omega_\ell}{r^2} \Omega_{ji}\Omega^{\ell i},
 \label{sdeltaformulaOmega}
\end{equation}
raising indices with the Euclidean metric. The second identity here follows from
the first and 
\begin{equation}
 \sDelta = \textrm{tr}(\nas^2) = \frac{\omega^j \omega_\ell}{r^2} \Omega_{ji} \Omega^{\ell i}
 + \frac{\omega^j}{r^2} (\Omega_{ji}\omega_\ell)\Omega^{\ell i},
 \label{}
\end{equation}
where the last term vanishes by an explicit calculation.

\begin{lemma}
	\label{nullcommsm}
 For each multi-index $I$, we have
 \begin{align}
  |[\evm, Z^I] q|&\lesssim \sum_{|J| \leq |I|-1}
	|\evm Z^J q| + \frac{1+|u|}{1+v} |\nas Z^J q|
	+\sum_{|K| \leq |J| - 2} \frac{(1+|u|)^2}{(1+v)^2}|n Z^J q|,
  \label{ellZcomm}\\
  |[n, Z^I] q|
	&\lesssim
	\sum_{|J| \leq |I|-1}
	|nZ^J q| +  |\nas Z^Jq|
	+ \sum_{|K| \leq |I|-2} |\evm Z^K q|.
  \label{nZcomm}
\end{align}

\end{lemma}
\begin{proof}
	The result follows from the claim that
	 $[\evm, Z^I]$ is a sum of terms of the following forms,
	 \begin{equation}
	  a(t,x) \evm Z^J,
		\quad a(t,x)\frac{1 + |u|}{(1+v)^2} \Omega Z^J,
		\quad
		|J| \leq |I|-1,
		\qquad
		a(t,x)\frac{(1 + |u|)^2}{(1+v)^2} n Z^K,
		\quad
		 |K| \leq |I|-2
	  \label{ellz}
	 \end{equation}
	 and $[n, Z^I]$ is a sum of terms of the following forms
	 \begin{equation}
	  a(t,x) n Z^J,
		\quad
		a(t,x)\frac{1}{1+v} \Omega Z^J,
		\quad
			|J| \leq |I|-1,
			\qquad
		a(t,x) \evm Z^K,
		\quad |K| \leq |I|-2,
	  \label{nz}
	 \end{equation}
	 where the $a(t,x)$ are functions satisfying the symbol condition
	 \eqref{minksymb}.

 The claims follows from a direct calculation. Using the facts that
 \begin{equation}
  n \omega_i = \evm \omega_i = 0,
	\qquad
	\Omega_{ij} u = \Omega_{ij} v{=0} \qquad
	[n,\Omega_{ij}] = [\evm, \Omega_{ij}] = 0,
  \label{minkfact0}
 \end{equation}
 it is straightforward to verify
 \begin{equation}
  [\evm, S] = \ell^{ m}, \qquad
	[\evm, \Omega_{ij}] = 0, \qquad
	[\evm, \pa_t] = 0, \qquad
	[\evm, \pa_i] = \frac{1}{r^2} c_i^{\ell \Omega} \cdot\Omega,
  \label{}
 \end{equation}
 as well as
  \begin{equation}
   [n, S] = n, \qquad
 	[n, \Omega_{ij}] = 0, \qquad
 	[n, \pa_t] = 0, \qquad
 	[n, \pa_i] = \frac{1}{r^2}c_i^{n \Omega} \cdot \Omega.
   \label{minkfact1}
  \end{equation}
	where the coefficients above all satisfy \eqref{minksymb}.
	It remains to commute with the fields $\Omega_{0i} = t\pa_i + x_i \pa_t$
	and for this we use the identity
	\begin{equation}
	 \Omega_{0i} = \omega_i (v\pa_v - u\pa_u) + \left(1 + \frac{u}{r}\right) \omega^j
	 \Omega_{ij}
	 \label{}
	\end{equation}
	and use the above relations to see that
	\begin{equation}
	 [\evm, \Omega_{0i}] = \omega_i \ell^{ m} - \frac{1}{2} \frac{u}{r^2} \omega^j \Omega_{ij},
	 \qquad
	 [n,\Omega_{0i}] = -\omega_i n + \frac{1}{2} \left( \frac{1}{r} + \frac{u}{r^2}\right)
	 \omega^j \Omega_{ij},
	 \label{}
	\end{equation}
	and so
	\begin{equation}
	 [\evm, Z] = a_Z^{\ell \ell} \ell + \frac{1 +|u|}{(1+v)^2} a_Z^{\ell \Omega} \cdot \Omega,
	 \qquad
	 [n, Z] = a_Z^{nn} n + \frac{1}{1+v}a^{n \Omega}_Z\cdot \Omega,
	 \label{nlcom1}
	\end{equation}
	for symbols $a_Z^{\alpha \beta}$, satisfying \eqref{minksymb}.
	To get the higher-order commutators we also need to commute $\Omega$
	with each $Z \in \mathcal{Z}$. For our purposes all that is needed is that
	the commutators $[Z,\Omega] = \sum_{Z' \in \mathcal{Z}} c_Z^{Z'} Z'$
	for constants $c_Z^{Z'}$.
	 Writing
	$[\evm , Z Z^J] = [\evm, Z]Z^J + Z [\evm, Z^J]$ and using \eqref{nlcom1}
	we find that $[\evm, Z^I]$ is a sum of terms of the form
	\begin{equation}
	 c \evm Z^J, \quad c \frac{1 + |u|}{(1+v)^2}  \Omega Z^J,
	 \quad
	 c \frac{1 +|u|}{(1 + v)^2} Z^K,
	 \label{}
	\end{equation}
	and that $[n, Z^I]$ is a sum of terms of the form
	\begin{equation}
	a n Z^J, \quad a \frac{1}{1+v}  \Omega Z^J,
	\quad
	a \frac{1 }{1 + v} Z^K,
	\label{}
 \end{equation}
	where $|J| \leq |I|-1$ and $|K| \leq |I|-2$ and where
	the coefficients $c$ satisfy the symbol condition
	\eqref{minksymb}. These are of the form we want apart
	from the last term in each expression, and after using \eqref{Zintonull}
	to handle this term we get \eqref{ellz}-\eqref{nz}.
\end{proof}

\subsection{Commutators with $\mZB$}
\label{mZBcommutatorsec}

Recall that $\mZB = \{\sBo = s\pa_u, \sBt = v\pa_v,
\Omega_{ij} = x_i \pa_j - x_j \pa_i\}$. We start by recording
some basic identities inolving these fields. All of these fields commute
with $n = \pa_u$ and $\sBo, \Omega_{ij}$ additionally commute with
$\evmB=\pa_v+\frac{u}{vs}\pa_u$,
\begin{equation}
    [\sBo, \evmB] = [\Omega_{ij}, \evmB] = [Z_{\mB}, n] = 0.
 \label{basicB0}
\end{equation}
As for the commutator $[\sBt, \evmB]$,
\begin{equation}
 [\sBt, \evmB] = -\left(\evmB + \frac{u}{vs^2} \pa_u\right)
 \label{basicB1}
\end{equation}
which we will see does not introduce any serious difficulties.
The fields $\sBo, \sBt$ do not commute with the angular
Laplacian but the commutator is given by
\begin{equation}
 [\sBo, \sDelta] 
 = \frac{s}{r} \sDelta,
 \qquad
 [\sBt, \sDelta ]
 = -\frac{v}{r} \sDelta {= -2\sDelta - \frac{u}{r} \sDelta}
  \label{basicBangular}
\end{equation}
Finally, we note that the family $\mathcal{Z}_{\mB}$ does not form an algebra because
$\Sbo, \Sbt$ do not commute, but their commutator is
\begin{equation}
 [\Sbo, \Sbt] = -\pa_u = -\frac{1}{s} \Sbo,
 \label{nonalgebra}
\end{equation}
which is harmless in our applications.

{
We now record an analogue of the identity \eqref{Zcommbox}. For this it will
be helpful to introduce two classes of symbols that capture the behavior of
some of the coefficients we encounter. We say a smooth function $a$ is a
``strong'' symbol if for all $j$ it satisfies
\begin{equation}
 (1+v)^j |\pa^j a|\leq C_j,
 \label{strongsymbol}
\end{equation}
for constants $C_j$. Note that this is stronger than the condition \eqref{minksymb}
since for example it requires that $(1+v)|\pa_u a|\lesssim 1$ as opposed to
just $(1+|u|) |\pa_u a|\lesssim 1$. Note that if $a = a(x/|x|)$ is smooth
it satisfies \eqref{strongsymbol} in the region $r \sim t$, which is the
region we will be concerned with in this section.

We say a smooth function $b$ is a ``weak'' symbol if, {in the region $|u|\lesssim s^{1/2}$},  for all multi-indices
$I$ it satisfies
\begin{equation}
 |Z_{\mB}^I b|\leq C_I,
 \label{weaksymbol}
\end{equation}
for constants $C_I$.
Note that while a bounded smooth function
$f(u)$ is not a
strong symbol due to the growth of $(s\pa_u)^k f(u)$, the function $u/s$ is a weak symbol.
 However, {since in region $|u|\lesssim s^{1/2}$ the function $|u/s|\lesssim s^{-1/2}$}, this results in a loss of $s^{-1/2}$ in various estimates, including the ones below.

With these definitions, we can write \eqref{basicB1} and \eqref{basicBangular}
in the form
\begin{equation}
 [Z_{\mB}, \evmB] = c_{Z_{\mB}}\evmB +  \frac{1}{(1+v)(1+s)}b_{Z_{\mB}} \pa_u,
 \qquad
 [Z_{\mB}, \sDelta] = a_{Z_{\mB}} \sDelta,
 \label{abstractmBcomms}
\end{equation}
where the $c_{Z_{\mB}}$ are constants $(c_{Z_{\mB}} = -1$ if $Z_{\mB} = \sBt$ and
zero otherwise), the $b_{Z_{\mB}}$ are weak symbols \eqref{weaksymbol}
and the $a_{Z_{\mB}}$ are strong symbols \eqref{strongsymbol}. {In fact,
\begin{equation}
  \label{aformulas}
  a_{\sBo} =\frac{s}{r} ,\qquad a_{\sBt} =-\left( 2 + \frac{u}{r} \right),
  \qquad
  a_{\Omega_{ij}} = 0,
\end{equation}}
{so we can write the second identity in \eqref{abstractmBcomms} in the form
\begin{equation}
  \label{}
  [Z_{\mB}, \sDelta] = a_{Z_{\mB}} \sDelta = \left(\slashed{a}_{Z_{\mB}} + \frac{s}{r} d_{Z_{\mB}} \right) 
  \sDelta
\end{equation}
where $\slashed{a}_{v\pa_v} = -2$ and $\slashed{a}_{Z_{\mB}} = 0$ otherwise, and where $d_{Z_{\mB}}$ are
weak symbols \eqref{weaksymbol}. Here we used that $\frac{u}{s}$ is a weak symbol.

} 
If we introduce $\widetilde{Z_{\mB} } = Z_{\mB} - c_{Z_{\mB} }$, then the first identity
reads
\begin{equation}
  \label{wtzmBcomm}
  [\widetilde{Z}_{\mB} , \evmB] =  \frac{1}{(1+v)(1+s)}b_{Z_{\mB}} \pa_u.
\end{equation}

 Noting that
all our fields commute with $n = \pa_u$, the first identity in
\eqref{abstractmBcomms} above then implies
 \begin{equation}
     [\widetilde{Z_{\mB}}, n\evmB] q = \pa_u P_{Z_{\mB}}[q],
	\qquad P_{Z_{\mB}}[q] =  \frac{1}{(1+v)(1+s)}b_{Z_{\mB}} \pa_uq.
  \label{}
 \end{equation}
 If we write
$\Box_{\mB} = -4n\evmB + \sDelta$ then taking advantage of the second identity
in \eqref{abstractmBcomms} we find
\begin{equation}
    \widetilde{Z_{\mB}} \Box_{\mB} q
    =  \Box_{\mB} \widetilde{Z_{\mB}}q + \pa_u P_{Z_{\mB}}[q]
 +( a_{Z_{\mB}} - c_{Z_{\mB}})  \sDelta q,
 \qquad P_{Z_{\mB}}[q] = \frac{1}{(1+v)(1+s)}b_{Z_{\mB}} \pa_uq,
 \label{BoxmBcom0}
\end{equation}
where we have relabelled the weak symbols $b_{Z_{\mB}}$ to absorb the harmless 
multiplicative constant $-4$.

{For some of our applications
it will be enough to use that $a_{Z_{\mB}} - c_{Z_{\mB}}$ is a strong symbol, but
in other places
we will need to record a more explicit version of this formula when $Z_{\mB} = \sBt = v\pa_v$.
In that case, $a_{Z_{\mB}} = \slashed{a}_{Z_{\mB}} - c_{Z_{\mB}} + \frac{s}{r} d_{Z_{\mB}} 
= -1 + \frac{s}{r} d_{Z_{\mB}}$ since $\slashed{a}_{v\pa_v} = -2$ and $c_{v\pa_v} = -1$.
Then \eqref{BoxmBcom0} reads
\begin{equation}
  \label{}
  \widetilde{v\pa_v} \Box_{\mB} q  = \Box_{\mB}\widetilde{ v\pa_v} q - \sDelta q 
  + \pa_u P_{v\pa_v}[q] + F_{v\pa_v}[q],
  \qquad \text{ where } F_{v\pa_v}[q] = \frac{s}{r} d_{v\pa_v} \sDelta q.
\end{equation}
We will see that the term $-\sDelta q$ will contribute a positive-definite term to our energy
estimates.
More generally, using \eqref{aformulas}, along with the fact that $c_{s\pa_u} = c_{\Omega_{ij}} = 0$,
we can write \eqref{BoxmBcom0} in the form
\begin{equation}
  \label{BoxmBcom1}
  \widetilde{Z_{\mB}} \Box_{\mB} q = \Box_{\mB} \widetilde{Z_{\mB}} q
  - a'_{Z_{\mB}} \sDelta q + \pa_u P_{Z_{\mB}}[q] + F_{Z_{\mB}}[q],
\end{equation}
where $P_{Z_{\mB}}[q]$ are as in \eqref{BoxmBcom0} and where
\begin{equation}
  \label{}
  a'_{v\pa_v} = 1, \quad a_{s\pa_u}' = a_{\Omega_{ij}}' = 0,
  \qquad F_{Z_{\mB}}[q] = \frac{s}{r} d_{Z_{\mB}}\sDelta q,
\end{equation}
for weak symbols  $d_{Z_{\mB}}$.}

We now get a higher-order version of the identity \eqref{BoxmBcom1}.

\begin{lemma}\label{boxmBcommutator}
    Define $\widetilde{Z_{\mB}} = Z_{\mB} + c_{Z_{\mB}}$ 
    where $c_{\sBo} = c_{\Omega} = 0$ and $c_{\sBt} = -1$,
    as in the above. Let $X^k$ denote
 an arbitrary $k$-fold product of the radial vector fields $X \in \{\Sbo, \Sbt\}$.
 If $Z_{\mB}^I = X^k \Omega^K$ then with $\Box_{\mB} = -4n\evmB +\sDelta$,
 \begin{equation}
  \widetilde{Z_{\mB}^I} \Box_{\mB} q = \Box_{\mB} \widetilde{Z_{\mB}^I} q
	+ \pa_u P_{\mB, I}[q] + F_{\mB, I}[q],
  \label{comm1formula0}
 \end{equation}
 where the above quantities are
 \begin{equation}
  P_{\mB, I}[q]
	= \frac{1}{(1+v)(1+s)} \sum_{|J| \leq |I|-1} b_J^I \pa_u Z_{\mB}^Jq,
  \qquad
  F_{\mB, I}[q] = \sum_{j \leq k-1} \sum_{|J| \leq |K|} a_{j}^k\sDelta X^j \Omega^Jq,
  \label{fmBIformula}
 \end{equation}
 where the coefficients $b_J^I$ satisfy the weak symbol condition \eqref{weaksymbol}
 and the coefficients $a_j^k$ satisfy the strong symbol condition \eqref{strongsymbol}.
 The last sum is over all $j$-fold products of the fields $X$ with the convention
 that the sum vanishes if $k = 0$.

 {Moreover, we have the identity
 \begin{equation}
  \label{}
    \widetilde{Z_{\mB}^I} \Box_{\mB} q = \Box_{\mB} \widetilde{Z_{\mB}^I} q
    + \pa_u P_{\mB, I}[q] + F_{\mB, I}^1[q] + F_{\mB, I}^2[q]
\end{equation}
where $P_{\mB, I}[q]$ is as above and where the quantities $F_{\mB, I}^1[q], F_{\mB, I}^{2}[q]$
are as follows. First,
\begin{equation}
  \label{fmBIformula1}
  F_{\mB, I}^1[q] = \sum_{|J_1| + |J_2| = |I|-1}
  -a_I^{J_1 J_2} \sDelta Z_{\mB}^{J_1}Z_{\mB}^{J_2}q,
\end{equation}
where the coefficients $a_I^{J_1 J_2} = 1$ if $Z^I_{\mB} = Z^{J_1}_{\mB} (v\pa_v) Z^{J_2}_{\mB}$
and $a_I^{J_1 J_2} = 0$ otherwise (so that $a_I^{J_1J_2} \equiv 0$ if there are no 
factors of $v\pa_v$ present in $Z_{\mB}^I$). The term
$F_{\mB, I}^2[q]$ is given by
\begin{equation}
  \label{fmBIformula2}
  F_{\mB, I}^2[q] = \frac{s}{r}  \sum_{|K| \leq |I|-1} d_{IK} \sDelta Z_{\mB}^K q
  + \sum_{|K|\leq |I|-2} d_{IK}' \sDelta Z_{\mB}^K q
\end{equation}
for weak symbols $d_{IJ}, d_{IJ}' $.
In particular,
\begin{equation}
  \label{fmBI2bd}
  |F_{\mB, I}^2[q]| \lesssim \frac{1+s}{(1+v)^2} \sum_{|J| \leq |I|} |\nas Z_{\mB}^J q|
  + \frac{1}{(1+v)^2} \sum_{|J| \leq |I|-1} |\Omega Z_{\mB}^J q|.
\end{equation}

}
 \end{lemma}
 {
 \begin{remark}
     \label{explanationofF1F2}
        The term $F_{\mB, I}$ is too large to treat as an error term in
        our estimates. However, after integrating by parts twice,
        it contributes a postive-definite
        term to our energy estimates and can then be safely ignored. This
        is a consequence of the following observations.
        For our applications we will
        need to handle the product $F_{\mB, I}^1 v\pa_v Z_{\mB}^I q$.
        The coefficients $a_I^{J_1J_2}$ in the definition of $F_{\mB, I}^1$ are such that
        \begin{equation}
            a_I^{J_1J_2} \left(\sDelta Z_{\mB}^{J_1} Z_{\mB}^{J_2}q\right) v\pa_v\left(
            Z_{\mB}^I q
            \right)
            =
            a_I^{J_1J_2} \left(\sDelta Z_{\mB}^{J_1} Z_{\mB}^{J_2}q\right)
           v\pa_v \left( Z_{\mB}^{J_1} (v\pa_v) Z_{\mB}^{J_2} q\right),
        \end{equation}
        for some $J_1, J_2$.
        Ignoring the commutator $[Z_{\mB}^{J_1}, v\pa_v]$ and writing $Z_{\mB}^{J_1}Z_{\mB}^{J_2}
        = Z_{\mB}^J$, this says
        \begin{equation}
          \label{}
             a_I^{J_1J_2} \left(\sDelta Z_{\mB}^{J} q\right) v\pa_v
                \left(Z_{\mB}^I q\right)
            = a_I^{J_1J_2}\left(\sDelta Z_{\mB}^J q\right) \left((v\pa_v)^2 Z_{\mB}^J q\right).
        \end{equation}
        This can be handled by integrating by parts in the angular direction and then in the $v$-direction;
        this generates lower-order boundary terms and bulk terms, as well as a highest-order bulk term,
        \begin{equation}
          \label{}
          a_I^{J_1J_2} |\nas (v\pa_v) Z_{\mB}^J q|^2,
        \end{equation}
        and the crucial point is that this enters our energy identities with a favorable sign. See
        Lemma \ref{FImBibpargument} and in particular \eqref{crucialIBP}.
    \end{remark}
}
 
\begin{proof}
	When $|I| = 1$, the result follows from \eqref{BoxmBcom0}, respectively \eqref{BoxmBcom1},
    after using that $a_{Z_{\mB}} - c_{Z_{\mB}}$ is a strong symbol to get \eqref{comm1formula0}.

	To get the identity \eqref{comm1formula0} for larger $|I|$, we use induction and write
    $[\widetilde{Z_{\mB}} \widetilde{Z_{\mB}^I},\Box_{\mB}] =
    \widetilde{Z_{\mB}}[\widetilde{Z_{\mB}^I},\Box_{\mB}]
	+ [\widetilde{Z_{\mB}}, \Box_{\mB}] \widetilde{Z_{\mB}^I}$ and then, by
    \eqref{BoxmBcom0},
	\begin{multline}
        \label{splitcomm}
	 \widetilde{Z_{\mB}}[\widetilde{Z_{\mB}^I},\Box_{\mB}]
 	+ [\widetilde{Z_{\mB}}, \Box_{\mB}] \widetilde{Z_{\mB}}^I\\
	=
	\pa_u \left( Z_{\mB} P_{\mB, I}[q] + P_{Z_{\mB}}[Z_{\mB}^I q]\right) + Z_{\mB} F_{\mB, I}[q]
	+
	 F_{ Z_{\mB}}[Z_{\mB}^I q],
	\end{multline}
	where we used that all our fields commute with $\pa_u$. The first two
	terms are of the correct form
    for \eqref{comm1formula0}. 
    As for the last two terms, we write
	\begin{align}
	 Z_{\mB} F_{\mB, I}[q]
	 &= \sum_{j \leq k-1} \sum_{|J| \leq |K|} a_j^k \sDelta Z_{\mB} X^j \Omega^J q
	 + \sum_{j \leq k-1}\sum_{|J| \leq |K|} (Z_{\mB}a_j^k + a_j^k a_{Z_{\mB}})\sDelta X^j \Omega^J q,\\
	 F_{ Z_{\mB}}[Z_{\mB}^I q]&=
	 a_{Z_{\mB}} \sDelta  Z_{\mB}^I q,
	 \label{}
	\end{align}
	which are both of the form appearing in \eqref{fmBIformula}.

    {
    To get the formula \eqref{fmBIformula2}, we argue in the same way except
that we use \eqref{BoxmBcom1} in place of \eqref{BoxmBcom0}. The difference is just that we
have the following terms contributed into \eqref{splitcomm}
by the term $F_{Z_{\mB}}[\widetilde{Z}_{\mB}^I]$
and $\widetilde{Z}_{\mB} F^1_{\mB, I}$,
\begin{equation}
  \label{higherorderF2}
  \frac{s}{r} d_{Z_{\mB}}\sDelta \widetilde{Z_{\mB}}^I q +\widetilde{Z_{\mB}}
  \left(\frac{s}{r} \sum_{|K| \leq |I|-1}d_{IK} \sDelta Z_{\mB}^K q 
  + \sum_{|K| \leq |I|-2} d_{IK}'\sDelta Z_{\mB}^K q\right),
\end{equation}
and the contribution from the terms $F_{\mB, I}^2$ and the term $-a_{Z_{\mB}}\sDelta \widetilde{Z_{\mB}}^I$
from \eqref{BoxmBcom1}, which are
\begin{multline}
  \label{higherorderF1}
   \sum_{|J_1| + |J_2| = |I|-1} -a_I^{J_1 J_2} \widetilde{Z_{\mB}}\sDelta Z_{\mB}^{J_1} Z_{\mB}^{J_2}q
  - a_{Z_{\mB}}'\sDelta \widetilde{Z_{\mB}}^Iq
  \\
  = 
  \sum_{|J_1| + |J_2| = |I|-1} -a_I^{J_1 J_2} \sDelta Z_{\mB} Z_{\mB}^{J_1} Z_{\mB}^{J_2}q
  -a_{Z_{\mB} }' \sDelta \widetilde{Z_{\mB} }^I q
  \\
  + 
  \sum_{|J_1| + |J_2| = |I|-1} -(c_{Z_{\mB}} + a_{Z_{\mB}}) a_I^{J_1 J_2} \sDelta Z_{\mB}^{J_1} Z_{\mB}^{J_2}q,
\end{multline}
where we used the definition $\widetilde{Z_{\mB}} = Z_{\mB} + c_{Z\mB}$ and the second identity
in \eqref{abstractmBcomms} to commute $Z_{\mB}$ with $\sDelta$ in the second step.

To handle \eqref{higherorderF2}, we use the fact that
$Z_{\mB} \frac{s}{r} = d'_{Z_{\mB}} \frac{s}{r}$ for a weak symbol $d'_{Z_{\mB}}$,
along with the second identity in \eqref{abstractmBcomms} to commute with $\sDelta$ in the second
and third
terms, which shows that all the terms in \eqref{higherorderF2} are of the form appearing in \eqref{fmBIformula2}.

We now consider \eqref{higherorderF1}.
The terms on the second line is of the form appearing in the second term in \eqref{fmBIformula2}, 
since they involve lower-order terms. The first and second terms are of the form appearing in \eqref{fmBIformula1},
since if $Z_{\mB} = \sBo$ or $\Omega_{ij}$, the first term accounts for all decompositions
of $Z_{\mB} Z_{\mB}^I$ into products of the form $Z_{\mB}^{J_1}(v\pa_v)Z_{\mB}^{J_2}$. 
If $Z_{\mB} = \sBt = v\pa_v$, the second term accounts for the additional decomposition
$Z_{\mB} Z_{\mB}^I = (v\pa_v) Z_{\mB}^{J_1} Z_{\mB}^{J_2}$ (recall that $a_{v\pa_v}' = 1$).
The bound \eqref{fmBI2bd}
follows immediately from \eqref{fmBIformula2} after expressing $\sDelta$ in terms of the
rotation fields using \eqref{sdeltaformulaOmega}.
    }
\end{proof}

To commute with the nonlinear terms in our equation, we will need to commute our operators
$Z_{\mB}$ with differentiation in rectangular coordinates. For this it is helpful
to record the following elementary identities,
\begin{equation}
	\pa_t =  \pa_u +  \pa_v,
	\qquad
	\pa_i = \omega_i \pa_v - \omega_i \pa_u + \frac{\omega^j}{r} \Omega_{ji},
	\qquad \omega = x/|x|,
 \label{alphaintonullandrotation}
\end{equation}
so in particular for $\alpha \in \{0,1,2,3\}$ we can write
\begin{equation}
 \pa_\alpha = c_\alpha^u(\omega) \pa_u + c_\alpha^v(\omega) \pa_v +
 \frac{1}{r}c^{ij}_{\alpha}(\omega) \Omega_{ij},
 \label{alphaintonullandrotation2}
\end{equation}
where the $c_\alpha$ are smooth functions on the sphere.
We emphasize the fact that apart from the factor of $1/r$ in front of
$\Omega$ in \eqref{alphaintonullandrotation2}, none of the above coefficients depend on $u$,
which will simplify many formulas going forward.
We will also use the facts that
\begin{equation}
 [Z_{\mB}, \pa_u] = 0,
 \qquad
 [\Sbo, \Omega] = [\sBt, \Omega] = 0,
 \qquad
 [\sBo, \pa_v] = -\frac{1}{v}\pa_u,
 \qquad
 [\sBt, \pa_v] = -\pa_v,
 \label{comm1}
\end{equation}
and we will also rely on the algebra property \eqref{liealgebra} of the rotation
fields.
Finally, we will use the following simple facts,
\begin{equation}
 \pa_u \omega= \pa_v \omega = \sBo \omega = \sBt \omega = 0,
 \qquad
 \Omega_{ij}r = 0,
 \qquad
 \Omega_{ij} \omega_k = c_{ijk}^\ell \omega_\ell,
 \label{framecoeffs}
\end{equation}
for constants $c_{ijk}^\ell$.

Using \eqref{alphaintonullandrotation2}
and the above facts, we find that
\begin{align}
 [\sBo,\pa_\alpha ]q &= [\sBo,  c^u_\alpha(\omega) \pa_u]q
 + [\sBo,  c^v_\alpha(\omega) \pa_v]q
 + [\sBo,  \frac{1}{r}c^{ij}_\alpha(\omega) \Omega_{ij}]q
 = -\frac{1}{v} c^v_\alpha(\omega) \pa_u q
 + \frac{s}{2r} \frac{1}{r} c_{\alpha}^{ ij}(\omega) \Omega_{ij}q,\label{sbopaalpha}\\
 [\sBt,\pa_\alpha ]q &= [\sBt,  c^u_\alpha(\omega) \pa_u]q
 + [\sBt,  c^v_\alpha(\omega) \pa_v]q
 + [\Sbt, \frac{1}{r}c^{ij}_\alpha(\omega)\Omega_{ij}]q
 = - c^v_\alpha(\omega)\pa_v q
 - \frac{v}{2r} \frac{1}{r} c_{\alpha}^{ ij}(\omega) \Omega_{ij}q,
 \\
 [\Omega, \pa_\alpha] q &= d^u_{\alpha ij}(\omega) \pa_u + d^v_{\alpha ij}(\omega)\pa_v
 + \frac{1}{r}d_{\alpha ij}^{i'j'}(\omega) \Omega_{i'j'},
 \label{omegapaalpha}
\end{align}
where the $c_\alpha, d_\alpha$ are smooth functions on the sphere.
Writing $\pa_u = \tfrac{1}{2}\pa_t - \tfrac{1}{2} \omega^i \pa_i$ and
$\pa_v = \tfrac{1}{2} \pa_t + \tfrac{1}{2}\omega^i \pa_i$, the above can all be
written in the form
\begin{equation}
 [Z_{\mB}, \pa_\alpha] q = c_{Z_{\mB} \alpha}^\beta \pa_\beta q
 + \frac{1}{1+v} b_{Z_{\mB} \alpha}^{ij}\Omega_{ij} q,
 \label{paalphazmB}
\end{equation}
where the coefficients $c, b$ satisfy the strong symbol condition
 \eqref{strongsymbol} in the region $|u| \lesssim s^{1/2}$.
 We note that this can be written in the alternate form
 \begin{equation}
  [Z_{\mB}, \pa_\alpha] q = \pa_\beta \left( c_{Z_{\mB} \alpha}^\beta q\right)
	+ \frac{1}{1+v} c'_{Z_{\mB} \alpha} q + \frac{1}{1+v}  b_{Z_{\mB} \alpha}^{ij}\Omega_{ij} q,
  \label{paalphazmB2}
 \end{equation}
 where again the coefficients satisfy the strong symbol condition
  \eqref{strongsymbol} in the region $|u| \lesssim s^{1/2}$. Here we have
	used that for any strong symbol $c$ we have $\pa c  = (1+v)^{-1} c'$
	for another strong symbol $c'$.

	We will need a higher-order version of this formula.
\begin{lemma}
	\label{ZBdivcommvariantlem}
We have
	\begin{equation}
		[Z_{\mB}^I, \pa_\alpha] q =
		\sum_{|J| \leq |I|-1} c_{\alpha J}^{\beta I} \pa_\beta Z_{\mB}^J q +
		\frac{b_{\alpha J}^{Iij}}{1+v} \Omega Z_{\mB}^J q	 \label{ZBdivcommvariant}
	\end{equation}
	where the coefficients are smooth in the region $r \sim t$
	and satisfy the strong symbol condition \eqref{strongsymbol}.

	In particular, we can write
	\begin{equation}
	 		[Z_{\mB}^I, \pa_\alpha] q =
			\pa_\beta P_{\pa_\alpha, I}^\beta[q] + F_{\pa_\alpha, I}[q],
	 \label{ZBdivcommvariant2}
	\end{equation}
	where
	\begin{equation}
	 P_{\pa \beta, I}^\alpha[q] =
	 \sum_{|J| \leq |I|-1} c_{\alpha J}^{\beta I} Z_{\mB}^J q,
	 \qquad
	 F_{\pa_\alpha, I}[q]
	 = \frac{1}{1+v} \sum_{|J| \leq |I|} b_{\alpha J}^{ I} Z_{\mB}^J q,
	 \label{ZBdivcommvariant2struct}
	\end{equation}
	for strong symbols $c, b$.
\end{lemma}

\begin{proof}
When $|I| = 1$ this is just \eqref{paalphazmB}. If \eqref{ZBdivcommvariant}
holds for some $|I| \geq 1$ we write
$[Z_{\mB} Z_{\mB}^I, \pa_\alpha]q
= Z_{\mB} [Z_{\mB}^I,\pa_\alpha]q + [Z_{\mB}, \pa_\alpha] Z_{\mB}^I q$ and then
\begin{multline}
 Z_{\mB} [Z_{\mB}^I,\pa_\alpha]q + [Z_{\mB}, \pa_\alpha] Z_{\mB}^I q\\
 = Z_{\mB} \left( \sum_{|J| \leq |I|-1} c_{\alpha J}^{\beta I} \pa_\beta Z_{\mB}^J q +
 \frac{b_{\alpha J}^{Iij}}{1+v} \Omega Z_{\mB}^J q	\right)
 + c_{Z_{\mB} \alpha}^\beta \pa_\beta Z_{\mB}^I q
 + \frac{1}{1+v}  b_{Z_{\mB} \alpha}^{ij}\Omega_{ij}Z_{\mB}^I q.
 \label{}
\end{multline}
The second and third terms are of the correct form. To handle the terms in the
sum, we just use the facts that if $c$ is a strong symbol then so is $Z_{\mB} c$,
that $Z_{\mB}v = c_{Z_{\mB}} v$ for a strong symbol
$c_{Z_{\mB}}$, the commutator identities \eqref{comm1} and
the fact that the rotation fields form an algebra to handle the commutators
with $\Omega$, and the identity \eqref{paalphazmB} once more
to commute with $\pa_\beta$.

The identity \eqref{ZBdivcommvariant2} follows immediately from
\eqref{ZBdivcommvariant} and \eqref{paalphazmB2}.
\end{proof}

{
We now record a version of the above that we use to commute with
the quadratic nonlinearity in our equation.
\begin{lemma}
	\label{centralcommcurrent}
	Suppose that $\gamma = \gamma^{\alpha\beta\delta}$ satisfy $(1+v)|Z_{\mB}^J \gamma| \leq C_J$
	for any $J$, where all quantities are expressed relative
	to the usual rectangular coordinate system.
	With $\check{\gamma}^{\alpha\beta\delta}
	= \gamma^{\alpha\beta\delta} + \gamma^{\alpha\delta\beta}$, we have
 	\begin{align}
 	  \widetilde{Z}_{\mB}^I\pa_\alpha(\gamma^{\alpha\beta\delta}\pa_\beta q
 		\pa_{\delta} q)
 	=
  	 \pa_\alpha(\check{\gamma}^{\alpha\beta\delta}\pa_{\beta}q \pa_{\delta}\widetilde{Z}_{\mB}^Iq)
		 + \pa_\alpha P_I^\alpha + F_{I},
		 \label{centralcommcurrentident0}
 	\end{align}
	where $\widetilde{Z}_{\mB}$ is defined in \eqref{BoxmBcom1}.
	The components of the current $P_I^\alpha$ expressed in rectangular
	coordinates and the remainder are given by
	\begin{equation}
	 P_I^\alpha = \frac{1}{1+v}
	 \sum_{{{|I_1| + |I_2| \leq |I| \atop |I_1|,|I_2|\le |I|-1}}}
	a^{\alpha\beta\delta} \pa_\beta Z_{\mB}^{I_1}q \pa_\delta Z_{\mB}^{I_2} q,
	\qquad
	F_I = \frac{1}{(1+v)^2} \sum_{{|I_1| + |I_2| \leq |I|}}
 a^{\alpha\beta\delta} \pa_\beta Z_{\mB}^{I_1}q \pa_\delta Z_{\mB}^{I_2} q
	 \label{centralcommcurrentident0formula}
	\end{equation}
	where the coefficients in the above are smooth functions
	satisfying the weak symbol condition \eqref{weaksymbol}.
\end{lemma}

\begin{proof}
By the identity \eqref{ZBdivcommvariant2} from Lemma \ref{ZBdivcommvariantlem}, we have
\begin{equation}
 Z_{\mB}^I \pa_\alpha (\gamma^{\alpha\beta \delta} \pa_\beta q \pa_\delta q)
 = \pa_\alpha \left( Z_{\mB}^I \left( \gamma^{\alpha\beta \delta} \pa_\beta q \pa_\delta q\right)\right)
 + \pa_\beta P_{\pa_\alpha, I}^\beta [ \gamma^{\alpha\beta \delta} \pa_\beta q \pa_\delta q]
 + F_{\pa_\alpha, I}[\gamma^{\alpha\beta \delta} \pa_\beta q \pa_\delta q],
\end{equation}
where the last two terms are as in \eqref{ZBdivcommvariant2struct}.
The quantity $Z_{\mB}^I \left( \gamma^{\alpha\beta \delta} \pa_\beta q \pa_\delta q\right)$
and the quantity $P_{\pa_\alpha, I}^\beta [ \gamma^{\alpha\beta \delta} \pa_\beta q \pa_\delta q]$
are sums of terms of the form
\begin{multline}
 Z_{\mB}^K (\gamma^{\alpha \beta \delta} \pa_\beta q \pa_\delta q)
 = (\gamma^{\alpha \beta \delta} + \gamma^{\alpha \delta \beta}) \pa_\beta q Z_{\mB}^K \pa_\delta q\\
 + \sum_{\substack{|K_1| + |K_2| + |K_3| \leq |K|, \\ |K_2|, |K_3| \leq |K|-1}}
 \left(Z_{\mB}^{K_1} \gamma^{\alpha \beta \delta}\right)
 \left(Z_{\mB}^{K_2}\pa_\beta q\right)
 \left(Z_{\mB}^{K_3} \pa_\delta q\right).
 \label{ffs10}
\end{multline}
To conclude we need to commute our vector fields with $\pa$ once
more.
Using \eqref{ZBdivcommvariant2struct} again, we have
\begin{equation}
 \left(Z_{\mB}^{K_2}\pa_\beta q\right)
 \left(Z_{\mB}^{K_3} \pa_\delta q\right)
 = \left(\pa_\beta Z_{\mB}^{K_2} q + \pa_{\beta'}P_{\pa_\beta, I}^{\beta'}[q] + F_{\pa_{\beta}, I}[q] \right)
 \left(\pa_\delta Z_{\mB}^{K_2} q + \pa_{\delta'}P_{\pa_\delta, I}^{\delta'}[q] + F_{\pa_{\delta}, I}[q]\right).
 \label{}
\end{equation}
Inserting this formula into \eqref{ffs10} then gives the result.
\end{proof}}

To handle some of the boundary terms along the timelike sides of
the shocks, it will be important to relate the vector fields from
$\mZB$ to those in $\mathcal{Z}_m$.
\begin{lemma}\label{ZmBtoZlem}
 The vector fields $\mZB$ and $\mathcal{Z}_m$ satisfy the following
 properties. First, there are smooth functions $c_{Z Z_{\mB}},c_{Z Z_{\mB}}'$ satisfying
 the symbol condition \eqref{weaksymbol} so that
 \begin{equation}
	 Z = \sum_{Z_{\mB} \in \mZB } \left(c_{ZZ_{\mB}}  + c_{ZZ_{\mB}}'\frac{u}{s} \right) Z_{\mB}.
  \label{ZtoZmB}
 \end{equation}
 As a consequence, we have the following bounds in the region $|u| \lesssim s^{1/2}$,
\begin{equation}
 |Z^I q| \lesssim \sum_{|J|\leq |I|} |Z_{\mB}^J q|,
 \label{ZtoZmBbd0}
\end{equation}
and for any $Z_{\mB}$ in $\mZB$,
 \begin{equation}
  |Z_{\mB} Z^J q| \lesssim \sum_{|K| \leq |J|+1} |Z_{\mB}^K q|  \label{ZtoZmBbd1}
 \end{equation}
\end{lemma}
\begin{proof}
The identity \eqref{ZtoZmB} follows after using the identity
\eqref{Zintonull} and then the fact that $s\pa_u$ and $v \pa_v$ are
in $\mZB$. The bounds \eqref{ZtoZmBbd0} and \eqref{ZtoZmBbd1} follow after repeatedly using
the identity \eqref{ZtoZmB} and the fact that $Z_{\mB} \frac{u}{s}
= a_{Z_{\mB}}\left(1 +  \frac{u}{s}\right)$ for symbols $a_{Z_{\mB}}$.
\end{proof}
Finally, we record an identity for the commutators $[\evmB, Z_{\mB}^I]$
which we will also use along the shocks.
\begin{lemma}
	Let $X^k$ denote an arbitrary $k$-fold product of the fields $X \in\{\sBo, \sBt\}$.
	Then
 	\begin{equation}
 		[X^k \Omega^K,\evmB]q
 		= \sum_{j \leq k -1} c^{k}_{j} \evmB X^j \Omega^K q + \frac{1}{(1+v)(1+s)}
 		a^{k}_{j} \pa_u X^j\Omega^K q,
 			 \label{}
 	\end{equation}
 	where the sum is taken over all $j$-fold products of the fields $\Sbo, \Sbt$
 	with $j \leq k-1$. In the above, the $c^k_j$ are constants and the $a^k_j$
	are weak symbols \eqref{weaksymbol}. In particular,
	\begin{equation}
	 |[X^k \Omega^K,\evmB]q|
	  \lesssim \sum_{j \leq k-1} |\evmB X^j \Omega^K q| + \frac{1}{(1+v)(1+s)}|\pa_u X^j\Omega^K q|
	 \label{evmbZcomm}
	\end{equation}
\end{lemma}
\begin{proof}
 This follows from repeated application of the first identity in \eqref{abstractmBcomms}
 along with the fact that $\evmB$ and the fields $X$ commute with the rotations
 $\Omega$.
\end{proof}

\section{Derivation of the wave equation for the potential}
\label{derivation1}

We assume that $\rho$
is given in terms of the density by a given equation of state $p = P(\rho)$.
We will assume that the equation of state satisfies $P', P'' > 0$ and $p \in C^{\infty}(\R\setminus \{0\}).$
The enthalpy $w = w(\rho)$ is 
\begin{equation}
 w(\rho) = \int_1^{\rho} \frac{P'(\lambda)}{\lambda} d\lambda.
 \label{enthdef}
\end{equation}
From Bernoulli's equation \eqref{introbern},
$w$ is related to $\pa \Phi$ by
\begin{equation}
 w(\rho) = -\pa_t \Phi - \frac{1}{2} |\nabla_x \Phi|^2,
 \label{bernapp}
\end{equation}
Since $P' >0$ it follows that $\rho \mapsto w(\rho)$ is an invertible function,
which we denote $\rho = \rho(w)$. We then think of \eqref{bernapp}
as determining the density $\rho$ from $\pa \Phi$, and we define
$\varrho$ by
$\varrho = \varrho(\pa\Phi) = \rho(w(\pa\Phi))$. {Note that $\varrho(0) = \rho(0) = 1$
since $w|_{\rho = 1}  =0 $ by \eqref{enthdef}.} We record that for the ``polytropic''
 equation of
state $P(\rho) = \rho^\gamma$ with $\gamma > 1$, we have
\begin{equation}
 w(\rho) = \int_1^\rho \gamma \lambda^{\gamma-2}\, d\lambda
 = \frac{\gamma}{\gamma -1} \left(\rho^{\gamma-1} - 1\right),
 \label{}
\end{equation}
so
\begin{equation}
 \rho(w) = \left( \frac{\gamma-1}{\gamma} w + 1\right)^{1/(\gamma-1)}.
 \label{}
\end{equation}

With the above notation, define
\begin{equation}
 H^0(\pa\Phi) = \varrho(\pa \Phi),
 \qquad
 H^i(\pa\Phi) = \varrho(\pa\Phi) \nabla^i\Phi,
 \label{Hdef}
\end{equation}
so the continuity equation takes the form
\begin{equation}
 \pa_\alpha H^\alpha(\pa \Phi) = 0,
 \label{contderiv}
\end{equation}
with $\pa_\alpha = \pa_{x^\alpha}$ where $x^\alpha$ denote
rectangular coordinates on $\R^4$.

The jump conditions are
\begin{equation}
 [H^\alpha(\pa\Phi)\zeta_\alpha] = 0, \qquad
 [\Phi] = 0,
 \label{jumpapp}
\end{equation}
where $\zeta$ is any non-vanishing one-form whose nullspace at each point
$(t,x)$ is the tangent space to $\Gamma$ at $(t,x)$.

\subsection{The continuity equation}
The purpose of this section is to write the equation \eqref{contderiv}
as a wave equation under mild assumptions on the equation of state.
In this section, we use $\xi$ to denote points in the cotangent space
$T^*_{(t,x)} \mathbb{R}^4$. Let $H^\alpha(\xi)$ denote the components of $H$
expressed in rectangular coordinates. We write
\begin{equation}
 H^{\alpha\beta}(\xi) = \pa_{\xi^\beta} H^{\alpha}(\xi),
 \qquad
 H^{\alpha\beta\delta}(\xi) = \pa_{\xi^\delta}\pa_{\xi^\beta} H^{\alpha}(\xi),
 \qquad
 H^{\alpha\beta\delta\rho}(\xi) = \pa_{\xi^\rho}\pa_{\xi^\delta}\pa_{\xi^\beta} H^{\alpha}(\xi).
 \label{xinotation}
\end{equation}
We note that quantities such as $H^{\alpha\beta}(0)\xi_\alpha\xi_\beta$ are invariant
under coordinate changes of $\mathbb{R}^4$, a fact which will be used repeatedly
in what follows. That is, the quantities $H^{\alpha_1\cdots \alpha_k}$
are well-defined tensor fields.

We compute
\begin{equation}
 H^{\alpha\beta}(0)\xi_\alpha \xi_\beta =
 {-\rho'(0)} \xi_0^2 + \varrho(0) \delta^{ij}\xi_i\xi_j
 {= -\frac{1}{p'(1)} \xi_0^2 + \delta^{ij}\xi_i\xi_j,}
 \label{}
\end{equation}
{where we used that $\rho(0) = 1$ and that $\rho'(0) = \rho(0)/p'(1)$.}
and so after an appropriate rescaling of $(t, x)$, we can take
\begin{equation}
 H^{\alpha\beta}(0) = m^{\alpha\beta},
 \label{}
\end{equation}
with $m^{\alpha\beta}$ the components of the inverse of the usual Minkowski metric,
\begin{equation}
 m^{00} = -1, \qquad
 m^{11} = m^{22} = m^{33} = 1.
 \label{}
\end{equation}
{We note at this point that with this choice of units the sound
speed is one at $\rho = 1$,
\begin{equation}
	p'(1) = 1.
 \label{metric}
\end{equation}
}

For each $\alpha = 0,1,2,3,$ we have
\begin{equation}
 H^\alpha(\xi) - H^\alpha(0) =
H^{\alpha\beta}(0)\xi_\beta
+ G^{\alpha\beta\delta}(\xi)\xi_\beta \xi_\delta,
 \label{Hexpansion}
\end{equation}
where
\begin{equation}
 G^{\alpha\beta\delta}(\xi) = \int_0^1 (1-t) H^{\alpha\beta\delta}(t\xi)\, dt.
 \label{Gdefapp}
\end{equation}
Later on it will be convenient to also use the notation
\begin{equation}
 j^\alpha(\xi) = G^{\alpha\beta\delta}(\xi)\xi_\beta \xi_\delta,
 \qquad
 j^{\alpha\beta}(\xi) = \pa_{\xi^\beta} j^\alpha(\xi).
 \label{gammadef}
\end{equation}
Then for each $\alpha$, $\xi\mapsto j^\alpha(\xi)$ is a smooth function
and $j^\alpha(0) = 0$.

With this notation, the continuity equation \eqref{contderiv} takes the form
\begin{equation}
 \Box \Phi + \pa_\alpha j^\alpha(\pa\Phi) = 0.
 \label{wave000}
\end{equation}

In what follows it will be helpful to keep track of the nonlinearity more
carefully.
Returning to \eqref{Hexpansion} we write
\begin{equation}
 A^{\alpha\beta\delta} = G^{\alpha\beta\delta}(0)
 =\pa_{\xi^\beta}\pa_{\xi^\delta} H^{\alpha}|_{\xi = 0}
 \qquad
 B^{\alpha\beta\delta}(\xi) = G^{\alpha\beta\delta}(\xi)-G^{\alpha\beta\delta}(0)
 = \int_0^1 \pa_{\xi^\kappa} G^{\alpha\beta\delta}(t\xi) \xi_\kappa dt,
 \label{}
\end{equation}
and then the continuity equation becomes
\begin{equation}
 \Box \Phi + \pa_\alpha (A^{\alpha \beta\delta} \pa_\beta \Phi\pa_\delta \Phi)
 + \pa_\alpha B^\alpha(\pa \Phi) = 0,
 \label{modela}
\end{equation}
with $B^\alpha(\pa \Phi) = B^{\alpha\beta\delta}(\pa\Phi)\pa_\beta\Phi\pa_\delta \Phi$
is a cubic nonlinearity
and where the $A^{\alpha\beta\delta}$ are constants.

We introduce the notation
\begin{equation}
  A^{u u u}(\omega) = A^{\alpha\beta\delta} \pa_\alpha u \pa_\beta u \pa_\delta u,
 \label{}
\end{equation}
as well as
\begin{equation}
 \widetilde{A}^{\alpha\beta\delta}(\omega) = A^{\alpha\beta\delta}
 - \delta^{\alpha u}\delta^{\beta u}\delta^{\delta u} A^{uuu}(\omega),
 \label{}
\end{equation}
{where we are abusing notation slightly and writing}
\begin{equation}
 \delta^{\alpha u} = \delta^{\alpha 0} - \delta^{\alpha i}\omega_i
 = \delta^{\alpha \alpha'}\pa_\alpha' u.
 \label{}
\end{equation}
Then $A^{uuu}(\omega)$ corresponds to the $(u,u,u)$ component of $A$
expressed relative to the null coordinate system defined in
\eqref{minknull0}.
{Noting that
\begin{equation}
 \pa_\alpha \delta^{\alpha u} = \pa_\alpha \delta^{\alpha 0} - \pa_\alpha (\delta^{\alpha i} \omega_i)
 = - \frac{2}{r},
 \label{}
\end{equation}}
The equation \eqref{wave000} takes the form
\begin{equation}
 \Box \Phi + \pa_u(A^{uuu}(\omega) (\pa_u\Phi)^2)
 +{\pa_\alpha (\widetilde{A}^{\alpha\beta\delta} \pa_\beta \Phi\pa_\delta\Phi)}
 + \pa_\alpha B^\alpha(\pa\Phi) = {\frac{2}{r} A^{uuu}(\omega) (\pa_u\Phi)^2},
 \label{almostunits}
\end{equation}
where $\opa$ denotes projection of $\pa$ away from the $u$-direction,
\begin{equation}
 \opa_0 =  \frac{1}{2} (\pa_t + \pa_r),
 \qquad
 \opa_i = \pas_i.
 \label{opadef}
\end{equation}

 The coefficients
$\widetilde{A}^{\alpha\beta\delta}$ are not constants but they satisfy
\begin{equation}
 r^\ell |\pa^\ell A|\lesssim 1,
 \label{}
\end{equation}
{and the nonlinearity $\pa_\alpha(\widetilde{A}^{\alpha\beta\delta}\pa_\beta\Phi\pa_\delta\Phi)$
verifies the classical null condition.}
We claim that the coefficient $A^{uuu}$ is actually a constant,
which is nonzero under a mild assumption on the equation of state.
To see this, we start by writing
\begin{equation*}
 H^{u}(\xi) = H^0(\xi) - \omega_iH^i(\xi)
 = \rho(w(\xi))(1- \xi_v + \xi_u), \quad
 \text{ where }
 w(\xi) 
 = - (\xi_v + \xi_u) - \frac{1}{2} (\xi_v- \xi_u)^2- \frac{1}{2}|\slashed{\xi}|^2.
\end{equation*}
From these formulas we compute 
\begin{equation}
 \pa_{\xi_u} w(\xi) =
 {-1-\xi_u+\xi_v ,
 \qquad \pa_{\xi_u}^2 w(\xi) = -1},
 \label{}
\end{equation}
so
\begin{equation}
  \pa_{{\xi_u}} H^{u}(\xi)
	{= \pa_{\xi_u} w(\xi) \rho'(w(\xi))(1-\xi_v+\xi_u) + \rho(w(\xi))}
 \label{}
\end{equation}
{
and
\begin{equation}
 \pa^2_{\xi_u} H^{u}(\xi)
 =  \pa_{\xi_u}^2 w(\xi) \rho'(w(\xi))(1-\xi_v+\xi_u)
 +  (\pa_{\xi_u} w(\xi))^2 \rho''(w(\xi))(1-\xi_v+\xi_u)
 +2 \pa_{\xi_u} w(\xi) \rho'(w(\xi)).
 \label{}
\end{equation}
It follows that
\begin{equation}
 2A^{uuu} = 2G^{uuu}(0) =  (\pa_{\xi_u}^2 H^u)(0)
 =  \rho''(0) - \rho(0),
 \label{}
\end{equation}
which is a constant. To determine when it is nonzero, we
use the fact that
$$
\rho'(w)=\frac  1{w'(\rho)}=\frac \rho{p'(\rho)}
$$
to express at $\rho=1$ (which is the same as $w = 0$, recall \eqref{enthdef})
$$
\rho''(0) - \rho'(0)=\frac{p'(1)-p{''}(1)-p'^2(1)}{p'^3(1)}
$$
With our choice of units (see \eqref{metric}), $p'(1) = 1$ and so $\rho''(0) - \rho'(0)$
is nonvanishing as long as the equation of state is convex at $\rho = 1$.

}
If we replace $\Phi$ with $-\frac{1}{A^{uuu}} \Phi$ and multiply the equation
by $-A^{uuu}$, this has the effect of replacing $A^{uuu}$ in the
expression \eqref{almostunits} by $-1$.
After performing this rescaling, \eqref{almostunits} becomes
\begin{equation}
 H^\alpha(\pa\Phi) -H^\alpha(0) = m^{\alpha\beta} \pa_\beta \Phi
 -\delta^{\alpha u} (\pa_u\Phi)^2
 + \widetilde{A}^{\alpha\beta\delta}\pa_\beta\Phi \pa_\delta \Phi
 + B^\alpha(\pa\Phi)
 = -\delta^{\alpha u} (\pa_u\Phi)^2
 + \widetilde{j}^\alpha(\pa\Phi),
 \label{}
\end{equation}
where
\begin{equation}
 \widetilde{j}^{\alpha}(\pa\Phi)
 =\widetilde{A}^{\alpha\beta\delta}\pa_\beta\Phi \pa_\delta \Phi
 + B^\alpha(\pa\Phi)
 \label{gammatilde}
\end{equation}
is such that $\pa_\alpha \widetilde{j}^\alpha$ verifies the classical
null condition.
In summary, we have the following formula for $H^\alpha$.
\begin{lemma}
 \label{basicstructure}
 {Suppose that the equation of state $p = p(\rho)$ satisfies
 $p''(1)\not=0$.}
 With $H$ defined as in \eqref{Hdef} and
 $\gamma$ as in \eqref{gammadef}, after an appropriate rescaling of the
 dependent and independent variables,
we have
\begin{equation}
 H^\alpha(\pa\Phi) - H^\alpha(0) = m^{\alpha\beta}\pa_\beta \Phi
 + j^\alpha(\pa\Phi),
 \label{basicHfact1}
\end{equation}
where
\begin{equation}
 j^\alpha(\pa \Phi) = A^{\alpha\beta\delta}\pa_\beta\Phi \pa_\delta \Phi
 + B^\alpha(\pa\Phi),
 \label{basicHfact2}
\end{equation}
where
$\xi \mapsto B^\alpha(\xi)$ vanishes to third order at $\xi = 0$
and where the $A^{\alpha\beta\delta}$ are constants.
The above terms have the following structure:
with $\delta^{u \alpha} = \tfrac{1}{2} \delta^{0\alpha} - \tfrac{1}{2}\omega_i \delta^{i\alpha}$,
we have
 \begin{align}
   j^\alpha(\pa\Phi) &=
   -\delta^{u \alpha} (\pa_u\Phi)^2 + \widetilde{j}^\alpha(\pa\Phi)\\
   &= -\delta^{u\alpha} (\pa_u \Phi)^2
   + \widetilde{A}^{\alpha \beta \delta}(\omega) \pa_\beta \Phi \pa_\delta \Phi
   + B^\alpha(\pa\Phi).
   \label{basicstructureformula}
 \end{align}
 where the $\widetilde{A}$ are smooth functions on $\mathbb{S}^2$,
 $B$ consists of terms which are cubic or higher-order, and
 where
 the nonlinearity $\pa_\alpha \left(\widetilde{A}^{\alpha \beta \delta}\pa_\beta \Phi \pa_\delta\Phi\right)$
 verifies the classical null condition $\widetilde{A}^{\alpha\beta\delta}\pa_\alpha u \pa_\beta u \pa_\delta u = 0$;
 in particular, for arbitrary smooth $q_1, q_2$,
 we have
 \begin{equation}
	\pa_\alpha \left(\widetilde{A}^{\alpha \beta \delta}\pa_\beta q_1 \pa_\delta q_2\right) =
	\opa_\alpha \left( A^{\alpha \beta \delta}_1 \pa_\beta q_1 \pa_\delta q_2\right)
	+ \pa_\alpha \left( A^{\alpha\beta\delta}_2 \opa_\beta q_1 \pa_\delta q_2\right)
	+ \pa_\alpha \left( A^{\alpha\beta\delta}_3 \pa_\beta q_1 \opa_\delta q_2\right),
	\label{basicstructureformulanull}
 \end{equation}
 where the $A^{\alpha \beta \delta}_i$ are smooth functions on $\mathbb{S}^2$,
 and where
	\begin{equation}
	 \opa_0 = 2\pa_v = \pa_t + \pa_r, \qquad
	 \opa_i = \nas_i = \pa_i - \omega_i\omega^j \pa_j,\qquad \omega = x/|x|
	 \label{}
	\end{equation}

  In particular the continuity equation \eqref{contderiv} can be written in
	either of the forms
	  \begin{align}
			\Box \Phi + \pa_\alpha j^\alpha(\pa \Phi) &=0,\label{wave00}\\
    \Box \Phi -\pa_u (\pa_u\Phi)^2 + \pa_\alpha \widetilde{j}^\alpha(\pa\Phi) &= {-\frac{2}{r}(\pa_u\Phi)^2},
   \label{wave0}
  \end{align}
	where $\pa_\alpha \widetilde{j}^\alpha$ verifies the classical null condition.

	If $\mathcal{Z}$ is an arbitrary family of vector fields and $Z^I$ denotes
 an $|I|$-fold product of the fields in $\mathcal{Z}$ then
\begin{equation}
	Z^I j^\alpha(\pa\Phi) = \gamma_{{0}}^{\alpha\beta}(\pa \Phi) Z^I \pa_\beta \Phi
	+ P^\alpha_{I,0}(\pa\Phi), \qquad \gamma^{\alpha\beta}(\xi)
	= \pa_{\xi^\beta} j^\alpha(\xi),
 \label{vectorfieldsonconteqn}
\end{equation}
 where {$\gamma_0$ is symmetric and 
 $\gamma_{0}$ and $P_{I,0}$} satisfy the following estimates,
\begin{equation}
 |\gamma_0| \lesssim |\pa \Phi|, \qquad |P_{I,0}| \lesssim \sum_{|I_1| + \cdots |I_r| \leq |I|-1, r \geq 2}
 |Z^{I_1} \pa \Phi|
 \cdots
 |Z^{I_r} \pa \Phi|,
\end{equation}
\begin{equation}
 |\nabla_X \gamma_0| \lesssim |\nabla_X \pa \Phi|,
 \qquad
 |\nabla_X P_{I,0}| \lesssim
 \sum_{|I_1| + \cdots |I_r| \leq |I|-1, r \geq 2}
 |\nabla_X Z^{I_1} \pa \Phi|
 \cdots
 |Z^{I_r} \pa \Phi|
 \label{}
\end{equation}
\end{lemma}
\begin{proof}
 It only remains to prove \eqref{vectorfieldsonconteqn}, and that follows
 directly from the chain rule and the fact that for each
 $\alpha$, $\xi\mapsto j^\alpha(\xi)$ are smooth
 functions and $j^\alpha(0) = 0$.
\end{proof}

\section{The equation for $r\Phi$ and the higher-order continuity equation}
\label{higherorderequations}

We now want to commute the equation \eqref{wave00} with a family of vector fields
and then expand the solution $\Phi$ around the model shock profile,
$\Phi = \sigma + \phi$ where
\begin{equation}
 \sigma =\begin{cases}
\frac{u^2}{2rs},\qquad \text{ in } D^C,\\
0,\qquad \text{ otherwise } \end{cases}
 \label{}
\end{equation}
In the regions $D^L, D^R$ we will commute with the full family of Minkowski
vector fields $\mathcal{Z}$
and in the region $D^C$ we commute with the family
\begin{equation}
 \mZB = \{\Omega_{ij}, s\pa_u, v\pa_v\}.
 \label{mZBdef}
\end{equation}

We start with the computation in the exterior regions, where the model shock profile
vanishes and where the linearized operator is the Minkowskian wave operator.
There, we will not need to keep track of the structure of the nonlinear
terms and it will suffice to start from
\eqref{wave00}.

\subsection{The higher-order equation in the exterior regions}
\label{derivationexterior}

Let $Z^I$ denote a product of Minkowskian vector fields, $Z \in \mathcal{Z}$.
In this section we want to find an equation for
$Z^I (r \Phi)$ and express it in null coordinates
\eqref{mainminknull}. Starting from \eqref{wave00} and using
\eqref{widehatbox} and
\eqref{vectorfieldsonconteqn} we find that
$\Phi^I = \widetilde{Z}^I\Phi$ satisfies
\begin{equation}
 \Box \Phi^I + \pa_\alpha\left( \gamma^{\alpha\beta} \widetilde{Z}^I \pa_\beta \Phi\right)
 + \pa_\alpha P_{I, 0}^\alpha = 0,
 \label{}
\end{equation}
and using the fact that $[Z, \pa_\alpha] = c_{\alpha Z}^\beta \pa_\beta$
for constants $c_{\alpha Z}^\beta$, this takes the form
\begin{equation}
 \Box \Phi^I + \pa_\alpha\left( \gamma^{\alpha\beta}\pa_\beta  \Phi^I\right)
 + \pa_\alpha P_{I,1}^\alpha = 0,
 \label{comm0}
\end{equation}
where $\gamma_0, P_{I,1}$ satisfy the estimates
\begin{equation}
 |\gamma| \lesssim |\pa \Phi|, \qquad
  |P_{I,1}| \lesssim \sum_{|I_1| + \cdots + |I_r| \leq |I|, r \geq 2,{|I_i| \leq |I|-1}} | \pa Z^{I_1} \Phi|
 \cdots
 |\pa Z^{I_r}  \Phi|,\label{basicgammabound0}
\end{equation}
\begin{equation}
 |\nabla_X \gamma| \lesssim |\nabla_X \pa \Phi|,
 \qquad
 |\nabla_X P_{I,1}| \lesssim
 \sum_{|I_1| + \cdots +|I_r| \leq |I|, r \geq 2,{|I_i| \leq |I|-1}}
 |\nabla_X \pa  Z^{I_1} \Phi|
 \cdots
 |\pa Z^{I_r}  \Phi|.
 \label{basicPbound0}
\end{equation}
For the estimates near $r = 0$ we will want bounds in terms of Lie derivatives,
\begin{equation}
 |\mathcal{L}_X \gamma| \lesssim |\mathcal{L}_X \pa \Phi|
 \qquad
 |\mathcal{L}_X P_{I,1}| \lesssim
 \sum_{|I_1| + \cdots +|I_r| \leq |I|, r \geq 2,{|I_i| \leq |I|-1}}
 |\pa  Z^{I_1} \Phi|
 \cdots
 |\mathcal{L}_X \pa  Z^{I_k} \Phi|\cdots
 |\pa Z^{I_r}\Phi|.
 \label{}
\end{equation}

We now want to express the nonlinearity in \eqref{comm0} in the null coordinate
system \eqref{mainminknull}. For this we use the fact that if
$q = q^\alpha\pa_{x^\alpha}$ is a vector field and $q^\alpha$ denote
the components of $q$ expressed relative to rectangular coordinates and
$q^\mu$ the components expressed relative to the coordinate system
\eqref{mainminknull} then $\pa_\alpha q^\alpha = \div q = r^{-2}\pa_\mu(r^2 q^\mu)$.
As a result,
\begin{equation}
 \pa_\alpha (\gamma^{\alpha\beta} \pa_\beta \Phi^I) +
 \pa_\alpha P^\alpha_{I,1}
 = \frac{1}{r^2}\pa_\mu (r^2\gamma^{\mu\nu}\pa_\nu \Phi^I)
 + \frac{1}{r^2}\pa_\mu(r^2 P^\mu_{I,0})
 =\frac{1}{r^2} \pa_\mu (r \gamma^{\mu\nu} \pa_\nu \Psi^I)
  + \frac{1}{r^2} \pa_\mu (r^2 P^\mu_{I, 1} - \gamma^{\mu r} \Psi^I),
 \label{}
\end{equation}
where all quantities on the right-hand side are expressed relative to
the coordinates \eqref{mainminknull} and where we have introduced
$\Psi^I = r \Phi^I$. Since
$r \Box q = (-4\pa_u\pa_v + \sDelta)(rq)$, we have the equation
\begin{equation}
 -4\pa_u\pa_v \Psi^I + \sDelta \Psi^I + \frac{1}{r} \pa_\mu(r\gamma^{\mu\nu}\pa_\nu \Psi^I)
 + \frac{1}{r} \pa_\mu (r^2 P_{I,1}^\mu - \gamma^{\mu r} \Psi^I) = 0
 \label{}
\end{equation}
Introducing
\begin{equation}
 P_I^\mu = r P_{I, 1}^\mu - \frac{1}{r} \gamma^{\mu r} \Psi^I,
 \qquad
 F_I = -\frac{1}{r} \gamma^{r \nu}\pa_\nu \Psi^I - 2P_{I, 1}^r + \frac{1}{r^2}\gamma^{rr} \Psi^I,
 \label{PIbasic}
\end{equation}
we have the following result.

\begin{lemma}
	\label{higherorderexterioreqns}
 With $\psi^I = r\widetilde{Z}^I \phi$, $\psi^I$ satisfies
 \begin{equation}
  \left(-4\pa_u\pa_v + \sDelta \right) \psi^I + \pa_\mu (\gamma^{\mu\nu}
	\pa_\nu \psi^I) + \pa_\mu P^\mu_I = F_I,
  \label{higherorderexterioreqn}
 \end{equation}
If $|\pa Z^J\phi| \leq C$ for $|J| \leq |I|/2+1$,
the above quantities satisfy the following bounds,
\begin{equation}
   |\gamma| \lesssim |\pa \phi|,
   \qquad
   |\nabla_X \gamma|  \lesssim |\nabla_X \pa \phi|,
 	\qquad
 	|\mathcal{L}_X\gamma|\lesssim |\mathcal{L}_X \pa \phi|,
   \label{gammabdsexteriorapp}
\end{equation}
	\begin{align}
	 |P_I| &\lesssim \sum_{\substack{|I_1| + |I_2| \leq |I|, \\ |I_1|, |I_2| \leq |I|-1}}
	  r |\pa Z^{I_1} \phi| |\pa Z^{I_2} \phi|
	 + |\pa Z^{I_1} \phi| |Z^{I_2} \phi|,
	 \label{Pbounds}
	 \\
 	|F_I| &\lesssim \sum_{|I_1| + |I_2| \leq |I|}
 	 |\pa Z^{I_1} \phi| |\pa Z^{I_2} \phi|
 	+ r^{-1}|\pa Z^{I_1} \phi| |Z^{I_2} \phi|,
	\label{Fbounds}\\
	 |\nabla_X P_I| &\lesssim
	 \sum_{\substack{|I_1| + |I_2| \leq |I|, \\ |I_1|, |I_2| \leq |I|-1}}
	 r | \nabla_X \pa Z^{I_1} \phi| |\pa Z^{I_2} \phi| +
	 | \nabla_X \pa Z^{I_1} \phi||Z^{I_2}\phi|
	 + |\pa Z^{I_2} \phi| |\nabla_X Z^{I_2}\phi| \\
	 &\qquad
	 \sum_{\substack{|I_1| + |I_2| \leq |I|, \\ |I_1|, |I_2| \leq |I|-1}}
	 |X| \left(|\pa Z^{I_1} \phi| |\pa Z^{I_2} \phi|
	+ r^{-1}|\pa Z^{I_1} \phi| |Z^{I_2} \phi|\right), \label{Pbounds2}\\
	 |\mathcal{L}_{X} P_I| &\lesssim
	 \sum_{\substack{|I_1| + |I_2| \leq |I|, \\ |I_1|, |I_2| \leq |I|-1}}
	 r \left(| \mathcal{L}_X \pa Z^{I_1} \phi| |\pa Z^{I_2} \phi| +
	 | \mathcal{L}_X \pa Z^{I_1} \phi||Z^{I_2}\phi|
	 + |\pa Z^{I_2} \phi| |\mathcal{L}_X Z^{I_2}\phi| \right)\\
	 &\qquad
	 \sum_{\substack{|I_1| + |I_2| \leq |I|, \\ |I_1|, |I_2| \leq |I|-1}}
	 |X| \left(|\pa Z^{I_1} \phi| |\pa Z^{I_2} \phi|
	+ |\pa Z^{I_1} \phi| |Z^{I_2} \phi|\right),
	\label{Pbounds3}
	\end{align}
	if $X = X^u\pa_u + X^v\pa_v$.
\end{lemma}

\subsection{The  higher-order equations in the central region}
\label{derivationC}

{In this section we will need to track the nonlinear terms a bit more carefully
because we want to use the fact that $\Sigma$ is an approximate solution
and so the argument here is organized a bit differently
than that in the previous section.}

We start by deriving the wave equation satisfied by the perturbation $\psi$.
\begin{lemma}
	\label{centraleqnlowestprop}
 We have
 \begin{equation}
  -4\pa_u \left( \pa_v + \frac{u}{vs} \pa_u\right) \psi
	+\sDelta \psi + \pa_\mu\left( \frac{u}{vs} a^{\mu\nu}\pa_\nu \psi\right) +
	 \pa_\mu(\gamma^{\mu\nu} \pa_\nu \psi)
	 = F + F_\Sigma
  \label{centraleqnlowest}
 \end{equation}
 where the above quantities satisfy the following properties. First,
 $a^{\mu\nu} = a^{\nu\mu}$
 are smooth functions satisfying the strong symbol condition \eqref{strongsymbol} as well as the null condition
 \eqref{intronullcondn}.
 The quantities $\gamma^{\mu\nu}$ take the form
 \begin{equation}
  \gamma^{\mu\nu} = \frac{1}{1+v} \gamma^{\mu\nu\nu'} \pa_{\nu'} \psi
	+ \frac{1+s}{(1+v)^2}A_1^{\mu\nu}\psi
	+ \frac{1+s}{(1+v)^2} A_2^{\mu\nu}
	+ r B^\mu
  \label{gammaCformula}
 \end{equation}
 where the $\gamma^{\mu\nu\nu'}$ and  ${A^{\mu\nu} }$ are smooth
 functions satisfying the weak symbol condition \eqref{weaksymbol} and $B^\mu$ consists
 of terms which are cubic or higher-order, in the sense that
 $ \xi \mapsto B(\xi) $ vanishes to third order at $\xi = 0 $,
  and its components can be written in the form
  \begin{align}
   B = \frac{1}{(1+v)^2}
 	B_0\left(( \pa(\psi/ r) + \pa( as/r\right))
 	\bigg( Q_1(\pa \psi + a, \pa \psi + a) +\frac{1}{1+v} Q_2(\pa \psi + a, \psi + bs)
 	+ \frac{1}{(1+v)^2}Q_3(\psi + bs, \psi + bs)\bigg),
   \label{Bformula}
  \end{align}
	where $B_0$ is a smooth function
  with $B_0(0) = 0$, the $Q_i$ are quadratic nonlinearities
  with smooth coefficients satisfying the weak symbol condition,
  and the $a, b$ are weak symbols.

 The quantity $F$ is of the form
 \begin{align}
	F &= \frac{C^{\mu\nu} \pa_\mu \psi \pa_\nu \psi}{(1+v)^2}
	+ \frac{D^{\mu} \pa_\mu \psi\psi}{(1+v)^3}
	+ \frac{D \psi^2}{(1+v)^4}
	+ \frac{E \psi}{(1+v)^2(1+s)}
	+ B^r,
  \label{FCformula}
 \end{align}
 where the coefficients above satisfy the weak symbol condition \eqref{weaksymbol},
 $B$ is a cubic nonlinearity depending on $\pa \psi$,
 and the function $F_\Sigma$ is smooth and takes the form
 \begin{equation}
  F_\Sigma = \frac{C}{(1+v)^2},
  \label{FSigmaformula}
 \end{equation}
 where $C$ satisfies the weak symbol condition \eqref{weaksymbol}.

\end{lemma}

The structure of the above terms (in particular, those which are linear in
$\psi$) is explained after \eqref{acceptablemetric}; see in particular
Lemma \ref{ffsclaim}.

\begin{proof}
 We first
handle the ``bad'' term in the nonlinearity and exploit the fact
that $\pa_u \Sigma$ satisfies Burgers' equation and find the effective linearized
equation for the peturbation $\psi = r\varphi$. We then need to manipulate the remaining
terms in the nonlinearity, and we want to express the result in null coordinates
in terms of the variable $\psi$ and remainder terms which decay much more quickly
than the ``bad'' term.

\emph{Step 1: Extracting the effective equation}

We start by manipulating the first two terms in \eqref{wave0},
\begin{equation}
 r\left(\Box \Phi - \pa_u ( (\pa_u \Phi)^2 )\right)
 = \left(-4\pa_v \pa_u + \sDelta \right) \Psi
 - \pa_u ( r (\pa_u\Phi)^2 ) - \frac{1}{2} (\pa_u \Phi)^2,
 \label{stillwave0}
\end{equation}
using $\pa_u r= -\frac{1}{2}$.
With $\Psi = r\Phi$, the quadratic term here can be written in the form
\begin{equation}
 \pa_u (r (\pa_u \Phi)^2) = \pa_u \left(\frac{1}{r}  (\pa_u \Psi)^2 \right)
 + \pa_u \left( \frac{1}{r^2} \Psi \pa_u \Psi - \frac{1}{4r^3} \Psi^2 \right),
 \label{}
\end{equation}
and, writing $\frac{1}{r} = \frac{2}{v-u} = \frac{2}{v} + \frac{u}{v} \frac{1}{r}$,
we further write the first term here in the form
\begin{equation}
 \pa_u \left(\frac{1}{r}  (\pa_u \Psi)^2 \right)
 = \frac{2}{v} \pa_u \left( (\pa_u \Psi)^2\right)
 + \pa_u\left( \frac{u}{v} \frac{1}{r} (\pa_u \Psi)^2\right).
 \label{}
\end{equation}
Returning to \eqref{stillwave0}, we have the identity
\begin{equation}
 r\left(\Box \Phi - \pa_u ( (\pa_u \Phi)^2 )\right)
 = \left(-4\pa_v \pa_u + \sDelta \right) \Psi
 - \frac{2}{v} \pa_u \left( (\pa_u \Psi)^2\right)
 + \pa_u \gamma_0(\Psi, \pa \Psi)
 + F_0,
 \label{ffs0}
\end{equation}
\begin{equation}
	\gamma_0(\Psi, \pa \Psi) = -\frac{u}{v} \frac{1}{r} (\pa_u \Psi)^2
 -\frac{1}{r^2} \Psi \pa_u \Psi,\qquad
 F_0=\pa_u \left( \frac{1}{4r^3} \Psi^2 \right)- \frac{1}{2} (\pa_u \Phi)^2.
 \label{ffs}
\end{equation}
Now we expand $\Psi = \Sigma + \psi$ with $\Sigma = \frac{u^2}{2s}$. Noting
that $\Sigma$ satisfies
\begin{equation}
 \pa_v \pa_u \Sigma + \frac{1}{2v}\pa_u (\pa_u\Sigma)^2 = 0,
 \qquad
 \sDelta \Sigma  = 0,
 \label{}
\end{equation}
we find
\begin{multline}
 \left(-4\pa_u \pa_v + \sDelta \right) \Psi
 - \frac{2}{v} \pa_u \left( (\pa_u \Psi)^2\right)
 = (-4\pa_u \pa_v +\sDelta)\psi
 - \frac{4}{v} \pa_u \left( \pa_u \Sigma \pa_u \Psi\right)
 - \frac{2}{v} \pa_u \left( (\pa_u \psi)^2\right)\\
 = -4 \left( \pa_u \left(\pa_v + \frac{u}{vs} \pa_u\right) - \frac{1}{4} \sDelta\right)
 \psi
 - \frac{2}{v} \pa_u \left( (\pa_u \psi)^2\right).
 \label{}
\end{multline}
The equation \eqref{ffs0} can then be written in the form
\begin{equation}
	r\left(\Box \Phi - \pa_u ( (\pa_u \Phi)^2 )\right)
	=
 -4 \left( \pa_u \left(\pa_v + \frac{u}{vs} \pa_u\right) - \frac{1}{4} \sDelta\right)
 \psi
 - \frac{2}{v} \pa_u \left( (\pa_u \psi)^2\right)+ \pa_u \gamma_0(\Psi, \pa \Psi)
 + F_0,
 \label{}
\end{equation}

By \eqref{wave0} and the previous equation, we have arrived at
the equation
\begin{equation}
	-4 \left( \pa_u \left(\pa_v + \frac{u}{vs} \pa_u\right) - \frac{1}{4} \sDelta\right)\psi
  -\frac{2}{v} \pa_u \left( (\pa_u \psi)^2\right)
	+ r \pa_\alpha \left(\widetilde{A}^{\alpha\beta\delta}\pa_\beta \Psi \pa_\delta \Psi\right)
	 +r \pa_\alpha B^\alpha + \pa_u \gamma_0 =
	F_1,
 \label{stillgoing00}
\end{equation}
where
\begin{align}
	\gamma_0 &=
	-\frac{u}{v} \frac{1}{r} (\pa_u \Psi)^2
 -\frac{1}{r^2} \Psi \pa_u \Psi
	\\
F_1 &=
\pa_u \left( \frac{1}{4r^3} \Psi^2 \right) - \frac{3}{2} (\pa_u \Phi)^2,
 \label{F1defgoing}
\end{align}
and where recall that $\widetilde{A} = \widetilde{A}(\omega)$ are smooth functions on
$\mathbb{S}^2$ verifying the null condition. It remains to handle the remainder
terms $r \pa_\alpha (\widetilde{A}^{\alpha\beta\delta}\pa_\beta\Phi \pa_\delta \Phi)$
, $r \pa_\alpha B^\alpha$,
$\pa_u \gamma_0$ and $F_1$.

\emph{Step 2: Dealing with the remaining terms in \eqref{stillgoing00}}

We now want to show that the above remainder terms are as in the statement
of the lemma. To handle these terms, it will be convenient to
make the following definitions. We say that a two-tensor $\gamma^{\mu\nu}$ is
an ``acceptable metric correction''  if it is a sum of terms of the following types,
\begin{equation}
	\frac{A^{\nu}\pa_{\nu}\psi}{1+v} , \quad
	\quad
	\frac{(1+s) B \psi}{(1+v)^2},
	\qquad
	\frac{(1+s)C}{(1+v)^2},
 \label{acceptablemetric}
\end{equation}
for smooth coefficients $A^\nu, B, C,$ satisfying the
weak symbol condition \eqref{weaksymbol}.
Terms of the last two types here will be generated
by the quadratic nonlinearities when we expand $\Psi = \Sigma + \psi$. Such terms are consistent
with \eqref{gammaCformula}. Similarly,
we say that a function $F$ is an ``acceptable remainder'' if it
is a sum of terms of the following types,
\begin{equation}
	\frac{A^{\mu\nu}\pa_\mu\psi\pa_\nu\psi}{(1+v)^2} ,
	\quad
	\frac{B^{\mu}\pa_\mu\psi\psi}{(1+v)^3} ,
	\quad
	\frac{ C\psi^2}{(1+v)^4},
	\quad
	\frac{A^\mu \pa_\mu \psi}{(1+v)^2},
	\quad
	\frac{B \psi}{(1+v)^2(1+s)}  ,
	\quad
	\frac{C}{(1+v)^2}
 \label{acceptableremainder}
\end{equation}
where the coefficients above are weak symbols.
Terms of the last three types account for
error terms generated by expanding around the model shock profile
$\Sigma$ and the fact that $\Sigma = s A$ for a weak symbol $A$.
These terms are consistent with \eqref{FCformula}, and so
Lemma \ref{centraleqnlowestprop} follows from the upcoming Lemma \ref{ffsclaim}.

\end{proof}

\begin{lemma}\label{ffsclaim}
	With notation as in \eqref{stillgoing00}-\eqref{F1defgoing},
	each of the quantities $r\pa_\alpha \left(\widetilde{A}^{\alpha\beta\delta}\pa_\beta\Psi\pa_\delta\Psi \right), \pa_u\gamma^0$ and
 $F_0$ can be written in the form $\pa_\mu(\tfrac{u}{vs}a^{\mu\nu}\pa_\nu \psi) + \pa_\mu(\gamma^{\mu\nu}\pa_\nu \psi) + F$,
 where $a^{\mu\nu}$ is a strong symbol \eqref{strongsymbol}
 and verifies the null condition \eqref{intronullcondn0},
 where $\gamma$ is an acceptable metric correction $\gamma$ and where $F$
 is an acceptable remainder. The cubic nonlinearity $r\pa_\alpha B^\alpha$
 can be written as in \eqref{Bformula}.
\end{lemma}

\emph{Proof.}
To expand around $\Sigma$, it will be helpful to note that $\Sigma$ satisfies
\begin{equation}
 \pa_\mu \Sigma = c_\mu \frac{u}{s},
 \qquad
 \opa_\mu \Sigma
 = d_\mu \frac{1}{v} \frac{u^2}{s^2},
 \label{derivsofsigma}
\end{equation}
for strong symbols $c_\mu, d_\mu$ satisfying \eqref{strongsymbol}.
In particular, we can write
\begin{equation}
 \pa_\mu \Sigma= a_\mu ,\qquad \opa_\mu \Sigma = \frac{1}{1+v} b_\mu,
 \label{sigmaweaksymbols}
\end{equation}
for weak symbols $a_\mu, b_\mu$ satisfying \eqref{weaksymbol}.

To deal with various powers of $u$ we will enounter in the following,
we will use the fact that
if $f$ is smooth, then $f(\frac{u}{s})$ satisfies the weak symbol condition \eqref{weaksymbol}
in the region $|u| \lesssim s^{1/2}$.

\begin{proof}[Acceptability of $\pa_u\gamma_0$]
Using \eqref{sigmaweaksymbols}, we can write the first
term in the definition of $\gamma_0$ in the form
\begin{equation}
 \frac{u}{v}\frac{1}{r} (\pa_u\Psi)^2
 = \frac{s}{(1+v)^2}a (\pa_u\Psi)^2
 = \left(\frac{s}{(1+v)^2}a \pa_u\psi + \frac{2s}{(1+v)^2}\frac{u}{s}a \right)\pa_u\psi
 +\frac{s}{(1+v)^2} \frac{u^2}{s^2} a,
 \label{}
\end{equation}
for a weak symbol $a$.
The quantity in the brackets
is an acceptable metric correction
because it is a sum of
the first and third types appearing in \eqref{acceptablemetric}, after using
that $u/s$ is a weak symbol. The $u$ derivative
of the last term here can be written in the form $1/(1+v)^2 a'$ for a weak symbol
$a'$ and is thus an acceptable remainder.

For the second term in the definition of $\gamma_0$ we write
\begin{equation}
 \frac{1}{r^2} \Psi \pa_u\Psi
 = \frac{1}{(1+v)^2} a \Psi \pa_u \Psi
 = \left(\frac{1}{(1+v)^2} a \psi
 + \frac{s}{(1+v)^2} \frac{u^2}{s^2} a \right)\pa_u\psi
 + \frac{1}{(1+v)^2}\frac{u}{s} a \psi  + \frac{s}{(1+v)^2} \frac{u^2}{s^2},
 \label{acceptablePsipaPsi0}
\end{equation}
for a weak symbol $a$. The quantity in the brackets is an acceptable
metric correction because it is a sum of
the second and third types appearing in \eqref{acceptablemetric}.
The $u$ derivative of the last two terms here can be written in the form
$a\psi/((1+v)^2(1+s)) + b/(1+v)^2$ for weak symbols $a,b$, and it is
therefore an acceptable remainder.

\end{proof}
\begin{proof}[Acceptability of $F_1$]
	Writing $\Phi = \Psi/r$,
	the remainder $F_1$ can be written in the form
	\begin{equation}
	 F_1 = \frac{1}{r^2}c_1(\pa_u\Psi)^2
	 + \frac{1}{r^3} c_2 \Psi \pa_u \Psi + \frac{1}{r^4} c_3\Psi^2,
	 \label{F1expressionabstract}
	\end{equation}
	for constants $c_1,c_2, c_3$. If we expand $\Psi = \Sigma + \psi$ and use
		\eqref{sigmaweaksymbols} to express derivatives of $\Sigma$ along with the
		fact that $\Sigma = (1+s)a$ for a weak symbol $a$, we find the following expressions
	for the above nonlinearities,
	\begin{equation}
	 (\pa_u\Psi)^2 = a_1 (\pa_u\psi)^2 + a_2 \pa_u\psi + a_3,
	 \qquad
	 \Psi \pa_u\Psi = b_1\psi \pa_u\psi + b_2(1+s) \pa_u\psi
	 + b_3 \psi + b_4(1+s),
	 \label{}
	\end{equation}
	\begin{equation}
	 \Psi^2 = d_1\psi^2 + d_2(1+s) \psi + d_3(1+s)^2,
	 \label{}
	\end{equation}
	where the above coefficients are weak symbols. Inserting these into
	\eqref{F1expressionabstract} shows that $F_1$ is an acceptable remainder.

\end{proof}

\begin{proof}[Acceptability of $r \pa_\alpha (\widetilde{A}^{\alpha\beta\delta}\pa_\beta \Psi \pa_\delta \Psi)$]
	This is more complicated to establish because it requires exploiting the null
	condition.
	We first re-write this quantity in terms of $\Psi$,
\begin{equation}
  \pa_\alpha \left( \widetilde{A}^{\alpha \beta \delta} \pa_\beta \Phi \pa_\delta \Phi\right)
 = \pa_\alpha \left( \frac{1}{r^2} \widetilde{A}^{\alpha \beta \delta} \pa_\beta \Psi \pa_\delta \Psi\right)
 + \pa_\alpha \left( \frac{1}{r^3} (\widetilde{A}^{\alpha \beta r} + \widetilde{A}^{\alpha r \beta}) \Psi \pa_\beta\Psi
 + \frac{1}{r^4} \widetilde{A}^{\alpha rr} \Psi^2\right)
 \label{stillgoing}
\end{equation}
where the $\widetilde{A}^{\alpha\beta\delta}$ are smooth and satisfy
the null condition \eqref{basicstructureformulanull}. We now want to pass to
null coordinates \eqref{mainminknull}.
	With the convention that indices $\alpha, \beta, \delta$ refer
	to quantities expressed in rectangular coordinates and $\mu, \nu,\nu'$
	refer to quantities expressed in the coordinate system \eqref{mainminknull}, using the identity $\pa_\alpha X^\alpha = \div X = \frac{1}{r^2} \pa_\mu(r^2 X^\mu)$,
	we have $r\pa_\alpha X^\alpha = r^{-1}\pa_\mu(r^2X^\mu) =\pa_\mu(rX^\mu) + X^r$, and as a result,
	\begin{equation}
	 r \pa_\alpha\left( \frac{1}{r^2} \widetilde{A}^{\alpha \beta \delta} \pa_\beta \Psi \pa_\delta \Psi\right)
	 = \pa_\mu\left( \frac{1}{r} \widetilde{A}^{\mu \nu \nu'} \pa_\nu \Psi \pa_{\nu'}\Psi\right)
	 + \frac{1}{r^2} \widetilde{A}^{r \nu\nu'}\pa_{\nu}\Psi \pa_{\nu'}\Psi.
	 \label{}
	\end{equation}
	Using the same formula for the remaining
	terms on the second line of \eqref{stillgoing} we have the formula
	\begin{equation}
	 r \pa_\alpha \left(\widetilde{A}^{\alpha\beta\delta}\pa_\beta\Phi\pa_\delta\Phi\right)
	 = \pa_\mu\left( \frac{1}{r} \widetilde{A}^{\mu \nu \nu'} \pa_\nu \Psi \pa_{\nu'}\Psi
	 {+ \frac{1}{r^2}(\widetilde{A}^{\mu \nu r} + \widetilde{A}^{\mu r \nu}) \Psi \pa_\nu\Psi }\right)
	  + F_{{2}}(\Psi, \pa\Psi),
	 \label{stillgoing2}
	\end{equation}
where
\begin{align}
 F_{{2}}(\Psi, \pa\Psi) &= \pa_{\mu}\left(  \frac{1}{r^3} \widetilde{A}^{\mu rr} \Psi^2
\right) + \frac{1}{r^2} \widetilde{A}^{r \nu\nu'}\pa_{\nu}\Psi \pa_{\nu'}\Psi -
 {\frac{1}{r^3}(\widetilde{A}^{\mu \nu r} + \widetilde{A}^{\mu r \nu}) \Psi \pa_\nu\Psi
 -\frac{1}{r^4} \widetilde{A}^{\mu rr}\Psi^2}.
 \label{sgf2}
\end{align}
{
	After expanding around $\Sigma$ it is clear that $F_{2}$ is an acceptable remainder
	since it has the same structure as the remainder $F_1$ which we previously handled.
	Similarly, the second term
	can be written in terms of an acceptable metric correction and acceptable remainder
	as in \eqref{acceptablePsipaPsi0}. 
We now expand
	$\Psi = \Sigma + \psi$ in the first term on the right-hand side
	of \eqref{stillgoing2}, we get
	\begin{equation}
	 \pa_\mu\left( \frac{1}{r} \widetilde{A}^{\mu \nu \nu'} \pa_\nu \Psi \pa_{\nu'}\Psi\right)
	 =\pa_\mu\left( \frac{1}{r}
	 \widetilde{A}^{\mu \nu \nu'}\pa_\nu \psi \pa_{\nu'}\psi\right)+
	 \pa_\mu\left( \frac{1}{r}
	 (\widetilde{A}^{\mu \nu \nu'} +\widetilde{A}^{\mu \nu' \nu})\pa_\nu \Sigma \pa_{\nu'}\psi\right)
	 + F_{\Sigma, 1},
	 \label{stillgoing3}
	\end{equation}
	and doing the same with the second term in \eqref{stillgoing2} we find
	\begin{multline}
	 \pa_\mu \left( \frac{1}{r^2}
	 (\widetilde{A}^{\mu\nu r} + \widetilde{A}^{\mu r \nu})\Psi \pa_\nu \Psi\right)
	 =
	 \pa_\mu \left( \frac{1}{r^2}
	 (\widetilde{A}^{\mu\nu r} + \widetilde{A}^{\mu r \nu})\psi \pa_\nu \psi\right)
	 + \pa_\mu \left(\frac{1}{r^2}
	  (\widetilde{A}^{\mu\nu r} + \widetilde{A}^{\mu r \nu})\Sigma \pa_\nu \psi \right)
		\\
		+ \pa_\mu \left( \frac{1}{r^2}
		(\widetilde{A}^{\mu\nu r} + \widetilde{A}^{\mu r \nu})\pa_\nu\Sigma\psi\right)
		+ F_{\Sigma, 2},
	 \label{stillgoing4}
	\end{multline}
	where $F_{\Sigma, 1}, F_{\Sigma, 2}$ collect the terms involving $\Sigma$
	alone
	\begin{align}
	 F_{\Sigma, 1} &= \pa_\mu\left( \frac{1}{r}
	 (\widetilde{A}^{\mu \nu \nu'} +\widetilde{A}^{\mu \nu' \nu})\pa_\nu \Sigma \pa_{\nu'}\Sigma\right),
	 \label{fSigma1}\\
	 F_{\Sigma, 2} &=  \pa_\mu \left( \frac{1}{r^2}
 	 (\widetilde{A}^{\mu\nu r} +
	 \widetilde{A}^{\mu r \nu})\Sigma \pa_\nu \Sigma\right).
	 \label{fSigma2}
	\end{align}

	We now consider each of the above quantities.

	\emph{Acceptability of the terms in \eqref{stillgoing3}}
	The quantity $\frac{1}{r} \widetilde{A}^{\mu\nu\nu'}\pa_\nu \psi$ from
	the right-hand side of \eqref{stillgoing3}
	is an acceptable metric
	correction, because it is of the first type in
	\eqref{acceptablemetric}.
	To see that the second term on the right-hand side of \eqref{stillgoing3}
	involves an acceptable metric correction we will use that the coefficients
	$\widetilde{A}$ verify the null condition.
	For this we use \eqref{sigmaweaksymbols}
	to write $\frac{1}{r}\pa_\nu \Sigma = c_\nu \frac{u}{vs}
	+ \frac{1}{(1+v)^2}d_\nu$
	where $c_\nu$ is a strong symbol and $d_\nu$ is a weak
	symbol. Since  $\widetilde{A}$ verifies
	the null condition $\widetilde{A}^{uuu} = 0$, we can write
	\begin{equation}
	 \pa_\mu\left( \frac{1}{r}
	(\widetilde{A}^{\mu \nu \nu'} +\widetilde{A}^{\mu \nu' \nu})\pa_\nu \Sigma \pa_{\nu'}\psi\right)
	= \pa_\mu\left( \frac{u}{vs} a^{\mu\nu} \pa_\nu \psi\right)
	+ \pa_\mu\left( \frac{1}{(1+v)^2} b^{\mu\nu}\pa_\nu \psi\right)
	 \label{}
	\end{equation}
	where $a$ verifies the null condition $a^{uu} = 0$
	and where the $b$ are weak symbols.
	Each of these is of the correct form. 

	As for the quantity $F_{\Sigma, 1}$, again using the
	null condition, it can be written in the form
	\begin{equation}
	 F_{\Sigma, 1}
	 = \frac{1}{(1+v)^2} a^{\mu\nu}\pa_\mu\Sigma \pa_\nu \Sigma
	 + \frac{1}{1+v} b^{\mu\nu\nu'}\opa_\mu \pa_{\nu'}\Sigma \pa_\nu \Sigma
	 + \frac{1}{1+v} c^{\mu\nu\nu'}\pa_\mu \pa_{\nu'}\Sigma \opa_\nu \Sigma
	 \label{}
	\end{equation}
	for weak symbols $a,b,c$, and in light of \eqref{sigmaweaksymbols},
	this is an acceptable remainder.

 \emph{Acceptability of the terms in \eqref{stillgoing4}}
 The quantity $\frac{1}{r^2} (\widetilde{A}^{\mu\nu r} + \widetilde{A}^{\mu r\nu})$
 is an acceptable metric term since it is of the second type in
  \eqref{acceptablemetric}. Writing $\Sigma = s a$ for a weak symbol $a$,
	the second term in \eqref{stillgoing4} also involves an acceptable
	metric correction. For the first term on the last line of \eqref{stillgoing4},
	we expand the derivative. Since $\pa \Sigma$ is a weak
	symbol and since if $a$ is a weak symbol, $\pa a = \frac{1}{1+s} a'$
	for another weak symbol $a'$, we have
		\begin{equation}
	 \pa_\mu \left( \frac{1}{r^2}
	 (\widetilde{A}^{\mu\nu r} + \widetilde{A}^{\mu r \nu})\pa_\nu\Sigma\psi\right)
	 = \frac{1}{(1+v)^2(1+s)} a \psi + \frac{1}{(1+v)^2} b^\mu \pa_\mu \psi,
	 \label{}
	\end{equation}
	for weak symbols $a, b^\mu$. The same observation shows that
	$F_{\Sigma, 2}$ is an acceptable remainder.
}

\emph{Dealing with the cubic term $r\pa_\alpha B^\alpha$}
It remains only to deal with the cubic nonlinearity $B$. Once again we
pass to null coordinates and write
\begin{equation}
 r \pa_\alpha B^\alpha = \pa_\mu(r B^\mu) + B^r,
 \label{}
\end{equation}
with the same notation as above. Now we note that the components of $B$
can all be written in the form
\begin{equation}
 B = B_0( \pa(\psi/r) + \pa(\Sigma/r) )
 Q( \pa(\psi/r) + \pa(\Sigma/r),\pa(\psi/r) + \pa(\Sigma/r))
 \label{}
\end{equation}
for a smooth function $B_0$ with $B_0(0) = 0$ and a quadratic nonlinearity $Q$ with coefficients
satisfying the weak symbol condition.
To conclude we just note that $\Sigma = a s$
and $\pa \Sigma = b$ for weak symbols $a, b$, and so
the above is of the form appearing in \eqref{Bformula}.

\end{proof}

We now want to commute the equation \eqref{centraleqnlowest} with the fields
\begin{equation}
 \mZB = \{\sBo = s\pa_u, \sBt = v\pa_v, \Omega_{ij}\}.
 \label{}
\end{equation}
Recall the notation $\widetilde{Z}_{\mB}$ from section \eqref{mZBcommutatorsec}.

\begin{lemma}
	\label{higherordereqncentral}
	With $\psi^I = \widetilde{Z}_{\mB}^I r\phi$, $\psi^I$ satisfies
	\begin{multline}
	 -4\pa_u\left( \pa_v + \frac{u}{vs} \pa_u\right) \psi^I
	 + \sDelta \psi^I + \pa_\mu\left( \frac{u}{vs} a^{\mu\nu} \pa_\nu \psi^I\right)
	 + \pa_\mu \left(\gamma^{\mu\nu} \pa_\nu \psi^I\right)
	 + \pa_\mu P_I^\mu + \pa_\mu P_{I, null}^\mu 
     {+F_{I, \mB}^1}\\
     = F_{I} + F_{\Sigma, I} {+ F_{\mB, I}^2},
	 \label{higherordercentraleqn}
	\end{multline}
	where the above quantities satisfy the following bounds when
	$|u| \lesssim s^{1/2}$.

	First, $a^{\mu\nu}, \gamma^{\mu\nu}$ are as in the previous lemma and,
		if $|\pa Z^J_{\mB}\psi|\leq C$ for $|J| \leq |I|/2+ 1$, $\gamma$
	satisfies the bounds
	\begin{align}
	 |\gamma| &\lesssim \frac{1}{1+v} |\pa \psi| + \frac{1{+s}}{(1+v)^2} |\psi|
	 + \frac{1{+s}}{(1+v)^2},
	 \label{gammaboundscentral0}
	 \\
	 (1+s)|\pa_u \gamma|  + (1+v)|\pa_v\gamma| + |{\Omega} \gamma|
	 &\lesssim
	 \sum_{|I|\leq 1}\left(\frac{1}{1+v}
	 | \pa Z_{\mB}^I \psi| + \frac{1{+s}}{(1+v)^2} |Z_{\mB}^I \psi|\right)+ \frac{1{+s}}{(1+v)^2},
	 \label{gammaboundscentral}
	\end{align}
	while $a$ satisfies
	\begin{equation}
	 |a^{\mu\nu} \pa_\mu q \pa_\nu q|
	 \lesssim |\opa q| |\pa q|.
	 \label{}
	\end{equation}

 The current $P_I$ satisfies the bounds
	\begin{align}
		|P_I|
		&\lesssim
 \sum_{\substack{|I_1| + |I_2| \leq |I|,\\ |I_1|, |I_2| \leq |I|-1}}	\frac{1}{1+v}
		|\pa \psi^{I_1}| |\pa \psi^{I_2}|
		\sum_{|J| \leq |I|-1} \frac{1}{(1+v)(1+s)} |\pa \psi^J|
		\label{PIkbdcentral}
    \end{align}
    \begin{align}
        (1+s)|\pa_u P_I| + (1+v) |\pa_v P_I| &+ |\Omega P_I|
		 \\ 
         \lesssim
         &\sum_{\substack{|I_1| + |I_2| \leq |I|+1,\\ |I_1|, |I_2| \leq |I|}}
		\frac{1}{1+v}
		|\pa \psi^{I_1}| |\pa \psi^{I_2}| +
		\sum_{|J| \leq |I|} \frac{1}{(1+v)(1+s)} |\pa \psi^J|.
		\label{ffspik}
	\end{align}

	The current $P_{I, null}$ accounts for lower-order commutations
	with the linear term verifying the null condition and satisfies the
	following estimates. The $u$-component $P^u_{I,null} = P^\alpha_{I,null}\pa_\alpha u$
	satisfies
	\begin{align}
	 |P_{I, null}^u|
	 &\lesssim
	 \frac{1}{1+v}
	 \sum_{|J| \leq |I|-1}\left(
	 \frac{1}{(1+s)^{1/2}} |\pa \psi^J| + |\opa \psi^J|\right)
	 +  \frac{1}{1+v}\sum_{|J| \leq |I|-2} |\pa \psi^J|,
	 \label{borderlinebd0}
 \end{align}
 \begin{multline}
	 (1+s)|\pa P_{I, null}^u| + (1+v)|\pa_v P_{I, null}^u|
	 + |\Omega P_{I, null}^u|
	 \\ \lesssim
	 \frac{1}{1+v}\sum_{|J| \leq |I|}\left(
	 \frac{1}{(1+s)^{1/2}} |\pa \psi^J| + |\opa \psi^J|\right)
	 + \frac{1}{1 + v}\sum_{|J| \leq |I|-1} |\pa \psi^J|,
	 \label{ffsderivbd0}
	\end{multline}
    where we are writing $\opa = (\nas, \evmB)$,
	while the remaining components satisfy
	\begin{align}
	 |P_{I, null}| &\lesssim \frac{1}{1+v} \sum_{|J| \leq |I|-1} |\pa \psi^J|,\\
	 (1+s)|\pa P_{I, null}| + (1+v)|\pa_v P_{I, null}|
	 + |\Omega P_{I, null}|
	 &\lesssim \frac{1}{1+v} \sum_{|J| \leq |I|} |\pa \psi^J|.
	 \label{borderlinebd2}
 \end{align}

 The remainder $F_I$ collects various error terms involving $\psi$
 and satisfies
 \begin{multline}
  |F_I| \lesssim
	\frac{1}{(1+v)^2}
	\sum_{|I_1| + |I_2| \leq |I|} |\pa \psi^{I_1} | |\pa \psi^{I_2}|
	+ \frac{1}{(1+v)^4}
	\sum_{|I_1| + |I_2| \leq |I|} |\psi^{I_1} | | \psi^{I_2}|\\
	+ \frac{1}{(1+v)^2} \sum_{|J| \leq |I|} |\pa \psi^J|
	+ \frac{1}{(1+v)^2(1+s)} \sum_{|J| \leq |I|} |\psi^J|,
  \label{FInonlin}
 \end{multline}
  $F_{\Sigma, I}$ collects the error terms involving the model
 profile $\Sigma$,
 \begin{equation}
  |F_{\Sigma, I}| \lesssim \frac{1}{(1+v)^2},
  \label{FSigmabd}
 \end{equation}
 {and $F_{\mB, I}^1, F_{\mB, I}^2$ collects error terms generated by commuting
 the linear terms $-4\pa_v (\pa_v + \tfrac{u}{vs}\pa_u) + \sDelta$
 with our fields. The error term $F_{\mB, I}^1$ is
 \begin{equation}
  \label{nonlinearfmBI1}
  F_{\mB, I}^1 = \sum_{|J_1| + |J_2| = |I|-1}
  -a_I^{J_1 J_2} \sDelta Z_{\mB}^{J_1}Z_{\mB}^{J_2}\psi,
\end{equation}
where the coefficients $a_I^{J_1 J_2} = 1$ if $Z^I_{\mB} = Z^{J_1}_{\mB} (v\pa_v) Z^{J_2}_{\mB}$
and $a_I^{J_1 J_2} = 0$ otherwise (so that $a_I^{J_1J_2} \equiv 0$ if there are no 
factors of $v\pa_v$ present in $Z_{\mB}^I$). The error term $F_{\mB, I}^2$ satisfies the bound
\begin{equation}
  \label{nonlinearfmBI2}
  |F_{\mB, I}^2|\lesssim 
     \frac{1+s}{(1+v)^2} \sum_{|J| \leq |I|} |\nas  \psi^J|
  + \frac{1}{(1+v)^2} \sum_{|J| \leq |I|-1} |\Omega  \psi^J|
\end{equation}

}
%
%
%
%
%
%

\end{lemma}

\begin{remark}
 For our applications, most of the currents and remainders
 appearing in the above are harmless.
 The terms $P_{I, null}$ and $F_{\mB, I}$, however, are generated
 by commuting our fields $Z_{\mB}$ with some of the linear
 operators in our equation and therefore need to be treated carefully.
 In particular,
 the fact that $P_{I,a}$ only has a factor $(1+v)^{-1}$ in front
 of the second term in \eqref{ffsderivbd0} at first glance is too large for us to
 handle (we expect a ``generic'' error term to behave like $\frac{1}{v}\frac{1}{s^{1/2}}
 \pa \psi_C$ and so we are off by a factor of $s^{1/2}$).
 This term is generated because $Z_{\mB} \tfrac{u}{vs} \sim \frac{1}{v}$
 and the null condition \eqref{intronullcondn} does not commute well with
 the rotation fields; in particular this term is generated when we consider
 quantities of the form $Z_{\mB}^I \left( \tfrac{u}{vs} a^{\alpha\beta}\pa_\alpha \psi\right)$,
 where all quantities are expressed relative to rectangular coordinates.
 If $Z_{\mB}^I$ contains at least one field $s\pa_u$ and one rotation field
 then after applying the product rule, we encounter a term like
 $\left(s\pa_u \tfrac{u}{vs}\right)\Omega a^{\mu\nu} \sim \frac{1}{v} \Omega a^{\alpha\beta}$.
 If it was not for the presence of the rotation field $\Omega$, this term would satisfy
 the null condition and could be handled, but in general $\Omega a^{\alpha\beta}$
 does not satisfy the null condition. Thankfully, this quantity only
 appears multiplied by lower-order derivatives of $\psi_C$ (since this only
 happens when at least two of our vector fields fall on the coefficients),
 and so terms of this form can be handled by integrating to the right shock
 and using that we have bounds for the field $s\pa_u$ applied to the solution.
 This is dealt with in Lemma \ref{timeintegrability-center-linear};
 See in particular the calculation starting with \eqref{extraPIlinbdleftover}.
\end{remark}

\begin{proof}

	With $I$ fixed, we will use a slight modification of the terminology from the proof
	of the previous result and will say that $F$ is an acceptable remainder
	if it can be written as a sum of terms of the form
	\begin{equation}
	 \frac{A^{\mu\nu}\pa_\mu Z_{\mB}^{I_1}\psi\pa_\nu Z_{\mB}^{I_2}\psi}{(1+v)^2} ,
 	\quad
 	\frac{B^{\mu}\pa_\mu Z_{\mB}^{I_1}\psi Z_{\mB}^{I_2}\psi}{(1+v)^3} ,
 	\quad
 	\frac{ CZ_{\mB}^{I_1}\psi Z_{\mB}^{I_2}\psi }{(1+v)^4},
 	\quad
 	\frac{A^\mu \pa_\mu Z_{\mB}^{J}\psi}{(1+v)^2},
 	\quad
 	\frac{B Z_{\mB}^{J}\psi}{(1+v)^2(1+s)}  ,
 	\quad
 	\frac{C}{(1+v)^2}
	 \label{higherorderacceptableremainder}
	\end{equation}
	where $|I_1| + |I_2| \leq |I|$
	 and $|J| \leq |I|$.
	 This is just
	 a higher-order version of \eqref{acceptableremainder}.
	It will also be convenient to say that a vector field $P$ is an ``acceptable
	current'' if it is a sum of terms of the form
	\begin{equation}
\frac{1}{1+v} a^{\mu\nu} \pa_\mu Z_{\mB}^{I_1}\psi \pa_\nu Z_{\mB}^{I_2} \psi,
\qquad \frac{1}{(1+v)(1+s)} b^\mu\pa_\mu Z^L\psi
	 \label{higherorderacceptablecurrent}
	\end{equation}
	for $|I_1|\! + \!|I_2| \leq \!\!|I|$ with $\max(|I_1|, |I_2|) \!\leq\! |I|-1$,
	and $|L| \leq |I|-1$ where the above coefficients are
	weak symbols.

	Terms of the second type in \eqref{higherorderacceptablecurrent}
	re generated by commuting our vector fields with
	the linear part of the equation, see \eqref{fmBIformula}. We remark
	that despite being linear, these terms are harmless for our estimates
	 since we expect
	bounds $|\pa Z_{\mB}^{K}| \psi\lesssim (1+s)^{-1/2}$ for small $|K|$,
	and so the second type of term in \eqref{higherorderacceptablecurrent} decays
	more quickly than the first type of term in \eqref{higherorderacceptablecurrent}.

	Then acceptable remainders satisfy the bounds \eqref{FSigmabd}
	and \eqref{FInonlin} and acceptable currents satisfy the bounds
	\eqref{PIkbdcentral}-\eqref{ffspik}.
	%

In Lemma \ref{centraleqnlowestprop}, we worked in null coordinates because that made it easier
to see what happened when we expanded around the model shock profile.
To commute with our fields (in particular the rotation fields $\Omega$)
it is somewhat awkward to work in null coordinates and so it winds up
being easier to derive the higher-order equations if we
 go back to expressing all quantities in rectangular coordinates.
 For this we use the formula $\pa_\mu J^\mu = \pa_\alpha J^\alpha - \frac{2}{r}J^r$
 with $J^r = \omega_i J^i$ and where the quantities on the right-hand side
 are expressed in rectangular coordinates. Then the equation
 \eqref{centraleqnlowest} reads
 \begin{equation}
	 \Box_{\mB}\psi
	+ \pa_\alpha \left(\frac{u}{vs}a^{\alpha\beta}\pa_\beta \psi\right)
	+ \pa_\alpha \left(\gamma^{\alpha\beta}\pa_\beta \psi\right)
	 = F^\prime + F_\Sigma,
  \label{effectiveinrect}
 \end{equation}
 where we remind the reader of the notation $\Box_{\mB} = -4\pa_u\left(\pa_v+ \frac{u}{vs} \pa_u\right) + \sDelta$
 and where we have introduced
 \begin{equation}
	F^\prime = F +
	\frac{2}{r} \frac{u}{vs}a^{r\beta}\pa_\beta \psi+ \frac{2}{r} \gamma^{r\beta}\pa_\beta \psi.
  \label{}
 \end{equation}
 Recalling that the $a$ are weak symbols, the formula
 \eqref{gammaCformula} for $\gamma$ and the formula
 \eqref{FCformula} for $F$, it follows that
 $Z_{\mB}^I F^\prime$ is an acceptable remainder in the sense of
 \eqref{higherorderacceptableremainder}. Also, from the formula
 \eqref{FSigmaformula} it is clear that $Z_{\mB}^I F_{\Sigma}$ is
 an acceptable remainder.

 We now commute the equation \eqref{effectiveinrect} with our fields.

 \emph{Step 1: Commutation with $\Box_{\mB}$}

 By Lemma \ref{boxmBcommutator}, we have the identity
 \begin{equation}
  \widetilde{Z}_{\mB}^I \Box_{\mB}  \psi = \Box_{\mB} \widetilde{Z}_{\mB}^I \psi
  + \pa_u P_{\mB, I}[\psi] + {F_{\mB, I}^1[\psi] + F_{\mB, I}^2[\psi],}
  \label{}
 \end{equation}
 where the current $P_{\mB, I}$ and remainders {$F_{\mB, I}^1, F_{\mB,I}^2$} are given in
 \eqref{fmBIformula}, 
 \eqref{fmBIformula1}-\eqref{fmBIformula2}.
 { By \eqref{fmBIformula}, $P_{\mB, I}[\psi]$ is an acceptable current since it is a sum
 of the terms of the second type in \eqref{higherorderacceptablecurrent}. The quantities
 $F_{\mB, I}^1, F_{\mB, I}^2 $ are not acceptable remainders but $F_{\mB, I}^1$
 is recorded in \eqref{nonlinearfmBI1},
 and by \eqref{fmBI2bd}, $F_{\mB, I}^2$ satisfies the bound \eqref{nonlinearfmBI2}.

}

 \emph{Step 2: Commutation with the nonlinear terms}

 We now commute with the nonlinearity $\pa_\alpha (\gamma^{\alpha\beta}\pa_\beta \psi)$.
 The metric
 perturbation $\gamma^{\alpha\beta}$ is given by
   \begin{equation}
    \gamma^{\alpha\beta} = \gamma^{\alpha\beta\delta}
  	\pa_\delta \psi +
  \frac{1+s}{(1+v)^2} A_1^{\alpha\beta} \psi
  + \frac{1+s}{(1+v)^2} A_2^{\alpha\beta}
	+ r B^\alpha
    \label{splitgammaintogoodandbad}
   \end{equation}
	 where the above coefficients are weak symbols and $B^\alpha$ is a cubic
	 nonlinearity.
	 This follows from the explicit formula \eqref{gammaCformula} after
	 expressing all quantities in rectangular coordinates.
	 To handle the term
	 $\widetilde{Z}_{\mB}^I\pa_\alpha (\gamma^{\alpha\beta\delta}\pa_\beta \psi \pa_\delta \psi)$,
	  we use Lemma \ref{centralcommcurrent},
	 which gives
	 \begin{equation}
		 \widetilde{Z}_{\mB}^I \pa_\alpha(\gamma^{\alpha\beta\delta}\pa_\beta \psi\pa_\delta \psi)
		= \pa_\alpha
		\left(\check{\gamma}^{\alpha\beta\delta}\pa_\beta \psi\pa_\delta
		\widetilde{Z}_{\mB}^I \psi\right)
	   + \pa_\alpha P_I^\alpha + F_{I},
	  \label{startofboxmbcomm}
	 \end{equation}
	 where $\check{\gamma}^{\alpha\beta \delta} = \gamma^{\alpha \beta\delta} +\gamma^{\alpha \delta \beta}$.
	 The quantities $P_I$ and $F_I$ are
	 \begin{equation}
	  P_I^\alpha = \frac{1}{1+v}
 	 \sum_{{{|I_1| + |I_2| \leq |I| \atop |I_1|,|I_2|\le |I|-1}}}
 	a^{\alpha\beta\delta} \pa_\beta Z_{\mB}^{I_1}\psi \pa_\delta Z_{\mB}^{I_2} \psi,
 	\qquad
 	F_I = \frac{1}{(1+v)^2} \sum_{|I_1| + |I_2| \leq |I|}
  a^{\alpha\beta\delta} \pa_\beta Z_{\mB}^{I_1}\psi \pa_\delta Z_{\mB}^{I_2} \psi
	  \label{}
	 \end{equation}
	 where the coefficients are strong symbols. Using the formula
	 $\div X = \pa_\mu X^\mu + 2r^{-1} X^r$ to express
	 \eqref{startofboxmbcomm} in null coordinates, we find
	 \begin{multline}
	  \pa_\alpha
		\left(\check{\gamma}^{\alpha\beta\delta}\pa_\beta \psi\pa_\delta
		\widetilde{Z}_{\mB}^I \psi\right)
	   + \pa_\alpha P_I^\alpha + F_{I}\\
		 = \pa_\mu
		 \left(\check{\gamma}^{\mu\nu\nu'}\pa_\nu \psi\pa_{\nu'} \widetilde{Z}_{\mB}^I \psi\right)
 	   + \pa_\mu P_I^\mu
		 + \frac{2}{r}\check{\gamma}^{r\nu\nu'}\pa_\nu \psi\pa_{\nu'} \widetilde{Z}_{\mB}^I \psi
		 + \frac{2}{r} P_I^r
		 + F_{I},
	  \label{}
	 \end{multline}
	 where the last three terms are an acceptable remainder, the second term involves
	 an acceptable current, and, recalling that $|\gamma^{\mu\nu\nu'}|
	 \lesssim (1+v)^{-1}$, the first term involves an acceptable metric correction.

	 To handle the remaining terms
	 from \eqref{splitgammaintogoodandbad} we use \eqref{ZBdivcommvariant2} which gives
	 \begin{multline}
	  \widetilde{Z}_{\mB}^I \pa_\alpha
		 \left( \frac{1+s}{1+v} A_1^{\alpha\beta} \psi \pa_\beta\psi\right)
		 = \pa_\alpha Z_{\mB}^I \left( \frac{1+s}{1+v} A_1^{\alpha\beta} \psi \pa_\beta\psi\right)
		 \\
		 + \pa_\beta J_{\pa_\alpha, I}^\beta\left[\frac{1+s}{1+v} A_1^{\alpha\beta} \psi \pa_\beta\psi\right]
		 + F_{\pa_\alpha, I}\left[\frac{1+s}{1+v} A_1^{\alpha\beta} \psi \pa_\beta\psi\right],
	  \label{}
	 \end{multline}
	 where the last two terms are defined as in \eqref{ZBdivcommvariant2struct}.
	 Repeatedly using \eqref{ZBdivcommvariant2} and re-writing in null
	 coordinates shows that this can be written in
	 terms of an acceptable metric correction, the divergence of an acceptable
	 current, and an acceptable remainder. The contribution from the third term
	 in \eqref{splitgammaintogoodandbad} can be handled in the same way.

	 It remains to handle the contribution from the cubic and higher-order terms.
	 These are of course simpler to deal with but we include here a brief sketch
	 of how to handle such terms.
	 We just just consider the term $\pa_\alpha (Z_{\mB}^I (rB^\alpha))$,
	 since the commutator $[\pa_\alpha, Z_{\mB}^I](r B^\alpha)$ just generates
	 similar terms.
	 For this we use the formula \eqref{Bformula}. Since we can write
	 $\pa(\psi/r) + \pa (\Sigma/r) = \frac{1}{1+v} a\pa \psi + \frac{1}{(1+v)^2}b \psi
	 + \frac{1}{1+v} c$ for weak symbols $a,b,c$, it follows
	 from \eqref{Bformula} that $Z_{\mB}^I (rB^\alpha)$ can be written as a sum
	 of terms of the form
	 \begin{equation}
	  \frac{1}{1+v} B^\prime
		 \left[\prod_{i = 1}^j
		 \left(\frac{1}{1+v}  a\pa Z_{\mB}^{I_i} \psi + \frac{1}{(1+v)^2}bZ_{\mB}^{I_i}\psi + \frac{1}{1+v} c\right)\right]
		 Q_{I_{j+1}, I_{j+2}},
		 \label{modelofcubic}
	 \end{equation}
	 where in the above, $B^\prime$ is a smooth function
	 depending on $\pa (\psi/r) + \pa (\Sigma/r)$,
	 the indices satisfy $|I_1| + \cdots +|I_{j+2}| \leq |I|$ and $j\geq 1$,
	 and the quantities $Q_{K, L}$ are sums of the following types of terms,
	 \begin{align}
	  &Q(\pa Z_{\mB}^{K} \psi, \pa Z_{\mB}^{L}\psi),
		\quad
			  \frac{1}{1+v}Q( Z_{\mB}^{K} \psi, \pa Z_{\mB}^{L}\psi),
				\quad
					  \frac{1}{(1+v)^2}Q( Z_{\mB}^{K} \psi, Z_{\mB}^{L}\psi),\\
				&A \cdot \pa \psi,
				\quad \frac{1+s}{1+v} A \cdot \pa Z_{\mB}^{K}\psi
						\quad
						\frac{1}{1+v} A Z_{\mB}^{K}\psi,
						\frac{1+s}{(1+v)^2} A Z_{\mB}^{K}\psi
						\\
						& A, \quad \frac{1+s}{1+v} A,
						\quad
						\frac{(1+s)^2}{(1+v)^2} A,
	  \label{}
	 \end{align}
	 where the $Q$ are quadratic nonlinearities with smooth coefficients verifying
	 the weak symbol condition and the quantities $A$ are also weak symbols.
	 Quantities of the form in \eqref{modelofcubic} are consistent with our bounds,
	 and to prove our result it remains only to commute our vector fields with the linear term
	 $\pa_\alpha( \tfrac{u}{vs} a^{\alpha\beta}\pa_\beta \psi)$ verifying the null
	 condition.

\emph{Step 3: Commutation with the term verifying the null condition}
 	This is more delicate than the above computations because we need to keep
	better track of the coefficients so we can exploit the null condition.
	By Lemma \ref{ZBdivcommvariantlem}, we have
	\begin{equation}
		Z_{\mB}^I \pa_\alpha \left(\frac{u}{vs} a^{\alpha\beta}\pa_\beta \psi\right)
		= \pa_\alpha\left( Z_{\mB}^I \left(\frac{u}{vs} a^{\alpha\beta}\pa_{\beta} \psi\right)\right)
		+ \pa_\beta P_{\pa_\alpha, I}^\beta[\tfrac{u}{vs} a^{\alpha\beta'}\pa_{\beta'} \psi]
		+ F_{\pa_\alpha, I}[\tfrac{u}{vs} a^{\alpha\beta'}\pa_{\beta'} \psi],
		\label{ffsx}
	\end{equation}
	where
	\begin{align}
	 P_{\pa_\alpha, I}^\beta[\tfrac{u}{vs} a^{\alpha\beta'}\pa_{\beta'} \psi]
	 &=\sum_{|J| \leq |I|-1} c_{\alpha J}^{\beta I}
	 Z_{\mB}^J \left(\frac{u}{vs}a^{\alpha\beta'}\pa_{\beta'} \psi\right)
	 \label{Pnull}\\
	 F_{\pa_{\alpha}, I}[\tfrac{u}{vs} a^{\alpha\beta'}\pa_{\beta'} \psi]
	 &=\frac{1}{1+v} \sum_{|J| \leq |I|} b_{\alpha J}^{ I} Z_{\mB}^J
	 \left(\frac{u}{vs}a^{\alpha\beta'}\pa_{\beta'} \psi\right),
	 \label{Fnull}
	\end{align}
	where the coefficients $c, b$ are weak symbols.
	The quantity in \eqref{Fnull} is an acceptable remainder, but the first
	term on the right-hand side of \eqref{ffsx} and the current in \eqref{Pnull}
	are more problematic, because even though $a^{\alpha\beta}$
	verifies the null condition, the quantities $Z_{\mB}^K a^{\alpha\beta}$ do not.
	Also, since $\sBo \tfrac{u}{s} = 1$, we lose a power of $s$ when the vector
	fields land on $\tfrac{u}{vs}$.

	For the upcoming calculation, with $I$ fixed
	we will say that a vector field $P$ is a ``borderline current''
	if its components $P^\alpha$ expressed in rectangular coordinates
	 can be written as a sum of terms of the following types of terms,
	\begin{alignat}{2}
	 &\frac{1}{1+v} a^{\alpha\beta}\pa_\beta Z_{\mB}^J \psi,
	 \quad
	 \frac{u}{vs} b^{\alpha\beta}\pa_\beta Z_{\mB}^J \psi,
	 &&\quad \text{ for }|J| \leq |I|-1,\label{borderlinecurrents0}
	 \\
	 &\frac{1}{1+v} c^{\alpha\beta} \pa_\beta Z_{\mB}^L \psi, &&\quad \text{ for }|L| \leq |I|-2,
	 \label{borderlinecurrents}
 \end{alignat}
 $a, b, c$ are weak symbols
	and $a$ additionally satisfies the null condition
	$a^{\alpha\beta}\pa_\alpha u\pa_\beta u = 0$.
The next lemma reduces the proof of Lemma \ref{higherordereqncentral}
to showing that the quantities appearing
in \eqref{ffsx} and \eqref{Pnull} involve borderline and acceptable currents.
	\begin{lemma}
		\label{borderlinelemma}
	 If $P = P^\alpha \pa_{\alpha}$ is a borderline current, then with
	 $P^u = P^\alpha\pa_\alpha u$, $P$ satisfies the estimates
	 \eqref{borderlinebd0}-\eqref{borderlinebd2}.
	\end{lemma}
	\begin{proof}
		The bound \eqref{borderlinebd2} follows immediately from
		the definitions, since $|u| \lesssim s^{1/2}$ in $D^C$.
		To get
		\eqref{borderlinebd0}, we just write $\pa_\beta = \pa_\beta u \pa_u + \pa_\beta v \pa_v
		+ \slashed{\pa}_\beta$ with $\slashed{\pa}_\beta$ denoting angular differentiation
		and since $a^{u\beta}\pa_\beta u = 0$ by the null condition.
		The bound \eqref{ffsderivbd0} is clear if $P$ is of the last
		type appearing in \eqref{borderlinecurrents} since then
		$Z_{\mB} P$ is bounded by the last term in \eqref{ffsderivbd0}.
		To
		get the bound \eqref{ffsderivbd0} for the first two
		types of terms in \eqref{borderlinecurrents0}, we just note
		that if the derivative falls on either the coefficients
		$a$ or $\frac{u}{s}$, there are no more
		than $|I|-1$ derivatives falling on $\psi$ and so such
		terms are bounded by the last term in \eqref{ffsderivbd0}. If
		the derivative instead falls on $\pa_\beta Z_{\mB}^J \psi$,
		then we write $Z_{\mB} \pa_\beta Z_{\mB}^J \psi = \pa_\beta Z_{\mB} Z_{\mB}^J\psi + [Z_{\mB}, \pa_\beta] Z_{\mB}^J \psi$.
		The contribution from the first term here is bounded
		by the first two terms in \eqref{ffsderivbd0},
		and by
		Lemma \ref{ZBdivcommvariantlem}, the commutator $[\pa_\beta, Z_{\mB}]Z_{\mB}^J\psi$ generates
		another quantity bounded by the last term in \eqref{ffsderivbd0}.

	\end{proof}

 We now claim that for each $K$ with $|K| \leq |I|$, we can write
 \begin{equation}
  Z_{\mB}^K \bigg( \frac{u}{vs} a^{\alpha\beta}\pa_\beta \psi\bigg)
	= \frac{u}{vs} a^{\alpha\beta} \pa_\beta Z_{\mB}^K \psi
	+ P_{borderline, K}^{\alpha} + P_{acceptable, K}^{\alpha}, \qquad |K| \leq |I|
  \label{nullformulacommute}
 \end{equation}
 for a borderline current $P_{borderline, K}^{\alpha}$
 and an acceptable (in the sense of \eqref{higherorderacceptablecurrent})
 current  $P_{acceptable, K}^{\alpha}$.
  Assuming this claim, the first term in \eqref{ffsx} takes the form
	\begin{equation}
	 \pa_\alpha \left( Z_{\mB}^I \left( \frac{u}{vs}a^{\alpha\beta}\pa_\beta \psi\right)\right)
	 = \pa_\alpha \left( \frac{u}{vs} a^{\alpha\beta}\pa_\beta Z_{\mB}^I \psi\right)
	 + \pa_\alpha (P_{borderline, I}^\alpha + P^\alpha_{acceptable, I}),
	 \label{toporderborderline}
	\end{equation}
	which is of the correct form. Similarly, the currents $P^\beta_{\pa_\alpha, I}$
	from \eqref{Pnull} can be written in the form
	\begin{align}
	 P_{\pa_\alpha, I}^\beta[\tfrac{u}{vs} a^{\alpha\beta'}\pa_{\beta'} \psi]
	 &= \sum_{|J| \leq |I|-1}
	 c^{\beta I}_{\alpha J}
	 \left(\frac{u}{vs} a^{\alpha \beta'} \pa_{\beta'}
	 Z^J_{\mB}\psi + P^{\alpha}_{borderline, J} + P^\alpha_{acceptable, J}\right).
	 \\
	 &=\sum_{|J| \leq |I|-1}c^{\beta I}_{\alpha J}
	 \left(\widetilde{P}^{\alpha}_{borderline, J} + P^\alpha_{acceptable, J}\right),
	 \label{lowerorderborderline}
	\end{align}
	where $\widetilde{P}^\alpha_{borderline, J} =
	\frac{u}{vs} a^{\alpha \beta'} \pa_{\beta'}
	Z^J_{\mB}\psi + P^\alpha_{borderline, J}$ is a borderline current for $|J| \leq
	|I|-1$. By Lemma \ref{borderlinelemma}, such terms satisfy the needed
	estimates, and it remains only to prove the claim \eqref{nullformulacommute}.
 \end{proof}

\begin{proof}[Proof of the claim \eqref{nullformulacommute}]

		We start by writing
		\begin{align}
		 Z_{\mB}^K \bigg( \frac{u}{vs} &a^{\alpha\beta}\pa_\beta \psi\bigg)
		 - \frac{u}{vs} a^{\alpha\beta} \pa_\beta Z_{\mB}^K \psi\\
		 &=
		 \sum_{\substack{|K_1| + |K_2| + |K_3| \leq |K|,\\ |K_3| \leq |K|-1}}c^K_{K_1K_2K_3}
		 \left(Z_{\mB}^{K_1} \frac{u}{vs}\right)
		 \left(Z_{\mB}^{K_2} a^{\alpha\beta}\right)
		 (\pa_\beta Z_{\mB}^{K_3}\psi)\\
		 &+ \sum_{\substack{|K_1| + |K_2| + |K_3| \leq |K|,\\ |K_3| \leq |K|-1}}c^K_{K_1K_2K_3}
		 \left(Z_{\mB}^{K_1} \frac{u}{vs}\right)
		 \left(Z_{\mB}^{K_2} a^{\alpha\beta}\right)
		 ([\pa_\beta, Z_{\mB}^{K_3}] \psi)
		 + \frac{u}{vs}a^{\alpha\beta} [\pa_\beta,Z_{\mB}^K] \psi .
		 \label{ZKformula}
		\end{align}
		for constants $c$. We start with the terms on the second line.
		Let $T_{K_1 K_2 K_3}^{\alpha} = \left(Z_{\mB}^{K_1} \frac{u}{vs}\right)
		\left(Z_{\mB}^{K_2} a^{\alpha\beta}\right)
		(\pa_\beta Z_{\mB}^{K_3}\psi)$.
		When $|K_1| = 0$, we ignore the structure of the $a$ terms and write
		$Z_{\mB}^{K_2} a^{\alpha\beta} = b^{\alpha\beta}_{K_2}$ for weak symbols
		$b$ and the result is that
		\begin{equation}
			T_{0 K_2 K_3}^\alpha =
		 \frac{u}{vs} b^{\alpha\beta}_{K_2}
		 (\pa_\beta Z_{\mB}^{K_3}\psi),
		 \label{}
		\end{equation}
		where the coefficients are weak symbols.

		When $|K_1| \geq 1$, we write $Z_{\mB}^{K_1} \frac{u}{vs} = \frac{1}{1+v} b_{K_1}$
		for a strong symbol $b$. If $|K_2| = 0$, we then have
		\begin{equation}
		 T_{K_1 0 K_3}^{\alpha}
		 = \frac{b_{K_1} }{1+v} a^{\alpha\beta}
		 (\pa_\beta Z_{\mB}^{K_3} \psi),
		 \label{}
		\end{equation}
		and if $|K_2| \geq 1$ we ignore the structure of the coefficients
		$a$ and write $Z_{\mB}^{K_2} a^{\alpha\beta} = b^{\alpha\beta}_{K_2}$ to write
		\begin{equation}
		 T_{K_1 K_2 K_3}^{\alpha} = \frac{b^{\alpha\beta}_{K_2}}{1+v} \pa_\beta Z_{\mB}^{K_3} \psi,
		 \qquad |K_1| + |K_2| \geq 2.
		 \label{}
		\end{equation}
		We therefore have
		\begin{align}
		 \sum_{\substack{|K_1| + |K_2| + |K_3| \leq |K|,\\ |K_3| \leq |K|-1}}& c^K_{K_1 K_2 K_3 } T_{K_1 K_2 K_3}^\alpha\\
		 &= \sum_{\substack{|K_1| + |K_2| + |K_3| \leq |K|,\\ |K_3| \leq |K|-1}}c^K_{0 K_2 K_3 } T_{0 K_2 K_3}^\alpha
		 + \sum_{\substack{|K_1| + |K_2| + |K_3| \leq |K|,\\ |K_1| \geq 1}}c^K_{K_1 0 K_3 } T_{K_1 0 K_3}^\alpha\\
		 &\qquad+ \sum_{\substack{|K_1| + |K_2| +  |K_3| \leq |K|, \\ |K_1| + |K_2| \geq 2}} c^K_{K_1 K_2 K_3 } T_{K_1 K_2 K_3}^\alpha\\
		 &= \frac{u}{vs} \sum_{|L| \leq |K|-1} b^{\alpha\beta}_L \pa_\beta Z_{\mB}^L \psi
		 + \frac{1}{1+v} \sum_{|L| \leq |K|-1} b_L a^{\alpha\beta} \pa_\beta Z_{\mB}^L \psi\\
		 &\qquad+ \frac{1}{1+v} \sum_{|L| \leq |K|-2} d_L^{\alpha\beta} \pa_\beta Z_{\mB}^L \psi,
		 \label{}
	 \end{align}
	 where the coefficients are weak symbols \eqref{weaksymbol}. This shows that the quantity
	 on the left-hand side is a borderline current in $D^C_t$.

	 It remains to prove that the terms on the last line of \eqref{ZKformula}
	 are of the form \eqref{nullformulacommute}. This is a bit
	 easier since they are lower-order and we just need
	 to establish that they are as in \eqref{borderlinecurrents}. By
	 \eqref{ZBdivcommvariant}, we can write the commutators in the form
	 \begin{equation}
	  [\pa_\beta, Z_{\mB}^{K'}] \psi = \sum_{|K''| \leq |K'|-1}
		c_{\beta K''}^{\alpha K'} \pa_\alpha Z_{\mB}^{K''} \psi,
	  \label{}
	 \end{equation}
	 where the coefficients are weak symbols, where we write
	 $\Omega_{ij} = r \omega_i \pa_j - r \omega_j \pa_i$ to express
	 the last term in \eqref{ZBdivcommvariant} in rectangular coordinates.
	 As a result, the last term on the last line of \eqref{ZKformula}
	 is a linear combination of quantities of the form
	 \begin{equation}
	  \frac{u}{vs} a^{\alpha\beta'} c_{\beta'}^{\alpha'}\pa_{\alpha'}
		 Z_{\mB}^{K'} \psi, \qquad |K'| \leq |K|-1,
	  \label{highestcomm}
	 \end{equation}
	 for a weak symbol $c_{\beta'}^{\alpha'}$, while the remaining terms
	 in \eqref{ZKformula} are instead of the form
	 \begin{equation}
	   c\cdot
		 \left(Z_{\mB}^{K_1} \frac{u}{vs} \right)
		 \left(Z_{\mB}^{K_2} a^{\alpha\beta}\right)
		\pa_{\alpha'} Z_{\mB}^{K''} \psi,
		\qquad |K_1| + |K_2| + |K''| \leq |K|, \quad |K''| \leq |K|-2.
	  \label{lowercomm}
	 \end{equation}
	 Since $|K| \leq |I|$, the terms in \eqref{lowercomm} are of
	 the last type appearing in \eqref{borderlinecurrents},
	 while the terms in \eqref{highestcomm} are of the second
	 type in \eqref{borderlinecurrents0}.

%
%
%
 \end{proof}

\section{The Rankine-Hugoniot conditions}
\label{RHsec}

The goal of this section is to prove some consequences of the Rankine-Hugoniot
conditions. We will use these conditions to give boundary conditions along
the timelike sides of each shock, to get an evolution equation for the positions
of the shocks, and finally to get control over angular derivatives of
the functions $B^A$ which define the shocks.

\begin{lemma}[The equations for the positions of the shocks]
	\label{derivofbetaeqn}
	At the shock $\Gamma^A$, with $s = \log v$, we have
 \begin{equation}
     \pa_s B^A - \frac{1}{2s} B^A = \frac{1}{2}[\pa_u \psi] - \frac{s}{u}[\pa_s\psi] + 
     s^{1/2} F_A^\prime,
  \label{betaevol}
 \end{equation}
 and
 \begin{equation}
  \nas B^A = -\frac{s}{u}[\nas \psi] + \slashed{F}_A,
  \label{betaangle}
 \end{equation}
 where $F_A^\prime$ consists of terms which are at least quadratic in derivatives of
 $\psi$,
\begin{equation}
  \label{betaevolF}
    F_A^\prime = 
    \frac{1}{2} \frac{s^{1/2}}{u}\frac{[\pa_u\psi]^2}{1+\frac{s}{u}[\pa_u\psi]}
    - \frac{1}{2}\frac{s^{3/2}}{u^2}[\pa_s\psi][\pa_u\psi] 
    - \frac{s^{5/2}}{u^3}\frac{[\pa_s\psi][\pa_u\psi]^2}{1+\frac{s}{u}[\pa_u\psi]},
\end{equation}
and similarly,
\begin{equation}
  \label{}
  \slashed{F}_A = -\frac{s^2}{u^2} \frac{[\nas \psi][\pa_u\psi]}{1 + \frac{s}{u}[\pa_u\psi]}
\end{equation}

\end{lemma}
\begin{remark}
We remind the reader that along the shocks, $|u| \sim s^{1/2}$. We also expect to have bounds
$|\pa_u\psi|\lesssim \epsilon s^{1/2}$ for a small parameter $\epsilon$ and so
$1 + \tfrac{s}{u} [\pa_u\psi] \sim 1 + \epsilon$. As a result, we have the bounds
\begin{multline}
  \label{}
  |F_A^\prime| \lesssim (1+s)^{1/2} \left(|\pa\psi_+ |^2 + |\pa \psi_-|^2 \right) 
  + (1+s) \left(|\pa\psi_+| |\pa_s\psi_+| + |\pa \psi_-||\pa_s\psi_+| \right) 
  \\
  + (1+s)^{3/2} \left(|\pa\psi_+|^2 |\pa_s \psi_+| + |\pa\psi_-|^2 |\pa_s\psi_-| \right),
\end{multline}
 where $(\pa\psi)_{\pm}$ denotes the restriction of derivatives of the potentials
 $\psi_{\pm}$ defined on either side of the shocks to the shocks.

\end{remark}

\begin{remark}
\label{evoluteqnrmk}
If we fix notation so that $[q] = q_R - q_C$ at the right shock and $[q] = q_L - q_C$
at the left shock, then we can write $[\pa_v \psi] = [ \evg \psi] + \frac{u}{vs}[\pa_u\psi]$,
and then \eqref{betaevol} reads
\begin{equation}
  \label{}
  \pa_s B^A - \frac{1}{2s}B^A = -\frac{1}{2}[\pa_u\psi] +s^{1/2} F_A,
\end{equation}
where
\begin{equation}
    \label{betaevolF2}
  F_A = \frac{1}{2} \frac{s^{1/2}}{u}\frac{[\pa_u\psi]^2}{1+\frac{s}{u}[\pa_u\psi]}
    - \frac{1}{2}\frac{s^{3/2}}{u^2}[\pa_s\psi][\pa_u\psi] 
    - \frac{s^{5/2}}{u^3}\frac{[\pa_s\psi][\pa_u\psi]^2}{1+\frac{s}{u}[\pa_u\psi]} 
    - \frac{u}{s}[\evg \psi].
\end{equation}
By the upcoming Lemma \ref{tlbc}, the quantity $[\evg \psi]$ can be treated as
a nonlinear error term. 
\end{remark}

\begin{lemma}[The boundary conditions]
	\label{tlbc}
 Let $\psi = r\Phi - \Sigma$ where $\Sigma = \frac{u^2}{2s}$ between the shocks
 and $\Sigma = 0$ otherwise. If the jump conditions \eqref{introRH1} for $\Phi$
 are satisfied then at the left shock
 \begin{multline}
  \pa_v \psi_L + \frac{1}{v} Q_L(\pa \psi_L, \pa \psi_L)
	- \nas^i\psi_L \nas_i B^L
	= \left(\pa_v + \frac{1}{vs} \pa_u \right)\psi_C
	- \nas^i \psi_C \nas_i B^L
	+ \frac{1}{v} Q_C(\pa \psi_C, \pa \psi_C)\\
	+ [G'(\psi_L, \psi_C, B^L)],
  \label{lefttlbc}
 \end{multline}
 and at the right shock,
  \begin{multline}
		\left(\pa_v + \frac{1}{vs} \pa_u \right)\psi_C
 	- \nas^i \psi_C \nas_i B^R
 	+ \frac{1}{v} Q_C(\pa \psi_C, \pa \psi_C)
	=
   \pa_v \psi_R + \frac{1}{v} Q_R(\pa \psi_R, \pa \psi_R)
 	- \nas^i\psi_R \nas_i B^R
 	 \\
 	+ [G'(\psi)],
   \label{righttlbc}
  \end{multline}
 where $G'$ has the following structure,
 \begin{equation}
  G'(\psi) = \frac{1}{v}\widetilde{Q}_1(\pa \Sigma, \pa \psi) + \frac{1}{v}\widetilde{Q}_2(\pa \Sigma,
	\pa \Sigma)
	+ \frac{1}{v}Q^\alpha(\pa \Psi, \pa \Psi) \pa_\alpha B^A
	+ R(\pa \Psi, \pa \Psi) + R^\alpha(\pa \Psi, \pa \Psi)\pa_\alpha B^A.
  \label{Gstructure}
 \end{equation}
 In the above, the $Q, \widetilde{Q}$ are quadratic forms
 \begin{equation}
  Q(\pa q, \pa q) = Q^{\alpha\beta}(\omega)\pa_\alpha q \pa_\beta q,
  \label{}
 \end{equation}
 for smooth functions $Q^{\alpha\beta}$ satisfying
 the strong symbol condition \eqref{strongsymbol}, and where the $\widetilde{Q}$
 additionally verify the null conditition $\widetilde{Q}(\pa u, \pa u) = 0$.
 Finally, the quantities $R, R^\alpha$ are of the form
 \begin{equation}
  R = \frac{u}{v^2} Q(\pa \psi, \pa \psi) + r B (\pa \Phi) + \frac{1}{r^2}
	\psi L\psi + \frac{1}{r^3} a \psi^2,
  \label{}
 \end{equation}
 where $L = L^\alpha(u,v, \omega)\pa_\alpha$ and $a = a(u,v,\omega)$
 for symbols $L^\alpha$ and $a$, and where $B(\xi)$ vanishes to third order
 at $\xi = 0$.
\end{lemma}

\begin{remark}[Eliminating derivatives of $B^A$ from the boundary conditions]
    \label{eliminaterem}
By Lemma \ref{derivofbetaeqn}, we can re-write the terms in $G'$ involving
 derivatives of $B^A$ in terms of $B^A$ and derivatives of $\psi$,
 \begin{multline}
     G'{(\psi)} = \frac{1}{v}\widetilde{Q}_1(\pa \Sigma, \pa \psi) + \frac{1}{v}\widetilde{Q}_2(\pa \Sigma,
	\pa \Sigma)
	+  \left(\frac{1}{v}Q^v(\pa \Psi, \pa \Psi)+ R^v(\pa \Psi, \pa \Psi)\right)
	\left(\frac{1}{2vs} B^A - \frac{1}{2v} [\pa_u\psi] + F\right)\\
	+\left(\frac{1}{v}\slashed{Q}^i(\pa \Psi, \pa \Psi)+ \slashed{R}^i(\pa \Psi, \pa \Psi)\right)
	\left( -\frac{s}{u}[\nas \psi] + \slashed{F}_i\right)
	+ R(\pa \Psi, \pa \Psi).
  \label{alternateGexpression}
 \end{multline}
Moreover, if we add the term $\nas \psi_L\cdot \nas B^L$ to both sides of
\eqref{lefttlbc} and use \eqref{betaangle} to express $\nas B^L$ in
terms of $[\nas \psi]$, we can further re-write
\eqref{lefttlbc} in the form
\begin{equation}
 Y_L^- \psi_L = Y_L^+ \psi_C + G,
 \label{YL0}
\end{equation}
where
\begin{equation}
    G = [G'{(\psi)}] - \frac{s}{u}[\nas\psi]^2 - [\nas \psi]\cdot \slashed{F},
 \label{Gexpression}
\end{equation}
 with $G'$ as in \eqref{alternateGexpression} and where
\begin{equation}
 Y_L^-\psi_L = \evm \psi_L + \frac{1}{v} Q_L(\pa \psi_L, \pa \psi_L),
 \qquad
 Y_L^+\psi_C = \evmB \psi_C + \frac{1}{v} Q_C(\pa \psi_C, \pa \psi_C).
 \label{}
\end{equation}
The point of the identity \eqref{YL0} is that
it does not involves any derivatives of $B^L$.
If it were not for this observation, there would be an
apparent loss of derivatives: to control $\ev \psi_L$
along the boundary would require a bound for $\nas B^L$,
but from the transport equation \eqref{betaevol}, a bound
for this quantity would appear to require a bound for
$\nas \pa_u \psi$ (on both sides of the shock), which is one
more derivative than we can afford at this level.

In the same way, we can write
\eqref{righttlbc} in the form
\begin{equation}
 Y_R^- \psi_C = Y_R^+ \psi_R + G,
 \label{}
\end{equation}
where
\begin{equation}
 Y_R^-\psi_C = \evmB \psi_L + \frac{1}{v} Q_C(\pa \psi_C, \pa \psi_C),
 \qquad
 Y_R^+\psi_R = \evm \psi_R + \frac{1}{v} Q_R(\pa \psi_R, \pa \psi_R).
 \label{YR0}
\end{equation}
\end{remark}

\begin{proof}[Proof of Lemma \ref{derivofbetaeqn}]
 Since the field $T^A = \pa_v + \pa_v B^A \pa_u$ is tangent
 to the shock $\Gamma^A$ at $\Gamma^A$, it follows that
 $[v T^A\Psi] = 0$. Rearranging this identity, we find
 \begin{equation}
  \pa_s B^A + \frac{[\pa_s \Psi]}{[\pa_u \Psi]} = 0.
  \label{startofBAeq}
 \end{equation}
 Writing $1/(1+\tfrac{s}{u}[\pa_u\psi]) = 1 + \tfrac{s}{u}[\pa_u\psi]/(1+\tfrac{s}{u}[\pa_u\psi])$,
 we find
 \begin{equation}
  \frac{1}{[\pa_u \Psi]} = \frac{1}{[\pa_u \Sigma + \pa_u \psi]}
	= \frac{1}{\tfrac{u}{s} + [\pa_u\psi]}
    = \frac{s}{u} + \frac{s^2}{u^2} \frac{[\pa_u\psi]}{1 + \frac{s}{u}[\pa_u\psi]}
    = \frac{s}{u} + \frac{s^2}{u^2} [\pa_u \psi] + 
    \frac{s^3}{u^3}\frac{[\pa_u\psi]^2}{1 + \frac{s}{u}[\pa_u\psi]}.
    \label{onebypauPsi}
 \end{equation}
 We also have
 \begin{equation}
  [\pa_s \Psi] = [\pa_s \Sigma] + [\pa_s \psi] = -\frac{u^2}{2s^2}
	+ [\pa_s \psi],
  \label{}
 \end{equation}
 and so from \eqref{startofBAeq} we find
 \begin{equation}
  \pa_s B^A + 
  \frac{s}{u}	\left(-\frac{u^2}{2s^2} + [\pa_s \psi]\right)
  \left(1 + \frac{s}{u} [\pa_u \psi] + 
    \frac{s^2}{u^2}\frac{[\pa_u\psi]^2}{1 + \frac{s}{u}[\pa_u\psi]} \right) 
	 = 0.
  \label{pavBaeqn}
 \end{equation}
 Now we write
 \begin{equation}
     \frac{s}{u}\left(-\frac{u^2}{2s^2} + [\pa_s \psi]\right)
  \left(1 + \frac{s}{u} [\pa_u \psi] + 
    \frac{s^2}{u^2}\frac{[\pa_u\psi]^2}{1 + \frac{s}{u}[\pa_u\psi]} \right) 
    = -\frac{u}{2s} - \frac{1}{2}[\pa_u\psi] + \frac{s}{u}[\pa_s\psi]
    -F_A^\prime,
 \end{equation}
 which gives \eqref{betaevolF}.
%
%
%
%
%
%
%
%
 To get the equation for $\nas B^A$, we use that
 $\nas_T = \nas + \nas B^A \pa_u$ is tangent to the shock, so $[\nas_T \Psi] = 0 $,
 and since $\nas \Sigma = 0$, using the fourth identity in
 \eqref{onebypauPsi} we find
 \begin{equation}
  \nas B^A + \frac{[\nas \psi]}{[\pa_u \Psi]}
  = \nas B^A + \frac{s}{u} [\nas \psi] 
  + \frac{s^2}{u^2} \frac{[\nas \psi][\pa_u\psi]}{1+ \frac{s}{u}[\pa_u\psi]},
  \label{}
 \end{equation}
 which gives \eqref{betaangle}.
\end{proof}

\begin{proof}[Proof of Lemma \ref{tlbc}]

By \eqref{basicHfact1}-\eqref{basicHfact2}, we have
\begin{equation}
 [H^\alpha(\pa \Phi)] = m^{\alpha\beta}[\pa_\beta \Phi] + [j^\alpha(\pa \Phi)]
 = m^{\alpha\beta} [\pa_\beta \Phi] + A^{\alpha\beta\delta}[\pa_\beta \Phi \pa_\delta\Phi]
 +[B^\alpha(\pa \Phi)]
 \label{}
\end{equation}
Since $[r] = [\Phi] = 0$, with $\Psi = r\Phi$,
\begin{equation}
 [rH^\alpha(\pa\Phi)]
 =
 m^{\alpha\beta}[\pa_\beta \Psi]
 +
 \frac{1}{r} A^{\alpha\beta\delta}[\pa_\beta \Psi \pa_\delta\Psi]
 +[F_1^\alpha(\pa \Phi)],
 \label{rHalpha}
\end{equation}
with
\begin{equation}
 [F_1^\alpha(\pa\Phi)] =
 [rB^\alpha(\pa \Phi)]
 -\frac{1}{r^2} (A^{\alpha\beta r}+A^{\alpha r \beta})
 [\pa_\beta \Psi \Psi]
 + \frac{1}{r^3} A^{\alpha rr} [\Psi^2].
 \label{}
\end{equation}
Writing $\frac{1}{r} = \frac{2}{v} + \frac{2u}{v}\frac{1}{v-u}$,
we further write \eqref{rHalpha} as
\begin{equation}
 [rH^\alpha(\pa\Phi)]
 =
 m^{\alpha\beta}[\pa_\beta \Psi]
 +
 \frac{2}{v} A^{\alpha\beta\delta}[\pa_\beta \Psi \pa_\delta\Psi]
 +[F^\alpha(\pa \Phi)],
 \label{}
\end{equation}
with
\begin{equation}
 [F^\alpha(\pa\Phi)] = [F_1^\alpha(\pa \Phi)] + \frac{2u}{v} \frac{1}{v-u}
 A^{\alpha\beta\delta}[\pa_\beta \Psi \pa_\delta\Psi].
 \label{}
\end{equation}
Writing $H^u = H^\alpha \pa_\alpha u$,
$H^v = H^\alpha\pa_\alpha v$ and
$\slashed{H}^i = H^i - H^jx_jx^i/|x|^2$,  by \eqref{basicstructureformula}
and the fact that $m^{uv} = m^{vu} = -2$ with our conventions,
we further have
\begin{align}
 [rH^u(\pa \Phi)] &=-2 \left[\pa_v\Psi + \frac{1}{v}(\pa_u\Psi)^2\right]
 + \frac{2}{v}\widetilde{A}^{u\beta\delta}(\omega)[ \pa_\beta \Psi \pa_\delta \Psi]
 + [F^u(\pa \Phi)],\\
 [r H^v(\pa\Phi)] &=-2 [\pa_u \Psi] + \frac{2}{v}
 \widetilde{A}^{v\beta\delta}(\omega)[ \pa_\beta \Psi \pa_\delta \Psi]
 + [F^v(\pa \Phi)],\label{Hvformula}\\
 [r \slashed{H}^i(\pa\Phi)]
 &= [\nas^i \Psi] + \frac{2}{v} \slashed{\widetilde{A}}^{i \beta \delta}(\omega)
 [ \pa_\beta \Psi \pa_\delta \Psi]
 +[\slashed{F}^i(\pa \Phi)]
 \label{}
\end{align}
In the above, the coefficients $\widetilde{A}^{\alpha\beta\delta}$ satisfy
the null condition, $\widetilde{A}^{\alpha\beta\delta}\pa_\alpha u \pa_\beta u\pa_\delta u = 0$.

Now we expand $\Psi = \psi + \Sigma$, where $\Sigma = 0$ in the exterior
regions $D^L, D^R$ and $\Sigma = \tfrac{u^2}{2s}$ in the central region
to find
\begin{align}
 [rH^u(\pa \Phi)]
 &= -2\left[\pa_v \Sigma + \frac{1}{v} (\pa_u\Sigma)^2\right]
 -2 \left[ \pa_v \psi + \frac{2}{v} \pa_u\Sigma \pa_u\psi\right]
 + \frac{2}{v} A^{u \beta \delta}(\omega) [\pa_\beta \psi \pa_\delta \psi]\\
 &\qquad+ \frac{4}{v} \widetilde{A}^{u\beta \delta}(\omega)[\pa_\beta \Sigma \pa_\delta
 \psi]
 + \frac{2}{v} \widetilde{A}^{u\beta \delta}(\omega)[\pa_\beta \Sigma \pa_\delta \Sigma]
 + [F^u(\pa\Phi)].
\end{align}
Noting that $\Sigma$ satisfies the equation
\begin{equation}
 2\pa_v \Sigma + \frac{1}{v} (\pa_u\Sigma)^2 = 0
 \label{}
\end{equation}
on either side of each shock, we have the identity
\begin{align}
 [r H^u(\pa \Phi) - 2 \pa_v \Psi]
 &= [r H^u(\pa\Phi) - 2\pa_v \Sigma - 2\pa_v \psi]\\
 &= -2\left[2\pa_v \Sigma + \tfrac{1}{v}(\pa_u\Sigma)^2\right]
 -4 \left[ \pa_v \psi + \frac{1}{v} \pa_u\Sigma \pa_u\psi\right]
 + \frac{2}{v} A^{u \beta \delta}(\omega) [\pa_\beta \psi \pa_\delta \psi]\\
 &\qquad+ \frac{4}{v} \widetilde{A}^{u\beta \delta}(\omega)[\pa_\beta \Sigma \pa_\delta
 \psi]
 + \frac{2}{v} \widetilde{A}^{u\beta \delta}(\omega)[\pa_\beta \Sigma \pa_\delta \Sigma]
 + [F^u(\pa\Phi)]\\
 &=-4 \left[ \pa_v \psi + \frac{1}{v} \pa_u\Sigma \pa_u\psi\right]
 + \frac{2}{v} A^{u \beta \delta}(\omega) [\pa_\beta \psi \pa_\delta \psi]\\
 &\qquad+ \frac{4}{v} \widetilde{A}^{u\beta \delta}(\omega)[\pa_\beta \Sigma \pa_\delta
 \psi]
 + \frac{2}{v} \widetilde{A}^{u\beta \delta}(\omega)[\pa_\beta \Sigma \pa_\delta \Sigma]
 + [F^u(\pa\Phi)].
\end{align}

In particular we can write the above as
\begin{align}
 [r H^u(\pa\Phi) - 2\pa_v \Psi]
 &= -4 \left[ \pa_v \psi + \frac{1}{v} \pa_u\Sigma \pa_u\psi\right]
 + \frac{1}{v} [Q^u(\pa \psi, \pa \psi)]
 \\
 &\qquad\qquad+ \frac{1}{v} [\widetilde{Q}_1(\pa \Sigma, \pa \psi)]
 + \frac{1}{v} [\widetilde{Q}_2(\pa \Sigma, \pa \Sigma)]
 + [F^u(\pa\Phi)],
 \label{}
\end{align}
for quadratic forms $Q^u, \widetilde{Q}_1, \widetilde{Q}_2$ where
the $\widetilde{Q}_i$ satisfy the null condition.

Similarly, starting with \eqref{Hvformula} we have
\begin{equation}
 [r H^v(\pa\Phi) + 2\pa_u \Psi]
 = \frac{1}{v} [Q^v(\pa \Psi, \pa \Psi)] + [F^v (\pa \Psi)].
 \label{}
\end{equation}
Since $T^A = \pa_v + \pa_v B^A \pa_u$ is tangent to the shock at the shock, we
have $[T^A \Psi] = 0$ and so with $\zeta^A = d(u -B^A)$, we find
\begin{align}
 0 &=
 [r H^\alpha \zeta_\alpha^A - 2 T^A \Psi]\\
 &= [r H^u(\pa\Phi)-2\pa_v \Psi
  - (r H^v(\pa \Phi)\pa_v B^A +2 \pa_u \Psi)\pa_v B^A
	- r \slashed{H}^i(\pa \Phi)
 \nas_i B^A]\\
 &= -4 \left[ \pa_v \psi + \frac{2}{v} \pa_u\Sigma \pa_u\psi\right]
 + \frac{1}{v} [Q^u(\pa \psi, \pa \psi)] - [\nas^i \psi]\nas_i B^A
 \\
 &\qquad\qquad+ \frac{1}{v} [\widetilde{Q}_1(\pa \Sigma, \pa \psi)]
 + \frac{1}{v} [\widetilde{Q}_2(\pa \Sigma, \pa \Sigma)]
 - \frac{1}{v} [Q^v(\pa \Psi, \pa \Psi)]\pa_v B^A\\
 &\qquad\qquad+ [F^u(\pa\Phi)] - [F^v(\pa \Phi)]\pa_v B^A
 + [\slashed{F}^i(\pa \Phi)]\nas_i B^A.
 \label{tlbcexpression}
\end{align}

With the convention that $[q] = (q_A - q_C)|_{\Gamma^A}$
at either $\Gamma^L$ or $\Gamma^R$, we note that
\begin{equation}
 [\pa_v \psi + \frac{1}{v} \pa_u \Sigma \pa_u \psi]
 = \pa_v \psi_A - \left( \pa_v + \frac{u}{vs}\pa_u\right)\psi_C,
 \label{}
\end{equation}
and the result now follows from \eqref{tlbcexpression}.

\end{proof}

\section{Stokes' theorem}
\label{pfofstoke}

We record here the version of Stokes' theorem we will use. This is standard
apart from the fact that we are using the measure $\frac{1}{r^2} dx$ in place
of $dx$.
\begin{lemma}
  \label{stokesapp}
  Fix a metric $h$ and let $J = J^\mu\pa_\mu$ be a vector field.
  If $D = \cup_{t_0 \leq t_1} D_t$ is a domain,
 bounded by a possibly disconnected hypersurface $\Lambda$,
 neither of which contain $\{r = 0\}$, then
 \begin{equation}
   \int_{D} \pa_\mu J^\mu \, \sin^2 \theta dr d\theta d\phi dt
   = \int_{D_{t_0}} h(J,N_h^{D_{t_0}}) \sin^2 \theta dr d\theta d\phi-
   \int_{D_{t_1}} h(J, N_h^{D_{t_1}})\sin^2 \theta dr d\theta d\phi
   + \int_{\Lambda} h(J, N_h^{\Lambda})dS,
  \label{usualstokes}
 \end{equation}
 where $N_h^{\Lambda}$ denotes the outward-directed normal vector field defined relative
 to the metric $h$ and where $dS$ denotes the surface measure on $\Lambda$
 induced by the measure $\frac{1}{r^2} dx$.
\end{lemma}
\begin{proof}
  Set $J_1 = \frac{1}{r^2} J$. By the usual version of Stokes' theorem,
  \begin{equation}
    \int_{D} \pa_\mu J^\mu\, \frac{1}{r^2}dx dt
    = \int_{D} \div J_1 dx dt
   = \int_{D_{t_1}} J_1^\mu \zeta^{D_{t_1}}_\mu dx
   +
    \int_{D_{t_0}} J_1^\mu \zeta^{D_{t_0}}_\mu dx
    + \int_{\Lambda} J_1^\mu \zeta^{\Lambda}_\mu dS,
   \label{}
  \end{equation}
  where $\zeta^\Sigma = \zeta^\Sigma_\mu dx^\mu$ denotes the outward-pointing conormal
  to the surface $\Sigma$
  normalized by $\delta^{\mu\nu} \zeta_\mu \zeta_\nu = 1$. The result now follows
  since $\zeta_\mu^{\Sigma} J^\mu = h(J, N_h^{\Lambda})$ where
  $N_h^{\Sigma}$ is obtained by raising the index of $\zeta$ with $h$.
\end{proof}
When $D$ contains the origin $\{r = 0\}$, we instead have the following.
\begin{lemma}
  \label{Jvanishorigin}
  With notation as in the previous lemma,
 \begin{multline}
  \int_{D} \pa_\mu J^\mu
  =
  -\lim_{\ve \to 0 }\int_{D \cap\{r = \ve\}} J^r\\
  + \int_{D_{t_0}} h(J,  N_h^{D_{t_0}}) \sin^2 \theta dr d\theta d\phi-
  \int_{D_{t_1}} h(J,  N_h^{D_{t_1}})\sin^2 \theta dr d\theta d\phi
  + \int_{\Lambda} h(J, N_h^{\Lambda})dS.
  \label{Jvanishoriginstatement}
\end{multline}
\end{lemma}
\begin{proof}
  This follows after applying the Lemma \ref{stokesapp} to the region
  $D_{\ve} = D \cap\{|x| \geq \ve\}$ and taking $\ve \to 0$.
\end{proof}

Fix a metric $h$, vector fields $P$,  $X$ and a function $\psi$.
We define the energy current $J_{X,h,P}$ by
\begin{equation}
 J^\mu_{X,h,P} =h^{\mu\nu} \pa_\nu \psi X\psi
 - \frac{1}{2} X^\mu h^{-1}(\pa \psi, \pa \psi)
 + P^\mu X\psi - X^\mu P\psi.
 \label{}
\end{equation}
and we define the energy-momentum tensor $Q_P^h$ by
 $Q_P^h(X, Y) = h(J_X, Y) $. Explicitly,
\begin{equation}
  Q(X, Y) = X\psi Y\psi - \frac{1}{2} h(X, Y) h^{-1}(\pa \psi, \pa \psi)
  + h(P, Y) X\psi - h(X, Y) P\psi
 \label{Qhdef}
\end{equation}

Suppose that $\Lambda = \Gamma^+ \cup \Gamma^-$ for two (possibly
empty) hypersurfaces $\Gamma^{\pm}$ where $\Gamma^+$ is a spacelike surface
lying to the future of $D$ and where $\Gamma^-$ is a timelike surface $D$.
Then
the outward-pointing normal vector to $\Gamma^+$ is \emph{past-directed}.
If we let $N_h^{\Sigma}$ denote the \emph{future-directed} normal vector field
to a spacelike surface $\Sigma$ and $N_h^\Sigma$ denote the outward-facing normal
to a timelike surface $\Sigma$, then by \eqref{usualstokes}, if the origin
is not contained in $D$ we have
\begin{multline}
  -\int_{D} \pa_\mu J_{X,h,P}^\mu \, dx'dt\\
  = \int_{D_{t_0}} Q_P^h(X, N_h^{D_{t_0}})dx' - \int_{D_{t_1}} Q^h_P(X, N_h^{D_{t_1}})dx'
  + \int_{\Lambda^+} Q_P^h(X, N_h^{\Lambda^+})\, dS'
   - \int_{\Lambda^-} Q_P^h(X, N_h^{\Lambda^-} )\, dS'.
 \label{emstokes}
\end{multline}

We will need to use a version of this in the leftmost region which
contains the set $\{r = 0\}$, and the above result does not directly cover this
case. Instead we have
\begin{lemma}
  \label{divthm2}
  With $\psi = r\varphi$, $Q$ as in \eqref{Qhdef} and
  $K_{X} = K_{X, h, P}$ as in \eqref{KPdef},
  if $\pa_\mu(h^{\mu\nu}\pa_\nu \psi) + \pa_\mu P^\mu = F$,
\begin{multline}
 \int_{D^L_{t_1}} Q(X, N) + \int_{t_0}^{t_1}\int_{D^L_t}-K_{X}
 +  \int_{t_0}^{t_1}\lim_{r\to 0} \left( X^r h^{rr} \varphi^2\right)\, dt
 + \int_{t_0}^{t_1} \int_{\Gamma^L_t}
 Q(X, N)
 \\
 = \int_{D^L_{t_0}} Q(X, N)
 +\int_{t_0}^{t_1} \int_{D^L_t} F X\psi.
 \label{}
\end{multline}
\end{lemma}

\begin{proof}
  By \eqref{Jvanishorigin}, we have
 \begin{multline}
   \int_{D^L_{t_1, \ve}} Q_P^h(X, N)    -
   \int_{D^L_{t_0, \ve}} Q_P^h(X, N)
 -\int_{t_1}^{t_2} \int_{\Gamma^L_t}  Q_P^h(X, N) \,  dt
 +
  \lim_{\ve \to 0 }\int_{t_1}^{t_2} \int_{|x| = \ve} Q_P^h(X, N)\, dt\\
 =
 -\int_{t_1}^{t_2} \int_{D^L_t}F X\psi\, dt.
  \label{}
 \end{multline}
To handle the integral over $|x| = \ve$, we expand $\psi = r\varphi$
and compute
\begin{multline}
 Q_P^h(X, N) =  h^{r\mu}\pa_\mu \psi X\psi - \frac{1}{2} X^r
 h^{-1}(\pa\psi,\pa\psi)
 +
 P^rX\psi - X^r P \psi\\
 = \ve^2 \left(h^{r\mu}\pa_\mu \varphi X\varphi -\frac{1}{2} X^r h^{-1}(\pa\varphi,
 \pa\varphi)\right)
 +\ve
 \bigg( h^{rr}\varphi X\varphi + h^{r\mu}\pa_\mu \varphi X^r \varphi
 - X^r h^{r\mu}\pa_\mu \varphi \varphi\bigg)\\
 + h^{rr}\varphi^2 X^r - \frac{1}{2} X^rh^{rr} \varphi^2
 + \ve P^r X\varphi - \ve X^r P\varphi
 \\
 =\ve^2 \bigg(h^{r\mu}\pa_\mu \varphi X\varphi -\frac{1}{2} X^r h^{-1}(\pa\varphi,
 \pa_\varphi)\bigg)
 + \ve \left(h^{rr}\varphi X\varphi + P^r X\varphi -  X^r P\varphi\right)
 + \frac{1}{2} X^r h^{rr} \varphi^2 ,
 \label{}
\end{multline}
and taking $\ve \to 0$ we arrive at the result.
\end{proof}
We also need a modification of the above result where we replace the usual
energy-momentum tensor $Q^h_P$ with the energy-momentum tensor
$\mQ^h_P$ defined in \eqref{mQdef}.

\begin{lemma}
  \label{divthm3}
  Let $\psi = r\varphi$, $\mQ$ as in \eqref{mQdef} and define
  $\mK_X$ and $\wK_{X}$ as in Proposition \ref{effectivemmmink} or
	\ref{effectivemmmB}.
  For a metric $g$ set $\gamma = h^{-1}-g^{-1}$.
  If $\pa_\mu(h^{\mu\nu}\pa_\nu \psi) + \pa_\mu P^\mu = F$ then
\begin{multline}
 \int_{D_{t_1}} \mQ_P^h(X, N_h^{D_{t_1}}) - \int_{t_0}^{t_1}\int_{D_t}\mK_{X}
 +\wK_X
 +  \int_{t_0}^{t_1}\lim_{r\to 0} \left( X^r (g^{rr}+\pg^{rr}) \varphi^2\right)\, dt
 + \int_{\Lambda}
 Q(X, N_h^\Lambda)
 \\
 = \int_{D_{t_0}} \mQ_P^h(X, N_h^{D_{t_0}})
 +\int_{t_0}^{t_1} \int_{D_t} F X\psi.
 \label{}
\end{multline}
\end{lemma}
\begin{proof}
 This follows as in the previous lemma, after noting that
 \begin{equation}
  \lim_{r\to 0} \mJ_{X}^r - J_{X, g}^r - J_{X,\pg}^r = 0,
  \label{}
 \end{equation}
 which follows since the remaining terms in the definition of $\mJ$ from
 \eqref{mQdef} vanish away from $\{u = 0\}$.
\end{proof}

\section{Hardy and Poincar\'{e}-type inequalities}

To close our estimates, we will need some bounds for homogeneous
quantities $\psi^I_A$ as opposed
to $\pa \psi^I_A$. We start with the following bounds at the shocks.
\begin{lemma}
	\label{sobolevshock}
	Suppose that \eqref{betaLassump} (resp. \eqref{betaRassump}) holds.
 Let $\Gamma = \Gamma^R$ (resp. $\Gamma^L$) and let $q$ be a function
 defined in a neighborhood of (one side of) $\Gamma$. For any $t_0$, we have
 \begin{equation}
  \|q\|_{L^2(\Gamma_t)}
	\lesssim \|q\|_{L^2(\Gamma_{t_0})} + (\log t)^{1/2}
	\left(\int_{t_0}^{t} \int_{\Gamma_{t'}} v |\pa_v q|^2 + \frac{1}{vs} |\pa_u q|^2\, dS dt\right)^{1/2}
  \label{}
 \end{equation}
\end{lemma}
\begin{proof}
	Let $r(t, \omega)$ denote the value of $|x|$ at the intersection of
	$\Gamma_t$ and the ray $\{x/|x| = \omega\}$. Then 
	\begin{multline}
		\| q\|_{L^2(\Gamma_t)}^2 \lesssim
		\int_{\mathbb{S}^2} |q(t, r(t,\omega)\omega)|^2\, dS(\omega)\\
		\lesssim \int_{\mathbb{S}^2} |q(t_0, r(t_0, \omega) \omega)|^2\, dS(\omega)
	 + \left(\int_{t_0}^{t} \int_{\mathbb{S}^2}(\pa_t q)(t, r(t', \omega) \omega) +
	  (\pa_t r(t',\omega))(\pa_r q)(t', r(t', \omega)\omega)\,  dS(\omega) dt'\right)^2.
	 \label{}
	\end{multline}
	Since $|t - r(t, \omega)| = B(t, r(t,\omega)\omega)$ where
	$B$ satisfies the estimates in \eqref{betaLassump}-\eqref{betaRassump}, it
	follows that $|\pa_t r(t, \omega) - 1|\lesssim v^{-1}s^{-1/2}$, where here
	we are writing $v = t + r(t,\omega)$ and $s = \log (t+ r(t,\omega))$, so we
	have the bound
	\begin{multline}
	\left| \int_{t_0}^{t}\int_{\mathbb{S}^2} (\pa_t q)(t, r(t', \omega) \omega) +
	  (\pa_t r(t',\omega))(\pa_r q)(t', r(t', \omega)\omega)\, dS(\omega) dt'\right|
		\\
		\lesssim
        \int_{t_0}^t\int_{\Gamma_{t'}} |\pa_v q| + \frac{1}{vs^{1/2}} |\pa_u q|\, dS dt'
		\lesssim
	\left(	\int_{t_0}^t \frac{dt'}{t'} \right)^{1/2}
    \left( \int_{t_0}^{t_1} \int_{\Gamma_{t'}} v |\pa_v q|^2 + \frac{1}{vs} |\pa_u q|^2\, dS dt'\right)^{1/2}.
	 \label{hardysintroduction}
	\end{multline}
    Therefore
	\begin{equation}
	 \|q\|_{L^2(\Gamma_t)}^2 \lesssim
	 \| q\|_{L^2(\Gamma_{t_0})}^2
	 + \log t 
	 \int_{t_0}^{t} \int_{\Gamma_{t'}} v |\pa_v q|^2 + \frac{1}{vs} |\pa_u q|^2\, dS dt',
	 \label{}
	\end{equation}
	as needed.
\end{proof}
We will also need the following simple variant of the above, which just
relies on the fact that
the functions $(t \log t (\log \log t)^\alpha))^{-1}$ and
$(t \log t \log \log t (\log \log t)^\alpha))^{-1}$
are time-integrable
when $\alpha > 1$. 

\begin{lemma}
	\label{sobolevshock2}
	Suppose that \eqref{betaLassump} (resp. \eqref{betaRassump}) holds.
 Let $q$ be a function
 defined in a neighborhood of one side of $\Gamma^L$. For any $t_0$, we have
 \begin{equation}
  \|q\|_{L^2(\Gamma_t)}
	\lesssim \|q\|_{L^2(\Gamma_{t_0})}
    +
	\left(\int_{t_0}^{t} \int_{\Gamma_{t'}} v\log v(\log \log v)^\alpha |\pa_v q|^2 +
	 \frac{1}{v} \log s (\log \log s)^\alpha |\pa_u q|^2\, dS dt'\right)^{1/2}.
  \label{sobolevpoinL}
 \end{equation}
\end{lemma}
\begin{proof}
	The proof is the same as the proof of
	Lemma \ref{sobolevshock} above, except that instead of
	\eqref{hardysintroduction} we bound
	\begin{equation}
        \int_{t_0}^t\int_{\Gamma^L_{t'}} |\pa_v q| \, dS dt'
	 \lesssim
	 \left(\int_{t_0}^{t} \frac{1}{t'} \frac{1}{\log t'} \frac{1}{(\log \log t')^\alpha}\,dt'\right)
     \left(\int_{t_0}^{t} \int_{\Gamma^L_{t'}}
	 v s (\log s)^\alpha|\pa_v q| \, dS dt'\right)
     \lesssim\int_{t_0}^{t} \int_{\Gamma^L_{t'}}
	 v s (\log s)^\alpha|\pa_v q| \, dS dt'
	 \label{}
	\end{equation}
	and
	\begin{multline}
        \int_{t_0}^t\int_{\Gamma^L_{t'}} \frac{1}{vs^{1/2}}|\pa_u q| \, dS dt'\\
		\lesssim
			 \left(\int_{t_0}^{t} \frac{1}{t'} \frac{1}{\log t'} \frac{1}{\log \log t'}
			 \frac{1}{(\log\log \log t')^\alpha}\,dt'\right)
             \left(\int_{t_0}^{t} \int_{\Gamma^L_{t'}} \frac{1}{v} \log s (\log \log s)^\alpha
			 |\pa_u q|^2\, dS dt'\right)\\
             \lesssim\int_{t_0}^{t'} \int_{\Gamma^L_{t'}} \frac{1}{v} \log s (\log \log s)^\alpha
			 |\pa_u q|^2\, dS dt'.
	 \label{}
	\end{multline}
\end{proof}

We now record some bounds which rely on Lemma \ref{sobolevshock}. In the rightmost region,
we will use the following simple estimate, which is based on the Hardy-type
inequalities from \cite{LindbladRodnianski2010}.
\begin{lemma}
	\label{hardyRapplem}
 If \eqref{betaRassump} holds, for $t_0 \leq t$ and $\mu > 1$,
 if $q$ satisfies the condition  $\lim_{r \to \infty} (1+ r-t)^{\mu-1} |q(t,r\omega)|^2 = 0$
 for each $t \geq 0, \omega \in \mathbb{S}^2$, then
\begin{equation}
(1 + \log t)^{\mu/4-1/2} \| q \|_{L^2(D^R_t)} \lesssim  \|(1 +r-t)^{\mu/2} \pa q\|_{L^2(D^R_t)}.
 \label{hardyRapp}
\end{equation}


\end{lemma}
\begin{proof}
	Take $\gamma > 0$ and set $w(r-t) = (1 + r-t)^{\gamma}$. Then we have
	\begin{equation}
	 \pa_r ( w(r-t) q^2) = w'(r-t) q^2 + 2w(r-t) q\pa_r q.
	 \label{}
	\end{equation}
	For fixed $t' \geq 0$ and $\omega \in \mathbb{S}^2$, let $r_R(t', \omega)$
	denote the value of $r = |x|$ at the intersection of the sets $\{x/|x| = \omega\}$,
	$\{t = t'\}$ and $\Gamma^R_t$. That is, $r_R$ is
	defined by the property that $t - r_R(t, \omega) = \beta_{\log(t + r_R(t,\omega)}(\omega)$.
	Integrating the above identity at fixed $t$ and $\omega = x/|x|$
	from $r = r_R(t, \omega)$ to $r = \infty$ and using the decay of $q$ at infinity,
	we find
	\begin{multline}
		\int_{r = r_R(t,\omega)}^\infty w'(r-t) q^2\, dr \leq
		 2\int_{r = r_R(t,\omega)}^\infty w(r-t) |q| |\pa_r q| \, dr\\
		 \leq 2 \left( \int_{r = r_R(t,\omega)}^{\infty} w'(r-t) |q|^2\, dr\right)^{1/2}\,
		 \left( \int_{r = r_R(t,\omega)}^\infty \frac{w(r-t)^2}{w'(r-t)} |\pa_r q|^2\, dr\right)^{1/2},
	\end{multline}
	where we used $\gamma > 0$ to divide by $w'$.
	This gives
	\begin{equation}
\int_{r = r_R(t,\omega)}^\infty w'(r-t) q^2 \, dr\leq
4 \int_{r = r_R(t,\omega)}^\infty \frac{w(r-t)^2}{w'(r-t)} |\pa_r q|^2\, dr.
	\end{equation}
	Integrating over $\omega \in \mathbb{S}^2$ and taking $\gamma =\mu - 1$ gives
	the bound
	\begin{equation}
	 \int_{D^R_t} (1 + r-t)^{\mu-2} |q|^2
	 \leq 4 \int_{D^R_t} (1 + r-t)^{\mu} |\pa q|^2,
	 \label{}
	\end{equation}
	and using that $r-t \gtrsim (1 + \log t)^{1/2}$ in $D^R$ gives the result.

 %
 %
 %
\end{proof}

We will also need the following weighted estimates on
the timelike side of the right shock and the spacelike side
of the left shock. The first bound is needed to close the energy
estimates in the central region and the second is needed to control
a term that arises when using the boundary conditions on the timelike
side of the left shock. We will also use the first bound on the spacelike
side of the left
shock to handle some of the boundary terms coming from the boundary
condition along the timelike side of the left shock.
\begin{lemma}
	\label{controlangularhardyright}
	If \eqref{betaRassump} holds, there is a continuous function
	$c_0(\epsilon_0)$ with $c_0(0) = 0$
	so that if $q$ is a function defined in a neighborhood of one side
	 of $\Gamma^R$,
 \begin{equation}
  \int_{t_0}^{t_1} \int_{\Gamma^R_t} \frac{s}{v^2} |q|^2\, dS dt
	\lesssim \frac{1}{1+t_0} \int_{\Gamma^R_{t_0}} |q|^2\, dS
	+ c_0(\epsilon_0)\int_{t_0}^{t_1} \int_{\Gamma^R_t} v |\pa_v q|^2 + \frac{1}{vs} |\pa_u q|^2\, dS dt.
  \label{controlangularderivs}
 \end{equation}
 If \eqref{betaLassump} holds, the same bound holds with
 $\Gamma^R$ replaced with $\Gamma^L$.
\end{lemma}

\begin{proof}
	Since
 \begin{equation}
  \frac{\log v}{v^2} = -\frac{d}{dv} \frac{1 + \log v}{v} = -T^R \frac{1 + \log v}{v}
  \label{}
 \end{equation}
 where $T^R = \pa_v + \pa_v B \pa_u$ is a generator of $\Gamma^R$,
 we have
 \begin{align}
  \int_{t_0}^{t_1} \int_{\Gamma^R_t} \frac{s}{v^2} |q|^2\, dS dt
	&\leq \int_{\Gamma^R_{t_0} } \frac{1 + s}{v} |q|^2\, dS\\
	&\qquad+ 2 \left( \int_{t_0}^{t_1} \int_{\Gamma^R_t} \frac{(1+s)^2}{v^3} |q|^2\, dS dt\right)^{1/2}
	\left( \int_{t_0}^{t_1} \int_{\Gamma^R_t} v |\pa_v q|^2 + \frac{1}{vs} |\pa_u q|^2\, dS dt\right)^{1/2}\\
	&\lesssim \int_{\Gamma^R_{t_0} } \frac{1 + s}{v} |q|^2\, dS\\
	&\qquad+ c_0(\epsilon_0)\left( \int_{t_0}^{t_1} \int_{\Gamma^R_t} \frac{(1+s)^2}{v^2} |q|^2\, dS dt\right)^{1/2}
	 \left( \int_{t_0}^{t_1} \int_{\Gamma^R_t} v |\pa_v q|^2 + \frac{1}{vs} |\pa_u q|^2\, dS dt\right)^{1/2},
  \label{}
 \end{align}
 which gives the result after absorbing.

\end{proof}

We will also use the following estimates in the central region.
\begin{lemma}
    \label{hardyClemma}
 If \eqref{betaLassump}-\eqref{betaRassump} hold, for any
 $t \geq t_0$,
 we have
 \begin{equation}
   \label{hardyCapp0}
   \| q\|_{L^2(D^C_t)} \lesssim (1 + \log t)^{1/4}\|q\|_{L^2(\Gamma^L_t)} 
   + (1+ \log t)^{1/2}\|\pa q\|_{L^2(D^C_t)},
 \end{equation}
 and
 \begin{equation}
  \|q\|_{L^2(D^C_t)} \lesssim (\log t)^{1/4} \|q\|_{L^2(\Gamma^L_{t_0})}
	+( \log t)^{3/4} \left(\int_{t_0}^{t} \int_{\Gamma^L_{t'}} v |\pa_v q|^2
	+ \frac{1}{vs} |\pa_u q|^2\, dS dt'\right)^{1/2}
	+ (\log t)^{1/2} \|\pa q\|_{L^2(D^C_t)}.
  \label{hardyCapp}
 \end{equation}
\end{lemma}

\begin{proof}
 For $t^\prime\in [t_0, t_1]$ and $\omega^\prime \in \mathbb{S}^2$, let
 $r_R(t^\prime,\omega)$ denote the value of $r = |x|$ at the intersection of
 the sets $x/|x| = \omega^\prime$, the right shock, and the surface
 $\{t = t^\prime\}$, and similarly with $r_L(t, \omega)$.
 Then $|r_L(t,\omega) - r_R(t, \omega)| \lesssim (\log t)^{1/2}$
 under our assumptions. Bounding
 \begin{multline}
  |q(t, r, \omega)|^2 \lesssim |q(r, r_L(t, \omega)\omega)|^2
	+  |r_R(t,\omega) - r_L(t, \omega)| \int_{r_L(t,\omega)}^{r_R(t, \omega)}
	|\pa_r q(t, r',\omega)|^2\, dr\\
	\lesssim|q(r, r_L(t, \omega)\omega)|^2
	+ (\log t)^{1/2} \int_{r_L(t,\omega)}^{r_R(t, \omega)}
	|\pa_r q(t, r',\omega)|^2\, dr.
  \label{holderqC}
 \end{multline}
 integrating over $D^C_t$ and using that $r^{-2}\textrm{Vol}(D^C_t) \lesssim (\log t)^{1/2}$
	(recall that our integrals are taken with respect to $|x|^{-2} dx$)
	we find
 \begin{equation}
  \int_{D^C_t} |q|^2
	\lesssim (\log t)^{1/2}\int_{\Gamma^L_t} |q|^2
	+ \log t \int_{D^C_t} |\pa q|^2,
  \label{basicpoinC}
 \end{equation}
 which is \eqref{hardyCapp0}, and \eqref{hardyCapp}
 then follows from \eqref{sobolevshock}.
\end{proof}
%

Finally, to handle some of the homogeneous terms we encounter
in $D^L_t$, we need the following bound.
\begin{lemma}
 If \eqref{betaLassump} holds, then for $t \geq t_0$ we have
 \begin{equation}
  \| q\|_{L^2(D^L_t\cap \{|u| \leq s^3\})} \lesssim
	(\log t)^{3/2} \|q\|_{L^2(\Gamma^L_{t_0})}
	+ (\log t)^{2} \left(\int_{t_0}^{t_1} \int_{\Gamma^L_t} v |\pa_v q|^2
	+ \frac{1}{vs} |\pa_u q|^2\, dS dt\right)^{1/2}
	+ (\log t)^3 \|\pa q\|_{L^2(D^L_t)}.
  \label{rhardyL0}
 \end{equation}
 If $q|_{r = 0} = 0$ and $q$ is smooth,
 \begin{equation}
  \| r^{-1} q\|_{L^2(D^L_t \cap\{|u| \geq s^3\})}\lesssim
	\|\pa q\|_{L^2(D^L_t\cap\{|u| \geq s^3\})}.
  \label{rhardyL}
 \end{equation}
 In particular, if $q|_{r = 0} = 0$,
 \begin{equation}
  \int_{D^L_t} |\pa (r^{-1} q)|^2 r^2\, dr dS(\omega)
	\lesssim \int_{D^L_t} |\pa q|^2\, dr dS(\omega).
  \label{pointofhardyL}
 \end{equation}
\end{lemma}
We remind the reader that all integrals are taken with respect
to $|x|^{-2} dx$ and not $dx$.
\begin{proof}
	For each $r, \omega$ we have the bound
	\begin{multline}
	 |q(t, r\omega)|
	 \leq |q(t, r_L(t, \omega)\omega)|
	 + \int_{r}^{r_L(t,\omega)} |\pa_r q(t, r'\omega)|\, dr'\\
	 \lesssim
	 |q(t, r_L(t, \omega)\omega)|
	 + |r_L(t,\omega) - r|^{1/2} \left( \int_{r}^{r_L(t, \omega)}
	 |\pa_r q(t, r'\omega)|^2\, dr'\right)^{1/2},
	 \label{}
	\end{multline}
	where $r_L(t,\omega)$ denotes the value of $|x|$ at the intersection
	of the left shock and the ray $x/|x| = \omega$ at time $t$.
	Squaring and integrating this expression over $D^L_t \cap\{|u| \leq s^3\}$ and using that
	$|r - r_L(t, \omega)|\lesssim (\log t)^3$ in that region, we find
	\begin{equation}
		\int_{D^L_t \cap \{|u|\leq s^3\}} |q|^2
	 \lesssim
	(\log t)^3 \int_{\Gamma^L_t} |q|^2
	+ (\log t)^6 \int_{D^L_t} |\pa q|^2,
	 \label{}
	\end{equation}
	and using \eqref{sobolevshock} at the left shock gives the first result.

	The second bound \eqref{rhardyL} is the usual Hardy inequality.
	Writing $\frac{1}{r^2} = -\frac{d}{dr} \frac{1}{r}$,
	integrating by parts and using that $\lim_{r \to 0} \frac{|q|^2}{r} = 0$
	since $q|_{r = 0} = 0$ and $q$ is smooth, we find that for arbitrary $R > 0$
	\begin{equation}
	 \int_0^{R} \frac{|q(t, r\omega)|^2}{r^2} \, dr
	 \leq 2\int_0^{R} \frac{1}{r} |q(t, r\omega)| |\pa_r q(t,r\omega)|\, dr,
	 \label{}
	\end{equation}
	which gives the result after absorbing and integrating
	over $\omega \in \mathbb{S}^2$. The bound
	\eqref{pointofhardyL} follows immediately from
	\eqref{rhardyL}.
\end{proof}

\section{Global Sobolev inequalities}
\label{klainermansobolev}
We record here the Klainerman-Sobolev type inequalities we use to control
pointwise norms of the solution in terms of $L^2$ norms involving vector
fields. We remind the reader that all integrals below are
taken with respect to the measure $dx/r^2$ as opposed to the usual
three-dimensional measure $dx$.

Integrating from $r = |x|$ to $r = \infty$, using Sobolev embedding on 
$\S^2$ and $|\pa q|\lesssim \frac{1}{1+|u|}|Zq|$ gives
\begin{lemma}
  If $q \in C_0^\infty(D_t^R)$ and $w$ satisfies $(1+|u|) w'(u)\lesssim
  w(u)$, then
 \begin{equation}
  w(u)^{1/2} (1+ |u|)^{1/2} |q(t,x)|\lesssim
  \sum_{|I| \leq 3}
	\| w^{1/2} Z^I q(t,\cdot)\|_{L^2(D^R_t)}
  \label{KSR}
 \end{equation}
\end{lemma}

In the central region we have the following pointwise bound
which follows from a {scale-invariant Sobolev inequality.}
\begin{lemma}
	Under the hypotheses of Proposition \ref{bootstrapprop},
 if $q \in C^\infty(D_t^C)$, the following inequality holds,
 \begin{equation}
  {(1 + \log t)^{1/4}} |q(t,x)|\lesssim
  \sum_{|I| \leq 3}
	\| Z_{\mB}^I q(t,\cdot)\|_{L^2(D^C_t)}
  \label{KSC}
 \end{equation}
\end{lemma}
\begin{proof}
	{
	At each time $t$, $D_t^C$ can be written as the region
	between two graphs over the unit sphere $\S^2$,
	\begin{equation}
	 D_t^C = \{x \in \mathbb{R}^3 : r_L(t, x/|x|) \leq |x| \leq r_R(t,x/|x|)\},
	 \label{}
	\end{equation}
	where $r_A(t',\omega)$ denotes the value of $|x|$ lying at the intersection
	of the sets $\Gamma^A$, $\{t = t'\}$, and $\{x/|x| = \omega\}$.

    We now rescale and{
	introduce {$R(t,y) = (1-|y|)(r_L(t, y/|y|) - r_R(t, y/|y|)) + r_L(t, y/|y|)$, so that 
	$x=R(t,y)y/|y|$ 
    maps the annulus $A = \{1 \leq |y| \leq 2\}$ to the region $D^C_t$.
    Writing
    $Q(t, y) = q(t, R(t,y)y/|y|)$} and using the Sobolev inequality with
    respect to the measure $dx/|x|^2$, we find
    \begin{equation}
      \label{usesssob0}
      \|q\|_{L^\infty(D^C_t)}  = \|Q\|_{L^\infty(A)}
      \lesssim \sum_{k \leq 3} \| \nabla^k Q\|_{L^2(A)}.
    \end{equation}
    Writing $\omega = y/|y|$, the function $R$ satisfies
    \begin{align}
      \label{eq:Rbd1}
      |R(t, y)|&\lesssim |r_L(t, \omega) - r_R(t, \omega)| + |r_L(t, \omega)| \lesssim (1 + \log t)^{1/2}
       + t,
       \\
       \label{eq:Rbd2}
      |\nabla^{1+k}_{y}R(t, y)| &\lesssim |r_L(t, \omega) - r_R(t,\omega)| + |\nabla^{1+k}_y r_L(t,\omega)|
      \lesssim  (1 + \log t)^{1/2},
    \end{align}
    for $k\le 2$.     {Furthermore, 
    $$
    \nabla_{y} Q(t, y)=\nabla_{y} (R(t, y)\omega)\cdot \nabla_x q(t,x) =
    \nabla_{y} R(t, y)\omega\cdot \nabla_x q(t,x) + R(t, y) \nabla_y \omega\cdot \nabla_x q(t,x)
    $$
    The second term above can be decomposed as
    $$
     R(t, y) \nabla_y \omega\cdot \nabla_x q(t,x)= \sum_{k\le 1}  {\Omega^k q(t,x)}
    $$
    Applying another derivative we then obtain 
    $$
    |\nabla^{2}_{y} Q(t, y)|\lesssim |\nabla^{2}_{y} R(t, y)| |\nabla_x q| + |\nabla_{y} R(t, y)|^2 |\nabla^2_x q|+
    |\nabla_{y} R(t, y)| |\nabla_x \Omega q|+|\Omega^2 q|
    $$
    Using \eqref{eq:Rbd2},
    $$
      |\nabla^{2}_{y} Q(t, y)|\lesssim  (1 + \log t)^{1/2} |\nabla_x q| +  (1 + \log t) |\nabla^2_x q|+
    (1 + \log t)^{1/2}|\nabla_x \Omega q|+|\Omega^2 q|
    \lesssim
        \sum_{|I| \leq 2} |(Z^I_{\mB} q)(t,x)|,
   $$
    where we used  the
    fact that our vector fields satisfy
    $(1 + \log t)^m |\nabla^m q| + (1+t)^m |\nas^m q|\leq 
    \sum_{|I| \leq m} |Z_{\mB} q|$. A similar inequality holds for the third derivatives,
   $$
      |\nabla^{3}_{y} Q(t, y)|\lesssim          \sum_{|I| \leq 3} |(Z^I_{\mB} q)(t,x)|.
   $$}

}

	Returning to \eqref{usesssob0}, changing variables and using that
	$|r_R - r_L|^{-1} \lesssim (1 +\log t)^{-1/2}$, we therefore have
	\begin{equation}
	 \|q \|_{L^\infty(D^C_t)}
	 \lesssim
	 \sum_{|I| \leq 3}
	  \left(
		\int_{D^C_t}\frac{1}{(1 + \log t)^{1/2}}|Z_{\mB}^I q|\right)^{1/2}.
	 \label{}
	\end{equation}
	}
\end{proof}

To the left of the left shock, we use the standard Klainerman-Sobolev inequality.
\begin{lemma}
 If $q \in C^\infty(D_t^L)$ then
 \begin{equation}
  (1+|u|)^{1/2} |q(t,x)| \lesssim
  \sum_{|I| \leq 3}
	\| Z^I q(t,\cdot)\|_{L^2(D^L_t)}
  \label{KSL}
 \end{equation}
\end{lemma}

\section{The modified energy and scalar currents}
\label{modifiedproofsecm}

In this section we prove multiplier identities for solutions of equations
of the form
\begin{equation}
 \pa_\mu (h^{\mu\nu} \pa_\nu\psi) + \pa_\mu P^\mu = F,
 \label{modelwaveapdx}
\end{equation}
where $h$ is either a perturbation of the Minkowski metric or the metric
$\mB$. These identities are used to prove the energy estimates in
Section \ref{ensec2}. In the Minkowskian case we use Proposition \ref{mainminkidentprop}
and in the central region we use Proposition \ref{mainmBidentprop}.
The assumptions in the upcoming results are designed to capture
the behavior of the multiplier fields we will be using (see Section
\ref{fields}). In particular the condition \eqref{Xgrowth} will be immediate
for all of our fields. We remind the reader that for our applications, $\gamma$
will behave roughly like $1/v \pa \psi$ and $P$ will collect various lower-order terms.
\begin{prop}[The modified multiplier identity in the Minkowskian case]
	\label{mainminkidentprop}
	Suppose that $\psi$ satisfies the equation \eqref{modelwaveapdx} and let $\gamma = h^{-1} - m^{-1}$.
	Let $X = X^u \pa_u+ X^v\pa_v$ where $X^u = X^u(u,v)$, $X^v =X^v(u,v)$, and suppose
	that $\gamma, X$ satisfy the assumptions \eqref{pert1}, and morever that $X$
	satisfies
	$|X^\ell_m|, |X^n_m| \gtrsim 1$ and 
	\begin{equation}
		{
		\left(\frac{|X^\ell_m|^{1/2}}{ |X^n_m|^{1/2}}+\frac{|X^n_m|^{1/2}}{|X^\ell_m|^{1/2}}\right)
		} \frac{1+|u|}{1+v}
		+  {\frac{|\pa X| }{|X^n_m|^{1/2}|X^\ell_m|^{1/2}} }(1+|u|)
		+ \frac{|\pa X^u|}{|X^n_m|^{1/2}}(1+ |u|)
		\lesssim 1
	 \label{Xgrowth}
	\end{equation}
	when $|u| \leq v/8$.
	Then the identity
	 \begin{equation}
		 \left( \pa_\mu(h^{\mu\nu}\pa_\nu \psi) + \pa_\mu P^\mu\right) X\psi
		 = \pa_\mu J_{X, m, P}^\mu + K_{X, m, P}
		 +\pa_\mu \mJ_{X, \gamma, P}^\mu + \mK_{X, \gamma, P},
	  \label{}
	 \end{equation}
 holds, where the energy current
 $J_{X, m}$ and scalar current $K_{X, m}$ are defined as in \eqref{energycurrent} and \eqref{scalarcurrent}.
  The modified energy current $\mJ_{X, \gamma, P}$ is given explicitly in
  \eqref{explicitmJ} and the modified scalar current $\mK_{X, \gamma, P}$ is given
  explicitly in \eqref{explicitmK}, and these quantities satisfy the following estimates.
  If $\zeta$ is any one-form with $|\zeta| = 1$,
  for any $\delta > 0$,
  in the region $|u| \leq v/8$, the modified energy current
 $\mJ_{X, \gamma, P}$ satisfies the estimates
	\begin{multline}
	|\zeta(\mJ_{X, \gamma, P})|
	\lesssim
	 \delta |X^\ell_m||\evm \psi|^2
	+ \left(1 + \frac{1}{\delta}\right)|\gamma| |\pa \psi|_{X,m}^2
	+  |\zeta(X)| |\gamma| |\pa \psi|^2
	+
	|\slashed{\zeta}|^2 |\pa \psi|_{X,m}^2
	\\
  +  \left(1 + \frac{1}{\delta}\right) |X| |P|^2
  +  |X^n_m|^{1/2} |P| |\pa \psi|_{X, m}.
	 \label{zetamJ}
	\end{multline}
	When $|u| \geq v/8$, we instead have
	\begin{equation}
  |\zeta(\mJ_{X, P})|
	\lesssim |\gamma| |X| |\pa \psi|^2 + |P| |X| |\pa \psi|.
  \label{zetamJfar}
 \end{equation}

	In the region $|u| \leq v/8$, the modified scalar current satisfies
	\begin{multline}
	 |\mK_{X, \gamma, P}|\lesssim
	 \left( |\nabla \gamma| + \frac{1}{1+|u|} |\gamma|
 	+ \frac{|X^\ell_m|^{1/2}}{|X^n_m|^{1/2}}\left( |\nabla_{\evm} \gamma|
	+ |\nas \gamma|\right)
 	\right)  |\pa \psi|_{X, m}^2
 	+ |X^n_m| |F|  |\pa \psi|_{X, m}\\
	+ \left( |\nabla P| +  \frac{|P|}{1+|u|}
	+ \frac{|X^\ell_m|^{1/2}}{|X^n_m|^{1/2}} \left( |\nabla_{\evm} P|
	+ |\nas P|\right) \right)|X^n_m|^{1/2} |\pa \psi|_{X, m}
	\\
  +
   |P| |\pa_u X^v| |\evm \psi|
	+
	|P||X|  \left( |F| + \frac{1}{1+v} |P|\right)
	 \label{modifiedKbound}
	\end{multline}
		and in the region $|u| \geq v/8$, we instead have 
	\begin{equation}
	 |\mK_{X, \gamma, P}|
	 \lesssim |\nabla \gamma | |X| |\pa \psi|^2 + |\gamma| |\pa X| |\pa \psi|^2
	 + |\nabla P| |X| |\pa \psi|
	 +
	  \left(\frac{1}{r} +\frac{1}{1+v}\right)|X| \left(|\gamma| |\pa \psi|^2+ |P| |\pa \psi|\right)
	 \label{trivialKbound}
	\end{equation}
    as well as
	\begin{equation}
	 |\mK_{X, \gamma, P}|
	 \lesssim |\mathcal{L}_X\gamma | |\pa \psi|^2 + |\gamma| |\pa X| |\pa \psi|^2
	 + |\mathcal{L}_X P|  |\pa \psi|
	 +\frac{1}{1+v}|X| \left(|\gamma| |\pa \psi|^2+ |P| |\pa \psi|\right)
	  + |\pa X| |P| |\pa \psi|,
	 \label{trivialKboundlie}
	\end{equation}
	where $\mathcal{L}_X\gamma$ denotes the Lie derivative of the tensor
	field $\gamma$ with respect to $X$ and $\mathcal{L}_XP$ denotes the Lie
	derivative of the vector field $P$ with respect to $X$.
\end{prop}
\begin{remark}
 It is better to use \eqref{trivialKboundlie} near $r = 0$ since it
 avoids a spurious singularity at the origin. We have written the above in
 this form because $|\mathcal{L}_X\gamma|$ and $|\nabla \gamma|$ are invariant
 under coordinate changes, and this is convenient since in rectangular
 coordinates, the components of $\gamma$ are constants and therefore
 both quantities are easy to compute. The quantites $|\pa X|$ are of course
 not invariant under coordinate changes but they are easy to handle.
\end{remark}

For the estimates in the central region, the metric $h$ will be a perturbation
of the metric $\mBB$ (defined in \eqref{mBBdef}) and the analogue of the above is the following.
\begin{prop}[The modified multiplier identity in the central region]
	\label{mainmBidentprop}
	Suppose that $\psi$ satisfies the equation \eqref{modelwave} and let
	$\gamma = h^{-1} - \mBB^{-1}$ with notation as in
	section \ref{generalenergies}. Let $X = v \pa_v + X^u \pa_v$ where
	$X^u = X^u(u,v)$ and suppose that $\gamma, X$ satisfy
	the bounds \eqref{pert1}, and moreover that
	$1+s \gtrsim |X^u| \gtrsim (1+s)^{-1/2}$ and $|\pa X^u|\lesssim \frac{1}{1+v}$.
	Then we have the identity
 \begin{equation}
	 \left( \pa_\mu(h^{\mu\nu}\pa_\nu \psi) + \pa_\mu P^\mu\right) X\psi
	 = \pa_\mu J_{X, \mBB}^\mu + K_{X, \mBB}
	 +\pa_\mu \mJ_{X, \gamma, P}^\mu + \mK_{X, \gamma, P},
  \label{mBidentapp}
 \end{equation}
 holds, where the energy current
 $J_{X, \mBB}$ and scalar current $K_{X, \mBB}$ are defined as in \eqref{energycurrent} and \eqref{scalarcurrent}.
 The modified energy current $\mJ_{X, \gamma, P}$ is given explicitly in
 \eqref{explicitmJmB} and the modified scalar current $\mK_{X, \gamma, P}$ is given
 explicitly in \eqref{explicitmKmB}, and these quantities satisfy the following estimates.

 If $\zeta$ is any one-form with $|\zeta| = 1$, then when $|u| \lesssim s^{1/2}$,
 the modified energy current
 $\mJ_{X, \gamma, P}$ satisfies the bound
		\begin{align}
			|\zeta(\mJ_{X, \gamma, P})|
			&\lesssim
			\delta v|\evmB \psi|^2
			+ \left(\delta + \epsilon +  \frac{\epsilon}{\delta}\right)\frac{1}{(1+v)(1+s)^{1/2}} |\pa \psi|_{X, \mB}^2
			+ |\zeta(X)| |\gamma| |\pa \psi|^2
			+
			 \epsilon |\zeta(J_{X,\gamma_a})|
			 \\
			 &+  |\slashed{\zeta}|^2 |\pa \psi|_{X, \mB}^2
			 +\left(1 + \frac{1}{\delta}\right) v |P|^2
			 + \frac{1}{(1+s)^{1/2}} |P| |\pa \psi|
		 \label{zetamBJ}
		\end{align}

	The modified scalar current $\mK_{X, \gamma, P}$
	satisfies
	\begin{align}
		\label{mBmodifiedKbound}
		|\mK_{X, \gamma, P}| &\lesssim
		\left( |\nabla \gamma| + \frac{|\gamma|}{1+s}
		+ \frac{|X^\ell_{\mB}|^{1/2}}{|X^n_{\mB}|^{1/2}}
		\left(|\nabla_{\evm} \gamma|+ |\nas \gamma|\right)
		\right)  |\pa \psi|_{X, \mB}^2
		+ \frac{1}{(1+v)^{1/4}} |F|  |\pa \psi|_{X, \mB}\\
		&+   \left( |\nabla P^u| + \frac{|P^u|}{1+s}\right)|X^n_{\mB}|^{1/2}|\pa \psi|_{X, \mB}\\
		&+
		 \left(|\nabla_{\evmB} P| + |\nas P|
		+ \frac{1}{1+v} |\nabla P| + \frac{1}{1+v} |P|
		 \right) |X^\ell_{\mB}|^{1/2}|\pa \psi|_{X, \mB}\\
		&+ \epsilon\left( \frac{1}{(1+v)^{3/2}} |\pa \psi|^2 + \frac{1}{(1+v)^{1/2}} (|\evmB \psi|^2 + |\nas \psi|^2)
		\right)\\
&+
\frac{1}{(1+s)^{1/2}} \left( |\nabla P^u| + \frac{|P^u|}{1+v}\right) |\pa \psi|
+v|P| \left(|\nabla P| + \frac{|P|}{1+v} + |F|\right).
	\end{align}
	\end{prop}
	\begin{remark}
	 It will be important for our applications to keep track of the component
	 $P^u$ separately from the others; see Lemma \ref{higherordereqncentral}. The remaining
	 components only enter nonlinearly or after being differentiated
	 in the $\evmB$ or $\nas$ directions.
	\end{remark}
	Propositions \ref{mainminkidentprop} and \ref{mainmBidentprop} follow from the upcoming sequence
	of lemmas and the proofs can be found at the end of the section.

	\subsection{The proof of Proposition \ref{mainminkidentprop}}

We are going to need a slightly different result in the
region near the light cone $u \sim 0$ and the region away from
the light cone. We fix a $C^\infty$ cutoff function $\chi_0$ with  $\chi_0(\rho) = 1$ when
$|\rho| \leq 1$ and $\chi_0(\rho) = 0$ when $|\rho| \geq 2$.
We then set $\chi(u, v) = \chi_0(|u|/4v)$ so that $\chi \equiv 1$ when $|u| \leq v/4$
and $\chi \equiv 0$ when $|u| \geq v/2$. Also $\nabla \chi$ is supported only
in the region $v/4 \leq |u| \leq v/2$, and $|\nabla \chi|\lesssim \frac{1}{1+v} |\chi_0'|$.
Writing $\gamma^{\mu\nu} = \gamma_\chi^{\mu\nu} + \gamma_{1-\chi}^{ {\mu\nu}}$,
with $\gamma_\chi = \chi \gamma$ and $\gamma_{1-\chi} = (1-\chi)\gamma$,
we then have the following bound
\begin{equation}
 |\nabla_X \gamma_{\chi}|
 \lesssim
 \chi |\nabla_X \gamma| + \frac{1}{1+v} |\chi_0'| |X| |\gamma|
 \label{}
\end{equation}
and similarly with $\chi$ replaced with $1-\chi$. We will use these bounds
repeatedly in what follows.

The contribution from $ \gamma_{1-\chi}$ and
$P_{1-\chi} = (1-\chi)P$ can be handled using the standard identity
\eqref{mainlinident} so we just write
\begin{equation}
	\pa_\mu(\gamma_{1-\chi}^{\mu\nu}\pa_\nu\psi + P_{1-\chi}^\mu) X\psi
	= \pa_\mu J_{X, \gamma_{1-\chi}, P_{1-\chi}}^\mu
	+ K_{X, \gamma_{1-\chi}, P_{1-\chi}},
 \label{}
\end{equation}
where, by \eqref{zetaJbound0}, \eqref{naiveKbd}, we have
\begin{equation}
 |\zeta(J_{X, \gamma_{1-\chi}, P_{1-\chi}})|
 \lesssim |\gamma_{1-\chi}| |X||\pa \psi|^2 +
 |P_{1-\chi}| |X| |\pa \psi|,
 \label{1minuschiJ}
\end{equation}
\begin{multline}
 |K_{X, \gamma_{1-\chi}, P_{1-\chi}}|
 \lesssim
 (1-\chi)\left(|\nabla \gamma| |X| |\pa\psi|^2
 + |\gamma| |\pa X| |\pa \psi|^2
 + (|\pa P| |X| + |\pa X| |P|) |\pa \psi|\right)
 \\
 + \frac{1}{1+v} |\chi_0'|
 \left(|\gamma| |X| |\pa\psi|^2 + |P| |X| |\pa \psi|\right)
 + \frac{1}{r}(1-\chi) \left(|\gamma| |X| |\pa \psi|^2 +|P| |X| |\pa \psi| \right).
 \label{1minuschiK}
\end{multline}
To get the bound \eqref{trivialKboundlie} involving the Lie derivative,
we just use the bound \eqref{scalarcurrentbd0Lie}.

We now carry out the calculation for $\gamma_\chi = \chi \gamma$.
We start by handling the ``good'' terms, which are those which do
not involve products between $X^v$ and $u$-derivatives of $\psi$.
\begin{lemma}
	\label{J1K1lem}
Under the hypotheses of Proposition \ref{mainminkidentprop},
with $\gamma_\chi = \chi \gamma$, we have
 \begin{equation}
  \pa_\mu(\gamma_\chi^{\mu\nu}\pa_\nu \psi)X\psi =
	\pa_u(\gamma_\chi^{uu}\pa_u \psi) X^v \pa_v \psi + \pa_\mu {J}^{1,\mu}_X
	+ {K}_X^1,
  \label{}
 \end{equation}
 where $\widetilde{J}^{1,\mu}_X$ and $\widetilde{K}_X^1$
 are given explicitly in \eqref{j1explicit},
 and satisfy the following bounds. For any $\delta > 0$,
 \begin{align}
  |\zeta({J}^{1,\mu}_X)| &\lesssim
	\chi \delta |X^\ell_m||\evm \psi|^2
	+\chi \left(\left(1 + \frac{1}{\delta}\right)|\gamma| |\pa \psi|_{X,m}^2
	+ |\zeta(X)| |\gamma| |\pa \psi|^2\right), \label{J1boundm}\\
	|{K}_X^1| &\lesssim
	\left( \chi|\nabla \gamma| + (\chi+|\chi_0'|)\frac{1}{1+|u|} |\gamma|
	+ \chi\frac{|X^\ell_m|^{1/2}}{|X^n_m|^{1/2}} |\nabla_{\evm} \gamma|
	\right)  |\pa \psi|_{X, m}^2
	+ \chi\frac{|X^n_m|}{|X^\ell_m|^{1/2}} |{F}|  |\pa \psi|_{X, m}\\
	&+ \chi |X^n_m|^{1/2} \left( |\nabla P| + \frac{1}{1+v} |P|\right) |\pa \psi|_{X, m}
  \label{K1boundm}
	\end{align}
\end{lemma}

\begin{proof}

\textit{Step 1: Separating the bad terms}

We start by separating out
the terms with $(\mu,\nu) \in\{(u, u),  (v, u), (u, v)\}$,
\begin{multline}
 \pa_\mu(\gamma_\chi^{\mu\nu}\pa_\nu\psi)
 = \pa_u(\gamma_\chi^{uu}\pa_u\psi) + \pa_v(\gamma_\chi^{vu}\pa_u\psi)
 + \pa_u (\gamma_\chi^{uv}\pa_v \psi)
 + \pa_\mu({\gamma}_{1}^{\mu\nu} \pa_\nu \psi)\\
 =
  \pa_u(\gamma^{uu}_\chi\pa_u\psi) + (\gamma_\chi^{vu} + \gamma_\chi^{uv})
	\pa_v\pa_u\psi
 + \pa_\mu({\gamma}_1^{\mu\nu} \pa_\nu \psi)
 + (\pa_v \gamma^{vu}_\chi)\pa_u \psi
 + (\pa_u \gamma^{uv}_\chi) \pa_v\psi,
 \label{startmodifiedapp}
\end{multline}
where ${\gamma}_1^{\mu\nu}$ vanishes when
$(\mu,\nu) \in \{(u,u), (v, u), (u, v)\}$,
\begin{equation}
 {\gamma}_1^{\mu\nu} = \gamma_\chi^{\mu\nu} - \delta^{\mu}_{u} \delta^{\nu}_{u} \gamma_\chi^{uu}
 -\delta^{\mu}_{v}\delta^{\nu}_{u} \gamma_\chi^{vu}
 -\delta^{\mu}_u\delta^{\nu}_v \gamma_{\chi}^{uv}.
 \label{gamma1def}
\end{equation}
We first deal with the contribution from $\gamma_1$ into \eqref{startmodifiedapp}.
Writing $\gamma_1(\pa\psi, \pa \psi) = \gamma_1^{\mu\nu}\pa_\mu\psi\pa_\nu\psi$,
\begin{equation}
 |\gamma_1(\pa \psi, \pa \psi)| \lesssim \chi|\gamma| \left(|\pa_v\psi| + |\nas \psi|\right)|\pa \psi|,
 \label{}
\end{equation}
and if $X = X^u\pa_u + X^v\pa_v$ then
\begin{multline}
 |(X^\alpha\pa_\alpha \gamma_1^{\mu\nu}) \pa_\mu\psi\pa_\nu\psi|
 \lesssim |X^\alpha \pa_\alpha \gamma_1| \left(|\pa_v\psi| + |\nas \psi|\right)|\pa \psi|\\
 \lesssim \left(\chi |\nabla_X \gamma|
 + \frac{\chi + |\chi_0'|}{1+v} |X| |\gamma|\right)\left(|\pa_v\psi| + |\nas \psi|\right)|\pa \psi|,
 \label{}
\end{multline}
where we bounded $|X^\alpha\pa_\alpha \gamma|\lesssim |\nabla_X \gamma| +
|\Gamma| |X||\gamma|$, where the Christoffel symbols $\Gamma$ satisfy
$|\Gamma| \lesssim \frac{1}{r} \lesssim \frac{1}{1+v}$ on the support of
$\chi$. We also note that
\begin{equation}
 |\pa_\mu X^\alpha \gamma_1^{\mu\nu}\pa_\nu \psi \pa_\alpha \psi|
 \lesssim \chi |\pa X^u| |\gamma| |\pa \psi|^2
 + \chi |\pa X| |\gamma||\pa_v\psi| |\pa \psi|.
 \label{}
\end{equation}

As a result, we have the identity
\begin{equation}
 \pa_\mu({\gamma}_1^{\mu\nu} \pa_\nu \psi) X\psi
 = \pa_\mu J_{X, \gamma_1}^\mu + K_{X, \gamma_1}
 \label{}
\end{equation}
where
\begin{align}
 J_{X, \gamma_1}^\mu &= \gamma_1^{\mu\nu}\pa_\nu \psi X\psi -
 \frac{1}{2}X^\mu \gamma_1(\pa\psi,\pa \psi),\label{JXgamma1}\\
 K_{X, \gamma_1} &=\pa_\mu X^\alpha \gamma_1^{\mu\nu} \pa_\nu \psi \pa_\alpha \psi
 - \frac{1}{2} \pa_\alpha X^\alpha \gamma_1(\pa\psi, \pa \psi)
 - \frac{1}{2} (X^\alpha\pa_\alpha \gamma_1^{\mu\nu}) \pa_\mu \psi\pa_\nu \psi.
 \label{}
\end{align}
which satisfy
\begin{align}
 |\zeta(J_{X, \gamma_1})| &\lesssim \chi
 \left(|\gamma| |\pa \psi| |X\psi| + |\zeta(X)||\gamma| |\pa \psi|^2\right),\label{zetaXbd0}\\
 |K_{1, X}| & \lesssim \left(\chi |\nabla_X \gamma|
 + \frac{\chi + |\chi_0'|}{1+v} |X| |\gamma| + \chi |\pa X| |\gamma|\right)\left(|\pa_v\psi| + |\nas \psi|\right)|\pa \psi|
 + \chi|\pa X^u| |\gamma| |\pa \psi|^2.
\end{align}
We now bound
\begin{equation}
 |\gamma| |\pa \psi| |X\psi|
 \lesssim |X^n| |\gamma||\pa \psi|^2 + |X^\ell| |\gamma| |\pa \psi| |\evm \psi|
 \lesssim \delta |X^\ell_m| |\evm \psi|^2 +\left(1 +  \frac{1}{\delta}\right)|\gamma||\pa \psi|^2_{X, m}
 \label{J1boundm0}
\end{equation}
for any $\delta >0$,
and so the first line of \eqref{zetaXbd0} is bounded by the right-hand side of
\eqref{J1boundm}. We also have
\begin{multline}
 |\nabla_X \gamma| (|\pa_v \psi| + |\nas \psi|) |\pa \psi|
 \lesssim
 |X^n_m| |\nabla \gamma| |\pa \psi|^2
 + \frac{|X^\ell_m|^{1/2}}{|X^n_m|^{1/2}}
 |\nabla_{\evm} \gamma| |X^\ell_m|^{1/2} (|\pa_v \psi| + |\nas \psi|)
 |X^n_m|^{1/2} |\pa \psi|\\
 \lesssim
 \left(|\nabla \gamma| + \frac{|X^\ell_m|^{1/2}}{|X^n_m|^{1/2}}
 |\nabla_{\evm} \gamma| \right)|\pa \psi|_{X, m}^2,
 \label{}
\end{multline}
as well as
\begin{equation}
 \frac{|X|}{1+v} |\gamma| \left( |\pa_v \psi| + |\nas \psi|\right) |\pa \psi|
 \lesssim
 \left[  {\left(\frac{|X^\ell_m|^{1/2}}{ |X^n_m|^{1/2}}+\frac{|X^n_m|^{1/2}} {|X^\ell_m|^{1/2}}\right)}    \frac{1+|u|}{1+v}\right] \frac{|\gamma|}{1+|u|}
 |\pa \psi|_{X,m}^2,
 \label{growthuseerror1}
\end{equation}
and similarly
\begin{equation}
 |\pa X| |\gamma| \left( |\pa_v \psi| + |\nas \psi|\right) |\pa \psi|
 \lesssim
 \left[{\frac{|\pa X| }{|X^n_m|^{1/2}|X^\ell_m|^{1/2}} (1+|u|)}\right] \frac{|\gamma|}{1+|u|}
 |\pa \psi|_{X,m}^2,
 \label{growthuseerror2}
\end{equation}
and after using the hypotheses \eqref{Xgrowth} on $X$, these satisfy \eqref{K1boundm}.

\textit{Step 2: Using the equation}

We now deal with the second term in \eqref{startmodifiedapp}.
Introducing
\begin{equation}
 \check{\gamma} = \gamma_{\chi}^{vu} + \gamma_{\chi}^{uv},
 \label{}
\end{equation}
and using the equation \eqref{modelwaveapdx}
written in the form \eqref{modelwavenull}, the second term in \eqref{startmodifiedapp}
is
\begin{align}
 \check{\gamma} \pa_v\pa_u\psi
 &= \check{\gamma}\left(\frac{1}{4}\sDelta \psi
 - \pa_\mu(\gamma^{\mu\nu}\pa_\nu \psi)
 +  F - \pa_\mu P^\mu\right)\\
 &= \nas \cdot \left(\frac{1}{4}\check{\gamma} \nas \psi\right)
 - \pa_\mu( \check{\gamma} \gamma^{\mu\nu}\pa_\nu \psi)
 + \check{\gamma} F
 - \check{\gamma}\pa_\mu P^\mu
 + (\pa_\mu \check{\gamma})\gamma^{\mu\nu}
 \pa_\nu \psi
 - \frac{1}{4}\nas\check{\gamma}\cdot \nas \psi,
 \label{}
\end{align}
so \eqref{startmodifiedapp} reads
\begin{equation}
 \pa_\mu(\gamma_\chi^{\mu\nu}\pa_\nu\psi)
 = \pa_u(\gamma_\chi^{uu}\pa_u\psi) + \pa_\mu({\gamma}_1^{\mu\nu} \pa_\nu \psi)
 + \pa_\mu({\gamma}_2^{\mu\nu} \pa_\nu \psi)
 + F_1 + F_2,
 \label{modified0}
\end{equation}
where ${\gamma}_2^{\mu\nu}$ is given by
\begin{equation}
 {\gamma}_2^{\mu\nu} = \frac{1}{4} \slashed{\Pi}^{\mu\nu}\check{\gamma} {-}  \gamma^{\mu\nu}\check{\gamma}
 ,
 \qquad
 \slashed{\Pi}^{\mu\nu} = m^{\alpha\beta} \slashed{\Pi}_\alpha^\mu \slashed{\Pi}_\beta^\nu,
 \label{gamma2def}
\end{equation}
{where $\slashed{\Pi}$ denotes projection to the tangent space
to the spheres $v+u = constant$ and $v-u = constant$, defined as in \eqref{angularproj}}.
The terms $F_1, F_2$ are
\begin{align}
 F_1 &= \check{\gamma} F,\label{F1}\\
 F_2 &= - \check{\gamma}\pa_\mu P^\mu
 + {\pa_\mu}\check{\gamma} \gamma^{\mu\nu}
 \pa_\nu \psi
 - \frac{1}{4}\nas\check{\gamma}\cdot \nas \psi+ (\pa_v \gamma^{vu}_\chi)\pa_u \psi
 + (\pa_u \gamma^{uv}_\chi) \pa_v\psi,
 \label{F2}
\end{align}
and we now verify that $(|F_1|  +|F_2|)|X\psi|$ is bounded by the
right-hand side of \eqref{K1boundm}. First, noting that
$|X \psi| \lesssim |X^\ell_m|^{1/2} |\pa \psi|_{X, m}$, we have
\begin{equation}
 |F_1| |X\psi|
 \lesssim
 |\gamma||F| |X^\ell_m|^{1/2} |\pa \psi|_{X, m}
 \lesssim
 \frac{|X^n|}{|X^\ell_m|^{1/2}} |F| |\pa \psi|_{X, m}.
 \label{}
\end{equation}

We also have $|\pa_\mu P^\mu| \lesssim |\nabla P| + \frac{1}{1+v} |P|$
on the support of $\chi$, and so on the support of $\chi$,
\begin{equation}
 |\gamma \pa_\mu P^\mu| |X \psi|
 \lesssim |\gamma| \left( |\nabla P| + \frac{1}{1+v} |P|\right)
 |X^\ell_m|^{1/2} |\pa \psi|_{X, m}
 \lesssim
 \left[ \frac{|X^n_m|}{|X^\ell_m|^{1/2}} \right]
 \left(|\nabla P| + \frac{1}{1+v} |P|\right) |\pa \psi|_{X, m},
 \label{}
\end{equation}
which is bounded by the right-hand side of \eqref{K1boundm}.
Bounding
\begin{equation}
 |\pa_v \gamma^{vu}_\chi|\lesssim
 |\nabla_v \gamma_\chi| + \frac{1}{r} |\gamma_\chi|
 \lesssim
 \chi|\nabla_{\evm} \gamma| + \frac{\chi + |\chi_0'|}{1+v} |\gamma|,
 \label{}
\end{equation}
we find
\begin{multline}
 |\pa_v \gamma^{uv}_\chi| |\pa \psi| |X\psi|
 \lesssim \frac{|X^\ell_m|^{1/2}}{|X^n_m|^{1/2}} |\pa_v \gamma^{uv}_\chi| |\pa \psi|_X^2
 \\
 \lesssim \chi \frac{|X^\ell_m|^{1/2}}{|X^n_m|^{1/2}}|\nabla_{\evm} \gamma| |\pa \psi|_X^2
 +
 (\chi + |\chi_0'|)
 \left[
 \frac{|X^\ell_m|^{1/2}}{|X^n_m|^{1/2}}
 \frac{1+|u|}{1+v}
 \right]
  \frac{|\gamma|}{1+|u|} |\pa \psi|_X^2,
 \label{}
\end{multline}
which is bounded by the right-hand side of \eqref{K1boundm}. By a similar
argument and the bound
\begin{equation}
 |\pa_u \gamma^{uu}_\chi| |\pa_v \psi| |X\psi|
 \lesssim |\pa \gamma_\chi| |\pa \psi|_{X, m}^2,
 \label{}
\end{equation}
this term is also bounded by the right-hand side of \eqref{K1boundm},
and similarly $|\nas \gamma| |\nas \psi| |X\psi|
\lesssim |\nas \gamma| |\pa \psi|_{X, m}$, and as a result
 $
 (|F_1|  + |F_2|)|X\psi|$
 is bounded by the right-hand side of \eqref{K1boundm}.

We now multiply the expression \eqref{modified0} by $X \psi
= (X^v\pa_v + X^u\pa_u)\psi$. We will need to treat the product
$\pa_u(\gamma^{uu}\pa_u\psi) X^v\pa_v \psi$ differently from
the other terms this generates and so we write
\begin{equation}
	\pa_\mu(\gamma_\chi^{\mu\nu}\pa_\nu \psi) X\psi
	= \pa_u(\gamma_\chi^{uu}\pa_u\psi) X^v\pa_v \psi
	+ {J}_{X}^{1,\mu} + {K}_{X}^1.
 \label{modified1}
\end{equation}
Here,
\begin{equation}
    {J}_{X}^{1,{\mu}} = J_{X, \gamma_1}^{{\mu}} + J_{X, \gamma_2}^{{\mu}} + J_{X^u\pa_u,\gamma_\chi}^{{\mu}},
	\qquad
	K_{X}^1 = K_{X, \gamma_1}
	+ K_{X, \gamma_2} + K_{X^u\pa_u,\gamma_\chi} + (F_1 + F_2) X\psi,
	\label{j1explicit}
\end{equation}
where we have used the identity \eqref{JKrelation} and where
the $J_{X, \gamma}$
are defined as in 
\eqref{JXgamma1}.
The quantity $J_X^{1,\mu}$ satisfies the bound
\begin{equation}
 |\zeta(J_X^1)|\lesssim
 \chi \left(|\gamma| |\pa \psi| |X\psi| + |\zeta(X)| |\gamma| |\pa \psi|^2 +
 |X^u||\gamma| |\pa \psi|^2\right),
\end{equation}
which can be bounded by the right-hand side of \eqref{J1boundm} as in
\eqref{J1boundm0}.

To handle $K^1_X$, we write $K_{X, \gamma_2} = \slashed{K}_X + \check{K}_X$ where
$\slashed{K}$ collects the terms involving angular derivatives and $\check{K}$
collects the terms involving products between $\gamma$ and $\check{\gamma}$,
\begin{align}
	\slashed{K}_X &= \frac{1}{8} \pa_\alpha \left(\slashed{\Pi}^{\mu\nu}\check{\gamma} X^\alpha\right)
 \pa_\mu \psi \pa_\nu \psi
 =\frac{1}{8} \pa_\alpha X^\alpha |\nas \psi|^2  \check{\gamma}
 + \frac{1}{8} ( X^\alpha\pa_\alpha \check{\gamma} ) |\nas \psi|^2
 \label{slashedK}\\
 \check{K}_X &= -\frac{1}{2} \pa_\alpha \left( \gamma^{\mu\nu} \check{\gamma} X^\alpha\right)
 \pa_\mu \psi \pa_\nu \psi
 + \pa_\mu X^\alpha \gamma^{\mu\nu} \check{\gamma} \pa_\nu \psi \pa_\alpha \psi.
 \label{checkedK}
\end{align}

Noting the bounds
\begin{equation}
 |\check{\gamma}| \lesssim \chi |\gamma|,
 \qquad
 |\nabla_X \check{\gamma}| \lesssim \chi |\nabla_X \gamma| + \frac{|X|}{1+v} |\chi_0'| |\gamma|,
 \label{}
\end{equation}
and that on the support of $\chi$, the Christoffel symbols $\Gamma$ satisfy
$|\Gamma| \lesssim \frac{1}{1+v}$, we have the bounds
\begin{align}
	|\slashed{K}_X|
	&\lesssim \chi |\pa X| |\gamma| |\nas \psi|^2 + \chi |\nabla_X \gamma| |\nas \psi|^2
	+ \frac{|X|}{1+v} \left( \chi + |\chi'_0|\right) |\gamma| |\nas \psi|^2,
	\\
 |\check{K}_X|
 &\lesssim
 \chi |\pa X| |\gamma|^2 |\pa \psi|^2 + \chi |\nabla_X \gamma| |\gamma| |\pa \psi|^2
 + \frac{|X|}{1+v} \left( \chi + |\chi'_0|\right) |\gamma|^2 |\pa \psi|^2.
 \label{}
\end{align}
In particular,
\begin{equation}
 |\slashed{K}_X|
 \lesssim \chi \left[ \frac{|\pa X|}{|X^\ell_m|} (1+|u|) \right]
 \frac{|\gamma|}{1+|u|} |\pa \psi|^2_{X, m}
 +\chi |\nabla \gamma| |\pa \psi|_{X, m}^2
 +\left( \chi + |\chi'_0|\right)
  \left[ \frac{1 +|u|}{1+v} \right] \frac{|\gamma|}{1+|u|} |\pa \psi|_{X,m}^2,
 \label{}
\end{equation}
and, bounding $|\gamma| \leq \frac{|X^n_m|}{|X^\ell_m|}$
and
$|\nabla_X \gamma||\gamma| \lesssim |X^n_m| |\nabla \gamma|$,
\begin{multline}
 |\check{K}_X|\lesssim
 \chi \left[\frac{|\pa X|}{|X^\ell_m|} {|X^n|^{1/2}} (1+|u|) \right]\frac{|\gamma|}{1+|u|}
 |\pa \psi|_{X, m}^2\\
 + \chi |\nabla \gamma| |\pa \psi|_{X,m}^2
 + \left( \chi + |\chi'_0|\right) \left[\frac{1+|u|}{1+v}\right] \frac{|\gamma|}{1+|u|}
 |\pa \psi|_{X, m}^2,
 \label{}
\end{multline}
as needed.

From the formula 
\begin{equation}
 K_{X^u\pa_u,\gamma_\chi} = \frac{1}{2} \pa_u(X^u \gamma_{\chi}^{\mu\nu})\pa_\mu\psi
 \pa_\nu \psi - \pa_\mu X^u \gamma_\chi^{\mu\nu}\pa_\nu \psi \pa_u \psi,
 \label{}
\end{equation}
we also have
\begin{multline}
 |K_{X^u\pa_u,{\gamma_\chi}}|
 \lesssim
 \chi|X^u| |\nabla \gamma| |\pa \psi|^2
 + \chi|\pa X^u| |\gamma| |\pa \psi|^2 + \frac{\chi + |\chi_0'|}{1+v} |X^u| |\gamma| |\pa \psi|^2\\
 \lesssim
 \chi |\nabla \gamma| |\pa \psi|_{X, m}^2
 + \chi \left[ \frac{|\pa X^u|}{|X^n_m|} (1 +|u|)\right] \frac{|\gamma|}{1+|u|}
 |\pa \psi|_{X, m}^2
 + (\chi + |\chi_0'|)
 \left[ \frac{1+|u|}{1+v}\right] \frac{|\gamma|}{1+|u|} |\pa \psi|_{X, m}^2,
 \label{}
\end{multline}
which also satisfies the needed bounds.

\end{proof}

We now manipulate the first term on the right-hand side of \eqref{modified1}.
\begin{lemma}
	\label{J2K2lem}
	Under the hypotheses of Proposition \ref{mainminkidentprop}, we have
 \begin{equation}
  \pa_u(\gamma_\chi^{uu}\pa_u \psi) X^v \pa_v \psi = \pa_\mu J_X^{2, \mu}
	+ K^2_X,
	\label{J2K2ident}
	\end{equation}
	where $J_X^{2,\mu}, K_X^2$ are given explicitly in \eqref{j2k2} and satisfy
\begin{align}
  |\zeta(J_X^{2})|
	&\lesssim \chi \left( |\slashed{\zeta}|^2 |\pa \psi|_{X,m}^2
  +  |\gamma| |\pa \psi|_{X,m}^2\right)\\
	|K_X^2|
	&\lesssim
	\left(
	\chi |\nabla \gamma| +
  (\chi + |\chi_0'|)  \frac{|\gamma|}{1+|u|}
	+\chi
	\frac{|X^\ell_m|^{1/2}}{|X^n_m|^{1/2}} |\nas \gamma|
  \right)
  |\pa \psi|_{X, m}^2\\
	&+ \chi|X^n_m|^{1/2}\left( |F| + |\nabla P|\right) |\pa \psi|_{X,m}
	+ \chi |X^n_m|^{1/2} \frac{|P|}{1+|u|} |\pa \psi|_{X, m}.
 \label{}
\end{align}
\end{lemma}

\begin{proof}

Using the equation \eqref{modelwaveapdx} again,
\begin{multline}
 \pa_u(\gamma_{\chi}^{uu}\pa_u\psi) X^v\pa_v \psi
 = \pa_u (\gamma_{\chi}^{uu}\pa_u \psi X^v\pa_v \psi)
 - \gamma_{\chi}^{uu}\pa_u \psi X^v \pa_u \pa_v \psi -  \gamma_{\chi}^{uu}\pa_u X^v\pa_u \psi \pa_v \psi\\
 = \pa_u(\gamma_{\chi}^{uu}\pa_u \psi X^v\pa_v\psi)
 - \frac{1}{4} X^v \gamma_{\chi}^{uu} \pa_u\psi \sDelta \psi
 + \gamma_{\chi}^{uu}\pa_u\psi X^v \pa_\mu(\gamma^{\mu\nu}\pa_\nu \psi)
 \\
 -\gamma_\chi^{uu} \pa_u\psi X^v \pa_\mu P^\mu
 + \gamma_{\chi}^{uu}\pa_u\psi X^v F - \gamma_{\chi}^{uu} \pa_u X^v\pa_u \psi\pa_v \psi.
 \label{modified2}
\end{multline}
The first and second terms on the third line here are bounded by
\begin{equation}
 |\gamma_{\chi}^{uu}\pa_u\psi X^v F|
 \lesssim \chi |\gamma||X| |\pa \psi| |F| \lesssim \chi|X^n_m|^{1/2} |F| |\pa \psi|_{X,m},
 \label{}
\end{equation}
where we used the assumptions \eqref{pert1}, and
\begin{multline}
 |\gamma_\chi^{uu} \pa_u\psi X^v \pa_\mu P^\mu|
 \lesssim \chi |\gamma| |X| |\pa \psi| |\nabla P|
 + \frac{\chi}{1+v} |\gamma| |X| |\pa \psi| |P|\\
 \lesssim \chi |X^n_m|^{1/2} |\nabla P| |\pa \psi|_{X, m}
 + \chi \left[\frac{1+|u|}{1+v}\right] \frac{|X^n_m|^{1/2}}{1+|u|} |P| |\pa \psi|_{X, m},
 \label{}
\end{multline}
as needed.

The second term on the right-hand side of \eqref{modified2} is
\begin{multline}
  -\frac{1}{4} X^v \gamma_{\chi}^{uu} \pa_u\psi \sDelta \psi
	= -\frac{1}{4}\nas \cdot\left(X^v \gamma_{\chi}^{uu} \pa_u \psi \nas \psi\right)
	+ \frac{1}{8} \pa_u \left( X^v \gamma_{\chi}^{uu} |\nas \psi|^2\right)
	\\
	- \frac{1}{8} \pa_uX^v  \gamma_{\chi}^{uu} |\nas \psi|^2
	+ \frac{1}{4} X^v \nas \gamma_{\chi}^{uu} \pa_u \psi \nas \psi
	- \frac{1}{8} X^v \pa_u\gamma_{\chi}^{uu} |\nas \psi|^2
	{+ \frac{1}{4} [\nas, \pa_u] \psi \cdot \nas \psi},
 \label{modified3}
\end{multline}
which can be written in the form
\begin{equation}
-\frac{1}{4} X^v \gamma_{\chi}^{uu} \pa_u\psi \sDelta \psi
= \pa_\mu\slashed{\widehat{J}}^{\mu}_X + \slashed{\widehat{K}}_X
\label{modified4}
\end{equation}
with
\begin{align}
 \slashed{\widehat{J}}^{\mu}_X
 &= -\frac{1}{4}X^v \gamma_{\chi}^{uu} \pa_u\psi \nas^\mu \psi
 + \frac{1}{8} \delta^{\mu u} X^v \gamma_{\chi}^{uu} |\nas \psi|^2,\\
 \slashed{\widehat{K}}_X
 &= \frac{1}{4} X^v \nas \gamma_{\chi}^{uu} \pa_u \psi \nas \psi
 - \frac{1}{8} X^v \pa_u\gamma_{\chi}^{uu} |\nas \psi|^2
 {+ \frac{1}{4}X^v\gamma_{\chi}^{uu} [\nas, \pa_u] \psi \cdot \nas \psi}.
 \label{}
\end{align}
These satisfy
\begin{align}
 |\zeta(\slashed{\widehat{J}})|
 &\lesssim
 \chi |X||\gamma|\left(|\pa \psi||\zeta(\nas \psi)|
 + |\nas \psi|^2\right),\\
 |\slashed{\widehat{K}}_X|
 &\lesssim
 \chi |X| |\nas \gamma| |\pa \psi| |\nas \psi|
 + \chi |X| |\nabla \gamma| |\nas \psi|^2
 + \frac{\chi + |\chi_0'|}{1+v}|X| |\gamma| |\nas \psi|^2,
 \label{}
\end{align}
using the same arguments as in the previous lemma to handle terms involving
derivatives of $\gamma_\chi$,{and where we bounded $|[\nas, \pa_u]\psi| \lesssim \frac{1}{1+v} |\nas \psi|$
on the support of $\chi$, which follows after writing $\nas = \frac{1}{r} \Omega$ and noting
that $[\pa_u, \Omega] = 0$.}
As a result,
\begin{equation}
 |\zeta(\slashed{\widehat{J}})|
 \lesssim \chi |X| |\gamma| \left( |\slashed{\zeta}|^2 |\pa \psi|^2 + |\nas \psi|^2\right)
 \lesssim \chi \epsilon |\slashed{\zeta}|^2 |\pa \psi|_{X,m}^2
 + \chi |\gamma| |\pa \psi|_{X,m}^2,
 \label{}
\end{equation}
and
\begin{equation}
 |\slashed{\widehat{K}}_X|
 \lesssim \chi\frac{|X^\ell_m|^{1/2}}{|X^n_m|^{1/2}} |\nas \gamma| |\pa \psi|_{X, m}^2
 + \chi |\nabla \gamma| |\pa \psi|_{X, m}^2 +
 (\chi + |\chi_0'|) \left[\frac{1+|u|}{1+v} \right] \frac{|\gamma|}{1+|u|}
 |\pa \psi|_{X, m}^2,
 \label{}
\end{equation}
as needed.

Similarly, the third term in \eqref{modified2} is
\begin{multline}
 \gamma_{\chi}^{uu}\pa_u\psi X^v \pa_\mu(\gamma^{\mu\nu}\pa_\nu \psi)
 = \pa_\mu \left(X^v\gamma_{\chi}^{uu}\gamma^{\mu\nu}  \pa_u \psi \pa_\nu \psi\right)
 - \frac{1}{2} \pa_u \left(X^v\gamma_{\chi}^{uu}\gamma( \pa\psi, \pa\psi)\right)
 \\
 - \pa_\mu (X^v\gamma_{\chi}^{uu}) \gamma^{\mu\nu} \pa_u\psi\pa_\nu\psi
 + {\frac{1}{2} \pa_u\left(X^v \gamma_{\chi}^{uu} \gamma^{\mu\nu}\right) \pa_\mu\psi\pa_\nu\psi}
 \label{}
\end{multline}
 where $\gamma(\pa\psi,\pa\psi) = \gamma^{\mu\nu}\pa_\mu\psi\pa_\nu\psi$. This
can be written in the form
\begin{equation}
 \gamma_{\chi}^{uu}\pa_u\psi X^v \pa_\mu(\gamma^{\mu\nu}\pa_\nu \psi)
 = \pa_\mu \widehat{J}_{X}^\mu + \widehat{K}_{X}^\mu
 \label{}
\end{equation}
with
\begin{align}
 \widehat{J}_{X}^\mu &= X^v \gamma_{\chi}^{uu} \gamma^{\mu\nu}\pa_u\psi \pa_\nu \psi
 - \frac{1}{2} \delta^{\mu u} X^v \gamma_{\chi}^{uu} \gamma(\pa \psi, \pa \psi),\\
 \widehat{K}_{X} &= -\pa_\mu (X^v\gamma_{\chi}^{uu}) \gamma^{\mu\nu} \pa_u\psi\pa_\nu\psi
 + {\frac{1}{2} \pa_u\left(X^v \gamma_{\chi}^{uu} \gamma^{\mu\nu}\right) \pa_\mu\psi\pa_\nu\psi}.
 \label{}
\end{align}
These satisfy
\begin{align}
 |\zeta(\widehat{J}_X)|
 &\lesssim \chi |X| |\gamma|^2 |\pa \psi|^2,\\
 |\widehat{K}_X|
 &\lesssim {\chi}(|X| |\nabla \gamma| + |\pa X| |\gamma|)|\gamma| |\pa \psi|^2
 + \frac{\chi + |\chi_0'|}{1+v}|X| |\gamma|^2 |\pa \psi|^2,
 \label{}
\end{align}
which satisfy the needed bounds after using the same arguments we have now used
many times.
Combining the above, we have arrived at \eqref{J2K2ident} where $J^{2}_X$ and $K^2_X$
are given by
\begin{equation}
 J^2_X = \slashed{J}_X + \widehat{J}_X,
 \qquad
 K_X^2 = \slashed{K}_X + \widehat{K}_X -\gamma_\chi^{uu} \pa_u\psi X^v \pa_\mu P^\mu
 + \gamma_{\chi}^{uu}\pa_u\psi X^v F-
 \gamma_{\chi}^{uu} \pa_u X^v\pa_u \psi\pa_v \psi.
 \label{j2k2}
\end{equation}

\end{proof}

%
%

Finally, we handle the contribution from $P_\chi$.

\begin{lemma}
	\label{littlejklem}
Under the hypotheses of Proposition \ref{mainminkidentprop}, we have
 \begin{equation}
  \pa_\mu P_\chi^\mu X\psi = \pa_\mu j_{P,X}^\mu + k_{P, X},
  \label{littleident}
 \end{equation}
 where $j_{P, X}$ and $k_{P,X}$ are given explicitly in
 \eqref{littlej}-\eqref{littlek} and satisfy
 the following bounds. For any $\delta > 0$,
\begin{align}
 |\zeta(j_{X,P})|
 &\lesssim\chi \delta |X| |\evm \psi|^2
 + \chi \left(1 + \frac{1}{\delta}\right) |X| |P|^2
 + \chi |\slashed{\zeta}|^2 |\pa \psi|_{X, m}^2
 + \chi |X^n_m|^{1/2} |P| |\pa \psi|_{X, m},
 \label{littlejbd}\\
 |k_{X, P}|&\lesssim
  \left( \chi|X^n_m|^{1/2} |\nabla P| + (\chi+|\chi_0'|)|X^n_m|^{1/2} \frac{|P|}{1+|u|}
 + \chi|X^\ell_m|^{1/2} \left(|\nabla_{\evm} P|
 + |\nas P|\right) \right)|\pa \psi|_{X, m}\\
 &
 +
 \chi |P||X|  \left( |F| + \frac{1}{1+v} |P|\right)
 +
 \chi |P| |\pa_u X^v| |\evm \psi|.
 \label{littlekbd}
\end{align}
\end{lemma}
\begin{proof}
Here the problematic term
is $\pa_u P_\chi^u X^v\pa_v\psi$ so we separate it out and write
\begin{equation}
 \pa_\mu P_\chi^\mu X\psi =
 \pa_u P_\chi^u X^v\pa_v \psi +
 \pa_u P_\chi^u X^u \pa_u \psi
 + \pa_\mu \widetilde{P}_\chi^\mu X \psi
 =  \pa_u P_\chi^u X^v\pa_v \psi +
 \pa_u P_\chi^u X^u \pa_u \psi
+ \widetilde{k}_{P, X}
 \label{startP}
\end{equation}
To handle the last term we do not need to integrate by parts and
we just bound it directly by
\begin{align}
 |\widetilde{k}_{P, X}|
 &\lesssim
 \chi(|\pa_v P| + |\nas P|) |X\psi|
 + \frac{\chi + |\chi_0'|}{1+v} |P| |X \psi|
 \\
 &\lesssim
 \chi |X^\ell_m|^{1/2} (|\pa_v P| + |\nas P|) |\pa \psi|_{X, m}
 + \frac{\chi + |\chi_0'|}{1+v} |X^\ell_m|^{1/2} |P| |\pa \psi|_{X, m}
 \\
 &\lesssim
 \chi|X^\ell_m|^{1/2} (|\pa_v P| + |\nas P|) |\pa \psi|_{X, m}
 +(\chi + |\chi_0'|) \left[ \frac{|X^\ell_m|^{1/2}}{|X^n_m|^{1/2}}
 \frac{1+|u|}{1+v} \right] |X^n_m|^{1/2} \frac{|P|}{1+|u|} |\pa \psi|_{X, m},
 \label{goodPbdlittlek}
\end{align}
as needed. It is also straightforward to bound the second term on the right-hand side of \eqref{startP}
by \eqref{littlekbd}.

Using the equation for $\psi$, the first term in \eqref{startP} is
\begin{multline}
 \pa_u P_\chi^u X^v\pa_v \psi = \pa_u (P_\chi^u X^v\pa_v \psi)
 - P_\chi^u X^v \pa_u\pa_v \psi
 - P_\chi^u \pa_u X^v \pa_v \psi\\
 =
 \pa_u (P_\chi^u X^v\pa_v \psi)
 + \frac{1}{4} P_\chi^u X^v\sDelta \psi
 - P_\chi^uX^v\pa_\mu(\gamma^{\mu\nu}\pa_\nu \psi)
 + P_\chi^u X^v \pa_\mu P^\mu
 - P_\chi^uX^v F
 - P_\chi^u \pa_u X^v \pa_v \psi.
 \label{modifiedP1}
\end{multline}
Now we perform the same steps as in the previous lemma.
The second term in \eqref{modifiedP1} is
\begin{equation}
 \frac{1}{4} P_\chi^u X^v\sDelta \psi
 = \nas\cdot(\frac{1}{4} P_\chi^u X^v \nas \psi)
 - \frac{1}{4} X^v \nas P_\chi^u  \cdot \nas \psi,
 \label{}
\end{equation}
and the third is
\begin{equation}
 - P_\chi^uX^v\pa_\mu(\gamma^{\mu\nu}\pa_\nu \psi)
 = \pa_\mu(-P_\chi^uX^v \gamma^{\mu\nu}\pa_\nu \psi)
 + \pa_\mu(P_\chi^u X^v) \gamma^{\mu\nu}\pa_\nu \psi,
 \label{}
\end{equation}
so we have the identity \eqref{littleident} with
\begin{align}
 j_{X,P}^\mu &=  P_\chi^u X^v \left( \delta^{\mu u} \pa_v \psi + \frac{1}{4} \nas^\mu \psi - \gamma^{\mu\nu}
 \pa_\nu \psi\right),\label{littlej}\\
 k_{X,P} &= \widetilde{k}_{P, X}
 + \pa_u P^u_\chi X^u \pa_u \psi - \frac{1}{4}X^v \nas P_\chi^u\cdot \nas \psi
 + \pa_\mu(P_\chi^u X^v)\gamma^{\mu\nu}\pa_\nu \psi\\
 &+ P_\chi^u X^v \pa_\mu P^\mu
 - P_\chi^uX^v F
 - P_\chi^u \pa_u X^v \pa_v \psi\\
 &= \widetilde{k}_{P, X} + \pa_u P^u_\chi X^u \pa_u \psi + \widehat{k}_{P, X},
 \label{littlek}
\end{align}
where the first two terms in \eqref{littlek} are bounded as in
\eqref{goodPbdlittlek}.
For $j_{X, P}$,we bound
\begin{multline}
 |\zeta(j_{X, P})|
 \lesssim \chi |P| |X|(|\pa_v \psi| + |\slashed{\zeta}||\nas \psi| + |\gamma| |\pa \psi|),
 \\
 \lesssim
  \delta |X| |\evm \psi|^2
 + \chi \left(1 + \frac{1}{\delta}\right) |X| |P|^2
 + \chi |\slashed{\zeta}|^2 |\pa \psi|_{X, m}^2
 + \chi |X^n_m|^{1/2} |P| |\pa \psi|_{X, m},
 \label{}
\end{multline}
as needed, after bounding $|\gamma| |X| \lesssim |X^n_m|$.
To handle $\widehat{k}_{P, X}$, we bound
\begin{multline}
|\widehat{k}_{P, X}|
\lesssim
\chi|\nabla P| |X^n_m| |\pa \psi|
+ \chi|\nas P| |X^\ell_m| |\nas \psi|
+ \chi |\nabla P| |X| |\gamma| |\pa \psi|
+ \chi |P| |X| |\nabla P|\\
+ \chi |P| |X| |F| + \chi |P| |\pa_u X^v| |\evm \psi|
+ \frac{1}{1+v} (\chi + |\chi_0'|)|P|\left(  |X^n_m| |\pa \psi|
+ |X| |\gamma| |\pa \psi| + |X| |P|\right),
 \label{}
\end{multline}
which satisfies \eqref{littlekbd}, once more using the same arguments
we used in the previous two lemmas.
\end{proof}

\begin{proof}[Proof of Proposition \ref{mainminkidentprop}]

Combining Lemmas \ref{J1K1lem}, \ref{J2K2lem}, and \ref{littlejklem}, we arrive
at the identity
\begin{equation}
 \pa_\mu(\gamma^{\mu\nu}_\chi\pa_\nu \psi + P^\mu) X\psi
 = \pa_\mu \mJ_{X, P}^\mu + \mK_{X, P},
 \label{}
\end{equation}
where
\begin{align}
 \mJ_{X, P} &=  \mJ_X + j_{X, P} = J_X^1 + J_X^2 + j_{X, P},
 \label{explicitmJ}\\
 \mK_{X, P} &= \mK_X + k_{X, P}= K_X^1 + K_X^2 + k_{X, P},
 \label{explicitmK}
\end{align}
where $K_X^1, J^1_X$ are as in Lemma \ref{J1K1lem}, $K_X^2, J^2_X$ are as in
Lemma \ref{J2K2lem}, and $k_{X, P}, j_{X, P}$ are as in Lemma
\ref{littlejklem}.
To get the result, it remains only to see that \eqref{trivialKbound}
holds in the region $|u| \geq v/8$. By \eqref{1minuschiK}
it holds when $|u| \geq v/2$ and since $\mK_{X, P}$
satisfies \eqref{trivialKbound} as well, the result follows.
\end{proof}

\subsection{The proof of Proposition \ref{mainmBidentprop}}
\label{modifiedproofsecmB}
	The argument is nearly identical to the proof of the previous
	lemma, so we just indicate what the differences are, the main
	ones being that there are additional quantities
	involving $\frac{u}{vs}$ generated whenever we use the equation
	for $\pa_u\pa_v \psi$ and also that we need to keep better track
	of the terms involving $P$ since the $P$ in this region
	will satisfy worse estimates than the one we consider in the
	exterior.

  We note at this point that by our assumptions on $X^u, X^v$, we have
	\begin{equation}
	 |X^u| \lesssim |X^n_{\mB}| + \frac{1}{(1+s)^{1/2}}\lesssim X^n_{\mB},
	\end{equation}
	which we will frequently use in what follows. We also are assuming
	the condition \eqref{Xgrowth} but with $1+|u|$ replaced with
	$1+s$,
	\begin{equation}
	 |X^\ell_{\mB}|^{1/2} |X^n_{\mB}|^{1/2} \frac{1+s}{1+v}
	 + |\pa X| \frac{|X^n_{\mB}|^{1/2}}{|X^\ell_{\mB}|^{1/2}} \frac{1+s}{1+v}
	 + \frac{|\pa X^u|}{|X^n_{\mB}|^{1/2}}(1+ s)
	 \lesssim 1,
	 \label{Xgrowth2}
	\end{equation}
	which will be used to insert factors of $(1+s)^{-1}$ in front of some of
	the upcoming terms.

	We start with the following analogue
	of Lemma \ref{J1K1lem}.
	\begin{lemma}
		\label{J1K1lemmB}
	 Under the hypotheses of Proposition \ref{mainmBidentprop}, we have
	 \begin{equation}
	  \pa_\mu(\gamma^{\mu\nu}\pa_\nu\psi)X\psi = \pa_u \left(\gamma^{uu} \pa_u\psi\right) X^v \pa_v \psi
	  + \pa_\mu J_{X}^{1, \mu} + K_{X}^1
	  \label{}
	 \end{equation}
	 where
	 \begin{align}
	  |\zeta({J}^{1,\mu}_X)| &\lesssim
		\delta |X^\ell_{\mB}||\evmB \psi|^2
		+ \left(1 + \frac{1}{\delta}\right)|\gamma| |\pa \psi|_{X,\mB}^2
		+ |\zeta(X)| |\gamma| |\pa \psi|^2
		+
		 \epsilon |\zeta(J_{X,\gamma_a})|,
		  \label{mBJ1bd}\\
		|{K}_X^1| &\lesssim
		\left( |\nabla \gamma| + \frac{1}{1+s} |\gamma|
		+ \frac{|X^\ell_{\mB}|^{1/2}}{|X^n_{\mB}|^{1/2}} |\nabla_{\evm} \gamma|
		\right)  |\pa \psi|_{X, \mB}^2
		+ \frac{|X^n_{\mB}|}{|X^\ell_{\mB}|^{1/2}} |F|  |\pa \psi|_{X, \mB}\\
		&+  |X^n_{\mB}|^{1/2} \left( |\nabla P^u| + \frac{1}{1+v} |P^u|\right) |\pa \psi|_{X, \mB}\\
		&+  |X^\ev_{\mB}|^{1/2} \left( |\nabla_{\evmB} P| + |\nas P| + \frac{1}{1+v} |\nabla P|
		+ \frac{1}{1+v} |P|\right) |\pa \psi|_{X, \mB}\\
		&+ \epsilon\left( \frac{1}{(1+v)^{3/2}} |\pa \psi|^2 + \frac{1}{(1+v)^{1/2}} (|\evmB \psi|^2 + |\nas \psi|^2)
		\right)
	  \label{mBK1bd}
	 \end{align}
	\end{lemma}
	\begin{remark}
	 The linear terms above (the last terms on the right-hand side of
	 \eqref{mBJ1bd} and \eqref{mBK1bd}) are generated by the linear term
	 $\pa_\mu(\tfrac{u}{vs}a^{\mu\nu}\pa_\nu \psi)$ in our equation after we use
	 the equation for $n\evmB$, and they do not cause any serious difficulties.
	 The quantity $\zeta(J_{X,\gamma_a})$ is handled in Lemma \ref{nulllemmacurrent}
	 and uses the fact that $\gamma_a$ satisfies a null condition
	 (see \eqref{intronullcondn0}).
	\end{remark}
\begin{proof}

\textit{Step 1: Separating the bad terms}

	Since we will
	only need this argument in the region $|u| \ll v$, we do not need
	to introduce cutoff functions as in the proof of the previous
    result. Following the steps from the proof of Lemma \ref{J1K1lem},
	the identity \eqref{startmodifiedapp} is replaced by
\begin{multline}
 \pa_\mu(\gamma^{\mu\nu}\pa_\nu\psi)
 = \pa_u(\gamma^{uu} \pa_u\psi) + \pa_v(\gamma^{vu}\pa_u\psi)
 + \pa_u(\gamma^{uv}\pa_v\psi)
 + \pa_\mu({\gamma}_{1}^{\mu\nu} \pa_\nu \psi)\\
 =
  \pa_u(\gamma^{uu}\pa_u\psi) + (\gamma^{uv} + \gamma^{uv}) n\ell^{\mB}\psi
	+ \pa_\mu(\gamma_1^{\mu\nu}\pa_\nu \psi)
 + \pa_\mu({\gamma}_{1,\mB}^{\mu\nu} \pa_\nu \psi)\\
 + (\pa_v \gamma^{vu})\pa_u \psi
  + (\pa_u \gamma^{uv})\pa_v\psi
	+ \pa_u \left(\gamma^{uv} + \gamma^{uv}\right)\frac{u}{vs} \pa_u\psi,
 \label{startmodifiedmB}
\end{multline}
where, with $\gamma_1$ as in \eqref{gamma1def}, we have introduced
\begin{equation}
 \gamma_{1, \mB}^{\mu\nu} =  -
 (\gamma^{uv} + \gamma^{vu}) \frac{u}{vs} \delta^{\mu u}\delta^{\nu u}.
 \label{}
\end{equation}
We then have the identity
\begin{equation}
 \pa_\mu((\gamma_1^{\mu\nu} + \gamma_{1, \mB}^{\mu\nu})\pa_\nu \psi)
 = \pa_\mu J_{X, \gamma_1 + \gamma_{1, \mB}}^\mu
 + K_{X, \gamma_1 + \gamma_{1, \mB}},
 \label{}
\end{equation}
where recall
\begin{align}
 J_{X, \gamma}^\mu &= \gamma^{\mu\nu}\pa_\nu \psi X\psi -\frac{1}{2}X^\mu
 \gamma(\pa\psi, \pa \psi),\\
 K_{X, \gamma}
 &= \pa_\mu X^\alpha \gamma^{\mu\nu} \pa_\nu \psi \pa_\alpha \psi
  - \frac{1}{2} (\pa_\alpha X^\alpha) \gamma(\pa \psi, \pa \psi)
 - \frac{1}{2} (X^\alpha \pa_\alpha \gamma^{\mu\nu}) \pa_\mu\psi \pa_\nu \psi.
 \label{}
\end{align}
The bounds for $J_{X, \gamma_1}$ and $K_{X, \gamma_1}$ can be handled
just as in Lemma \ref{J1K1lem}, with the only difference being that we
use the bound \eqref{Xgrowth2} in place of the assumption \eqref{Xgrowth}
to introduce powers of $(1+s)^{-1}$.
For the contribution from $\gamma_{1, \mB}$, we just note that
\begin{multline}
 |\zeta(J_{X, \gamma_{1,\mB}})|
 \lesssim \frac{1}{vs^{1/2}} |\gamma| |\pa \psi| |X \psi|
 + \frac{1}{vs^{1/2}} |\gamma| |X^n_{\mB}| |\pa \psi|^2\\
 \lesssim \frac{1}{vs^{1/2}} |X^\ell_{\mB}|
 |\gamma| |\pa \psi| |\pa \psi|_{X, \mB}
 \lesssim |\gamma| |\pa\psi|_{X, \mB}^2,
 \label{}
\end{multline}
where we used that $\frac{|X^\ell_{\mB}|}{vs^{1/2}}
= \frac{1}{(vs)^{1/2}} \lesssim |X^n_{\mB}|$.
Using the straightforward bound \eqref{naiveKbd}
we get that $K_{X,\gamma_{1, \mB}}$ is bounded by the right-hand side of
\eqref{mBK1bd}.

\textit{Step 2: Using the equation}

Recalling that the equation in this region reads
\begin{equation}
	-4n\evmB \psi + \sDelta \psi + \pa_\mu\left( \gamma_a^{\mu\nu}\pa_\nu \psi\right)
	+ \pa_\mu (\gamma^{\mu\nu}\pa_\nu \psi) + \pa_\mu P^\mu = F,
 \label{dudvmBmodel}
\end{equation}
the identity \eqref{modified0} is replaced by
\begin{multline}
 \pa_\mu(\gamma^{\mu\nu}\pa_\nu\psi)
 = \pa_u\left({\gamma}_{\mB}^{uu} \pa_u\psi
 \right) + \pa_\mu({\gamma}_1^{\mu\nu} \pa_\nu \psi)
 + \pa_\mu(\gamma_{1,\mB}^{\mu\nu}\pa_\nu \psi) + \pa_\mu({\gamma}_2^{\mu\nu} \pa_\nu \psi)\\
 + \pa_\mu \left( \check{\gamma} \gamma_a^{\mu\nu}\pa_\nu \psi\right)
 + F_1 + F_{2, \mB} + F_{a},
 \label{mBgammaident0}
\end{multline}
where $\check{\gamma} = \gamma^{uv} + \gamma^{uv}$,
 $\gamma_2$ is as in \eqref{gamma1def}, \eqref{gamma2def}
but with $\gamma_\chi$ replaced with $\gamma$, where the above
quantities are defined as follows. First,
\begin{equation}
 {\gamma}_{\mB}^{uu} = \gamma^{uu} - \gamma^{vu} \frac{u}{vs},
 \label{}
\end{equation}
 $F_1$ is as in \eqref{F1}, and $F_{2,\mB}$ is as in \eqref{F2} except
that there are additional terms, generated by the third term in
\eqref{startmodifiedmB},
\begin{align}
 F_{2,\mB} = - \check{\gamma}\pa_\mu P^\mu
 + \check{\gamma} \gamma^{\mu\nu}
 \pa_\nu \psi
 - \frac{1}{4}\nas\check{\gamma}\cdot \nas \psi+ (\pa_v \gamma^{vu})\pa_u \psi
 + (\pa_u \gamma^{uv}) \pa_v\psi
 + \frac{u}{vs} \pa_u \check{\gamma} \pa_u\psi,
 \label{f2mBbd}
\end{align}
and, finally
\begin{equation}
 F_{a} = -\gamma_a^{\mu\nu}\pa_\mu \check{\gamma} \pa_\nu \psi.
 \label{}
\end{equation}

The quantity $|F_{2, \mB} X\psi|$ can be bounded just as how we controlled
$|F_{2} X\psi|$ starting in equation \eqref{F2}, except that we want to keep track
of the $u$ component of $P$ separately, and so we write
\begin{multline}
 |\pa_\mu P^\mu \check{\gamma}| |X \psi|
 \lesssim
 \left(|\pa_u P^u|
 + |\pa_v P^v| + |\nas P|\right) |\check{\gamma}||X^\ell_{mB}|^{1/2} |\pa \psi|_{X, \mB}
 \\
 \lesssim |X^n_{\mB}|^{1/2} |\pa P^u| |\pa \psi|_{X, \mB}
 + |X^\ell_{\mB}|^{1/2} \left( |\nabla_{\evmB} P|  + |\nas P|
 + \frac{1}{1+v} |\pa P| + \frac{1}{1+v} |P|\right)
 |\pa \psi|_{X, \mB},
 \label{specialPumB}
\end{multline}
where we used that $|\gamma| |X^\ell_{\mB}|^{1/2} \lesssim \frac{1}{(1+v)^{1/2}}
\lesssim |X^n_{\mB}|^{1/2}$ and bounded $|\pa_v P^v|\lesssim |\evmB P^v|
+ \frac{1}{1+v} |\pa_u P^v|$.

We note at this point that $\gamma_{\mB}^{uu}$ satisfies the bounds
\begin{equation}
 |\gamma^{uu}_{\mB}|\lesssim |\gamma|,
 \qquad
 |\nabla\gamma^{uu}_{\mB}| \lesssim |\nabla\gamma| +  \frac{1}{1+v}|\gamma|,
 \label{gammamBbd}
\end{equation}
if $|u| \lesssim s^{1/2}$, say.

As in \eqref{modified1}, the identity \eqref{mBgammaident0} gives
\begin{equation}
 \pa_\mu(\gamma^{\mu\nu}\pa_\nu \psi)X\psi = \pa_u \left(\gamma^{uu}_{\mB} \pa_u\psi\right) X^v \pa_v \psi
 + \pa_\mu J_{X}^{1, \mu} + K_{X, \mB}^1,
 \label{}
\end{equation}
where
\begin{align}
	{J}_{X}^{1,\mu} &= J_{X, \gamma_1} + J_{X, \gamma_2}+ J_{X, \check{\gamma} \gamma_a} + J_{X^u\pa_u,\gamma}
	,
	\label{J1K1mB2}
	\\
	K_{X}^1 &= K_{X, \gamma_1}
	+ K_{X, \gamma_2} + K_{X^u\pa_u,\gamma}
	+ K_{X, \check{\gamma} \gamma_a}+ (F_1 + F_{2,\mB} + F_a) X\psi.
 \label{J1K1mB}
\end{align}

Apart from the contribution from the quantity $\check{\gamma}\gamma_a$,
the rest of the argument from Lemma \ref{J1K1lem} then goes through without
change, using \eqref{gammamBbd} to handle the terms contributed by $\gamma_{\mB}^{uu}$,
and so the quantities in \eqref{J1K1mB} satisfy the bounds in Lemma \ref{mainmBidentprop}.
Using \eqref{energycurrentbd0} and \eqref{scalarcurrentbd0} it is not
hard to see that the terms contributed by $a$ into
\eqref{J1K1mB2} and \eqref{J1K1mB} are bounded by \eqref{zetamBJ} and
\eqref{mBmodifiedKbound}, respectively. To handle the contribution
from $\gamma_a$, we just bound
\begin{equation}
 |\zeta(J_{X, \check{\gamma} \gamma_a})|
 \lesssim |\check{\gamma}| |\zeta(J_{X, \check{\gamma} \gamma_a})|
 \lesssim \epsilon |\zeta(J_{X,\gamma_a})|,
 \label{}
\end{equation}
which appears in \eqref{mBJ1bd}.
Using Lemma
 \ref{nulllemmascalarcurrent} and straightforward estimates for
$\check{\gamma}$, $K_{X, \check{\gamma} \gamma_a}$ is bounded by the
last term on the right-hand side of \eqref{mBK1bd}.
\end{proof}


The next step is the analogue of Lemma \ref{J2K2lem} with
$\gamma^{uu}$ replaced with ${\gamma}^{uu}_{\mB}$.
\begin{lemma}
	\label{J2K2lemmB}
Under the hypotheses of Proposition \ref{mainmBidentprop}, we have
 \begin{equation}
  \pa_u \left(\gamma^{uu}_{\mB} \pa_u \psi\right) X^v\pa_v\psi
	= \pa_\mu J_{X}^{2,\mu} + K_X^2,
  \label{}
 \end{equation}
 where $J_X^2, K_X^2$ are given explicitly in \eqref{j2mBexplicit}-\eqref{k2mBexplicit}
 and satisfy
 \begin{align}
  |\zeta(J_X^2)| &\lesssim |\slashed{\zeta}|^2 |\pa \psi|_{X, \mB}^2
	+ |\gamma| |\pa \psi|_{X, \mB}^2 + \epsilon |\zeta(J_{X, \gamma_a})|\\
	|K_X^2| &\lesssim \left( |\nabla \gamma|  + \frac{|\gamma|}{1+s}
	+ \frac{|X^\ell_{\mB}|^{1/2}}{|X^n_{\mB}|^{1/2}} |\nas \gamma|\right) |\pa \psi|_{X,\mB}^2
  \label{}
 \end{align}
\end{lemma}

\begin{proof}
	Recalling
$X^v = v$, the identity
\eqref{modified2} is replaced with
\begin{align}
 \pa_u &\left(\gamma^{uu}_{\mB} \pa_u \psi\right) X^v\pa_v\psi\\
 &= \pa_u\left( \gamma^{uu}_{\mB}v \pa_u\psi \pa_v\psi\right)
 - \gamma^{uu}_{\mB} v\pa_u \psi n\ell^{\mB} \psi
 + \gamma^{uu}_{\mB} v \pa_u \psi \pa_u\left( \frac{u}{vs} \pa_u\psi\right)\\
 &=\pa_u\left( v\gamma^{uu}_{\mB}\pa_u\psi \pa_v\psi\right) + \frac{1}{2}
 \pa_u\left( \gamma^{uu}_{\mB} \frac{u}{s} (\pa_u\psi)^2\right)
 -\frac{1}{2}\pa_u \left( \gamma^{uu}_{\mB}\right) \frac{u}{s} (\pa_u\psi)^2
 - \gamma^{uu}_{\mB} v \pa_u \psi n\ell^{\mB} \psi.
 \label{}
\end{align}

Following the same steps that led to \eqref{j2k2}, we get
\begin{multline}
 \pa_u \left(\gamma^{uu}_{\mB} \pa_u \psi\right) X^v\pa_v\psi
 = \pa_\mu \slashed{J}_{X, \mB} + \pa_\mu \widehat{J}_{X, \mB} +
 \pa_\mu \check{J}_{X, \mB}^\mu + \pa_\mu (\check{J}^\prime_{X, \gamma_a})^\mu
 \\
 +
 \slashed{K}_{X,\mB} + \widehat{K}_{X,\mB}
 + \check{K}_{X, \mB} + \check{K}^\prime_{X, \gamma_a},
 \label{}
\end{multline}
with
\begin{align}
 \slashed{\widehat{J}}^{\mu}_{X,\mB}
 &= -\frac{1}{4}X^v \gamma_{\mB}^{uu} \pa_u\psi \nas^\mu \psi
 + \frac{1}{8} \delta^{\mu u} X^v \gamma_{\mB}^{uu} |\nas \psi|^2,
 \\
\widehat{J}_{X,\mB}^\mu &= X^v \gamma_{\mB}^{uu} \gamma^{\mu\nu}\pa_u\psi \pa_\nu \psi
 - \frac{1}{2} \delta^{\mu u} X^v \gamma_{\mB}^{uu} \gamma(\pa \psi, \pa \psi),
 \\
 \check{J}_{X, \mB}^\mu &= \delta^{\mu u} \left( \gamma^{uu}_{\mB}v \pa_u\psi
 \left(\pa_v\psi + \frac{u}{2vs} \pa_u\psi\right) \right),
 \\
 \slashed{K}_{X, \mB}
 &= \frac{1}{4} v \nas \gamma_{\mB}^{uu} \pa_u \psi \nas \psi
 - \frac{1}{8} v \pa_u\gamma_{\mB}^{uu} |\nas \psi|^2,
 \\
 \widehat{K}_{X,\mB}
 &=-\pa_\mu (v\gamma_{\mB}^{uu}) \gamma^{\mu\nu} \pa_u\psi\pa_\nu\psi
 + \frac{1}{2} v \pa_u\gamma_{\mB}^{uu} \gamma (\pa\psi,\pa\psi),
 \\
 \check{K}_{X, \mB}&=
 -\gamma_{\mB}^{uu}v \pa_u\psi  \pa_\mu P^\mu
 + \gamma_{\mB}^{uu}\pa_u\psi v F
  -\frac{1}{2} \pa_u { \gamma^{uu}_{\mB} }\frac{u}{s} (\pa_u\psi)^2,
 \label{}
\end{align}
and where the contributions from $\gamma_a$ are collected in
\begin{align}
 (\check{J}^\prime_{X, \gamma_a})^\mu
 &= \gamma_{\mB}^{uu}\left( \gamma_a^{\mu\nu} v \pa_v\psi
 \pa_\nu \psi - \frac{1}{2} \delta^{\mu u} v \gamma_a(\pa \psi, \pa \psi)\right),
 \label{checkjprime}\\
 \check{K}^\prime_{X, \gamma_a}
 &=
 \frac{1}{2} \pa_u \left( \gamma_{\mB}^{uu} \gamma_a^{\mu\nu}\right) v \pa_\mu\psi
 \pa_\nu \psi
 - \pa_\mu( \gamma_{\mB}^{uu} v) \gamma_a^{\mu\nu}\pa_u \psi \pa_\nu \psi.
 \label{checkkprime}
\end{align}

In light of \eqref{gammamBbd}, the above energy currents $\slashed{J}, \widehat{J}$ satisfy
the same bounds as those in Lemma \ref{J2K2lem}, while $\check{J}$ satisfies
\begin{equation}
 |\zeta(\check{J}_{X, \mB})| \lesssim |X| |\gamma| |\pa \psi| |\pa_v \psi|
 + \frac{u}{s}|\gamma| |\pa \psi|^2
 \lesssim |X||\gamma| |\ell^{\mB}\psi||\pa \psi|
 + |X^n_{\mB}| |\gamma| |\pa \psi|^2,
 \label{}
\end{equation}
where we used that $(1+v)^{-1}(1+s)^{-1/2}|X| \lesssim X^n_{\mB}$ by the
assumption on $X$ and the assumption that $|u|\lesssim s^{1/2}$.
Similarly, the scalar currents $\slashed{K}, \widehat{K}$ satisfy
the same bounds as those in Lemma \ref{J2K2lemmB}, and $\check{K}$ satisfies
\begin{equation}
 |\check{K}_{X, \mB}|
 \lesssim
 |X| |\gamma||\pa \psi| \left(|\nabla P| + |F|\right)
 + |\gamma| |\pa \psi| |P|
 + |X^n_{\mB}| \left(
 |\nabla \gamma| + \frac{1}{1+v} |\gamma|\right) |\pa \psi|^2.
 \label{}
\end{equation}
Here, we used that $v \lesssim |X|$ since $X^v = v$ and we are assuming that
$|X^n_{\mB}|\lesssim X^v$.
To control the last term in the definition of $\check{K}$, we used that
by assumption $|X^n_{\mB}| \gtrsim (1+s)^{1/2}$, which, combined with
\eqref{gammamBbd} and the assumption $|u|\lesssim (1+s)^{1/2}$,
gives
\begin{multline}
 |\pa_u (\gamma_{\mB}^{uu}) \frac{u}{s}| |\pa \psi|^2 \lesssim
  \frac{1}{(1+s)^{1/2}} |\nabla \gamma| |\pa \psi|^2
	+ \frac{1}{(1+v)(1+s)^{1/2}} |\gamma||\pa \psi|^2\\
	\lesssim |X^n_{\mB}| \left( |\nabla \gamma|  +\frac{1}{1+v} |\gamma|\right)|\pa \psi|^2.
 \label{}
\end{multline}
The quantities \eqref{checkjprime}-\eqref{checkkprime}
can be handled in the same way we handled the quantities
$J_{X, \check{\gamma} \gamma_a}$ and $K_{X, \check{\gamma} \gamma_a}$ above
and this gives the result with
\begin{align}
 J_{X}^2 &= \widehat{\slashed{J}}_{X,\mB} + \widehat{J}_{X, \mB} + \check{J}_{X, \mB}
 +\check{J}^\prime_{X},\label{j2mBexplicit}\\
 K_{X}^2 &= \widehat{\slashed{K}}_{X,\mB} + \widehat{K}_{X, \mB} + \check{K}_{X, \mB}
 +\check{K}^\prime_{X}.
 \label{k2mBexplicit}
\end{align}

\end{proof}

It remains to prove the analogue of Lemma \ref{littlejklem}.
For this we will argue almost exactly as in that result, but we will
need to keep track of some of the terms a little differently.
\begin{lemma}
	\label{littlejklemmB}
	Under the hypotheses of Proposition \ref{mainmBidentprop},
	\begin{equation}
	 \pa_\mu P^\mu X\psi = \pa_\mu j^\mu_{P, X} + k_{P, X}
	 \label{}
	\end{equation}
	where $j_{P, X}$ and $k_{P, X}$ are given explicitly
	in \eqref{littlejmB}-\eqref{littlekmB}
	and satisfy the following bounds. For $\delta > 0$,
	\begin{align}
	 |\zeta(j_{X,P})|
	 &\lesssim \delta v |\evmB \psi|^2
	 + \delta \frac{1}{(1+v)(1+s)^{1/2}} |\pa \psi|_{X, \mB}^2
	 +  \left(1 + \frac{1}{\delta}\right) (1+v) |P^u|^2
	 +  |\slashed{\zeta}|^2 |\pa \psi|_{X, \mB}^2,
	 \label{littlejbdmB}\\
	 |k_{X, P}|&\lesssim
	  \left(  |\nabla P^u| + \frac{|P^u|}{1+s}\right)|X^n_{\mB}|^{1/2}|\pa \psi|_{X, \mB}\\
	 &+  \left(|\nabla_{\evmB} P|
	 + |\nas P| + \frac{1}{1+v}|\nabla P| + \frac{1}{1+v} |P|\right)|X^\ell_{\mB}|^{1/2} |\pa \psi|_{X, {\mB}}\\
	 &+
	 \frac{1}{(1+s)^{1/2}} \left( |\nabla P^u| + \frac{|P^u|}{1+v}\right) |\pa \psi|
	 +v|P| \left(|\nabla P| + \frac{|P|}{1+v} + |F|\right).
		\label{littlekbdmB}
 \end{align}
\begin{remark}
 For our applications, $P = P_{I, lin} + P_{I,nl}$ where $P_{I,nl}$
 collects lower-order nonlinear commutation errors and
 $P_{I, lin}$ collects lower-order linear commutation errors,
 both of which appear after commuting the equation with $Z_{\mB}^I$.
 For our estimates, the nonlinear errors are not particularly
 dangerous but the linear terms are somewhat complicated to handle,
 because (see \eqref{nullformulacommute}) the $u$-component
 $P^{u}_{I, lin}$ behaves like
 $\frac{1}{1+v} \sum_{|J| \leq |I|-2} \pa Z_{\mB}^J \psi$, up
 to better-behaved terms. Using our bootstrap assumptions, this
 (just) fails to be bounded in $L^1_tL^2_x$. This issue is dealt with
 in Lemma \ref{timeintegrability-center-linear}.
\end{remark}
\end{lemma}
\begin{proof}
	As in Lemma \ref{littlejklem} it is only the term
	$\pa_u P^u X^\ell_{\mB} \evmB \psi$ that needs to treated
	by integrating by parts, so we write
\begin{equation}
 \pa_u P^u X^v \pa_v\psi = \pa_u P^u X^v \ell^{\mB} \psi
  -
 X^v \frac{u}{vs} \pa_u P^u \pa_u \psi,
 \label{}
\end{equation}
and following the argument from Lemma \ref{littlejklem} and recalling $X^v = v$, this leads to
$\pa_\mu P^\mu X\psi = \pa_\mu j_{X, P}^\mu + k_{X, P}$ with
\begin{align}
 j_{X,P}^\mu &= P^u v \left( -\delta^{\mu u} \ell^{\mB} \psi + \frac{1}{4} \nas^\mu \psi - \gamma^{\mu\nu}
 \pa_\nu \psi - \gamma_a^{\mu\nu}\pa_\nu \psi\right),\label{littlejmB}\\
 k_{X,P} &= \pa_\mu \widetilde{P}^\mu X\psi - \frac{1}{4}v \nas P^u\cdot \nas \psi
 + \pa_\mu(P^u v)\gamma^{\mu\nu}\pa_\nu \psi
 + v P^u \pa_\mu P^\mu
 - v P^u F\\
 &\qquad-\frac{u}{s} \pa_u P^u \pa_u \psi + \pa_\mu(v P^u)\gamma_a^{\mu\nu}\pa_\nu \psi.
 \label{littlekmB}
\end{align}
Now, we have
\begin{equation}
 |\zeta(j_{X, P})|\lesssim
 v |P^u||\evmB \psi| + v |P^u| |\zeta(\nas \psi)|
 + v|\gamma| |P^u| |\pa \psi|
 + v |\gamma_a | |P^u| |\pa \psi|.
 \label{zetacentral0}
\end{equation}
The first two terms here are bounded by
\begin{equation}
  v |P^u||\evmB \psi| + v |P^u| |\zeta(\nas \psi)|
	\lesssim \delta v |\evmB \psi|^2 + v |\slashed{\zeta}|^2 |\nas \psi|^2
	+ \left(1 + \frac{1}{\delta}\right) v |P^u|^2.
 \label{}
\end{equation}
For the third term in \eqref{zetacentral0}, we bound
\begin{multline}
 v |\gamma| |P^u| |\pa \psi|
 \lesssim
 \frac{\epsilon}{(1+s)^{1/2}} |P^u| |\pa \psi|
 \lesssim \delta \frac{1}{1+v} \frac{1}{1+s} |\pa \psi|^2
 + \frac{1}{\delta} (1+v)|P^u|^2\\
 \lesssim \delta \frac{1}{(1+v)(1+s)^{1/2}} |\pa \psi|_{X, \mB}^2
 + \frac{1}{\delta} (1+v)|P^u|^2,
 \label{}
\end{multline}
where we used that $|X^n| \gtrsim (1+s)^{-1/2}$.
Since $v |\gamma_a| = v |\frac{u}{vs} a| \lesssim \frac{1}{s^{1/2}}$,
in the same way we have
\begin{equation}
 v |\gamma_a | |P^u| |\pa \psi|
 \lesssim \frac{1}{(1+s)^{1/2}} |P^u| |\pa \psi|
 \lesssim \delta  \frac{1}{(1+v)(1+s)^{1/2}} |\pa \psi|_{X,\mB}|^2
 + \frac{1}{\delta} (1+v)|P^u|^2,
 \label{}
\end{equation}
and combining the above we get \eqref{littlejbdmB}.

For the quantity $k_{X, P}$, we first bound
\begin{equation}
 |\pa_v P^v | |X\psi| + |\nas \slashed{P}| |X \psi|
 \lesssim |X^\ell_{\mB}|^{1/2} \left(|\nabla_{\evmB} P| + |\nas P|
 + \frac{1}{1+v} |\nabla P|  + \frac{1}{1+v} |P|\right) |\pa \psi|_{X, \mB},
 \label{}
\end{equation}
which is bounded by the right-hand side of \eqref{littlekbdmB},
since we easily have $|X^\ell_{\mB}|^{1/2} (1+v)^{-1} \lesssim |X^n_{\mB}|^{1/2}
(1+s)^{-1}$. The second term in $k_{X, P}$ is bounded by
\begin{equation}
 |v \nas P \cdot \nas \psi|
 \lesssim v^{1/2} |\nas P| |\pa \psi|_{X,\mB}
 = |X^\ell_{\mB}|^{1/2} |\nas P| |\pa \psi|_{X, \mB}.
 \label{}
\end{equation}
Next, using the bound $|\gamma| \lesssim \epsilon (1+v)^{-1}(1+s)^{-1/2}$, we have
\begin{equation}
 |\pa (v P^u) ||\gamma| |\pa \psi|
 \lesssim \frac{\epsilon}{(1+s)^{1/2}}\left( |\nabla P^u| + \frac{|P^u|}{1+v}\right)
 |\pa \psi|,
 \label{}
\end{equation}
and since $|\gamma_a|\lesssim \frac{1}{(1+v)(1+s)^{1/2}}$ we also have
\begin{equation}
 |\pa(vP^u)| |\gamma_a| |\pa \psi|
 \lesssim \frac{1}{(1+s)^{1/2}} \left( |\nabla P^u| + \frac{|P^u|}{1+v}\right) |\pa \psi|.
 \label{}
\end{equation}
We also have
\begin{equation}
 \left|\frac{u}{s}\right| |\pa_u P^u| |\pa \psi|
 \lesssim
 \frac{1}{(1+s)^{1/2}}\left( |\nabla P^u| + \frac{|P^u|}{1+v} \right)|\pa \psi|,
 \label{}
\end{equation}
and bounding
\begin{equation}
 |v P^u \pa_\mu P^\mu|
 + |v P^u F|
 \lesssim v |P| \left( |\nabla P| + \frac{|P|}{1+v} + |F|\right),
 \label{}
\end{equation}
we get the result.
\end{proof}

\begin{proof}[Proof of Proposition \ref{mainmBidentprop}]

By Lemmas \ref{J1K1lemmB}-\ref{littlejklemmB}
the identity \eqref{mBidentapp} holds with $\mJ$ and $\mK$ given by
\begin{align}
 \mJ_{X, P} &=  \mJ_X + j_{X, P} = J_X^1 + J_X^2 + j_{X, P},
 \label{explicitmJmB}\\
 \mK_{X, P} &= \mK_X + k_{X, P}= K_X^1 + K_X^2 + k_{X, P},
 \label{explicitmKmB}
\end{align}
where $J_X^1, K_X^1$ are given in \eqref{J1K1mB}, $J_X^2, K_X^2$ are given by
\begin{equation}
 J_{X}^2 = \slashed{\widehat{J}}_{X,\mB} + \widehat{J}_{X,\mB} + \check{J}_{X, \mB},
 \qquad
 K_X^2 = \slashed{K}_{X, \mB} + \widehat{K}_{X,\mB} + \check{K}_{X, \mB},
 \label{}
\end{equation}
and $j_{X, P}, k_{X, P}$ are given in \eqref{littlejmB}-\eqref{littlekmB}.
After bounding $|\gamma| \lesssim \frac{\epsilon}{(1+v)(1+s)^{1/2}}$ in
\eqref{mBJ1bd}, we get the stated bounds.
\end{proof}

\subsection{Estimates for a linear term verifying the null condition}
In the central region, we need to deal with a linear term
$\pa_\mu(\gamma_a^{\mu\nu}\pa_\nu \psi)$ with $\gamma_a^{\mu\nu} =\frac{u}{vs} a^{\mu\nu}$
where $a^{\mu\nu} = a^{\nu\mu}$ satisfies the null
 condition \eqref{intronullcondn0}. In the next lemma we control the
 energy current contributed by this term, and in Lemma \ref{nulllemmascalarcurrent}
 we handle the scalar current. Thanks to the smallness of the coefficient
 $u/vs$ along the shocks and the fact that $a$ verifies the null condition,
 these terms can be treated perturbatively.
 \begin{lemma}
	 \label{nulllemmacurrent}
	 Define $\gamma_a$ as in the above paragraph.
	 Suppose that
	 $X^\ell_{\mB} = v$ and that $(1+s) \gtrsim |X^n_{\mB}|\gtrsim (1+s)^{-1/2}$.
	Suppose that the condition \eqref{largestart} holds.
    With the energy current $J_{X, \gamma_a}$ defined as in \eqref{energycurrent}
			there are continuous functions $c^i_0$ with $ c^i_0(0) = 0$,
			for $i = 1,2,3$ so that
	\begin{equation}
	 |dt(J_{X, \gamma_a})| \lesssim c_0^1(\epsilon_0) |\pa \psi|_{X, \mB}^2,
	 \label{QAtimeslice}
	\end{equation}
	and if the assumptions \eqref{betaLassump} and \eqref{betaRassump} hold,
	then at $\Gamma^A \in \{\Gamma^L, \Gamma^R\}$, with $\zeta^A$ as in
	\eqref{zetadef},
	\begin{equation}
	 |\zeta^{A}(J_{X, \gamma_a})| \lesssim
	  c_0^2(\epsilon_0)|\pa \psi|_{X, \mB, \SL}^2
	 \label{QAspacelike}
	\end{equation}
 \end{lemma}
 \begin{proof}
	We first prove that under our hypotheses,
	if $\zeta$ is a one-form with $|\zeta| \leq 1$,
		then when $|u|\lesssim (1+s)^{1/2}$,
		\begin{equation}
		 (1+v)(1+s)^{1/2}|\zeta(J_{X, \gamma_a})|
		 \lesssim  \left(\frac{(1+s)^{1/2}}{(1+v)^{1/2}} +
		 |\overline{\zeta}|\frac{(1+v)^{1/2}}{(1+s)^{1/4}}\right)
		 |\pa \psi|_{X, \mB}^2
		 + |\opa \psi|_{X, \mB}^2,
		 \label{zetaJA}
		\end{equation}
			where we are abusing notation slightly and writing
			\begin{equation}
			 |\opa \psi|_{X, \mB}^2 = |X^\ell_{\mB}|\left( |\evmB \psi|^2 +  |\nas \psi|^2
			 + \frac{1}{(1+v)(1+s)^{1/2}}|n \psi|^2\right).
			 \label{}
			\end{equation}
To prove this, we start by noting that by \eqref{energycurrent}, for any one-form $\zeta$ with $|\zeta|\leq 1$ we have
	\begin{multline}
	 (1+v)(1+s)^{1/2}|\zeta(J_{X,\gamma_a})|
	 \lesssim |\gamma_a(\zeta, \pa \psi)| |X\psi| + |\gamma_a(\pa \psi, \pa \psi)|
	 |\zeta(X)|\\
	 \lesssim
	\left( |\overline{\zeta}| |\pa \psi| + |\zeta| |\opa \psi|\right)|X\psi|
	+ |\pa \psi| |\opa \psi| |\zeta(X)|
	 \label{encurrentbd}
	\end{multline}
	by \eqref{intronullcondn}. Now we bound
	\begin{equation}
	 |\overline{\zeta}| |\pa \psi||X\psi|
	 \lesssim
	 |\overline{\zeta}| |\pa \psi|\left(|X^\ell_{\mB}| |\ell^{\mB}|
	 +|X^n_{\mB}| |n\psi|\right)
	 \lesssim |\overline{\zeta}| \left( 1 + \frac{|X^\ell_{\mB}|^{1/2}}{|X^n_{\mB}|^{1/2}}\right) |\pa \psi|_{X, \mB}^2
	 \label{}
	\end{equation}
	and bounding $|\zeta| \leq 1$ and $|\opa \psi|\lesssim |\evmB \psi| + |\nas \psi|
	+ \frac{1}{(1+v)(1+s)^{1/2}} |n\psi|$, we find
	\begin{equation}
	 |\zeta| |\opa \psi||X\psi|
	 \lesssim |\opa \psi| \left(|X^\ell_{\mB}| |\ell^{\mB}|
	 +|X^n_{\mB} |n\psi|\right)
	 \lesssim
	 \frac{|X^n_{\mB}|^{1/2}}{|X^\ell_{\mB}|^{1/2}} |\pa \psi|_{X, \mB}^2
	 + |\opa \psi|_{X, \mB}^2,
	 \label{}
	\end{equation}

	To handle the last term in \eqref{encurrentbd}, we note that by the
	assumptions on the vector field $X$ we have
	$|X^\ell_{\mB}| |\opa \psi|^2 \lesssim |\pa \psi|^2_{X, \mB}$, since
	$|\opa \psi| \lesssim |\evmB \psi| + |\nas \psi| + \frac{1}{(1+v)(1+s)^{1/2}} |n \psi|$
	and $|X^\ell_{\mB}|(1+v)^{-1}(1+s)^{-1/2} \lesssim |X^n_{\mB}|$. We therefore have
	\begin{equation}
	 |\pa \psi| |\opa \psi||\zeta(X)|\lesssim
	|X^n_{\mB}| |\pa \psi||\opa \psi|  +
	|\zeta_v||X^\ell_{\mB}| |\pa \psi||\opa \psi|
	\lesssim
	\left(\frac{|X^n_{\mB}|^{1/2}}{|X^\ell_{\mB}|^{1/2}}
	+ |\zeta_v| \frac{|X^\ell_{\mB}|^{1/2}}{|X^n_{\mB}|^{1/2}}\right)|\pa \psi|^2_{X, \mB}
	 \label{}
	\end{equation}
	Combining the above, we have
	\begin{equation}
		(1+v)(1+s)^{1/2}|\zeta(J_{X, \gamma_a})|
 	 \lesssim
	 \left(\frac{|X^n_{\mB}|^{1/2}}{|X^\ell_{\mB}|^{1/2}} + |\overline{\zeta}|\frac{|X^\ell_{\mB}|^{1/2}}{|X^n_{\mB}|^{1/2}}\right)
	 |\pa \psi|_{X, \mB}^2
	 + |\opa \psi|_{X, \mB}^2
	 \label{zetaJA0}
	\end{equation}
	which gives \eqref{zetaJA} after using the assumptions on $X$.

	To prove \eqref{QAtimeslice} we take $\zeta = dt$ and bounding
	$|\overline{\zeta}| \lesssim 1$ we find
	\begin{equation}
	 |dt(J_{X, \gamma_a})|
	 \lesssim \frac{1}{(1+v)^{1/2}} |\pa \psi|_{X, \mB}^2
	 \lesssim c_0(\epsilon_0)^1|\pa \psi|_{X, \mB}^2.
	 \label{}
	\end{equation}

	We just prove \eqref{QAspacelike} at the right shock, the proof
	at the left shock being identical.
	We first note that by definition of $|\pa\psi|_{X, \mB, +}$ from
	\eqref{plusnorm},
	\begin{align}
	 \frac{1}{(1+v)(1+s)^{1/2}} |\pa \psi|_{X, \mB}^2
	 &\lesssim |\pa \psi|_{X, \mB, \SL}^2,
	 \\
	 |\opa \psi|_{X, \mB}^2
	 &\lesssim \left(1 + \frac{1}{|X^n_{\mB}|}\right) |\pa \psi|_{X, \mB, \SL}^2
	 \lesssim (1+s)^{1/2}|\pa \psi|_{X, \mB, \SL}^2,
	 \label{}
	\end{align}
	where in the last step we used the assumption on $X$.

	Taking $\zeta = \zeta^{\Gamma^R}$, we have
	$|\overline{\zeta}|\lesssim \frac{(1+s)^{1/2}}{1+v}$ and so
	bounding $\frac{(1+s)^{1/2}}{(1+v)^{1/2}} +
	|\overline{\zeta}|\frac{(1+v)^{1/2}}{(1+s)^{1/4}} \lesssim
	\frac{1}{(1+v)^{1/4}}$, say, \eqref{QAspacelike} gives
	\begin{equation}
	 |\zeta(J_{X, \gamma_a})|
	 \lesssim \frac{1}{(1+v)^{1/8}} |\pa \psi|_{X, \mB, \SL}^2,
	 \label{}
	\end{equation}
	say, which gives the result.

\end{proof}

\begin{lemma}
 \label{nulllemmascalarcurrent}
 Suppose that $X = X^u(u,v)\pa_u + X^v(u,v)\pa_v$ satisfies $|\pa X| \lesssim 1$,
  $X^v = v$, $|X^u|\lesssim (1+s)$,
 $|\pa_v X^u|\lesssim \frac{1}{1+v}$. Then
 \begin{equation}
  |K_{X, \gamma_a}| \lesssim \frac{1}{(1+v)^{3/2}} |\pa \psi|^2 + \frac{1}{(1+v)^{1/2}} |\opa \psi|^2.
  \label{nulllemmascalarcurrentbd}
 \end{equation}
\end{lemma}
\begin{proof}
 From \eqref{scalarcurrent} and the assumption that $|\pa X|\lesssim 1$, we have
 \begin{equation}
  |K_{X, \gamma_a}|
	\lesssim
	|\gamma_a(\pa \psi, \pa \psi)|
	+ |\gamma_{a, X}(\pa \psi, \pa \psi)|
	+ |\pa_\mu X^\alpha \gamma_a^{\mu\nu} \pa_\nu \psi \pa_\alpha \psi|,
  \label{Knull}
 \end{equation}
 where $\gamma_{a, X}^{\mu\nu} = X^\alpha \pa_\alpha \gamma_a^{\mu\nu}$. The third
 term here is bounded by
 \begin{equation}
  |\pa_\mu X^\alpha \gamma_a^{\mu\nu} \pa_\mu \psi \pa_\alpha \psi|
	\lesssim
	|\opa X^\alpha \pa_\alpha \psi| |\gamma_a| |\pa \psi|
	+ |\gamma_a| |\opa \psi | |\pa \psi|
	\lesssim |\gamma_a| |\opa \psi| |\pa \psi|
	+ \frac{1}{1+v} |\gamma_a| |\pa \psi|^2.
  \label{}
 \end{equation}
 Now, $X^\alpha\pa_\alpha \gamma_a^{\mu\nu} = \left(X\tfrac{u}{vs}\right)
 a^{\mu\nu} + \tfrac{u}{vs} X^\alpha \pa_\alpha a^{\mu\nu}$. Since
 $|X\tfrac{u}{vs}| \lesssim \frac{1}{v}$ by assumption, using the upcoming
 Lemma \ref{nullformlemma} we therefore have
 \begin{equation}
  |\gamma_{a, X}(\pa \psi, \pa \psi)|
	\lesssim \frac{1}{1+v} |\pa \psi| |\opa \psi|,
  \label{}
 \end{equation}
 and since $|\gamma_a| \lesssim (1+v)^{-1}$, after bounding
 $|\gamma_a(\pa \psi, \pa \psi)|\lesssim (1+v)^{-1} |\opa\psi| |\pa \psi|$,
 by \eqref{Knull} and the above we have
 \begin{equation}
  |K_{X, \gamma_a}| \lesssim \frac{1}{1+v} |\pa \psi| |\opa \psi|
	\lesssim \frac{1}{(1+v)^{3/2}} |\pa \psi|^2 + \frac{1}{(1+v)^{1/2}} |\opa\psi|^2,
  \label{}
 \end{equation}
 as needed.
\end{proof}
\begin{lemma}[Commutation with null forms in null coordinates]
	\label{nullformlemma}
	Suppose that $a = a^{\alpha\beta}$ are smooth functions satisfying the symbol
	condition \eqref{strongsymbol} and the null condition \eqref{intronullcondn0}.
	Let $a^{\mu\nu}$ denote the components of $a$ expressed in the null
	coordinates $(u, v, \theta_1, \theta_2)$.
	For any vector field $X$, $a^{\mu\nu}_X = X^\alpha\pa_\alpha a^{\mu\nu}$ also satisfies
	\eqref{intronullcondn0}. In particular,
		\begin{equation}
		 | a_X(\xi, \tau)|
		 \lesssim \frac{|X|}{1+v} \left(|\overline{\xi}||\tau| + |\xi| |\overline{\tau}|\right).
		 \label{nullformcommextrav}
		\end{equation}
\end{lemma}
\begin{proof}
The bound \eqref{nullformcommextrav} follows from the null condition
\eqref{intronullcondn0} as in \eqref{intronullcondn} along with
the fact that $|X^\alpha\pa_\alpha a|\lesssim |X|(1+v)^{-1}$ by the symbol
condition. To prove that
this condition holds, we just note that
\begin{equation}
 a_X^{\mu\nu}\pa_\nu u \pa_\mu u = ((X^\alpha\pa_\alpha) a^{\mu\nu})\pa_\mu u\pa_\nu u
 = -a^{\mu\nu} X^\alpha\pa_\alpha (\pa_\mu u \pa_\nu u),
 \label{}
\end{equation}
since $a^{\mu\nu}\pa_\mu u \pa_\nu u = 0$. Since $\pa_\mu u, \pa_\nu u$ are constants
in our coordinate system, the result follows.

\end{proof}

\bibliographystyle{abbrv}

\textsc{Daniel Ginsberg, Department of Mathematics, Brooklyn College (CUNY), Brooklyn, NY 11210}
\newline
\textit{Email address:} \texttt{daniel.ginsberg@brooklyn.cuny.edu}
\newline
\textsc{Igor Rodnianski, Department of Mathematics, Princeton University, Princeton, NJ 08544}
\textit{Email address:} \texttt{irod@princeton.edu}

\end{document}